%% file: Notes-All.tex
\documentclass{book}
\usepackage[utf8]{inputenc}

\usepackage{todonotes}

\newcommand{\gO}{{O$_2$}}
\newcommand{\rO}{{O$_1$}}

\usepackage{asymptote}
\usepackage{ marvosym }
\usepackage{soul}

\usepackage{textcomp}
\usepackage{comment}
\usepackage{amsmath,amssymb,amsthm}
\usepackage{booktabs}
\usepackage{multicol}
\usepackage{longtable}
\usepackage{xspace}
\usepackage{colortbl}
\usepackage{fullpage}
\usepackage[colorlinks,linkcolor=blue,urlcolor=blue,citecolor=blue,destlabel=true]{hyperref}
\usepackage{mathtools}
\usepackage{graphicx}
\usepackage{placeins}
\usepackage{caption}
\usepackage{subfigure}
\usepackage{wasysym}
\usepackage[all]{xy}
\entrymodifiers={+!!<0pt,\fontdimen22\textfont2>}
\usepackage{enumitem}
\usepackage{tikz, pgf}

\newcommand*\circleit[1]{\tikz[baseline=(char.base)]{
            \node[shape=circle,draw,inner sep=2pt] (char) {#1};}}

\usetikzlibrary{arrows,decorations.markings, automata}
\usetikzlibrary{arrows.meta}
\usepackage{enumitem}
\usepackage{venndiagram}
\usepackage{float}
\usepackage{pdfpages}
\usepackage{xcolor}
\definecolor{CSUGreen}{cmyk}{.92, .18, .94, .61}
\definecolor{CSUGold}{cmyk}{.11, .06, .64, .13}
\usetikzlibrary{positioning}
\newdimen\nodeDist
\usepackage{tcolorbox}
\newtcolorbox{videobox}{colback=CSUGold!30!white,colframe=CSUGreen!90!black}

\newtheoremstyle{myremark} 
    {7pt}                    
    {7pt}                    
    {}  	                 
    {}                           
    {\it}       	         
    {.}                          
    {.5em}                       
    {}  

\theoremstyle{plain}
\newtheorem{lemma}{Lemma}[section]
\newtheorem{theorem}[lemma]{Theorem}
\newtheorem{corollary}[lemma]{Corollary}
\newtheorem{proposition}[lemma]{Proposition}
\theoremstyle{definition}
\newtheorem{definition}[lemma]{Definition}
\newtheorem{example}[lemma]{Example}
\newtheorem{question}[lemma]{Question}
\newtheorem*{exercise}{Exercise} 
\newtheorem{notation}[lemma]{Notation}
\newtheorem{remark}[lemma]{Remark}

\newtheorem*{pigeonhole-generalized}{Generalized Pigeonhole Principle}
\theoremstyle{myremark}
\newtheorem*{aside}{Aside}
\newtheorem*{answer}{Answer}

\newtheorem*{brainstorm}{Brainstorm}

\newtheorem*{claim}{Claim}

\newtheorem*{etymology}{Etymology}

\newcommand{\N}{\mathbb{N}}
\newcommand{\Q}{\mathbb{Q}}
\newcommand{\R}{\mathbb{R}}
\newcommand{\Z}{\mathbb{Z}}

\newcommand{\defn}{\textbf}

\newcommand{\cost}{\mathrm{cost}}

\title{{\Huge Counting Rocks!} \\ {\Large An Introduction to Combinatorics} \\ {\Large Colorado State University}}
\author{Henry Adams, Kelly Emmrich, Maria Gillespie, Shannon Golden, and Rachel Pries}
\date{\today}

\begin{document}

\maketitle

\include{Preface}

\tableofcontents

\include{01-Introduction/Introduction}

\include{02-CountingPrinciples/BasicCounting}

\include{03-Combinations/Combinations}

\include{04-BinomialCoefficientsPascal/BinomialCoeffs}

\include{05-ProofTechniques/ProofTechniques}

\include{06-InclusionExclusion-Recursion/IncExcRecurrence}

\include{07-GeneratingFunctions/GeneratingFunctions}

\include{08-GraphsWalksCycles/GraphsWalksCycles}

\include{09-Trees/Trees}

\include{10-GraphOptimization/GraphOptimization}

\include{11-PlanarGraphsEuler/PlanarGraphs}

\include{12-ColoringGraphs/ColoringGraphs}











\end{document}

%% file: Preface.tex
\chapter*{Preface}

This textbook, \emph{Counting Rocks!}, is the written component of an 
interactive introduction to combinatorics at the undergraduate level.
Throughout the text, we link to videos where we describe the material and provide examples; see the \href{https://www.youtube.com/playlist?list=PL5J6K3znOvOmzBUoxlk-W0N4j7L1Y9yfW}{Youtube playlist} on the Colorado State University (CSU) Mathematics Youtube channel.

The major topics in this text are counting problems 
(Chapters 1-4), proof techniques (Chapter 5), recurrence relations and generating functions (Chapters 6-7), and an introduction to graph theory (Chapters 8-12).
The material and the problems we include are standard for an undergraduate combinatorics course.  In this text, one of our goals was to describe the mathematical structures underlying 
problems in combinatorics.  For example, we separate the description of sequences, permutations, sets and multisets in Chapter 3. 

In addition to the videos, we would like to highlight some
other features of this book.  Most chapters contain an investigation section, where students are led through a series of deeper problems on a topic.  In several sections, we show students how to use the free on-line computing software SAGE in order to solve problems; this is especially useful for the problems on recurrence relations.
We have included many helpful figures throughout the text, and we end each chapter (and many of the sections) with a list of exercises of varying difficulty. 

\section*{Development of this textbook}

For several years, faculty taught Math 301, Introduction to Combinatorics, at Colorado State University using the textbook 
\textit{Discrete Mathematics: Elementary and Beyond} by László Lovász, József Pelikán, and Katalin Vesztergombi.
We found this book very inspiring, but we needed to supplement it with 
notes on certain topics.  In 2019, Adams wrote an extensive set of lecture notes for the course, based on the aforementioned book.  
In 2020, Gillespie and Pries decided that their students needed an interactive textbook on the material covered in the CSU 301 combinatorics course in order to learn 
remotely during the COVID-19 pandemic.

This inspired the creation of an accessible, free online textbook along with short accompanying lecture videos.  Adams, Gillespie, and Pries 
applied for an Open Educational Resources (OER) grant at CSU to help with the creation of the materials.  The math department at CSU also provided some funding.  These resources also funded graduate student co-authors Kelly Emmrich and Shannon Golden.
Our title, \emph{Counting Rocks: An Introduction to Combinatorics}, is in part an homage to the Rocky Mountains, near which we love to work.
It is possible that we may update this textbook or create more videos, as future versions of the course are taught.

There are many textbooks written about undergraduate combinatorics, each with its own positive attributes.  This book and videos helped our students and we would like to share them with anyone else who finds them useful.

\section*{Acknowledgments}

The creation of these materials was supported by the CSU math department and the Open Educational Resources grant provided by the CSU Libraries.
We also thank the Fall 2020 Math 301 students at CSU for their helpful feedback and suggestions.

%% file: 01-Introduction/Introduction.tex
\chapter{What is combinatorics?} \label{chap:intro}

Combinatorics is the mathematical study of \textit{discrete}
mathematical objects and their combinations and arrangements.
Other kinds of mathematics (like calculus and analysis) use \textit{continuous} mathematical objects, which involve infinitesimally small increments.
By contrast, \textit{discrete} mathematics is the study of objects that come in larger chunks, such as whole numbers, polygons, and networks.

If you have ever asked ``I wonder how many different ways I could \ldots ?", then you have thought about combinatorics.
For example, if there are 10 classes you are interested in taking this semester, and you want to enroll in either 3 or 4 of them, then there might be over 300 different schedules you could choose from, depending on how many of the classes meet at conflicting times.
How do you count the number of possibilities, without listing them all by hand?
Alternatively, how do you write a computer program to list all of the possible schedules?
Both of these approaches --- either finding a clever way to count or constructing a computer algorithm to do the work for you --- rely heavily on the ideas and language of combinatorics and discrete mathematics.

The history of combinatorics began in ancient times, in places such as Egypt, India, China, and Iran.  The work of scholars in ancient times led to the development of counting techinques during 1300-1700 by individuals such as Fibonacci, Pascal, Leibniz, and De Moivre. In the 1700s and 1800s, mathematicians like Euler and Cayley developed the study of graph theory.

The computer revolution greatly amplified the importance of combinatorics and graph theory.  The most obvious connections between this class and computer science involve networks and organization of
data structures.
There are also applications of this class to other fields, like linguistics, physics, chemistry, and biology.
In general, from this class, you will gain experience solving problems about combinatorial structures and algorithms; we hope you will continue to use the experience and ideas you gain from this class for the benefit of many other people.

Here are some of the ideas we will study this semester.

\section{Introducing combinatorics with a handshake}
\label{S1.1}

Suppose five students walk into their combinatorics class and each shakes hands with all the others.
How many handshakes took place?

There are several ways to approach this problem.
One way is to reason that, because 5 students each shake 4 hands, and each handshake involves 2 students, we counted $5\cdot 4 = 20$ handshakes but counted each handshake twice.
So exactly $10$ handshakes took place.  

Another way to see this is to call the students $a$, $b$, $c$, $d$, $e$, and directly count the handshakes as follows.
Student $a$ shakes hands with $4$ others, then student $b$ shakes hands with $3$ others besides $a$, then
student $c$ shakes $2$ more hands besides $a$ and $b$, then student $d$ shakes 1 more hand (with $e$), and finally student $e$ does not need to shake any more hands. So the total number of handshakes is $1+2+3+4 = 10$.

Finally, we can visually model this problem by drawing a dot to represent each student, and a line segment connecting each pair of dots to represent the handshakes:

\begin{center}
\includegraphics{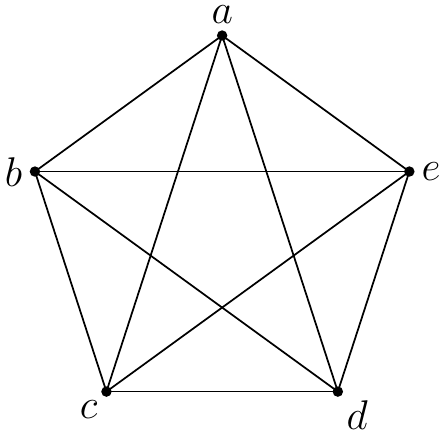}
\end{center}

With this visual tool, called a \textit{graph}, we can count that there are 10 line segments (the handshakes) between the labeled dots (the students).
The graph above is called $K_5$, the ``complete graph on $5$ vertices.''
The dots representing the students are called \textit{vertices} of the graph, and line segments between the vertices are called \textit{edges}.
The intersection points of the line segments which are not labeled should be ignored; they are not vertices.
We say the graph is \textit{complete} because every two vertices are connected by an edge.

In this class, Introduction to Combinatorics, we will learn about counting problems, mathematical proofs, generating functions, and graph theory.
Here are some examples of questions that we will answer.
These questions will become more complicated when we later increase the number of students beyond 5.

\subsubsection*{Chapters 2--4:
Basic counting, sets, and binomial coefficients.} 

\begin{enumerate}
\item How many ways can the 5 students in the class line up? 
\item How many ways can 3 students be chosen for a project from the class of 5 students?
\item If each student $a$, $b$, $c$, $d$, and $e$ receives either $0$ or $1$ point for the quiz, how many sequences of five numbers are possible as the list of their grades?
\end{enumerate}

\subsubsection*{Chapters 5--7:
Counting techniques using more theory, including bijections, recurrence relations, and generating functions.}

\begin{enumerate}
\item How many ways can 10 pieces of halloween candy be distributed among the students? 
\item How many problems on the group worksheet will be solved if the first student does one problem, each other student does twice as many as the previous, and no problems are solved twice? 

\item How many ways can a student make change for one dollar? 
\end{enumerate}

\subsubsection*{Chapters 8--10:
Graph theory and optimization.}

\begin{enumerate}
\item If student $a$ has important news about the class, how many phone calls are needed to share it with the other students? 
\item How many cables are needed to connect the 40 buildings on campus with fiber optic internet?
\item Without lifting your pencil from the page, can you trace your pencil along the edges of the graph $K_5$ so that each edge is drawn exactly once?

\end{enumerate}

\subsubsection*{Chapters 11--12:
Combinatorial geometry, including planar graphs and coloring problems.}

\begin{enumerate}
\item Is it possible to redraw the edges of the graph $K_5$ on the page so that they do not intersect? 

\item On a piece of paper, $4$ of the students draw a line. What is the largest number of regions the paper can be divided into?  What about for $5$ students?

\item Is it possible to color the regions in the previous problem with $3$ colors so that any two regions sharing a boundary edge have different colors?  What about $2$ colors?
\end{enumerate}

\subsection*{Exercises}

Before getting started, let's discuss the questions above. At this stage, it is more important to develop good communication patterns with your classmates than to answer the problems.\footnote{ 
If you can already answer all these problems, for an arbitrary number of students, you will receive an automatic transfer to the class on graduate combinatorics Math 501.}
A group works best if everyone speaks for about the same amount of time and if questions are valued as much as answers.
Make sure that everyone in your group feels comfortable sharing their ideas and understands the material being discussed.

For each problem that your group can solve, write down an explanation of your answer and your thought process. 
For each problem that your group cannot solve: 
decide if the question is clear or ambiguous; decide if all the information needed to solve the problem is provided; and provide some ideas 
or a strategy for solving the problem.  

\begin{enumerate}

\item 
Introduce yourself to your group, including something about why you are taking this class, and something about your life outside of this class.

\item Try to solve the questions from Chapters 2-4.

\item Develop a strategy to approach the problems from Chapters 5-7.

\item Try to solve the problems from Chapters 8-10.

\item Generate ideas about the problems in Chapters 11-12.

\end{enumerate}

\section{Three classical counting formulas}

As an introduction to the techniques in this class, we now discuss three classical problems in mathematics that can be solved easily with combinatorial methods.

\subsection{Triangular numbers}

\begin{videobox}
\begin{minipage}{0.1\textwidth}
\href{https://www.youtube.com/watch?v=MJ3HzLuY9TY}{\includegraphics[width=1cm]{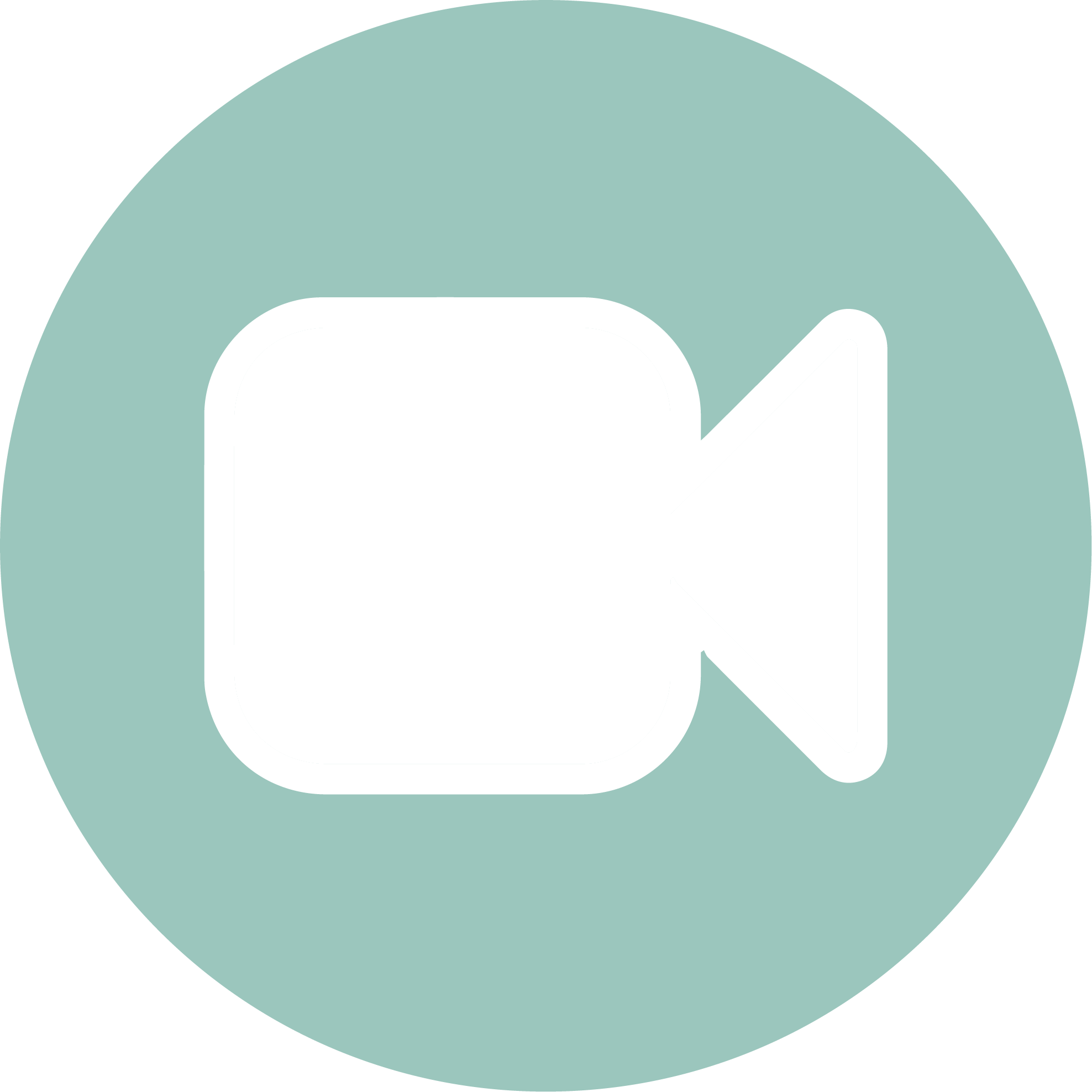}}
\end{minipage}
\begin{minipage}{0.8\textwidth}
Click on the icon at left or the URL below for this section's short video lecture. \\\vspace{-0.2cm} \\ \href{https://www.youtube.com/watch?v=MJ3HzLuY9TY}{https://www.youtube.com/watch?v=MJ3HzLuY9TY}
\end{minipage}
\end{videobox}

We saw in the previous section that we could count the handshakes among $5$ students by adding the numbers $1+2+3+4$.
For $101$ students shaking hands, we can count the number of handshakes by adding 
the numbers from $1$ to $100$.
One of the first signs of Gauss' strength in mathematics was in elementary school when he quickly computed that
\[1 + 2 + 3 + \cdots + 99 + 100 = 5050.\]
More generally, 

\begin{lemma}\label{lem:Tn}
The sum of the first $n$ integers is
\begin{equation} \label{Egauss}
1+2+\ldots+n=\frac{n(n+1)}{2}.
\end{equation}
\end{lemma}

One way to show this is by using a method of proof called \textit{induction}, which will be covered in Section~\ref{sec:induction}. Here, we prove this identity directly in two more straight forward ways.

\begin{proof}
\textbf{Method 1:} This is an example of a proof by rearrangement. We can add the numbers in pairs, starting with the first and last, then the second and second-to-last, and so on:
\[(1+n) + (2 + (n-1)) + (3 + (n-2)) + \cdots.\]
Each pair sums to $n+1$.
If $n$ is even, then there are $\frac{n}{2}$ pairs, with the last pair being $\frac{n}{2}$ and $\frac{n}{2}+1$. So the total is the product of $n+1$ (the sum of each pair)
and $\frac{n}{2}$ (the number of pairs) which equals $(n+1)\frac{n}{2}$.
If $n$ is odd, then there are $\frac{n-1}{2}$ pairs, and the entry $\frac{n+1}{2}$ is not included in a pair.
So the total is
\[(n+1)\frac{n-1}{2} + \frac{n+1}{2} = \frac{n^2-1}{2}+\frac{n+1}{2} = \frac{n^2+n}{2}.\]
In either case, the total is $\frac{n(n+1)}{2}$.

\smallskip

\textbf{Method 2:} This is an example of a \textit{combinatorial proof},
in which we count something in two different ways.
In this example, the first way will yield a count of $1+2+\dots+n$, and the second way will yield a count of $\frac{n(n+1)}{2}$.
Since both are valid ways of counting the same object, we will have proven that $1+2+\dots+n=\frac{n(n+1)}{2}$.
How might we do this?

Draw a rectangle of dots, with $n$ rows and $n+1$ columns.
In the lower left half of the rectangle, color the dots green, with $1$ filled dot in the first row, $2$ in the second row, $3$ in the third row, up to $n$ dots in the $n$th row.
The number of green dots is 
$1 + 2 + \cdots + n$, which is the left hand side of Equation~\ref{Egauss}.
The rectangle has $n(n+1)$ dots total and exactly half of them are green.
So the number of green dots is $n(n+1)/2$, which is the right hand side of Equation~\ref{Egauss}.
Hence $1+2+\dots+n=\frac{n(n+1)}{2}$.

\begin{center}
\includegraphics{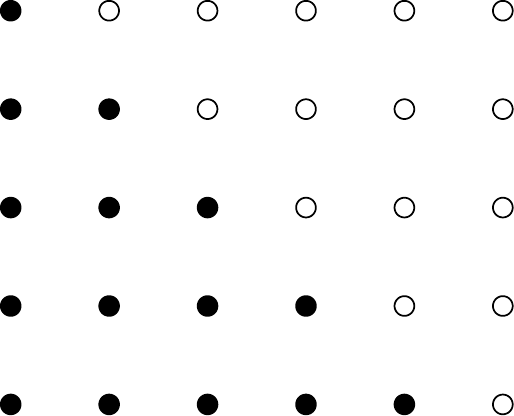}
\end{center}
\end{proof}

Due to the triangular nature of the picture above, the numbers $$T_n=\frac{n(n+1)}{2}$$ are called \textit{triangular numbers.} The first several triangular numbers are $1,3,6,10,15,\ldots$. 
The number of handshakes among $n+1$ people equals $T_n$.

\subsection{Factorials and orderings}\label{sec:factorials-intro}

\begin{videobox}
\begin{minipage}{0.1\textwidth}
\href{https://www.youtube.com/watch?v=sG4jx_p9PbA}{\includegraphics[width=1cm]{video-clipart-2.png}}
\end{minipage}
\begin{minipage}{0.8\textwidth}
Click on the icon at left or the URL below for this section's short video lecture. \\\vspace{-0.2cm} \\ \href{https://www.youtube.com/watch?v=sG4jx_p9PbA}{https://www.youtube.com/watch?v=sG4jx\_p9PbA}
\end{minipage}
\end{videobox}

Suppose the five students who walked into the class need to line up at the front desk to hand in their homework. In how many different orders can they line up?

One way to solve this is to reason that there are $5$ possibilities for the first student in line.
Regardless of which student is first, there are $4$ possibilities for the next student.
Then there are $3$ choices for the third student, then $2$ choices for the fourth, and finally $1$ choice for the last in line.
So all together there are $5\cdot 4 \cdot 3 \cdot 2 \cdot 1 = 120$ possibilities.

This leads us to the important definition of a \textit{factorial}. 

\begin{definition}
Define $1!=1$, and if $n >1$ is an integer, 
then define $n!=n\cdot (n-1)!$.
We pronounce the symbol $n!$ as ``$n$ factorial.''
\end{definition}

For instance, $2!=2$, $3!=6$, $4! = 24$ and $5!=120$.
For $n \geq 6$, we can write 
\[n!=n\cdot(n-1)\cdot(n-2) \cdots 3 \cdot 2 \cdot 1.\]

Factorials are a helpful starting point for many similar problems about 
arranging things in order.
For instance, if only three of the five students line up at the front desk, then the number of different possible line-ups is $5\cdot 4 \cdot 3$, which can be expressed as $5!/2!$. 

\begin{example}
Suppose there are $47$ students in M301. If there are $47$ labeled chairs, the number of ways they can be seated is $47!$.
This is because there are $47$ choices for which student is in chair 1, then 46 choices for which student is in chair 2, 
etc, until there are 2 choices for which student is in chair 46, and only 1 choice for which student is in chair 47.
\end{example}

\subsection{Binomial coefficients and counting subsets}
\label{S.1.2.3}

\begin{videobox}
\begin{minipage}{0.1\textwidth}
\href{https://youtu.be/iXtxzRhhEjI}{\includegraphics[width=1cm]{video-clipart-2.png}}
\end{minipage}
\begin{minipage}{0.8\textwidth}
Click on the icon at left or the URL below for this section's short video lecture. \\\vspace{-0.2cm} \\ \url{https://youtu.be/iXtxzRhhEjI}
\end{minipage}
\end{videobox}

\textit{Binomial coefficients} will play an important role throughout this class.  They provide another example of the importance of factorials in counting.

\begin{definition}
Let $n$ be a positive integer.
Let $\{1, \ldots, n\}$
be the set of 
the smallest $n$ positive integers.
The binomial coefficient $\binom{n}{k}$, pronounced ``$n$ choose $k$'', is the number of ways to choose $k$ elements 
from $\{1,2,\ldots,n\}$. 
\end{definition}

\begin{example}
The binomial coefficient 
$\binom{n}{2}$ is the number of ways to 
choose two numbers from the 
set $\{1, \ldots, n\}$.
Then $\binom{n}{2} = \frac{n(n-1)}{2}$, because
there are $n$ choices for the first element and $n-1$ choices for the second, and then we divide by $2$ since the order of the two elements does not matter.
\end{example}

\begin{lemma} \label{Lformulabinom}
The binomial coefficient has this formula:
\[\binom{n}{k}= \frac{n!}{k!(n-k)!}.\]
\end{lemma}

\begin{proof}
\textbf{Method 1:} There are $n$ choices for the first element, then $n-1$ choices for the second, 
and so on, until there are $n-k+1$ choices for the $k$th element. These choices are independent, so the number of ways to select these $k$ numbers is
the product
\[n (n-1) \cdots (n-k+1) = \frac{n!}{(n-k)!}.\]
The same set of $k$ numbers can be chosen in $k!$ different ways.
We divide by $k!$ since the order of choosing the numbers does not matter.

\smallskip

\textbf{Method 2:} There are $n!$ ways to line up the numbers in $\{1,2,\ldots,n\}$.
We then choose the first $k$ elements in line.
We divide by $k!$ because the order does not matter among the chosen numbers, and divide by $(n-k)!$ because the order does not matter among the numbers that are not chosen.
\end{proof}

\begin{example} \label{E47chairs}
Suppose there are $47$ students in Math 301. 
If there are $50$ labeled chairs, what is the number of ways the 47 students can be seated?

\smallskip

\textbf{Answer 1:} $\binom{50}{47} \cdot 47!$.

\smallskip

The reason is that there are $\binom{50}{47}$ ways to choose $47$ chairs to use and, once we have chosen 47 chairs, there are 47! ways to assign the students to chairs. 

\smallskip

\textbf{Answer 2:} $\frac{50!}{3!}$. 

\smallskip

To see this, add three ghost students to the class. There are $50!$ ways to assign the students and ghosts to chairs.
There are $3!$ ways to rearrange the $3$ ghosts without changing the arrangement of the students. We divide by $3!$ since we counted each arrangement of students $3!$ times.

\smallskip

Luckily, the two answers are the same because
\[ \binom{50}{47}\cdot 47! = \frac{50!}{47!\ 3!}\cdot 47! = \frac{50!}{3!}. \]
\end{example}

\subsection*{Exercises}
\label{ss:counting-problems}

\begin{enumerate}

\item What is the sum of the numbers from 10 to 100? That is, what is 10+11+12+...+100?																				
\item Ten people walk into a room and everyone shakes hands with everyone else. How many handshakes occurred?																				
\item Compute the sum 3+6+9+12+...+99.																				
\item How many ways can you rearrange the letters in the word EIGHT?																				
\item In a betting game, several players keep out-bidding each other. The first player puts one chip into the empty pot. The second puts two chips in, then the third puts three chips in, and so on. How many players are needed to continue bidding in this fashion for the pot to contain at least 50 chips?																				

\item How many ways are there to choose $5$ states to visit out of $50$ states?

\item How many ways are there to rank the top $5$ states out of $50$ states that you would be most excited to visit, in order from one through five?

\item How many ways are there to give 4 identical teddy bears to 7 kids, so that no kid receives more than one teddy bear?

\item Verify that $\binom{n}{2}=\frac{n!}{2!(n-2)!}$ is indeed equal to $\frac{n(n-1)}{2}$, and that 
$\binom{n+1}{2}$ equals $T_n$.

\item Let $s_n$ be the sum of the first $n$ odd integers, so $s_1=1$ and $s_2 = 1+3 = 4$. 
\begin{enumerate}
    \item Find $s_3$ and $s_4$. \item Make a conjecture about the value of $s_n$. \item Try to prove your conjecture.
\end{enumerate}

\end{enumerate}

\section{Introduction to graphs}

\begin{videobox}
\begin{minipage}{0.1\textwidth}
\href{https://youtu.be/yHK59HOoaP0}{\includegraphics[width=1cm]{video-clipart-2.png}}
\end{minipage}
\begin{minipage}{0.8\textwidth}
Click on the icon at left or the URL below for this section's short video lecture. \\\vspace{-0.2cm} \\ \href{https://youtu.be/yHK59HOoaP0}{https://youtu.be/yHK59HOoaP0}
\end{minipage}
\end{videobox}

In Section~\ref{S1.1}, the graph of $K_5$ helped us visualize the problem of $5$ students shaking hands.
More generally, in Chapter~\ref{chap:graphtheory}, we will define $K_n$ to be the graph with $n$ vertices, such that every pair of vertices is connected by an edge.  The number of edges in $K_n$ is $\binom{n}{2}$.

Graphs are useful for other problems too.  For example, suppose the $5$ students meet $3$ employers at a job fair.  
Each employer shakes the hand of each student.
How many handshakes are there?
Draw a graph to help you
visualize these handshakes, using a vertex to represent each person and an edge to represent each handshake.
This is an example of a \textit{bipartite graph}; 
this kind of graph is useful when studying connections between two different types of things, like students and employers, or students and classes.

\subsection{Walks on graphs}

Here is the famous
K\"onigsberg bridge problem (now Kalingrad, Russia) --- the drawing below has two islands in the middle of the Pregel river, along with seven bridges connecting the islands and the river banks.\footnote{The last time Professor Pries taught this course, there were two students from Kalingrad in the class.  They were very excited to hear about this example and described how the bridges look to the class.}

\begin{center}
\includegraphics[width=.4\textwidth]{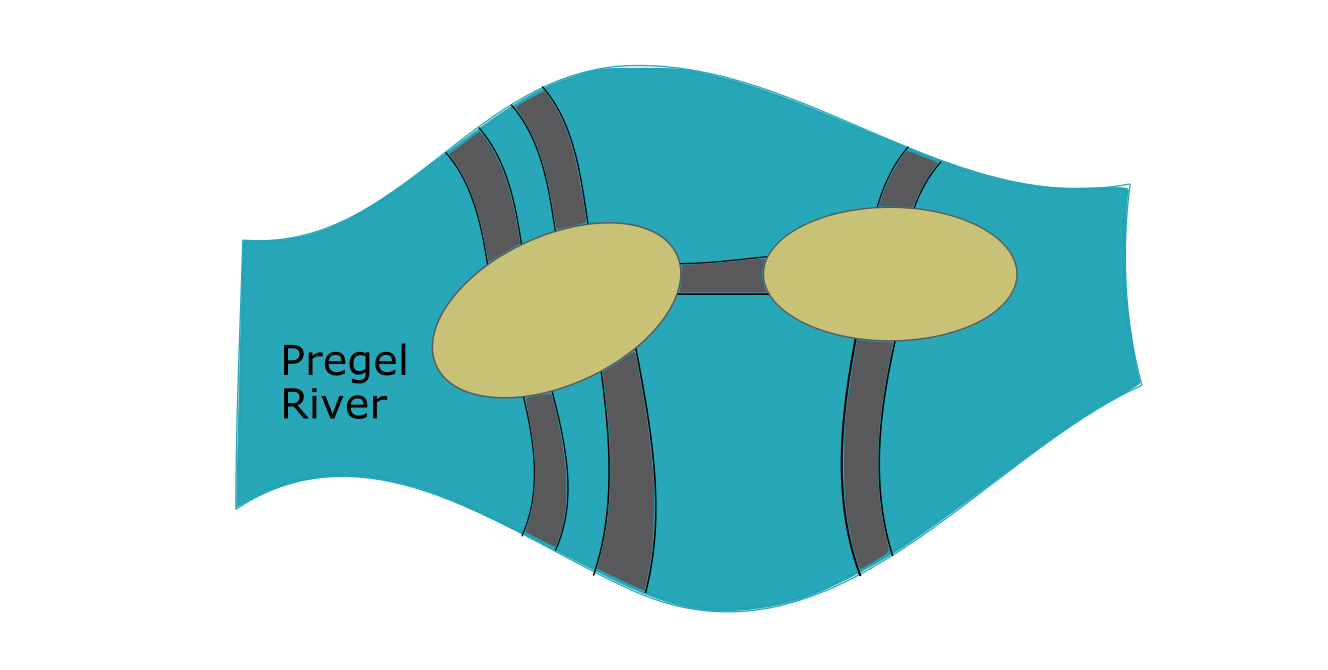}
\end{center}

\begin{question}
\label{ques:Konigsberg}
Can you plan a walk that crosses each bridge exactly once?
\end{question}

\begin{answer}
In 1736, Euler proved it is impossible.
Try to explain to your group why it is impossible.
\end{answer}

\begin{center}
\includegraphics[width=1in]{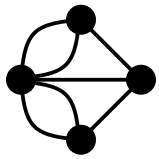}
\end{center}

\begin{question} \label{QKdual}
The graph above is called the graph of K\"onigsberg.
Without lifting your pencil from the page, can you trace your pencil along the edges of the graph above so that each edge is drawn exactly once?
\end{question}

\subsection{When are two graphs the same?}

When drawing a graph, the position of the vertices and the edges does not matter. The only information that is relevant in a graph is which pairs of vertices are connected.  In fact, the following two pictures illustrate the same graph.

\begin{center}
\includegraphics[width=2.5in]{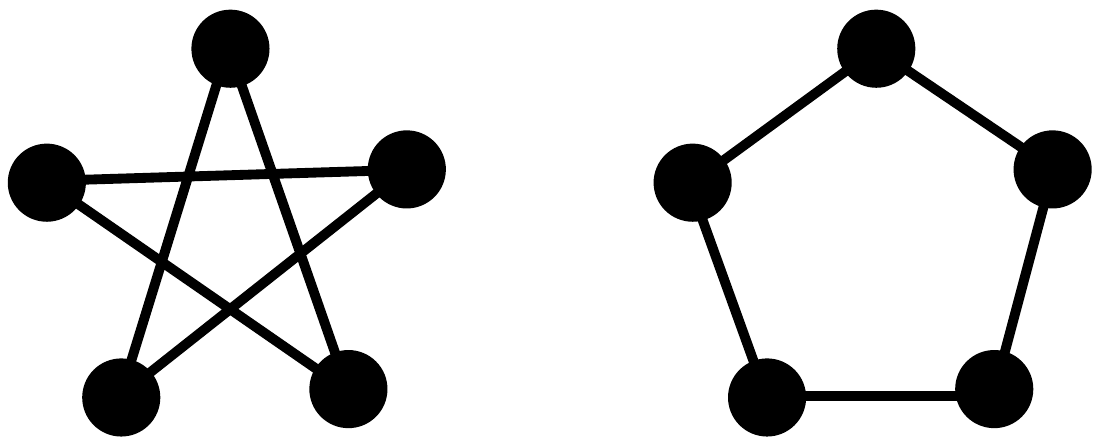}
\end{center}

Here are some ways to tell if two graphs are different:

\begin{enumerate}
  \item The number of vertices for the two graphs is different. 
  \item The number of edges for the two graphs is different.
\end{enumerate} 

\begin{question}\hfill \label{Qgraphdifferent}
\begin{enumerate}
  \item Find two graphs with the same number of vertices and the same number of edges which you are convinced are still different.
  \item Find another way of telling if two graphs are different.
  (There are many possible answers to this question).
\end{enumerate}
\end{question}

\begin{question} \label{Qgraphdifferent2}
Three of the following four graphs are the same, meaning that one of the following four graphs is different from all of the others.
Can you identify any two graphs which are the same?
Can you identify any two graphs that are different?

\begin{center}
    \includegraphics[width=.75\textwidth]{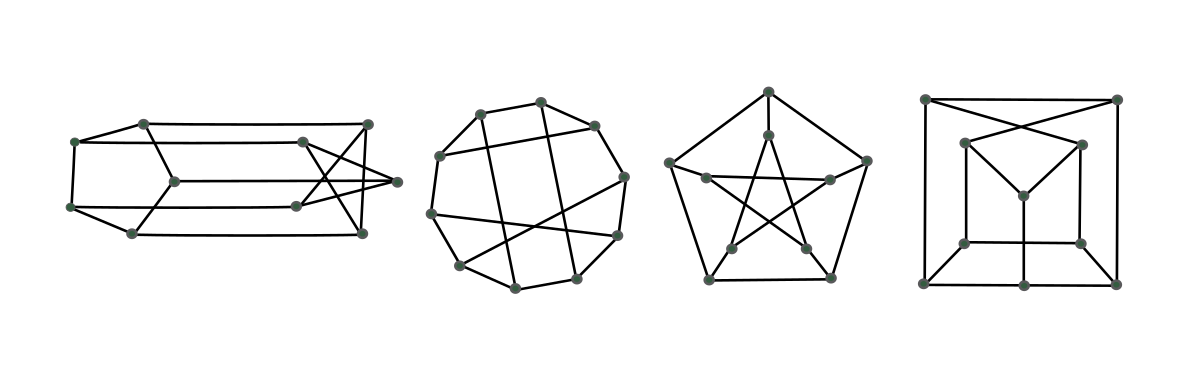}
\end{center}


\end{question}

\subsubsection*{Exercises}

\begin{enumerate}
\item How many vertices does the cube graph have?										
\item How many edges does the cube graph have?										
\item True or False: The graph of the cube is PLANAR, meaning it can be drawn on a piece of paper with no edges crossing.										
\item The graph of the cube is BIPARTITE, meaning its vertices can be colored green and gold so that vertices of the same color are not connected by an edge.										
\item What is the degree of each vertex in the cube graph?										

\item Imagine a graph where the vertices represent apartment buildings and the edges represent streets. 
A snow plow operator needs to plow each street, return home, and does not want to travel 
any more distance than necessary.
\begin{enumerate}
\item Find a way of removing edges
from the graph of K\"onigsberg to make the snow plow operator happy. 

\item Find a way of adding edges to the graph of K\"onigsberg to make the snow plow operator happy. 

\item Does the location of the 
plow operator's house make a difference?

\item In mathematical terms, describe what needs to be true about a graph for the
snow plow operator to be happy.
\end{enumerate}

\item Again, let the vertices of a graph represent apartment buildings and the edges represent streets. The mail carrier needs to deliver mail to each apartment building, return home, and does not want to travel along any more roads than necessary.
\begin{enumerate}
  \item Draw several graphs where the mail carrier is happy and several where the mail carrier is unhappy.
  \item Does the location of the mail carrier's house make a difference?

\item In mathematical terms, describe what needs to be true about a graph for the
mail carrier to be happy.
\end{enumerate}

\item How is the graph of 
K\"onigsberg related to the picture of the K\"onigsberg bridge problem?
How is Question~\ref{QKdual} related to Question~\ref{ques:Konigsberg}?

\item Answer Question~\ref{Qgraphdifferent}.

\item Answer Question~\ref{Qgraphdifferent2}.
\end{enumerate}\

\section{Introduction to SAGE}

Earlier in this chapter, we introduced 
triangular numbers, factorials, and binomial coefficients.  Later in the book, we will introduce other famous numbers like the Fibonacci numbers and Catalan numbers.  It is often nice to have an exact formula for these numbers.  The process of computing them on a calculator can be slow and leads to small errors.  

A faster and more powerful method is to work with computing software.  In this course, we will work with SAGE, a free on-line open-source program, based on python.

To get started with Sage:
go to the website
\begin{verbatim}
http://sagecell.sagemath.org/
\end{verbatim}

Try typing easy commands like \texttt{2+3} or \texttt{factor(2021)} and then evaluating.

For larger jobs or if you want to save your work, you can start a free CoCalc account at:

\begin{verbatim}
http://www.sagemath.org/
\end{verbatim}

The sage reference manual is at:
\begin{verbatim}
http://doc.sagemath.org/html/en/reference/
\end{verbatim}

A more advanced overview of using sage for combinatorics can be found at
\begin{verbatim}
https://doc.sagemath.org/html/en/reference/combinat/sage/combinat/tutorial.html
\end{verbatim}

\begin{example}
Which is bigger, $3^6$ or $6!$?
One method is to compute:
\begin{verbatim}
sage: 3^6;
output 729
sage: factorial(6);
output 720
\end{verbatim}

Another method is to write:
\begin{verbatim}
sage: 3^6 - factorial(6) > 0;
output: True
\end{verbatim}
\end{example}

\begin{example}
For which values of $k$, is the binomial 
coefficient $\binom{11}{k}$ bigger than $100$?
A slow method is to compute 
$\binom{11}{k}$ for each $0 \leq k \leq 11$, and see if the answer is bigger than $100$.
Another method is to compute them all at once:
\begin{verbatim}
sage: [binomial(11, k) for k in range(12)]
output: [1, 11, 55, 165, 330, 462, 462, 330, 165, 55, 11, 1]
\end{verbatim}
Note that the command \texttt{range(m)} includes all integers $0 \leq m \leq m-1$.
From this we see that $\binom{11}{k} > 100$
when $3 \leq k \leq 8$.
Here is another method that gives the values of $k$ as the output.
\begin{verbatim}
sage: [k for k in range(12) if binomial(11,k) >100]
output: [3, 4, 5, 6, 7, 8]
\end{verbatim}
\end{example}

\subsection*{Exercises}

Use Sage to answer the following questions.

\begin{enumerate}
    \item Factor $2022$.
    \item Compute $4^{4^4}$.
    \item Which is bigger, $(4!)!$ or $4^{4^4}$?
    \item Compute the product of the smallest positive fifty odd numbers (starting with $1$).
    \item What is the smallest value of $k$ for which $\binom{13}{k}$ is bigger than $10$?
    \item What is the smallest possible value of $n+k$ for which $\binom{n}{k}$ is bigger than $100$?
\end{enumerate}

%% file: 02-CountingPrinciples/BasicCounting.tex
\chapter{Counting principles}\label{chap:countingprinciples}

In this chapter, we establish some of the basic principles of counting.  In Sections~\ref{Saddsubtract} - \ref{Scombine}, we use the basic arithmetic principles of addition, subtraction, multiplication, and division to count things more efficiently.
In Section~\ref{Ssetcourse}, we begin to 
reinterpret counting problems using set theory and count the number of subsets of a finite set.
In Sections~\ref{Ssetcourse2} - \ref{Sdivequiv}, we use set theory to describe the four basic arithmetic principles and the 
inclusion-exclusion principle.

\section{Motivation} \label{Smotivate2}

Students $a$, $b$, $c$, and $d$\footnote{To protect their privacy, students are identified by their first initial in this book.} sat down in the library to start their combinatorics homework.  They read the first problem that they were assigned.

\begin{question}
\label{Qhw1}
How many ways are there to choose 6 numbers from the set $\{1, \ldots, 20\}$ so that at least four of them are odd?
\end{question}

Student $d$ watched as $a$, $b$, $c$ each worked on the problem for several minutes, using paper, calculators, and tablets. 
Finally, student $d$ said `Hey, I don't know how to get started with this.'

\textbf{Student $a$}: My thought is there are 10 odd numbers between $1$ and $20$.  So there are 10 choices for the first odd number, 9 choices for the second, 8 choices for the third, and 7 choices for the fourth.  The last two numbers can be odd or even and there are 16 numbers left.  So that gives $16$ choices for the fifth and $15$ choices for the sixth.  So the answer is
\[10 \cdot 9 \cdot 8 \cdot 7 \cdot 16 \cdot 15 = 1,209,600.\]

\textbf{Student $b$}: You're forgetting that the order doesn't matter.  So first we need to choose four numbers out of the 10 odd numbers, then we need to choose $2$ out of the remaining $16$ numbers.  So the answer is
\[\binom{10}{4} \cdot \binom{16}{2} = \frac{10!}{4!6!} \cdot \frac{16!}{2!14!} = 25,200.\]

\textbf{Student $c$}: 
I have a different way of doing it.  There are $\binom{20}{6}$ ways to choose six integers out of the 20, but the probability that a number is odd is $1/2$.  So the answer is 
\[\binom{20}{6}\cdot\frac{1}{2^4} = \frac{20!}{6!14!}\cdot\frac{1}{2^4} =
2422.5.\]
Wait, that's not an integer.... 

\textbf{Student $d$}: What if there are more than four odd numbers?
Forget this, let's try the next question.

\begin{question}
\label{Qhw2}
In a game of poker, you are dealt five cards from a standard $52$ card deck.
How many ways can you be dealt a full house?
\end{question}

\textbf{Student $c$}: I've never played cards.
What is a full house?\footnote{If you are not familiar with cards or poker, there is a lot of helpful information in  Section~\ref{poker}.}

\textbf{Student $a$}: It's a triple of the same number together with a pair of a different number. Like 6 of hearts, 6 of clubs, and 6 of diamonds, together with 3 of hearts and 3 of spades.  So I'm thinking for the first card, we have $52$ choices; no matter what number it is, to get a triple we
need to choose two more out of the three remaining cards of that number, and there are three ways to do that.
Then there are $48$ choices for the next card; no matter what number it is, to get a pair we need to choose one more out of the three remaining cards of that number, and there are three ways to do that.
So the answer is:
\[52 \cdot 3 \cdot 48 \cdot 3 =
22464.\]

\textbf{Student $b$}:
That's over counting again because it doesn't matter which card is dealt first. Instead, let's choose two numbers out of $13$, then choose three suits for the first number and then choose two suits for the second number.
So the answer is:
\[\binom{13}{2}\cdot \binom{4}{3}
\cdot \binom{4}{2}
=\frac{13 \cdot 12}{2} \cdot 4 \cdot 6 =1872.\]

\textbf{Student $d$}: On the web, it says that the probability of a full house is $0.001441$ and there are $\binom{52}{5} =2,598,960$ possible hands.  If we multiply those together, then we get $3745.10136000000$, so the answer is probably $3745$.

\textbf{Student $c$}: Maybe we need to read more of the book.

\subsection*{Exercises}

\begin{enumerate}
    \item Figure out something that student $a$ said that was correct and something that was wrong.
    \item Figure out something that student $b$ said that was correct and something that was wrong.
    \item Figure out something that student $c$ said that was correct and something that was wrong.
    \item Figure out something that student $d$ said that was correct and something that was wrong.
    \item Do you think the students are close to solving Question~\ref{Qhw1}?  What advice would you give them?
    \item Do you think the students are close to solving Question~\ref{Qhw2}?  What advice would you give them?
\end{enumerate}

\section{The addition and subtraction principles} \label{Saddsubtract}

\begin{videobox}
\begin{minipage}{0.1\textwidth}
\href{https://www.youtube.com/watch?v=apKGi7NJeGM}{\includegraphics[width=1cm]{video-clipart-2.png}}
\end{minipage}
\begin{minipage}{0.8\textwidth}
Click on the icon at left or the URL below for this section's short video lecture. \\\vspace{-0.2cm} \\ \href{https://www.youtube.com/watch?v=apKGi7NJeGM}{https://www.youtube.com/watch?v=apKGi7NJeGM}
\end{minipage}
\end{videobox}

The first way we learn to count is ``by $1$'s,'' that is, add $1$ to a running total for every item in a list. 
If the number of things we are trying to count is finite, it is always possible to make a list of them and count them one by one.
However, this can take a very long time and it is easy to forget some things that should be on the list, or to write some things twice.
It speeds things up to use a computer to generate the list, but even that can be slow if the list is very long.

One strategy is to split up the items into shorter lists. We can count the number on each shorter list and then add those numbers together.
For example, student E has climbed $5$ tall mountains and $8$ short mountains so has climbed $5+8=13$ mountains in total.

The more general principle is this:
\begin{quote}
	\textbf{Addition Principle:}  The sum $a+b$ counts the total number of things in a collection formed by adding a collection of $b$ things to a collection of $a$ things.
\end{quote}

\begin{example}
Suppose we wish to count the number of two-digit positive integers that end in either $0$ or $1$.  We can first count those that end in $0$: there are $9$ of them, one for each possible tens digit.  There are also, similarly, $9$ of them that end in $1$.  So there are $9+9=18$ such numbers in total.
\end{example}

Suppose now that we have 90 pieces of fruit (specifically apples, bananas, and oranges) in your fruit-storing refrigerator.  If we wanted to count how many apples and bananas we had, we could add up the number of apples and the number of bananas.  Let's say there were a total of 87 apples and bananas - this is a lot of fruit to count (and carry!).  Is there an easier way to do this?
\begin{quote}
	\textbf{Subtraction Principle:}  If $b$ things are removed from a collection of $a$ things, then there are $a-b$ things left.
\end{quote}
We can easily spot that there are 3 oranges in our refrigerator while the rest of the fruit are apples and bananas.  Thus, we have $90-3=87$ apples and bananas.

\begin{example}
 Suppose we wish to count the two-digit positive integers that end in either $0,1,2,3,4,5,6,7,$ or $8$.  We could do it by adding $9$ to itself $9$ times, but we could alternatively use the subtraction principle by counting a larger set and subtracting the ones we don't want!

In particular, there are $90$ two-digit positive integers, namely the numbers from $10$ up through $99$.  Exactly $9$ of them end in the digit $9$, and so we subtract these from the total of $90$ to get $90-9=81$ of them that don't end in a $9$.
\end{example}

\subsection*{Exercises}

\begin{enumerate}
    \item Identify when students $a$, $b$, $c$, and $d$ used (or should have used) the addition and subtraction principles in Section~\ref{Smotivate2}.
    \item Identify a place in Chapter~\ref{chap:intro} where the addition (resp.\ subtraction) principle was used.
    \item Suppose you need to choose one more class to complete your schedule for next semester at CSU. You need to choose exactly one course from the following list: math, science, and music theory. The math class is offered at six different times, the science course is offered at two different times and the music theory class is offered at three different times throughout the week. If none of the class times create a conflict with your current schedule, how many options do you have for choosing a class? 
    \item How many two-digit numbers are not divisible by 5?
    
    \item How many ways are there to choose four numbers from 
    $1, \ldots, 100$, such that 
    at most one is even?
    
    \item How many numbers from $1$ to $3^6$ are not divisible by $3$?
    
    \item Suppose you want to make a bouquet of flowers native to Colorado.  You have 53 flowers and 3 types: poppymallow, harebell, and sunflowers.  If you have 22 poppymallows and 17 harebells, how many total flowers would be in a bouquet of all the sunflowers and poppymallows?
\end{enumerate}

\section{The multiplication and division principles}
\label{Smultdivide}

There are two other principles, called the multiplication and division principles.  
  Section 2.2's video includes the multiplication principle, but the division principle is covered in the following video.

\begin{videobox}
\begin{minipage}{0.1\textwidth}
\href{https://www.youtube.com/watch?v=Q0gk4yBuVj0}{\includegraphics[width=1cm]{video-clipart-2.png}}
\end{minipage}
\begin{minipage}{0.8\textwidth}
Click on the icon at left or the URL below for this section's short video lecture. \\\vspace{-0.2cm} \\ \href{https://www.youtube.com/watch?v=Q0gk4yBuVj0}{https://www.youtube.com/watch?v=Q0gk4yBuVj0}
\end{minipage}
\end{videobox}

Here is an example of the multiplication principle at work.

\begin{example} \label{Koutfit}
Student $k$ owns $15$ shirts and $6$ pairs of pants.
How many possible outfits (consisting of one shirt and one pair of pants) are possible?
Well, to form an outfit, student $k$ first chooses a shirt (in $15$ ways) and then a pair of pants (in $6$ ways).
So there are $15\cdot 6=90$ possible outfits.
\end{example}

The more general principle is this:

\begin{quote}
	\textbf{Multiplication Principle:}  The product $a\cdot b$ counts the number of ways to choose one thing in $a$ ways and then another in $b$ ways.
\end{quote}

\begin{example}
Suppose we wish to count the number of two-digit positive integers that end in either $0, 4,$ or $8$.  There are $9$ ways to choose the first digit, and $3$ possibilities for the second digit, for a total of $9\cdot 3 =27$ possibilities.
\end{example}

Suppose now we want to sell our 90 pieces of fruit at the local farmer's market.  We can fit 10 pieces of fruit in each fruit basket. Then we need $90/10=9$ baskets to carry all the fruit.
  
\begin{quote}
    \textbf{Division Principle 1:}  If $a$ things are grouped evenly into collections of size $b$, there are $a/b$ collections.
\end{quote}

Let's consider another scenario. 

\begin{example}
Suppose 36 CSU students are traveling to Denver for a combinatorics conference. If each university van seats 4 people, $36/4=9$ vans are needed to transport the students. 
\end{example}

\begin{example}
How many lotto tickets are there, if a lotto ticket is a collection of $5$ different numbers from $\{1, \ldots, 90\}$, where order does not matter. 

We already know one way to solve this problem!  The number of lotto tickets is the number of ways to choose $5$ numbers out of $90$.  This equals
\[\binom{90}{5}=\frac{90!}{5!\cdot85!}=\frac{90\cdot 89\cdot 88\cdot 87\cdot 86}{5!} =43,949,268.\] 

We can do this problem a different way using the division principle to gain a better understanding why we divide by $5!$.
There are $90\cdot 89\cdot 88\cdot 87\cdot 86$ ways to pick $5$ different numbers from $1$ to $90$ one at a time. 
We put two lotto tickets in the same pile if the numbers on them are the same, just written in a different order.  
Then each pile has 
$5!=120$ lotto tickets in it.
Since the order of the numbers on a lotto ticket does not matter, we have counted each ticket $5!$ times.
Hence the total number of lotto tickets is $\frac{90\cdot 89\cdot 88\cdot 87\cdot 86}{5!}$.
\end{example}

An alternative formulation of the division principle is as follows:
\begin{quote}
    \textbf{Division Principle 2:} If $a$ things are grouped evenly into $c$ collections of equal size, the number of things in each collection is $a/c$.
\end{quote}

For example, if we divide $90$ pieces of fruit evenly into $9$ baskets, then there are $10$ pieces of fruit in each basket.  Similarly, if $36$ CSU students driving to Denver divide themselves evenly into $9$ cars, then there are $4$ people in each car. 
We end this section with a more challenging example that can be solved with either the division or multiplication principle.

\begin{example}
How many different ways can $10$ people be grouped into $5$ pairs of two?

One approach uses the multiplication principle: Find a predetermined way of ordering the people, such as age). 
The youngest person has $9$ possible partners, then the youngest remaining has $7$ possible partners, then the youngest remaining has $5$ possible partners, then the youngest remaining has $3$ possible partners, and then the last two people will be partners.
Hence the total number of ways is \[9 \cdot 7\cdot 5\cdot 3 \cdot 1=945.\]

Another solution uses the division principle: There are $10!$ ways to organize the people in a line.  We pair the first and second people, the third and fourth people, etc.
There are $5!$ ways to rearrange the ordering of the pairs, which does not affect who is paired with whom.
Also, there are $2$ ways to switch the order 
of the people in each of the five pairs, which also does not affect who is paired with whom.
So the total number of pairs is 
\[\frac{10!}{5! \cdot 2^5} =945.\]
\end{example}

\subsection*{Exercises}
\begin{enumerate}

\item How many three-digit positive integers are even?

\item You have 6 distinct marbles and 2 identical buckets.  How many different ways can you put 3 marbles in each bucket?																									
\item In how many ways can you attach the letters of MATH in some order to a piece of string? (If you turn the string over so that the order of the letters is reversed, it is not considered a different possibility. So for instance, MATH and HTAM are considered the same ordering.)																									
\item There are 10 people in a dance class.  How many ways can they pair off into 5 pairs of dance partners?  (Gender does not matter in making the pairings.)																									
\item There are 10 people in a dance class. In how many ways can the instructor choose 4 people and put them into two pairs to demonstrate a dance move?																									
\item Explain why the sum of all the degrees of the vertices of any graph is even. (Hint: How many times is each edge counted?)																									

\item There are 30 granola bars.  Decide which sentence uses division principle 1 and which uses division principle 2:
\begin{itemize}
    \item If there are 5 granola bars in each box, then the number of boxes is 6. 
\item If there are 6 boxes of granola bars, then each box contains 5 granola bars.
\end{itemize}

\item How many ways are there to group 18 students into $9$ teams of $2$?

\item How many ways are there to group $18$ students into $6$ teams of $3$?

\item Identify when students $a$, $b$, $c$, and $d$ used (or should have used) the multiplication and division principles in Section~\ref{Smotivate2}.
    \item Identify a place in Chapter~\ref{chap:intro} when the multiplication (resp.\ division) principle was used.

\end{enumerate}

\section{Combining the principles}
\label{Scombine}

One difficulty in solving counting problems is that there are often several different ways to approach the problems.  Small errors in the set-up can make an answer completely wrong. 
In this section, we review the four principles and explain how to distinguish between them and combine them.

We summarize the four basic principles in the following table.
Take care that you're using the right principle at each step!

\begin{center}
\begin{tabular}{@{}>{\raggedright}p{0.15\textwidth}>{\raggedright}p{0.8\textwidth}@{}}
	
	\toprule
	
	\textbf{Expression} & \textbf{Combinatorial meaning}
	\tabularnewline[\doublerulesep]
	
	\midrule
	$a+b$ & The total size of a collection formed by adding $b$ things to $a$ things. 
	\tabularnewline
	$a-b$ & The number of elements left after removing $b$ things from  $a$ things. 
	\tabularnewline
	$a\cdot b$ & The number of ways to choose one thing from $a$ things and then one from $b$ things.
	\tabularnewline
	$a/b$ & If $a$ things can be sorted into collections of size $b$, then $a/b$ is the number of collections.
	\tabularnewline

	\bottomrule
\end{tabular}
\end{center}

\begin{remark}
Here is another helpful tip for distinguishing addition and multiplication when counting.  Addition is used when making one choice \emph{or} another. Multiplication is used when making one choice \emph{then} another (when the choices are taken in succession and do not depend on the earlier choices).
\end{remark}

\begin{example} Suppose there are 3 car models and 4 bike models. In how many ways can you buy one vehicle? The answer is $3+4=7$. In how many ways can you buy one car and one bike? The answer is $3\cdot4=12$.
\end{example}

Here is an example of a problem that requires combining the principles together.

\begin{example}
  Bob has 5 brown shirts and 4 blue shirts, and 3 brown pants and 4 blue pants.  He wants to pick out an outfit in which his shirt and pants have the same color.  How many ways can he do this?
  
  We can split the problem into two cases: either he'll wear a brown outfit \textbf{or} a blue outfit, and we will \textbf{add} the results together by the addition principle.  
  
  For the brown case, he needs to pick one of $5$ brown shirts \textbf{and then} one of $3$ brown pants, so we \textbf{multiply} these to get $5\cdot 3=15$ possibilities.
  
  For the blue case, he needs to pick one of $4$ blue shirts \textbf{and then} one of $4$ blue pants, so we \textbf{multiply} these to get $4\cdot 4=16$ possibilities.
  
  Adding these together, we get $15+16=31$ outfits in all.
\end{example}

\subsection*{Exercises}
For the following exercises, make sure you identify each time you use one of the four principles.

\begin{enumerate}

\item Five people and five dogs meet in a park. The people all shake hands with each other and the people all shake hands with the dogs, but the dogs do not shake each other's hands. How many handshakes took place?															
\item The five people and five dogs meet in a park. How many ways can the people pair off with the dogs to go for a walk?															
\item How many three-digit numbers have two even digits and one odd digit?															
\item How many rearrangements of the word MATH start with a consonant (M, T, or H)?

\item How many ways are there to group 100 students into $50$ pairs of $2$?


\item How many odd three-digit numbers are there that don't end with 9?

\item A drama teacher has 30 student actors and actresses that make up a cast for the school play.  Each student has one role, where 10 students are in the chorus, 15 students are in supporting roles, and the rest of the students are lead characters. There are 450 total costumes for the cast.  If each cast member wears the same number of costumes and all costumes are worn, how many total costumes are worn by the lead characters?

\item A donut shop has 100 total donuts, where there are 38 vanilla frosted, 21 blueberry, 26 chocolate, and the rest are jelly-filled.  A very hungry customer wants all of your jelly-filled and blueberry donuts.  How many boxes do you need if each box holds six donuts? Note: it does not matter what donuts go in what box.  

\item You order dinner for a friend at a concession stand during a CSU soccer game.  You notice both burgers and burritos are being sold. To order a burger you must choose one option from each category:
\begin{itemize}
    \item Bun: wheat or sesame seed
    \item Patty: beef, veggie, or turkey
    \item Sauce: ketchup, mustard, relish, mayonnaise, none
\end{itemize}
To order a burrito, you must choose one option from each category:
\begin{itemize}
    \item Tortilla: white, wheat, or corn
    \item Filling: shrimp, carnitas, carne asada, or fajita veggies
    \item Cheese: cheddar, pepper jack, none
    \item Rice: brown, jasmine
\end{itemize}
How many possible dinners could you order for your friend if you buy them either a burger or a burrito (but not both)?
  
\end{enumerate}

\section{Sets, subsets, and the number of subsets} \label{Ssetcourse}
\begin{videobox}
\begin{minipage}{0.1\textwidth}
\href{https://www.youtube.com/watch?v=vwSn9UnHa1I}{\includegraphics[width=1cm]{video-clipart-2.png}}
\end{minipage}
\begin{minipage}{0.8\textwidth}
Click on the icon at left or the URL below for this section's short video lecture. \\\vspace{-0.2cm} \\ \href{https://www.youtube.com/watch?v=vwSn9UnHa1I}{https://www.youtube.com/watch?v=vwSn9UnHa1I}
\end{minipage}
\end{videobox}

\textit{Sets} are unordered collections, like a handful of marbles or a bag of fruit.  Many counting problems can best be phrased in terms of sets and subsets.  In Sections~\ref{Ssetcourse2} and \ref{Scoursemult}, we will 
rephrase the four principles in terms of sets.

\subsection{Basic definition of sets}

\begin{definition}[Informal]

A \defn{set} is a collection of distinct objects (in no particular order).  The objects in a set are called \defn{elements}.  We write the objects in a set inside squiggly brackets: $\{\}$.\footnote{Later in the book, we will consider \emph{multi-sets}, where the objects in the set can be the same, for 
example $\{a,a,b\}$.}
\end{definition}

For example, $\{a,b,c\}$ and $\{1,2,3,4\}$ are both sets.  The set $\{a,b,c\}$ is the same set as $\{b,c,a\}$ because the order of the elements does not matter.

\begin{notation}
We write $b\in A$ to mean that $b$ \emph{is an element} of the set $A$.  Similarly, $b\not\in A$ if $b$ is not an element of $A$.
\end{notation}

\begin{example}
For example, 
$4\in \{2,4,5\}$ but
$3\notin \{2,4,5\}$.
\end{example}

It is sometimes helpful to draw a set in ``blob notation'', consisting of a circle with points inside labeled by the elements of the set.  For instance, the figure below depicts the set $\{2,4,5\}$.

\begin{center}
     \includegraphics[width=.15\textwidth]{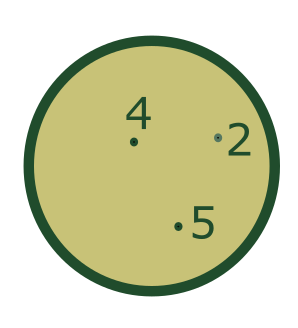}
\end{center}

\begin{definition}
If $A$ is a finite set, then
its \defn{size} (or cardinality)
(written $|A|$) is the number of elements in $A$.
\end{definition}

For example, $|\{2,4,5\}|=3$.  
The size of a set is found by counting the number of elements in the set.  For this reason, counting problems can often be phrased in terms of computing the cardinality of a set.

There is also a notion of cardinality for infinite sets, but we will not cover this topic in this document.  There are even different sizes of infinite sets!

\subsection{Set builder notation}

Writing down sets of large or infinite size can be hard.  One possible notation for the infinite set of positive integers is $\{1,2,3,4,\ldots\}$, but this is a bit ambiguous.
How do we know whether a number like $7$ is included? 
To describe larger or more complicated sets, without laboriously writing out every element, we use 
\textit{set builder notation}, as in the definition below.

\begin{definition}
 The notation $$\{x\in A \mid x\text{ has property }P\}$$ represents the set of all elements $x$ in the set $A$ that satisfy property $P$.  Sometimes we simply write $\{x \mid x\text{ has property } P\}$ if the set $A$ containing $x$ is clear.
\end{definition}

\begin{example}
 $\{x\in \{2,4,5\} \mid  x\mbox{ is even}\}=\{2,4\}$.
\end{example}

\begin{example}
Let $V$ be the set of all integers whose last digit is zero.  
Then $30 \in V$ but $31 \not \in V$.
Another description of $V$ is
\[V=\{10 \cdot k \mid k \mbox{ is an integer} \}.\]
\end{example}

\subsection{Famous sets}

It is helpful to fix notation for a few important infinite sets that are used frequently in mathematics:

\begin{definition}
$\emptyset=\{\ \}$ is the empty set, the unique set containing no elements.

$\N=\{0,1,2,3,4,\ldots\}$ is the set of nonnegative integers (or natural numbers).\footnote{In this book, we include $0 \in \N$, but many people exclude $0$ from being a natural number.}

$\Z=\{\ldots,-2,-1,0,1,2,\ldots\}$ is the set of integers.

$\Q=\{a/b \mid a\in\Z\mbox{ and }0\neq b\in\Z\}$ is the set of
fractions (rational numbers).

$\R$ is the set of real numbers.
\end{definition}

\begin{example}
The number $1.73 \in \Q$ is a good approximation for $\sqrt{3}$.
In Section~\ref{Sproofcont}, we will prove that $\sqrt{3}\notin\Q$. 
\end{example}

\begin{example}
In the 1760s, Lambert 
proved that the real number $\pi$ is not in $\Q$.
\end{example}

\begin{example}
There are several ways to describe the set of all even natural numbers: 
\[\{x\in \mathbb{N} \mid x \text{ is even}\}=\{0,2,4,6,8,\ldots\} = \{2 \cdot k \mid k \in \N\}.\]
\end{example}

\subsection{Subsets}

Informally, a \textit{subset} is a set contained in a bigger set.

\begin{definition}
A set $A$ is a \defn{subset} of a set $B$ (written $A\subseteq B$) if every element of $A$ is also an element of $B$.
\end{definition}

For instance, $\{2,4\}\subseteq \{2,3,4\}$, but $\{2,4,5\}\not\subseteq \{2,3,4\}$.

\begin{example}
Let $A$ be any set.
Then $A \subseteq A$ because every element of $A$ is an element of $A$.
Also $\emptyset \subseteq A$  because there are no elements of $\emptyset$.
\end{example}

\begin{example} For instance, all natural numbers are integers and all integers are rational numbers. All rational numbers are real numbers. We can write this as 
a \emph{nested sequence of inclusions}  \[\emptyset\subseteq\N\subseteq\Z\subseteq\Q\subseteq\R.\]
\end{example}

\subsection{The number of subsets}

We now address a famous counting question involving sets, which is to count the number of subsets of a set.  Here is an example of a concrete application of this problem.

\begin{example} \label{E4studentsub}
Four students say that they might go to office hours on Monday.  How many different groups might show up in office hours? 
\end{example}
The above problem can be modeled by counting the possible subsets of the set $\{a,b,c,d\}$ of four students.

To answer this question in general, we need to count the number of subsets in a set $A$ of size $n$.  Before we dive into the theorem, let's look at a few cases:  
\begin{itemize}
\setlength\itemsep{0pt}
    \item $n=1$: The subsets of $A=\{a\}$ are $\emptyset$ and $\{a\}$, so a set with $1$ element has $2$ subsets.
    \item $n=2$: The subsets of $A=\{a,b\}$ are $\emptyset$, $\{a\}$, $\{b\}$, and $\{a,b\}$, so a set with $2$ element has $4$ subsets.
    \item $n=3$: The subsets of $A=\{a,b,c\}$ are $\emptyset$, $\{a\}$, $\{b\}$, $\{c\}$, $\{a,b\}$, $\{a,c\}$, $\{b,c\}$, and $\{a,b,c\}$, so a set with $3$ element has $8$ subsets.
\end{itemize}

In each case, our answer was a power of $2$, with $2^1=2$, $2^2=4$, and $2^3=8$.
For a set $\{a,b,c,d\}$ with four elements, like in Example~\ref{E4studentsub}, we might guess that the number of subsets is $2^4=16$.

The following theorem proves that this pattern holds for any set of finite size.

\begin{theorem}\label{Tsubset}
If $S$ is a set with $n$ elements,
then $S$ has $2^n$ subsets.
\end{theorem}

\begin{proof}[Proof of Theorem~\ref{Tsubset}]
Let's label the elements of $S$ as $a_1,a_2,a_3,\ldots,a_n$, so that $S=\{a_1,a_2,a_3,\ldots,a_n\}$. Picking a subset of $S$ is the same as
\begin{itemize}

\setlength\itemsep{0pt}
\item First, choosing if $a_1$ is in the subset or not (2 choices),
\item Next, choosing if $a_2$ is in the subset or not (2 choices),

\hspace{5mm}\vdots

\item Finally, choosing if $a_n$ is in the subset or not (2 choices).
\end{itemize}
At each step $i$, there are 2 choices (include $a_i$ or not). Since we make a decision $n$ times, the total number of subsets is $\underbrace{2\cdot 2\cdot\ldots\cdot2}_{n\mbox{ times}}=2^n$ by the multiplication principle.
\end{proof}

\begin{remark}
Here is a way to visualize this proof when $n=3$.
Let's say $S=\{a_1,a_2,a_3\}$. 
\begin{center}
\begin{tikzpicture}[
    node/.style={%
      draw,
      rectangle,
    },
  ]
    \node [node] (A) {$a_1$ in subset?};
    \path (A) ++(-170:\nodeDist) node [node] (B) {$a_2$ in subset?};
    \path (A) ++(-10:\nodeDist) node [node] (C) {$a_2$ in subset?};
    \path (B) ++(-120:\nodeDist) node [node] (D) {$a_3$ in subset?};
    \path (B) ++(-60:\nodeDist) node [node] (E) {$a_3$ in subset?};
    \path (C) ++(-120:\nodeDist) node [node] (F) {$a_3$ in subset?};
    \path (C) ++(-60:\nodeDist) node [node] (G) {$a_3$ in subset?};
    \path (D) ++(-105:\nodeDist) node [node] (H) {$\emptyset$};
    \path (D) ++(-75:\nodeDist) node [node] (I) {$\{a_3\}$};
    \path (E) ++(-105:\nodeDist) node [node] (J) {$\{a_2\}$};
    \path (E) ++(-75:\nodeDist) node [node] (K) {$\{a_2,a_3\}$};
    \path (F) ++(-105:\nodeDist) node [node] (L) {$\{a_1\}$};
    \path (F) ++(-75:\nodeDist) node [node] (M) {$\{a_1,a_3\}$};
    \path (G) ++(-105:\nodeDist) node [node] (N) {$\{a_1,a_2\}$};
    \path (G) ++(-75:\nodeDist) node [node] (O) {$\{a_1,a_2,a_3\}$};

    \draw (A) -- (B) node [pos=0.5,above] {no}(A);
    \draw (A) -- (C) node [above,pos=0.5] {yes}(A);
    \draw (B) -- (D) node [left,pos=0.25] {no}(A);
    \draw (B) -- (E) node [right,pos=0.25] {yes}(A);
    \draw (C) -- (F) node [left,pos=0.25] {no}(A);
    \draw (C) -- (G) node [right,pos=0.25] {yes}(A);
    \draw (D) -- (H) node [left,pos=0.25] {no}(A);
    \draw (D) -- (I) node [right,pos=0.25] {yes}(A);
    \draw (E) -- (J) node [left,pos=0.25] {no}(A);
    \draw (E) -- (K) node [right,pos=0.25] {yes}(A);
    \draw (F) -- (L) node [left,pos=0.25] {no}(A);
    \draw (F) -- (M) node [right,pos=0.25] {yes}(A);
    \draw (G) -- (N) node [left,pos=0.25] {no}(A);
    \draw (G) -- (O) node [right,pos=0.25] {yes}(A);
\end{tikzpicture}
\end{center}
\end{remark}

\begin{remark}
Theorem~\ref{Tsubset} even works when $n=0$, since the empty set $\emptyset$ is the unique set with $0$ elements, and there is $1=2^0$ subset of $\emptyset$, namely $\emptyset$.
\end{remark}

\begin{remark}
There is a more abstract way to think about Theorem~\ref{Tsubset}.
We can make another set $T$, whose elements are the subsets of $S$, which are sets themselves.
Then $2^n$ is the cardinality of $T$.
\end{remark}

\begin{example}
\label{Qevensub3}
Students $a,b,c$ say that they might show up to office hours.
Among the $8$ possible outcomes, 
how many of them have an even number of students coming to office hours.
The answer is $4$, namely
\[\emptyset, \{a,b\}, \{a,c\}, 
\{b,c\}.\]
\end{example}

\subsection*{Exercises}

\begin{enumerate}
\item Given the following sets, describe in words what the set contains. (There will be more than one correct description)
\begin{enumerate}
    \item $A=\{2n\mid n\in \mathbb{Z}\}$
    \item $B=\{x \mid 0\leq x<2\}$
    \item $C=\{(x,y) \mid y=x^2\}$
    \item $D=\{ y=mx+b \mid m=3 \}$
\end{enumerate}
\item Given the following mathematical descriptions, write a set in set builder notation that matches the description. (There will be more than one correct answer)
\begin{enumerate}
    \item The set containing all natural numbers greater than or equal to ten.
    \item The set containing all integers divisible by 3.
    \item The set containing nothing.
    \item The set of all perfect squares.
\end{enumerate}

\item Let $S=\{1, 
\ldots, 40\}$.
List the elements in
$\{x \in S \mid x \ \text{is a power of a prime}\}.$

\item You have 20 different books and you want to donate a subset of them to the library. The only requirements are that you want to donate at least 1 book and you do not want to donate all 20 books. How many different donations can you make?

\item This problem is about the subsets of $\{a,b,c,d\}$.  First, list all the subsets.
\begin{enumerate}
    \item How many of them have odd order?
    \item How many of them contain $\{a\}$?
    \item How many have size $0$, size $1$, size $2$, size $3$, and size $4$?
    \end{enumerate}
    
    \item 
    This problem is about the subsets of $\{a,b,c,d,e\}$.
    Without listing these subsets, make conjectures about the following problems.
    \begin{enumerate}
        \item What proportion of them have even order?
        \item What proportion of them contain $a$?
        \item How many have size $0$, $1$, $2$, $3$, $4$, and $5$?
    \end{enumerate}

    \item What is the output in SAGE when you type S=Subsets([1,2,3,4,5,6]) in line 1 and S.cardinality() in line 2?																
    \item If A is a set of size 4 and B is a set of size 7, what is the smallest possible size of the union of A and B?																
    \item You have 6 different books and you want to donate a subset of them to the library. The only requirements are that you want to donate at least 1 book and you do not want to donate all 6 books. How many different donations can you make?																
    \item How many subsets of $\{1,2,3,4,5\}$ contain the number 1?																
  \item If each student a, b, c, d, and e receives either 0 or 1 point for the quiz, how many sequences of five numbers are possible as the list of their grades?																
\end{enumerate}

\section{Addition and subtraction from the perspective of set theory} \label{Ssetcourse2}

\subsection{Union, intersection, and the addition principle}

We now can define the set operations that allow us to phrase the addition principle in terms of sets.

\begin{definition}
The \defn{union} of sets $A$ and $B$ is 
\[ A\cup B=\{x~:~x\in A\mbox{ or }x\in B\}. \]
\end{definition}

Here ``or" means $x\in A$ or $x\in B$ or both. (This is a different meaning of the word ``or" than in ``soup or salad" or ``paper or plastic".)

\begin{example}
We see that $\{2,4,5\}\cup\{2,5,7,9\}=\{2,4,5,7,9\}$.  The union consists of all the elements in the region of the diagram covered by the overlapping circles.
\end{example}

\begin{center}
    \includegraphics[width=.26\textwidth]{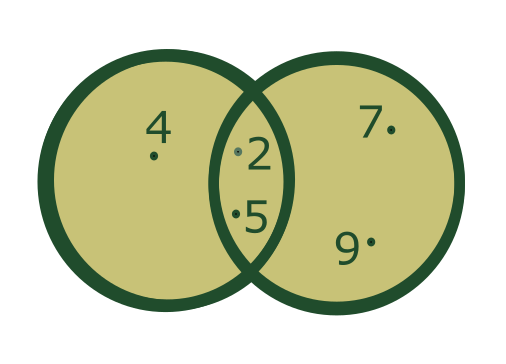}
\end{center}

\begin{definition}
The \defn{intersection} of sets $A$ and $B$ is 
\[ A\cap B=\{x~:~x\in A\mbox{ and }x\in B\}. \]
\end{definition}

\begin{example}
We have $\{2,4,5\}\cap\{2,5,7,9\}=\{2,5\}$.  In the diagram, it is the common region shared by the overlapping circles.
\end{example}

\begin{center}
    \includegraphics[width=.3\textwidth]{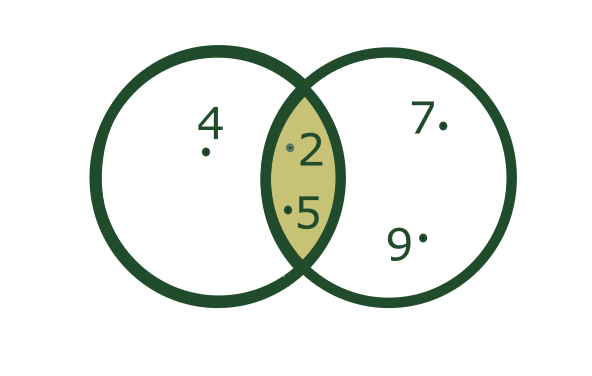}
\end{center}

Here is the set analogue of the addition principle.

\begin{lemma} [The addition principle for sets]
If $A$ and $B$ are finite sets with no elements in common (that is, $A \cap B=\emptyset$), then the union $A \cup B$ has size $|A|+|B|$.
\end{lemma}

\begin{example}
A student has climbed the set of big mountains $A=\{\mbox{Longs peak}, \mbox{ Pikes peak}, \mbox {Mt.\ Evans}\}$ of size $|A|=3$, and the set of small mountains 
$B=\{\mbox{Arthur's rock}, \mbox{ Horsetooth}\}$ of size $|B|=2$.
Since there is no overlap between the sets of large and small mountains, E has climbed $|A|+|B|=5$ mountains.
\end{example}

\subsection{Subsets, complements, and the subtraction principle}

We now introduce the set operations needed to understand the subtraction principle in terms of sets. 
One example of subsets can be constructed with the \textit{set difference} operator as follows.

\begin{definition}
The \defn{difference} of sets $A$ and $B$ is 
\[A - B = \{x \in A \mid x \not \in B\}.\]
\end{definition}

\begin{example}
We see that $\{2,4,5\} - \{2,5,7,9\}=\{4\}$, and $\{2,5,7,9\} - \{2,4,5\}=\{7,9\}$.
\end{example}

\begin{definition}
\label{def:complement}
If $B\subseteq A$, then the \defn{complement} of $B$ in $A$ is the
difference $A - B$. 
If the set $A$ is clear from context, then this is written as 
$B^c$, where the $c$ stands for complement.
\end{definition}

\begin{center}
    \includegraphics[width=.2\textwidth]{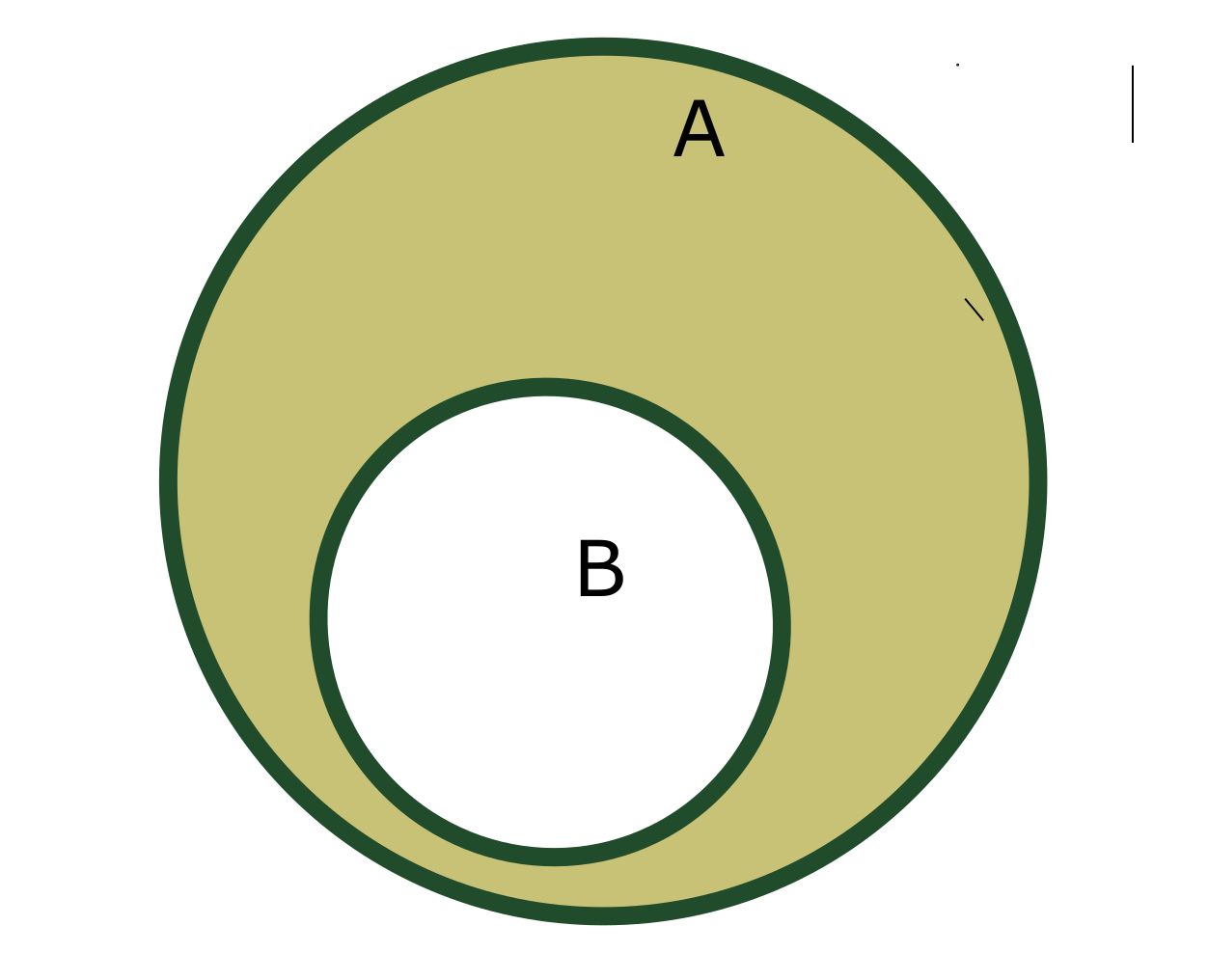}
\end{center}

We can now state the set analogue of the subtraction principle.

\begin{lemma}
[The subtraction principle for sets]
If $A$ and $B$ are finite sets and  $B \subseteq A$, then 
$|A - B| = |A|-|B|$.
\end{lemma}

\begin{example}
How many numbers in $A=\{1, \ldots, 25\}$ are not a multiple of $5$?

To answer this, let $B=\{5,10, 15, 20, 25\}$.  Then $B \subseteq A$ and $B$ is exactly the subset of elements of $A$ that are a multiple of $5$.
So $|A - B|=25-5=20$.
\end{example}

\subsection*{Exercises}
\begin{enumerate}
\item Let $A$ and $B$ be sets with sizes $|A|=4$ and $|B|=7$.
\begin{itemize}
\item[(a)] What are all the possible values of $|A\cap B|$?
\item[(b)] What are all the possible values of $|A\cup B|$?
\item[(c)] What are all the possible values of $|B -  A|$?
\end{itemize}

 \item See Question \ref{Qhw1}.  Let $A$ be the set of subsets of size $6$ in $\{1, \ldots, 20\}$ which contain at least $4$ odd numbers. 
    \begin{enumerate}
        \item Write down several elements of $A$. 
    \item Describe $A$ as the union of three smaller subsets.  
    \item Use the addition principle for sets to find $|A|$.
    \end{enumerate}

\item What is $A - B$ if $A\cap B=\emptyset$?
\item Find particular sets $A$, $B$, and $C$ so that
$A-(B-C) \not = (A-B) -C$.

\item Let $A=\{1,2,3,4,5\}$ and $B=\{4,5,6,7\}$ be subsets of $W=\{x \mid 1\leq x \leq 10,~x\in \mathbb{Z}\}$.  Compute the following:
\begin{enumerate}
    \item $B^c$
    \item $A^c$
    \item $A\cup B$
    \item $(A\cup B)^c$
    \item $A^c\cup B^c$
    \item $A \cap B$
    \item $(A \cap B)^c$
    \item $A^c \cap B^c$
\end{enumerate}
Are any of the sets equal? If so, identify them.
\end{enumerate}

\section{Venn diagrams and the Principle of Inclusion-Exclusion}\label{sec:pie}

\begin{videobox}
\begin{minipage}{0.1\textwidth}
\href{https://www.youtube.com/watch?v=yhTBE8L3pro}{\includegraphics[width=1cm]{video-clipart-2.png}}
\end{minipage}
\begin{minipage}{0.8\textwidth}
Click on the icon at left or the URL below for this section's short video lecture. \\\vspace{-0.2cm} \\ \href{https://www.youtube.com/watch?v=yhTBE8L3pro}{https://www.youtube.com/watch?v=yhTBE8L3pro}
\end{minipage}
\end{videobox}

In this section, we cover 
the \textit{Principle of Inclusion-Exclusion}, which is a way to count the number of elements in a union of sets when the sets intersect non-trivially.
Drawing sets as overlapping circles to form a \textit{Venn diagram} is a helpful way to visualize the intersection of several sets.

\begin{example}
\label{Evenn1}
Students $a$, $b$, $c$, $d$, $e$, $f$, and $g$ decide to walk to the local ice cream shop in Old Town after their combinatorics class.  As they walk, student $d$ takes a quick poll on who likes mint ice cream and who likes caramel ice cream. Students $a$ and $f$ love both mint and caramel. It would be a shame to choose just one! So they vote for both.
Tallying the votes, there are 5 votes for mint ice cream and 4 votes for caramel, for a total of 9 votes.  Student $d$ quickly realizes that something is wrong and figures out the following Venn diagram.
\end{example}
\begin{center}
    \includegraphics[width=.3\textwidth]{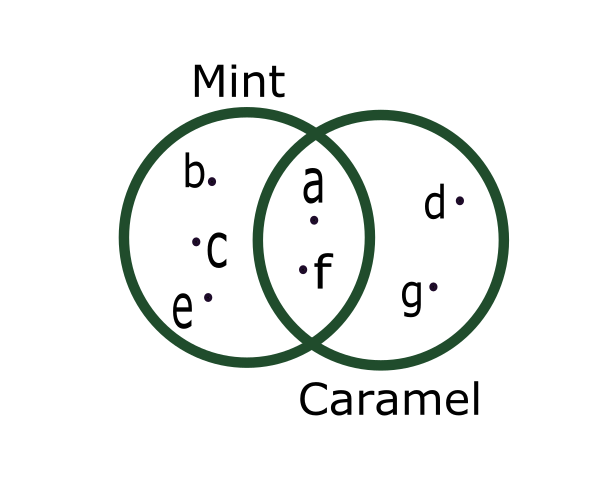}
\end{center}

Here is another example of a problem that can be solved with a Venn diagram.

\begin{example}\label{ex:210}
How many numbers in the set $S=\{1, \ldots, 210\}$ have no factor in common with 21?

A number has a factor in common with $21 =3 \cdot 7$ if it is a multiple of $3$ or a multiple of $7$ or both.
Let $A$ be the set of multiples of $3$ in $S$ and let $B$ be the set of multiples of $7$ in $S$. 
  The intersection $A \cap B$ is the set of multiples of $21$.

We want to count the elements that are not in $A$ or $B$.
This is the complement of $A \cup B$ in $S$.  
The number of these is 
\[|S - (A \cup B)| = 210 - |A \cup B|.\]

There are $70$ multiples of $3$ in $S$ so $|A|=70$.
There are $30$ multiples of $7$ in $S$ so $|B|=30$.  But if we subtract both $70$ and $30$ from $210$, we've removed each multiple of $21$ twice, since it is divisible by both $3$ and $7$.  There are $10$ multiples of $21$ that we removed twice, so we ad $10$ to the count. 
So the answer is
$210-70-30+10=120$.
\end{example}

We state the Principle of Inclusion-Exclusion for two sets and then for three sets.

\begin{theorem}[Principle of Inclusion-Exclusion for two sets]\label{thm:pie}
Suppose $A$ and $B$ are two finite sets.
Then 
\[|A \cup B| = |A| + |B| - |A \cap B|.\]
\end{theorem}
\begin{center}
    \includegraphics[width=\textwidth]{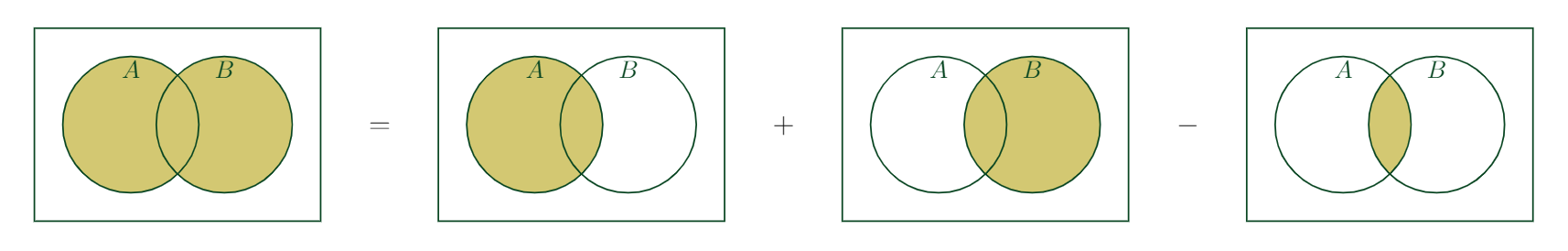}
\end{center}

\begin{remark}
In Example \ref{ex:210}, $A$ and $B$ are subsets of $S=\{1, \ldots, 210\}$; 
$A$ is the set of multiples of $3$; $B$ is the set of multiples of $7$.
We calculated 
\[
|A \cup B| = |A| + |B| - |A\cap B| = 70 + 30 -10 = 90.
\]
The set of all numbers in $S$ that are not divisible by $3$ or $7$ is the complement of $A \cup B$ in $S$, so it has size $|S - (A \cup B)|=210-90 = 120$.
\end{remark}

There is also a version of the inclusion-exclusion principle for three sets.

\begin{theorem} \label{Tvenn3}
Suppose $A,B,C$ are three finite sets.
Then
\[|A \cup B \cup C|=|A|+|B|+|C|-|A\cap B| - |B \cap C| - |A \cap C| + |A \cap B \cap C|.\]
\end{theorem}

Stare at the following pictures for a while to understand why Theorem~\ref{Tvenn3} is true!

\begin{center}
    \includegraphics[width=\textwidth]{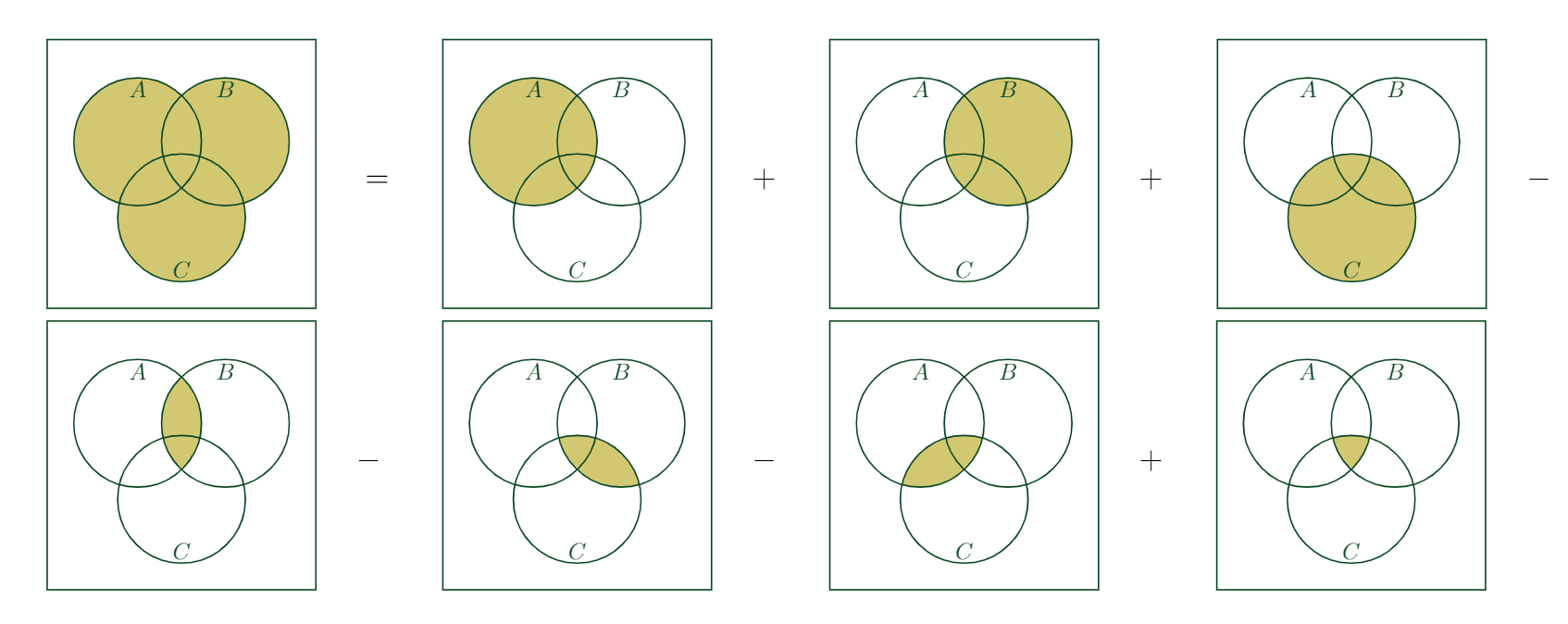}
\end{center}

\begin{example}
Continuing Example~\ref{Evenn1}, two of the students $a$, $b$, $c$, $d$, $e$, $f$, and $g$ also like strawberry ice-cream.
There is $1$ student who likes mint and strawberry but not caramel.
There are no students who like strawberry and caramel but not mint.
Show that there is exactly $1$ student who likes all three flavors.

To solve this, we fill in the numbers in each part of the Venn diagram.  We have the additional information that there is no student who likes only strawberry, because all $7$ students like either mint or caramel or both.
Among the $2$ students who like strawberry, $1$ also likes mint but not caramel and $0$ like caramel but not mint.  This shows there is only one student who likes all three flavors.
\end{example}

\begin{example}\label{ex:210new}
How many numbers in the set $S=\{1, \ldots, 210\}$ have no factor in common with 30?

A number has a factor in common with $30 =2 \cdot 3 \cdot 5$ if it is a multiple of $2$ or $3$ or $5$ or any combination of these.
Let $A$ be the set of multiples of $2$ in $S$.
Let $B$ be the set of multiples of $3$ in $S$. 
Let $C$ be the set of multiples of $5$ in $S$.
Then 
\[|A| = 105,\  |B|=70,\ {\rm and} \  |C|=42.\]
It is useful that the numbers $2$ and $3$ and $5$ are all relatively prime.
So $A \cap B$ is the set of multiples of $6$ in $S$.
Also $A \cap C$ is the set of multiples of $10$ in $S$.
Also $B \cap C$ is the set of multiples of $15$ in $S$.
So 
\[|A \cap B| = 35, \ |A \cap C| = 21, \ {\rm and} \ |B \cap C|= 14.\]
Finally, $A \cap B \cap C$ is the set of multiples of $30$ in $S$ so it has size $7$.
Using Theorem~\ref{Tvenn3}, we see that
\[|A \cup B \cup C| = 105+70+42 - 35 - 14 - 21 + 7 =154.\]
We want to count the size of the complement of $A \cup B \cup C$ in $S$ since the complement contains the numbers that have no factor in common with $30$.  
It has size $210-154=56$.
\end{example}

\subsection*{Exercises}
\begin{enumerate}
\item What is $|A\cup B|$ if $|A|=10$, $|B|=7$, and $|A\cap B|=3$?

\item Suppose $|A|=10$ and $|B|=7$.  What is the largest and smallest possible size for $|A \cup B|$?

\item Suppose $A$ and $B$ are sets such that $|A - B|=17$, $|A\cap B|=5$, and $|A\cup B|=40$. What is $|B|$?

\item Suppose $A$, $B$, and $C$ are sets such that $|A|=10$, $|B|=17$, $|A\cap B|=5$, and $|A\cup B\cup C|=30$. What is the largest and smallest possible size of $|C|$?

\item Students $a$, $b$, and $c$ compare their CSU Fall class schedules.
Student $a$ is registered for combinatorics, art history, cooking, and physics.
Student $b$ is registered for combinatorics, health, astronomy, band, and dance.
Student $c$ is registered for combinatorics, art history, dance, English literature, and African American studies.
Use the principle of inclusion-exclusion to compute how many classes they are taking in total.

\item How many numbers in $\{1, \ldots, 35\}$ are relatively prime to $35$ (i.e., have no factors in common with $35$)?
\item How many integers in $\{1, 2, ...., 55\}$ are not divisible by either 11 or 5?																			
\item How many integers in $\{1, 2, ...., 105\}$ are not divisible by 3 or 5 or 7?																									
\item Suppose 50 socks lie in a drawer. Each one is either white or black, ankle-high or knee-high, and either has a hole or doesn’t. 22 socks are white, four of these have a hole, and one of these four is knee-high. Ten white socks are knee-high, ten black socks are knee-high, and five knee-high socks have a hole. Exactly three ankle-high socks have a hole. Find the number of black, ankle-high socks, with no holes.																									
\end{enumerate}

\section{The multiplication principle from the perspective of set theory}
\label{Scoursemult}

We now introduce the set theory notation needed to describe the multiplication principle in terms of sets.

\subsection{Cartesian products and the multiplication principle}

If $A$ and $B$ are sets, then their \defn{Cartesian product}, denoted $A \times B$, is the set consisting of all ordered pairs of elements from $A$ and $B$:
\[ A\times B = \{(a,b)~|~a\in A \text{ and } b\in B\}. \]

\begin{example}
The product $\{2,3,4\}\times \{x,y\}$ is the set of ordered pairs
\[\{(2,x),(2,y),(3,x),(3,y),(4,x),(4,y)\}.\]
\end{example}

\begin{example}
The product $\{2,3,4\}\times \{2,5\}$ is the set of ordered pairs
\[\{(2,2),(2,5),(3,2),(3,5),(4,2),(4,5)\}.\]
\end{example}

\begin{example}
  The product $\mathbb{R}\times \mathbb{R}$ can be thought of as all the points $(x,y)$ in the usual coordinate plane $\mathbb R^2$ that we use when graphing functions in calculus.  This plane is often called the Cartesian plane, which is where the term ``Cartesian product'' comes from.  
\end{example}

We can now state the set analogue of the multiplication principle.

\begin{lemma}
[Multiplication principle for sets]
If set $A$ has $|A|$ elements and set $B$ has $|B|$ elements then set $A \times B$ has $|A| \cdot |B|$ elements.
\end{lemma}

\begin{example}
We arrange three boxes of granola bars in a line.  In each box, five granola bars are stacked up.  We index each granola bar by a pair 
$(i,j)$ where $i \in A=\{1,2,3\}$ indexes the box and $j \in B=\{1, \ldots, 5\}$
indexes the spot in the box.
Then the number of granola bars is $|A \times B| = 3 \cdot 5 =15$.
\end{example}

\begin{example}
Let $A$ be the set of numbers in $\{1, \ldots, 40\}$ which are odd and not a multiple of $5$.  Suppose we wish to find the size of $A$.

The numbers in $A$ are exactly those whose last digit is in $C=\{1,3,7,9\}$ and whose first digit is in 
$B = \{0,1,2,3\}$. So $A$ can be identified with $B \times C$.\footnote{More precisely, there is a bijection from $A$ to $B \times C$ where the number $10c + d$ maps to the pair $(c,d)$.  The topic of bijections is covered in Section~\ref{sec:bij}.}
So $|A|=|B| \cdot |C| = 4 \cdot 4 = 16$.
\end{example}

\subsubsection*{Exercises}

\begin{enumerate}
\item If Student $a$ owns 5 shirts and 2 pairs of pants, how many outfits consisting of one shirt and one pair of pants can they wear?  Justify your answer using sets.

  \item See Question~\ref{Qhw2}.
    Let $A$ be the set of poker hands which are a full house (five cards including a triple of one number and a pair of another number).
    \begin{enumerate}
    \item Write down one example of a full house.
    \item With a classmate (or friend), play 20 questions until they guess which full house you wrote down.  For example, one question could be: do the cards in your triple have the number $7$ on them?
    \item Describe $A$ as the Cartesian product of several smaller sets.
    Hint: each smaller set tells you some information, but not all information about the full house.  For example, set $B$ could index the number on the cards in the triple.
    \item Use the multiplication principle for sets to find $|A|$.
    \end{enumerate}
\item Let $W=\{(1,1), (2,1), (2,2), (3,1), (1,2), (3,2)\}$.  Find sets $A$ and $B$ such that $A\times B=W$.  verify the multiplication principle with your sets.
\item Express $\mathbb{R}\times\mathbb{R}$ as a set.  What does this represent?
\item Express $[0,1]\times[0,1]$ as a set.  What does this represent?
\end{enumerate}

\section{Set partitions, the division principle, and equivalence relations}
\label{Sdivequiv}

\begin{videobox}
\begin{minipage}{0.1\textwidth}
\href{https://www.youtube.com/watch?v=iqP5fyUL1MU}{\includegraphics[width=1cm]{video-clipart-2.png}}
\end{minipage}
\begin{minipage}{0.8\textwidth}
Click on the icon at left or the URL below for this section's short video lecture. \\\vspace{-0.2cm} \\ \href{https://www.youtube.com/watch?v=iqP5fyUL1MU}{https://www.youtube.com/watch?v=iqP5fyUL1MU}
\end{minipage}
\end{videobox}

\subsection{Set partitions and the division principle}
\label{Ssetdivide}

The division principle is the most difficult to state in terms of set theory.
To do this, we define 
a \textit{set partition}, which is a a helpful way to
organize sets.

\begin{definition}
\label{Dsetpartition}
Let $A$ be a finite set.
A non-empty subset $B$ of $A$ is called a \textbf{block}.
A \textbf{set partition} of a finite set $A$ is a set of blocks $\{B_1,B_2,\ldots,B_k\}$ of $A$ such that:
\begin{enumerate}
    \item the blocks are disjoint, that is, $B_i\cap B_j=\emptyset$ for all $i\neq j$; and
    \item the union of the blocks is $A$, that is $B_1\cup B_2 \cup \cdots \cup B_k=A$.
    \end{enumerate}
\end{definition}

\begin{example}
The following is a set partition of $\{1,2,3,4,5,6,7\}$:  $$\{\{1,3,4\},\{2,6\},\{5,7\}\}.$$
\end{example}

The division principle arises when we split a set into blocks of equal sizes.  It can be stated as follows.

\begin{lemma}
[Division principle for sets]
\label{Ldivsetprin}
Suppose that $\{B_1,\ldots,B_k\}$ is a set partition of $A$ such that each of the blocks $B_1,\ldots,B_k$ has the same cardinality $m$.
Then the number of blocks and the size of the blocks satisfy these relationships:
$|A|/m=k$, and $|A|/k=m$.
\end{lemma}

\begin{example}
How many ways can the letters in the word ``COOL'' be rearranged? If we pretend the two O's are different from each other, 
calling them $O_1$ and $O_2$, then there are $4!=4\cdot 3\cdot 2 \cdot 1=24$ ways to order them, as discussed in section \ref{sec:factorials-intro}.  We can partition these arrangements into blocks of size $2$ that have the letters in the same position except for the two O's switched with each other.
\begin{center}
$\{\{$CL{\rO}{\gO},\,\,CL{\gO}{\rO}$\},\,\, \{$C\rO L\gO,\,\,C\gO L\rO$\},\,\,\{$C\rO\gO L,\,\,C\gO\rO L$\},\,\,\{$LC\rO \gO,\,\,LC\gO \rO$\},$ \\
\phantom{\{}$\{$L\rO C\gO,\,\,L\gO C\rO $\},\,\,\{$L\rO \gO C,\,\,L\gO\rO C$\},\,\,\{$\rO CL\gO,\,\,\gO CL\rO$\},\,\,\{$\rO C\gO L,\,\,\gO C\rO L$\},$ \\
\phantom{\{}$\{${\rO}LC{\gO},\,\,\gO LC{\rO}$\},\,\,\{$\rO L\gO C,\,\,\gO L \rO C$\},\,\,\{$\rO\gO CL,\,\,\gO\rO CL$\},\,\,\{$\rO \gO LC,\,\,\gO\rO LC$\}\}$
\end{center}
Thus there are $24/2=12$ blocks, or $12$ distinct rearrangements of the letters in COOL.  See Section \ref{sec:anagrams} for more examples of counting anagrams.
\end{example}

\subsection{Equivalence relations and modular arithmetic}\label{sec:mod}

Equivalence relations are often used to 
divide a set into blocks.
In fact, every equivalence relation on a set makes a set partition.

\begin{definition}
A \emph{equivalence relation} on a set $S$ is a way of identifying elements $x \sim y$
with the three following rules:
\begin{enumerate}
    \item (reflexive) $x \sim x$;
    \item (symmetric) if $x \sim y$, then $y \sim x$;
    \item (transitive) 
    if $x \sim y$ and $y \sim z$, then 
    $x \sim z$.
\end{enumerate}
\end{definition}

\begin{example}
How many ways can students $A,B,C,D$ be seated at a rotating round table?
There are $24=4!$ ways for the students to sit down at the table.  We say that two seating arrangements are equivalent if they look the same after a rotation.
For example, 
\[ABCD \sim DABC \sim CDAB \sim BCDA.\]
Each seating arrangement is equivalent to three others.  So we can divide the set of seating arrangements into blocks, each of which has size $4$.  By the division principle, there are $6$ blocks.
\end{example}

One important equivalence relation is called congruence modulo $m$.

\begin{definition}
Let $m$ be a positive integer.
Let $x,by$ be integers.
We say that $x$ is \defn{congruent to $y$ modulo $m$} when $m$ divides $y-x$.
This is written $x \equiv y \bmod m$.
\end{definition}

\begin{lemma}
If $r$ is the remainder when $x$ is divided by $m$, then $x \equiv r \bmod m$.  In particular, if $m$ divides $x$, then $x \equiv 0 \bmod m$.
\end{lemma}

\begin{proof}
If $r$ is the remainder when $a$ is divided by $m$, then $a = km +r$ for some integer $k$.
This means that $a-r = km$, which is divisible by $m$.  So $a \equiv r \bmod m$ by definition.
If $m$ divides $a=a-0$, then $a \equiv 0 \bmod m$ by definition.
\end{proof}

\begin{example}
The hours of a day are described with congruence modulo $12$.  For example, if it is 11 o'clock and you need to wait 2 hours before eating lunch, then you will eat lunch at 1 o'clock.  We write $11 + 2 = 13 \equiv 1 \bmod 12$.
\end{example}

\begin{example}
Every even integer is congruent to $0$ modulo $2$.
Every odd integer is congruent to $1$ modulo $2$.
\end{example}

\begin{example}
The last digit of a number indicates its congruence modulo $10$.
\end{example}

We will now give some examples of the division principle using blocks that are constructed using congruence modulo $m$.

\begin{example}
How many numbers in the set $S=\{1, \ldots, 40\}$ are congruent to $1$ modulo $5$?

We list the numbers between $1$ and $40$ that are congruent to $1$ modulo $5$: 
\[\{1,6,11,16,21,26,31,36\}.\]
So the answer is $8$.
In this example,
there is a set partition of $S$ into $5$ blocks of size $8$, where two numbers are in the same block exactly when they are 
congruent modulo $5$.
\end{example}

This example works more generally.

\begin{lemma} \label{Lblockcong}
Suppose $N$ is a multiple of $m$.
Let $S$ be the set $S=\{1, \ldots, N\}$.
Then $S$ has a set partition into 
$m$ blocks, each of size $N/m$, where
two numbers are in the same block exactly when they are 
congruent modulo $m$.
\end{lemma}

\subsection*{Exercises}

\begin{enumerate}
\item Let $A=\{a,b,c,d\}$.  List all possible set partitions.

\item How many ways can the letters in ``COLORADO'' be rearranged.

\item How many ways can five students be seated at a rotating round table?

\item How many numbers in $S=\{1, \ldots, 40\}$ are congruent to $3$ modulo $8$? What are they?

\item Let $S=35$ and $m=5$.  Find the set partition of $S = \{1, \ldots, 35\}$ from Lemma~\ref{Lblockcong}.
Repeat this question with $m=7$.

\item Show that congruence is an equivalence relation.

\item
If $n$ is an integer, show $n^2 \equiv 0 \bmod 4$ if $n$ is even and $n^2 \equiv 1 \bmod 4$ if $n$ is odd.

\item Explain why Lemma~\ref{Lblockcong} is true using Lemma~\ref{Ldivsetprin}.
\end{enumerate}

\section{Additional Problems for Chapter 2}
\label{Sadditional2}

\begin{enumerate}

\item Use the multiplication and division principles to find how many handshakes there are between seven people, if every pair of person shakes hands. 

\item Generalize your reasoning in the previous problem to give a proof using the multiplication and division principles that the number of handshakes between $n$ people, if everyone shakes hands, is $n(n-1)/2$.

\item Bob has $4$ pairs of socks, $2$ pairs of shoes, $5$ pairs of pants, $10$ shirts, and one hat.  How many ways can he get dressed for the day (wearing one of each type of item of clothing)?  

\item How many three-digit positive integers have only even digits?

\item Suppose $A$ and $B$ are sets such that $|A|=21$, $|B|=15$, and $|A\cup B|=30$. What is $|A\cap B|$?

\item Using a Venn diagram, for any three sets $A,B,C$, explain why it is always true that 
$$A\cap(B\cup C)=(A\cap B)\cup(A\cap C).$$

\item Using a Venn diagram, for any three sets $A,B,C$, explain why it is always true that $$A\cup(B\cap C)=(A\cup B)\cap(A\cup C).$$

\item How many numbers in $\{1, \ldots, 105\}$ are relatively prime to $105$?

\item Let $n=pq$ where $p$ and $q$ are distinct primes.  Show that there are 
$pq-p-q+1$ numbers in $\{1, \ldots, n\}$
which are relatively prime to $n$.

\item 
\begin{enumerate}
\item How many binary strings of length $m$ are there?
\item If $s$ is a string of length $m$, let $s^c$ be the complement string where each $0$ is replaced by a $1$ and each $1$ is replaced by a zero.  Write out the binary strings of length $m=3$ and match each with its complement. 
\item For each 
binary string $s=s_1s_2s_3$, plot the point $(s_1,s_2,s_3)$ in $3$-dimensional space.  What 3-dimensional shape has these points as its corners and what is the relationship between the corners plotted for $s$ and its complement $s^c$?
\end{enumerate}

\item \label{Exercisedefcomplement}
  Let $S$ be the set of binary strings of length $m$.
  \begin{enumerate}
      \item When $m=2$, here are all the subsets of $S$ of size $k=2$: 
     \begin{equation}
 \label{Estringsub} \{00,01\},\{00,10\},\{00,11\},\{01,10\},\{01,11\},\{10,11\}.
      \end{equation}
      Write out all the subsets of $S$ of size $k=3$. 
 \item  If $T$ is a subset of $S$ containing $k$ strings, let $T^c$ be the subset of $S$ containing their $k$ complements.
 When $m=2$ and $k=2$, 
 match up each set in \eqref{Estringsub} with its complement.
 When $m=2$ and $k=3$, 
 match up each set for your answer to part (1) with its complement.
 
 \item Suppose that $k$ is odd.  If $T$ is a subset of $S$ of size $k$, prove that $T$ and $T^c$ are not the same.
 \item
 Suppose $k$ is even.  If $R$ is a subset of $S$ of size $k/2$ and if $T=R \cup R^c$,  show that $T$ is a subset of $S$ of size $k$ and that $T = T^c$.

 \item How many sequences of 5 letters can be made with the letters A,E,I,O,U,M,S, allowing repeats of letters, for example, MMSSU?										
\item How many sequences of 5 letters can be made with the letters A,E,I,O,U,M,S, and have exactly 2 vowels, for example, MMSOU?										
\item How many ways are there to assign 5 jobs to 4 people, so that each person gets at least one job (and every job is assigned to some person)?										
\item How many ways can the letters of PUPPY be rearranged (including the original spelling)?										
\item In the set $\{1, 2, ....., 24\}$, there are 3 integers which are 2 mod 8. What is their sum?										
 \end{enumerate}

\end{enumerate}

\newpage 
\section{Investigation: Divisors of a positive integer}

In this section, we study \emph{divisor functions}, which were investigated by the famous mathematician Ramanujan (1887-1920).
Let $N$ be a positive integer.
A divisor function is a function defined on positive integers whose value at $N$ depends on the divisors of $N$.

\subsection{The number of divisors} \label{Ssigma0}

\begin{definition}
Let $\sigma_0(N)$ be the number of positive divisors of $N$, including $1$ and $N$ itself.
\end{definition}

\begin{example}
\begin{itemize}
    \item 
$\sigma_0(22) = 4$ because the positive divisors of $22$ are $1,2,11,22$; 
\item $\sigma_0(23) = 2$ because the positive divisors of $23$ are $1$ and $23$; and \item $\sigma_0(24)=8$ because the positive divisors of
$24$ are $1,2,3,4,6,8,12,24$.
\end{itemize}
\end{example}

In the following questions, we will develop a formula for $\sigma_0(N)$ by working out cases of increasing difficulty.

\begin{question}
\begin{enumerate}
\item
Compute $\sigma_0(N)$ for $N=1, 
\ldots, 20$.  Look for some patterns.
    \item 
When is $\sigma_0(N)=1$?  Explain why.
\item When is $\sigma_0(N)=2$?  Explain why.
\item When is $\sigma_0(N)$ odd?  Explain why.
\end{enumerate}
\end{question}

\begin{question} 
\begin{enumerate}
    \item Compute $\sigma_0(N)$
    for $N=3,9,27,81$.
    \item If $N=p^e$ for some prime $p$, find a formula 
    for $\sigma_0(N)$.
    \item Explain why the formula is true.
    \item Give an example to show that the formula does not work when $p$ is not prime.
\end{enumerate}
\end{question}

\begin{question}
\label{Qtwoprimes}
\begin{enumerate}
    \item Compute $\sigma_0(N)$ for $N=21, 26, 33, 35$.
    \item If $N=p \cdot q$ where $p$ and $q$ are distinct primes, find a formula for $\sigma_0(N)$.
    \item Explain why the formula is true.
    \item 
    Let $A$ be the set of multiples of $p$ dividing $N$ and $B$ be the set of multiples of $q$ dividing $N$.
    Draw a Venn diagram for the divisors of $N$.
\end{enumerate}
\end{question}

\begin{question}
\begin{enumerate}
    \item Compute $\sigma_0(N)$ for $N=30, 42, 70$.
    \item If $N=p\cdot q \cdot r$ where $p,q,r$ are distinct primes, find a formula for $\sigma_0(N)$.
    \item Explain why the formula is true.
    \item 
    Let $A$ and $B$ be as in Question~\ref{Qtwoprimes}.
    Let $C$ be the set of multiples of $r$ dividing $N$. 
    Draw a Venn diagram for the divisors of $N$.
\end{enumerate}
\end{question}

\begin{question}
Suppose $N=p_1 \cdot p_2 \cdots p_n$ where $p_1, \ldots, p_n$ are distinct primes.  Find a formula for $\sigma_0(N)$.
Explain why the formula is true using Theorem~\ref{Tsubset}.
\end{question}

\begin{question}
Suppose $N=p_1^{e_1} \cdot p_2^{e_2} \cdots p_n^{e_n}$ where $p_1, \ldots, p_n$ are distinct primes. Find a formula for $\sigma_0(N)$.
Explain why the formula is true.
\end{question}

\subsection{The sum of the divisors}

\begin{definition}
Let $\sigma_1(N)$ be the sum of the positive divisors of $N$, including $1$ and $N$ itself.
\end{definition}

\begin{example}
\begin{itemize}
    \item 
$\sigma_1(22) = 36$ because  $1+2+11+22=36$; 
\item $\sigma_1(23) = 24$ because $1+23=24$; and 
\item $\sigma_1(24)=60$ because $1+2+3+4+6+8+12+24=60$.
\end{itemize}
\end{example}

\begin{question}
Let $p$ be a prime.
\begin{enumerate}
    \item Find a formula for $\sigma_1(p)$. 
    \item Find a formula for 
    $\sigma_1(p^e)$.  Explain why it is true.
\end{enumerate}
\end{question}

\begin{question}
Let $p_1, \ldots, p_n$ be distinct primes.
\begin{enumerate}
    \item Find a formula for 
    $\sigma_1(p_1 \cdot p_2)$.
    \item Find a formula for 
    $\sigma_1(p_1 \cdot p_2 \cdots p_n)$. Explain why the formula is true.
\end{enumerate}
\end{question}

\begin{question}
Let $p_1, \ldots, p_n$ be distinct primes.
Find a formula for 
$\sigma_1(p_1^{e_1} \cdot p_2^{e_2} \cdots p_n^{e_n})$.
Explain why the formula is true.
\end{question}

\subsection{The M\"obius function}
\label{Smobius}

The M\"obius function was defined by M\"obius in 1832, but had been studied 30 years earlier by Gauss.
It plays a key role in the M\"obius inversion formula and the Riemann hypothesis, which is a famous unsolved problem in mathematics.

\begin{definition}
The M\"obius function $\mu(N)$ is defined as follows:
\begin{itemize}
    \item 
If there is a prime $p$ such that $p^2$ divides $N$, let $\mu(N)=0$; 
\item $\mu(1)=1$; and
\item If $N=p_1 \cdot p_2 \cdots p_n$ where $p_1, \ldots, p_n$ are distinct primes, let $\mu(N)=(-1)^n$.
\end{itemize}
\end{definition}

\begin{example}
Let $N=22$.  The divisors $d$ of $N$ are $1,2,11,22$.  We compute 
$\mu(1)=1$, $\mu(2)=-1$, $\mu(11)=-1$, and $\mu(22)=1$.  
We compute that
$\mu(1)+\mu(2)+\mu(11)+\mu(22)=0$.
\end{example}

\begin{question}
For the other values of $N$ between $20$ and $30$, compute the value of 
$\sum_{d \mid N} \mu(N)$, where the sum is over all divisors of $N$ including $1$ and $N$.  Make a conjecture about this.
\end{question}

\subsection{Multiplicative functions}

Let $N$ and $M$ be positive integers.  We say that 
$N$ and $M$ are \emph{relatively prime} if ${\rm gcd}(N,M)=1$, meaning that $N$ and $M$ have no prime factors in common.

Let $f$ be a function on the positive integers.
We say that $f$ is \emph{multiplicative} if $f(N \cdot M) = f(N) \cdot f(M)$ whenever $N$ and $M$ are relatively prime.

\begin{question}
\begin{enumerate}
    \item Show that $\sigma_0$ is a multiplicative function.
    \item Show that $\sigma_1$ is a multiplicative function.
    \item Show that $\mu$ is a multiplicative function.
    \item For each of the functions $\sigma_0$, $\sigma_1$, $\mu$, give an example that shows the multiplicative property fails when $M$ and $N$ are not relatively prime.
\end{enumerate}
\end{question}

%% file: 03-Combinations/Combinations.tex
\chapter{Counting combinations}\label{chap:combinations}

We now arrive at the heart of combinatorics, and the origin of its name, which comes from counting \textit{combinations}.

\section{The types of combinations} \label{sec:table-intro}

\begin{videobox}
\begin{minipage}{0.1\textwidth}
\href{https://www.youtube.com/watch?v=4Aw9vYMbDL8}{\includegraphics[width=1cm]{video-clipart-2.png}}
\end{minipage}
\begin{minipage}{0.8\textwidth}
Click on the icon at left or the URL below for this section's short video lecture. \\\vspace{-0.2cm} \\ \href{https://www.youtube.com/watch?v=4Aw9vYMbDL8}{https://www.youtube.com/watch?v=4Aw9vYMbDL8}
\end{minipage}
\end{videobox}

How many ways can you choose a collection of $k$ objects from a collection of $n$ objects?  

What we mean by this question might vary depending on the situation; in particular, whether \textit{order matters} and whether or not \textit{repeats are allowed}.  Consider the following examples of each setting.

\begin{example} \label{Efruit1} (Order does not matter, repeats not allowed.)
Suppose you want to pick out two pieces of fruit to bring to work from the refrigerator, which contains an apple, a banana, an orange, a pear, and a mango.   In this case, \textit{order does not matter}, so the number of ways is $\binom{5}{2}=\frac{5!}{2!\cdot 3!}=10$.  
\end{example}
\begin{example} \label{Efruit2} (Order matters, repeats not allowed.)
On the other hand, suppose you want to eat one fruit on each of Monday and Tuesday next week.  In this case, \textit{order matters}. Now there are $5$ choices for which fruit to eat on Monday, but then when you get to Tuesday there are only $4$ choices left, so there are a total of $5\cdot 4=20$ choices.   
\end{example}
\begin{example} \label{Efruit3}(Order matters, repeats allowed.)
Suppose instead that your refrigerator is packed with plenty of fruits of each of the five types!  So for instance, you can eat an apple on both Monday and Tuesday if you like - indeed, \textit{repeats are allowed}.  Now there are $5$ choices for which type of fruit to eat each day, for a total of $5\cdot 5=25$ possibilities. 
\end{example}
\begin{example} \label{Efruit4}(Order does not matter, repeats allowed)
Finally, suppose \textit{repeats are allowed} (there are plenty of each type of fruit in the fridge) but \textit{order doesn't matter} (you are just grabbing two pieces of fruit to put in your lunch bag for Monday).  Now there are $5$ ways to pick the two pieces of fruit to be the same kind as each other, and if they're different, there are $10$ possibilities as we saw in Example~\ref{Efruit1}.  In total, there are $10+5=15$ possibilities. 
\end{example}

The examples above show that the answer of ``how many ways can I choose two fruits from five'' can either be $10$, $20$, $25$, or $15$ depending on whether order matters and whether repeats are allowed!  We can summarize our findings in the following table:

\begin{center}
\textbf{Number of ways to choose 2 pieces of fruit from 5 fruits} 

\begin{tabular}{c|cc}
     & Order does not matter & Order matters\\\hline
Repeats not allowed & 10 & 20 \\
Repeats allowed     & 15 & 25 \\
\end{tabular}
\end{center}

In this chapter, we will study each of the four cases above in depth.
The goal is to find a general formulas for each entry of the table, by computing the number of ways to choose $k$ objects from $n$ objects in each case. Throughout this chapter, we will assume $k$ and $n$ are natural numbers.

\subsection*{Exercises}
\begin{enumerate} 
\item Find the four entries in the table if another piece of fruit, a strawberry, is added to the refrigerator.

\item Find the four entries in the table if a third piece of fruit will be chosen (for Wednesday in Examples~\ref{Efruit2} and \ref{Efruit3}).

\item The Student Government of Combinations College consists of a president, vice president, secretary, and treasurer. Seven people are running for student government. How many possible sets of four people can make up the student government? (Ignore for now the choice of which of these four people is president, vice president, etc.)																									
\item Using the setup of the previous problem, how many different possible choices of president, vice president, secretary, and treasurer can there be, out of the seven people running?																									
\item How many ways can you pick two marshmallows from a large bag of colored marshmallows, where there are 6 possibilities for the color? (For instance, you could pick two red marshmallow, or one red and one blue, or any other combination of the 6 available colors. Assume the order in which you pick them does not matter).																									
\item There are 6 different colors of marshmallows in a large bag of marshmallows. How many ways can you stick 4 marshmallows on a toothpick that's sticking out of the top of a birthday cake?																									
\item In the previous problem, how many possibilities are there if the four marshmallows all have distinct colors?

\end{enumerate}

\section{Sequences}

\begin{videobox}
\begin{minipage}{0.1\textwidth}
\href{https://www.youtube.com/watch?v=-FP3UgU9ZNA}{\includegraphics[width=1cm]{video-clipart-2.png}}
\end{minipage}
\begin{minipage}{0.8\textwidth}
Click on the icon at left or the URL below for this section's short video lecture. \\\vspace{-0.2cm} \\ \href{https://www.youtube.com/watch?v=-FP3UgU9ZNA}{https://www.youtube.com/watch?v=-FP3UgU9ZNA}
\end{minipage}
\end{videobox}

Sequences allow us to order or list the elements in sets.

\begin{definition}
A \defn{sequence} is a list of numbers or other symbols written in order.
\end{definition}

A sequence can be either \textit{finite} or \textit{infinite}.  We write a finite sequence as a list separated by commas and enclosed by parentheses, as in: $$(5,2,3,6)$$ or $$(2,4,6,8,10,\ldots).$$

The former is a finite sequence, of \textit{length} $4$ since there are $4$ numbers in the sequence.  The latter is an infinite sequence.  This notation can be generalized using subscripts as follows.

\begin{notation}
We write $$(a_1,a_2,\ldots,a_n)$$ to denote a finite sequence of length $n$,
whose first entry is $a_1$, second entry is $a_2$, etc, and ending with the $n$th entry being $a_n$.  
  We write $$(a_1,a_2,a_3,a_4,\ldots)$$ to denote the infinite sequence whose $i$th entry is $a_i$ for each integer $i \geq 1$.
\end{notation}

The set of numbers or symbols we use in a sequence is called the \textit{alphabet}.
Commonly used alphabets are the binary alphabet $\{0,1\}$, the letters in the English alphabet, ASCII, or the natural numbers.

The next remark explains some 
other variations on sequences.

\begin{remark} \label{rem:seqquence-string}
\begin{enumerate}
    \item It is sometimes convenient to ``zero-index'', by starting a sequence with $a_0$ rather than $a_1$, see  
Chapter \ref{chap:recurrence}.

\item
It is sometimes convenient to write a sequence as a \textit{string} or \textit{word}, in which we drop the parentheses and commas. For instance, the sequence $(a,p,p,l,e)$ can be more easily written as the string of letters $apple$, and the string of digits $(1,0,1,1,0)$ can be written as the binary string $10110$.

\textit{Warning:} We do not want to write the sequence $(1,0, 10,100,1000)$ as the string $10101001000$, because
we would lose information about the length of each entry.
This topic is important in coding theory where it is called the theory of uniquely decipherable codes.

\item 
A sequence of length $n$ can also be thought of as a function that 
assigns a
number (or symbol) to every positive integer from $1$ to $n$.
An infinite sequence can also be thought of as a function that 
assigns a number (or symbol) to every positive integer. 
\end{enumerate}
\end{remark}

\subsection{Counting where order matters and repeats are allowed}

The type of counting  seen in Example~\ref{Efruit3}, where \textbf{order matters} and \textbf{repeats are allowed}, can be restated in terms of counting sequences.  Indeed, if we choose $k$ objects from $n$ objects in order, with repeats allowed, we can write them as a sequence from an alphabet with $n$ symbols.

We can count such sequences as follows.

\begin{theorem} \label{Tstring}
The number of sequences of length $k$ composed from an alphabet with $n$ symbols is $n^k$.\\ More generally, there are $n^k$ combinations of $k$ objects chosen from $n$ objects, where order matters and repeats are allowed.
\end{theorem}

\begin{proof}
  There are $n$ ways to choose the first symbol in the sequence, then $n$ ways to choose the second, and so on up to the $k$th.  Therefore there are $n\cdot n\cdot n\cdot \cdots \cdot n=n^k$ possibilities.  
\end{proof}

Using Theorem~\ref{Tstring}, we can complete the general formula for one entry of the table.

\begin{center}
\textbf{Number of ways to choose $k$ objects from $n$ objects} \\\vspace{0.4cm}
\renewcommand{\arraystretch}{2}
\begin{tabular}{c|cc}
     & Order does not matter & Order matters\\\hline
Repeats not allowed &  &  \\
Repeats allowed     &  & $n^k$ \\
\end{tabular}
\end{center}

Let's look at a few examples;
for convenience we write these examples as strings.

\begin{example}
\label{Ebinarystring}
The number of binary strings (consisting of $0$'s and $1$'s) of length $k$ is $2^k$ (here $n=2$).  We can verify this for $k=3$ by listing all $8=2^3$ possible binary strings of length three: $$000,001,010,011,100,101,110,111.$$
\end{example}

\begin{example}
The number of 3-digit strings of letters from the English alphabet is $(26)^3=17,576$.  Some examples are $pxy$ or $csu$.
\end{example}

\begin{example}
The number of 3-digit strings of letters or numbers, such as $p17$, $2xy$, or $pqm$, is $(26+10)^3=36^3=46,656$.
\end{example}

We can use the formula $n^k$ for any situation in which order matters and repeats are allowed, even if we are not explicitly counting sequences or strings.

\begin{example}
At a buffet table, there are many samosas, chicken skewers, and carrots.  The number of ways to eat 10 snacks in a row is $3^{10}$.
\end{example}

 The formula $n^k$ was derived using the multiplication principle, by considering how many choices there are for each element of the sequence.  In Theorem~\ref{Tstring}, there is one alphabet of size $n$ that is used for all $k$ positions in the string.  We can generalize this by using a different alphabet in each position.

\begin{theorem}\label{thm:strings}
There are $n_1\cdot n_2\cdot \ldots \cdot n_k$ possible sequences of length $k$ made by
\begin{itemize}
\item choosing one of $n_1$ symbols as the 1-st entry,
\item choosing one of $n_2$ symbols as the 2-nd entry,

\hspace{5mm}\vdots
\item choosing one of $n_k$ symbols as the $k$-th entry.
\end{itemize}
\end{theorem}

\begin{remark}
Theorem~\ref{Tstring} is the special case of Theorem~\ref{thm:strings} where $n_1=n_2=\ldots=n_n=n$.
\end{remark}

\begin{example}
How many 4-digit numbers are there that have no leading zeros? For example, 2347 is one such number but 0023 is not.
\end{example}

\begin{answer}
There are 9 choices (1-9) for the first digit, and 10 choices (0-9) for the remaining three digits. Hence the total number is $9\cdot 10\cdot 10\cdot 10=9000$.
%
\end{answer}

\begin{example}
How many 3-digit sequences of letters from the English alphabet in which the first and last letters are consonants and the middle is a vowel (A, E, I, O, or U)?
Since there are $5$ vowels and $21$ consonants, there are $21\cdot 5 \cdot 21 = 2205$ possibilities.
\end{example}

\subsection*{Exercises}
\begin{enumerate} 
\item
\begin{itemize}
\item[(a)] Find the number of all 4-digit strings of letters. String $aabb$ is different from $abab$.
\item[(b)]  Find the number of all 4-digit strings of letters in which no two consecutive letters are the same.  For example, strings $xdwa$ and $xdxd$ count but strings $xdww$ and $xddx$ do not.
\item[(c)]  Find the number of all 4-digit strings of letters in which letters which differ by one slot are not the same, and also letters which differ by two slots are not the same. For example, string $xwzz$ does not count since there are $z$'s which differ by one slot, string $xzwz$ does not count since there are $z$'s which differ by two slots, but string $zxwz$ is okay.
\end{itemize}

\item How many ways are there to send one postcard to each of 10 different friends, from a large supply of each of 4 different kinds of postcards (a ram, buffalo, fox, and goat print postcard). If Student A receives a ram postcard and Student B receives a buffalo postcard, that is different than if Student A receives a buffalo postcard and Student B receives a ram postcard.

\item How many 5-digit strings of digits 0-9 are there such that no two consecutive digits are the same? For example, strings $01323$ and $93572$ are allowed, but $00232$ and $93552$ are not.

\item Student $d$ starts a job at a license plate manufacturer.
Each license plate is of the form XXX~\#\#\#, where each entry marked X can be any capital letter from the English alphabet and each digit marked \# can be any digit 0-9.  
If Student $d$ can make 33,000 license plates a day, how many days of work will this give Student $d$?

\item Your neighbor works for your school's accounting division, which receives scholarship donations from all over the world.
Each donor is assigned an account fund, labeled by a $5$ digit string where the first digit is `C' (standing for Code), starting at CA000 through CA999, then CB000 through CB999, etc.
The available codes assigned via this system will run out after codes CZ000 through CZ999 are assigned.
About 700 new codes are used every year.
At the beginning of the school year, the next code to be assigned is CW182.
Your neighbor plans to retire 5 years from now.
Will the available codes run out before she retires? 
\end{enumerate}

\section{Permutations and other sequences with distinct entries}

The seven students in the set $P=\{a,b,c,d,e,f,g\}$ go out for ice cream. When they arrive at the ice cream shop, they need to wait in line.
Student $a$ argues they should wait in alphabetical order, but student $g$ has a strong opposition to this ordering.
This prompts the question: How many different ways are there for them to wait in an ordered line?  

To solve this kind of problem, we need the following definition.

\begin{definition}
A \defn{permutation} of $n$ objects is an ordering or arrangement of all $n$ of them. 
\end{definition}

A permutation can be expressed as a sequence or string.

\begin{example}
For example, the set $\{a,b,c\}$ has $6$ permutations:
\[ abc\quad acb\quad bac\quad bca\quad cab\quad cba \]
\end{example}

\begin{example}
To write down a permutation of the set $\{1,2,3,4,5,6,7,8,9,10\}$, we use sequence notation rather than string notation, since $10$ has two digits.  One possible permutation is: $$(3,6,7,2,1,10,4,5,9,8).$$ 
\end{example}

Recall that $n!=n\cdot(n-1)\cdot(n-2)\cdot\ldots\cdot3\cdot2\cdot1$.

\begin{theorem}\label{thm:factorial}
The number of permutations of $n$ objects is $n!$.
\end{theorem}

\begin{proof}
By Theorem \ref{thm:strings}, there are $n$ choices for the first object, then $n-1$ choices for the second object, then $n-2$ for the third, and so on until there is only one $1$ choice remaining for the $n$th object.
So there are $n\cdot(n-1)\cdot(n-2)\cdot\ldots\cdot3\cdot2\cdot1=n!$ permutations in total.
\end{proof}

\begin{remark}\label{rmk:zero-factorial}
 If there are $0$ objects, then there is only one way to arrange them (do nothing).  This can be thought of as the ``empty sequence'' or ``empty string''.  Therefore, we set $0!=1$. 
\end{remark}

\begin{example}
  When seven students go out for ice cream, there are $7!=5040$ ways for them to line up, since each way corresponds to a permutation of elements from $P=\{a,b,c,d,e,f,g\}$.
\end{example}

\begin{example}
 Suppose there are $5$ different fruits in your refrigerator: an apple, a banana, an orange, a pear, and a mango.  How many ways can you bring one fruit on each day of the work week?
 Let's call the fruits $A,B,O,P$, and $M$ for short.  Then a possible assignment of each of the fruits to a day of the work week is a permutation of $\{A,B,O,P,M\}$.  
 Therefore, the number of possibilities is $5!=5\cdot 4 \cdot 3 \cdot 2 \cdot 1=120$.
\end{example}

The problems in this section can be rephrased as counting problems where order matters and repeats are not allowed, in which we arrange all of the objects.
In the next section, we consider the more general problem of choosing $k$ objects from $n$ objects where order matters but repeats are not allowed.  These can be thought of as \textit{ordered subsets}.

\subsection{Ordered subsets}

The ice cream adventurers
$P=\{a,b,c,d,e,f,g\}$
decide to wait in line for ice cream yet again.  However, to their collective horror, the ice cream shop is closing and can only serve $4$ more customers.  How many ordered lines are there that contain $4$ people from the set $P$?
To make such a line, we choose a subset $S$ of $P$ of size $4$ and then order the elements in $S$.  

This problem is similar to the one seen in Example~\ref{Efruit2}, where {\bf order matters} and {\bf repeats are not allowed}.
More generally, consider the following definition.

\begin{definition}
Suppose $0 \leq k \leq n$. 
  An \defn{ordered subset} of size $k$ from a set $A$ is a string $a_1,\ldots,a_k$ of distinct elements of $A$.
\end{definition}

\begin{theorem}\label{thm:ordered-subsets}
Suppose $0 \leq k \leq n$.
The number of ordered subsets of size $k$ from a set with $n$ elements is
\begin{equation}
    \label{Eordsubsizek}
n\cdot(n-1)\cdot(n-2)\cdot\ldots\cdot(n-k+1)=\frac{n!}{(n-k)!}.
\end{equation}
More generally, $\frac{n!}{(n-k)!}$ is the number of ways to choose $k$ objects from $n$ objects, where order matters and repeats are not allowed.
\end{theorem}

\begin{proof}
   An ordered subset of size $k$ is a sequence of length $k$ chosen from an alphabet of size $n$ in which no two entries are equal.  There are $n$ possibilities for the first entry. Once that entry is chosen, there are only $n-1$ remaining possibilities for the second one, then $n-2$ for the third, and so on.
   For the final $k$th entry, there are $n-k+1$ possibilities.
   The decreasing number of possibilities for each entry guarantees that there are no repeated entries.
   The formula in \eqref{Eordsubsizek} follows from the multiplication principle.
\end{proof}

\begin{remark}
Since $n-k+1=n-(k-1)$,
there are $k$ entries in the product $n\cdot(n-1)\cdot(n-2)\cdot\ldots\cdot(n-k+1)$.  It is helpful to think of these entries as being labeled from 0 to $k-1$.
\end{remark}

\begin{example}
There are $7\cdot 6\cdot 5 \cdot 4=840$ possible ways to form a line of four ice cream adventurers.
\end{example}

Using Theorem~\ref{thm:ordered-subsets}, we can complete the general formula for the upper right entry of the table.

\begin{center}
\textbf{Number of ways to choose $k$ objects from $n$ objects} \\\vspace{0.4cm}
\renewcommand{\arraystretch}{2}
\begin{tabular}{c|cc}
     & Order doesn't matter & Order matters\\\hline
Repeats not allowed &  & $\frac{n!}{(n-k)!}$ \\
Repeats allowed     &  & $n^k$ \\
\end{tabular}
\end{center}

\begin{example} 
  As in Example~\ref{Efruit2}, the number of ways to choose one out of five fruits to eat on Monday and then another to eat on Tuesday is $$\frac{5!}{3!}=\frac{5\cdot 4 \cdot 3 \cdot 2 \cdot 1}{3\cdot 2 \cdot 1} =5\cdot 4 = 20.$$
\end{example}

\noindent\textbf{Tip:} When dividing one factorial by another, expand both out as products and cancel as many terms as possible before multiplying.  For example, the computation above was easier than computing $\frac{5!}{3!}=\frac{120}{6}=20$.

\begin{example}
How many different arrangements are there for 47 students sitting in a classroom with 50 seats?
\end{example}

\begin{answer}
This is the number of ordered subsets of size $47$ from a set of size 50.  Hence the answer is $50\cdot 49\cdot\ldots\cdot 4=\frac{50!}{3!}$.
\end{answer}

\begin{remark}
In Section~\ref{ss:counting-problems}, we answered this question by adding 3 ``ghosts" to the class. These answers are really the same!
\end{remark}

\subsection*{Exercises}
\begin{enumerate}
\item How many different ways can you permute the 26 letters of the English alphabet?

\item How many ways are there to scramble the letters in ABSEINO, in the following situations? 
\begin{enumerate}
\item If each letter is used exactly once in a seven letter word.

\item If each letter is used at most once in a five letter word.

\item How many actual five letter English words, such as NOISE, can you make from these letters?
\end{enumerate}

    \item 
    In how many ways can you visit 5 of the 50 state capitals? The trip 
    \[\mbox{Sacramento} \to \mbox{Dover} \to \mbox{Baton Rouge} \to \mbox{Bismarck} \to \mbox{Denver}\]
    is different than the trip 
    \[\mbox{Dover} \to \mbox{Baton Rouge} \to \mbox{Denver} \to \mbox{Bismarck} \to \mbox{Sacramento}.\]

    \item In how many ways can you rank your favorite 10 restaurants (from \#1 to \#10) out of a collection of 50 restaurants?
    
    \item There are 128 NCAA DI football teams. How many possible top-25 ranked lists are there?  One such list might be:
\begin{enumerate}
\item[1.] CSU
\item[2.] Oklahoma
\item[3.] Washington

\hspace{5mm}\vdots
\item[25.] Alabama
\end{enumerate}

\end{enumerate}

\section{Sets: When order doesn't matter}
After receiving numerous complaints about the length of its lines, the ice cream shop decided to reward their customers with an \emph{Ice Cream for a Year} contest. This contest allows 4 people to win free ice cream for a year!  The seven students in the set $P=\{a,b,c,d,e,f,g\}$ eagerly enter, along with 100 other people.
Out of the 107 entrants, how many ways are there to choose 4 winners?  All winners receive the same prize.

This problem is similar to the one seen in Example~\ref{Efruit1}, where {\bf order does not matter} and 
{\bf repeats are not allowed}.
 We can solve problems of this type with the binomial coefficient formula from Section~\ref{S.1.2.3}.

\begin{theorem}
\label{Tupperleft}
Suppose $0 \leq k \leq n$.
The number of subsets of size $k$ in a set of size $n$ is $$\binom{n}{k}=\frac{n!}{k!(n-k)!}.$$ More generally, $\binom{n}{k}$ is the number of ways to choose $k$ objects from $n$ objects, where order does not matter and repeats are not allowed.
\end{theorem}

\begin{proof}
We already proved this in Lemma~\ref{Lformulabinom}.  Here we repeat the proof from Method 1, using the vocabulary of sets.
By Theorem \ref{thm:ordered-subsets}, the number of \emph{ordered} subsets of size $k$ is $\frac{n!}{(n-k)!}$. However, that counts each unordered subset $k!$ times, since $k!$ is the number of permutations of $k$ objects (Theorem \ref{thm:factorial}). Hence, by the division principle, the number of unordered subsets of size $k$ is $$\frac{\frac{n!}{(n-k)!}}{k!}=\frac{n!}{k!(n-k)!}=\binom{n}{k}.$$
\end{proof}

Using Theorem~\ref{Tupperleft}, we can complete the upper-left entry of the table.

\begin{center}
\textbf{Number of ways to choose $k$ objects from $n$ objects} \\\vspace{0.4cm}
\renewcommand{\arraystretch}{2}
\begin{tabular}{c|cc}
     & Order does not matter & Order matters\\\hline
Repeats not allowed & $\binom{n}{k}$ & $\frac{n!}{(n-k)!}$ \\
Repeats allowed     &  & $n^k$ \\
\end{tabular}
\end{center}

\begin{example}
  The number of ways to choose $2$ fruits from $5$ to bring to work on the same day is $$\binom{5}{2}=\frac{5!}{2!\cdot 3!}=\frac{5\cdot 4 \cdot 3 \cdot 2 \cdot 1}{2\cdot 1 \cdot 3 \cdot 2 \cdot 1}=5\cdot 2=10.$$
\end{example}

\begin{example}
How many ways are there to select the top 25 NCAA DI football teams (of which there are 128) without ranking them?

This is the number of (unordered) subsets of size 25, giving $\binom{128}{25}=\frac{128!}{25!(128-25)!}= \frac{128!}{25!103!}$
\end{example}

\begin{example}
\label{Ewinice}
The number of ways to pick 4 winners out of 107 entrants for the \emph{Ice Cream for a Year} contest is $\binom{107}{4}=\frac{107!}{4! \cdot 103!}=5,160,610$.
\end{example}

Choosing the four winners
in Example~\ref{Ewinice} is the same as eliminating the 103 people who did not win.
So the number of ways to choose $103$ people out of $107$ gives the same answer.
  This leads us to the following corollary, which is often used to simplify problems.

\begin{corollary}\label{cor:symmetry}
  These two binomial coefficients are equal: $\binom{n}{k}=\binom{n}{n-k}$.
\end{corollary}


Surprisingly, the formula $\binom{n}{k}$ can also be useful in counting certain sequences or strings.

\begin{example}
How many binary strings (from the alphabet $\{0,1\}$) have exactly $k$ zeroes and $n-k$ ones?  To form such a binary string, we simply have to choose $k$ positions for the $0$'s, out of $n$ possible positions.  This is the same as choosing an \textit{unordered} subset of $k$ of the positions in the string (even though the string itself is ordered).  Therefore, there are $\binom{n}{k}$ such strings.
\end{example}

\subsection*{Exercises}
\begin{enumerate}
\item You have 7 different rolls of wrapping paper and 4 identical teddy bears. How many different ways can you wrap the teddy bears such that no two teddy bears are wrapped in the same wrapping paper?

\item Student $a$ has six wristbands of different colors and student $b$ has seven necklaces of different colors.
In how many ways can $a$ trade two wristbands for three of $b$'s necklaces?

\item Suppose student $d$ has 9 pencils (all different) and student $e$ has 6 erasers (all different).
In how many different ways could student $d$ trade 4 of their pencils for 2 of student $e$'s erasers?

\item Give an algebraic proof that $\binom{n}{2}+\binom{n+1}{2}=n^2.$ (Do not use induction, even if you know what that is.)

\item 
How many ways are there to send postcards to 10 different friends, from a large supply of each of 4 different kinds of postcards (a ram, buffalo, fox, and goat print postcard), in the following scenario?
You want to send either 1, 2, 3, or 4 postcards, such that no person gets two postcards of the same animal? For example, student $a$ might receive the ram, buffalo, and fox postcards, while student $b$ might receive only the ram postcard, and students $a$ and $b$ can receive the same set of postcards.
\end{enumerate}

\section{Multisets: sets with repeats allowed}\label{sec:sticks-and-stones}

\begin{videobox}
\begin{minipage}{0.1\textwidth}
\href{https://www.youtube.com/watch?v=8OSEZyL_O7E}{\includegraphics[width=1cm]{video-clipart-2.png}}
\end{minipage}
\begin{minipage}{0.8\textwidth}
Click on the icon at left or the URL below for this section's short video lecture. \\\vspace{-0.2cm} \\ \href{https://www.youtube.com/watch?v=8OSEZyL\_O7E}{https://www.youtube.com/watch?v=8OSEZyL\_O7E}
\end{minipage}
\end{videobox}

Student $c$ was one of four lucky winners of the ice cream contest!
To celebrate, the seven students decide to get a large sundae in a bowl, filled with many different scoops of ice cream.
There are 10 flavors of ice cream and a large sundae has $15$ scoops.
Since it is a hot day in Colorado, all of the scoops will melt together, so the order of the scoops does not matter.
Each flavor can be used for as many of the $15$ scoops as they want.
How many different sundaes can student $c$ order?  

This problem is similar to the one seen in Example~\ref{Efruit4}.
To solve problems of this type, we define a \textit{multiset}; intuitively, this is a set whose elements do not need to be different. 

\subsection{Definition of multiset}

\begin{example}
  The multiset $\{\text{apple},\text{apple},\text{orange},\text{pear},\text{pear}\}$ denotes a collection of two identical apples, one orange, and two identical pears.
\end{example}

\begin{definition}
  A \defn{multiset} is
      a set, together with a positive integer multiplicity assigned to each element. 
      Equivalently, it is a set $S$, together with a function 
      $f:S \to {\mathbb N} - \{0\}$.
\end{definition}

A multiset can be written down in several ways:
\begin{itemize}
    \item 
by listing an element of the set $m$ times if it has multiplicity $m$;
 \item by 
 writing the multiplicity as a superscript on the element; (This method can sometimes lead to confusion.)
 \item by writing down the function $f$.
 \end{itemize}

\begin{example}
Here are several ways to write the multi-set with three $4$'s, one $5$, and two $6$'s:
\begin{itemize}
    \item $\{4,4,4,5,6,6\}$ or
    $\{4,5,4,6,4,6\}$ (the order does not matter).
    \item $\{4^{x3},5^{x1},6^{x2}\}$
    Warning! The $4^{x3}$ written here does not mean $4$ raised to the power $x3$.
    \item the function
    $f:\{4,5,6\} \to \N$
    with $f(4)=3$, $f(5)=1$, and $f(6)=2$.
\end{itemize}
\end{example}

\subsection{Example of sticks and stones method}

To count the number of multi-sets of a certain kind,
 we use what we call the ``sticks and stones'' method.
(In other sources, this is sometimes called the ``stars and bars'' method.) We illustrate this method first with an example.

\begin{example}
\label{Estickstone}
Suppose we have an unlimited supply of $5$ kinds of fruit: apples, bananas, mangoes, oranges, and pears.  Let's write them as A, B, M, O, and P for short.  How many ways can we make a bag of $6$ fruits?

Such a bag can be described as a multiset, for instance, $\{A,B,B,M,M,P\}$.
An alternative way to write this multiset is with a string of sticks (written $|$\,) and stones (written $\bullet$).  The stones represent fruit and the sticks are dividers between different types of fruit (say in alphabetical order A, B, M, O, P).  So the multiset $\{A,B,B,M,M,P\}$ is written $${\bullet} \mid \bullet\,\, \bullet \mid \bullet\,\, \bullet \mid \,\,\mid \bullet.$$  The first stone indicates that there is one apple, then the next two stones indicates that there are two bananas, and so on.
There are no stones between the last two sticks because there are no oranges in the bag.

Each bag can be written using six stones and four sticks.  There are $4$ sticks because there are $5$ types of fruit to separate. 
On the other hand, every string of six stones and four sticks describes a bag of fruit. 
For example, the string 
$$\mid \, \, \mid {\bullet} \mid \bullet\,\, \bullet \, \, \bullet\,\, \bullet \mid  \bullet$$
describes the bag 
$\{M, O, O, O, O, P\}$.

In this way, there is a 1-to-1 matching between the bags and the strings.
So to count the number of bags of six fruits (with five kinds of fruit), we can count the number of strings of $6$ stones and $4$ sticks  This is equal to the number of ways to choose $6$ positions out of $10$ possible positions for the location of the stones.  The sticks are then placed in the remaining $4$ positions.  
Thus the answer is $\binom{10}{6}=210$.
\end{example}

Example~\ref{Estickstone} illustrates several facts.  First, the number of multi-sets of size $k$ whose elements come from a fixed set of size $n$ equals the number of combinations of $k$ elements from a set of $n$ elements where order does not matter but repeats are allowed. 
Second, this number can be re-expressed as the number of ways to write a string of $k$ stones and $n-1$ sticks.
Third, this number equals the number of ways to choose $k$ locations for stones out of $k + (n-1)$ locations for sticks and stones.  We state this more formally in the next theorem.

\subsection{Main result on multisets and examples}

\begin{theorem}\label{thm:choose-with-repeats}
The number of multisets of size $k$ in which every element is chosen from an alphabet of size $n$ is $\binom{n+k-1}{k}$.\\
More generally, $\binom{n+k-1}{k}$ is the number of ways to choose $k$ objects from $n$ objects, where order does not matter and repeats are allowed.
\end{theorem}

\begin{proof}
  We can write the multiset as a string of $k$ stones and $n-1$ sticks, where the $n-1$ sticks separate the $n$ types of possible elements of the alphabet in some chosen ``alphabetical order''.  Then there are $n-1+k=n+k-1$ total sticks and stones in the string.  We count the number of ways to choose which $k$ of the positions in the string contain a stone (as opposed to a stick).  There are therefore $\binom{n+k-1}{k}$ possibilities.  
\end{proof}

\begin{remark}
By the symmetry of binominal coefficients,  $\binom{n+k-1}{k}=\binom{n+k-1}{n-1}$; this corresponds to choosing the location of the sticks instead.
\end{remark}

Using Theorem~\ref{thm:choose-with-repeats}, we can include the formula in the lower left portion of the table and thus complete the table for counting combinations. 

\begin{center}
\textbf{Number of ways to choose $k$ objects from an $n$ objects} \\\vspace{0.4cm}
\renewcommand{\arraystretch}{2}
\begin{tabular}{c|cc}
     & Order doesn't matter & Order matters\\\hline
Repeats not allowed & $\binom{n}{k}$ & $\frac{n!}{(n-k)!}$ \\
Repeats allowed     & $\binom{n+k-1}{k}$ & $n^k$ \\
\end{tabular}
\end{center}
\

\begin{example}
A store sells 6 colors of balloons. How many different ways can you buy 10 balloons?
\end{example} 

\begin{answer}
We want to choose a collection of $k=10$ balloons, out of an unlimited supply of $n=6$ colors of balloons, so the answer is
$\binom{n+k-1}{n-1}=\binom{6+10-1}{6-1}=\binom{15}{5}$.
\end{answer}

\begin{example} \label{Etoomanysundaes}
When Student C orders the large sundae, there are $k=15$ scoops and $n=10$ flavors, so the number of possibilities is $\binom{10+15-1}{10-1}=\binom{24}{9}=1,307,504$ different sundaes they can order --- they will have to order about 3,593 sundaes a day to make sure they try every one in a year!
\end{example}

\subsection{A variation on multiset problems}

\begin{example}
In Example~\ref{Etoomanysundaes}, that is far too many sundaes for the students to eat every day! What if instead they wanted to order their sundae so that there was at least one scoop of each flavor?

The first $10$ scoops are determined (one of each flavor).
So student $c$ now needs to choose $k'=5$ more scoops from among the $n'=10$ flavors.
So the number of choices is 
 $\binom{14}{5}=2,002$ sundaes to order - about 6 sundaes a day!
\end{example}

\begin{corollary}
Let $m\ge n$.  The number of ways to choose $m$ elements from an $n$-element set where order does not matter, repeats are allowed, and at least one of each of the $n$ elements is chosen is $\binom{m-1}{n-1}$.
\end{corollary}

\begin{proof}
 We start by picking one of each element.
 Then we must pick $m-n$ additional elements from the $n$ possibilities, where repeats are allowed and order does not matter.  By Theorem \ref{thm:choose-with-repeats}, the number of ways to do this is $$\binom{n+(m-n)-1}{m-n}=\binom{m-1}{m-n}.$$ By the symmetry of binomial coefficients,  $\binom{m-1}{m-n}=\binom{m-1}{m-1-(m-n)}=\binom{m-1}{n-1}$.
\end{proof}

\begin{example}
A store sells 6 colors of balloons. How many different different bouquets of 10 balloons are there that contain at least one balloon of each color?  In this case, $m=10$ and $n=6$.  Hence the number of ways is $\binom{10-1}{6-1}=\binom{9}{5}$.
\end{example}

\subsection{Using multisets to solve distribution problems}

\begin{corollary}\label{Cdistribution}
The number of ways to distribute $k$ identical pennies to $n$ people is
$\binom{n+k-1}{k}$.  
\end{corollary}

\begin{proof}
Label the people by the letters $A_1, \ldots, A_n$.
Suppose $S$ is a multiset of size $k$, whose elements are the letters $A_1, \ldots, A_n$.
Then we can distribute the pennies in this way: for $1 \leq i \leq n$, the number of pennies that person $A_i$ receives equals the number of times that $A_i$ is included in the multiset.
Conversely, given a way of distributing $k$ pennies to the $n$ people, we can create a multiset $S$ of size $k$ from $\{A_1, \ldots, A_n\}$ so that the number of times $A_i$ is included in $S$ equals the number of pennies $A_i$ received.
So the number of ways to distribute $k$ pennies to $n$ people equals the number of multisets of size $k$ from a set of size $n$; this equals $\binom{n+k-1}{k}$ by Theorem~\ref{thm:choose-with-repeats}.  
\end{proof}

\begin{example}
When $k=9$ and $n=4$, for the multiset
$\{A_1, A_1, A_3, A_4, A_4\}$, person $A_1$ receives $2$ pennies, person $A_2$ receives $0$ pennies, person $A_3$ receives $1$ penny, and person $A_4$ receives $2$ pennies.
\end{example}

\begin{example}
When $k=9$ and $n=4$,
if person $A_2$ receives all the pennies, then the 
multiset is 
$\{A_2, A_2, A_2, A_2, A_2\}$.
\end{example}

\begin{example} \label{Edistpen}
The number of ways to distribute $k=9$ pennies to $n=4$ people is
$\binom{n+k-1}{k}=\binom{4+9-1}{9}=\binom{12}{9}=220$ ways.
\end{example}

\begin{example}
How many different ordered lists of 4 nonnegative integers have sum equal to 9? Examples include $(4,0,2,3)$, $(0,4,2,3)$, and $(0,0,0,9)$. 
\end{example}

\begin{answer}
The answer is the same as for Example~\ref{Edistpen}, namely $220$, because the number of these lists is the same as the number of ways to distribute $n=9$ identical pennies to $k=4$ people:\\
1st number on list = number of pennies to person 1,\\
2nd number on list = number of pennies to person 2,\\
3rd number on list = number of pennies to person 3,\\
4th number on list = number of pennies to person 4.
\end{answer}

\begin{corollary} \label{Cdistribute2}
If $m \geq n$, 
the number of ways to distribute $m$ identical pennies to $n$ people, so that each person must receive at least one penny, is $\binom{m-1}{n-1}$.
\end{corollary}

\begin{example}
The number of ways to distribute $m=9$ identical pennies to $n=4$ people so that each person gets at least one is $\binom{9-1}{4-1}=\binom{8}{3}=56$.
\end{example}

\begin{example} \label{Edistpen2}
How many different ordered lists of 4 positive integers have sum equal to 9? Examples include $(3,2,1,3)$, $(1,2,3,3)$, $(2,2,2,3)$, but not $(4,0,2,3)$ as 0 is not positive.
\end{example}

\begin{answer}
The answer is the same as for Example~\ref{Edistpen2},
namely 56.  The reason is that the number of these ordered lists is the same 
as the number of ways to distribute 9 identical pennies to 4 people where each person gets at least one.
\end{answer}

\subsection*{Exercises}
\begin{enumerate}

\item How many ways can you spend $\$3$ on sushi?						
\item How many ways can you spend $\$10$ on sushi?						
\item How many ways can you spend $\$10$ on sushi if you want at least one piece of each type?						
\item How many ways can you distribute 14 pennies to 6 kids?						
\item How many ways can you distribute 14 pennies to 6 kids if each kid gets at least 2?

\item You want to buy a bag of 15 lollipops and there are 5 different flavors to choose from. How many different bags can you make?  How many different bags can you make if you must have at least one of each flavor?


\item You want to order a tray of 20 tacos and there are 4 different types to choose from.  How many different trays can you order?  How many if the tray must contain at least one of each type of taco?

\item You want to make a bouquet of 50 balloons, and there are 10 different colors to choose from. All balloons of the same color are the same, and there are an unlimited number of balloons of each color. How many different bouquets can you make?
How many if the bouquet must include at least one of each color?

\item 
A store sells bouquets 
of 5 kinds of flowers.
How many ways are there to make a 
bouquet of 12 flowers?
How many ways if the bouquet must include at least one of each kind of flower?

\item
 How many different ways can we place 10 identical balls into 4 labeled bins, if each bin must contain at least one ball?

\item How many ways are there to distribute 10 identical pieces of paper to students $a$, $b$, and $c$?
The only requirement is that all 10 pieces of paper need to be handed out.

\item How many ways are there to paint 6 identical picture frames with $3$ colors if each color must be used at least once.
Let $G$ be green, $Y$ be yellow and $B$ be blue; then YYGGBB is the same as YGYGBB since the frames are identical.

\item Explain why 
Corollary~\ref{Cdistribute2} is true.

\end{enumerate}

\section{Summary}

The major results in this chapter are the counting formulas in this table.
\begin{center}
\textbf{Number of ways to choose $k$ objects from $n$ objects} \\\vspace{0.4cm}
\renewcommand{\arraystretch}{2}
\begin{tabular}{c|cc}
     & Order does not matter & Order matters\\\hline
Repeats not allowed & $\binom{n}{k}$ & $\frac{n!}{(n-k)!}$ \\
Repeats allowed     & $\binom{n+k-1}{k}$ & $n^k$ \\
\end{tabular}
\end{center}

\section{Additional problems for Chapter 3}
\begin{enumerate}

\item How many different strings are permutations of the letters in PILLOW?
For example, LOPILW is one such permutation.

\item Baseball team A has 7 pitchers, and baseball team B has 5 catchers. The two teams have agreed to make a trade of the following form: 3 of team A's pitchers will be traded for 2 of team B's catchers. The local newspaper decides to make a ranked list of their top 4 favorite possible trades of this kind. How many different ranked lists can the newspaper make?

\item How many ordered strings with 10 symbols are there if exactly 7 symbols are numbers (0--9) and exactly 3 symbols are letters ($a$--$z$)? The string $abc1222233$ is different from $12c2a3b322$.

\item
Find the number of ways to distribute 20 candies to 5 students $a,b,c,d,e$ if:
\begin{enumerate}
\item the candies are identical.
\item the candies are identical and every person receives at least one.
\end{enumerate}

\item Find the number of ways to give 20 pillows to 12 students $a$ through $l$ if:
\begin{enumerate}
\item the pillows are identical.
\item the pillows are identical and every person receives at least one.
\item the pillows are all different and there are no restrictions.  For example, one student can receive all 20 pillows.

\item How many different DNA sequences (whose alphabet consists of the letters A,C,G, and T) of length 5 are there?							
\item How many different strings of four distinct letters from the English alphabet (which has 26 letters) are there?							
\item How many binary strings of length 8 have the same number of 0's as 1's?							
\item How many strings of length 7 from the alphabet {0,1,2} have exactly three 0's?							
\item How many strings of length 7 from the alphabet {0,1,2} have less than three 0's?							
\end{enumerate}

\item How many ways can you roll a sum of 13 with three six-sided dice?  A die is a cube with 6 faces where each face contains one of the numbers $1$ through $6$. 
\end{enumerate}

\newpage
\section{Investigation: Counting problems in the game of poker}\label{poker}

One day after Math 301, several students decide to play a friendly game of poker. During the game, one of them asks ``How many ways are there to get a full house?" Stumped by this question, the students turn to their detailed notes from class to solve this problem. 

Before we solve this problem, we first will introduce some vocabulary about playing cards. 
If you already are familiar with the game of poker, you can skip the next subsection.

\subsection*{Facts about cards and poker}

 Poker is a card game played with a standard deck of 52 cards. The cards can be sorted into four \textbf{suits}, labeled by the symbols $\heartsuit, \spadesuit, \diamondsuit, \textnormal{and } \clubsuit$, called hearts, spades, diamonds, and clubs, respectively. There are 13 cards in each suit (quick check: $13\cdot4=52$). 
 The cards in each suit are labeled with values, from least to greatest, 
 2,3,4,5,6,7,8,9,10, J=Jack, Q=Queen, K=King, A=Ace=1.
 The set of cards held by a player is called a \textbf{hand}.  Typically a hand contains five cards.
 
\begin{example}

\begin{center}
    \includegraphics[width=.5\textwidth]{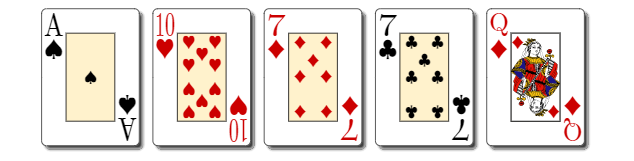}
\end{center}

This hand contains the Ace of spades, the 10 of hearts, the 7 of diamonds, the 7 of clubs, and the Queen of diamonds.
\end{example}

Another term used in many card games is called a run.  A \textbf{run} is a set of three or more cards with consecutive values where the suit does not matter: for example, the collection $3\heartsuit, 4\spadesuit, 5\spadesuit$ is a run of three cards, and $9\heartsuit, 10\spadesuit, 
{\rm J}\clubsuit, 
{\rm Q}\clubsuit, 
{\rm K}\diamondsuit$ is a run of five cards.  An important note is that an Ace can be either the lowest or highest card in a run (but can not be in the middle of a run). 
An example where the Ace is  ``high" is ($\textnormal{Q}\clubsuit, \textnormal{K}\diamondsuit, \textnormal{A}\diamondsuit$) and an example where the Ace is ``low" is ($\textnormal{A}\clubsuit, 2\heartsuit, 3\heartsuit$).  This set of three cards is not a run: $\textnormal{K}\heartsuit, \textnormal{A}\diamondsuit, 2\clubsuit$. \\

In this class, you do not need to know how to play poker.  We just want to count the number of ways to make different hands of five cards. 
In the following manual, we will describe each hand of poker, give an example, and explain the method of how to count that hand.

\section*{Poker Hands}
\subsection*{Royal flush} 
\textit{Definition of hand:} A run consisting of an Ace, King, Queen, Jack, and 10, all of the same suit.  \\
\textit{Example hand:} 
\begin{center}
    \includegraphics[width=.52\textwidth]{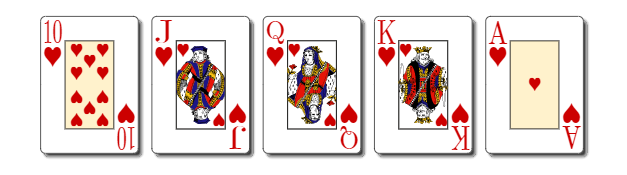}
\end{center}
\textit{Number of hands:} $\displaystyle \binom{4}{1}=4.$\\
\textit{Combinatorial proof:} The values on the cards are determined.  Therefore, we just need to choose one of four suits.

\subsection*{Straight flush} 
\textit{Definition of hand:} A straight is a run of five cards.  A flush means all of the cards are of the same suit.\\
\textit{Example hand:} 
\begin{center}
    \includegraphics[width=.47\textwidth]{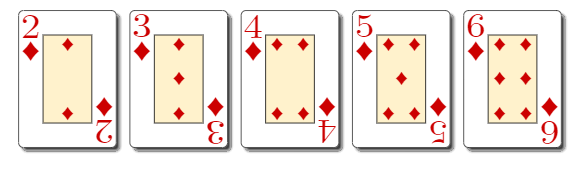}
\end{center}
\textit{Number of hands:} $\displaystyle \binom{10}{1}\cdot\binom{4}{1}-4=36.$\\
\textit{Combinatorial proof:} We need to choose the value of the starting card and the suit of the run (the other four values and suits are then determined).  There are 10 ways to begin the run (Ace through 10) and four suits to choose.  However, if we start the run with a 10, we would have a royal flush, so we subtract the number of royal flushes to get the answer. Note that if we tried to start a run with a Jack, we would not have enough cards, since Aces end runs.

\subsection*{Four of a kind}
\textit{Definition of hand:} A four of a kind consists of four same-valued cards and any other card. \\
\textit{Example hand:} 
\begin{center}
    \includegraphics[width=.5\textwidth]{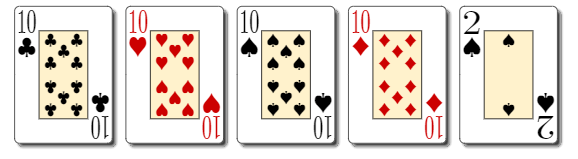}
\end{center}
\textit{Number of hands:} $\displaystyle \binom{13}{1}\binom{48}{1}=624.$\\
\textit{Combinatorial proof:} We first pick a value for the four cards.  Then, we choose the other card in the hand; since we chose 4 cards already, there are $52-4=48$ cards left in the deck.

\subsection*{Full House} 
\textit{Definition of hand:} This hand contains three cards of one value and two cards of a second value. \\
\textit{Example hand:} 
\begin{center}
    \includegraphics[width=.5\textwidth]{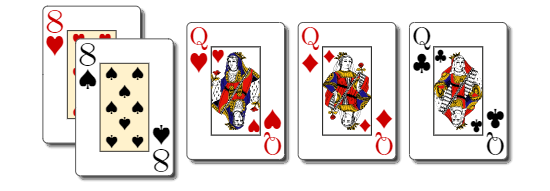}
\end{center}
\textit{Number of hands:} $\displaystyle \binom{13}{1}\binom{4}{3}\binom{12}{1}\binom{4}{2}=3,744.$\\
\textit{Combinatorial proof:} First, choose the value of the card in the triple and choose 3 suits for those 3 cards.  Then, from the remaining 12 values left, choose the value for the pair, and choose 2 of the four suits for those two cards.

\subsection*{Flush} 
\textit{Definition of hand:} A hand is a flush if all five cards have the same suit. The only restriction  is that the hand cannot be a run of five cards (or else this would be a straight or royal flush).\\
\textit{Example hand:} 
\begin{center}
    \includegraphics[width=.5\textwidth]{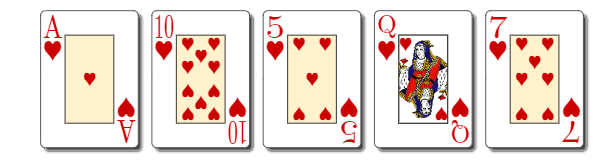}
\end{center}
\textit{Number of hands:} $\displaystyle \binom{4}{1}\binom{13}{5}-40=5,108.$\\
\textit{Combinatorial proof:} We first choose a suit for the hand.  Then, we choose 5 values from the 13 possible values.  However, some of these hand combinations could be straight or royal flushes, so we must subtract the total number of these two hands to get the answer.

\subsection*{Straight} 
\textit{Definition of hand:} A straight is a run of five cards consisting of at least two suits.  If they were all the same suit this would be a straight or royal flush.\\
\textit{Example hand:} 
\begin{center}
    \includegraphics[width=.5\textwidth]{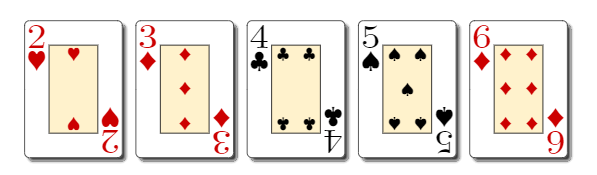}
\end{center}
\textit{Number of hands:} $\displaystyle \binom{10}{1}\binom{4}{1}^5-40=10,200.$\\
\textit{Combinatorial proof:} We need to pick the value of the starting card and a suit for each card. There are 10 ways to begin the run and four suits to choose from for each card. However, some of these hands are straight and royal flushes, so we subtract off the total number of those. 

\subsection*{Three of a kind} 
\textit{Definition of hand:} A three of a kind consists of exactly three same-valued cards and two other cards with distinct values.  If the two other cards were the same value, this would be a full house.\\
\textit{Example hand:} 
\begin{center}
    \includegraphics[width=.5\textwidth]{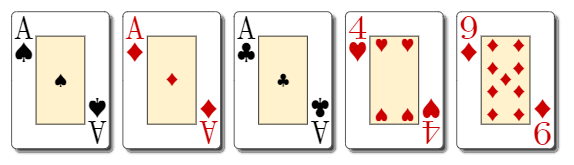}
\end{center}
\textit{Number of hands:} $\displaystyle \binom{13}{1}\binom{4}{3}\frac{\binom{48}{1}\binom{44}{1}}{2!}=54,912.$\\
\textit{Combinatorial proof:} First, we choose the value for three of the cards and their three suits.  We then choose the remaining two cards from the deck: the first is chosen from the remaining 48 cards in the deck (since we need to remove the possibility of a four of a kind); the second is chosen from the remaining 44 cards in the deck (since we need to remove the possibility of a full house). However, notice that the order in which we choose the last two cards does not affect the overall hand (for example, choosing $4\heartsuit$ then $9\heartsuit$ gives the same hand as $9\heartsuit$ then $4\heartsuit$).  So, we divide by $2!$. 

\subsection*{Two pairs} 
\textit{Definition of hand:} This hand consists of three values: two pairs and a third card of a distinct value.\\ 
\textit{Example hand:} 
\begin{center}
    \includegraphics[width=.4\textwidth]{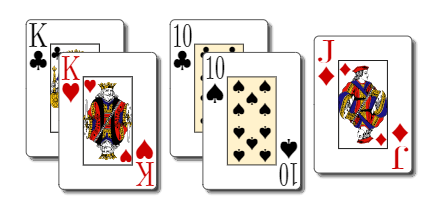}
\end{center}
\textit{Number of hands:} $\displaystyle\binom{13}{2}\binom{4}{2}^2\binom{11}{1}\binom{4}{1}=123,552.$\\
\textit{Combinatorial proof:} First choose two values for the numbers that will be in pairs.  Once these two values are chosen (for example king and 10), one of them is larger. Choose two suits for the pair with the larger number.  Choose two suits for the pair with the smaller number.  In order to avoid having three of a kind, we choose the number for the last card from the remaining 11 values. Finally, we choose the suit of the last card. 

\subsection*{Pair} 
\textit{Definition of hand:} Exactly two values are identical (different suits), all other cards are different values.  For example, consider having $4\heartsuit, 4\spadesuit, 6\diamondsuit, \text{K}\clubsuit, \text{J}\diamondsuit$.  If the other three cards were the same, this would be a full house.  If they contained an additional pair, it would be a two pair.  If they contained the same value(s) as the first pair, it would be a three or four of a kind. \\ \textit{Example hand:} 
\begin{center}
    \includegraphics[width=.45\textwidth]{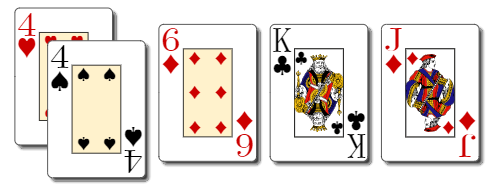}
\end{center}
\textit{Number of hands:} $\displaystyle\binom{13}{1}\binom{4}{2}\frac{\binom{48}{1}\binom{44}{1}\binom{40}{1}}{3!}=1,098,240$\\
\textit{Combinatorial proof:} First, we choose the value and suits in our pair.  Then, we choose the remaining 3 cards --- we start with 48 remaining cards in our deck (or else we would risk having three or four of a kind), then 44 cards (or else we would risk having two pairs), then 40 (again we would risk having two pairs).  However, notice the order in which we draw the cards does not matter, so we divide by $3!$ to get the final answer.

\subsection*{High Card} 
\textit{Definition of hand:} This hand is none of the above hands.  \\
\textit{Example hand:} 
\begin{center}
\includegraphics[width=.52\textwidth]{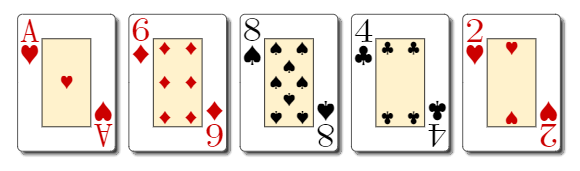}
\end{center}
\textit{Number of hands:} $1,302,540$\\
\textit{Combinatorial proof:} Left as an exercise.

\section*{Exercises}
\begin{enumerate}

\item
\begin{enumerate}
\item In your own words, explain how to count how many 4-of-a-kind poker hands there are.
(Recall that we consider the hands $K\heartsuit, K\diamondsuit, K\spadesuit, K\clubsuit,2\diamondsuit$ and $2\diamondsuit, K\spadesuit, K\clubsuit, K\diamondsuit, K\heartsuit$ to be the same.)
\item What fraction of all possible poker hands are 4-of-a-kinds?
\end{enumerate}

\item
\begin{enumerate}
\item In your own words, explain how to count how many different pair poker hands there are.
\item What fraction of all possible poker hands are pair poker hands?
\end{enumerate}

\item
\begin{enumerate}
\item In your own words, explain how to count how many different two pair poker hands there are.
A two pair consists of a pair of cards of one value, another pair of cards of a different value, and then a fifth card of a distinct value.
\item What fraction of all possible poker hands are two pair poker hands?
\end{enumerate}

\item 
\begin{enumerate}
\item In your own words, explain how to count how many different full house poker hands there are.
A full house consists of triple of cards of the same value, plus a pair of cards of the same value.
\item What fraction of all possible poker hands are full houses?
\end{enumerate}

    \item Find another way to count the number of ways a player could obtain the following hands:
    \begin{enumerate}
        \item Three of a kind
        \item Two pairs 
        \item Pair
    \end{enumerate}
    \noindent Check your answers numerically. 
    \item Use the Inclusion-Exclusion Principle to compute the number of ways to get a High Card, as follows.  Define the set $A$ to be the set of all hands having at least two cards of the same number.  Let $B$ be the set of all hands for which all cards are the same suit.  Let $C$ be the set of all straights.  Then the set of high cards is the complement of $A\cup B\cup C$.    
    
    Calculate the number of high card hands and check your answer numerically.
    \item The students decide to play Poker with some revised rules. They now play with hands containing 4 cards instead of 5. Compute the number of ways a player can now have the following hands and provide a combinatorial proof for each one:
        \begin{enumerate}
            \item Four of a kind  
            \item Three of a kind 
            \item Two pairs
            \item Pair
        \end{enumerate}

        \item For the following questions, compute the number of ways a player can now have (i) a straight flush, (ii) a flush, and (iii) a straight, and provide a combinatorial proof for each one when:
        \begin{enumerate}
            \item Aces only count as a high card (this eliminates $\textnormal{A}\heartsuit,2\clubsuit,3\diamondsuit,4\clubsuit,5\spadesuit$ from counting as a run).
            \item Aces only count as a low card (this eliminates $10\clubsuit,\textnormal{J}\diamondsuit,\textnormal{Q}\clubsuit,\textnormal{K}\spadesuit,\textnormal{A}\diamondsuit$ from counting as a run).
            \item Aces can now be used in the middle of runs (this allows $\textnormal{Q}\clubsuit,\textnormal{K}\spadesuit,\textnormal{A}\diamondsuit,2\clubsuit,3\diamondsuit$ to count as a run).
        \end{enumerate}

    \item For the following questions, we will use a modified deck. 
    \begin{enumerate}
        \item Suppose all of the cards with value 8 have been removed. How many ways can you obtain a straight? Keep in mind $5\clubsuit, 6\heartsuit, 7\diamondsuit, 9\heartsuit, 10\clubsuit$ is not a valid straight.
        \item Suppose one of the Joker cards has been added to the deck as a free card (a card that can take the place of any card). Compute the number of the following hands and provide a combinatorial proof:
        \begin{enumerate}
            \item Four of a kind 
            \item Flush 
            \item Full House
        \end{enumerate}
    \end{enumerate}

\end{enumerate}

%% file: 04-BinomialCoefficientsPascal/BinomialCoeffs.tex
\chapter{Pascal's triangle and the Binomial Theorem}\label{chap:binomcoefficients} \label{Cpascal}

The binomial coefficient $\binom{n}{k}$ counts the number of ways to choose $k$ objects from a set of $n$ objects (where order does not matter and repeats are not allowed).
In Theorem~\ref{Tupperleft} (and Lemma~\ref{Lformulabinom}), 
we proved that 
\[\binom{n}{k}=\frac{n!}{k!(n-k)!}.\] Binomial coefficients are extremely useful in combinatorics and arise in many different areas of mathematics and applications.   
In Section~\ref{Spascal}, we study Pascal's triangle and Pascal's recurrence for binomial coefficients, and describe applications to counting lattice paths.
In Section~\ref{sec:binomial-theorem}, we describe their role as coefficients in the \textit{Binomial Theorem}. 
Sections~\ref{sec:pascal} and \ref{Sidentity2} contain identities involving binomial coefficients.
In Sections~\ref{sec:anagrams} and \ref{Smultinomial}, we generalize to study multinomial coefficients and their applications to anagrams.

\section{Pascal's triangle} \label{Spascal}

Pascal's triangle, named after the French mathematician/philosopher Blaise Pascal (1623--1662),
 provides a way of quickly computing binomial coefficients for small values of $n$ and $k$.  The method is to construct an infinite triangular chart in which the entry in the $n$-th row and $k$-th (diagonal) column is  $\binom{n}{k}$, as follows.
 
\begin{center}
\includegraphics{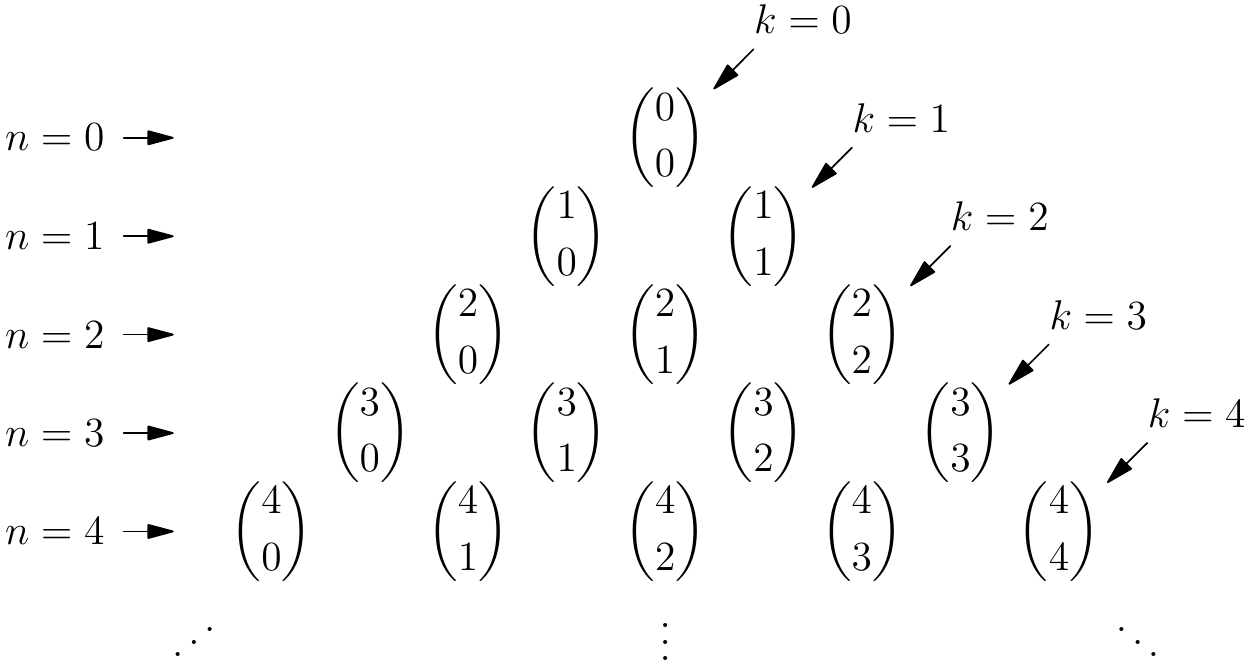}
\end{center}

The values in this triangle can be computed quickly using an identity known as \textbf{Pascal's Recurrence}.

\begin{theorem}[Pascal's Recurrence]\label{thm:Pascal}
Suppose $n$ is a natural number.
\begin{enumerate}
    \item[(a)] Then \[\binom{n}{0}=1 \ {\rm and} \  \binom{n}{n}=1.\] 
    \item[(b)]
    If $k$ is a natural number such that $0 \leq k < n$, then \[\binom{n}{k}+\binom{n}{k+1}=\binom{n+1}{k+1}.\]
\end{enumerate} 
\end{theorem}

Before proving Theorem~\ref{thm:Pascal}, we provide some more explanation about it.
Part (a) of Pascal's Recurrence says that the numbers on both the left and right sides of the triangle are all $1$, as seen below:  

\begin{center}
\begin{tabular}{rccccccccccccc}
$n=0$:&    &    &    &    &    &    &  1\\\noalign{\smallskip\smallskip}
$n=1$:&    &    &    &    &    &  1 &    &  1\\\noalign{\smallskip\smallskip}
$n=2$:&    &    &    &    &  1 &    &   &    &  1\\\noalign{\smallskip\smallskip}
$n=3$:&    &    &    &  1 &    &   &    &   &    &  1\\\noalign{\smallskip\smallskip}
$n=4$:&    &    &  1 &    &   &    &   &    &   &    &  1\\\noalign{\smallskip\smallskip}
$n=5$:&    &   1 &    &    &    &    &    &    &    &    &   &  1\\\noalign{\smallskip\smallskip}
$n=6$:& 1 &    &   &    &   &    &   &    &   &    &    &   & 1\\\noalign{\smallskip\smallskip}
\end{tabular}
\end{center}

Part (b) of Pascal's Recurrence says that the other numbers in the triangle are the sum of the two numbers just above them (to the left and right).  For instance, to find the middle entry in the $n=2$ row, we compute $2=1+1$. 

\begin{center}
\begin{tabular}{rccccccccccccc}
$n=0$:&    &    &    &    &    &    &  1\\\noalign{\smallskip\smallskip}
$n=1$:&    &    &    &    &    &  1 &    &  1 \\\noalign{\smallskip\smallskip}
$n=2$:&    &    &    &    &  1 &    &  2 &    &  1\\\noalign{\smallskip\smallskip}
$n=3$:&    &    &    &  1 &    &   &    &   &    &  1\\\noalign{\smallskip\smallskip}
$n=4$:&    &    &  1 &    &   &    &   &    &   &    &  1\\\noalign{\smallskip\smallskip}
$n=5$:&    &   1 &    &    &    &    &    &    &    &    &   &  1\\\noalign{\smallskip\smallskip}
$n=6$:& 1 &    &   &    &   &    &   &    &   &    &    &   & 1\\\noalign{\smallskip\smallskip}
\end{tabular}
\end{center}

We can then fill in each entry, for each row in succession, by adding the pair of numbers above it. For example, the $3$rd entry ($k=2$) in the $n=6$ row is $15=5+10$. Here are the first six rows of Pascal's triangle.

\begin{center}
\begin{tabular}{rccccccccccccc}
$n=0$:&    &    &    &    &    &    &  1\\\noalign{\smallskip\smallskip}
$n=1$:&    &    &    &    &    &  1 &    &  1\\\noalign{\smallskip\smallskip}
$n=2$:&    &    &    &    &  1 &    &  2 &    &  1\\\noalign{\smallskip\smallskip}
$n=3$:&    &    &    &  1 &    &  3 &    &  3 &    &  1\\\noalign{\smallskip\smallskip}
$n=4$:&    &    &  1 &    &  4 &    &  6 &    &  4 &    &  1\\\noalign{\smallskip\smallskip}
$n=5$:&    &   1 &    &   5 &    &  10  &    &  10  &    &  5  &   &  1\\\noalign{\smallskip\smallskip}
$n=6$:& 1 &    &  6 &    &  15 &    &  20 &    &  15 &    &  6  &   & 1\\\noalign{\smallskip\smallskip}
\end{tabular}
\end{center}

It is fast to compute the first rows of entries in Pascal's triangle recursively, which provides a useful alternative to the formula $\binom{n}{k}=\frac{n!}{k!(n-k)!}$.  

In the next lemma, we give another way to interpret Pascal's recurrence.
Imagine placing a ball at the top of this grid.  It can fall along the diagonal lines, moving down either to the left or to the right.
It can not fall up.

\begin{lemma} \label{Lpaths}
The entry $\binom{n}{k}$ counts the number of possible ways the ball could have traveled to that location!
\end{lemma}

\begin{proof}
To reach the $n$-th row,
the ball needs to fall $n$ steps down, either to the left or right.
To reach the spot with the entry 
$\binom{n}{k}$, the ball needs to fall to the right $k$ times.
The route is given by a sequence of length $n$, containing the letter $R$ (for right) $k$ times and containing the letter $L$ (for left) $n-k$ times.
The number of such sequences is 
$\binom{n}{k}$, since we need to choose which $k$ entries of the sequence are the letter $R$.
\end{proof}

Lemma~\ref{Lpaths} gives another way to interpret Pascal's recurrence.
Part (a) says that there is only one way that the ball can fall to reach a spot along the outer diagonals.
To reach a spot that is not on the outer diagonals, the ball fell from the row above, coming either from the left or from the right. 
Part (b) says that the number of ways the ball could have traveled to this spot is the sum of the ways it could have fallen to those two spots in the row above.

\begin{example} \label{Eroutes}
There is only one way for the ball to reach the entry $\binom{4}{0}$; it is given by the route $LLLL$.

There are four ways for the ball to reach the entry $\binom{4}{1}$;
they are given by the routes
$LLLR$, $LLRL$, $LRLL$, and $RLLL$.
\end{example}

\begin{example}
The blocks north of campus are on a grid system.  Starting at the corner of Howes and Laurel, the students need to walk two blocks north and three blocks west to reach the ice cream store. 
How many ways are there to do this (without traveling south or east)?  After rotating Pascal's triangle, we see the answer is $10$.
We can write out these paths as 
\begin{eqnarray*}
& NNWWW, NWNWW, NWWNW, NWWWN, WNNWW,& \\
&WNWNW, WNWWN, WWNNW, WWNWN, WWWNN.&
\end{eqnarray*}
\end{example}

Finally, we will prove Pascal's Recurrence algebraically.  For a more elegant combinatorial proof, see
Example~\ref{Epascal2}.

\begin{proof}[Proof of Theorem \ref{thm:Pascal}]
To show (a), note that $\binom{n}{0}=\frac{n!}{0!\cdot n!}=\frac{n!}{1\cdot n!}=1$ since $0!=1$ by Remark~\ref{rmk:zero-factorial}.  Similarly, $\binom{n}{n}=\frac{n!}{n!\cdot 0!}=1$.

For (b), we compute: 
\begin{align*}
\binom{n}{k}+\binom{n}{k+1}&=\frac{n!}{k!(n-k)!}+\frac{n!}{(k+1)!(n-k-1)!} \\
&=\frac{n!\cdot(k+1)}{(k+1)!(n-k)!}+\frac{n!\cdot (n-k)}{(k+1)!(n-k)!} \\
&=\frac{n!(k+1+n-k)}{(k+1)!(n-k)!} \\
&=\frac{n!(n+1)}{(k+1)!(n+1-(k+1))!} \\
&=\frac{(n+1)!}{(k+1)!(n-1-(k+1))!} \\
&=\binom{n+1}{k+1}.
\end{align*}
\end{proof}
\subsection*{Exercises}
\begin{enumerate}
\item Compute the rows $n=7$ and $n=8$ in Pascal's triangle.

\item Which entries in Pascal's triangle are odd?

\item Suppose $\binom{n}{k-1}=120$ and $\binom{n+1}{k}=330$. What is $\binom{n}{k}$?\\
\emph{Hint: The best way to solve this problem does not require you to find $n$ or $k$.}

\item Find all the routes the ball could fall to reach the entry $\binom{4}{2}$ in Pascal's triangle.

\item Starting at the corner of Howes and Laurel, how many ways are there to walk three blocks north and four blocks west (without traveling south or east)?

 \item How many paths of length $m+n$ are there from the lower left corner to the upper right corner in an $m\times n$ grid?  (Paths can only trace along edges of the squares of the grid.)
 
 \item \label{Exmodpbinom} If $p$ is a prime number and $1 \leq k \leq p-1$, explain why $\binom{p}{k} \equiv 0 \bmod p$.

\item  Starting at the corner of Howes and Laurel, how many ways are there to walk four blocks north and two blocks west (without traveling south or east)?													
\item You have 7 different hats and you want to loan some of these to a friend, who needs at least 2 hats.	How many ways are there to do this?

\item Remember how a ball can fall through the grid system of Pascal’s triangle. We suppose that the ball has an equal probability of falling down to the left or to the right. Compute the probability that the ball lands at the 3rd spot in the 5th row (with the outer diagonal being the case k = 0). Write your answer as a decimal number between 0 and 1.	

\item A ball falls through the grid system of Pascal's triangle and lands on the n=10th row. Write a 2-3 sentence explanation of why the ball has a $50\%$ probability of landing in a spot where k is even.													
\end{enumerate}
\section{The Binomial Theorem}\label{sec:binomial-theorem}

\begin{videobox}
\begin{minipage}{0.1\textwidth}
\href{https://www.youtube.com/watch?v=wFk-Z25-s7A}{\includegraphics[width=1cm]{video-clipart-2.png}}
\end{minipage}
\begin{minipage}{0.8\textwidth}
Click on the icon at left or the URL below for this section's short video lecture. \\\vspace{-0.2cm} \\ \href{https://www.youtube.com/watch?v=wFk-Z25-s7A}{https://www.youtube.com/watch?v=wFk-Z25-s7A}
\end{minipage}
\end{videobox}

One of the most important applications of Pascal's triangle and binomial coefficients is the expansion of a power of a binomial.  Starting with a binomial $x+y$, let us consider its first few powers:

\begin{align*}
(x+y)^0&=1,\\
(x+y)^1&=x+y,\\
(x+y)^2&=x^2+2 xy+y^2,\\
(x+y)^3&=x^3+3 x^2y+3  xy^2+y^3, \\
(x+y)^4&=x^4+4  x^3y+6  x^2y^2+4 xy^3+y^4. 
\end{align*}

The coefficients of these expansions are the numbers in Pascal's triangle!  The pattern becomes even more clear if we align them this way:

\begin{center}
\begin{tabular}{rccccccccccc}
$(x+y)^0$:&    &    &    &    &    &    &  1\\\noalign{\smallskip\smallskip}
$(x+y)^1$:&    &    &    &    &    &  $x$ &  +  &  y\\\noalign{\smallskip\smallskip}
$(x+y)^2$:&    &    &    &    &  $x^2$ &  +  &  $2xy$ & +   &  $y^2$\\\noalign{\smallskip\smallskip}
$(x+y)^3$:&    &    &    &  $x^3$ & +  &  $3x^2y$ &  +  &  $3xy^2$ &  +  &  $y^3$\\\noalign{\smallskip\smallskip}
$(x+y)^4$:&    &    &  $x^4$ & + &  $4x^3y$ & + &  $6x^2y^2$ & + &  $4xy^3$ & +  &  $y^4$\\\noalign{\smallskip\smallskip}
\end{tabular}
\end{center}
Indeed, if we ignore the $x$'s and $y$'s in the table above, the coefficients are precisely the numbers in Pascal's triangle.  This is why the numbers $\binom{n}{k}$ are called binomial coefficients.

\begin{remark}
Some students are familiar with the `foil' (first, outer, inner, last) method to compute the case $n=2$:
\[(x+y)^2 = x \cdot x + y \cdot x + x \cdot y + y \cdot y.\]
Later in this section, we will generalize this, in a way that could be called a `giant foil method'.
\end{remark}

We now can state the pattern of the coefficients as the \textbf{Binomial Theorem}, one of the oldest named theorems in mathematics, first discovered by the Persian 
mathematician and engineer
Al-Karaji $\sim$953--1029.

\begin{theorem}\label{Tbinomialthm}
(Binomial Theorem) Let $x$ and $y$ be variables.  Then the expansion of $(x+y)^n$ is
\[(x+y)^n=\binom{n}{0}x^n+\binom{n}{1}x^{n-1}y^1+\ldots+\binom{n}{k}x^{n-k}y^k+\ldots+\binom{n}{n-1}x^1y^{n-1}+\binom{n}{n}y^n.\]
\end{theorem}

\begin{proof}
Note
\[(x+y)^n=\underbrace{(x+y)\cdot(x+y)\cdot(x+y)\cdot\ldots\cdot(x+y)\cdot(x+y)}_{n \mbox{ times}}.\]
In expanding this product, we need to take one entry from each term, multiply these entries together, and then add all those products (`the giant foil method').
For each of the $n$ terms $(x+y)$, we have the option to choose either $x$ or $y$. Since there are $n$ terms, the product of these will be a monomial of the form $x^{n-k} y^k$ for some number $k$.  To produce the monomial $x^{n-k}y^k$ in the product, we must choose the variable $y$ for a total of $k$ times and then choose the variable $x$ for the remaining $n-k$ times.
There are $\binom{n}{k}$ ways to choose $k$ terms out of the $n$ terms.
So the coefficient of $x^{n-k}y^k$ is $\binom{n}{k}$.
\end{proof}

For example, if you expand $(x+y)^{10}$ and combine like terms, then the coefficient in front of $x^8y^2$ will be $\binom{10}{2}=45$. This comes from the fact that out of the 10 terms $x+y$ being multiplied together, we need to choose the variable $y$ from 2 of them (and the variable $x$ from the remaining 8 of them) in order to produce the monomial $x^8y^2$.

\begin{remark}
By combining the Binomial Theorem with Pascal's triangle, we can quickly expand powers of the sum of two things! For instance, by reading row $n=6$ from Pascal's triangle, we can immediately deduce that $$(x+y)^6=x^6+6x^5y+15x^4y^2+20x^3y^3+15x^2y^4+6xy^5+y^6.$$
By reading row $n=3$ from Pascal's triangle, we can deduce that
\begin{eqnarray*}
(a^2+5b)^3 & = & (a^2)^3 + 3 (a^2)^2 (5b) + 3 (a^2)(5b)^2 + (5b)^3 \\
& = & a^6 + 15 a^4  + 75 a^2b^2 + 125b^3.
\end{eqnarray*}
\end{remark}

\subsection*{Exercises}
\begin{enumerate}
			
\item Given a power of $x+y$ and a monomial in its expansion, find the coefficient on that term using the Binomial Theorem:
\begin{enumerate}
    \item $(x+y)^3$ and $x^2y$
    \item $(x+y)^6$ and $x^4y^2$
    \item $(x+y)^8$ and $x^3y^5$
\end{enumerate}


\item What is the coefficient of $x^5y^3$  in $(x+y)^8$?			
\item What is the coefficient of $x^2$ in $(x-1)^5$?			
		
\item What is the coefficient of $x^3$ in $(x-1)(x-2)(x-3)(x-4)$?			
\item What is the constant coefficient in $(x-1)(x-2)(x-3)(x-4)$?
\item Expand $(x+2)^3$ by substituting $y=2$ into the binomial theorem.

\item Expand $(x-1)^4$ by substituting $y=-1$ into the binomial theorem.

\item Expand $(a^2+3b)^4$ using the binomial theorem.
\item What is the coefficient of $a^4b^2$ in $(a^2+3b)^4$?	
\item For a real number $r$, 
find the coefficients of $(x-r)^n$, the polynomial whose only root is $r$.

\item The formulas of Vieta (1540-1603):
\begin{enumerate}
    \item Find formulas for $S_1$ and $S_2$ in terms of $r_1$ and $r_2$ if 
    \[(x-r_1)\cdot(x-r_2)=x^2 - S_1x + S_2.\]
    \item Find formulas for $S_1, S_2, S_3$ in terms of $r_1, r_2, r_3$ if 
    \[(x-r_1)\cdot(x-r_2)\cdot (x-r_3)=x^3 - S_1x^2 + S_2x -S_3.\]
    \item Without expanding the product on the left side of following equation, 
    make a conjecture about the 
    formulas for $S_1, S_2, S_3, S_4$ in terms of $r_1, r_2, r_3, r_4$ if 
    \[(x-r_1)\cdot(x-r_2)\cdot (x-r_3) \cdot(x-r_4) =x^4 - S_1x^3 + S_2x^2 -S_3x+S_4.\]

\end{enumerate}

\item
\begin{enumerate}
\item Show that:
\[(x+y)^2 \equiv x^2 + y^2 \bmod 2.\]
\[(x+y)^3 \equiv x^3 + y^3 \bmod 3.\]
\[(x+y)^5 \equiv x^5 + y^5 \bmod 5.\]

\item \label{Exmodpbinom2} If $p$ is a prime number and $1 \leq k \leq p-1$, then by Exercise~\ref{Exmodpbinom}, $\binom{p}{k} \equiv 0 \bmod p$. 
Use this to show that
        $(x+y)^p \equiv x^p+y^p \bmod p$.

\end{enumerate}

\end{enumerate}

\section{First identities in Pascal's triangle}\label{sec:pascal}

Some surprising and beautiful patterns appear in Pascal's triangle.

\begin{center}
\begin{tabular}{rccccccccccccc}
$n=0$:&    &    &    &    &    &    &   1 \\\noalign{\smallskip\smallskip}
$n=1$:&    &    &    &    &    &  1 &  &  1\\\noalign{\smallskip\smallskip}
$n=2$:&    &    &    &    &  1 &    &  2 &    &  1\\\noalign{\smallskip\smallskip}
$n=3$:&    &    &    &  1 &    &  3 &    &  3 &    &  1\\\noalign{\smallskip\smallskip}
$n=4$:&    &    &  1 &    &  4 &    &  6 &    &  4 &    &  1\\\noalign{\smallskip\smallskip}
$n=5$:&    &   1 &    &   5 &    &  10  &    &  10  &    &  5  &   &  1\\\noalign{\smallskip\smallskip}
$n=6$:& 1 &    &  6 &    &  15 &    &  20 &    &  15 &    &  6  &   & 1\\\noalign{\smallskip\smallskip}
\end{tabular}
\end{center}

The first two patterns are easy to state.

\begin{itemize}
    \item \textbf{Symmetry.}  Recall that $\binom{n}{k}=\binom{n}{n-k}$ by Corollary \ref{cor:symmetry}.  This is illustrated by the reflectional symmetry about the vertical axis (through $1, 2, 6, 20,\ldots $) in Pascal's triangle.
    \item \textbf{Sums of rows are powers of 2.} For instance, the sum of the entries in the $n=4$ row of Pascal's triangle is $1+4+6+4+1=16=2^4$.  In general, the sum of the entries in row $n$ is $2^n$.
\end{itemize}

We can prove the second pattern using the Binomial Theorem.  We first re-state it using binomial coefficient notation.

\begin{proposition} \label{Ppattern2}
(Sums of rows)
If $n$ is a natural number, then \[\binom{n}{0}+\binom{n}{1}+\binom{n}{2}+\ldots+\binom{n}{n}=2^n.\]
\end{proposition}

\begin{proof}
By setting $x=1$ and $y=1$ in the binomial theorem, we see that 
\begin{align*}
2^n&=(1+1)^n\\
&=\binom{n}{0}1^n+\binom{n}{1}1^{n-1}1^1+\binom{n}{2}1^{n-2}1^2+\ldots+\binom{n}{n}1^n\\
&=\binom{n}{0}+\binom{n}{1}+\binom{n}{2}+\ldots+\binom{n}{n}.
\end{align*}
\end{proof}

Building on Lemma~\ref{Lpaths} and Example~\ref{Eroutes}, remember how a ball can fall through the grid system of Pascal's triangle.  
Unless stated otherwise, we suppose that the ball has an equal probability of falling down to the left or to the right.
Here is a nice interpretation of the second pattern in terms of the falling ball.

\begin{lemma}
Suppose that the probability that the ball falls to the left is the same as the probability that the ball falls to the right.
Then the probability that the ball passes through the location labeled by the binomial coefficient
$\binom{n}{k}$ is $\binom{n}{k}/2^n$.  
\end{lemma}

\begin{proof}
If the probability that the ball falls to the left is the same as the probability that the ball falls to the right, then the probability of falling each way is $1/2$.
After falling $n$ rows, the probability of each route is $1/2^n$.
The probability that the ball passes through the location labeled by the binomial coefficient
$\binom{n}{k}$ is
the number of routes it could have taken to that location divided by $2^n$; this gives a probability of 
$\binom{n}{k}/2^n$.  
\end{proof}

\begin{remark}
If the ball falls $n$ steps, then it will land at one of the spots in the $n$th row.  So the sum of the probabilities of it landing at all the spots is $1$. 
By multiplying each probability by $2^n$, we see that this reproduces the pattern given in Proposition \ref{Ppattern2}.
\end{remark}

\begin{example} \label{Epathprob}
Building on Example~\ref{Eroutes}, 
there is only one path for the ball to reach the entry $\binom{4}{0}$; and the probability that the ball passes through this location is $1/16$; there are four paths for the ball to reach the entry $\binom{4}{1}$ and the  probability that the ball passes through this location is $4/16 =1/4$.
\end{example}

Here is a third pattern in Pascal's triangle.

\begin{itemize}
    \item \textbf{Alternating sums of the rows are $0$.}  For instance, in the $n=4$ row, the alternating sum of the entries is $1-4+6-4+1=0$.  In general, the alternating sum of the entries in row $n$ is zero.
\end{itemize}

The third pattern can also be proven using the Binomial Theorem. We first re-state it using  
 binomial coefficient notation.

\begin{proposition} \label{Ppattern3}
(Alternating sums of rows)
If $n$ is a positive natural number, then \[\binom{n}{0}-\binom{n}{1}+\binom{n}{2}-\binom{n}{3}+\ldots\pm\binom{n}{n}=0.\]
(The sign alternates between $+$ and $-$ in the sum.  The meaning of the coefficient $\pm$ is that 
the last term is positive if $n$ is even and the last term is negative 
if $n$ is odd.)
\end{proposition}

\begin{proof}
By setting $x=1$ and $y=-1$ in the binomial theorem, we see that
\begin{align*}
0&=(1+(-1))^n\\
&=\binom{n}{0}1^n+\binom{n}{1}1^{n-1}(-1)^1+\binom{n}{2}1^{n-2}(-1)^2+\binom{n}{3}1^{n-3}(-1)^3+\ldots+\binom{n}{n}(-1)^n\\
&=\binom{n}{0}-\binom{n}{1}+\binom{n}{2}-\binom{n}{3}+\ldots\pm\binom{n}{n}.
\end{align*}
\end{proof}

We will give another proof of Proposition~\ref{Ppattern3} in Proposition~\ref{Paltv2}.

\subsection*{Exercises}

\begin{enumerate}
\item You have $14$ hats of different colors and you want to loan some of these to a friend.
Use the results in this section to answer the following problems.
\begin{enumerate}
    \item In how many ways can you do this if your friend needs at least two hats?
\item In how many ways can you do this if your friend needs an even number (possibly $0$) of hats?
\end{enumerate}

\item For $n=4,5,6$, find the sum of all the entries in the $n$th row of Pascal's triangle, excluding the two entries on the outer diagonals and the two entries adjacent to those.  Find a formula for this sum which works for all $n \geq 4$.
   
\item  For $0 \leq k \leq 5$, compute the probability that the ball lands at the $k$th spot in the $5$th row (with the outer diagonal being the case $k=0$).

\item Repeat the previous problem when the probability of falling down to the left is $1/3$ and the probability of falling down to the right is $2/3$.

\item Label the entries of the $n$th row of Pascal's triangle by the numbers $k=0$ to $k=n$.
Assume the probability of the ball falling left or right is the same. For $n \geq 1$, in this problem, we will figure out the probability that the ball lands on an entry where $k$ is even.
\begin{enumerate}
    \item Compute this probability when $n = 5$ and when $n=6$.
    \item Make a conjecture for what this probability is in general.
    \item Explain why your conjecture is true using the symmetry pattern and Proposition \ref{Ppattern2}.
\end{enumerate}

\end{enumerate}

\section{Additional identities in Pascal's triangle} \label{Sidentity2}

There are other patterns in Pascal's triangle which are more subtle.  For the next pattern, consider the sum of the squares of the entries in each row:
\begin{center}
\begin{tabular}{rcccccccccl}
$n=0$:&    &    &    &    &  $1^2$ &&&&& $=1$\\\noalign{\smallskip\smallskip}
$n=1$:&    &    &    &  $1^2$ &    &  $+1^2$ &&&& $=2$\\\noalign{\smallskip\smallskip}
$n=2$:&    &    &  $1^2$ &    &  $+2^2$ &    &  $+1^2$ &&& $=6$\\\noalign{\smallskip\smallskip}
$n=3$:&    &  $1^2$ &    &  $+3^2$ &    &  $+3^2$ &    &  $+1^2$ && $=20$\\\noalign{\smallskip\smallskip}
$n=4$:&  $1^2$ &    &  $+4^2$ &    &  $+6^2$ &    &  $+4^2$ &    &  $+1^2$ & $=70$\\\noalign{\smallskip\smallskip}
\end{tabular}
\end{center}
These sums ($1,2,6,20,70,\ldots$) also appear in Pascal's triangle, as the middle entry in every row where $n$ is even.  This gives rise to pattern 4.

\begin{proposition} \label{Ppattern4}
    (Sum of squares) If $n$ is a natural number, 
    then \[\binom{n}{0}^2+\binom{n}{1}^2+\ldots+\binom{n}{n-1}^2+\binom{n}{n}^2=\binom{2n}{n}.\]
\end{proposition}

We will prove Theorem~\ref{Ppattern4} in Proposition~\ref{Pproofsumsquares}. 

The fifth pattern is known as the \textbf{hockey stick identity}, because the set of binomial coefficients involved in the pattern traces out the shape of a hockey stick in Pascal's triangle.

\begin{proposition} \label{Ppattern5}
(Hockey stick identity)
Suppose $n$ is a positive integer and $m$ is a non-negative integer.

Version 1: \[\binom{n}{0}+\binom{n+1}{1}+\binom{n+2}{2}+\cdots+\binom{n+m}{m}=\binom{n+m+1}{m}.\]
    
Version 2: \[\binom{n}{n}+\binom{n+1}{n}+\binom{n+2}{n}+\cdots+\binom{n+m}{n}=\binom{n+m+1}{n+1}.\]
\end{proposition}

The second version of the hockey stick identity is obtained from the first by applying the first pattern, symmetry, to each binomial coefficient. We will prove Theorem~\ref{Ppattern5} in Proposition~\ref{Phockey}.

\begin{example}
We illustrate the first hockey stick identity.
\begin{itemize}
\item If $n=2,m=0$, then $1=1$.
\item If $n=2,m=1$, then $1+3=4$.
\item If $n=2,m=2$, then $1+3+6=10$.
\item If $n=2,m=3$, then $1+3+6+10=20$.
\item If $n=2,m=4$, then $\textcolor{blue}{1+3+6+10+15}=\textcolor{red}{35}$.
\end{itemize}
\end{example}

We can see version 1 of the hockey stick identity when $n=2$ and $m=4$ traced out in circles in the following image:
\begin{center}
\begin{tabular}{rccccccccccccccc}
$n=0$:&   &    &    &    &    &    &    &  1\\\noalign{\smallskip\smallskip}
$n=1$:&   &    &    &    &    &    &  1 &    &  1\\\noalign{\smallskip\smallskip}
$n=2$:&   &    &    &    &    &  \textcolor{blue}{\circleit{1}} &    &  2 &    &  1\\\noalign{\smallskip\smallskip}
$n=3$:&   &    &    &    &  1 &    &  \textcolor{blue}{\circleit{3}} &    &  3 &    &  1\\\noalign{\smallskip\smallskip}
$n=4$:&   &    &    &  1 &    &  4 &    &  \textcolor{blue}{\circleit{6}} &    &  4 &    &  1\\\noalign{\smallskip\smallskip}
$n=5$:&   &    &   1 &    &   5 &    &  10  &    &  \textcolor{blue}{\circleit{10}}  &    &  5  &   &  1\\\noalign{\smallskip\smallskip}
$n=6$:&   & 1 &    &  6 &    &  15 &    &  20 &    &  \textcolor{blue}{\circleit{15}} &    &  6  &   & 1\\\noalign{\smallskip\smallskip}
$n=7$:&  1 &   &  7  &    &  21  &   &  35  &   &  \textcolor{red}{\circleit{35}}  &   &  21  &    & 7  &   &  1\\\noalign{\smallskip\smallskip}
\end{tabular}
\end{center}

\section*{Exercises}

\begin{enumerate}
    \item Draw the example of version 2 of the hockey stick identity on Pascal's triangle when $n=2$ and $m=4$.

\item When $n=2$ and $m=4$, 
explain why the hockey stick identity version 1 is true using Lemma~\ref{Lpaths}.  Specifically, separate the $35$ paths which end at the binomial coefficient $\binom{7}{4}$ into sets of paths of sizes $15, 10, 6, 3, 1$ in a natural way.

\item (Application to probability theory) 
The \textit{chi-square distribution} of $x_0, \ldots, x_n$ is $Q_n=x_0^2 + \cdots x_n^2$.
\begin{enumerate}
    \item For $n=1,2,3,4,5$, use
    Proposition~\ref{Ppattern4} to compute the chi-square distribution $Q_n$ of the binomial coefficients in the $n$th row of Pascal's triangle.
    \item As $n$ gets larger, what happens to the chi-square distribution $Q_n$?
    \item What is 
    $\lim\limits_{n \to \infty} Q_n$?  Explain your answer.
    \end{enumerate}

\end{enumerate}

\section{Counting anagrams with multinomial coefficients}\label{sec:anagrams}

The next two sections are about \textit{multinomial coefficients}, a generalization of binomial coefficients.

\begin{videobox}
\begin{minipage}{0.1\textwidth}
\href{https://www.youtube.com/watch?v=YAO-1ugO3TI}{\includegraphics[width=1cm]{video-clipart-2.png}}
\end{minipage}
\begin{minipage}{0.8\textwidth}
Click on the icon at left or the URL below for this section's short video lecture. \\\vspace{-0.2cm} \\ \href{https://www.youtube.com/watch?v=YAO-1ugO3TI}{https://www.youtube.com/watch?v=YAO-1ugO3TI}
\end{minipage}
\end{videobox}

We start with one application of multinomial coefficients, which is to count anagrams.

\begin{definition}
An \textbf{anagram} is a rearrangement of the letters in a word.  
\end{definition}

An anagram is different from a permutation, because some letters can occur in a word multiple times.

\begin{example}
For example, let's consider the permutations of the word ZOO. Typically, the number of permutations of three letters is $3!=6$. 
However, two of the letters in ZOO are the same. After we label the two $O$'s as $O_1$ and $O_2$, the six permutations are:
\begin{center}
\textcolor{red}{ZO$_1$O$_2$} \hspace{5mm} \textcolor{red}{ZO$_2$O$_1$}

\textcolor{blue}{O$_1$ZO$_2$}\hspace{5mm}
\textcolor{blue}{O$_2$ZO$_1$}

\textcolor{brown}{O$_1$O$_2$Z}\hspace{5mm}
\textcolor{brown}{O$_2$O$_1$Z}.
\end{center}
However, the number of anagrams is half of $6$, since each anagram of ZOO appears $2$ times on this list.  Recognizing that each row gives a different anagram, we can list the $6/2=3$ anagrams of the word ZOO. They are: ZOO, OZO, and OOZ.
\end{example}

\begin{example}
The word SASSY has $20$ anagrams.
Typically, the number of permutations of five letters is $5!=120$.  In the list of these permutations, each anagram is written $3!=6$ times, corresponding to the rearrangement of the three S's.  So the number of anagrams is $5!/3! = 20$.
\end{example}

\begin{example} \label{EanagramFC}
How many anagrams are there for FORTCOLLINS?
One approach is to say that there are 11 letters in total, but two of these are L and two of these are O.  
So in the list of $11!$ permutations of the letters, each word is written four times (the three extra times correspond to switching the locations of the two L's, switching the locations of the two O's, and switching both the two L's and the two O's). Hence the total number of anagrams is
$11!/4$.
\end{example}

\begin{example} \label{Ebiandmult}
Let $W$ be a binary string of length $n$, having $k$ zeros and $n-k$ ones.  Then the number of anagrams of $W$ is $\binom{n}{k}$ because the anagram is determined exactly by the location of the $k$ zeros.
\end{example}

The same ideas from the previous examples lead to the following definition and theorem:

\begin{definition}
Let $n$ be a natural number.
Let $k_1, \ldots, k_m$ be $m$ positive natural numbers such that $$k_1 + \cdots + k_m = n.$$
Let $W$ be a word with $n$ letters of $m$ different types, in which the letter of type $i$ appears $k_i$ times, for $1 \leq i \leq m$.
The multinomial coefficient $$\binom{n}{k_1,\ldots,k_m}$$
is defined to be
the number of anagrams of $W$.
\end{definition}

\begin{theorem}\label{thm:anagrams}
The multinomial coefficient has the following formula: 
\begin{equation} \label{Emulti}
\binom{n}{k_1,\ldots,k_m}=\frac{n!}{k_1!\cdot k_2!\cdot\ldots\cdot k_m!}.
\end{equation}
\end{theorem}

We prove Theorem~\ref{thm:anagrams} in Section~\ref{Smultinomial}.

\begin{example}
How many anagrams are there for the word MISSISSIPPI?
In this case $n=11$, $m=4$, 
$k_1=4$ (the I's), 
$k_2=1$ (the M), $k_3=2$ (the P's) and $k_4=4$ (the S's).
So the answer is
$11!/(4!\cdot 1! \cdot 2! \cdot 4!)$.
\end{example}

\subsection*{Exercises}

\begin{enumerate}
\item 
How many anagrams does the word BANANA have?
For example, NABNAA and NANABA are two anagrams of BANANA.
    
\item How many ways can you rearrange the letters in the word MOON?												
    \item 
    How many anagrams does the word HORSETOOTH have?
    
    \item How many anagrams does the word TENNESSEE have?
   
    \item Consider all anagrams of the word SPEEDED. 
    How many ways are there for you to make a ranked list of your top 5 favorite anagrams, numbered \#1 to \#5, of SPEEDED?

    \item How many paths are there from $(0,0,0)$ to $(3,4,5)$ in three dimensional space, using only unit steps in directions $(1,0,0)$, $(0,1,0)$, and $(0,0,1)$?
    
    \item
Suppose $m=2$.  
Show that $\binom{n}{k_1,k_2}$ is the familiar binomial coefficient $\binom{n}{k_1}$ and show that
the formula in \eqref{Emulti} is the same as the formula 
$\frac{n!}{k_1!(n-k_1)!}$ 
found 
in Lemma~\ref{Lformulabinom}
and Theorem~\ref{Tupperleft}.

\end{enumerate}

\section{Ordered set partitions and the multinomial theorem} \label{Smultinomial}

Recall that the binomial coefficient $\binom{n}{k}$ counts the number of ways to choose a subset of size $k$ from a set $S$ of size $n$.  A different way to think of this is that we are separating the set $S$ into two parts: the $k$ elements that are chosen, and the $n-k$ elements that are not chosen. What happens if we want to separate the set $S$ into more than two parts?  

\begin{example}\label{ex:fruit4.4}
  Suppose there are $7$ different fruits in your refrigerator.  You want to eat two fruits on Monday, three on Tuesday (since that is your workout day), and two on Wednesday.  In how many different ways can you choose which fruits to eat on each day?
  
  One way to solve this problem is to successively count possibilities for each day.  On Monday, there are $\binom{7}{2}$ ways to choose which two fruits you eat.  On Tuesday, there are $5$ fruits left, so there are $\binom{5}{3}$ ways to choose which three fruits to eat.
  On Wednesday, there are $2$ fruits left, so there is only $\binom{2}{2}=1$ way to choose the remaining fruit.  So the total number of possibilities is $$\binom{7}{2}\cdot \binom{5}{3}\cdot \binom{2}{2}=\frac{7!}{2!\cdot 5!}\cdot \frac{5!}{3!\cdot 2!}\cdot\frac{2!}{2!\cdot 0!}.$$  Note that many of these factorials cancel, and we can express the result more simply as $$\frac{7!}{2!3!2!}=210.$$
\end{example}

In the next definition, 
we generalize Definition~\ref{Dsetpartition}, by replacing a set partition by an \emph{ordered set partition}.

\begin{definition}
An \defn{ordered set partition} of a set $S$ is a collection of non-empty subsets $B_1,\ldots,B_m$ of $S$, 
such that the intersection of any two of them is trivial and 
such that the union of all of them is $S$.
\end{definition}  

In other words, the ordered set partition separates the elements of $S$ into subsets $B_1,\ldots,B_m$ with no overlap.  The sets $B_1,\ldots,B_m$ are called the \textbf{blocks} of the ordered set partition.

\begin{theorem}\label{thm:osp}
The number of ordered set partitions of an $n$-element set $S$ into blocks of sizes $k_1,k_2,\ldots,k_m$ is the multinomial coefficient $$\binom{n}{k_1,\ldots,k_m}=\frac{n!}{k_1!\cdot k_2! \cdots k_m!}.$$ 
\end{theorem}

\begin{proof}
We count the number of ordered set partitions of a set $S$ of size $n$ into blocks of sizes $k_1,k_2,\ldots,k_m$ as follows.  There are $\binom{n}{k_1}=\frac{n!}{k_1!(n-k_1)!}$ ways to choose the first block $B_1$.  Out of the remaining $n-k_1$ elements, there are $\binom{n-k_1}{k_2}=\frac{(n-k_1)!}{k_2!(n-k_1-k_2)}$ ways to choose the next block $B_2$.  If $m\geq 3$, there are $\binom{n-k_1-k_2}{k_3}=\frac{(n-k_1-k_2)!}{k_3!(n-k_1-k_2-k_3)!}$ ways to choose the third block $B_3$, and so on.  Multiplying these together, all the factorials cancel except for $n!$ in the numerator and the terms $k_1!, \ldots, k_m!$ in the denominator.  The result follows.
\end{proof}

\begin{example}
In Example \ref{ex:fruit4.4}, the set $S$ is the set of fruits in your refrigerator, $m=3$, and $B_1,B_2,B_3$ are the sets of fruit you will eat on each day (in order). 
Also $n=7$, $k_1=2$, $k_2=3$, and $k_3=2$.
So the number of ordered set partitions is $\binom{7}{2,3,2}=210$.
\end{example}

\begin{example}
Suppose we have 11 different presents, and we would like to give 4 to person $a$, then 1 to person $b$, then 2 to person $c$, then 4 to person $d$.  The number of ways we can do this is $$\binom{11}{4,1,2,4}=\frac{11!}{4!1!2!4!}=34650.$$
\end{example}

In the next remark, we compare
anagrams with ordered set partitions.  We will use this to prove Theorem~\ref{thm:anagrams}.

\begin{remark} \label{Ranblock}
Suppose $n=k_1 + \cdots + k_m$.
Suppose $W$ is a word of length $n$ having $k_i$ copies of the letter $a_i$, for $1 \leq i \leq m$, where the letters are taken in alphabetical order.
Anagrams of $W$ are 
 directly related to ordered set partitions of the set 
 $S=\{1, \ldots, n\}$.  Here is how that works.
 
 Given an anagram of $W$, we can make an ordered block partition as follows: for $1 \leq i \leq m$, find the positions of the $i$th letter and put those numbers in the block $B_i$.
 For example, consider the anagram IMSSSPIIPSI (of MISSISSIPPI): the I's are in positions 1,7,8, and 11, the M is in position 2; the P's are in positions 5 and 9; and the S's are in positions 3,4,5, and 10.  From this, we make the ordered set partition with blocks $B_1=\{1,7,8,11\}$, 
 $B_2=\{2\}$, $B_3=\{5,9\}$,
 and $B_4=\{3,4,5,10\}$. 
 Conversely, given an ordered set partition with blocks $B_i$ of size $k_i$ for $1 \leq i \leq m$, 
 we can make an anagram of the word $W$, by putting the $i$th letter in the positions numbered by the elements in the block $B_i$.  For example, the blocks $B_1=\{2, 5, 8,11\}$, $B_2=\{1\}$, $B_3=\{9,10\}$ and $B_4=\{3,4,6,7\}$ gives the original spelling of MISSISSIPPI.
\end{remark}

\begin{proof} (Proof of Theorem~\ref{thm:anagrams})
By Remark~\ref{Ranblock}, the number of anagrams of $W$ is the same as the number of ordered set partitions into blocks of sizes $B_1, \ldots, B_m$.  
By Theorem~\ref{thm:osp}, we obtain the desired formula for the number of these.
\end{proof}


The next example illustrates the difference between an ordered set partition and a set partition.

\begin{example}
Suppose there are 12 students in a class.
\begin{enumerate}
\item 
How many ways can they be separated into $4$ groups of $3$ who will be assigned to present homework solutions on Monday, Tuesday, Wednesday, and Thursday?

\item How many ways can they be separated into $4$ groups of $3$?
\end{enumerate}
\end{example}

\begin{answer}
\begin{enumerate}
    \item Each assignment is an ordered set partition of the 12 students into blocks $B_1,\ldots, B_4$ of size three (speaking on Monday, Tuesday, Wednesday, Thursday respectively).  So the answer is 
    $$\binom{12}{3,3,3,3} = 
    \frac{12!}{3! \cdot 3! \cdot 3! \cdot 3!}.$$
    \item 
    This is a set partition since the blocks of students are not assigned to different days of the week.  There are $4!$ ways to order the $4$ blocks.
    So we divide the answer to part (1) by $4!$.
\end{enumerate}

\end{answer}

As a final application, we see what happens if we take a power of a sum of more than two variables.  The answer to this question is why the expression $\binom{n}{k_1,\ldots,k_n}$ is often called a \emph{multinomial coefficient}.

\begin{example} The third power of $x+y+z$ is
$$(x+y+z)^3=x^3+y^3+z^3+3x^2y+3xy^2+3x^2z+3xz^2+3y^2z+3yz^2+6xyz.$$  The coefficients of this expansion are all examples of multinomial coefficients.
\end{example}

Here is the \textbf{Multinomial Theorem}, which directly generalizes the Binomial Theorem from Theorem~\ref{Tbinomialthm}. 

\begin{theorem}
Consider the expansion of $(x_1+x_2+\cdots+x_m)^n$.  The coefficient of $x_1^{k_1}\cdots x_m^{k_m}$ is the multinomial coefficient $\binom{n}{k_1,\ldots,k_m}$.  In other words,  $$(x_1+x_2+\cdots+x_m)^n=\sum_{k_1+\ldots+k_m=n} \binom{n}{k_1,\cdots,k_m}x_1^{k_1}\cdots x_m^{k_m}.$$
\end{theorem}

For example, the term $3x^2y$ in the above expansion of $(x+y+z)^3$ can be written as $3x^2y^1z^0$, which does indeed agree with the multinomial coefficient $$\binom{3}{2,1,0}=3.$$

The following proof is quite difficult and it is alright for you to come back to it after gaining more experience with proofs.

\begin{proof}
We can count how many ways to obtain the monomial $x_1^{k_1}\cdots x_m^{k_m}$ by multiplying terms in the expansion as follows.  
We label the factors of $(x_1+x_2+\cdots+x_m)^n$ by $F_1, \ldots, F_n$. (Each factor equals $x_1+x_2+\cdots+x_m$ but we think of them as being different since one of them is first, second, etc).
Let $S=\{F_1, \ldots, F_n\}$ be the set of the $n$ factors. 
To make the monomial $x_1^{k_1}\cdots x_m^{k_m}$ in the product, we need to choose the variable $x_1$ from $k_1$ of the factors, then choose the variable $x_2$ from $k_2$ of the factors, and so on.  This forms an ordered set partition of $S$ with sizes $k_1,k_2,\ldots,k_m$, and so the result follows by Theorem \ref{thm:osp}.  
\end{proof}

\subsection*{Exercises}

\begin{enumerate}

\item Find the number of ways to give 20 pillows to 12 students $a$ through $l$ if the pillows are all different and we give 2 pillows to each of students $a,b,c,d,e,f,g,h$ and 1 pillow to each of students $i,j,k,l$.

\item How many ways are there to assign the 12 months (Jan-Dec) to three nurses $a$, $b$, and $c$ so that each worker gets four months of ``on call" duty?

\item Five siblings decide that one of them needs to call their mother every day of the year.
How many ways are there to distribute the $365=5\cdot 73$ days of the year among themselves so that each of them makes $73$ calls?  Write out the formula, but do not figure out the actual number of ways.

\item Nine friends are going to a math conference!  The friends will share three hotel rooms, where the first hotel room fits 4 people, the second fits 3 people, and the third fits 2 people. They are all staying in the hotel for 5 nights and are allowed to rearrange rooms every night.  (Repeating the same room assignments for different nights is allowed). In how many different ways can the 9 friends create a 5-night schedule of room assignments?
Write out the formula, but do not figure out the actual number of ways.

\item You have eight different decorative Halloween stickers (a witch, a ghost, a pumpkin, a skeleton, a vampire, a mummy, a bat, and a black cat) and four windows to decorate them with. In how many ways can you put two stickers on each window?

\item How many ways can you put 16 differently-colored billiard balls into the six different pockets of a pool table so that each corner pocket gets 3 balls and the two middle pockets get 2 each?																	
\end{enumerate}

\newpage
\section{Investigation: Combinatorial problems on a chessboard}

The problems in this section take place on a chess board.
If you know the basics of chess pieces, you can skip the next remark.

\begin{remark}
A chess board is an $8\times 8$ board tiled in $1\times 1$ squares, colored black and white, so that no black square shares an edge with a white square (and vice-versa).

Chess is a game for two players which takes place on this board.  Each player starts the game with pieces: 1 king, 1 queen, 2 rooks, 2 bishops, 2 knights, and 8 pawns.  One player's pieces are gold and the other's are silver.  Any game piece can be placed on any square of the board, regardless of material and color. 

Each type of piece can move in different ways.
During a turn, a rook can move to any square in the same row or the same column
that it started the turn on.  
During a turn, a queen can move to any square in the same row, column, or two diagonals 
that it started the turn on. 
Two rooks of different materials are \textit{attacking} each other if one of them can move to the square of the other in one move.  Similarly, 
two queens of different colors are \textit{attacking} each other if one of them can move to the square of the other in one move. 
\end{remark}

In some of the problems below,
we are working with rooks made from different materials, with an unlimited supply of rooks made from each material.
To make everything in the following problems precise, it is also important to note that the chess board is glued to the table so it cannot be turned over or rotated.

\begin{enumerate}
\item How many ways are there to place two gold rooks and two silver rocks on the board (with no other game pieces):
\begin{enumerate}
    \item so that no square has more than one rook on it?
\item so that no square has more than one rook on it and no gold rook is attacking a silver rook?
\end{enumerate}

\item 
What is the largest number of rooks (of the same material) that can be placed on the board so that no two are in the same row or in the same column?  How many ways can these rooks be placed on the board?


\item Show that it is possible to put 16 rooks (four gold, four silver, four titanium, and four diamond) on the board, so that no pair of rooks with different materials is attacking each other.
How many ways can this be done?

\item (The peaceful rooks problem) What is the largest number $m$ such that $m$ gold rooks and $m$ silver rooks can be placed on the board so that no gold rook is attacking a silver rook?
How many ways can this be done?

\item Suppose there is one queen on the board and no other pieces.
\begin{enumerate}
    \item Find a position for the queen where the number of squares it can move to is as small as possible.
    Which position is this and 
    how many squares can it move to?
    \item Find a position for the queen where the number of squares it can move to is as big as possible.
    Which position is this and 
    how many squares can it move to?
    \end{enumerate}
    
    \item In this problem, we change the size of the board to $n \times n$ and put $n$ gold queens on the board so that no two (of the same material) are attacking.
    \begin{enumerate}
        \item Show this is not possible if $n=2$ or $n=3$.
        \item Find a way to do this if $n=4$.
        \item Find a way to do this if $n=5$.
        \item Look up the 8 queens puzzle.  Explain what it is and draw a solution.
    \end{enumerate}
    
    \item In this problem, we change the size of the board to $n \times n$.  Let $m$ be the maximum number such that $m$ gold queens and $m$ silver queens can be placed on the $n \times n$ board so that no two of different materials are attacking.
    For example, when $n=3$, then $m=1$, because the gold queen can be placed in the corner and the silver queen can be placed in the middle of one of the opposite sides; however it is not possible to place $2$ gold queens and $2$ silver queens on a $3\times 3$ board without them attacking each other.
    \begin{enumerate}
    \item When $n=4$, show that $m \geq 2$ by finding a way to put $2$ gold queens and $2$ silver queens on the board.
    \item When $n=5$, show that $m\geq4$ by finding a way to put $4$ gold queens and $4$ silver queens on the board.
    \item Look up the peaceable queens problem.  Find out the value of $m$ when $n=8$.
\end{enumerate}
\end{enumerate}

%% file: 05-ProofTechniques/ProofTechniques.tex
\newpage

\chapter{Proof Techniques in Combinatorics}\label{chap:proofs}

In this chapter, we explain how to prove statements in combinatorics. After explaining why proofs are useful, we focus on proof techniques that are often useful in combinatorics: induction, counting in two ways, bijective proofs, proof by contradiction, and the pigeonhole principle.

\begin{videobox}
\begin{minipage}{0.1\textwidth}
\href{https://www.youtube.com/watch?v=TFNwxluopek}{\includegraphics[width=1cm]{video-clipart-2.png}}
\end{minipage}
\begin{minipage}{0.8\textwidth}
Click on the icon at left or the URL below for this section's short video lecture. \\\vspace{-0.2cm} \\ \href{https://www.youtube.com/watch?v=TFNwxluopek}{https://www.youtube.com/watch?v=TFNwxluopek}
\end{minipage}
\end{videobox}

\section{Why are proofs necessary?}

Writing proofs is not easy.  Sometimes students feel that professors assign proofs the same way that a drill sergeant assigns push-ups:
\emph{after 50 of these, you'll be stronger!}.
At a basic level, proofs are needed to show that a statement is true without a doubt.
Beyond this, proofs are valuable because they explain the reasons why statements are true. 

 Many people find proofs to be beautiful and powerful.  Often mathematicians have favorite proofs, which illustrate core topics in their research areas. Some mathematicians view proofs as the foundation of mathematics, the solid base needed to build fantastic structures.  Other mathematicians view proofs as a living language, which connects the thoughts from ancient civilizations to the  thoughts in modern times.   

In this chapter, we will explain some proof techniques. 
We cannot promise that you will love proofs by the end of the chapter, but hopefully writing a proof will seem more manageable.  Maybe by the end of this class, you will have a favorite proof too.

To get started, we will share some examples that show numerical data is sometimes misleading and share some history about famous proofs in combinatorics and number theory.

\begin{example}
 How many positive divisors does $n!$ have? (This is denoted by $\sigma_0(n!)$, see Section~\ref{Ssigma0}.)
 
 To get started, let's write out the first few cases.  If $n=1$, then $n!$ is just $1$, which only has one divisor ($1$ itself).  If $n=2$, then $2!=2$, which has $2$ divisors ($1$ and $2$).  If $n=3$, then $3!=6$ which has $4$ divisors ($1,2,3$, and $6$).  With a little more effort, we find that $4!=24$ has $8$ divisors and $5!=125$ has $16$ divisors.  So the pattern of answers so far is $1,2,4,8,16$.  How many divisors do you think $6!$ has?
 
 You might have guessed $32$, since the answer seems to double at each step.  But alas, $6!$ only has $30$ divisors!  This illustrates why we need mathematical proofs, to show that a pattern continues indefinitely.  We need to be sure that there is not some counterexample to the pattern that is so large that we simply have not found it yet.
\end{example}

\begin{example}
Here are three numbers that are 
\emph{almost integers}, meaning that they are so close to integers that you might not realize they are transcendental numbers:
\[{\displaystyle e^{\pi {\sqrt {43}}}\approx 884736743.9997};\]
\[{\displaystyle e^{\pi {\sqrt {67}}}\approx 147197952743.999998};\]
\[{\displaystyle e^{\pi {\sqrt {163}}}\approx 262537412640768743.9999999999992}.\]
The last one is called Ramanujan's constant
after the famous mathematician Ramanujan
because, in April 1975, Gardner wrote in Scientific American that Ramanujan conjectured that $e^{\pi \sqrt{163}}$ is an integer; this was an April fool's hoax.
See OEIS A060295.
\end{example}

\begin{example}
\label{ex:four-color}
The Four Color Theorem states that for any map of countries drawn on a 2-dimensional world map, it is always possible to use four colors to color the countries in the map, such that no two countries sharing a border have the same color.  

This was an open problem for a long time, and in 1976 it was finally proven by  
Appel and Haken using a computer.  They first used mathematical proof reasoning techniques to categorize the possible maps into various types, and then used a computer to exhaustively check all of these cases.  Computers can make a proof much faster to complete when there are many cases to check.
\end{example}

\begin{example} \label{Efermatprime}
In the 1600s, Fermat studied numbers of the form $F_a=2^{2^a} +1$, which are now called Fermat numbers.
The first few examples of Fermat numbers are
\[F_0=3, \ F_1 = 5, \ F_2=17, \ F_3 = 257, \ F_4 = 65537.\]
After Fermat noticed that these five numbers are prime, he conjectured that $F_a$ is always prime.  But this was found to be false!
In 1732, Euler showed that
\[F_5 = 4294967297 = 641 \cdot 6700417.\]

Currently no one knows how many Fermat numbers are prime.  Even with modern computing power, no more examples of Fermat primes have been found.  Fermat numbers are still useful because they can have very large prime factors.
In 2020, as part of the PrimeGrid collaboration, Brown, Reynolds, Penn\'{e} \& Fougeron found the megaprime 
$13 \cdot 2^{5523860} + 1$ as a factor of 
$F_{5523858}$.
\end{example}

\begin{example}
Fermat's little theorem states: 
if $p$ is a prime number, then $p$ divides $a^{p-1}-1$ for all $a$ such that ${\rm gcd}(a,p)=1$.
For example, let $p=5$.  For $a=1,2,3,4$, then $a^4 -1$ equals $0, 15, 80, 255$ (respectively) and these numbers are all divisible by $5$.

We could ask if the converse to Fermat's little theorem is true: 
if $n$ divides $a^{n-1}-1$ for all $a$ such that ${\rm gcd}(a,n)=1$, then is $n$ prime?
If you experimented with different choices of $n$, you might start to believe that this is true.  But there is a counterexample!
The number $n=561$ 
divides $a^{560}-1$ for all $a$ such that ${\rm gcd}(a, 561) =1$.  However, $561$
is not prime, because it factors as $561=3 \cdot 11 \cdot 17$.
The number 561 is called a Carmichael number; it is an example of a pseudo-prime (fake prime).  This topic is useful in cryptography.
\end{example}

\begin{example}
A final famous problem named after Fermat is Fermat's Last Theorem.  In 1637, Fermat conjectured that the equation 
\[X^n+Y^n=Z^n\]
has no solutions when $n \geq 3$ is an integer and when $X,Y,Z$ are integers, unless $XYZ=0$.
This problem motivated people to study number theory for hundreds of years and their work led to the development of algebraic number theory and arithmetic geometry.  In 1994, Wiles published a proof of Fermat's Last Theorem, 
building on work of Frey, Ribet, Serre, Taniyama and Shimura.  Wiles received the Abel prize, worth about \$700,000, for the proof.
\end{example}

\subsection*{Exercises}

\begin{enumerate}
\item Show that $\pi^3$ and ${\rm sin}(11)$ are almost integers.

\item The golden ratio is $\phi=(1+\sqrt{5})/2 = 1.618\ldots$.
Show that $\phi^{17}$, $\phi^{18}$, and $\phi^{19}$ are almost integers.

\item Gelfond's constant is $e^\pi$.
Show that $e^\pi- \pi$ is an almost integer.
(Unlike the other examples in this section, there is no known explanation for this, so it is viewed as a coincidence.)

\item Read the history of the proof of the $4$-color theorem.  What was unusual about the proof of Appel and Haken?

\item Read the history of Tait's conjecture in graph theory.  What was important about Tutte's work on this conjecture?

\item Read the history of the proof of the Poincar\'e conjecture.  How do you think Hamilton felt about Perelman's proof of this conjecture?

    \item Read the history of the proof of the ABC-conjecture.  Decide whether you think Mochizuki proved this conjecture or not.
    
    \item There is a conjecture about primes, called 
Vandiver’s Conjecture, which has been verified for all primes less than 163 million.  Yet many mathematicians do not believe it is true for all primes.
What kind of justification is needed to convince someone that the data is conclusive or inconclusive?

    \item Explain the mistake in one of the false proofs at \url{http://mathforum.org/dr.math/faq/faq.false.proof.html} or \url{http://www.math.toronto.edu/mathnet/falseProofs/}
\end{enumerate}

\section{Induction}\label{sec:induction}

\begin{videobox}
\begin{minipage}{0.1\textwidth}
\href{https://www.youtube.com/watch?v=cBff3lTYlBs\&list=PL5J6K3znOvOmzBUoxlk-W0N4j7L1Y9yfW\&index=13}{\includegraphics[width=1cm]{video-clipart-2.png}}
\end{minipage}
\begin{minipage}{0.8\textwidth}
Click on the icon at left or the URL below for this section's short video lecture. \\\vspace{-0.2cm} \\ \href{https://www.youtube.com/watch?v=cBff3lTYlBs}{https://www.youtube.com/watch?v=cBff3lTYlBs}
\end{minipage}
\end{videobox}

The natural numbers 
${\mathbb N} =\{0, 1, 2, 3, \ldots, \}$ are lined up like a row of dominoes.

The goal of this section is to see how to prove facts or formulas that are true for each natural number $n \in {\mathbb N}$.  It is not possible to check one value of $n$ at a time because there are infinity many of them.  A good analogy for an inductive proof is the idea of knocking over all the dominoes: first, you need to knock over one of them (usually the first); then, you need a method to make sure that when each domino falls, it knocks over the next.  Thus, knocking over one domino sets off a chain reaction that causes all of the infinitely many dominoes to fall over. 

To state this mathematically, 
we consider a property or fact or formula called $P_n$, which could be true or false for each $n \in {\mathbb N}$. For example, 
\[P_n: 1+ 2 + \cdots n = n(n+1)/2.\]
We already saw in Lemma~\ref{lem:Tn} that this property $P_n$ is true for all integers $n \geq 1$.
As another example:
\[P_n: \text{The Fermat number $F_n$ is prime.}\]
We already saw in Example~\ref{Efermatprime} that this property $P_n$ is true when $n=1,2,3,4$ but false when $n=5$.  The first is true for all $n$ and we will see how to prove it using induction, but since the second is not true for all $n$, we cannot prove it by any method.

\begin{theorem}(The inductive principle)
Suppose
\begin{itemize}
\item (Base case) the property $P_n$ is true for $n=0$, and
\item (Inductive step) if property $P_k$ is true for an arbitrary integer $k \geq 1$, then so is property $P_{k+1}$.
\end{itemize}
Then $P_n$ is true for all integers $n \in {\mathbb N}$.
\end{theorem}

The above theorem is often called the Principle of Induction or just \textit{induction}, and it is often thought of as a method for proving that a statement $P_n$ is true for all integers $n\in \mathbb{N}$, as follows. \\  

\noindent \textbf{Method of Proof by Induction:} To prove $P_n$ is true for all natural numbers $n\ge 0$:
\begin{enumerate}
    \item Prove the \textbf{base case}: that $P_n$ is true for $n=0$.  (This usually is an easy computation.)
    \item Let $k\in \mathbb{N}$ be arbitrary, and assume $P_k$ is true for this specific value $k$.  (The statement $P_k$ is called the \textbf{inductive hypothesis}.)
    \item Using the assumption that $P_k$ is true, prove that $P_{k+1}$ is true.
\end{enumerate}

Notice that steps $2$ and $3$ above together prove the inductive step, since to prove an ``if... then'' statement in mathematics, one assumes that the hypothesis is true and then proves the conclusion holds under that assumption.  

In step 2, choosing $k$ to be arbitrary means that we are showing that no matter what domino falls, it will always knock over the next one.  Combining this fact with the base case that the first domino falls, we can conclude that every domino falls (the statement $P_n$ is true for all $n$).

\begin{remark}
The inductive step can be replaced with:
if property $P_{k-1}$ is true for an arbitrary integer $k \geq 1$, then so is property $P_{k}$.  In other words, in step $2$, we can let $k\ge 1$ (rather than $k\ge 0$) and assume $P_{k-1}$, and then in step $3$, we prove $P_k$.  This is a purely aesthetic choice, but sometimes makes the algebra in a proof simpler to work with.
\end{remark}

\begin{example}
The sum of the first $n$ odd integers is $n^2$.  In other words, the following property is true for all integers $n \geq 1$:
\[P_n: 1+3+5+\ldots+(2n-1)=n^2.\]
\end{example}

\begin{proof}
We proceed by induction.

\emph{(Base case.)} For $n=1$, note that $1=1^2$.

  \emph{(Inductive step.)} Let $k\ge 1$ be an arbitrary natural number, and suppose $P_k$ is true, meaning that
\[1+3+5+\ldots+(2k-1)=k^2.\]
Adding $2k+1$ to both sides gives
\[1+3+5+\ldots+(2k-1)+(2k+1)=k^2+2k+1=(k+1)^2,\]
showing that $P_{k+1}$ is true as desired.
Hence by induction, $P_n$ is true for all $n \geq 1$.
\end{proof}

The next two lemmas can similarly be proven by induction (see the Exercises).

\begin{lemma}\label{LnewproofTn}
(See Lemma~\ref{lem:Tn})
The following property is true for all integers $n \geq 1$:
\[P_n: 1+ 2 + \cdots n = n(n+1)/2.\]
\end{lemma}




\begin{lemma} \label{Lpower2induct}
The following property is true for all integers $n \geq 1$:
\[1+2+4+8+\ldots+2^{n-1}=2^n-1.\]
\end{lemma}





\begin{remark}
Sometimes a statement $P_n$ is false for small values of $n$.  If there is a positive integer $c$ for which $P_c$ is true and the inductive step works, then we 
can modify the method of induction to prove that statement $P_n$ is true for all $n\ge c$ (instead of $n\ge 0$).
To do this, we change the base case to prove $P_c$ rather than prove $P_0$, and also change the assumption on the inductive step to choose $k\ge c$ and assume $P_k$.
\end{remark}

\begin{example}
Prove that $n! \ge 2^n$ for all $n\ge 4$.
\end{example}

\begin{proof}
In this case, the statement is false when $n=1,2,3$ but true when $n=4$.
So we proceed by induction on $n$, starting with the case $n=4$.

\emph{(Base case.)}  
The identity is true when $n=4$ because $24=4! \geq 2^4=16$.

\emph{(Inductive step.)}
Let $k\ge 4$ be arbitrary, and suppose $P_k$ is true, meaning that $k! \geq 2^k$.
Then \[(k+1)! = (k+1)k! \geq (k+1)2^k \geq 5 \cdot 2^k \geq 2^{k+1},\]
showing that $P_{k+1}$ is true as desired.  Hence by induction, $P_k$ is true for all $n \geq 4$.
\end{proof}

\begin{remark}
Suppose we want to show that two equations $A$ and $B$ are equal.
One common mistake that people make in writing proofs is to assume $A=B$ and then simplify.
Warning: this does not always work!
Here is an extreme example of a \emph{false} claim and a \emph{false} proof to illustrate this:
\begin{claim}
$3=5$.
\end{claim}
\begin{proof}
Note $3=5$ implies $3\cdot0=5\cdot0$ implies $0=0$ which is true.
\end{proof}
\end{remark}

Here is an example of a \emph{false} claim and a \emph{false} proof by induction, that emphasizes the need to get both the base case and the inductive hypothesis step correct.

\begin{example} \label{proofhorse}
All horses are the same color.
\end{example}

\begin{proof} 
We will prove by induction that all horses in a herd of size $n$ are the same color for any $n \geq 1$.

\emph{(Base case.)} In any herd of 1 horse, all horses in that herd are the same color.

\emph{(Inductive step.)} Let $k\ge 1$ be arbitrary, and suppose all horses in a herd of size $k$ are the same color. Consider a herd of $k+1$ horses.

\begin{center}
\includegraphics[width=5in]{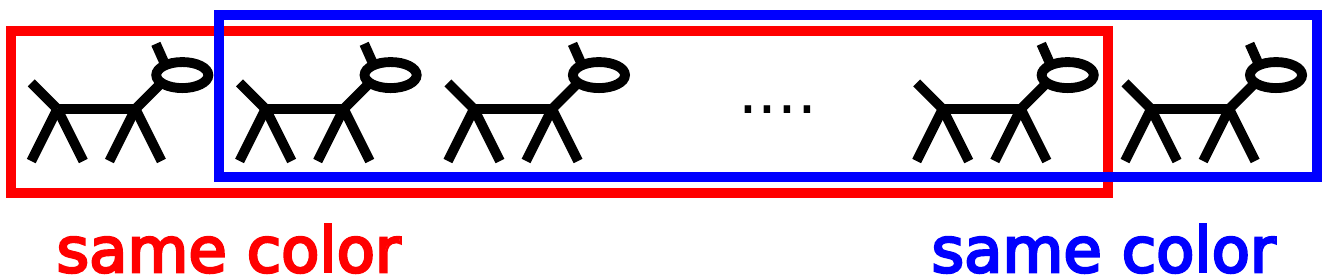}
\end{center}

Removing the last horse, the inductive assumption says that the first horse is the same color as the horses in the middle.
Removing the first horse, the inductive assumption says that the last horse is the same color as the horses in the middle.
Hence all $k+1$ horses in the herd are the same color.
By induction, the horses in a herd of size $n$ are the same color for all $n\ge 1$.
\end{proof}

Clearly not all horses are the same color, so where did our proof above go wrong?  Our mistake in Example~\ref{proofhorse} 
is that the inductive step fails when going from $k=1$ to $k=2$, even though it works for all larger values of $k$. When $k=1$, then the middle herd is empty.  This is why the bound on $k$ in the declaration ``Let $k\ge 1$ be arbitrary'' is so important. 

We will now give a proof by induction of the Hockey Stick Identity from Proposition~\ref{Ppattern5}. In this case, there are two variables $n$ and $m$ in the formula, so we need to be careful to specify which one is increasing in the inductive step.

\begin{proposition} \label{Phockey}
Suppose $n$ is a positive integer and $m$ is a non-negative integer.  Then
$$\binom{n}{0}+\binom{n+1}{1}+\binom{n+2}{2}+\ldots+\binom{n+m}{m}=\binom{n+m+1}{m}.$$
\end{proposition}

\begin{proof}
We proceed by induction on $m$ for a fixed $n$.

\emph{(Base case.)}  
The identity is true when $m=0$ because \[\binom{n}{0}=1=\binom{n+0+1}{0}.\]

\emph{(Inductive step.)} Let $m\in \mathbb{N}$ be arbitrary, and assume that \begin{equation}
\label{Ehockeyind1}
\binom{n}{0}+\binom{n+1}{1}+\binom{n+2}{2}+\cdots+\binom{n+m}{m}=\binom{n+m+1}{m}.
\end{equation}

We must now show that 
\begin{equation}
\label{Ehockeyind2}
    \binom{n}{0}+\binom{n+1}{1}+\binom{n+2}{2}+\cdots+\binom{n+m}{m}+\binom{n+m+1}{m+1}=\binom{n+m+2}{m+1}.
\end{equation}

By substituting \eqref{Ehockeyind1}, we see that the left hand side of \eqref{Ehockeyind2} simplifies to
$\binom{n+m+1}{m}+\binom{n+m+1}{m+1}$ which equals
$\binom{n+m+2}{m+1}$ by Pascal's recurrence Theorem~\ref{thm:Pascal}.
Hence the statement is true for all non-negative integers $m$ by induction.
\end{proof}

\begin{remark}
\label{Rstrongind}
There is a variation of the inductive process, called \textit{Strong induction}.  In the inductive step of this variant, we suppose that $P_k$ is true for all $1 \leq k \leq n$, and show that $P_k$ is true when $k=n+1$.
\end{remark}

\subsection*{Exercises}

\begin{enumerate}
    \item Prove Lemma~\ref{LnewproofTn} using induction.
    
    \item Prove Lemma~\ref{Lpower2induct}
using induction.

 \item Prove the \textit{geometric series formula} using induction: $$1+r+r^2+\cdots+r^n=\frac{r^{n+1}-1}{r-1}$$

\item Prove using induction that the following identity is true for all integers $n\ge1$: \[P_n: \frac{1}{1\cdot2}+\frac{1}{2\cdot3}+\ldots+\frac{1}{n(n+1)}=\frac{n}{n+1}.\] 

\item Prove using induction
that the following identity is true for all integers $n\ge1$:
\[P_n: 1^2+2^2+3^2+\ldots+(n-1)^2+n^2=\frac{n(n+1)(2n+1)}{6}.\]

\item Prove using induction that $n^2 \geq 5n+5$ for all integers $n\ge 6$.

\item Prove using induction that $2^n \ge 10n+7$ for all integers $n\ge 7$.

\item Prove that every positive integer $n\geq 2$ can be factored as a product of one or more prime numbers using strong induction.															
															
(Hint: In the induction step, consider two cases: if the number k+1 is prime then we are done; otherwise it factors as a product of two smaller numbers, each of which factors into primes by the induction hypothesis.)															
\end{enumerate}

\section{Counting in two ways}

\begin{videobox}
\begin{minipage}{0.1\textwidth}
\href{https://www.youtube.com/watch?v=C0g_hGyeO30}{\includegraphics[width=1cm]{video-clipart-2.png}}
\end{minipage}
\begin{minipage}{0.8\textwidth}
Click on the icon at left or the URL below for the next two sections' short video lecture. \\\vspace{-0.2cm} \\ \href{https://www.youtube.com/watch?v=C0g_hGyeO30}{https://www.youtube.com/watch?v=C0g\_hGyeO30}
\end{minipage}
\end{videobox}

We can prove some formulas from earlier in the book using the principle called \textit{counting in two ways}.  
\begin{lemma}
	(Counting in Two Ways.)  If a set has $n$ elements and it also has $m$ elements then $n=m$.
\end{lemma}

This sounds like a trivial principle, but it is incredibly useful in proving algebraic identities in combinatorics.  Here is the way we can use this. Suppose we want to prove $F=G$ where $F$ and $G$ are both mathematical formulas.  Then we simply have to find a collection of objects such that both $F$ and $G$ are the answer to the question ``how many are there in the collection''?

\begin{example}
	Suppose $a$ and $b$ are positive integers.
	Here is a combinatorial proof that $$a\cdot b=\underbrace{b+b+\cdots+b}_{a \text{ summands}}$$ where the sum has $a$ copies of $b$.  By repeated application of the Addition Principle, we see that the right-hand side, $b+b+\cdots+b$, counts the number of students in a school that has $a$ classes with $b$ students each.  By the Multiplication Principle, $a\cdot b$ counts the number of ways to first walk into one of the $a$ classrooms and then single out one of the $b$ students in that class.  Thus $a\cdot b$ is also equal to the total number of students, and so $a\cdot b=b+b+\cdots+b$.
\end{example}

  We can also prove more intricate identities combinatorially.

\begin{example} \label{Epascal2}
	Here is a combinatorial proof of Pascal's recurrence in Theorem~\ref{thm:Pascal} that $$\binom{n+1}{k}=\binom{n}{k-1}+\binom{n}{k}.$$ 
	\end{example}
	\begin{proof}
	Let's find the number of ways of choosing 
	$k$ people from a classroom that has $n$ students and $1$ professor.  
	Since there are $n+1$ people in the room, the answer is $\binom{n+1}{k}$.
	On the other hand, we can separate the ways of choosing the group into (i) 
	those that \textit{do not} contain the teacher and (ii) those that \textit{do} contain the teacher.  There are $\binom{n}{k}$ choices in case (i) and $\binom{n}{k-1}$ in case (ii) (since we are choosing only $k-1$ of the students). By the addition principle, the number of ways is the sum of the answers for (i) and (ii) and the result follows.
	\end{proof}

We now give combinatorial proofs of 
the ``sum of squares'' formula from Proposition~\ref{Ppattern4}.
\begin{proposition} \label{Pproofsumsquares}
If $n$ is a natural number, then $\binom{n}{0}^2+\binom{n}{1}^2+\ldots+\binom{n}{n-1}^2+\binom{n}{n}^2=\binom{2n}{n}$.
\end{proposition}

We will prove this in two ways, neither of which is easy.
The second proof proves an even harder identity, Vandermonde's identity, which reduces to Proposition~\ref{Pproofsumsquares} when we make a good choice for the variables.

\begin{proof}
Note that $\binom{2n}{n}$ is the number of ways to choose $n$ soccer players out of a team of $2n$ soccer players.  We will count that in another way by separating the $2n$ soccer players into two groups, by putting a gold jersey on $n$ of the soccer players and a green jersey on the other $n$ soccer players. From the first group, we choose $k$ soccer players with gold jerseys.  In order to choose $n$ soccer players altogether, we need to choose $n-k$ soccer players with green jerseys.  Symbolically,  
\[\underbrace{1,2,\ldots,n,}_{\text{choose $k$}}\hspace{3mm}\underbrace{n+1,n+2,\ldots,2n}_{\text{choose $n-k$}}.\]
So there are 
$\binom{n}{k} \cdot \binom{n}{n-k}$ ways to do that.  

The choice of the number $k$ was arbitrary and so we need to consider all the possible values of $k$, which could be as small as $0$ and as big as $n$.
 So we need to add up the number of ways to
 choose $n$ soccer players as $k$ varies between $0$ and $n$.
 This gives us:
\begin{align*}
\binom{2n}{n}
&=\sum_{k=0}^n(\text{\# choices of $k$ players with gold jerseys})\cdot(\text{\# choices of $n-k$ players with green jerseys})\\
&=\sum_{k=0}^n\binom{n}{k}\cdot\binom{n}{n-k}\\
&=\sum_{k=0}^n\binom{n}{k}\cdot\binom{n}{k}\\
&=\sum_{k=0}^n\binom{n}{k}^2\\
&=\binom{n}{0}^2+\binom{n}{1}^2+\ldots+\binom{n}{n-1}^2+\binom{n}{n}^2.
\end{align*}
\end{proof}

Here is another identity.
\begin{proposition} \label{PVandermonde}
(Vandermonde's identity.)
Suppose $n,m,\ell$ are non-negative integers such that $0 \leq \ell \leq m+n$.
\[\binom{n+m}{\ell} = \sum_{k=0}^\ell \binom{n}{k} \binom{m}{\ell-k}.\] (In the summation above, we set $\binom{n}{k}=0$ if $k>n$ and similarly set $\binom{m}{\ell -k}=0$ if $\ell - k >m$.)
\end{proposition}

\begin{proof}
By analogy: in your kitchen cabinet, you have $n$ granola bars (all different) and $m$ bags of trail mix (all different). 
The left hand side counts the number of ways to choose $\ell$ snacks out of $n+m$ choices.  

Let $D_k$ be the number of ways to choose the snacks 
with exactly $k$ granola bars and $\ell -k$ bags of trail mix.
Then $D_k =\binom{n}{k}  \binom{m}{\ell-k}$ because  you are choosing $k$ granola bars out of $n$ and $\ell-k$ bags of trail mix out of $m$.
We add up the number of ways for each possibility for $k$.
This equals the right hand side
$\sum_{k=0}^\ell D_k$. 
\end{proof}

Using Vandermonde's identity, we can find a second proof of Proposition~\ref{Ppattern4}.

\begin{proof} In Proposition \ref{PVandermonde}, substitute $m=n =\ell$ and simplify.  This yields
\[\binom{2n}{n} = \sum_{i=0}^n \binom{n}{i}^2.\]
\end{proof}

\subsection*{Exercises}

\begin{enumerate}
 \item Find a counting in two ways proof for Gauss's formula for the triangular numbers: $$1+2+3+\cdots+n=\frac{n(n+1)}{2}.$$

    \item Give a proof of the following identities by counting in two ways:
    \begin{enumerate}
        \item $\binom{n}{k}=\binom{n}{n-k}$
        \item $\binom{n}{m}\binom{n-m}{k}=\binom{n}{k}\binom{n-k}{m}$
        \item $\binom{n}{0}+\binom{n}{1}+\cdots + \binom{n}{n}=2^n$
        \item $\binom{n}{k}\cdot k=n\cdot \binom{n-1}{k-1}$
        \item $\binom{n}{1}+2\cdot \binom{n}{2}+3\cdot \binom{n}{3}+\cdots +n\cdot \binom{n}{n}=n\cdot 2^{n-1}$
    \end{enumerate}

    \item Recall from Proposition~\ref{Ppattern5} that the Hockey Stick identity (version 1) states that $$\binom{n}{0}+\binom{n+1}{1}+\cdots +\binom{n+m}{m}=\binom{n+m+1}{m}.$$ Find a counting in two ways proof of this identity.
\end{enumerate}

\section{Bijective proofs}
\label{sec:bij}
\begin{videobox}
\begin{minipage}{0.1\textwidth}
\href{https://www.youtube.com/watch?v=C0g_hGyeO30}{\includegraphics[width=1cm]{video-clipart-2.png}}
\end{minipage}
\begin{minipage}{0.8\textwidth}
Click on the icon at left or the URL below for this and the last section's short video lecture. \\\vspace{-0.2cm} \\ \href{https://www.youtube.com/watch?v=C0g_hGyeO30}{https://www.youtube.com/watch?v=C0g\_hGyeO30}
\end{minipage}
\end{videobox}

Here is a sample problem that you will be able to solve at the end of this section.

\begin{question}
How many ways can you choose $4$ numbers from $\{1, \ldots, 100\}$ so that none of them are consecutive?
\end{question}

We already know that there are $\binom{100}{4}$ ways to 
choose $4$ numbers from $\{1, \ldots, 100\}$ but it is not clear how to deal with the condition that they are not consecutive.
In this problem, it is not enough to count the size of one set in two ways.  Instead, we will use a \textit{bijection} to compare the size of
two different sets.

Let $A$ and $B$ be sets.  
A bijection is a matching or pairing of the elements of $A$ with the elements of $B$.
Here is the formal definition.

\begin{definition}
Let $A$ and $B$ be sets.  Let $f:A \to B$ be a map.
\begin{enumerate}
    \item The map $f$ is \defn{$1$-to-$1$} if, whenever $a_1 \not = a_2$, then $f(a_1) \not = f(a_2)$. 
    \item The map $f$ is \defn{onto} if for every $b \in B$,
there is an $a \in A$ such that $f(a)=b$.
\item The map $f$ is a \defn{bijection} if it is $1$-to-$1$ and onto.
\end{enumerate}
\end{definition}

Intuitively, the onto condition means that the map $f$ hits every element of $B$.  Another way to say this is that the range of $f$ is all of $B$.
The $1$-to-$1$ condition means that no element of $B$ is hit twice; it is equivalent to saying, if $f(a_1)=f(a_2)$, then $a_1=a_2$.

\begin{example} \label{Esubnk}
Let $\Omega = \{1, \ldots, n\}$.
Fix $k$ such that $0 \leq k \leq n$.
Let $A$ be the set of subsets of $\Omega$ of size $k$.  Let $B$ be the set of subsets of 
$\Omega$ of size $n-k$.
There is a bijection between $A$ and $B$
taking a subset $S$ of size $k$ to the complement $\Omega - S$ which has size $n-k$. 
\end{example}

\begin{lemma}
There is a bijection $f:A \to B$ if and only if $f$ has an inverse map $g:B \to A$.
\end{lemma}

\begin{proof}
Suppose $f:A \to B$ is a bijection.  Here is how to define the map $g:B \to A$.  Given $b \in B$, since $f$ is surjective, there is an $a \in A$ such that $f(a) =b$.  Also $a$ is unique since $f$ is $1$-to-$1$.  Define $g(b)=a$.
To show that $g$ is the inverse of $a$, we need to show that 
$g \circ f: A \to A$ and $f \circ g:B \to B$ are both the identity map.  That is true since 
$(g \circ f)(a)=g(f(a)) = g(b) = a$
and $(f \circ g)(b) = f(g(b))=f(a)=b$.

Conversely, if $g:B \to A$ is an inverse of $f$, we want to 
show that $f$ is a bijection.
To show that $f$ is $1$-to-$1$, 
suppose that $f(a_1)=f(a_2)$.  Then $g(f(a_1)) = g(f(a_2))$.
So $(g \circ f)(a_1)= (g \circ f)(a_2)$.
Since $g \circ f$ is the identity map, 
this shows that $a_1 = a_2$, completing the proof that $f$ is $1$-to-$1$.
Finally, to show that $f$ is onto, choose $b \in B$.
Let $a=g(b)$, which is in $A$.  
Then $f(a)=f(g(b))=(f \circ g)(b)$.
Since $f\circ g$ is the identity map, this shows $f(a)=b$,
completing the proof that $f$ is onto.
\end{proof}

\begin{example}
Let $A$ be the set of positive integers whose last digit is $3$.
Let $B$ be the set of positive integers whose last digit is $7$.
There is a bijection $f:A \to B$ given by $f(a)=a+4$.
Its inverse map $g:B \to A$ is given by $g(b) = b-4$.
\end{example}

\begin{example}
Let $A$ be the set of integers.  Let $B$ be the set of multiples of $7$.  There is a bijection $f:A \to B$ given by $f(a)=7a$.
The inverse map is $g:B \to A$ given by 
$g(b)=b/7$.
\end{example}

\begin{lemma}
	(Bijective proof.) If $A$ and $B$ are sets such that $A$ has $m$ elements, $B$ has $n$ elements, and there is a bijection $f:A\to B$, then $m=n$.
\end{lemma}

Again, this sounds like a very simple principle, but it can be tricky to execute in practice.  Here, we provide several harder examples. 

Recall the following identity (alternating sum of binomial coefficients) from Proposition~\ref{Ppattern3}.
\begin{proposition}
\label{Paltv2}
If $n >0$, then $\binom{n}{0}-\binom{n}{1}+\binom{n}{2}-\cdots \pm \binom{n}{n}=0$.
\end{proposition}

Another way to write Proposition~\ref{Paltv2} is to move the negative terms to the right hand side: $$\binom{n}{0}+\binom{n}{2}+\binom{n}{4}+\cdots = \binom{n}{1}+\binom{n}{3}+\binom{n}{5}+\cdots$$ 
For a set $\Omega$ of size $n$, we can interpret the left hand side as the number of subsets of $\Omega$ having even size, and the right hand side as the number of subsets of $\Omega$ having odd size.  Thus, to prove Proposition~\ref{Paltv2}, we just need to find a bijection between subsets of odd and even size, as in the following proposition.

\begin{proposition} \label{Psubodd}
Let $\Omega = \{1, \ldots, n\}$.
Let $A$ (resp.\ $B$) be the set of subsets of $\Omega$ of even (resp.\ odd) size.  Then there is a bijection between $A$ and $B$.
In particular, the number of subsets of $\Omega$ of even size equals the number of subsets of odd size. 
\end{proposition}

Since there are $2^n$ subsets of $\Omega$,
it follows that the number of subsets of even size is $2^n/2 = 2^{n-1}$.

\begin{example}
Let $\Omega=\{1,2,3\}$. Its subsets can be paired as follows:
\begin{itemize}
\item $\emptyset \leftrightarrow \{1,2,3\}$
\item $\{1\} \leftrightarrow \{2,3\}$
\item $\{2\} \leftrightarrow \{1,3\}$
\item $\{3\} \leftrightarrow \{1,2\}$
\end{itemize}
\end{example}

This example generalizes to a proof in the case that $n$ is odd.

\begin{proof}[Proof of Proposition \ref{Psubodd} when $n$ is odd] Suppose $n$ is odd.  If $S \subset \Omega$ has size $k$, then the complement $\Omega - S$ has size $n-k$.  If $k$ is even, then $n-k$ is odd and vice-versa.
Each subset has exactly one complement, so we can pair each subset of even size with exactly one subset of odd size.
In conclusion, when $n$ is odd, there is a bijection between the set $A$ (resp.\ $B$) of subsets of $\Omega$ of even (resp.\ odd) size taking $S$ to $\Omega-S$.  
\end{proof}

The proof above does not work when $n$ is even, so we will set up a different bijection.  Here is a motivating example when $n=4$ to see how this bijection works.

\begin{example}
Let $\Omega=\{1,2,3,4\}$. Its subsets can be paired as follows:
\[ 
\xymatrix@R-1pc{
& & \{1,2\} \\
& \{1\} & \{1,3\} & \{1,2,3\}  \\
\emptyset \ar@{<->}[ur] & \{2\} \ar@{<->}[uur] & \{1,4\} & \{1,2,4\} & \{1,2,3,4\} \\
& \{3\} \ar@{<->}[uur] & \{2,3\} \ar@{<->}[uur] & \{1,3,4\}  \\
& \{4\} \ar@{<->}[uur] & \{2,4\} \ar@{<->}[uur] & \{2,3,4\} \ar@{<->}[uur]\\
& & \{3,4\} \ar@{<->}[uur]
}
\]
\end{example}

\begin{proof}[Proof of Proposition \ref{Psubodd} for $n$ even (and odd)]
Since $\Omega$ is nonempty, pick a distinguished element, say $1\in \Omega$. Group the subsets of $\Omega$ into pairs where the two subsets in each pair differ only in whether they contain $1$ or not. Every subset is in exactly one pair.  Each pair contains one subset of even size and one subset of odd size.
So the number of subsets of odd size and of even size is the same.
\end{proof}

As another example of a bijective proof, we prove
the ``sticks and stones'' theorem (Theorem \ref{thm:choose-with-repeats}) again here, using ``M\&M's and toothpicks'' instead.  The logic we used in Section~\ref{sec:sticks-and-stones} to prove this result 
is really a bijection in disguise.

\begin{proposition}
	Suppose $n$ is a positive integer and 
	$k$ is a non-negative integer such that 
	$0 \leq k \leq n$.  Then the number of ways to choose a set of size $k$ from a set of size $n$ (with repeats allowed) is
	 \[\binom{n+k-1}{k}.\]
	\end{proposition}
	
	\begin{proof}
	We first find a set $A$ whose size is the number of ways to choose a set of size $k$ from a set of size $n$.
	To do this, let $A$ be the set of cups filled with $k$ M\&M's, where there are $n$ different possible colors of M\&M's.  
	
	Next, we find a set $B$ whose size is $\binom{n+k-1}{k}$.
	Let $B$ be the set of binary sequences consisting of $k$ zeroes and $n-1$ ones.  Such a sequence has $n+k-1$ entries total and is determined by choosing which $k$ positions contain the $0$'s.  So the size of $B$ is
	$\binom{n+k-1}{k}$.
	
	We need to find a bijective map between $A$ and $B$.
	Fix an ordering of colors, for example, rainbow order with brown at the end.
	An element of the set $A$ is a cup of $k$ colored M\&M's.  Given such a cup, spill the M\&M's onto a plate and sort them by color, in a row.  Now there are $k$ M\&M's separated into $n$ sections (some of which may be empty).  Put a toothpick between each section, using $n-1$ toothpicks total.  Finally, put on color-filter goggles so that all the M\&M's look grey. 
	What we see is a row of $k$ grey M\&M's and $n-1$ toothpicks.  Write a sequence of $0$'s and $1$'s by writing a $0$ for each grey M\&M and writing a $1$ for each toothpick.  This results in a sequence of $k$ zeroes and $n-1$ ones, which is an element of the set $B$.
	
	Now we will construct the inverse map from $B$ to $A$.  Starting from a row of $0$'s and $1$'s, we can interpret this as a row of $k$ grey M\&M's separated by $n-1$ toothpicks, and then color the M\&M's according to the specified ordering of colors and put them into a cup.  Since these processes reverse each other, the two sets are in bijection, as desired.
\end{proof}

\begin{example} \label{Ebijectcons}
Let's finish the problem about choosing $4$ numbers from 
$\{1, \ldots, 100\}$, none of which are consecutive.
Call these numbers $a_1, a_2, a_3, a_4$, written in increasing 
order.  We want $a_2 > a_1+1$, $a_3 > a_2 + 1$, and 
$a_4 > a_3 + 1$.
The number of ways to do this is the size of the set 
\[A = \{(a_1,a_2, a_3,a_4) \mid 1 \leq a_1 < a_2 -1 < a_3 -2 < a_4 - 3 \leq 97.\}\]

Let's define new numbers $b_1,b_2,b_3,b_4$ so that these conditions are easier to state.
Let $b_1 = a_1$ (the conditions on $a_1$ are easy, so there is no reason to change $a_1$).
Let $b_2 = a_2 - 1$; the condition $a_2 > a_1 +1$
can be re-expressed as
$b_2 > b_1$.
Next, let $b_3 = a_3 -2$;  the condition 
$a_3 > a_2+1$ can be re-expressed as $b_3+2 > b_2+2$ or, more simply, $b_3 > b_2$.
Finally, let $b_4 = a_4 -3$; the condition $a_4 > a_3 +1$ can be re-expressed as $b_4 > b_3$.
Notice that $a_4 \leq 100$ is equivalent to $b_4 \leq 97$.

Define a new set $B$ with these conditions:
\[B = \{(b_1,b_2,b_3,b_4) \mid 1 \leq b_1 < b_2 < b_3 < b_4 \leq 97\}.\] 
There is a map $f:A \to B$ where $f(a_1,a_2,a_3,a_4)=
(a_1, a_2-1, a_3-2, a_4-3)$.
It has an inverse map $g:B\to A$
where $g(b_1,b_2,b_3,b_4) = (b_1, b_2+1, b_3 + 2, b_4 + 3)$.
So $f$ is a bijection.  

The size of $B$ is $\binom{97}{4}$ because choosing an element of $B$ is the same as choosing $4$ distinct numbers in 
$\{1, \ldots, 97\}$.  So that is the size of $A$ also.
\end{example}

For additional facts about bijections and cardinality, see the following supplemental video.

\begin{videobox}
\begin{minipage}{0.1\textwidth}
\href{https://www.youtube.com/watch?v=L5s0m5cjq2Q}{\includegraphics[width=1cm]{video-clipart-2.png}}
\end{minipage}
\begin{minipage}{0.8\textwidth}
Click on the icon at left or the URL below for more on bijections and cardinality. \\\vspace{-0.2cm} \\ \href{https://www.youtube.com/watch?v=L5s0m5cjq2Q}{https://www.youtube.com/watch?v=L5s0m5cjq2Q}
\end{minipage}
\end{videobox}

\subsection*{Exercises}

\begin{enumerate}
\item Follow the outline of Example~\ref{Ebijectcons} to find the number of ways to choose $5$ numbers from $\{1, \ldots, 80\}$ so that none of them are consecutive?

\item Follow the outline of Example~\ref{Ebijectcons} to find the number of ways to choose $4$ numbers from $\{1, \ldots, 100\}$ so that
$a_1 \leq a_2 \leq a_3 \leq a_4 \leq a_5$.

\item Prove that $\binom{n}{k}=\binom{n}{n-k}$ using a bijection.

  \item 
  Let $S$ be the set of binary sequences of length $n$.
  Explain why $|S|=2^n$.
  Find a bijection between 
  $S$ and subsets of $\{1,2,\ldots,n\}$.
  Conclude that 
  $2^n$ is also the number of subsets of $\{1,2,\ldots,n\}$.
  \end{enumerate}

\section{Proof by Contradiction}
\label{Sproofcont}

Here is the strategy for a proof by contradiction: assume the opposite (negation) of what you want to show; then show that this implies something that is definitely false; this means that the 
initial assumption must have been false.

\begin{example}
If I want to get on my flight which leaves Denver International Airport at 10 am, then I need to leave the house before 7 am.
\end{example}

\begin{proof}
Assume I leave my house in Fort Collins at 7 (or later).
Then I arrive at the parking lot by 8:15.  After taking the shuttle, I arrive at the airport at 8:30.  After checking my bag, it is 8:45.  After going through security, it is 9:00.
After taking the little train, it is 9:15.  Then I run and run 
to my gate, which is the very last one, and arrive there at 9:45. This is too late; the doors have closed.
\end{proof}

In this section, we give some examples of proofs by contradiction that come from the subject of number theory.

\begin{proposition}
The number $\sqrt{3}$ is irrational (not a fraction).
\end{proposition}

\begin{proof}
Assume that $\sqrt{3}$ is rational (a fraction).
Then we can write $\sqrt{3} = a/b$, where $a,b$ are integers with $b \not = 0$.  If ${\rm gcd}(a,b) \not = 1$, then we can simplify the fraction.  So, without loss of generality, we can suppose that ${\rm gcd}(a,b)=1$.  Squaring both sides, we see that 
$3=a^2/b^2$.  So $3b^2=a^2$.  
This means that $3$ divides $a^2$.
Since $3$ is prime, this shows that $3$ divides $a$ (since $a$ has a unique factorization into primes).
Write $a=3a_1$, where $a_1$ is an integer.
Substituting, we see that $3b^2 = (3a_1)^2 = 9a_1^2$.
So $b^2 = 3a_1^2$.  Then $3$ divides $b^2$, so $3$ divides $b$.
But then ${\rm gcd}(a,b)$ is a multiple of $3$.  So this is a contradiction.  So $\sqrt{3}$ is not a fraction.
\end{proof}

Below is Euclid's proof that there are infinitely many primes, which he discovered around 300 B.C.  The reason to use a proof by contradiction is because no one has found a formula or algorithm to generate prime numbers.

\begin{proposition}
There are infinitely many prime numbers.
\end{proposition}
\begin{proof}
Assume there are only finitely many primes;  
let $r$ be the number of primes and 
let $S=\{p_1, \ldots, p_r\}$ be the set of all the primes.
Consider the number $N=1+ \prod_{i=1}^r p_i$.
Then $N$ is divisible by at least one prime number $q$ (since any positive integer has a unique prime factorization.)
But $N$ is not divisible by any of the primes in $S$;
this is because $N \equiv 1 \bmod p_i$ for each $1 \leq i \leq r$.  So $q$ is a new prime which is not in $S$.
This contradicts the statement that $S$ is the set of all the primes.  So there are infinitely many primes.
\end{proof}

\begin{itemize}
    \item A combinatorial existence proof using contradiction (``Prove that there must have been two people who shook hands the with the same number of people.'')
\end{itemize}

\begin{remark}
Sometimes proofs by contradiction are so fun that people use them more often than needed.  It is better to use a direct proof when you can. 
\end{remark}

\subsection*{Exercises}

\begin{enumerate}
\item Prove that $\sqrt{5}$ is irrational. 	
\item Show that $\sqrt[3]{2}$ is irrational.
	
\item Use the quantity $N=3 + (2 \prod_{i=1}^r p_i)^2$
to show there are infinitely many primes which are congruent to 
$3$ modulo $4$.

\end{enumerate}

\section{The Pigeonhole Principle}

\begin{videobox}
\begin{minipage}{0.1\textwidth}
\href{https://www.youtube.com/watch?v=8j1F_KDijxc}{\includegraphics[width=1cm]{video-clipart-2.png}}
\end{minipage}
\begin{minipage}{0.8\textwidth}
Click on the icon at left or the URL below for this section's short lecture video. \\\vspace{-0.2cm} \\ \href{https://www.youtube.com/watch?v=8j1F_KDijxc}{https://www.youtube.com/watch?v=8j1F\_KDijxc}
\end{minipage}
\end{videobox}

The pigeonhole principle gives a way to show that certain objects need to be grouped together.  This principle relies on the strategy of a proof by contradiction and it has a lot of applications in combinatorics.

\begin{example}
In 2017, the population of Colorado was 5,612,000.
Prove that there are at least 6 people in Colorado with the same number of hairs on their heads. 
You may assume that each person's head has less than 1,000,000 hairs.
\end{example}

\begin{answer}
If there were 

at most $5$ people with 0 hairs,

at most $5$ people with 1 hair,

at most $5$ people with 2 hairs,

\hspace{35mm}\vdots

and at most $5$ people with 999,999 hairs,

\noindent then the number of people in Colorado would be at most 5,000,000. Since $5,612,000 > 5,000,000$, there must be at least 6 people in Colorado with the same number of hairs on their heads.
\end{answer}

\begin{proposition} \label{Ppigeon} (Pigeonhole principle Version 1)
If we place \emph{more} than $n$ objects into $n$ boxes, then some box contains 2 or more objects.
\end{proposition}

\begin{center}
\includegraphics[width=2in]{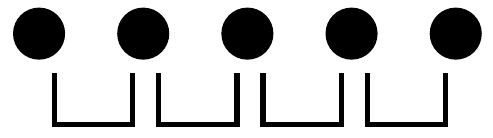}
\end{center}

\begin{proof}
If each box contained at most 1 object, then the total number of objects would be at most $n$, which is a contradiction. 
\end{proof}

\begin{etymology}
A pigeonhole is a place where birds ``roost", i.e., rest or sleep.  This principle is often stated with $n+1$ pigeons roosting in $n$ holes.  
\end{etymology}

\begin{remark}
The pigeonhole principle does not say anything about how many of the boxes are empty.  It is possible that all of the pigeons are stuffed into the same hole.
\end{remark}

\begin{example}
If 11 numbers are chosen from 1 to 100, then show that two of them have difference less than 10.
\end{example}

\begin{answer}
Separate the numbers $1-100$ into 10 ``boxes'', namely $1-10$, $11-20$, $21-30$, up to $91-100$.  Since there are 11 numbers and only 10 boxes, two of the numbers must be in the same box.  Then the difference between those numbers is less than $10$.
\end{answer}

\begin{example}
A target is in the shape of an
equilateral triangle, with side-length 2 meters.  If we hit the target with 5 arrows, 
prove that two arrows strike within 1 meter of each other.
\begin{center}
\includegraphics[width=1.2in]{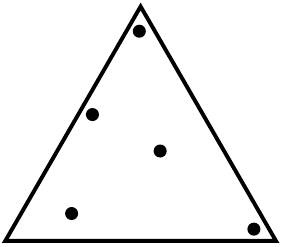}
\end{center}

To approach this using the pigeonhole principle, we divide the target into 4 parts, each of which is an equilateral triangle of side-length $1$, as drawn below.  We will think of these four smaller triangles as the ``boxes'' and the striking points of the arrows as the objects.
\begin{center}
\includegraphics[width=1.2in]{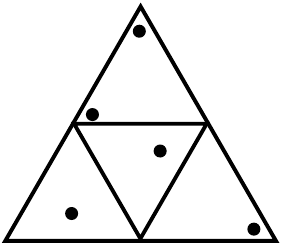}
\end{center}
By the pigeonhole principle, one of the smaller triangles must contain the striking points of at least two arrows. The distance between these striking points is at most the length of the side of the smaller triangle.  It follows that the striking points of those two arrows are within 1 meter of each other.
\end{example}

\begin{example}
There are $n$ guests at a party, for some integer $n \geq 2$.  Some pairs of them shake hands; (no one shakes hands with themself).
Prove that at least two guests shook hands with the same number of people. 

Each guest shook between $0$ and $n-1$ hands.  Let's label the boxes by the numbers $0$ through $n-1$ and put a person in box $i$ if they shook hands with exactly $i$ people.  At first, it does not look like we can use the pigeonhole principle because there are $n$ people and $n$ boxes.

However, there is one piece of information we did not use. 
Either box 0 or box $n-1$ is empty. The reason is that if someone shook $0$ hands, then no one shook hands with all $n-1$ of the other guests; 
if someone shook hands with all $n-1$ of the other guests, then no one shook $0$ hands.
Hence we are assigning $n$ people to $n-1$ boxes (labeled either by 0 through $n-2$ or by 1 through $n-1$). By the pigeonhole principle, two people must be assigned to the same box, and hence shook the same number of hands.
\end{example}

\begin{proposition} \label{Ppigeonv2}
(Pigeonhole principle Version 2)
Suppose $k$ is a non-negative integer.
If we place \emph{more} than $kn$ objects into $n$ boxes, then some box contains at least $k+1$ objects
\end{proposition}

\begin{remark}
To obtain Version 1 from Version 2, substitute the value $k=1$.
\end{remark}

\begin{proof}
If each box contained at most $k$ objects, then the total number of objects would be at most $kn$, which is a contradiction.
\end{proof}

To solve the problem about Colorado at the beginning of the section, we use Version 2 of the pigeonhole principle as follows

\begin{answer} 
Label the boxes by the numbers $1$ through $n=1,000,000$. 
People are assigned to the $i$th box if they have exactly $i$ hairs on their heads.  Let $k=5$. Since $5,612,000 > 5,000,000 = (5)\cdot (1,000,000)=k\cdot n$, there is a box with at least $6=k+1$ people by the pigeonhole principle (Version 2). 
These $6$ people all have the same number of hairs on their heads. 
\end{answer}

\begin{remark}
For pigeonhole principle problems, our recommendation is to:
\begin{itemize}
\item Identify $k+1$, and hence $k$.
\item Define your $n$ boxes (such that if enough objects land in a single box, then you have accomplished the goal of the problem!)
\item Make sure that the number of objects is bigger than $k\cdot n$. 
\end{itemize}
\end{remark}

We provide some harder examples.

\begin{example}
If 17 people are seated in a row of 20 chairs, then prove some consecutive set of 5 chairs are filled with people.
\end{example}

\begin{brainstorm} $ $
\begin{center}
\includegraphics[width=5in]{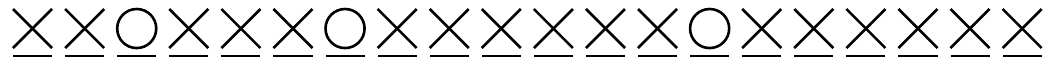}
\end{center}
\end{brainstorm}

\begin{answer}
There are 3 empty chairs.
These separate the row into $n=4$ ``boxes" which we can call the far left, center left, center right, and far right. Each of these parts contains filled chairs, which are consecutively filled since there are no other empty chairs.  (It is possible for one part to be empty or to contain only one chair.)  Let  $k=4$, so that $k+1=5$. Since $17>16=4\cdot 4=k\cdot n$, Version 2 of the pigeonhole principle says that some box contains at least $k+1=5$ objects, which means there are at least 5 consecutively filled chairs.
\end{answer}

\begin{example}
If 17 people are seated in a circle of 20 chairs, then prove some consecutive set of 6 chairs are filled with people.
\end{example}

\begin{answer}
The 3 empty chairs partition the circle into $n=3$ ``boxes" of consecutively filled chairs. Let $k+1=6$, so $k=5$. Since $17>15=5\cdot 3=k\cdot n$, Version 2 of the pigeonhole principle says that some box contains at least $k+1=6$ objects, which are consecutively filled chairs.
\end{answer}

One of the most fun applications of the pigeonhole principle is to \textit{Ramsey theory}.  Here we give one example of this.

\begin{example}
Consider the complete graph $K_6$ with $6$ vertices (every pair of its vertices is connected by an edge).
If you color the edges green and gold, show that there must be either a green triangle or a gold triangle of edges. 
\end{example}

\begin{answer}
Pick one vertex $v$.
It has 5 edges connecting it to the $5$ other vertices.  Let $n=2$ (for the 2 colors) and let $k=2$.
Since $5 > 2 \cdot 2$, at least 3 of these $5$ edges have the same color by Version 2 of the pigeonhole principle.  Let's say that color is green; (if not, we can continue with the same reasoning by swapping the colors).

These 3 green edges connect $v$ to $3$ other vertices, let's call them $w,x,y$. 
If the edge between $w$ and $x$ is green, then we are done because then $v,w,x$ form a triangle with only green edges. Similarly, if the edge between $w$ and $y$ is green, then we are done and if the edge between $x$ and $y$ is green then we are done.
So the only remaining case is if the three edges between $w,x,y$ are all gold.  But then they form a triangle with only gold edges and we are done! 
\end{answer}

\subsection*{Exercises}

\begin{enumerate}
\item If 5 people receive $11$ letters, show that someone gets at least $3$.
\item If a drawer contains 16 black and 16 white socks, how many do you need to pick in order to have a matching pair?  

\item The city of Phoenix has 1.5 million people.  The number of hairs on a head is at most $250,000$.
Show that at least 7 people have the same number of hairs on their head.

\item More than 6,000 of the students attending a private university were born in one of the 50 US states. Prove that at least 121 of these students were born in the same state.

\item In a set of 42 positive integers, prove that at least 5 of them have the same last digit.

\item In a set of 112 positive integers, prove that at least 12 of them have the same last digit.

\item Pick any 7 integers $x_1, \ldots, x_7$. Show there are two whose sum or difference is a multiple of $10$.

\item If 7 numbers are chosen from 1 to 100, show that two of them have a difference which is less than 17.

\item Prove that, among any 10 points inside a unit square, there must exist 3 of them that are all within $\sqrt{2}/2$ of each other.								
\item 
\begin{enumerate}
\item If 94 people are seated in a row of 100 chairs, then show that some consecutive set of 14 chairs is filled with people.
\item If 94 people are seated in a circle of 100 chairs, then show that some consecutive set of 16 chairs is filled with people.
\end{enumerate}

\end{enumerate}

\section{Additional problems for Chapter 5}

\begin{enumerate}
 \item Give a combinatorial proof (either by counting in two ways or using a bijection) for each of the following formulas.
  \begin{enumerate}
    \item  $n!=n\cdot (n-1)\cdot (n-2)\cdot\cdots\cdot 3\cdot 2 \cdot 1 $
    \item $n^k=n\cdot n\cdot \cdots\cdot n$, the product of $k$ copies of $n$
    \item $n\cdot (n-1)\cdot (n-2)\cdot \cdots (n-k+1)=\frac{n!}{(n-k)!}$
    \item $\binom{n}{k}=\frac{n!}{k!(n-k)!}$
  \end{enumerate}

\item Give a combinatorial proof
  of the following identities.
  \begin{enumerate}
    \item $\sum_{k=0}^n k \binom{n}{k}^2= n\binom{2n-1}{n-1}$
    \item $\sum_{k=0}^n \binom{n}{k}s^kt^{n-k}=(s+t)^n$ (Here $s$, $t$, and $n$ are positive integers.)
  \end{enumerate}

\item Give a combinatorial proof that the following identity holds for all positive integers n:					
$$2^0\binom{n}{0}+2^1\binom{n}{1}+...+2^n\binom{n}{n}=3^n.$$

\item If 30 numbers are chosen from 1 to 82, then show that there is a group of at least 5 of them such that the difference between any two numbers in this group is less than 12.

\item Suppose that 64 chairs are arranged in a circle. Prove that if 57 people sit in this circle of 64 chairs, then there are at least 9 consecutively filled chairs.

\item Suppose that 80 chairs are arranged in a \emph{row}. Prove that if 75 people sit in this row of 80 chairs, then there are at least 13 consecutively filled chairs.

\item An office needs to connect 15 computers to 10 printers, with the minimal number of cables.  At any moment, 10 computers might need printers.
\begin{enumerate}
\item Show that $r=5$ connections might not be enough for certain print jobs.

\item Show that $r=6$ connections is enough for every print job.
\end{enumerate}

\item 
 \begin{enumerate}
 \item Look up the statement of the 
 Well-Ordering Principle.  Explain how it is similar to and different from induction.
 \item Prove the Well-Ordering principle
 using only strong induction and the axiom that $0$ is smaller than every positive integer, as follows.
   \begin{enumerate}
     \item Let $P(n)$ be the statement that any subset of the natural numbers containing $n$ has a least element.
     \item Prove $P(0)$.
     \item Prove that if $P(0),P(1),\ldots,P(n)$ are all true, then $P(n+1)$ is true.
   \end{enumerate}
   \end{enumerate}

    \item 
   Let $p$ be a prime.  Fermat's Little Theorem states that $a^p \equiv a \bmod p$, for any positive integer $a$.
   For example, when $p=7$: then $7$ divides $1^7-1$; $7$ divides $2^7-2$;  $7$ divides $3^7-3$; etc.
   
   \begin{enumerate}
 \item For any $a \geq 2$, generalize Exercise~\ref{Exmodpbinom2} from Section~\ref{sec:binomial-theorem} to 
       show $(x_1 + \cdots + x_a)^p \equiv x_1^p + \cdots + x_a^p \bmod p$.
        \item Substitute $x_1=1, \ldots, x_a =1$ to show that 
        $a^p \equiv a \bmod p$, which gives a proof of Fermat's Little Theorem.
       \end{enumerate}

 \item In this problem, we will prove Fermat's Little Theorem using a combinatorial argument.
 \begin{enumerate}
     \item Let $S$ be the set of sequences of length $p$ whose entries are in the alphabet $\{1, \ldots, a\}$ such that not all the entries of the sequence are the same.
     Explain why $|S|=a^p-a$.
     \item We say that two sequences in $S$ are equivalent if the entries of one can be rotated to be the entries of the other.
     For example, when $p=5$, then $(1,2,3,4,5)$ is equivalent to $(2,3,4,5,1)$.
     Explain why each equivalence class in $S$ has size $p$.
     Explain why it is important for $p$ to be prime for this to be true.
     \item Use the division principle to show that $p$ divides $a^p-a$.
 \end{enumerate}
   
\end{enumerate}

\newpage
\section{Investigation: Pascal's triangle modulo 2 and Sierpinski's triangle}

Which elements of Pascal's triangle are even, and which are odd?  In other words, what is the value of the binomial coefficient $\binom{n}{k}$ mod $2$?  (See Section \ref{sec:mod} for modular arithmetic.)

\begin{center}
\begin{tabular}{rccccccccccccc}
$n=0$:&    &    &    &    &    &    &   1 \\\noalign{\smallskip\smallskip}
$n=1$:&    &    &    &    &    &  1 &  &  1\\\noalign{\smallskip\smallskip}
$n=2$:&    &    &    &    &  1 &    &  2 &    &  1\\\noalign{\smallskip\smallskip}
$n=3$:&    &    &    &  1 &    &  3 &    &  3 &    &  1\\\noalign{\smallskip\smallskip}
$n=4$:&    &    &  1 &    &  4 &    &  6 &    &  4 &    &  1\\\noalign{\smallskip\smallskip}
$n=5$:&    &   1 &    &   5 &    &  10  &    &  10  &    &  5  &   &  1\\\noalign{\smallskip\smallskip}
$n=6$:& 1 &    &  6 &    &  15 &    &  20 &    &  15 &    &  6  &   & 1\\\noalign{\smallskip\smallskip}
\end{tabular}
\end{center}

If an integer $a$ is even, then the remainder when dividing $a$ by $2$  is $0$, so $a \equiv 0 \bmod 2$;
if $a$ is odd, then the remainder when dividing $a$ by $2$ is 1, so $a \equiv 1 \bmod 2$. We can therefore write down ``Pascal's triangle mod 2'' by replacing all even numbers with $0$ and all odd numbers with $1$:

\begin{center}
\begin{tabular}{rccccccccccccc}
$n=0$:&    &    &    &    &    &    &   1 \\\noalign{\smallskip\smallskip}
$n=1$:&    &    &    &    &    &  1 &  &  1\\\noalign{\smallskip\smallskip}
$n=2$:&    &    &    &    &  1 &    &  0 &    &  1\\\noalign{\smallskip\smallskip}
$n=3$:&    &    &    &  1 &    &  1 &    &  1 &    &  1\\\noalign{\smallskip\smallskip}
$n=4$:&    &    &  1 &    &  0 &    &  0 &    &  0 &    &  1\\\noalign{\smallskip\smallskip}
$n=5$:&    &   1 &    &   1 &    &  0  &    &  0  &    &  1  &   &  1\\\noalign{\smallskip\smallskip}
$n=6$:& 1 &    &  0 &    &  1 &    &  0 &    &  1 &    &  0  &   & 1\\\noalign{\smallskip\smallskip}
\end{tabular}
\end{center}

\begin{question}
Continue writing down Pascal's triangle mod $2$ by copying the above table and then writing the next $5$ rows of $0$'s and $1$'s.  Find three patterns for the arrangement of $0$'s and $1$'s.
\end{question}

\begin{question}
How can you determine the next row of Pascal's triangle mod $2$ from the previous, without computing the original Pascal's triangle?  Explain your answer.
\end{question}

\begin{question}
After the row for $n=3$, when is the next row that only contains $1$'s? What does the row after that one look like?
\end{question}

\begin{proposition}
\label{Psierpinski2}
If $n=2^m$ is a power of $2$, then the numbers in row $n$ of Pascal's triangle mod $2$ are all $0$ except for the first and last entry. 
\end{proposition}

The next three questions outline three different ways of proving the above proposition.

\begin{question} 
Prove Proposition~\ref{Psierpinski2} with this method:
Write $\binom{2^m}{k}$ as $(2^m)!/(k!(2^m-k)!)$.  We want to show there are more $2$'s in the prime factorization of the numerator than in the denominator. Do this by finding a way to pair up each factor of $2$ in the denominator with a factor of $2$ in the numerator.  Show that there are always more factors of $2$ in the numerator.
\end{question}

\begin{question}
(An algebraic proof.) Prove Proposition~\ref{Psierpinski2} with this method:
\begin{enumerate}
    \item Explain why it is enough to show that $(x+y)^{2^m} \equiv x^{2^m} + y^{2^m} \bmod 2$.
    (Here we reduce the coefficients of the polynomial mod 2 but not the exponents).
    \item Show that $(x+y)^2 \equiv x^2 + y^2 \bmod 2$.
    \item Use induction on $m$ to prove statement 1, where the inductive hypothesis is that
    $(x+y)^{2^{m-1}} 
    \equiv x^{2^{m-1}} + y^{2^{m-1}} \bmod 2$.  For the inductive step, square both sides of the equation.
\end{enumerate}
\end{question}

 \begin{question}
  (A bijective/combinatorial proof.)
  Prove Proposition~\ref{Psierpinski2} with this method.  
 \begin{enumerate}
  \item Explain why $\binom{2^m}{k}$ is the number of ways to make a subset of $S$ of size $k$. (For example, if $m=2$ and $k=2$, then the number of subsets listed in
      \eqref{Estringsub} is  $\binom{4}{2}=6$.)
    \item Because of part (1), 
  we want to show the number of subsets of $S$ of size $k$ is even.
  Use part (3) of the previous problem to explain why this is true if $k$ is odd.
  
  \item To handle the case when $k$ is even, 
  we will use induction on $m$; the inductive hypothesis is that the statement is true for $m-1$, in other words, that $\binom{2^{m-1}}{k}$ is even for all $0<k<2^{m-1}$. 
  Write down what the base case is and verify that it is true.
\item If $T$ is a subset, 
recall the definition of the complement subset $T^c$ from 
Problem~\ref{Exercisedefcomplement} in 
Section~\ref{Sadditional2}.
For the inductive step,
we want to show that $\binom{2^m}{k}$ is even for all $0<k<2^m$. Suppose $k$ is even (because we are already done with the case that $k$ is odd).  Show that the number of subsets $T$ such that $T \not = T^c$ is even.
Show that the number of subsets $T$ such that $T=T^c$ is even using the inductive hypothesis and part (4) of the previous problem.
\end{enumerate}
\end{question}
Now we investigate a pattern in Pascal's triangle modulo $2$.
Look at the numbers in the rows $n=0$ to $n=3$.  Compare them with the numbers on the left side (then right side) of rows $n=4$ to $n=7$.
What do you notice?

\begin{question} \label{QzeroPas}
\begin{enumerate}
    \item If you have a green ball in the $n=4$ and $k=0$ spot in Pascal's triangle, and a gold ball in the $n=4$ and $k=4$ of Pascal's triangle, which is the closest spot of Pascal's triangle that they could both land?
    \item Repeat the previous part when  the green ball is at the $n=2^m$ and $k=0$ spot and the gold ball is at the $n=2^m $ and $k=n$ spot.  
    \item Explain why there is an upside down triangle of zeros in the middle of the rows from $n=2^m$ to $n=2^{m+1} -1$.
    
\end{enumerate} 
  
\end{question}

\begin{question}\label{self-similar}
In Pascal's triangle modulo $2$, explain why the rows $n=0$ to $n=2^m-1$ repeat again on the left hand and right hand sides of the 
rows $n=2^m$ to $n=2^{m+1} -1$.
\end{question}

The next question is a more precise version of the previous two.

\begin{question}
\begin{enumerate}
  \item Convince yourself that the following claim produces the upside down triangle of zeros described in Question~\ref{QzeroPas}:
  if $n=2^m + \ell$, where $0 \leq \ell \leq 2^m-2$, the $k$th entry of Pascal's triangle in the $n$th row is zero modulo $2$ for $\ell+1 \leq k \leq n-(\ell+1)$.

\item Convince yourself that the following claim produces the self-similar pattern 
described in Question~\ref{self-similar}:
if $0 \leq n \leq 2^m-1$ and $0 \leq k\le n$, then $$\binom{n}{k}\equiv \binom{n+2^m}{k}\pmod 2$$ and $$\binom{n}{k}\equiv \binom{n+2^m}{k+2^m}\pmod {2}.$$
\item Show that (reducing coefficients but not exponents modulo $2$)
\[(x+y)^{n+2^m} \equiv 
(x^{2^m} + y^{2^m})\cdot (x+y)^n \bmod 2.\]
\item Use part (3) to prove the claims in parts (1) and (2).
\end{enumerate}
\end{question}

\subsection*{Sierpinski's triangle}

\textit{Sierpinski's triangle} is a famous example of a \textbf{fractal}, which is a shape that exhibits \textit{self-similarity}.  In particular, Sierpinski's triangle is the unique 2-dimensional pattern that is made up of three shrunken copies of itself, each half of the total size, and arranged in a triangle.  Here is a sketch of a finite approximation of Sierpinski's triangle:
\begin{center}
\includegraphics{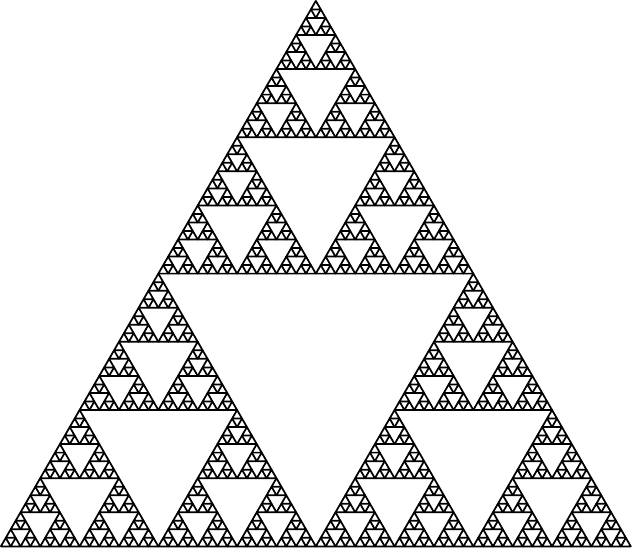}
\end{center}
One can obtain this drawing by starting with the smallest triangle at the top, then copying it to its lower left and lower right, and then repeating this process on the resulting shape several times.  Sierpinski's triangle itself is actually the ``limit'' of this diagram as the number of operations go to infinity.

Notice that the pattern of Sierpinski's triangle, formed by copying a triangle pattern down to the left and down to the right at each step, is exactly the same as the pattern we discovered about Pascal's triangle mod $2$ in Question \ref{self-similar}.  Indeed, the $1$'s in Pascal's triangle mod $2$ outline a Sierpinski triangle shape!  Demonstrate this for the first $16$ rows as follows:

\begin{question}
Use Sage to print out the first $16$ rows of Pascal's triangle mod $2$. 
(It's ok if it prints as a sequence of lists which is left justified, rather than as an equilateral triangle.)
Draw an approximation of Sierpinski's triangle on top of this printout by connecting the $1$'s with lines going in the three triangular directions.
\end{question}

There is another way to construct Sierpinski's triangle.
Let's start with an equilateral triangle, colored black.  We divide it into 4 smaller equilateral triangles and change the color of the middle one to white.  Suppose the sides of the original triangle have length $1$; then the sides of the smaller triangles have length $1/2$.  We then repeat this process for the $3$ smaller black triangles.  Repeating this process indefinitely produces the Sierpinski triangle.

%% file: 06-InclusionExclusion-Recursion/IncExcRecurrence.tex
\chapter{Recurrence relations}\label{chap:recurrence}

Many interesting sequences of integers are defined using \emph{recurrence relations}.
In this chapter, we explore fun examples of recurrence relations, including: the Fibonacci sequence, the tower of Hanoi, regions of the plane, derangements, and Catalan numbers.  
For the most part, we avoid treating the subject systematically, but in Section~\ref{Sotherlinear}, we describe linear recurrence relations and their characteristic polynomials in general.

\begin{definition}
A \emph{recursive relation} for an  infinite sequence $a_1, a_2, \ldots$
is a formula for $a_n$ in terms of the previous values in the sequence.
\end{definition}

\section{Fibonacci numbers}

\begin{videobox}
\begin{minipage}{0.1\textwidth}
\href{https://www.youtube.com/watch?v=VindfwTtkSA}{\includegraphics[width=1cm]{video-clipart-2.png}}
\end{minipage}
\begin{minipage}{0.8\textwidth}
Click on the icon at left or the URL below for this section's short lecture video. \\\vspace{-0.2cm} \\ \href{https://www.youtube.com/watch?v=VindfwTtkSA}{https://www.youtube.com/watch?v=VindfwTtkSA}
\end{minipage}
\end{videobox}

The Fibonacci numbers form a well-known sequence of numbers 
\[(1,1,2,3,5,8,13,21,34,55, \ldots).\]
They were first published in a book by Leonardo of Pisa in 1202 concerning the following question about rabbits:

\begin{question} \label{Qrabbit}
Suppose that rabbits mature in one month, that an adult pair of rabbits produces exactly one baby pair of rabbits each month, and that no rabbits ever die.  Starting with one pair of baby rabbits in month 1, how many pairs of rabbits will there be in month 17?
\end{question}

\begin{answer} We compute the following table of the number $F_n$ of pairs of rabbits in month $n$, starting with 1 rabbit.
\begin{center}
\begin{tabular}{|c|c|c|c|c|c|c|c|c|c|c|c|c|c|c|c|c|c|}
\hline
$n$ & 1 & 2 & 3 & 4 & 5 & 6 & 7 & 8 & 9 & 10 & 11 & 12 & 13 & 14 & 15 & 16 & 17 \\
\hline
$F_n$ & 1 & 1 & 2 & 3 & 5 & 8 & 13 & 21 & 34 & 55 & 89 & 144 & 233 & 377 & 610 & 987 & 1597 \\
\hline
\end{tabular}
\end{center}
\end{answer}

\begin{definition}
Set $F_1=1$ and $F_2=1$.  For $n\geq 3$,
the $n$-th Fibonacci number $F_n$ is defined by the formula $F_{n}=F_{n-1}+F_{n-2}$. 
\end{definition}

While the question about rabbits is completely unrealistic, Fibonacci numbers show up in nature in many unexpected places, including leaves on stems, pine cones, artichokes, pineapples, and family trees of bees\footnote{Ask CSU math professor Dr.\ Patrick Shipman about some of his work on the Fibonacci numbers!}.
There are also applications of Fibonacci numbers to many other scientific fields, including economics, logic, optics, and pseudo random number generators.

Here is another concrete question 
whose answer involves the Fibonacci numbers.

\begin{question} 
A staircase has $n$ steps.  How many ways can you go up the stairs if you can take one or two steps at a time and never go down?
\end{question}

For example, when $n=4$, the number of ways is 5 and the different ways to go up can be represented by $(1,1,1,1)$, $(1,1,2)$, $(1,2,1)$, $(2,1,1)$, or $(2,2)$.

Let $S_n$ be the number of ways to go up $n$ steps in this way.
We compute the following table.

\begin{center}
\begin{tabular}{|c|c|c|c|c|c|c|c|c|c|c|c|}
\hline
$n=$ \# stairs & 1 & 2 & 3 & 4 & 5 \\
\hline
 & 1 & \textcolor{blue}{1,1} & \textcolor{red}{1,1,1} & 1,\textcolor{red}{1,1,1} & 1,1,1,1,1 \\
& & \textcolor{blue}{2} & \textcolor{red}{1,2} & 1,\textcolor{red}{1,2} & 1,1,1,2 \\
& & & \textcolor{red}{2,1} & 1,\textcolor{red}{2,1} & 1,1,2,1 \\
sequence of number of steps taken & & & & 2,\textcolor{blue}{1,1} & 1,2,1,1 \\
& & & & 2,\textcolor{blue}{2} & 1,2,2 \\
& & & & & 2,1,1,1 \\
& & & & & 2,1,2 \\
& & & & & 2,2,1 \\
\hline
$S_n=$ \# ways to climb & 1 & 2 & 3 & 5 & 8 \\
\hline

\end{tabular}
\end{center}

Notice, in order to climb 4 steps, you must first climb 2 steps at a time and then 2 more steps, or 1 step and then 3 more steps as indicated in blue and red, respectively. Based on this example, we conjecture that $S_4=S_3+S_2.$ We will prove more generally that the stair climbing problem is in fact a recurrence relation in Lemma \ref{lem:stairs}, but for now let's take a  deeper look into the example of 4 stairs. 

\begin{example}
\noindent \textbf{Picture of $S_4=S_3+S_2$.}
\begin{center}
\includegraphics[width=2in]{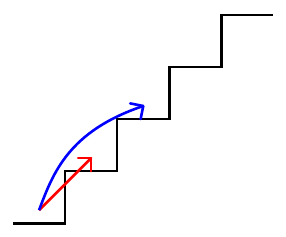}
\end{center}

After the red step of size 1, there are $S_3=3$ ways to proceed: $1,1,1$ or $1,2$ or $2,1$.

After the blue step of size 2, there are $S_2=2$ ways to proceed: $1,1$ or $2$.

In total there are $S_4=S_3+S_2$ ways:
\[\underbrace{1,\textcolor{red}{1,1,1}\hspace{8mm} 1,\textcolor{red}{1,2}\hspace{8mm} 1,\textcolor{red}{2,1}}_{\text{from $S_3$}}\hspace{10mm}\underbrace{2,\textcolor{blue}{1,1}\hspace{8mm} 2,\textcolor{blue}{2}}_{\text{from $S_2$}}\]
\end{example}

The values of $S_n$ are all Fibonacci numbers, but they are shifted to the left compared to the values of $F_n$ in the table for Question~\ref{Qrabbit}.
For example, $S_5 = F_6=8$.

Based on this, we write down the following lemma.
We prove the lemma using strong induction as described in Remark~\ref{Rstrongind}.

\begin{lemma}\label{lem:stairs}
The relationship between the number of ways to climb the stairs and the Fibonacci numbers is that $S_n=F_{n+1}$ for all $n \geq 1$.  
\end{lemma}

\begin{proof}
In order for this to be true, we need to see whether the values of $S_n$ satisfy the same recurrence relation as the values of $F_n$.
We claim that 
$S_{n}=S_{n-1}+S_{n-2}$ for $n \geq 3$.

If the claim is true, then we can use strong induction to show that $S_n=F_{n+1}$ for all $n \geq 1$.  Here is how that works.
The base cases $n=1$ and $n=2$ are true because 
$S_1 = 1 = F_2$ and $S_2 = 2 = F_3$.
Suppose that $n \geq 3$ and $S_i=F_{i+1}$
for all $1 \leq i \leq n-1$.
Using the claim, the inductive hypothesis, and the formula for $F_{n+1}$ shows that
\[S_n = S_{n-1} + S_{n-2} 
= F_{n} + F_{n-1}
= F_{n+1}.\]

So we need to check that $S_{n}=S_{n-1}+S_{n-2}$ for $n \geq 3$.
Every sequence of steps up the $n$ stairs starts with either
\begin{itemize}
\item a size 1 step, after which there are $S_{n-1}$ ways to climb the remaining $n-1$ stairs; or
\item a size 2 step, after which there are $S_{n-2}$ ways to climb the remaining $n-2$ stairs.
\end{itemize}
By the addition principle for sets, the number of sequences of steps up the $n$ stairs is 
$S_{n-1} + S_{n-2}$, verifying the claim.
\end{proof}

Another example of the Fibonacci numbers occurring in graph theory can be found in 
Example~\ref{EgraphwalkFibo}.

\subsection{The golden ratio and a formula for the Fibonacci numbers}

One drawback with recurrence relations is that it is time-consuming to compute the values in the sequence.
For example, it would be nice to find the value of $F_{100}$, without needing to compute the value of $F_n$ for $1 \leq n \leq 99$.
In this section, we describe a 
\emph{closed form formula} for the Fibonacci numbers.

The \emph{golden ratio} is the number $\alpha = \frac{1+\sqrt{5}}{2}$.
Its \emph{conjugate} 
is $\bar{\alpha} = \frac{1-\sqrt{5}}{2}$.
Note that $\alpha \approx 1.61803\ldots$ and $\bar{\alpha} \approx -0.61803\ldots$. 
Using the quadratic formula, we see that $\alpha$ and $\bar{\alpha}$ are
the roots of the quadratic polynomial 
$f(x)=x^2-x-1$. 

The golden ratio is an important number for the ratio of length to height in Greek architecture and natural objects like sea shells.

\begin{theorem} \label{TclosedFibo}
For all $n \geq 0$, the Fibonacci number $F_n$ is given by the formula
\begin{equation} \label{EclosedFn}
F_n = \frac{1}{\sqrt{5}}\left(\left(\frac{1+\sqrt{5}}{2}\right)^n-\left(\frac{1-\sqrt{5}}{2}\right)^n\right).
\end{equation}
\end{theorem}

\begin{remark}
This is a surprising statement!  Initially, the right-hand side of \eqref{EclosedFn} does not even look like an integer!  Also, it is not clear why there is a connection between the Fibonacci numbers and the golden ratio.
\end{remark}

\begin{remark}
Because $|\bar{\alpha}|<1$, the powers of $\bar{\alpha}$ approach $0$. 
As a consequence, $F_n \approx \alpha^n$ for $n$ large.
\end{remark}

We will not prove Theorem~\ref{TclosedFibo} but there is some explanation of why it is true
in Section~\ref{Sotherlinear}, see Example~\ref{Efiboclosed}.  It is also possible to prove it using induction.

\subsection*{Exercises}

\begin{enumerate}

\item Consider the sequence defined by $a_0=2$, $a_1=1$, and for all $n\geq1$, $a_n=a_{n-1}+2a_{n-2}$.			Compute the first several terms of the sequence. What is $a_5$?

\item Let $D_n$ be the number of ways that an $n$-by-$2$ board can be covered by $2$-by-$1$ dominos, so that no dominos overlap.
\begin{enumerate}
\item Show that $D_1 = 1$, $D_2 = 2$, and $D_3 = 3$.  
\item Show that $D_n=F_{n+1}$ for $n \geq 1$.
\end{enumerate}

    \item 
    In a subway car, there is a row of $n$ seats.
    Find the number $S_n$ of ways that the seats can be filled by an arbitrary number of people, so that no two people sit next to each other.
    For example,
    $S_1=2$ since the one seat can be empty or full.  Letting $e$ denote empty and $f$ denote full, we see that $S_2=3$, because the choices are $ee$, $fe$, or $ef$.  Show that 
    $S_n=F_{n+2}$ for $n \geq 3$.

\item 
    A fence is constructed from $n$ posts.
Prove that $F_{n+2}$ is equal to the number of ways to paint each post either silver or navy, in such a way that no two silver posts are next to each other.

\item How many subsets of the set $\{1,2,3,...,9\}$ contain no two consecutive integers?

\item What number do the successive quotients $\frac{F_2}{F_1}$, $\frac{F_3}{F_2}$, $\frac{F_4}{F_3}$, \ldots, $\frac{F_n}{F_{n-1}}$, $\frac{F_{n+1}}{F_n}$, \ldots approach as $n$ gets larger and larger? 

\item
 Using (strong) induction, prove that $F_n\le 2^{n-1}$ for all $n \geq 1$.

\item Here is a table of the sum of the first $n$ Fibonacci numbers.

\begin{center}
\begin{tabular}{|c|c|c|c|c|c|c|c|c|c|}
\hline
$n$ & 1 & 2 & 3 & 4 & 5 & 6 & 7 & 8 \\
\hline
$F_1+F_2+\cdots +F_n$ & 1 & 2 & 4 & 7 & 12 & 20 & 33 & 54 \\
\hline
\end{tabular}
\end{center}

For $n\ge 1$, prove that $F_1+F_2+\cdots+F_n=F_{n+2}-1$ by induction on $n$.





\item Let $F_n$ be the $n$-th Fibonacci number. Show that for all $n\ge 1$ we have \[F_1+F_3+F_5+\cdots+F_{2n-1}=F_{2n}.\]

\item Let $F_n$ be the $n$-th Fibonacci number. Prove that $F_1^2+\cdots+F_n^2=F_nF_{n+1}$ for all $n\ge1$.

\item True or False: Starting with $1,1,2$, the Fibonacci numbers rotate $odd, odd, even$, then $odd, odd, even$, then $odd, odd, even$, etc.

\item What happens to the values of the Fibonacci sequence modulo $3$?

\end{enumerate}

\section{Linear recurrence relations} \label{Sotherlinear}

\begin{videobox}
\begin{minipage}{0.1\textwidth}
\href{https://www.youtube.com/watch?v=8QMJNWQQ0lc}{\includegraphics[width=1cm]{video-clipart-2.png}}
\end{minipage}
\begin{minipage}{0.8\textwidth}
Click on the icon at left or the URL below for this section's short lecture video. \\\vspace{-0.2cm} \\ \href{https://www.youtube.com/watch?v=8QMJNWQQ0lc}{https://www.youtube.com/watch?v=8QMJNWQQ0lc}
\end{minipage}
\end{videobox}

The examples in the last section were all related to the Fibonacci numbers.  In this section, we explore some of the many other recurrence relations. We explain how to implement recurrence relations in SAGE.
We define linear recurrence relations and describe the characteristic polynomial method of solving them.

\subsection{Another recurrence relation}
Here is a problem that introduces a new recurrence relation.

\begin{example}
How many ways can you go up a staircase with 9 steps if you can take 1, 2, or 3 steps at a time? 
\end{example}

\begin{brainstorm} $ $
\begin{center}
\begin{tabular}{|c|c|c|c|c|c|c|c|c|c|c|}
\hline
\# stairs & 1 & 2 & 3 & 4 \\

\hline
 & \textcolor{green}{1} & \textcolor{red}{1,1} & \textcolor{blue}{1,1,1} & 1,\textcolor{blue}{1,1,1} \\
& & \textcolor{red}{2} & \textcolor{blue}{1,2} & 1,\textcolor{blue}{1,2} \\
& & & \textcolor{blue}{2,1} & 1,\textcolor{blue}{2,1} \\
sequence of number of steps taken & & & \textcolor{blue}{3} & 1,\textcolor{blue}{3} \\
& & & & 2,\textcolor{red}{1,1} \\
& & & & 2,\textcolor{red}{2} \\
& & & & 3,\textcolor{green}{1} \\
\hline
\# ways & 1 & 2 & 4 & 7 \\
\hline
\end{tabular}
\end{center}
\end{brainstorm}

\begin{answer}
Let $T_n=\#$ ways to climb a staircase with $n$ steps.
We claim that $T_n=T_{n-1}+T_{n-2}+T_{n-3}$.
The reason is that every sequence of steps up the $n$ stairs starts with either
\begin{itemize}
\item a size 1 step, after which there are $T_{n-1}$ ways to climb the remaining $n-1$ stairs, or
\item a size 2 step, after which there are $T_{n-2}$ ways to climb the remaining $n-2$ stairs, or
\item a size 3 step, after which there are $T_{n-3}$ ways to
climb the remaining $n-3$ stairs.
\end{itemize}

Using this recurrence relation and the base cases found by hand, we compute:
\begin{center}
\begin{tabular}{|c|c|c|c|c|c|c|c|c|c|c|}
\hline
$T_1$ & $T_2$ & $T_3$ & $T_4$ & $T_5$ & $T_6$ & $T_7$ & $T_8$ & $T_9$ \\
\hline
1 & 2 & 4 & 7 & 13 & 24 & 44 & 81 & 149 \\
\hline
\end{tabular}
\end{center}
Hence there are 149 ways to climb this staircase with 9 stairs.
\end{answer}

\begin{question}
Find $T_{10}$, the number of ways to climb $10$ stairs with these rules.
\end{question}

\subsection{SAGE commands for recurrence relations}

It is tedious to use a recurrence relation to compute numerous values in a sequence.
Now we will see how to use SAGE to do this work efficiently.

\begin{videobox}
\begin{minipage}{0.1\textwidth}
\href{https://youtu.be/IstO1arARxI}{\includegraphics[width=1cm]{video-clipart-2.png}}
\end{minipage}
\begin{minipage}{0.8\textwidth}
Click on the icon at left or the URL below for this section's short lecture video. \\\vspace{-0.2cm} \\ \href{https://youtu.be/IstO1arARxI}{https://youtu.be/IstO1arARxI}
\end{minipage}
\end{videobox}

Let's first set-up the recursive definition of the sequence $T_n$. (Note that the indentation of these commands is quite important in SAGE and may not be accurately reflected here.)

Sage commands:

\begin{verbatim}
def T(x):
    if x ==1:
        return 1;
    if x ==2:
        return 2;
    if x==3:
        return 4
    else:
        return T(x-1)+T(x-2)+T(x-3)
\end{verbatim}

Now we can print the first several values of the sequence $T_n$.

\begin{verbatim}
[print(T(n)) for n in [1..10]]
\end{verbatim}

\begin{example}
This is the recursive definition of the fibonacci sequence:
\begin{verbatim}
def fibo(x):
    if x ==1:
        return 1;
    if x ==2:
        return 1;
    else:
        return fibo(x-1)+fibo(x-2)
\end{verbatim}
Use SAGE to print the first 18 Fibonacci numbers.
\end{example}

\begin{example}
\label{ex:wasp}
A colony of wasps is building a nest in your roof.  Let ${\rm wasp}_n$ be the number of worker wasps (other than the queen) after $n$ days.
Suppose ${\rm wasp}_1=1$ and ${\rm wasp}_2=1$.
For $n \geq 3$, suppose that the number of wasps increases by this recurrence relation:
\[{\rm wasp}_n={\rm wasp}_{n-1}+20 \cdot {\rm wasp}_{n-2}.\]
\end{example}

\begin{question}
Set-up the wasp recurrence relation in SAGE and use SAGE to find the number of wasps on days 1-10. 
\end{question}

\begin{question}
What happens if you try to find the 100th entry of any of the sequences above?  Try to compute $T(100)$, ${\rm fibo}(100)$ or ${\rm wasp}(100)$ using SAGE.
\end{question}

\subsection*{Exercises}
\begin{enumerate}

\item Let $a_0,a_1,a_2,\ldots $ be the sequence defined by $a_0=4$,$a_1=5$, and $a_n=a_{n-1}+2a_{n-2}$. Then using the roots of the characteristic polynomial, we know that $a_n$ has the explicit formula $b\cdot2^n+c\cdot(-1)^n$ for some values of $b$ and $c$. Find $b$.															
\item Consider the sequence defined by $L_0=0,$ $L_1=1$, and $L_n=3L_{n-1}-L_{n-2}$. Implement this recursion in Sage in order to compute $L_{20}.$		
															
\end{enumerate} 
\subsection{Solving linear recurrence relations of degree 2}

Suppose $a_n$ is a sequence defined from a recurrence relation
and some initial values.
The computer is very fast at computing the first several entries of $a_n$, but gets very slow at computing entries later in the sequence.
In this section, we explain how to find a closed-form formula for $a_n$ when the recurrence relation is linear of degree $2$.  In the next section, we cover similar material for linear recurrence relations of higher degree.

\begin{definition}
An \emph{linear recurrence relation of order $2$} is an equation of the form
\begin{equation}
\label{Elinearrec2}
a_{n}=c_{1}a_{n-1}+c_{2}a_{n-2}.
\end{equation}
Here the values $c_1$ and $c_2$ are constants and $c_2 \not = 0$.
The \emph{initial values} of a linear recurrence relation of order $2$ are a choice of the numbers 
$a_1$ and $a_2$. 
\end{definition}

For example, the Fibonacci recurrence relation is linear of
order $2$ with constants $c_1=c_2=1$, and initial values $a_1=a_2=1$.
The wasp recurrence relation is linear of order $2$ with constants $c_1=1$ and $c_2=20$ and initial values $a_1=a_2=1$. 

\begin{definition}
The \emph{characteristic polynomial} of \eqref{Elinearrec2} is the degree $2$ polynomial
\[c(x) = x^2 - c_1x - c_2.\] 
\end{definition}

By the quadratic formula, the roots of $c(x)$ are 
\[\frac{c_1 \pm \sqrt{c_1^2 + 4c_2}}{2}.\]
We call these roots $r_1$ and $r_2$.

\begin{theorem} \label{Tlinchar2}
Let $r_1$ and $r_2$ be the roots of the characteristic polynomial.
Suppose that $r_1 \not = r_2$.
Then there are constants $z_1$ and $z_2$ such that, if $n \geq 1$, then
\begin{equation} \label{Esolrec2}
    a_n = z_1 r_1^n + z_2 r_2^n.
    \end{equation}
Furthermore, there is only one choice of $z_1$ and $z_2$ such that \eqref{Esolrec} is true for all $n$.
\end{theorem}

In general, the way to find $z_1$ and $z_2$ is to use \eqref{Esolrec2} when $n=1$ and $n=2$.  We solve for $z_1$ and $z_2$, using that we know $a_1,a_2,r_1,r_2$ using these equations:
\[a_1 = z_1 r_1^1 + z_2 r_2^1 \text{ and } a_2 = z_1 r_1^2 + z_2 r_2^2.\]
For example, the characteristic polynomial for the Fibonacci recurrence relation is $c(x)=x^2-x-1$.  Its roots are the golden ratio $\alpha = (1+\sqrt{5})/2$
and its conjugate $\bar{\alpha} = (1-\sqrt{5})/2$.

We explain how to find the roots and the numbers $z_1$ and $z_2$ in SAGE.
First, we define a ring of polynomials, define the characteristic polynomial $c(x)$ and use SAGE to find its roots.

\begin{verbatim}
R = PolynomialRing(RR, 'x')
c=x^2-x-1
c.roots();
\end{verbatim}

The output looks a little complicated. This notation tells us what the roots are and that each root happens exactly once.
\begin{verbatim}
[(-1/2*sqrt(5) + 1/2, 1), (1/2*sqrt(5) + 1/2, 1)]
\end{verbatim}

Let's extract the roots from this complicated notation (switching their order to make things look better at the end).  Remembering that lists are indexed starting at $0$, for the root $r_2$ we take the $0$th entry and then the $0$th entry of that;
for the root $r_1$ we take the $1$st entry and then the $0$th entry of that.

\begin{verbatim}
r2=c.roots()[0][0];
r1=c.roots()[1][0];
\end{verbatim}

Now we need to find the constants $z_1$ and $z_2$.  To do this, we define them as variables and use \eqref{Esolrec2} when $n=1$ and $n=2$.

\begin{verbatim}
var('z1,z2')
eqnfibo = [z1*r1 + z2*r2==1, z1*r1^2 + z2*r2^2==1]; 
Efibo = solve(eqnfibo, z1,z2); Efibo
\end{verbatim}

This shows us that $z_1=(1/5)\sqrt{5} = 1/\sqrt{5}$ and $z_2=-(1/5)\sqrt{5} = -1/\sqrt{5}$.
We then define the closed form formula, use it to check the case when $n=3$ and then find the $100$th Fibonacci number.

\begin{verbatim}
s=sqrt(5)
def Fibo(n): return expand((1/s)*(r1^n - r2^n));
Fibo(3);
Fibo(100);
\end{verbatim}

\begin{question}
Find the characteristic polynomial of the wasp recurrence relation from Example~\ref{ex:wasp}.
Use SAGE to find its roots and then the constants $z_1$ and $z_2$.  
Check your work by computing the third wasp number.  Find the 100th wasp number.
\end{question}

\subsection{Linear recurrence relations of higher order}

More generally, we can set up a recurrence relation as follows.
Fix a positive integer $d$. 

\begin{definition}
An \emph{linear recurrence relation of order $d$} is an equation of the form
\begin{equation}\label{Elinearrec}
a_{n}=c_{1}a_{n-1}+c_{2}a_{n-2}+\cdots +c_{d}a_{n-d}.
\end{equation}
Here the values $c_1, \ldots, c_d$ are constants and $c_d \not = 0$.
The \emph{initial values} of a linear recurrence relation of order $d$ are a choice of the numbers 
$a_1, \ldots, a_{d-1}$.
\end{definition}

Sometimes this is called an order $d$ 
homogeneous linear recurrence with constant coefficients; what a mouthful!

The stair climbing problem in this section is a linear recurrence relation of order $d=3$ with constants
$c_1=c_2=c_3=1$, and initial values $a_1=1$, $a_2 =2$, $a_3=4$.

Given a linear recurrence relation of order $d$ and initial values $a_1, \ldots, a_{d-1}$, we briefly describe how to find a closed form formula for the values $a_n$ for large $n$.

\begin{definition}\label{def:charpoly}
The \emph{characteristic polynomial} of \eqref{Elinearrec} is the degree $d$ polynomial
\[c(x) = x^d - c_1x^{d-1} - \cdots - c_{d-1} x - c_d.\] 
\end{definition}

\begin{theorem}
Let $r_1, \ldots, r_d$ be the roots of the characteristic polynomial.  Suppose that the $d$ roots are all different.
Then there are constants $z_1, \ldots, z_d$ such that, if $n \geq 1$, then
\begin{equation} \label{Esolrec}
    a_n = z_1 r_1^n + \cdots z_d r_d^n.
    \end{equation}
Furthermore, there is only one choice of constants $z_1, \ldots, z_n$ such that \eqref{Esolrec} is true for all $n$.
\end{theorem}

We won't see the proof of this theorem until Chapter \ref{chap:generatingfunctions}.  Here is how to use the theorem to find an explicit formula for the Fibonacci numbers without using SAGE.

\begin{example} 
\label{Efiboclosed}
The characteristic polynomial of the Fibonacci recurrence relation is $c(x)=x^2-x -1$.  By the quadratic formula, this has roots $r_1 = (1+\sqrt{5})/2$ and 
$\bar{\alpha}=(1-\sqrt{5})/2$.
So $a_n = z_1 \alpha +z_2 \bar{\alpha}$.
Substituting $n=1$ and $n=2$, we obtain the following system of equations for $z_1$ and $z_2$:
\begin{eqnarray*} 
1 & = & z_1 \frac{1+\sqrt{5}}{2} + z_2 \frac{1-\sqrt{5}}{2} \\
1 & = & z_1 \left(\frac{1+\sqrt{5}}{2}\right)^2 + z_2 \left(\frac{1-\sqrt{5}}{2}\right)^2
\end{eqnarray*}
A long computation shows that $z_1=1/\sqrt{5}$ and $z_2=-1/\sqrt{5}$.
This shows where Theorem~\ref{TclosedFibo} comes from!
\end{example}

\subsection*{Exercises}

\begin{enumerate}

\item Define the \emph{Lucas numbers} by $L_0=2$, $L_1=1$, and $L_{n}=L_{n-1}+L_{n-2}$ for $n\ge 2$. Find $L_{10}$.

\item How many ways are there to climb a staircase with $9$ steps if you can take either one or three steps at a time?
\begin{enumerate}
    \item Find how many ways to climb $n=1,2,3,4$ steps 
    if you can take either one or three stairs at a time.
    \item Find a recurrence relation for this problem and explain why it is true. 
    \item 
Compute what happens for $n=9$ steps using the recurrence relation.
\end{enumerate}

\item In how many ways can you cover a $3 \times 10$ grid of squares with identical dominoes, where each domino is of size $3 \times 1$ (or $1 \times 3$ if you turn it sideways), and you must use exactly $10$ dominoes?\\
\emph{Remark: Each domino covers 3 squares, and so the 10 dominoes cover the 30 squares exactly}.

\item
You have \$9 dollars to spend. Each day at lunch you buy exactly one item: either an apple for \$1, or a yogurt for \$2, or a sandwich for \$4. You continue buying one item each day until you have exactly \$0. This could take anywhere from 9 days, if you buy 9 apples in a row, to 3 days, if you buy 2 sandwiches in a row then an apple. Buying two sandwiches then an apple is different from buying an apple then two sandwiches. In how many different ways can you spend your money?

\item
 You have $n$ dollars to spend. Each day you buy exactly one item: either a banana for \$1 or a bagel for \$4, and you continue until you have exactly \$0. Let $S_n$ be the number of ways you can spend your $n$ dollars (buying a banana the first day and a bagel the second is different from a bagel the first day and a banana the second). Find and explain a recurrence relation for $S_n$.

\item Solve the linear recurrence relation of order $2$ given by
$a_n = 5a_{n-1} -6a_{n-2}$ if $a_1=5$ and $a_2=13$.

\end{enumerate}

\section{The tower of Hanoi}\label{sec:hanoi}

The tower of Hanoi is a game with $3$ pegs.  On the left hand peg, there are disks of different sizes (each with a hole in its center)
stacked in order of size from largest on the bottom to smallest on the top.
On each move, you can transfer one of the disks to another peg, without ever putting a larger disk on top of a smaller one.  The object of the game is to move all the disks to another peg.

\begin{videobox}
\begin{minipage}{0.1\textwidth}
\href{https://www.youtube.com/watch?v=2HJqtgiD6Y8&list=PL5J6K3znOvOmzBUoxlk-W0N4j7L1Y9yfW&index=20}{\includegraphics[width=1cm]{video-clipart-2.png}}
\end{minipage}
\begin{minipage}{0.8\textwidth}
Click on the icon at left or the URL below for this section's short lecture video. \\\vspace{-0.2cm} \\ \href{https://www.youtube.com/watch?v=2HJqtgiD6Y8}{https://www.youtube.com/watch?v=2HJqtgiD6Y8}
\end{minipage}
\end{videobox}

\begin{question}
Can the tower of Hanoi be represented as a recurrence relation?
\end{question}
\begin{brainstorm}
Let $H_n$ be the minimal number of moves needed to finish the game for $n$ disks.

$H_1=1$: if there is one disk, just move it to another peg.

$H_2=3$: first move the small disk to the middle peg, then move the big disk to the right peg, then put the small disk on top of the big disk on the right peg.

In the exercises, we will show that
$H_3 = 7$ and $H_4 = 15$.
\end{brainstorm}

\begin{lemma} \label{LHanoirec}
The recurrence relation for the tower of Hanoi is $H_1=1$ and $H_n = 2H_{n-1} + 1$ for $n\ge 2$.
\end{lemma}

\begin{proof}
By definition, it takes $H_{n-1}$ moves to transfer all the disks except the biggest one to another peg.  Then it takes $1$ move to move the biggest disk to the empty peg.  Then, by definition, it takes $H_{n-1}$ moves to transfer all the other disks on top of the biggest one.
\end{proof}

\begin{lemma}
A closed form formula for the minimal number of moves in the tower of Hanoi game is 
$H_n=2^n-1$.
\end{lemma}

\begin{proof}
The formula is true in the base case $n=1$.
The inductive hypothesis is that 
$H_{n-1} = 2^{n-1} -1$.
By Lemma~\ref{LHanoirec}, 
$H_n=2 H_{n-1} + 1$.
So $H_n=2(2^{n-1} -1) + 1 = 2^n -2 +1 = 2^n-1$.  So the result is true by induction.
\end{proof}

\subsection*{Exercises}
\begin{enumerate}

\item How many steps does it take to solve the Tower of Hanoi problem with 5 disks?										
								
\item True or false: the recursion for the Tower of Hanoi problem is a homogeneous linear recurrence that can be solved with the methods of section \ref{Sotherlinear}.

\item When $n=3$, label the disks $d_1$ (smallest) to $d_3$ (largest).  Write out a sequence of $7$ moves to show that $H_3=7$.
    \item 
    When $n=4$, label the disks $d_1$ (smallest) to $d_4$ (largest).
    Write out a sequence of $15$ moves to show that $H_4=15$.
    \item In the sequences of moves found in problems 1 and 2, which disks does $d_1$ rest on along the way?
    \item When $n=5$,
    label the disks $d_1$ (smallest) to $d_5$ (largest).  In order to solve the tower of Hanoi in 31 steps, what rules do you need to follow about
    where the disks $d_1$ and $d_2$ can rest along the way?
    
\item In order to solve the tower of Hanoi in the minimal number of steps, the moves must be chosen carefully.  Write an algorithm to solve the tower of Hanoi most efficiently.   
\end{enumerate}

\section{Regions of the plane} \label{sec:regions-plane}

\begin{videobox}
\begin{minipage}{0.1\textwidth}
\href{https://www.youtube.com/watch?v=CSUmdpc5dB8}{\includegraphics[width=1cm]{video-clipart-2.png}}
\end{minipage}
\begin{minipage}{0.8\textwidth}
Click on the icon at left or the URL below for this section's short lecture video. \\\vspace{-0.2cm} \\ \href{https://www.youtube.com/watch?v=CSUmdpc5dB8}{https://www.youtube.com/watch?v=CSUmdpc5dB8}
\end{minipage}
\end{videobox}

On a piece of paper, draw non-parallel lines so that at most $2$ lines intersect at a point.
We define $P_n$ to be the number of regions formed if you draw $n$ such lines.

\begin{example} $ $
\begin{figure}[h]
\begin{center}
\subfigure[$P_1=2$]{\includegraphics[width=1.4in]{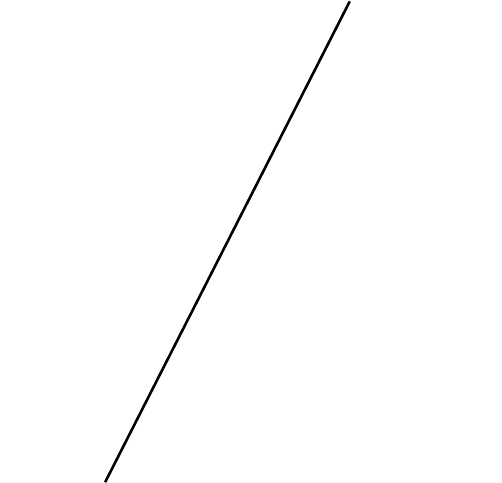}}
\hspace{1mm}
\subfigure[$P_2=4=P_1+2$]{\includegraphics[width=1.4in]{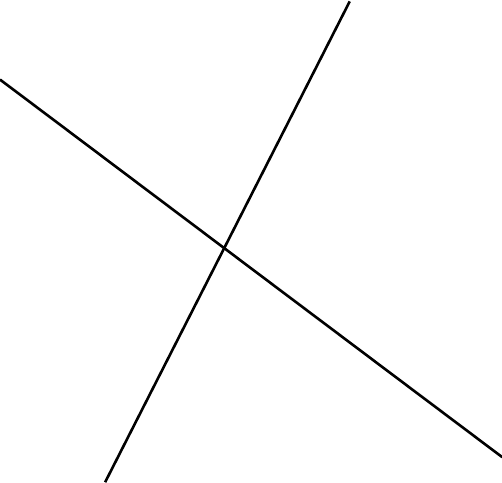}}
\hspace{1mm}
\subfigure[$P_3=7=P_2+3$]{\includegraphics[width=1.4in]{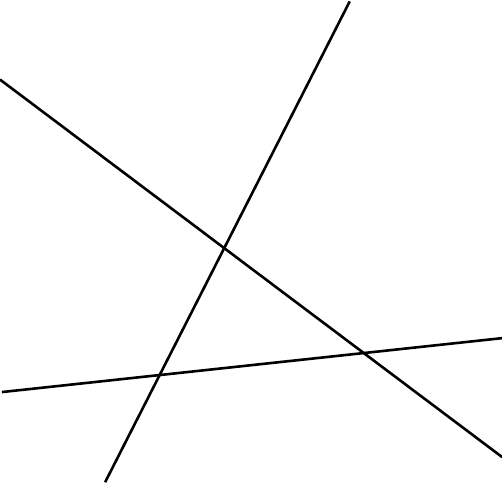}}
\hspace{1mm}
\subfigure[$P_4=11=P_3+4$]{\includegraphics[width=1.4in]{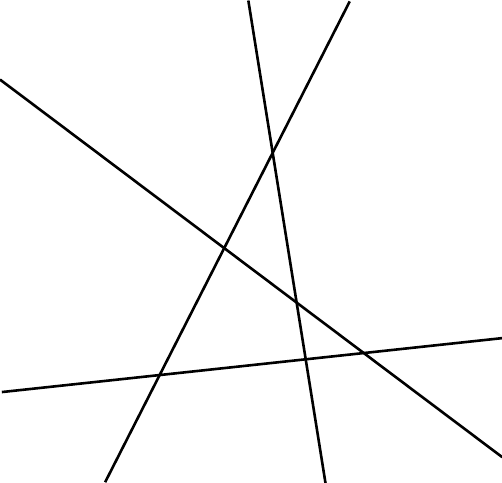}}
\end{center}
\end{figure} 
\end{example}

\begin{lemma}
The recurrence relation for the number $P_n$ of regions formed by drawing $n$ non-parallel lines so that at most two lines intersect in a point is $P_0=1$ and  $P_n=P_{n-1}+n$ for $n \geq 1$.
\end{lemma}

\begin{proof}
If we have no lines, there is just one region, the entire plane, so $P_0=1$.  Then, suppose we start with $n-1$ lines, for which there are $P_{n-1}$ regions. The $n$-th line crosses $n-1$ other lines, and hence divides $n$ of these regions in two. This shows that $P_n=P_{n-1}+n$.
\end{proof}

\begin{example}
Using the recurrence relation, we can compute more values of $P_n$.
\begin{center}
\begin{tabular}{|c|c|c|c|c|c|c|c|c|c|c|c|}
\hline
$P_0$ & $P_1$ & $P_2$ & $P_3$ & $P_4$ & $P_5$ & $P_6$ & $P_7$ & $P_8$ & $P_9$ & $P_{10}$ \\
\hline
1 & 2 & 4 & 7 & 11 & 16 & 22 & 29 & 37 & 46 & 56 \\
\hline
\end{tabular}
\end{center}
\end{example}

\begin{proposition} 
A closed form formula for the number of regions is
 $P_n=\frac{n^2+n+2}{2}$ for all $n\ge1$.
 \end{proposition}

\begin{proof}
One way to prove this is by induction.\\

\noindent\emph{Base case.} Note $\displaystyle P_1=2=\frac{1^2+1+2}{2}$.\\

\noindent\emph{Inductive step.} Assume $\displaystyle P_n=\frac{n^2+n+2}{2}$.
To prove: $\displaystyle P_{n+1}=\frac{(n+1)^2+(n+1)+2}{2}=\frac{n^2+3n+4}{2}$.

Note
\begin{align*}
P_{n+1} &= P_n + (n+1) &&\text{by the recurrence relation}\\
&=\frac{n^2+n+2}{2}+(n+1)&&\text{by the inductive assumption}\\
&=\frac{n^2+n+2}{2}+\frac{2n+2}{2}\\
&=\frac{n^2+3n+4}{2}.
\end{align*}
Hence we are done by induction.
\end{proof}

\subsection*{Exercises}

\begin{enumerate}
\item What is the largest number of regions that can be cut out using 11 lines in the plane?		

\item Let's draw $n$ circles in the plane.  We say that they are in \emph{general position} if each pair of circles intersects in exactly two points.
Let $a_n$ be the number of regions of the plane formed by $n$ circles in general position.  For example, $a_1=2$ because there are two regions (inside or outside the circle).

\begin{enumerate}
    \item Show that $a_2=4$, $a_3=8$, and $a_4=14$.
    \item Explain why the $n$th circle intersects the other $n-1$ circles in $2(n-1)$ points.
    \item Label these points as $P_1, \ldots P_{2(n-1)}$ so that they are arranged clockwise on the $n$th circle.
    Explain why each of the arcs below separates a region formed by the $n-1$ circles into $2$ regions:
    \[P_1P_2, \ P_2P_3, \ldots, P_{2(n-1)}P_1.\]
    \item Show that $a_n=a_{n-1} + 2(n-1)$.
    \item Use induction to show that $a_n=n^2-n+2$. 
    
\end{enumerate}

\end{enumerate}

\section{Derangements}

\begin{videobox}
\begin{minipage}{0.1\textwidth}
\href{https://youtu.be/zv0G3nCpkes}{\includegraphics[width=1cm]{video-clipart-2.png}}
\end{minipage}
\begin{minipage}{0.8\textwidth}
Click on the icon at left or the URL below for this section's short lecture video. \\\vspace{-0.2cm} \\ \href{https://youtu.be/zv0G3nCpkes}{https://youtu.be/zv0G3nCpkes}
\end{minipage}
\end{videobox}

\begin{example} \label{Ederange3}
How many ways are there to rearrange the letters in MAT, so that no letter is in its initial position?
The answer is only 2 (out of the $3!=6$ permutations): namely ATM and TMA.
\end{example}

\begin{example}
\label{Ederange4}
How many ways are there to rearrange the letters in MATH, so that no letter is in its initial position?
The answer is only 9 (out of the $4!=24$ permutations): 
namely:\\
AMHT, ATHM, AHMT, 
TMHA, THMA, THAM, 
HMAT, HTMA, HTAM.
\end{example}

It is possibly to check Examples~\ref{Ederange3} and \ref{Ederange4} using the inclusion-exclusion principle, but it is not so easy.

\begin{definition}
A \textit{derangement}
of a sequence
$a_1,a_2, \ldots, a_n$ is a permutation such that no element $a_i$ appears in its initial position (the $i$th spot).
Let $D_n$ be the number of derangements of a sequence of $n$ distinct letters.
\end{definition}

We can compute that $D_1=0$, $D_2=1$, $D_3=2$, and $D_4=9$.  
To compute $D_n$ when $n$ is larger, it is helpful to have a recurrence relation.

\begin{proposition}
\label{Pderange}
For $n \geq 3$, the number $D_n$ of derangements of $n$ objects satisfies the recurrence relation
\begin{equation}
 \label{Ederange1}  
D_n = (n-1)(D_{n-1} + D_{n-2}).
\end{equation}
\end{proposition}

\begin{proof}
Let's work with derangements of the sequence $1,2,3,\ldots,n$.
Given a derangement, let $c$ be the number of the spot where $1$
appears.  Since $c\neq 1$ (else this would  not be a valid derangement), there are $n-1$ choices for $c$.  Once $c$ is fixed,
we have two cases.

Case 1: the numbers $1$ and $c$ switched spots.
For example, when $n=5$ and $c=3$, the derangement 
$34152$ is in this case.
Ignoring the first and the $c$th spot, the other $n-2$ numbers are a derangement of their initial positions, and there are $D_{n-2}$ ways to make that derangement.

Case 2: the numbers $1$ and $c$ did not switch spots.
For example, when $n=5$ and $c=3$, the derangement
$23154$ is in this case.
We are left with $n-1$ numbers ($2,3,\ldots, n$)
that need to fit into $n-1$ spots ($1,2, \ldots, n$, excluding $c$).  Each number is restricted from exactly one spot: for the number $c$, this is because it cannot be in the first spot in order to be in Case 2;
for a number $j$ other than $c$, it cannot be in the $j$th spot.  So the number of options is $D_{n-1}$, the number of derangements of $n-1$ objects.

There is no overlap between cases 1 and 2.  
So, by the addition principle, for a fixed choice of $c$, the number of derangements is $D_{n-1} + D_{n-2}$.
Since there are $n-1$ choices for $c$, the number of derangements is $(n-1)(D_{n-1} + D_{n-2})$ by the multiplication principle.

\end{proof}

Using the recurrence relation, we can compute some more derangement values.

\begin{center}
\begin{tabular}{|c|c|c|c|c|c|c|c|c|c|c|}
\hline
$D_1$ & $D_2$ & $D_3$ & $D_4$ & $D_5$ & $D_6$ & $D_7$ & $D_8$ & $D_9$ & $D_{10}$ \\
\hline
0 & 1 & 2 & 9 & 44 & 265 & 1,854 & 14,833 & 133,496 & 1,334,961 \\
\hline
\end{tabular}
\end{center}

Here are two more formulas for derangements; one is recursive and the other is explicit. We will see proofs of these facts in the exercises below.

\begin{proposition}
\label{Pderange2}
For $n \geq 2$, the number $D_n$ of derangements of $n$ objects satisfies this other recurrence relation
\begin{equation}
 \label{Ederange2}  
D_n = nD_{n-1} + (-1)^n
\end{equation}
\end{proposition}

\begin{theorem} \label{Tderange}
Here is a closed form formula for the number of derangements:
\[D_n=n!\left(1 - \frac{1}{1!} + \frac{1}{2!} - \frac{1}{3!} + \cdots + \frac{(-1)^n}{n!}\right).\]
\end{theorem}

\subsection*{Exercises}

\begin{enumerate}
\item How many permutations of the word MONDAY have at least one letter in the correct spot?	

\item What is the distance between $D_5$ (the number of derangements of 5 objects) and $5!/e$ on the number line, to the nearest 1/1000 as a decimal?										
									
    \item 
One cool fact about derangements is that $D_n$ is the integer closest to $n!/e$.
Check that this is true for $n=1,\ldots, 5$.

\item Use the previous problem to estimate the probability that a permutation of $n$ objects is a derangement when $n$ is large.

\item Prove Proposition~\ref{Pderange2} that 
$D_n = n D_{n-1} + (-1)^n$ by induction as follows:
\begin{enumerate}
    \item Show that \eqref{Ederange2} is true when $n=2$.
    \item Write out the inductive hypothesis by substituting $n-1$ for $n$ into \eqref{Ederange2}.
    \item Use part (2) and \eqref{Ederange1} to prove the inductive step.
\end{enumerate}

\item Prove Theorem~\ref{Tderange} using Proposition~\ref{Pderange2}.

\item Let $W_k$ be the permutations of $1, \ldots, n$ that fix the number $k$.
\begin{enumerate}
    \item 
    For $1 \leq k \leq n$, show that $|W_k| = (n-1)!$. 
    
\item Show there are $\binom{n}{2}$ ways to choose two sets $W_k$ and $W_\ell$ with $k \not = \ell$; explain which permutations are in $W_k \cap W_\ell$ and find $|W_k \cap W_\ell|$.

\item How many ways are there to choose $i$ of the sets $W_1, \ldots, W_n$; what is the size of the intersection of these sets?

\item The inclusion-exclusion principle gives a way to 
find the size of the union of the sets $W_1, \ldots, W_n$.
Use that to find the size of the complement of this union and show that
$D_n = n! \sum_{i=0}^n \frac{(-1)^n}{i!}$.
\end{enumerate}
\end{enumerate}

\section{The Catalan numbers}\label{sec:Catalan}

\begin{videobox}
\begin{minipage}{0.1\textwidth}
\href{https://www.youtube.com/watch?v=NETfwiSbGEA}{\includegraphics[width=1cm]{video-clipart-2.png}}
\end{minipage}
\begin{minipage}{0.8\textwidth}
Click on the icon at left or the URL below for this section's short lecture video. \\\vspace{-0.2cm} \\ \href{https://www.youtube.com/watch?v=NETfwiSbGEA}{https://www.youtube.com/watch?v=NETfwiSbGEA}
\end{minipage}
\end{videobox}

The \textit{Catalan numbers} are one of the most celebrated sequences in combinatorics.  They are defined by the following recursion.

\begin{definition}
 The \textbf{Catalan numbers} are the sequence of numbers $C_0,C_1,C_2,\ldots$
  having initial value
  $C_0=1$ and satisfying the recurrence relation $$C_{n+1}=C_{0}C_n+C_1C_{n-1}+C_2C_{n-2}+\cdots + C_{n}C_0.$$
\end{definition}
Starting with $n=0$, we can use the above recursion to compute that the first $10$ Catalan numbers are 
\[1, 1, 2, 5, 14, 42, 132, 429, 1430, 4862.\]

One of the most important combinatorial interpretations of the Catalan numbers is in terms of \textit{Dyck paths}.  

\begin{definition}
 A \textbf{Dyck path} of length $2n$ is a path on the integer lattice grid from $(0,0)$ to $(n,n)$ that stay on or above the main diagonal.
\end{definition}

For example, here are the Dyck paths of length 6:

\begin{center}
    \includegraphics[width=.75\textwidth]{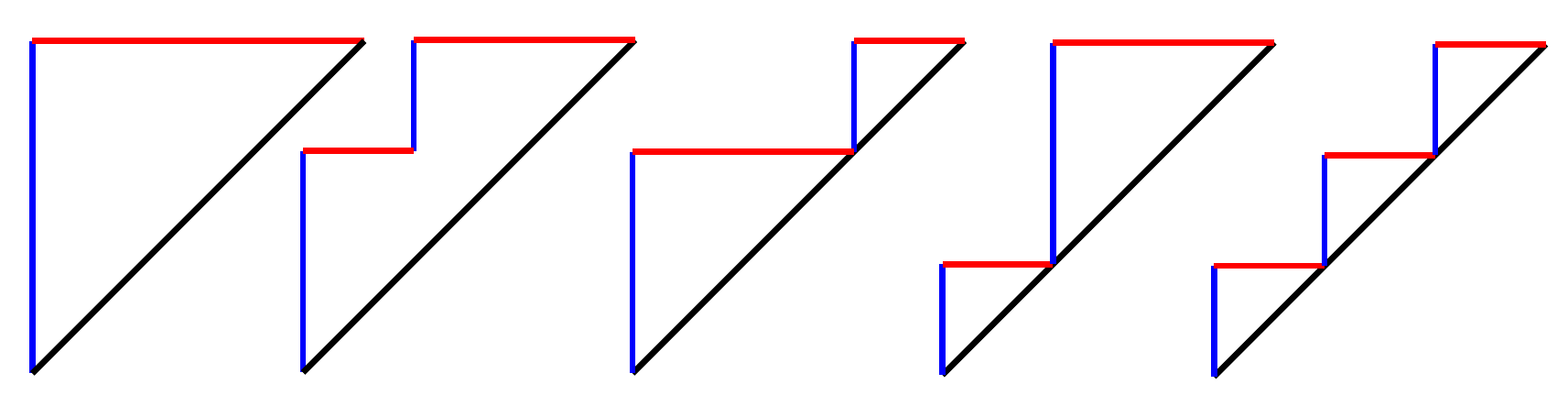}
\end{center}

We can alternatively represent these paths as \textbf{Dyck words}, which are sequences of $n$ $U$'s and $n$ $R$'s such that as we read from right to left, the number of $U$'s we have read is always at least equal to the number of $R$'s we have read.  Here are the six corresponding Dyck words of length $6$:
\[\text{\textcolor{blue}{UUU}\textcolor{red}{RRR}}, \ 
\text{\textcolor{blue}{UU}\textcolor{red}{R}\textcolor{blue}{U}\textcolor{red}{RR}}, \ 
 \text{\textcolor{blue}{UU}\textcolor{red}{RR}\textcolor{blue}{U}\textcolor{red}{R}}, \ 
 \text{\textcolor{blue}{U}\textcolor{red}{R}\textcolor{blue}{U}\textcolor{blue}{U}\textcolor{red}{R}\textcolor{red}{R}}, \ 
 \text{\textcolor{blue}{U}\textcolor{red}{R}\textcolor{blue}{U}\textcolor{red}{R}\textcolor{blue}{U}\textcolor{red}{R}}.\]

\begin{proposition}
The number of Dyck paths of length $2n$ is equal to the $n$th Catalan number $C_n$.
  \end{proposition}

\begin{proof}
Let $D_n$ be the number of Dyck paths of length $n$.  There is only one Dyck path of length $0$, namely the empty path.  So $D_0=1=C_0$, and the initial value is satisfied.

To prove that $D_n$ satisfies the recursion, consider the first time at which a Dyck path of length $n+1$ returns to the diagonal after the first step.  The height of this \textit{first return} is some number $i$ between $1$ and $n+1$ inclusive.  Let $D_{n+1,i}$ be the number of Dyck paths of length $n+1$ whose first return has height $i$.  Then by the addition principle we have \begin{equation}\label{CD}D_{n+1}=D_{n+1,1}+D_{n+1,2}+\cdots+D_{n+1,n+1}.\end{equation}
Now, to compute $D_{n+1,i}$, notice that since the first return to the diagonal is at height $i$, the first part of the path (from point $(0,0)$ to $(i,i)$) is uniquely determined by drawing a Dyck path from point $(0,1)$ to $(i-1,i)$, which gives $D_{i-1}$ possibilities for the first part. 

\begin{center}
    \includegraphics[width=6cm]{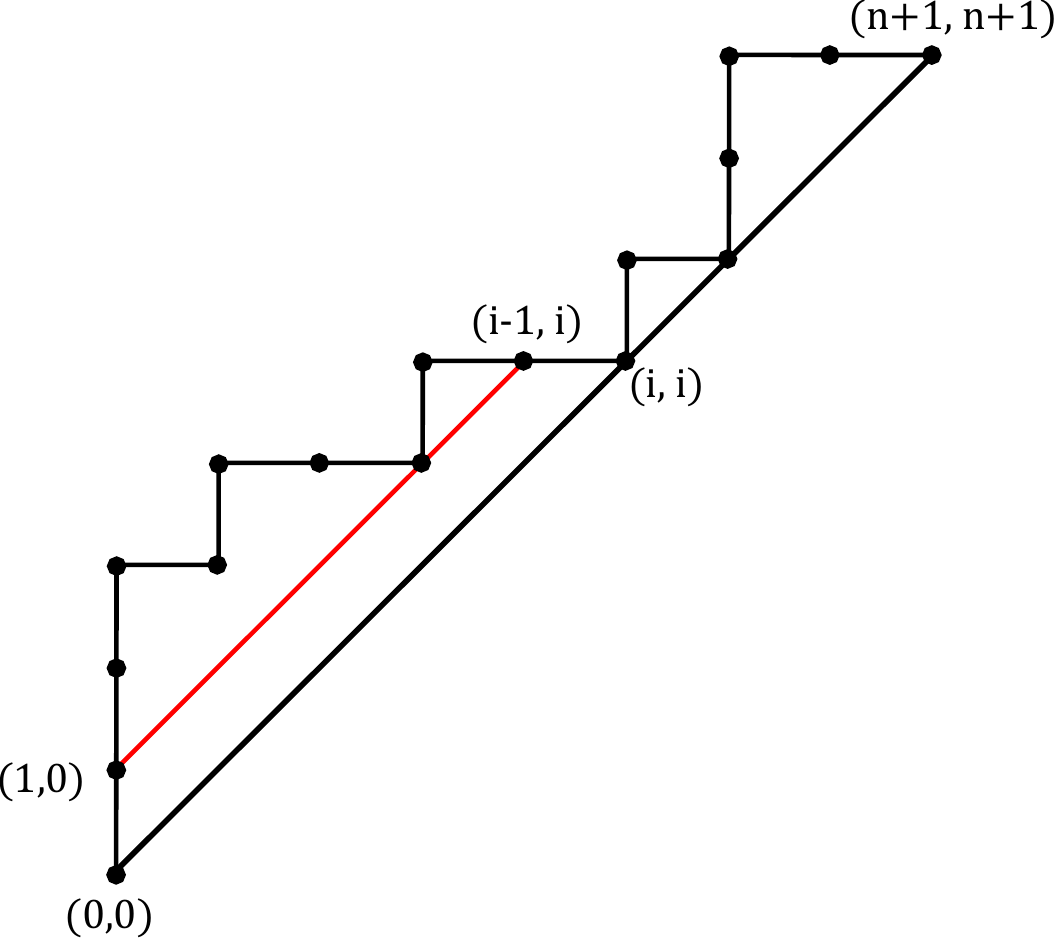}
\end{center}

The remaining steps in the path simply form a Dyck path of height $n+1-i$, of which there are $D_{n+1-i}$.  Therefore, by the multiplication principle, we have $$D_{n+1,i}=D_{i-1}D_{n+1-i}.$$  Thus in particular we have $D_{n+1,1}=D_0D_{n}$, $D_{n+1,2}=D_1D_{n-1}$, and so on.  Plugging these formulas into Equation \ref{CD} above, we obtain  \begin{equation*}D_{n+1}=D_0D_{n}+D_1D_{n-1}+\cdots+D_n D_{0}\end{equation*} as desired.  Therefore, $D_n$ satisfies the same recursion as the Catalan numbers, so $D_n=C_n$ for all $n$.
\end{proof}

\begin{theorem} \label{Tcatalan}
A closed form formula for the $n$th Catalan number is 
\begin{equation}\label{Ecatalan}
C_n = \binom{2n}{n} - \binom{2n}{n+1} = \frac{1}{n+1} \binom{2n}{n}.
\end{equation}
\end{theorem}

We will not include a proof of Theorem~\ref{Tcatalan} in this book.
It is not hard to check that the two formulas for $C_n$ in \eqref{Ecatalan} are the same.

Catalan numbers do not just count Dyck paths; in fact, there are 214 different known combinatorial descriptions of the Catalan numbers\footnote{All of these descriptions can be found in the book ``Catalan Numbers'' by Richard Stanley.}!  Let's consider some descriptions in the case $n=3$. 
\begin{example}
The fact that $C_3=5$ is equivalent to there being $5$ ways to:

\begin{enumerate}
\item put parentheses around $4$ letters
\[((xy)z)w, \      (x(yz))w, \  (xy)(zw), \      x((yz)w), \      x(y(zw));\]
\item divide a convex pentagon into triangles;
\begin{center}
    \includegraphics[width=.75\textwidth]{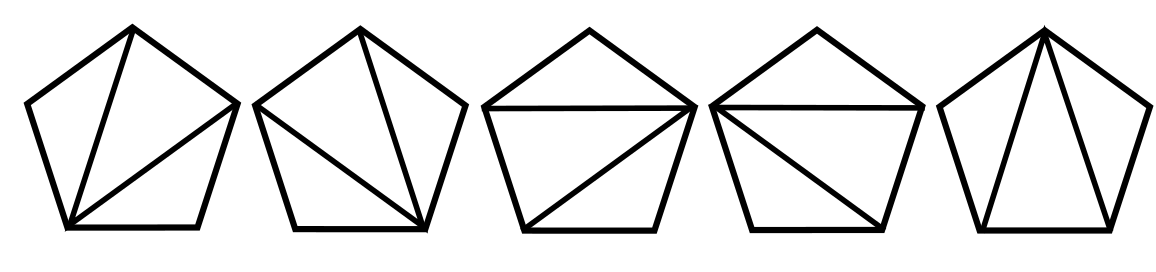}
\end{center}
\item tile a stair of height $3$ with $3$ rectangles;
\begin{center}
    \includegraphics[width=.75\textwidth]{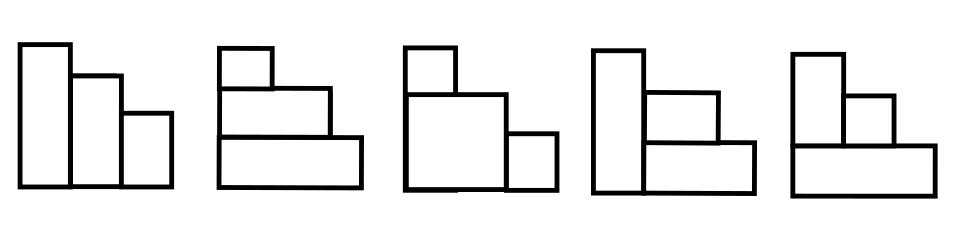}
\end{center}
\item arrange the numbers $\{1, \ldots, 6\}$ in a $2 \times 3$ grid so that each row and each column is increasing;
\[\begin{array}{|c|c|c|}
\hline 
 1&2 &3 \\ \hline
 4& 5& 6\\ \hline
\end{array}, \ 
\begin{array}{|c|c|c|}
\hline 
 1& 2& 4\\ \hline
 3& 5& 6\\ \hline
\end{array}, \ 
\begin{array}{|c|c|c|}
\hline 
 1& 2& 5\\ \hline
 3& 4& 6\\ \hline
\end{array}, \ 
\begin{array}{|c|c|c|}
\hline 
 1& 3& 4\\ \hline
 2& 5& 6\\ \hline
\end{array}, \ 
\begin{array}{|c|c|c|}
\hline 
 1& 3& 5\\ \hline
 2& 4& 6\\ \hline
\end{array}\]

\item pair the vertices of a hexagon so that the line segments joining paired vertices do not intersect;
\begin{center}
    \includegraphics[width=.75\textwidth]{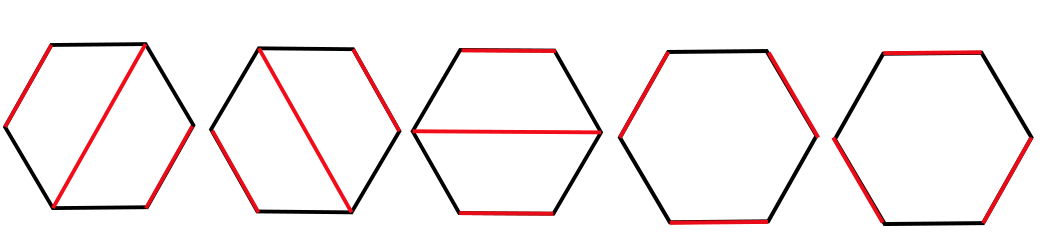}
\end{center}
\item draw a rooted binary tree with $4$ leaves such that every vertex has either $0$ or $2$ children (we will see more about trees in Chapter \ref{chap:trees}).

\begin{center}
    \includegraphics[width=.75\textwidth]{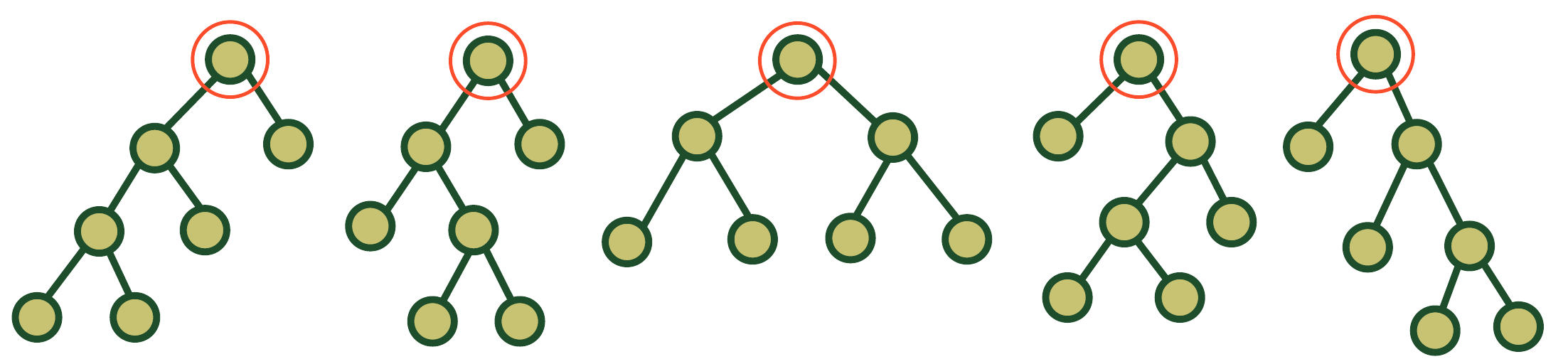}
\end{center}

\item write down a sequence $(a_1,a_2,\ldots,a_{6})$ of $1$'s and $-1$'s such that each partial sum $a_1+a_2+\cdots +a_j$ is nonnegative and the entire sum $a_1+a_2+\cdots+a_{6}$ is $0$.
$$(1,1,1,-1,-1,-1),(1,1,-1,1,-1,-1),(1,1,-1,-1,1,-1),(1,-1,1,1,-1,-1),(1,-1,1,-1,1,-1)$$

\end{enumerate}
\end{example}

 \subsection*{Exercises}
 \begin{enumerate}
\item How many lattice paths from (0,0) to (4,4) stay *on or below* the diagonal?																									
\item How many Dyck paths from (0,0) to (5,5) stay *strictly above* the diagonal, that is, they don't touch the diagonal at any point between (0,0) and (5,5)?																									
\item Suppose candidates Uppity and Rightley are running for class president. 13 of the students in the class plan on voting; 6 of them will be voting for Uppity and 7 will be voting for Rightley. How many sequences of ballots (which can be thought of as sequences of 6 U's and 7 R's in some order) have the property that, as they are counted in order, Rightly is never ahead until the very last ballot is counted?																									
 
\item When $n=4$, then $C_4=14$.
 In these exercises, find the 14 ways to do each of the following.
 
 \begin{enumerate}
         \item Draw a Dyck path;
         \item write a Dyck word;
         \item put parentheses around $5$ letters $xyzwv$;
\item divide a convex hexagon into triangles;

\item tile a stair of height $4$ with rectangles;

\item arrange the numbers $\{1, \ldots, 8\}$ in a $2 \times 4$ grid so that each row and each column is increasing;
\item pair the vertices of an octagon so that the line segments joining paired vertices do not intersect;
\item draw a rooted binary tree with $5$ leaves such that every vertex has either $0$ or $2$ children.
\item write down a sequence $(a_1,a_2,\ldots,a_{8})$ of $1$'s and $-1$'s such that each partial sum $a_1+a_2+\cdots +a_j$ is nonnegative and the entire sum $a_1+a_2+\cdots+a_{8}$ is $0$.
\end{enumerate} 
 \end{enumerate}




%% file: 07-GeneratingFunctions/GeneratingFunctions.tex
\chapter{Generating Functions}\label{chap:generatingfunctions}

This chapter centers on \textit{generating functions}, which are a technique for using polynomials and series for solving problems in combinatorics.  Generating functions are the cornerstone of a field of mathematics called \textit{algebraic combinatorics}.\\

\begin{videobox}
\begin{minipage}{0.1\textwidth}
\href{https://youtu.be/7Rw7pEMJneI}{\includegraphics[width=1cm]{video-clipart-2.png}}
\end{minipage}
\begin{minipage}{0.8\textwidth}
Click on the icon at left or the URL below for this section's short lecture video. \\\vspace{-0.2cm} \\ \href{https://youtu.be/7Rw7pEMJneI}{https://youtu.be/7Rw7pEMJneI}
\end{minipage}
\end{videobox}

First we motivate why generating functions are useful by studying some problems about making change.

\section{Making change}\label{sec:change}

\begin{example} \label{Egen16}
Suppose you have $6$ pennies and $2$ nickels in your pocket.  The pennies look identical and the nickels look identical.  How many different ways can you give some of the coins to your friend?

It is clear that the most you can give is $16$ cents and the least you can give is $0$ cents.  There are two ways that you can give $6$ cents, either $1$ penny and $1$ nickel or $6$ pennies.  To be more systematic, the number of cents you give will be
\[0,1,2,3,4,5, \text{or }  6 \ \text{in pennies plus $0,5, \text{or } 10$ in nickels}.\]
We want to add the value of pennies to the value of nickels.  The key insight is that we \textit{add} exponents in a product of polynomials.  So we use the exponents above in two polynomials and then multiply them.
\[(x^0+x+x^2+x^3+x^4+x^5+x^6) \cdot (x^0+x^5+x^{10}).\] 
We expand this to get:
\[x^0+x+x^2+x^3+x^4+2x^5+2x^6+x^7+x^8+x^9+2x^{10}+2x^{11}+x^{12}+x^{13}+x^{14}+x^{15}+x^{16}.\]

The fact that the largest value you can give is $16$ is represented by the term 
$x^{16}= x^6 \cdot x^{10}$, for $6$ cents in pennies and $10$ cents in nickels.
The fact that the smallest value you can give is $0$ is represented by the term $x^0=x^0\cdot x^0$ for $0$ cents in pennies and $0$ cents in nickels.
The fact that there are two ways that you can give $6$ cents is represented by the term $2 x^6=x^6\cdot x^0 + 
x^1 \cdot x^5$, for either $6$ cents in pennies and $0$ cents in nickels or $1$ cent in pennies and $5$ cents in nickels.

More generally, this polynomial
encodes the number of ways to give \textit{any} number of cents using $6$ pennies and $2$ nickels!  
For example, the term $2x^{11}=x^6 \cdot x^5+x^1 \cdot x^{10}$ tells us that there are two ways to give $11$ cents, which makes sense since you can give $6$ cents in pennies and $5$ cents in nickels or $1$ cent in pennies and $10$ cents in nickels. 
Just by reading the coefficients above, we see that there are $2$ ways of giving $10$ cents but only one way of giving $7$ cents.
It even tells us that there are no ways of giving $17$ cents!
\end{example}

Let's now try a somewhat more complicated example.

\begin{example} \label{Egen100}
  Suppose you have 6 pennies (\$0.01 each), 2 nickels (\$0.05 each), 4 dimes (\$0.10 each), and 3 quarters (\$0.25 each) in your wallet.  You can't tell the difference between any two coins of the same type; for instance, the pennies are indistinguishable from each other. In how many different ways can you make change for a dollar?

Again, one way to approach this is to write out the possibilities.  As a way of organizing it, we'll start with the maximum possible number of quarters, then the maximum possible number of dimes, and so on.  So the first possibility is using all 3 quarters, then the maximum of 2 dimes, and then a nickel.  The next possibility uses 3 quarters and 2 dimes but then 5 pennies instead of the nickel.  Then we consider the possibilities with 3 quarters and only 1 dime, and so on. We get the following list:
\begin{itemize}
    \item 3 quarters, 2 dimes, and a nickel
    \item 3 quarters, 2 dimes, and 5 pennies
    \item 3 quarters, 1 dime, 2 nickels, and 5 pennies
    \item 2 quarters, 4 dimes, and 2 nickels
    \item 2 quarters, 4 dimes, 1 nickel, and 5 pennies
\end{itemize}
We do not have enough smaller coins to use 2 quarters and only 3 dimes, or only 1 quarter.  So there are only 5 possibilities total.

Let's now use polynomial multiplication to solve this instead.  For each type of coin, if the coin is worth $k$ cents, we associate the polynomial $1+x^k+x^{2k}+\cdots+x^{mk}$ where $m$ is the number of that coin we have.  This way, the exponents in our polynomial encode the possible amounts of cents we can make using just those coins.  So we would get the polynomials:
\begin{itemize}
    \item \textbf{Pennies:} $1+x+x^2+x^3+x^4+x^5+x^6$
    \item \textbf{Nickels:} $1+x^5+x^{10}$
    \item \textbf{Dimes:} $1+x^{10}+x^{20}+x^{30}+x^{40}$
    \item \textbf{Quarters:} $1+x^{25}+x^{50}+x^{75}$
\end{itemize}

We take the product of all four polynomials: 
\[(1+x+x^2+x^3+x^4+x^5+x^6)\cdot(1+x^5+x^{10})\cdot (1+x^{10}+x^{20}+x^{30}+x^{40})\cdot(1+x^{25}+x^{50}+x^{75})\]  Typing this product into Sage using the command \texttt{expand}, we find that the product expands as: \\

$x^{131} + x^{130} + x^{129} + x^{128} + x^{127} + 2 \, x^{126} + 2 \, x^{125} + x^{124} + x^{123} + x^{122} + 3 \, x^{121} + 3 \, x^{120} + 2 \, x^{119} + 2 \, x^{118} + 2 \, x^{117} + 3 \, x^{116} + 3 \, x^{115} + x^{114} + x^{113} + x^{112} + 3 \, x^{111} + 3 \, x^{110} + 2 \, x^{109} + 2 \, x^{108} + 2 \, x^{107} + 4 \, x^{106} + 4 \, x^{105} + 2 \, x^{104} + 2 \, x^{103} + 2 \, x^{102} + 5 \, x^{101} + 5 \, x^{100} + 3 \, x^{99} + 3 \, x^{98} + 3 \, x^{97} + 6 \, x^{96} + 6 \, x^{95} + 3 \, x^{94} + 3 \, x^{93} + 3 \, x^{92} + 6 \, x^{91} + 6 \, x^{90} + 3 \, x^{89} + 3 \, x^{88} + 3 \, x^{87} + 6 \, x^{86} + 6 \, x^{85} + 3 \, x^{84} + 3 \, x^{83} + 3 \, x^{82} + 6 \, x^{81} + 6 \, x^{80} + 3 \, x^{79} + 3 \, x^{78} + 3 \, x^{77} + 6 \, x^{76} + 6 \, x^{75} + 3 \, x^{74} + 3 \, x^{73} + 3 \, x^{72} + 6 \, x^{71} + 6 \, x^{70} + 3 \, x^{69} + 3 \, x^{68} + 3 \, x^{67} + 6 \, x^{66} + 6 \, x^{65} + 3 \, x^{64} + 3 \, x^{63} + 3 \, x^{62} + 6 \, x^{61} + 6 \, x^{60} + 3 \, x^{59} + 3 \, x^{58} + 3 \, x^{57} + 6 \, x^{56} + 6 \, x^{55} + 3 \, x^{54} + 3 \, x^{53} + 3 \, x^{52} + 6 \, x^{51} + 6 \, x^{50} + 3 \, x^{49} + 3 \, x^{48} + 3 \, x^{47} + 6 \, x^{46} + 6 \, x^{45} + 3 \, x^{44} + 3 \, x^{43} + 3 \, x^{42} + 6 \, x^{41} + 6 \, x^{40} + 3 \, x^{39} + 3 \, x^{38} + 3 \, x^{37} + 6 \, x^{36} + 6 \, x^{35} + 3 \, x^{34} + 3 \, x^{33} + 3 \, x^{32} + 5 \, x^{31} + 5 \, x^{30} + 2 \, x^{29} + 2 \, x^{28} + 2 \, x^{27} + 4 \, x^{26} + 4 \, x^{25} + 2 \, x^{24} + 2 \, x^{23} + 2 \, x^{22} + 3 \, x^{21} + 3 \, x^{20} + x^{19} + x^{18} + x^{17} + 3 \, x^{16} + 3 \, x^{15} + 2 \, x^{14} + 2 \, x^{13} + 2 \, x^{12} + 3 \, x^{11} + 3 \, x^{10} + x^{9} + x^{8} + x^{7} + 2 \, x^{6} + 2 \, x^{5} + x^{4} + x^{3} + x^{2} + x + 1$ \\

Look closely at the $x^{100}$ term above.  Its coefficient is $5$, which is exactly the number of ways of making $100$ cents (one dollar) out of the coins!  Indeed, each way of getting $x^{100}$ from the product corresponds to a choice of some number of the pennies, nickels, dimes, and quarters that add up to $100$ cents. 
\end{example}

\begin{proposition} \label{Pgenchange1}
Suppose you have $k_1$ pennies, $k_5$ nickels, $k_{10}$ dimes, and $k_{25}$ quarters in your pocket. 
Then the number of ways you can make change for $n$ cents is the coefficient of $x^n$ in the polynomial
\[(1+x+ \cdots + x^{k_1})\cdot 
(1+x^5 + \cdots + x^{5k_5}) \cdot
(1+x^{10} + \cdots + x^{10k_{10}}) \cdot
(1+x^{25} + \cdots + x^{25k_{25}}) .\]
\end{proposition}

Notice that the above reasoning works no matter what values the coins are.  Here is an example using a different kind of currency.

\begin{example}\label{ex:HarryPotter}
  In Harry Potter currency, a Knut is the analog of a penny, a Sickle is worth $29$ Knuts, and a Galleon is worth $17$ Sickles (so a Galleon is $17\cdot 29=493$ Knuts).  If we have a wizarding money pouch containing two Galleons, two Sickles, and three Knuts, the corresponding polynomials are:

\begin{itemize}
    \item \textbf{Knuts:} $1+x+x^2+x^3$
    \item \textbf{Sickles:} $1+x^{29}+x^{58}$
    \item \textbf{Galleons:} $1+x^{493}+x^{986}$
\end{itemize}

Multiplying these polynomials together and expanding in Sage, we find:
$$(1+x+x^2+x^3)(1+x^{29}+x^{58})(1+x^{493}+x^{986})$$

$=x^{1047} + x^{1046} + x^{1045} + x^{1044} + x^{1018} + x^{1017} + x^{1016} + x^{1015} + x^{989} + x^{988} + x^{987} + x^{986} + x^{554} + x^{553} + x^{552} + x^{551} + x^{525} + x^{524} + x^{523} + x^{522} + x^{496} + x^{495} + x^{494} + x^{493} + x^{61} + x^{60} + x^{59} + x^{58} + x^{32} + x^{31} + x^{30} + x^{29} + x^{3} + x^{2} + x + 1$ \\

Notice that all of the coefficients in this polynomial are $1$.  That means there is exactly one way to make change for each of the exponents appearing: $1047, 1046, 1045, 1044, 1018,1017$, etc., and no ways to make any other amount of wizarding money using the coins in that pouch.

\end{example}

We can summarize our findings in the following theorem.

\begin{theorem}
  Suppose there are $m$ types of coins of values $k_1,k_2,\ldots,k_m$, and you have $n_1,n_2,\ldots,n_m$ of each type of coin respectively. 
  The number of ways to make change for a total value of $N$ using these coins is the coefficient of $x^N$ in the product $$(1+x^{k_1}+x^{2k_1}+\cdots +x^{n_1 k_1})(1+x^{k_2}+x^{2k_2}+\cdots+x^{n_2 k_2})\cdots (1+x^{k_m}+x^{2k_m}+\cdots+x^{n_m k_m}).$$
\end{theorem}

\begin{proof}
  If we multiply together one term from each factor in the product, we get a product of the form $$x^{a_1 k_1}x^{a_2k_2}\cdots x^{a_mk_m}=x^{a_1k_1+a_2k_2+\cdots+a_mk_m}$$ where each $a_i\le n_i$.  The exponent $a_1k_1+a_2k_2+\cdots+a_mk_m$ equals $N$ exactly when adding $a_i$ copies of $k_i$ for each $i$ adds up to $N$, which corresponds to using $a_i$ coins of the $i$-th type to make change for $N$.  Therefore, the number of times $x^N$ occurs in the expansion of the product is exactly the number of ways to make change for $N$. 
\end{proof}
\subsection*{Exercises}

\begin{enumerate}
    \item Suppose you have the same coins in your wallet as in the first example above, namely
    $6$ pennies, $2$ nickels, $4$ dimes, and $3$ quarters.
    \begin{enumerate}
        \item How many ways can you give someone $70$ cents?
        \item How many ways can you give someone $110$ cents?
        \item What are the values $m$ such that there is exactly one way to give someone $m$ cents?
    \end{enumerate}
  
    \item The polynomial in Example \ref{ex:HarryPotter} shows that there should be exactly one way of making exact change for $553$ Knut values using at most two Galleons, at most two Sickles, and at most three Knuts.  Figure out how to make change for $553$ using these coins, and explain how it comes from multiplying terms from the three factors in the product.
    
    \item At the video arcade, 
the red token is worth $3$ dollars, the green token is worth $4$ dollars, and the black token is worth
$5$ dollars.  You have $8$ red tokens, $6$ green tokens, and $4$ black tokens.  The question is how many ways you can spend $50$ dollars in tokens.
\begin{enumerate}
    \item Find the polynomial $f(x)$ that shows how many ways to spend every number of dollars in tokens.
    \item Expand $f(x)$ in Sage using the command \texttt{expand}.
    \item How many ways can you spend 50 dollars in tokens?
\end{enumerate}
  
    \item The Haunted Game Cafe in Fort Collins sells buffalo wings in small or large baskets.  The small basket contains $9$ wings and the large basket contains $15$ wings.
    A group of friends wants to order some buffalo wings, but there are 
    only 6 small baskets and 5 large baskets left. 
    \begin{enumerate}
        \item Find a polynomial that shows how many ways they can order every possible number of buffalo wings.
        \item How many ways can they order $60$ buffalo wings?
    \end{enumerate}
\end{enumerate}

\section{Generating functions}

\begin{videobox}
\begin{minipage}{0.1\textwidth}
\href{https://youtu.be/evP2mK7l7I0}{\includegraphics[width=1cm]{video-clipart-2.png}}
\end{minipage}
\begin{minipage}{0.8\textwidth}
Click on the icon at left or the URL below for this section's short lecture video. \\\vspace{-0.2cm} \\ \href{https://youtu.be/evP2mK7l7I0}{https://youtu.be/evP2mK7l7I0}
\end{minipage}
\end{videobox}

Suppose a cashier with an \textit{unlimited} supply of each coin was trying to figure out how many ways they can make change for a dollar.  Then the polynomials from Section \ref{sec:change} would turn into \textit{infinite} summations:
\begin{itemize}
    \item \textbf{Pennies}: $1+x+x^2+x^3+\cdots$
    \item \textbf{Nickels}: $1+x^{5}+x^{10}+x^{15}+\cdots$
    \item \textbf{Dimes}: $1+x^{10}+x^{20}+x^{30}+\cdots$
    \item \textbf{Quarters}: $1+x^{25}+x^{50}+x^{75}+\cdots$
\end{itemize}

We therefore need to define a notion of addition and multiplication of these types of infinite series, which we call \textit{generating functions}.

\begin{definition}
  The \textbf{generating function} of the sequence of numbers $a_0,a_1,a_2,\ldots$ is the series  $$\sum_{i=0}^\infty a_i x^i=a_0+a_1x+a_2x^2+a_3x^3+\cdots.$$ (Here $x$ is just a formal symbol, and the coefficients of the series are the important part!)
\end{definition}

\begin{remark}
Any polynomial can be interpreted as a generating function of a sequence that is eventually $0$; for instance, the generating function of the sequence $$1,1,3,0,0,0,0,0,0,\ldots$$ is $$1+x+3x^2+0x^3+0x^4+0x^5+\cdots=1+x+3x^2.$$
\end{remark}

\begin{example}
  The generating function of the sequence $1,1,1,1,1,1,1,0,0,0,0,\ldots$ is the polynomial $$1+x+x^2+x^3+x^4+x^5+x^6,$$ which is the polynomial we used to record the six pennies in Example \ref{Egen100}.  On the other hand, the polynomial that recorded the two nickels was $$1+x^5+x^{10},$$ which is the generating function of the sequence $$1,0,0,0,0,1,0,0,0,0,1,0,0,0,0,0,0,0,0,0,0,\ldots$$ having a $1$ in the $0$th, $5$th, and $10$th positions of the sequence and $0$ everywhere else.
\end{example}

\begin{example} \label{EgenN}
  The generating function of the sequence $1,2,3,4,5,\ldots$ (in other words, the sequence $a_0,a_1,a_2,\ldots$ defined by $a_i=i+1$) is
  $$1 + 2x + 3x^2 + 4x^3 + \cdots.$$
  We will revisit this example in the next section.
\end{example}

\begin{example}
  The generating function of the Fibonacci sequence $1,1,2,3,5,8,\ldots$ is $$1+x+2x^2+3x^3+5x^4+8x^5+\cdots.$$  By the end of this chapter we will see how to use generating functions to find an explicit formula for the Fibonacci numbers themselves.
\end{example}

\begin{example} (Revisiting the binomial theorem.)
Fix a number $n$.  The generating function of the sequence $\binom{n}{n},\binom{n}{n-1},\binom{n}{n-2},\ldots, \binom{n}{0}$ is:
\[\binom{n}{n} + \binom{n}{n-1} \cdot x + \binom{n}{n-2} x^2 + \cdots + \binom{n}{0}x^n.\]

This looks very familiar!  The right hand side is what we get if we substitute $y=1$ into Theorem~\ref{Tbinomialthm} and rearrange the terms:
\[(1+x)^n = \binom{n}{n} + \binom{n}{n-1} x + \binom{n}{n-2} x^2 + \cdots + \binom{n}{0} x^n.\]

In fact we can replace $\binom{n}{n}$ with $\binom{n}{0}$, replace $\binom{n}{n-1}$ with $\binom{n}{1}$, and so on by the symmetry of binomial coefficients, to get the somewhat nicer looking identity
\[(1+x)^n = \binom{n}{0} + \binom{n}{1} x + \binom{n}{2} x^2 + \cdots + \binom{n}{n} x^n.\]
Here is a different way to think about this that is similar to the problem of making change.  
Let's say there are $n$ people in a room and you need exactly $k$ of them to complete a survey.  
You can ask each person to hold up a card saying $0$ if they did not complete the survey and saying $1$ if they did.  You need the sum of the cards to be $k$.

Equivalently, from the polynomial $1+x$ you can ask each person to choose the monomial $1=x^0$ if they did not complete the survey and to choose the monomial $x=x^1$ if they did.  Then you need the product of the chosen monomials, which is the sum of the exponents, to equal $x^k$.  The conclusion is that the coefficient of
$x^k$ in $(1+x)^n$ is the number of ways to choose exactly $k$ people out of the $n$ people to complete the survey, and this number of ways is $\binom{n}{k}$.
In other words, $\binom{n}{k}$ is the coefficient of 
$x^k$ in $(1+x)^n$.  
\end{example}

\subsection{Are generating functions actually functions?}

In this book, we will not be thinking of generating functions as functions.  While it is common to use notation like $$A(x)=a_0+a_1x+a_2x^2+a_3x^3+\cdots$$ to denote a generating function, the notation $A(x)$ is just a convenient shorthand. 
 We will not be substituting a numerical value of $x$ into the formula for a generating function.  Indeed, if we try to plug in certain values of $x$ into an infinite series, the sum may not even converge.  

Instead, we like to think of $x$ 
as a formal symbol.  
The reason we take powers of $x$ is that the ways we manipulate generating functions are very similar to the ways we manipulate polynomials and power series in calculus.
We will define properties of generating functions in the next section. 

As Herbert Wilf says in his book Generatingfunctionology, 
\begin{quote}
    \it ``A generating function is a clothesline on which we hang up a sequence of numbers up for display.''
\end{quote}
In other words, it is the \textit{coefficients} (the sequence $a_0,a_1,a_2,\ldots$ that we are hanging up for display) that we are studying here, not any sort of output of the function.  In combinatorics, we study the series itself 
to gain information about its coefficients.

\begin{remark}
We can use any variable we want to write down the generating function of a sequence.  So far we have been considering generating functions in $x$, but for instance, the generating function of the sequence $1,2,3,4,5,\ldots$ in the variable $y$ is $$1+2y+3y^2+4y^3+\cdots$$ and its generating function in the variable $\smiley$ is $$1+2\smiley+3\smiley^2+4\smiley^3+\cdots$$
\end{remark}

\subsection*{Exercises}

\begin{enumerate}
    \item Write down the generating function of the sequence of even numbers $0,2,4,6,8,\ldots$ in the variable $x$, in the variable $y$, and in the variable $\smiley$.
    
    \item Write down the generating function (in $x$) of the sequence $1,-1,1,-1,1,-1,\ldots$.
    
    \item What sequence gives rise to the following generating function? $$1+x^2+x^4+x^6+\cdots$$
\end{enumerate}

\section{Operations on generating functions}

\begin{videobox}
\begin{minipage}{0.1\textwidth}
\href{https://www.youtube.com/watch?v=FFUH4IpS-V0}{\includegraphics[width=1cm]{video-clipart-2.png}}
\end{minipage}
\begin{minipage}{0.8\textwidth}
Click on the icon at left or the URL below for this section's short lecture video. \\\vspace{-0.2cm} \\ \href{https://www.youtube.com/watch?v=FFUH4IpS-V0}{https://www.youtube.com/watch?v=FFUH4IpS-V0}
\end{minipage}
\end{videobox}

If the cashier has an \textit{unlimited} supply of pennies, nickels, dimes, and quarters, we wish to ``multiply'' the corresponding infinite generating functions just like we multiplied the polynomials in Example \ref{Egen100}.  We therefore now need to define operations on generating functions.  The four main operations we will consider are addition, multiplication, substitution of a monomial, and differentiation.

In this section, let $A(x) = \sum_{n=0}^\infty a_n x^n$ and $B(x) = \sum_{n=0}^{\infty} b_n x^n$ be two generating functions.

\subsection{Addition (and subtraction) of generating functions}

Addition is the simplest operation on generating functions; we simply add corresponding coefficients: 
\begin{equation}\label{eq:add-generating-functions}
A(x)+B(x) = \sum_{n=0}^\infty a_n x^n+\sum_{n=0}^\infty b_n x^n=\sum_{n=0}^\infty (a_n+b_n) x^n.
\end{equation}
Writing out the summations, the first few terms look like: $$(a_0+a_1x+a_2x^2+\cdots)+(b_0+b_1x+b_2x^2+\cdots)=(a_0+b_0)+(a_1+b_1)x+(a_2+b_2)x^2+\cdots$$

\begin{example}
  The sum of the generating functions $$1+x+x^2+x^3+x^4+x^5+x^6+\cdots$$ and $$1-x+x^2-x^3+x^4-x^5+x^6+\cdots$$ is $$2+2x^2+2x^4+2x^6+\cdots.$$
\end{example}

We can similarly subtract generating functions.  In fact, any generating function minus itself equals $0$ (the generating function of the sequence $0,0,0,0,0,\ldots$).

\subsection{Multiplication of generating functions}

Multiplication of generating functions works much like multiplication of polynomials.  Let's start with an example.

\begin{example}
  Consider the generating function of the sequence $1,2,3,4,5,6,\ldots$.  We will compute its product with itself: $$(1+2x+3x^2+4x^3+5x^4+\cdots)(1+2x+3x^2+4x^3+5x^4+\cdots)$$
  Naively we would expect the expansion of this product to be performed by multiplying one term from the left factor with one term from the right factor in all possible ways, and adding up the resulting terms.  We can organize these products in a table as follows:
  
  \begin{center}
  \begin{tabular}{c|cccccc}
      $\times$ & $1$ & $2x$ & $3x^2$ & $4x^3$ & $5x^4$ & $\cdots$ \\\hline 
    $1$ & $1$ & $2x$ & $3x^2$ & $4x^3$ & $5x^4$ & \\
    $2x$   & $2x$   & $4x^2$ & $6x^3$ & $8x^4$ & $10x^5$ & \\
    $3x^2$ & $3x^2$ & $6x^3$ & $9x^4$ & $12x^5$ & $15x^6$ & \\
    $4x^3$ & $4x^3$ & $8x^4$ & $12x^5$ & $16x^6$& $20x^7$& \\
    $5x^4$ & $5x^4$ & $10x^5$& $15x^6$ & $20x^7$& $25x^8$& \\
    \vdots & & & & & &       
  \end{tabular}
  \end{center}
  Finally, to compute the sum of all the entries of this table, we notice that the down-and-left diagonals all have the same exponent of $x$, and we can combine the terms along each diagonal, obtaining the sum: $$1+(2x+2x)+(3x^2+4x^2+3x^2)+(4x^3+6x^3+6x^3+4x^3)+\cdots=1+4x+10x^2+20x^3+\cdots$$
\end{example}

We can generalize the above example as follows.  To compute the product $$A(x) B(x) = (a_0+a_1x+a_2x^2+\cdots)(b_0+b_1x+b_2x^2+\cdots),$$ consider the table:

  \begin{center}
  \begin{tabular}{c|cccccc}
      $\times$ & $b_0$ & $b_1x$ & $b_2x^2$ & $b_3x^3$ & $b_4x^4$ & $\cdots$ \\\hline 
    $a_0$ & $a_0b_0$ & $a_0b_1x$ & $a_0b_2x^2$ & $a_0b_3x^3$ & $a_0b_4x^4$ & \\
    $a_1x$  & $a_1b_0x$ & $a_1b_1x^2$ & $a_1b_2x^3$ & $a_1b_3x^4$ & $a_1b_4x^5$ & \\
    $a_2x^2$ & $a_2b_0x^2$ & $a_2b_1x^3$ & $a_2b_2x^4$ & $a_2b_3x^5$ & $a_2b_4x^6$ & \\
    $a_3x^3$ & $a_3b_0x^3$ & $a_3b_1x^4$ & $a_3b_2x^5$ & $a_3b_3x^6$ & $a_3b_4x^7$& \\
    $a_4x^4$ & $a_4b_0x^4$ & $a_4b_1x^5$ & $a_4b_2x^6$ & $a_4b_3x^7$ & $a_4b_4x^8$& \\
    \vdots & & & & & &       
  \end{tabular}
  \end{center}
  Again we add down the down-and-left diagonals to get the following summation: $$A(x)B(x)=(a_0b_0)+(a_0b_1+a_1b_0)x^2+(a_0b_2+a_1b_1+a_2b_0)x^2+\cdots.$$
In general, we see that the coefficient of $x^n$ in the product above is $a_0b_n+a_1b_{n-1}+a_2b_{n-2}+\cdots+a_nb_0.$  This leads to the general definition of multiplication of generating functions.

\begin{definition} We define the product of the generating functions $A(x)$ and $B(x)$ by 
\begin{eqnarray*}
A(x)B(x)=
\left(\sum_{n=0}^\infty a_n x^n\right)\left(\sum_{n=0}^\infty b_n x^n\right) & = & \sum_{n=0}^\infty (a_0b_n+a_1b_{n-1}+a_2b_{n-2}+\cdots+a_nb_0)x^n \\
& = & \sum_{n=0}^{\infty}\left(\sum_{k=0}^n a_kb_{n-k}\right)x^n.
\end{eqnarray*}
The coefficient $\sum_{k=0}^n a_kb_{n-k}$ of $x^n$ is called the $n$-th \textbf{convolution} of the sequences $a_0,a_1,\ldots$ and $b_0,b_1,\ldots$. 
\end{definition}

We can use the multiplication rule to prove one of the most important generating function identities, known as a \textit{geometric series}.

\begin{theorem}\label{Tgeoseries}
We have the identity $$1+x+x^2+x^3+\cdots=\frac{1}{1-x}.$$
\end{theorem}

\begin{proof}
   The meaning of this identity is that the product of the generating functions $1+x+x^2+x^3+\cdots$ and $1-x$ is equal to $1=1+0x+0x^2+0x^3+\cdots$.  Indeed, using the formula for multiplying generating functions, we find that the coefficient of $x^0$ is $1$, and all the other coefficients are $1-1=0$.
\end{proof}

\begin{remark}
If you have a good memory, you might remember the formula in Theorem~\ref{Tgeoseries} from Calculus II.
In Calculus II, you learn that
$1+x+x^2 + \cdots$ is the Maclaurin series for the function $f(x)=1/(1-x)$, in other words the Taylor series for $f(x)$ centered at $x=0$; also you might have learned that the radius of convergence for the Maclaurin series is $R=1$.

It is important to clarify the difference between those statements in calculus and Theorem~\ref{Tgeoseries}.  In this chapter, we are defining everything algebraically and only need the multiplication and addition rules for series to prove Theorem~\ref{Tgeoseries}.  We will not focus on the radius of convergence since we are not evaluating this series at any number.  
\end{remark}

A ``closed form formula'' for a generating function is a formula that does not involve an infinite summation.  In Theorem~\ref{Tgeoseries}, we found the closed form formula $1/(1-x)$ 
for $\sum_{i=0}^\infty x^i$.
We will see in Section~\ref{sec:recurrence-gf} that finding a closed form formula for a generating function can often come in handy for computing its coefficients.

\begin{example}
  A similar argument to Theorem \ref{Tgeoseries} shows that the alternating series $1-x+x^2-x^3+x^4-x^5+\cdots$ has the closed form formula $1/(1+x)$.
\end{example}

\begin{example}\label{ex:mult}
  Let's compute the square of the geometric series $\sum_{i=0}^\infty x^i$.  By the multiplication rule, we have \begin{align*}
      (1+x+x^2+x^3+\cdots)(1+x+x^2+x^3+\cdots)&=1\cdot 1+(1\cdot 1+1\cdot 1)x+(1\cdot 1+1\cdot 1+1\cdot 1)x^2+\cdots \\
      &=1+2x+3x^2+4x^3+\cdots
  \end{align*} We can simplify the left hand side as $1/(1-x)^2$ using Theorem \ref{Tgeoseries}, leading to the closed form generating function identity $$1+2x+3x^2+4x^3+\cdots=\frac{1}{(1-x)^2}$$ for the sequence $1,2,3,4,\ldots$.
\end{example}

\begin{example}\label{ex:mult-by-x}
{\bf Multiplying a generating function by $x$.}  
We can think of $x$ as the generating function $0+x+0x^2+0x^3+\cdots$, and use the multiplication formula to find:
\begin{align*}
    x(a_0+a_1x+a_2x^2+a_3x^4+\cdots)
 &= (0+1\cdot x+0x^2+0x^3+\cdots)(a_0+a_1x+a_2x^2+a_3x^4+\cdots) \\
 &= (0\cdot a_0)+(0\cdot a_1+1\cdot a_0)x+(0\cdot a_2+1\cdot a_1+0\cdot a_2)x^3+\cdots \\
 &= 0+a_0x+a_1x^2+a_2x^3+\cdots. 
\end{align*}
This is indeed intuitively what we would expect if we simply ``distributed'' the $x$ across the infinite summation.  In terms of coefficients, it also shows that multiplying a generating function by $x$ shifts the coefficients; precisely, 
when $A(x)$ is the generating function of the sequence 
$a_0,a_1,a_2,\ldots$, 
then the coefficient of $x^n$ in $xA(x)$ is $a_{n-1}$.  
In summation notation, 
$$
x\left(\sum_{n=0}^\infty a_n x^n\right) = 
\sum_{n=0}^\infty a_n x^{n+1} = \sum_{m=1}^\infty a_{m-1} x^{m},
$$
where the last equality uses the 
substitution $m=n+1$.
Sometimes, we 
replace $m$ by $n$ in the last line and write
$$x\left(\sum_{n=0}^\infty a_n x^n\right)=
\sum_{n=1}^\infty a_{n-1}x^n.
$$
 \end{example}

\begin{example}
For instance, we can multiply both sides of the geometric series formula by $x$ to find $$x+x^2+x^3+x^4+\cdots=\frac{x}{1-x}.$$
\end{example}

Similarly, we can multiply a generating function by $x^k$ to shift the coefficients by $k$ spots, and we can multiply through by a constant.  For instance, we can multiply the geometric series formula by $2x^2$ to get the identity $$2x^2+2x^3+2x^4+2x^5+\cdots = \frac{2x^2}{1-x}.$$

\subsection{Substituting monomials into generating functions}

A \textbf{monomial} in $x$ is any term of the form $ax^n$ where $a$ is a number and $n$ is an integer called the degree.  We may substitute monomials of positive degree into generating function identities to produce more identities.

\begin{example}
  Write $G(x)=1+x+x^2+x^3+x^4+\cdots$.  Then we have the closed form formula $G(x)=\frac{1}{1-x}$ by Theorem \ref{Tgeoseries}.  We can substitute $2x$ for $x$ to obtain
  \begin{eqnarray*}
  G(2x) & = & 1 + (2x) + (2x)^2 + (2x)^3 + \cdots \\ 
  \frac{1}{1-2x} & = & 
  1+2x+4x^2+8x^3+\cdots.\end{eqnarray*}
\end{example}

\begin{example}
  If we want to find a closed form formula for $1+x^2+x^4+x^6+x^8+\cdots$, we notice that it is the result of substituting $x^2$ for $x$ in
  the geometric series, to get $G(x^2)=\frac{1}{1-x^2}$.
\end{example}

\begin{remark}
Warning:  
do not try to substitute constants or formulas that are not monomials into generating functions without caution!  What goes wrong if we substitute $y=1+x$ in the generating function $1+y+y^2+y^3+y^4+\cdots$?  What goes wrong if we substitute $y=1$?
More generally, the composition of generating functions $G(F(x))$ makes sense only when $F(x)$ has no constant term. 
\end{remark}

\subsection{Differentiation of generating functions}

We now define differentiation of generating functions, by extending the power rule for derivatives in calculus.

\begin{definition}
 We define $$\frac{d}{dx} \left(\sum_{n=0}^\infty a_n x^n\right)=\sum_{n=1}^\infty na_nx^{n-1}=\sum_{n=0}^\infty (n+1)a_{n+1}x^n.$$  
 In the last equality, we re-indexed by replacing $n$ with $n+1$.
 More visually: $$\frac{d}{dx}\left(a_0+a_1x+a_2x^2+a_3x^3+\cdots\right)=a_1+2a_2x+3a_3x^2+4a_4x^3+\cdots.$$
\end{definition}

Remember that generating functions are formal power series that we are not thinking of as actual functions.
The amazing thing is that, despite this, \textit{all of the usual differentiation rules from calculus still apply here}. In particular, for any generating functions $F(x)$ and $G(x)$, the following identities hold: 

\begin{itemize}
  \item \textbf{Sum Rule:} $\displaystyle{\frac{d}{dx} (F(x)+G(x))=\frac{d}{dx}F(x)+\frac{d}{dx}G(x)}$
  \item \textbf{Product Rule:} $\displaystyle{\frac{d}{dx} (F(x)G(x))=G(x)\cdot\frac{d}{dx}F(x)+F(x)\cdot \frac{d}{dx}G(x)}$
  \item \textbf{Quotient Rule:} If $G(x)$ has a nonzero constant term $g_0x^0$, then $$\displaystyle{\frac{d}{dx} \left(\frac{F(x)}{G(x)}\right)=\frac{G(x)\frac{d}{dx}F(x)-F(x)\frac{d}{dx}G(x)}{G(x)^2}}$$
  \item \textbf{Chain Rule for monomials:} Let $H(x)=\frac{d}{dx}F(x)$ and let $ax^n$ be any monomial in $x$.  Then $\frac{d}{dx}(F(ax^n))=ax^{n-1}\cdot H(ax^n)$.
\end{itemize}

\begin{example}
  The derivative of the geometric series $1+x+x^2+x^3+x^4+x^5+\cdots$ is $1+2x+3x^2+4x^3+5x^4+\cdots.$  We therefore have \begin{align*}
      1+2x+3x^2+4x^3+5x^4+\cdots&=\frac{d}{dx}(1+x+x^2+x^3+x^4+\cdots) \\
      &=\frac{d}{dx}\left(\frac{1}{1-x}\right) \\
      &=\frac{1}{(1-x)^2}
  \end{align*}
  which gives another proof of the identity we found via multiplication in Example \ref{ex:mult}.
\end{example}

\subsection*{Exercises}

\begin{enumerate}
\item Find a closed form formula for the infinite series $1-x+x^2-x^3+\cdots$.
\item Prove that 
$1/(1-nx)=\sum_{k=0}^\infty n^k x^k$ for any fixed positive integer $n$.
    \item For any positive integer $m$, show that $1+x^m+x^{2m}+x^{3m}+\cdots=\frac{1}{1-x^m}$.
    
    \item 
    Use the derivative to find the coefficients 
    $a_0, a_1, a_2, \ldots$ that yield these closed form formulas:
    \begin{enumerate}
        \item 
    \[\frac{1}{(1-x)^3} = a_0 + a_1 x + a_2 x^2 + \cdots.\]
    \item \[\frac{1}{(1-x)^4} = a_0 + a_1 x + a_2 x^2 + \cdots.\]
    \item 
    \[\frac{1}{(1-x)^N} = a_0 + a_1 x + a_2 x^2 + \cdots.\]
    \end{enumerate}

    \item 
    Recall this fact about the sequence $1,2,3, \ldots$ whose generating function has the closed form $1/(1-x)^2$.
    Let $a_k = k+1$.  Note that $a_k$ is the number of ways to buy $k$ flags, from an infinite collection of gold and green flags.
    \begin{enumerate}
        \item Let $a_0, a_1, \ldots$ be the sequence found in part (a) of the previous problem, whose generating function has the closed form $1/(1-x)^3$.
        Find a formula for $a_k$.  Show that $a_k$ is the number of ways to buy $k$ flags, from an infinite collection of gold, green, and pink flags.
        \item Repeat using the sequence found in part (b) of the previous problem, 
        where the generating function has the closed form $1/(1-x)^4$ and there are $4$ colors of flags.
        \item Repeat using the sequence found in part (c) of the previous problem, 
        where the generating function has the closed form $1/(1-x)^N$ and there are $N$ colors of flags.
    \end{enumerate}

   \item Find another way to verify the formulas you found for (a), (b), and (c) in the two previous problems using the geometric series and multiplication of power series.  Your answer should include the formula for the number of ways to count $n$ objects where order does not matter and repeats are allowed.  Your answer should include an explanation, like the ones in the section about making change, about the meaning of the coefficient of $x^n$.    
    
    \item We define $E(x)$ to be the generating series for the sequence $a_n=\frac{1}{n!}$, whose first terms are $1,1,\frac{1}{2}, \frac{1}{6}, \ldots$.  In other words
    \[E(x) = 1 + x + \frac{1}{2!} x^2 + \frac{1}{3!} x^3 + \cdots.\]
    Use the process of taking a derivative of an infinite series to show that $d(E(x))/dx = E(x)$.  Explain why this makes $E(x)$ behave like the function $e^x$.

\end{enumerate}

\section{Solving recurrences with generating functions}
\label{sec:recurrence-gf}

\begin{videobox}
\begin{minipage}{0.1\textwidth}
\href{https://www.youtube.com/watch?v=whqRcIVKLOw}{\includegraphics[width=1cm]{video-clipart-2.png}}
\end{minipage}
\begin{minipage}{0.8\textwidth}
Click on the icon at left or the URL below for this section's short lecture video. \\\vspace{-0.2cm} \\ \href{https://www.youtube.com/watch?v=whqRcIVKLOw}{https://www.youtube.com/watch?v=whqRcIVKLOw}
\end{minipage}
\\
\begin{minipage}{0.1\textwidth}
\href{https://youtu.be/9DCkPPY5WVA}{\includegraphics[width=1cm]{video-clipart-2.png}}
\end{minipage}
\begin{minipage}{0.8\textwidth}
Click on the icon at left or the URL below for another short video discussing the second method utilized in this section. \\\vspace{-0.2cm} \\ \href{https://youtu.be/9DCkPPY5WVA}{https://youtu.be/9DCkPPY5WVA}
\end{minipage}
\end{videobox}

In this section, we will see how to find explicit formulas for recursively defined sequences using generating functions.  
Then we use this to investigate some of the nonlinear recurrences that we defined in Chapter \ref{chap:recurrence}.

The steps to find explicit formulas for recursively defined sequences are as follows:
\begin{enumerate}
    \item \textbf{Define the generating function.} If $a_0,a_1,a_2,\ldots$ is a recursively defined sequence, set $A(x)=a_0+a_1x+a_2x^2+\cdots$ to be its generating function.
    \item \textbf{Solve for the generating function.} Use the recursion that defines the sequence, along with the generating function operations we have studied, to find an equation that $A(x)$ satisfies.  Solve this equation to find a closed form formula for $A(x)$ in terms of $x$.  
    \item \textbf{Expand to compute the coefficients of the generating function.}  Expand the closed form formula for $A(x)$ found in the previous step as a generating function (this step will often involve the geometric series formula or some other well-known identity). If this expansion has a nice formula for the coefficient of $x^n$, then that gives an explicit formula for $a_n$.  
\end{enumerate}

\begin{example}
Consider the sequence defined by $a_0=3$ and for all $n\ge 1$, $a_n=2a_{n-1}$.  While it is not hard to find an explicit formula for this sequence without using generating functions, we'll start with this as a simple example to demonstrate the three steps of the method.

\begin{enumerate}
    \item \textbf{Define the generating function:} We define $$A(x)=\sum_{n=0}^\infty a_n x^n=a_0+a_1x +a_2x^2+a_3x^3+\cdots.$$ 
    \item \textbf{Solve for the generating function:} 
    We would like to find a closed form formula for $A(x)$.
    The recursion formula $a_n=2a_{n-1}$ tells us that each coefficient is twice the previous one.  By Example~\ref{ex:mult-by-x},  
    the coefficient of $x^n$ in $xA(x)$ is $a_{n-1}$ for every $n \geq 1$.
    We can also multiply the result by $2$ to get:
$$2xA(x)=\sum_{n=1}^{\infty} 2a_{n-1}x^n=\sum_{n=1}^\infty a_n x^n$$ where the second equality above follows from the recursion $a_{n}=2a_{n-1}$.  Now, the right hand side above looks almost like $A(x)$ --- but wait!  The summation starts at $n=1$, not $n=0$.  Let's fix this by subtracting and adding the $a_0$ term back in (which does not change the equality) to get an equation for $A(x)$:
$$2xA(x)=-a_0+a_0+\sum_{n=1}^\infty a_n x^n=-a_0+\sum_{n=0}^\infty a_n x^n =-a_0 + A(x)=-3 + A(x).$$
We have nearly completed step 2 --- solving the equation $2xA(x)=-3 + A(x)$ for $A(x)$, we have $(2x-1)A(x)=-3$, so $A(x)=3/(1-2x)$.
    \item \textbf{Expand to compute the coefficients:}  By substituting $2x$ in for $x$ in the geometric series formula, and then multiplying by $3$, we see that $$A(x)=\frac{3}{1-2x}=3\sum_{n=0}^\infty (2x)^n=3\sum_{n=0}^\infty 2^nx^n=\sum_{n=0}^\infty 3\cdot 2^n x^n.$$ 
\end{enumerate}
 Having completed the three steps, we now have the big payoff: we know that $A(x)=\sum_{n=0}^\infty a_n x^n$, but also $A(x)=\sum_{n=0}^\infty 3\cdot 2^n x^n$.  Since these are the same generating function, we can conclude that $a_n=3\cdot 2^n$ for all $n$, and voila!  We've found an explicit formula for $a_n$.
\end{example}

\textit{Tip:} It is often useful to get a more concrete handle on this generating function in steps 1 and 2 by using the recursion to calculate the first few terms of the sequence.  Starting with $a_0=3$, we have $a_1=2a_0=6$, $a_2=2a_1=12$, and $a_3=2a_2=24$. 
So
\begin{eqnarray*}
A(x)  = & 3+ & 6x+12x^2+24x^3+\cdots \\
2xA(x)  = & & 6x + 12x^2 + 24 x^3 + \dots
\end{eqnarray*}
By lining up the terms for $A(x)$
and $2xA(x)$, we see that $2x A(x) - A(x) = 3$.

\subsection{The Tower of Hanoi}
In Section \ref{sec:hanoi}, we studied the Tower of Hanoi game.  We found that the recurrence relation for the number $H_n$ of steps it takes to solve satisfies $H_1=1$ and $H_n=2H_{n-1}+1$ for all $n\ge 2$. Notice that if we set $H_0=0$ we get a sequence that satisfies the recurrence for all $n\ge 1$.  Let's use this recursion to try to solve for the generating function of the sequence $H_0,H_1,H_2,\ldots$, by going through the three steps again:
\begin{enumerate}
    \item \textbf{Define the generating function:} First set $H(x)$ to be the generating function:
$$H(x)=H_0+H_1x+H_2x^2+H_3x^3+\cdots.$$  

    \item \textbf{Solve for the generating function:}  We can use the recursion to replace the coefficients $H_1,H_2,\ldots$ with $2H_0+1$, $2H_1+1$, $2H_2+1$ and so on respectively:
$$H(x)=H_0+(2H_0+1)x+(2H_1+1)x^2+(2H_2+1)x^3+\cdots.$$ We can set $H_0=0$, and then by the addition rule for generating functions, we can split the remaining terms up as a sum of generating functions: 
\begin{align}
    H(x)&=((2H_0)x+(2H_1)x^2+(2H_2)x^3+\cdots)+(x+x^2+x^3+\cdots) \\
    H(x)&=2x(H_0+H_1x+H_2x^2+\cdots)+x(1+x+x^2+\cdots) \\
    H(x)&=2xH(x)+\frac{x}{1-x}
\end{align}
where the last simplification is by the definition of $H(x)$ and the geometric series formula.  (Notice that we did this step using the $\cdots$ notation by writing out the first few terms of each generating function, rather than using $\sum$ notation like in the previous example.  Can you rewrite it using $\sum$ notation?)

We can now gather all the $H(x)$ terms on the left hand side to solve for $H(x)$:
\begin{align}
    H(x)-2xH(x)&=\frac{x}{1-x} \\
    H(x)(1-2x)&=\frac{x}{1-x} \\
    H(x)&=\frac{x}{(1-x)(1-2x)}
\end{align}
\item \textbf{Expand to compute the coefficients:} We'd like to break up the fraction as a sum of two fractions to make it look more like the geometric series formula.  Luckily, there's a tool from algebra that does just that: the method of \textbf{partial fractions}.  This states that there are numbers $a$ and $b$ such that
$$\frac{x}{(1-x)(1-2x)} = \frac{a}{1-x}+\frac{b}{1-2x}.$$ 
There are a few ways to see that $a=-1$ and $b=1$, for example, by finding a common denominator for the right hand side.  
(Try adding the fractions to double check this yourself!).
So we have $$H(x)=\frac{1}{1-2x}-\frac{1}{1-x}.$$ Then we can expand both terms using the geometric series formula: $$H(x)=\sum_{n=0}^\infty 2^n x^n -\sum_{n=0}^\infty x^n=\sum_{n=0}^\infty (2^n-1)x^n.$$ 
\end{enumerate}

Therefore, the coefficient $H_n$ is equal to $2^n-1$ for all $n$.

\subsection{Regions of the plane}

In Section \ref{sec:regions-plane}, we showed that the number $P_n$ of regions in the plane cut out by drawing $n$ lines (with no two parallel and no three concurrent) satisfies the recurrence $P_0=1$ and $P_n=P_{n-1}+n$.  We go through the steps once again.

\begin{enumerate}
    \item \textbf{Define the generating function:}  We set $$P(x)=\sum_{n=0}^\infty P_n=P_0+P_1x+P_2x^2+P_3x^3+\cdots.$$
    \item \textbf{Solve for the generating function:} In order to make use of the recursion $P_n=P_{n-1}+n$, we multiply $P(x)$ by $x$ to obtain the generating function for $P_{n-1}$:  $$\sum_{n=1}^\infty P_{n-1}x^n=xP(x)$$ and we also compute the generating function for $n$: $$\sum_{n=1}^\infty n x^n=x+2x^2+3x^3+4x^4+\cdots=x(1+2x+3x^2+4x^3+\cdots)=\frac{x}{(1-x)^2}$$ by Example \ref{ex:mult}.
    Adding these two generating functions, we have \begin{align*}
        \sum_{n=1}^\infty (P_{n-1}+n)x^n&=xP(x)+\frac{x}{(1-x)^2} \\
        \sum_{n=1}^\infty P_nx^n &=xP(x)+\frac{x}{(1-x)^2} \\
        P(x)-P_0 &= xP(x)+\frac{x}{(1-x)^2} \\
        P(x)-xP(x) &= P_0+\frac{x}{(1-x)^2} \\
        P(x)(1-x) &= 1+\frac{x}{(1-x)^2} \\
        P(x) &= \frac{1}{1-x}+\frac{x}{(1-x)^3}
    \end{align*}
    \item \textbf{Expand to compute the coefficients:} The first term $\frac{1}{1-x}$ can be expanded using the geometric series formula, but we haven't seen the generating function $\frac{x}{(1-x)^3}$ yet.  For this, we can take another derivative of the generating function $\frac{1}{(1-x)^2}$, which we've already computed (Example \ref{ex:mult}):
    \begin{align*}
    \frac{1}{(1-x)^2} &= \sum_{n=0}^\infty (n+1)x^n\\
        \frac{d}{dx}\left((1-x)^{-2}\right) &= \frac{d}{dx}\left(\sum_{n=0}^\infty (n+1)x^n\right) \\
        \frac{2}{(1-x)^3} &=\sum_{n=1}^\infty n(n+1)x^{n-1} \\
        \frac{1}{(1-x)^3} &= \frac{1}{2}\sum_{n=1}^\infty n(n+1)x^{n-1} \\
        \frac{x}{(1-x)^3} &=  \sum_{n=1}^\infty \frac{n(n+1)}{2} x^n
    \end{align*}
    where the last equality forms by multiplying both sides by $x$.  Notice that since this sum has constant term $0$, we can also start the indexing of $n$ at $0$ to write it as $\sum_{n=0}^\infty \frac{n(n+1)}{2} x^n$.
    
    We can finally substitute to find 
    \begin{align*}
        P(x) &= \frac{1}{1-x}+\frac{x}{(1-x)^3} \\
            &= \sum_{n=0}^\infty x^n + \sum_{n=0}^\infty \frac{ n(n+1)}{2}x^n \\
             &= \sum_{n=0}^\infty \left(1+\frac{n(n+1)}{2}\right)x^n \\
             &= \sum_{n=0}^\infty \left(\frac{n^2+n+2}{2}\right)x^n 
    \end{align*}
\end{enumerate}
We can conclude that $P_n=\frac{n^2+n+2}{2}$ for all $n$, as we found in Section \ref{sec:regions-plane}.

\subsection*{Exercises}

      \begin{enumerate}
          \item Find a closed form for the generating functions for the following recursively defined sequences:
          \begin{enumerate}
              \item $a_0 = 1$, $a_1 = 5$, $a_{n} = 4a_{n-1} -3a_{n-2}$ for $n\ge 2$.
          \item $b_0 = 1$, $b_1 = 6$, $b_{n} = 4b_{n-1} - 4b_{n-2}$ for $n\ge 2$. 
          \item $c_0 = 1$, $c_1=1$, $c_2=1$, $c_{n}=c_{n-1}+c_{n-2}+c_{n-3}$ for $n\ge 3$.
          \item $d_0=1$, $d_{n}=2d_{n-1}+n+1$ for $n\ge 1$.
          \item $e_0=1$, $e_n=e_{n-1}+e_{n-2}+\cdots+e_2+e_1+e_0$ for $n\ge 1$.
          \item $h_0=1$,  $h_n=\left(-\frac{1}{n}\right)h_{n-1}$ for $n\ge 1$.
          \end{enumerate}
         
    \item In part (a) above, use partial fractions on the generating function to find an explicit formula for the sequence $a_n$.
    
    \item In part (b) above, use monomial substitution into the formula $\frac{1}{(1-x)^2}=1+2x+3x^2+\cdots$ to find an explicit formula for $b_n$.
          
    \item Consider the \textit{double recursion} defined by $f_0=1$, $f_1=2$, $g_0=1$, $g_1=1$, $f_n=f_{n-1}+g_{n-2}$, $g_n=g_{n-1}+f_{n-1}$.  Find the generating functions $F(x)$ and $G(x)$ of the sequences $f_n$ and $g_n$.
      \end{enumerate}

\section{Generating functions and linear recurrences}

We now show how to solve linear recurrences using generating functions.  
\begin{example}
Suppose the sequence $a_0,a_1,a_2,\ldots$ is defined by $a_0=0$, $a_1=1$, and the recurrence relation  $a_n=5a_{n-1}-6a_{n-2}$ for $n\ge 2$.  
The first several terms of this sequence are $0, 1, 5, 19, 89, \ldots, $ 
Then we define $A(x)$ to be the generating function for this sequence: $$A(x)=a_0+a_1x+a_2x^2+\cdots=\sum_{n=0}^\infty a_n x^n.$$ 
Its first several terms are:
\[0 + x + 5x^2 + 19x^3 + 89x^4 + \cdots.\]
\end{example}
The next lemma shows how we can use generating functions to find a closed-form 
formula for $a_n$ for all $n$. 

\begin{lemma}
If $a_0=0$, $a_1=1$, and $a_n=5a_{n-1}-6a_{n-2}$ for $n\ge 2$, then 
$a_{n}=3^n-2^n$ for all $n$.
\end{lemma}

\begin{proof}
The first step is to use the recurrence relation to find a closed form formula for $A(x)$.  
For all $n\ge 2$, we can substitute $a_n=5a_{n-1}-6a_{n-2}$, so we can write \begin{align*}
    A(x)&=a_0+a_1x+\sum_{n=2}^\infty a_n x^n \\
        &=a_0+a_1x+\sum_{n=2}^\infty (5a_{n-1}-6a_{n-2}) x^n \\
        &=0+1x+5\left(\sum_{n=2}^\infty a_{n-1}x^n\right)-6\left(\sum_{n=2}^\infty a_{n-2}x^n\right).
 \end{align*}   We re-index the first infinite sum in the line above, replacing $n$ by $n+1$, and the second infinite sum in the line above by replacing $n$ by $n+2$.  After that, we can factor out an $x$ from each term in the first infinite sum and an $x^2$ in the second infinite sum.  
So
 \begin{align*}
    A(x)&=       x+5\left(\sum_{n=1}^\infty a_nx^{n+1}\right)-6\left(\sum_{n=0}^\infty a_{n}x^{n+2}\right) \\
        &=x+5x\left(\sum_{n=1}^\infty a_n x^n\right)-6x^2\left(\sum_{n=0}^\infty a_n x^n\right)
         \end{align*} 
      Now the first infinite sum is the sum as $A(x)$ without its constant term (which is zero anyway) and the second infinite sum is the same as $A(x)$.  So 
  \begin{align*}
    A(x)&=        
     x+5xA(x)-6x^2A(x) \\
        &=x+(5x-6x^2)A(x)
\end{align*}

We can now solve this equation for $A(x)$ and then factor the denominator to find $$A(x)=\frac{x}{1-5x+6x^2}=\frac{x}{(1-2x)(1-3x)}$$
We now use the method of \textit{partial fractions} to write this as a sum of two fractions, of the form $$A(x)=\frac{r}{1-2x}+\frac{s}{1-3x}.$$ To solve for $r$ and $s$, we combine the fractions to get $\frac{r+s-(3r+2s)x}{(1-2x)(1-3x)}$, so $r+s-(3r+2s)x=x$.  Therefore $r+s=0$ and $3r+2s=-1$.  Solving, we find $r=-1$ and $s=1$.  Therefore  $$A(x)=\frac{-1}{1-2x}+\frac{1}{1-3x},$$ and we can now expand both fractions using the geometric series theorem.  Indeed, we have $$A(x)=-\sum_{n=0}^\infty 2^nx^n+\sum_{n=0}^\infty 3^n x^n=\sum_{n=0}^\infty (3^n-2^n)x^n.$$  Since this is exactly the same generating function as $\sum_{n=0}^\infty a_n x^n$, we can conclude that $a_{n}=3^n-2^n$ for all $n$.  
\end{proof}

\begin{remark}
In the above example, here is a different way to solve for $A(x)$. 
We recall that multiplying $A(x)$ by $x$ shifts the coefficients by one index and multiplying by $x^2$ shifts the coefficients by two indices.  In particular we have $$xA(x)=a_0x+a_1x^2+a_2x^3+\cdots = \sum_{n=1}^\infty a_{n-1}x^n$$ and $$x^2A(x)=a_0x^2+a_1x^3+a_3x^4+\cdots =\sum_{n=2}^\infty a_{n-2}x^n.$$  Therefore, if we add together $A(x)$, $-5xA(x)$, and $6x^2A(x)$, most of the coefficients will cancel since $a_n-5a_{n-1}+6a_{n-2}=0$ for all $n\ge 2$:

\begin{center}
\begin{tabular}{rcccccccccc}
    $A(x)$&=& $a_0$ &$+$& $a_1x$ &$+$& $a_2x^2$ &$+$& $a_3x^3$ &$+$& $\cdots$ \\
    $-5xA(x)$&$=$& &$-$&$5a_0x$&$-$&$5a_1x^2$&$-$&$5a_2x^3$&$+$&$\cdots$ \\
    $+6x^2A(x)$&$=$& & & & &$6a_0x^2$&$+$&$6a_1x^3$&$+$&$\cdots$ \\\hline
    $A(x)-5xA(x)+6x^2A(x)$&$=$&$a_0$&$+$&$(a_1-5a_0)x$&$+$&$0x^2$&$+$&$0x^3$&$+$&$\cdots$
\end{tabular}
\end{center}

Therefore, $A(x)(1-5x+6x^2)=a_0+(a_1-5a_0)x=x$.
\end{remark}

Notice in the previous calculation that we multiplied $A(x)$ by something very close to the \textbf{characteristic polynomial} of the recursion $a_n-5a_{n-1}+6a_{n-2}=0$ (see Definition \ref{def:charpoly}), which is $x^2-5x+6$.  Indeed, we multiply here by the \textit{reverse} of the characteristic polynomial, defined by reversing the sequence of coefficients to make it $1-5x+6x^2$ instead.  

We now apply this principle to solve another recursion.

\begin{example}
Let's use generating functions to derive the explicit formula for the $n$th Fibonnacci number $F_n$.  To do this, it is easier to start with the index $n=0$, so we write $F_0=0$, $F_1=1$, and $F_{n}=F_{n-1}+F_{n-2}$ for all $n\ge 2$.   Consider the generating function $$G(x)=\sum_{n=0}^\infty F_n x^n.$$  As in the previous problem, we multiply by the reverse of the characteristic polynomial, $1-x-x^2$, to get
\begin{eqnarray*} G(x)-xG(x)-x^2G(x)&  =&  F_0+F_1x-F_0x+\sum_{n=2}^\infty (F_n-F_{n-1}-F_{n-2})x^n\\
& = & 0 + 1 \cdot x - 0 \cdot x + \sum_{n=2}^\infty 0 \cdot x^n =x.
\end{eqnarray*}
Therefore $G(x)=x/(1-x-x^2)$.  We can factor the denominator as $(1-ax)(1-bx)$ where $a+b=1$ and $ab=-1$.   Solving, we get $a=\frac{1+\sqrt{5}}{2}$ and $b=\frac{1-\sqrt{5}}{2}$.  Using partial fractions, we can write $$G(x)=\frac{A}{1-ax}+\frac{B}{1-bx}.$$ To find $A$ and $B$, we add the fractions and compare coefficients: $A(1-bx)+B(1-ax)=x$, so $A+B=0$ and $Ab+Ba=-1$.  Thus $A=-B$ and so $A(b-a)=-1$.  Since $a-b=\sqrt{5}$ we have $A=\frac{1}{\sqrt{5}}$ and $B=-\frac{1}{\sqrt{5}}$.

We now have $$G(x)=\frac{1/\sqrt{5}}{1-ax}-\frac{1/\sqrt{5}}{1-bx}.$$ Expanding each term as a geometric series, we find that $$G(x)=\sum_{n=0}^\infty \frac{1}{\sqrt{5}}a^nx^n-\sum_{n=0}^\infty \frac{1}{\sqrt{5}}b^nx^n=\sum_{n=0}^\infty \frac{1}{\sqrt{5}}\left(\left(\frac{1+\sqrt{5}}{2}\right)^n-\left(\frac{1-\sqrt{5}}{2}\right)^n\right)x^n.$$ By the definition of equality of generating functions, it follows that $$F_n=\frac{1}{\sqrt{5}}\left(\left(\frac{1+\sqrt{5}}{2}\right)^n-\left(\frac{1-\sqrt{5}}{2}\right)^n\right).$$
\end{example}

We can now generalize our methods above to prove the following theorem, which we stated before without proof in Chapter \ref{chap:recurrence}.

\begin{theorem}
Let $r_1, \ldots, r_d$ be the roots of the characteristic polynomial of the recurrence relation $a_n=c_{1}a_{n-1}+c_2a_{n-2}+\cdots +c_da_{n-d}=0$.  Suppose that the $d$ roots are all different.
Then there are constants $z_1, \ldots, z_d$ such that, if $n \geq 1$, then
\begin{equation}\label{Esolgen}
a_n = z_1 r_1^n + \cdots z_d r_d^n.    
\end{equation}
Furthermore, there is a unique choice of $z_1, \ldots, z_d$ such that
\eqref{Esolgen} is true.
\end{theorem}

\begin{proof}
The characteristic polynomial 
of the recurrence relation is 
$c(x)=x^d - c_{1}x^{d-1} - c_2x^{d-2} - \cdots 
- c_{d-1} x - c_d$.  Reversing its coefficients, we get the polynomial $$r(x)=1-c_1x-c_2x^2-\cdots-c_dx^d.$$  Now, let $A(x)=\sum_{n=0}^\infty a_n x^n$ be the generating function of the recursively defined sequence $a_0,a_1,\ldots$.  Then by using the generating function multiplication rule on $r(x)A(x)$, we see that all the coefficients starting at $x^d$ and beyond are $0$.  Thus $r(x)A(x)=p(x)$ for some polynomial $p(x)$, and so $$A(x)=\frac{p(x)}{r(x)}.$$  Now, we claim that $r(x)$ factors as $(1-r_1x)(1-r_2x)\cdots (1-r_dx)$.  Indeed, since $r_1,\ldots,r_d$ are defined to be the roots of $c(x)$, we have $c(x)=(x-r_1)(x-r_2)\cdots (x-r_d)$.  But notice that $r(x)=x^dc(1/x)=x^d(1/x-r_1)(1/x-r_2)\cdots (1/x-r_d)$.  Multiplying each factor through by one of the $x$'s from the $x^d$ factor, we get $r(x)=(1-r_1x)(1-r_2x)\cdots (1-r_dx)$.

Putting this all together, we have $$A(x)=\frac{p(x)}{r(x)}=\frac{p(x)}{(1-r_1x)(1-r_2x)\cdots (1-r_d)x}.$$ Since the roots are assumed to be all different, this decomposes using partial fractions as a sum of fractions of the form $$\frac{z_1}{1-r_1x}+\frac{z_2}{1-r_2x}+\cdots +\frac{z_d}{1-r_dx}$$ for some constants $z_1,\ldots,z_d$.  Expanding each of these terms using the geometric series formula and adding the generating functions, we find that the $n$-th coefficient of the series is indeed $z_1r_1^n+\cdots +z_dr_d^n$.
\end{proof}

\section{The Catalan numbers}

We end this chapter by deriving the generating function and explicit formula for the celebrated Catalan numbers (see Section \ref{sec:Catalan}).  Recall that the \textbf{Catalan numbers} are defined by the recursion \begin{equation}
    \label{eq:Catalan} C_{n+1}=C_{0}C_n+C_1C_{n-1}+C_2C_{n-2}+\cdots + C_nC_0,
\end{equation} with initial condition $C_0=1$.  In particular, for the first few values of $n$, this recursion says that \begin{align}
C_1&=C_0C_0 \label{eq:c1}\\
C_2&=C_0C_1+C_1C_0 \label{eq:c2}\\
C_3&=C_0C_2+C_1C_1+C_2C_0 \label{eq:c3}
\end{align} and so on.  
\begin{enumerate}
    \item \textbf{Define the generating function:}  We set $$C(x)=C_0+ C_1 x +C_2x^2+C_3x^3+\cdots.$$ 
    \item \textbf{Solve for the generating function:}  Notice that the right hand side of the Catalan recursion (\ref{eq:Catalan}) very closely resembles the coefficients we get when we multiply $C(x)$ by itself.  Indeed:
$$C(x)^2=C(x)\cdot C(x)=C_0C_0+(C_0C_1+C_1C_0)x+(C_0C_2+C_1C_1+C_2C_0)x^2+\cdots$$ by the multiplication rule for generating functions.  The expressions in parentheses on the right hand side exactly match equations (\ref{eq:c1}), (\ref{eq:c2}), (\ref{eq:c3}), and in general the coefficient of $x^n$ will match the right hand side of equation (\ref{eq:Catalan}).  Substituting these in, we have $$C(x)^2=C_1+C_2x+C_3x^2+C_4x^3+\cdots.$$  We now want to re-express the right hand side in terms of $C(x)$ to solve for $C(x)$.  In order to line up the subscripts with the right powers of $x$, we can multiply both sides by $x$:
$$xC(x)^2=C_1x+C_2x^2+C_3x^3+C_4x^4+\cdots.$$  The right hand side is almost $C(x)$, but it is missing the $C_0$ term at the start.  So let's add that to both sides:
$$xC(x)^2+C_0=C_0+C_1x+C_2x^2+C_3x^3+C_4x^4+\cdots.$$
Now the right hand side is $C(x)$, and we can substitute $C_0=1$ on the left, to get $xC(x)^2+1=C(x)$,  
or $$xC(x)^2-C(x)+1=0.$$
We can now use the quadratic formula to solve for $C(x)$: 

$$C(x)=\frac{1\pm \sqrt{1-4x}}{2x}$$

    \item \textbf{Expand to compute the coefficients:} To interpret $\sqrt{1-4x}$ as a generating function, we want to think of it as ``the generating function whose square is $1-4x$''.  To get an intuition for what this looks like, we can start determining the coefficients one by one, starting with the equation: 
\begin{align*}
    (a_0+a_1x+a_2x^2+a_3x^3+\cdots )^2 &= 1-4x; \\
    a_0^2+(a_0a_1+a_1a_0)x+(a_0a_2+a_1a_1+a_2a_0)x^2+\cdots&=1-4x.
\end{align*}
By comparing the coefficients of $x^n$ from the left and right sides of this equation, we can solve for $a_0, a_1, a_2, \ldots$ one at a time.
From the coefficients of $x^0$, we see that $a_0^2=1$, and we take the positive solution $a_0=1$.\footnote{ This is similar to the convention of setting $\sqrt{4}$ to be $2$ rather than $-2$, even though both numbers have a square of $4$.}  
From the coefficients of $x^1$, we see that $a_0a_1+a_1a_0=-4$, and substituting $a_0=1$, we find that $a_1=-2$.  From the coefficients of $x^2$, we see that $a_0a_2+a_1a_1+a_2a_0=0$, and substituting $a_0=1$, $a_1=-2$, we find that $a_2=-2$.  Continuing in this manner, we claim that once the value of $a_0, a_1, \ldots, a_{n-1}$ are found, then there is a unique choice of $a_n$ that makes the coefficients of $x^n$ from the left and the right side of the equation equal.
This provides a well-defined ``square root'' of $1-4x$ as a generating function: $$\sqrt{1-4x}=1-2x-2x^2-4x^3-\cdots.$$ This also means that in the formula $C(x)=\frac{1\pm \sqrt{1-4x}}{2x}$, we should use the negative term in the symbol $\pm$, so that the constant term $1$ cancels and all the terms in the numerator are divisible by $2x$. 
\end{enumerate}

We've found the first several coefficients for the generating function of $\sqrt{1-4x}$ but now we want to find a formula for all of them.
To do this, we will introduce a generalization of binomial numbers:

\begin{definition}[Generalized Binomial Coefficients.]
Let $\alpha$ be any rational number, and define $$\binom{\alpha}{k}=\frac{\alpha(\alpha-1)(\alpha-2)\cdots(\alpha-k+1)}{k!}$$ 
if $k \not = 0$ and $\binom{\alpha}{0}=1$.
\end{definition}

If $\alpha$ is a positive integer $n$, then this definition is the same as the ordinary definition of $\binom{n}{k}$.

\begin{example}
If we set $\alpha=1/2$ and $k=3$ in the above formula, we have $$\binom{1/2}{3}=\frac{(1/2)(-1/2)(-3/2)}{3!}=\frac{3/8}{6}=\frac{1}{16}.$$
\end{example}

\begin{example} \label{Ehalfbinom1}
Let $\alpha = 1/2$ and $k=n+1$.
Then 
\begin{eqnarray*}
\binom{1/2}{n+1}&=
&\frac{(1/2)(1/2-1)(1/2-2)\cdots (1/2-n)}{(n+1)!} \\
&=& \frac{(1/2)(-1/2)(-3/2)(-5/2)\cdots((1-2n)/2)}{(n+1)!} \\
&=&\frac{(1)(3)(5)\cdots(2n-1) (-1)^n}{2^{n+1}(n+1)!}\\
& = & \frac{(1)(2)(3)(4)(5)\cdots(2n-1)(2n) (-1)^n}{(2)(4)\cdots (2n) 2^{n+1}(n+1)!}\\
& = & \frac{(2n)! (-1)^n}
{(2\cdot 1)(2 \cdot 2)
(2 \cdot 3)\cdots (2\cdot n) 2^{n+1}(n+1)!}\\
& = & \frac{(2n)! (-1)^n}{2^n 2^{n+1}n!(n+1)!}
\end{eqnarray*}
\end{example}

To find a formula for the coefficients of $\sqrt{1-4x}$, we'll use (without proof) this next theorem.

\begin{theorem}[Generalized Binomial Theorem.] \label{Tgbt} As a generating function, $$(1+x)^\alpha = \sum_{k=0}^\infty \binom{\alpha}{k}x^k.$$ 
\end{theorem}

\begin{example} \label{Ehalfbinom2}
We set $\alpha=1/2$ and substitute $-4x$ for $x$ to see that:
\begin{eqnarray*}
\sqrt{1-4x} & = &  \sum_{k=0}^\infty \binom{1/2}{k}(-4x)^k\\ 
& = &  \sum_{k=0}^\infty \binom{1/2}{k}(-1)^k4^kx^k
\end{eqnarray*}
\end{example}

We use this now to find the explicit formula $C_n=\frac{1}{n+1}\binom{2n}{n}$ for the $n$th Catalan number.

\begin{corollary}
The $n$th Catalan number is $C_n=\frac{1}{n+1}\binom{2n}{n}$. 
\end{corollary}

\begin{proof}
The $n$th Catalan number $C_n$ is the coefficient of $x^n$ in $C(x)=\frac{1\pm \sqrt{1-4x}}{2x}$.
Using Example~\ref{Ehalfbinom2}, we see that
\begin{eqnarray*}
    \frac{1- \sqrt{1-4x}}{2x} &=& 
    \frac{1-\left(\sum_{k=0}^\infty 
    \binom{1/2}{k}(-1)^{k}4^{k}x^{k}\right)}{2x}\\
    &=& 
    \frac{1-\left(1+\sum_{k=1}^\infty 
    \binom{1/2}{k}(-1)^{k}4^{k}x^{k}\right)}{2x}\\
    & = & - \sum_{n=0}^\infty
    \frac{\binom{1/2}{n+1}(-1)^{n+1}2^{n+1}2^{n+1}x^{n+1}}{2x}\\
    & = & \sum_{n=0}^\infty 
    \binom{1/2}{n+1}(-1)^{n}2^n2^{n+1}x^{n}
    \end{eqnarray*}
In the third equality, we made the substitution $n=k-1$.    
    
  Using Equation~\ref{Ehalfbinom1}, 
  we substitute for the binomial coefficients to see that the coefficient of $x^n$ is:
    \begin{eqnarray*}
    C_n &=& 
    \frac{(2n)! (-1)^n}{2^n 2^{n+1}n!(n+1)!}
    (-1)^{n}2^n2^{n+1} \\
    &=&\frac{(2n)!}{n!(n+1)!} \\
    &=&\frac{1}{(n+1)}\frac{(2n)!}{n!n!} \\
    &=&\frac{1}{n+1} \binom{2n}{n}.
\end{eqnarray*}
This completes the proof.
\end{proof}

\section{Additional problems for Chapter~7}

\begin{enumerate}

\item Suppose you have 8 pennies, 3 nickels, and 1 dime in your pocket and you want to give a subset of these to your friend. For each of pennies, nickels, and dimes, write down a polynomial whose exponents are the values of that coin you could give.																																										
\item Define $f_p$, $f_n$, $f_d$ to be the polynomials for pennies, nickels, and dimes in Sage (sagecell.sagemath.org), as in:			
$$f_d=1+x^{10}.$$Use Sage to also define $f_p$ and $f_n$, and expand the product F of these 3 polynomials and use the output to answer the questions below. \\

\noindent Note: If you just end the Sage cell with F rather than the F.coefficients line, it will return the entire polynomial. The coefficients line will just return the coefficient of $x^{17}$ and the number 17.				\begin{enumerate}	
\item What is the degree of the polynomial F? In other words, what is the largest power of x that appears?																									
\item How many ways can you give 17 cents?																									
\item How many ways can you give 20 cents?																									
\item If the number of dimes changes to 2, what does the degree of the polynomial change to?																									
\item If the number of dimes changes to 2, what is the new number of ways you can give 20 cents?\\

\noindent \textbf{Hint:} 
These commands may be useful:
\begin{verbatim}
R=PolynomialRing(RR, x)	
F=expand(fp*fn*fd);				
F;			
F.coefficients()[17]
\end{verbatim}

\end{enumerate}	

\item What is the coefficient of $x^4$ in the generating function whose closed form is $1/(1-3x)$?																									
\item What is the coefficient of $x^4$ in the generating function $x^2/(1-3x)$?																									
\item What is the coefficient of $x^4$ in the generating function $1/(1-3x)^2$?																									
\item Let $$F(x)=F_0+F_1x+F_2x^2+F_3x^3+\cdots$$ be the generating function of the Fibonacci numbers (with $F_0=0,F_1=1$, and $F_n=F_{n-1}+F_{n-2} $ for all $n\geq 2$.) Compute the coefficient of $x^5$ in the generating function $1/1-xF(x)$.																									
\item Let $F(x)$ be the generating function of the Fibonacci sequence as defined in the previous question. What is the coefficient of $x^5$ in $d/dx$ $F(x)$?																									
																							
\item Find a closed formula for the generating function of the recursively defined sequence defined by $b_0=5,b_1=12$, and $b_n=5b_{n-1}+6b{n-2}$ for all $n\geq2$.																									
\item Find a closed formula for the generating function of the sequence defined by $e_0=1$ and $e_n=e_{n-1}+e_{n-2}+\cdots +e_0$ for all $n\geq1$.																									

   \item Prove, using generating functions, that $$\sum_{k=0}^n k \binom{n}{k}=n\cdot 2^{n-1}.$$
   
   \item We found that the generating function for the Fibonacci numbers is $\sum_{n=0}^\infty F_n x^n=\frac{x}{1-x-x^2}$.  Write this as $\frac{x}{1-(x+x^2)}$ and use the geometric series formula to expand in terms of $x+x^2$ to prove that $$F_n=\sum_{k=0}^{n-1} \binom{n-k-1}{k}.$$
   
      \item \textbf{Infinite products:} Let $G(x)=\sum_{i=0}^\infty b_i x^i$ be a generating function.  We say $G(x)=\prod_{n=0}^\infty F_n(x)$ if the partial products $P_m(x)=\prod_{n=0}^{m} F_n(x)$ have the property that for all $i$, there exists $N$ such that if $m>n$, the coefficient of $x^i$ in $P_m$ equals $b_i$.  In other words, the $P_m$'s coefficients eventually match that of $G$.
   
   \begin{enumerate}
     \item What generating function is equal to the infinite product \[(1+x)(1+x^2)(1+x^4)(1+x^8)(1+x^{16})\cdots?\]
     \item Let $p(n)$ be the number of \textit{partitions} of $n$, the number of ways of writing $n$ as a sum of other positive integers, where the order of the summands doesn't matter.  Show that $$\sum_{n=0}^\infty p(n)x^n=\frac{1}{(1-x)}\cdot \frac{1}{(1-x^2)}\cdot \frac{1}{(1-x^3)}\cdot \frac{1}{(1-x^4)}\cdots$$
     \item Prove that the number of partitions of $n$ having distinct summands is equal to the number of partitions of $n$ having all its summands being odd. (Hint: can you factor the generating functions of each of the two types of partitions?)
   \end{enumerate}
 
 \newpage
 \section{Investigation: Partitions and their generating functions}
 
 Generating functions can be used to enumerate many different combinatorial objects.  One very famous such object is called a \textit{partition}.  The following bonus video explores this topic.
 
 \begin{videobox}
\begin{minipage}{0.1\textwidth}
\href{https://www.youtube.com/watch?v=Lghi2Otq5Q4}{\includegraphics[width=1cm]{video-clipart-2.png}}
\end{minipage}
\begin{minipage}{0.8\textwidth}
Click on the icon at left or the URL below for this section's video. \\\vspace{-0.2cm} \\ \href{https://www.youtube.com/watch?v=Lghi2Otq5Q4}{https://www.youtube.com/watch?v=Lghi2Otq5Q4}
\end{minipage}
\end{videobox}
 
\end{enumerate}

%% file: 08-GraphsWalksCycles/GraphsWalksCycles.tex
\chapter{Graph Theory Basics}\label{chap:graphtheory}

Graphs are a convenient data structure for a wide variety of problems in mathematics and computer science.
Some graphs are valuable because they illustrate beautiful patterns and complicated symmetries.
Other graphs are valuable for their applications.
If an optimization problem can be described in terms of a graph, then diverse and powerful tools can be used to solve that problem.

In this chapter, we introduce the basic objects of graph theory, including vertices, edges, degrees, subgraphs, paths, cycles, and data structures. Later in the chapter, we describe results about walks on graphs, including Euler cycles, Hamiltonian circuits, and the number of walks. 
In Chapter~\ref{chap:trees}, we study a particular type of graph --- trees --- which are deceptively simple because they have no loops, but have deep impacts because they can model hierarchical relationships.
We solve optimization (maximization or minimization) problems on graphs in Chapter~\ref{chap:optimization}, and consider graphs that can be drawn with no crossing edges in Chapter~\ref{chap:planar}.

\begin{videobox}
\begin{minipage}{0.1\textwidth}
\href{https://www.youtube.com/watch?v=bfr4Y4PIXl4}{\includegraphics[width=1cm]{video-clipart-2.png}}
\end{minipage}
\begin{minipage}{0.8\textwidth}
Click on the icon at left or the URL below for a video covering sections 8.1--8.6. \\\vspace{-0.2cm} \\ \href{https://www.youtube.com/watch?v=bfr4Y4PIXl4}{https://www.youtube.com/watch?v=bfr4Y4PIXl4}
\end{minipage}
\end{videobox}

\section{Graphs in combinatorics}

As Students C and D are finishing lunch at Corbett Hall, a classic Colorado hail storm arrives.  Not wanting to walk, they decide to take the bus.  They first need to stop at the library to pick up a textbook.  Next, they need to attend their combinatorics lecture in Weber.  Finally, after class, they plan on meeting Student F at the Recreation Center. What bus routes should Students C and D take to avoid the hail?

\begin{figure}[h!]
\begin{center}
    \includegraphics[width=.8\textwidth]{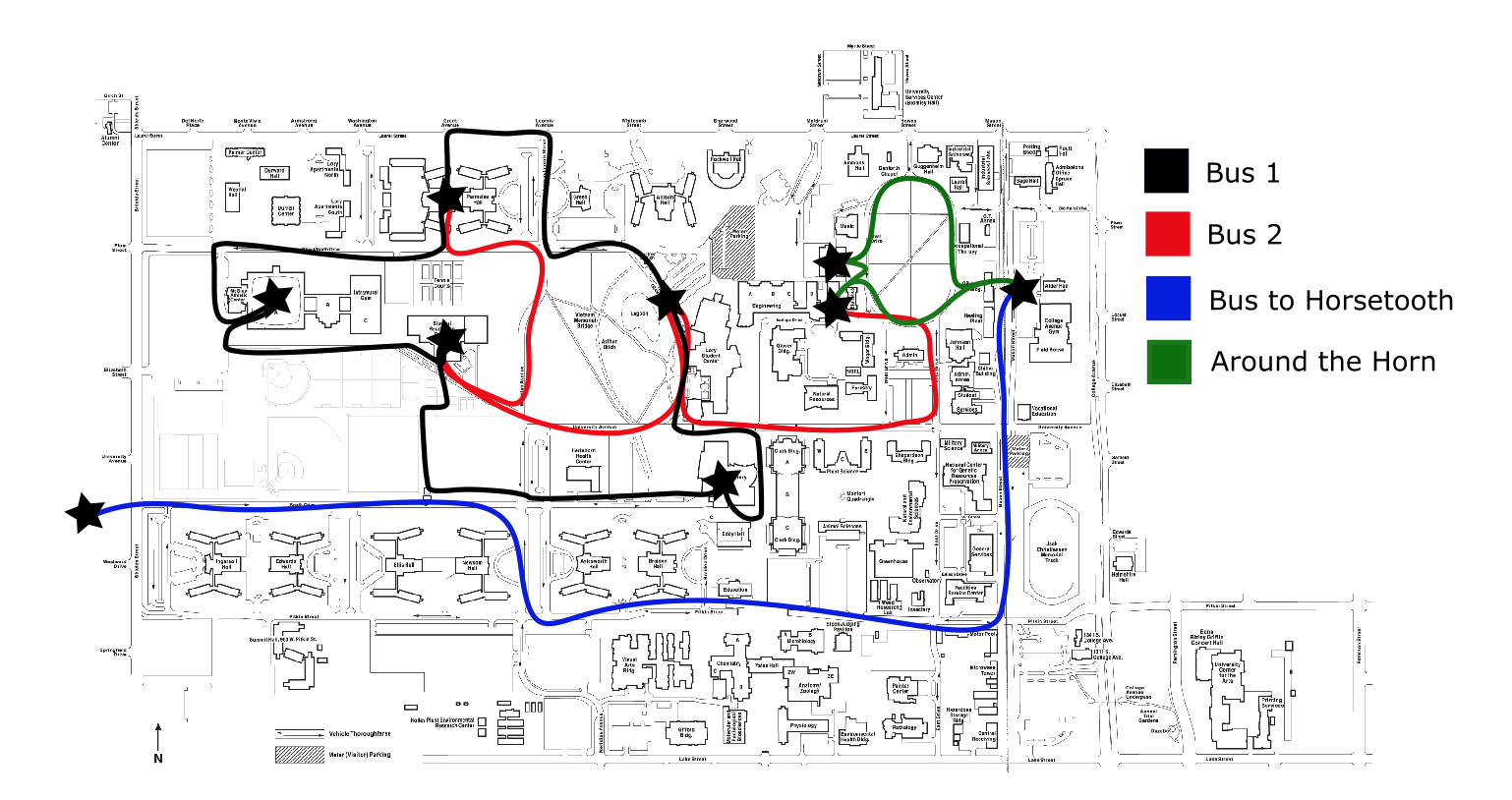}
    \caption{A map of CSU bus routes. CSU map from \url{http://www.mappery.com/Colorado-State-University-Map}.}
    \label{fig:bus-map}
\end{center}
\end{figure}

Looking at the map can be pretty confusing and they do not want to get lost.  We can instead simplify this map to something called a graph\footnote{The word graph in mathematics can also mean the graph of a function: for example, the graph of the function $y=x^2$ is a parabola.}.

\begin{figure}[h!]
\begin{center}
    \includegraphics[width=.7\textwidth]{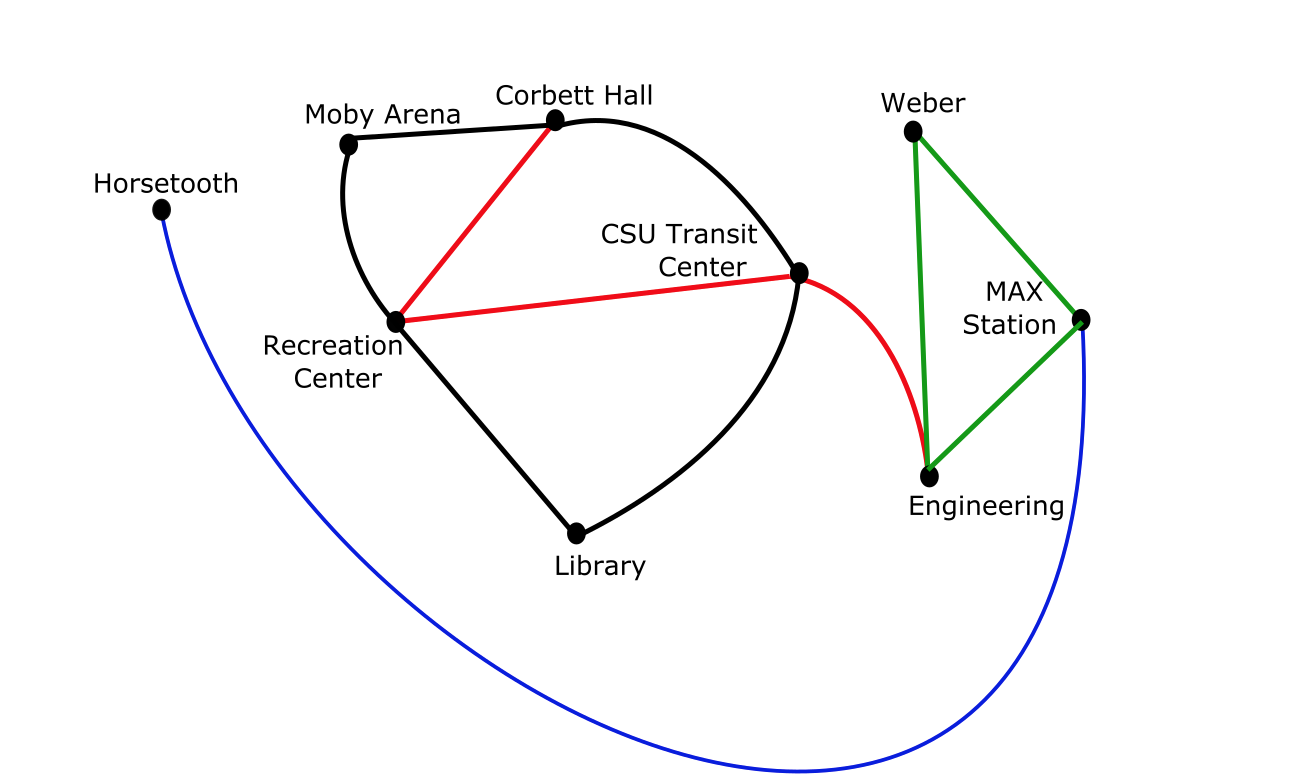}
    \caption{Bus routes near campus represented as a graph.}
    \label{fig:bus}
\end{center}
\end{figure}

In this class, a \defn{graph} is a collection of \defn{vertices} (dots) and \defn{edges} (line segments) between the vertices. The edges describe 
\defn{adjacency} (a connection or relationship) 
between the vertices.  In the graph associated with the bus routes, the locations of the bus stops are represented by vertices and the routes between stops are represented by edges.  An example of two adjacent bus stops are Corbett Hall and Moby Arena. 

As a second example of a graph, consider the social network in Figure~\ref{fig:Social}. Each vertex corresponds to a person\footnote{or at least a mammal},
and an edge between two people represents that they have met before.
The one intersection of edges that does not have a dot means nothing and can be ignored.
Suppose you are running for Student Body President at CSU and want to collect votes. Which person would you want to contact first?

\begin{figure}[h!]
\begin{center}
\includegraphics[width=.5\textwidth]{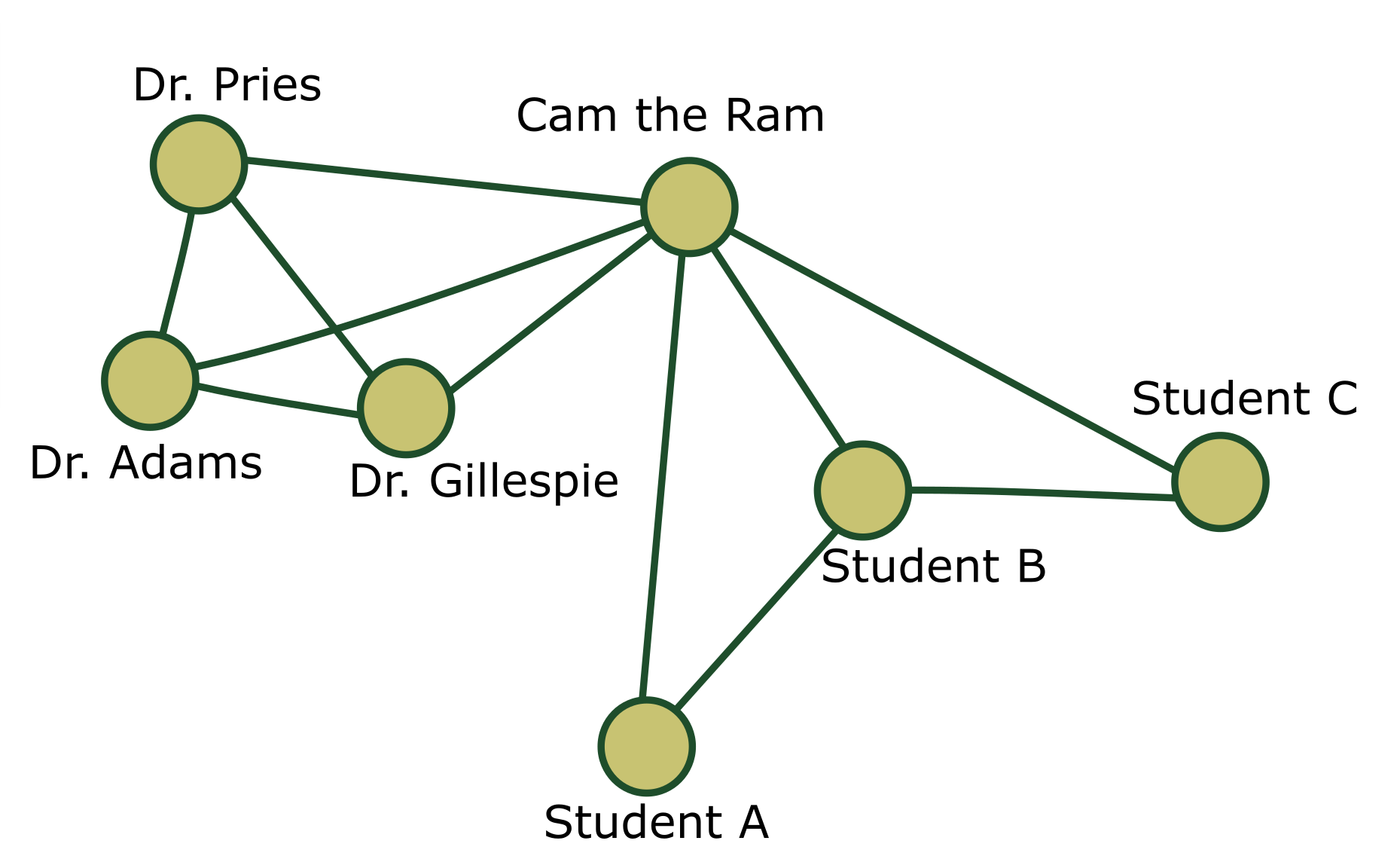}
\end{center}
\caption{An example of a CSU social network.}
\label{fig:Social}
\end{figure}

\subsection{Graphs, vertices, edges, and adjacency}
\begin{definition}
\label{def:graph}
A \defn{graph} $G=(V,E)$ consists of a finite set $V$ of \defn{vertices} (or nodes) and a finite set $E$ of \defn{edges}, where each edge $e\in E$ is of the form $e=\{u,v\}$ with $u,v\in V$.
\end{definition}

\begin{figure}[h!]
\begin{center}
\includegraphics[height=0.9in]{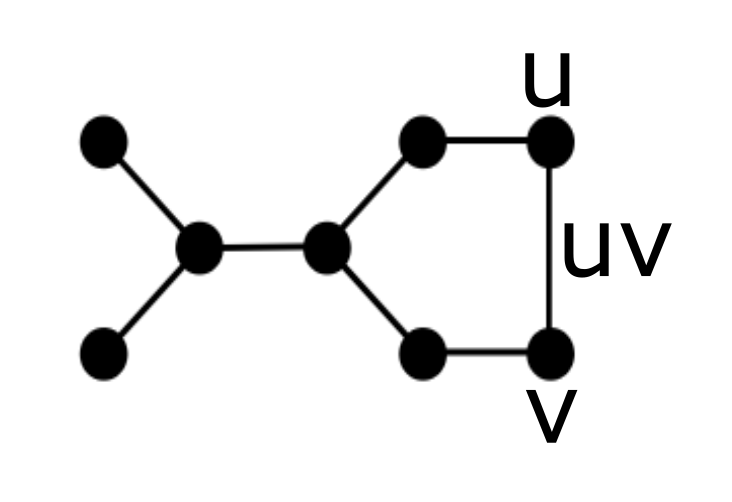} 
\end{center}
\caption{A graph with 8 vertices and 8 edges.
}
\end{figure}

The edge $\{u,v\}$ is often written as $uv$.  Unless mentioned otherwise, it is the same as the edge $\{v,u\} = vu$.

\begin{remark} \label{Ngraphrules}
In this book, unless we say otherwise:
\begin{enumerate}
\item we do not allow any edge to be a self-loop (an edge $uu$ that starts and ends at the same vertex $u$);
\item we do not allow more than one edge between any pair of vertices $\{u,v\}$.
\end{enumerate}

\begin{figure}
\begin{center}
\includegraphics[height=0.8in]{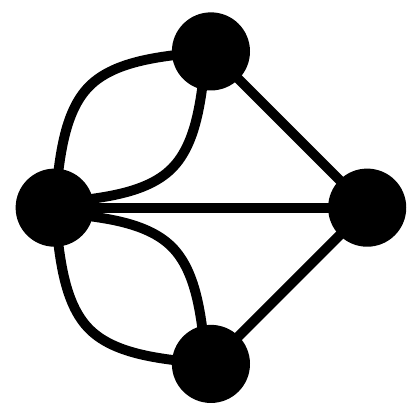}
\hspace{10mm}
\includegraphics[height=0.7in]{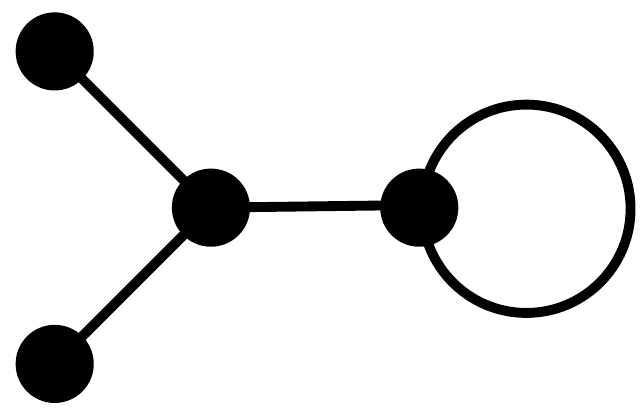}
\end{center}
\caption{We do not allow repeated edges (left) or self loops (right) in our graphs, unless stated otherwise.}
\label{Fnotationgraph}
\end{figure}

Indeed, in  Definition~\ref{def:graph}, the fact that each edge $e$ is a set automatically implies that self-loops are not allowed and the fact that $E$ is a set automatically implies that multiple edges are not allowed.
Later in the book, we will also consider a more general class of graphs, by allowing multiple edges and self-loops.
\end{remark}

\begin{definition}
Two vertices, $u$ and $v$, are \defn{adjacent} if there is an edge, $uv$, connecting them. If $u$ and $v$ are adjacent, we write $u\sim v.$
\end{definition}

A graph is defined only by its set of vertices and its set of edges; the same graph may be drawn in different ways.
Indeed, the two graphs below are the same graph; they are just drawn differently in the plane.
\begin{center}
\includegraphics[width=3in]{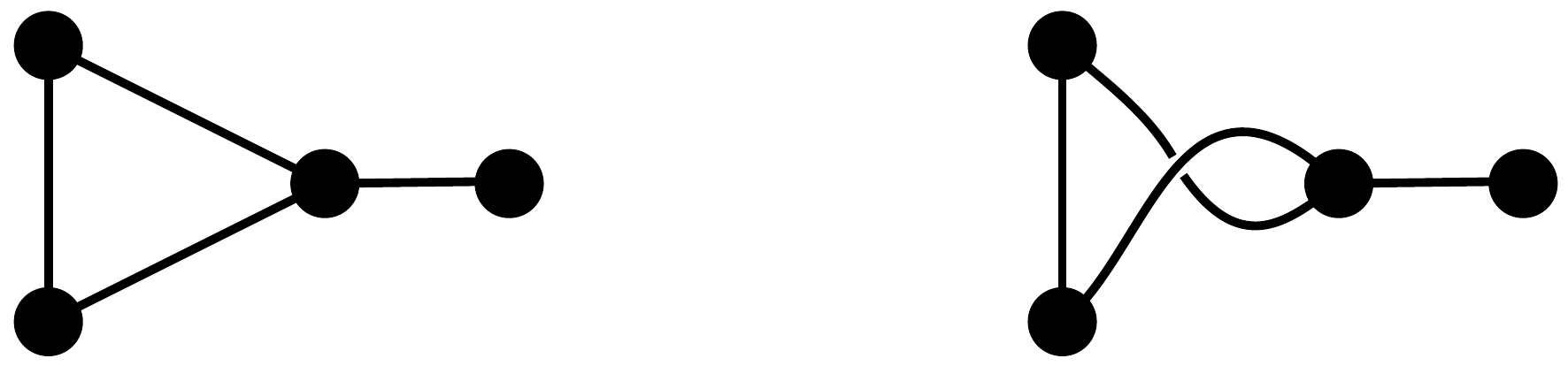}
\end{center}

A \defn{labeling} of a graph is an assignment of numbers to each vertex.
If a graph has $n$ vertices, then then we will always use the labels $0$ through $n-1$, where each such integer is used exactly once as a label.

\subsection*{Exercises}

\begin{enumerate}
 \item Let $V$ be a set of $4$ vertices.  How many different graphs $(V,E)$ are there if:
 \begin{enumerate}
     \item the vertices are labeled as $V=\{v_1,v_2,v_3,v_4 \}$?  For example, in this case, the graph with the one edge $v_1v_2$ is considered different from the graph with the one edge $v_2v_3$.
    \item the vertices are not labeled? For example, in this case, all graphs having $4$ vertices and one edge are considered the same. 
 \end{enumerate}

\item What is the largest number of edges that a graph with $5$ vertices can have?

\item Let $V$ be a set of $5$ vertices.  How many different graphs $(V,E)$ with exactly three edges are there if:
 \begin{enumerate}
     \item the vertices are labeled as $V=\{v_1,v_2,v_3,v_4, v_5 \}$?  
    \item the vertices are not labeled? 
 \end{enumerate}
 
\end{enumerate}

\section{Everyday graphs}

Some graphs are used frequently enough that they have names of their own. 
Here are some examples of these.
Let $n\geq 1$.

\subsection{Complete graphs}
\begin{definition}
The \defn{complete graph} $K_n$ is the graph with $n$ vertices having an edge between every pair of distinct vertices. 
\end{definition}

\begin{example} Below are the complete graphs, $K_1,K_2, K_3$ and $K_4$, respectively.
\begin{center}
\includegraphics[width=3in]{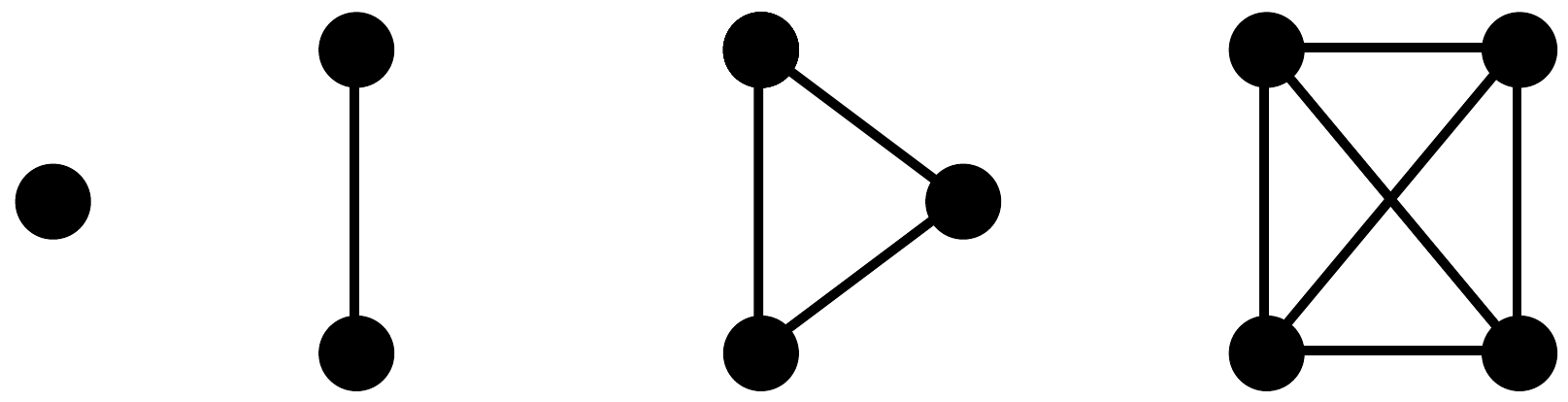}
\end{center}
\end{example}

\begin{example}
What is the largest number of edges that a graph with 9 vertices can have?
In other words, what is the maximal possible number of bus connections that can exist between the 9 stops in Figure~\ref{fig:bus}, assuming that any two bus stops have at most one bus route between them? 
\end{example}

\begin{brainstorm}
How many possible edges can we draw on 9 vertices?
\begin{center}
\includegraphics[width=1in]{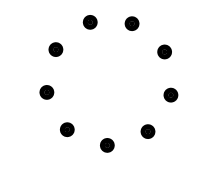}

\end{center}
\end{brainstorm}

\begin{answer}
There are at most $\binom{9}{2}$ possible edges.
The graph with all $\binom{9}{2}$ edges is the complete graph $K_9$.
\end{answer}

By a similar argument, we can show that $K_n$ has $\binom{n}{2}$ edges.

\subsection{Path and cycle graphs}

\begin{definition} The \defn{path graph} $P_n$ is the graph with $n$ vertices and $n-1$ edges that can be drawn so that all vertices and edges lie on a straight line. 
\end{definition}

\begin{example} For instance, the path graph $P_5$ has five vertices edges and four edges.
\begin{center}
\includegraphics[width=2in]{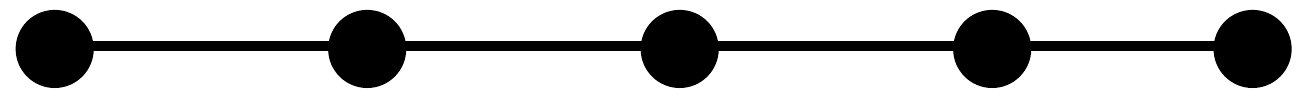}
\end{center}
\end{example}

The bus route in Figure~\ref{fig:bus} contains a path graph $P_6$ connecting from Moby Arena, to Corbett Hall, to the Library, to the CSU Transit Center, to Engineering, to the MAX Station.

\begin{definition} If $n \geq 3$, the \defn{cycle graph}  $C_n$ is the graph with $n$ edges and $n$ vertices such that every vertex is adjacent to exactly two others.
\end{definition}
\begin{example} Cycle graphs are usually arranged in the shape of a circle, as in the case of $C_6$ below. 
\begin{center}
\includegraphics[width=1in]{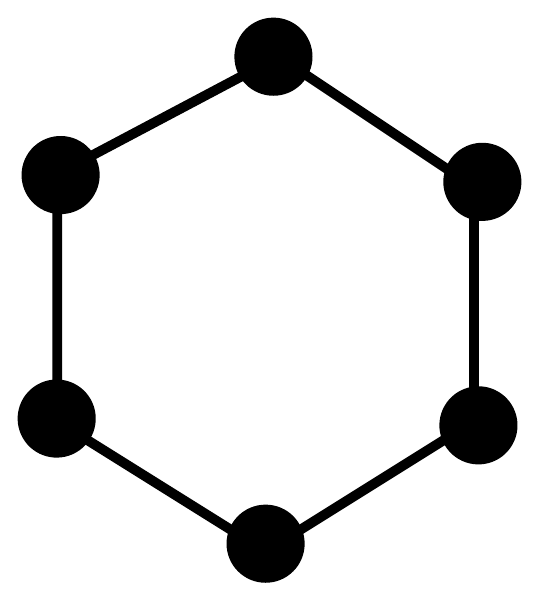}
\end{center}
\end{example}
The green loop, or the Around the Horn bus, in Figure~\ref{fig:bus}, connecting Weber to Engineering to the MAX Station back to Weber, forms a cycle graph $C_3$ inside the bus route map.

\subsection{Bipartite graphs}

The three most popular pizza parlors in Fort Collins are Pizza Casbah, Krazy Karl's, and Cosmos's Pizza.
Four student houses, A, B, C, and D often order pizza.
Their patterns of pizza ordering can be modeled as a graph.
The graph contains seven vertices --- three vertices for the three pizza parlors, and four vertices for the four student houses.
If a house ordered pizza from a particular parlor in the last month, then this support is modeled by an edge in the graph.
(Later, we might also choose to weight the edge by the number of pizzas ordered from that parlor over the last month.)

\begin{figure}[H]
\begin{center}
\includegraphics{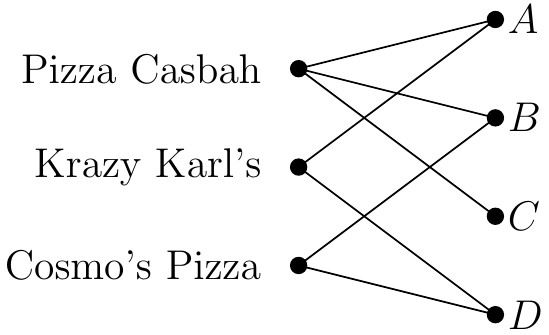}
\caption{Bipartite graph with three vertices on the left corresponding to pizza parlors, and four vertices on the right corresponding to houses.
An edge represents an order between that house and that parlor in the last month.}
\end{center}
\end{figure}

This graph has a particular structure: there are no edges between two pizza parlors, since a pizza parlor never orders pizza from another parlor;
similarly, there are no edges between two houses, since no house ever orders pizza from another house!

Many problems in graph theory involve matching two different types of objects, like students with jobs, or pets with owners.  

\begin{definition}
A graph $G=(V,E)$ is \defn{bipartite}
if $V$ is the disjoint union of 
two sets $L$ (left) and $R$ (right) and every edge connects a vertex in $L$ with a vertex in $R$.
\end{definition}

In other words, there are no edges between two vertices in $L$ and no edges between two vertices in $R$.

\begin{center}
    \includegraphics[width=1.2in]{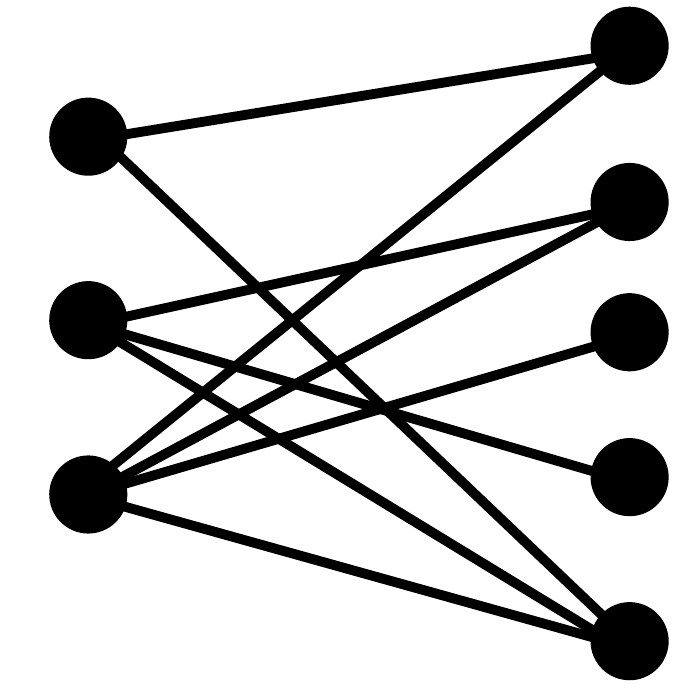}
\end{center}

\begin{definition}
    The \defn{complete bipartite graph} $B_{n,m}$ is the bipartite graph with $n$ vertices on the left and $m$ vertices on the right, such that every vertex on the left is adjacent to every vertex on the right.   
    \end{definition}
    
In the example with three pizza parlors and four student houses, the complete bipartite graph $B_{3,4}$ would model the situation in which all four houses have ordered pizza from all three parlors within the last month.

\begin{center}
\includegraphics{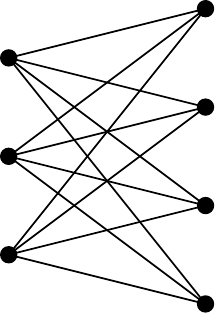}
\end{center}

 \subsection*{Exercises}

    \begin{enumerate}
    \item True or False: A graph with 30 edges has at least 9 vertices.
    
    \item What is the largest $n$ such that the graph in 
    Figure~\ref{fig:bus} contains a cycle graph $C_n$ (no edge can be used twice)?
    
    \item What is the largest $n$ such that the graph in 
    Figure~\ref{fig:bus} contains a path graph $P_n$ (no edge can be used twice)?
    \item Explain why the complete bipartite graph $B_{1,m}$ is called a \defn{fan}.
    \item How many vertices does $B_{n,m}$ have?
    \item 
    How many edges does $B_{n,m}$ have?
    \item Can you draw $B_{3,3}$ so that none of the edges cross?
\end{enumerate}

\section{The degree of a vertex} \label{Sdegree}

The students walk around campus, ending at the bus stop at the MAX Station, where they plan to catch a bus.
How many different bus stops can they visit next?
Looking at Figure~\ref{fig:bus}, they decide they can visit three stops next: Weber, Engineering, or Horsetooth.
In other words, the \emph{degree} of the CSU Transit Center vertex is three.
The students decide to take the bus to Horsetooth.
After dipping their toes in the reservoir, they realize that they have only one choice for their next stop, since their is only one route from Horsetooth.
In other words, the \emph{degree} of the Horsetooth vertex is one.

\begin{definition}
If $v$ is a vertex in a graph $G$, the \defn{degree} of $v$, denoted $\deg(v)$, is the number of edges adjacent to $v$.
\end{definition}

Given a set of $n$ non-negative numbers, we can ask whether this can be the set of degrees of the vertices in a graph with $n$ vertices.
In the rest of this section,
we give some examples and conditions about this question.

\begin{example}
Is there a graph on 5 vertices with vertices of degrees 1, 2, 2, 2, 3?
\end{example}

\begin{answer}
Yes!
\begin{center}
\includegraphics[width=0.9in]{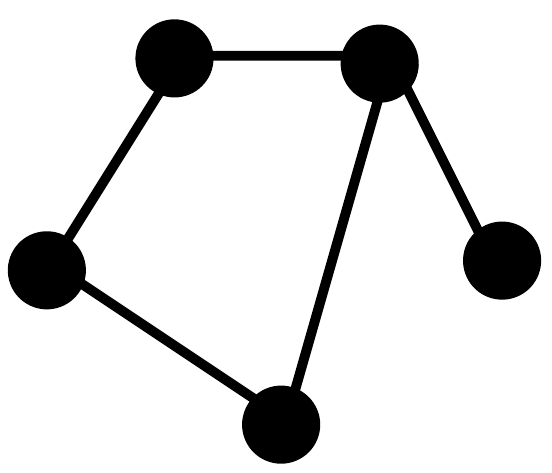}
\end{center}
\end{answer}

\begin{example}
Is there a graph on 9 vertices of degrees 1, 2, 2, 2, 3, 3, 3, 4, 4?
\end{example}

\begin{answer}
Yes, Figure~\ref{fig:bus} is one such example!
\end{answer}

\begin{example}
Is there a graph on 5 vertices with vertices of degrees 0, 1, 2, 3, 4?
\end{example}

\begin{answer}
No.
The vertex of degree 4 would be connected to every other vertex, and hence there cannot also be a vertex of degree 0.
\begin{center}
\includegraphics[width=0.9in]{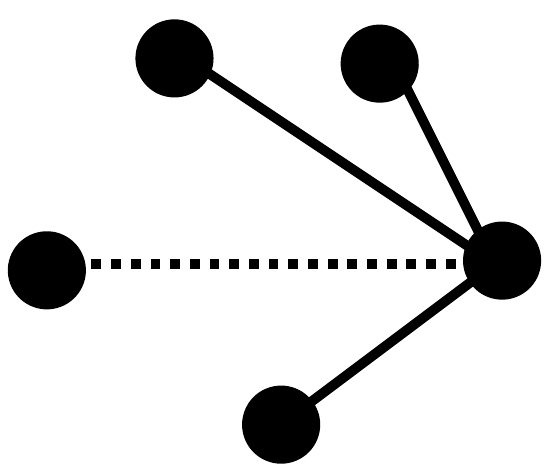}
\end{center}
\end{answer}

\begin{example}
Is there a graph on 5 vertices with vertices of degrees 2, 2, 3, 3, 3?
\end{example}

\begin{brainstorm}
Try to draw a picture of such a graph!
\end{brainstorm}

\begin{answer}
No, by the following theorem (since the sum of vertex degrees $2+2+3+3+3=13$ is not even):
\end{answer}

\begin{theorem}
\label{thm:graphDegreeTwiceEdges}
The sum of the degrees of all vertices in a graph is twice the number of edges.
In particular the sum of all vertex degrees is even, and hence the number of vertices of odd degree must be even.
\end{theorem}
\begin{proof}[Proof of Theorem~\ref{thm:graphDegreeTwiceEdges}]
In the sum of the degrees of all vertices, each edge $\{u,v\}$ is counted exactly twice: once in the degree of vertex $u$, and once in the degree of vertex $v$.
\end{proof}
\begin{example}
Consider the following graph, which has 6 edges, and the sum of the vertex degrees is $2\cdot 6=12$.
\begin{center}
\includegraphics[width=2in]{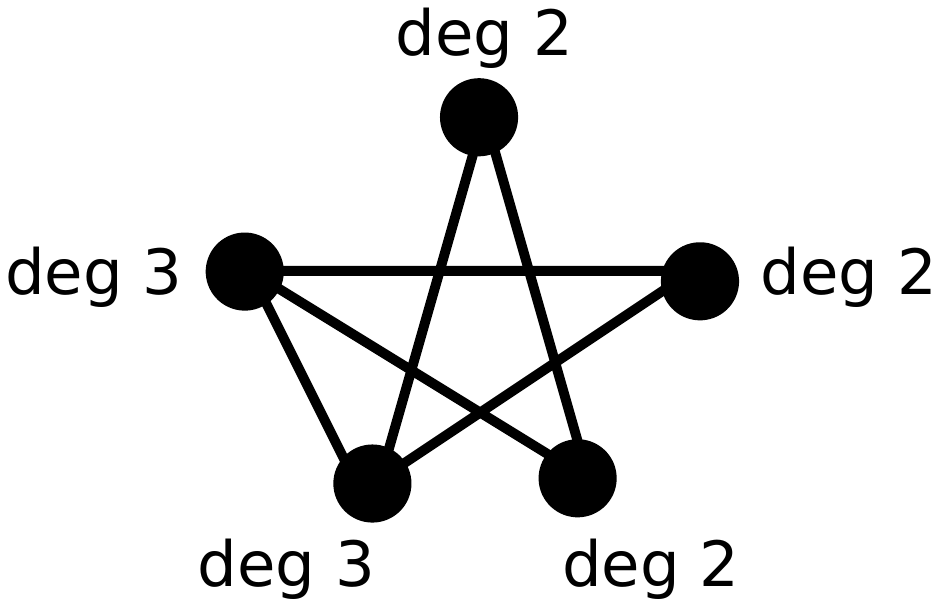}
\end{center}
\end{example}

Theorem~\ref{thm:graphDegreeTwiceEdges} implies that in any graph, the number of vertices of odd degree must be even.
We give a second proof of this consequence.
This proof is longer, but illustrates how induction can be used to prove results in graph theory.

\begin{proof}[Proof that in any graph $G$, the number of vertices of odd degree must be even.]
This proof will use induction on the number of edges in a graph.

The base case is a graph with no edges (only vertices), and so the number of vertices of odd degree is zero, which is an even number.

In the inductive step, we will add one edge.\footnote{There are a few ways that this is not as precise as it could be.
First, we need to prove that we can always build a graph by adding edges one at a time.
Also, in an inductive proof, the number of edges 
would become arbitrarily large, but that is not possible unless we allow multiple edges between vertices.}
\begin{center}
\includegraphics[width=5in]{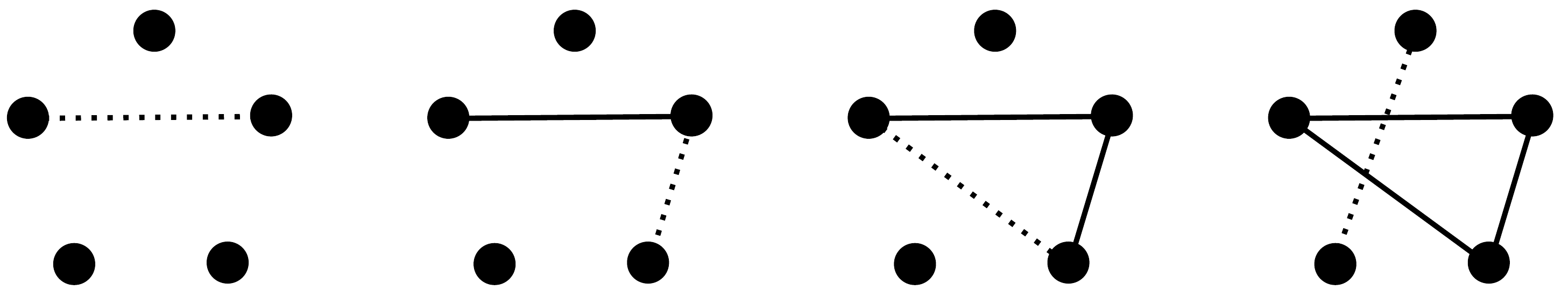}
\end{center}
Adding an edge changes the degree of two vertices, as described by one of the following three cases:
\begin{itemize}
\item \{even degree, even degree\} $\to$ \{odd degree, odd degree\}.
\item \{even degree, odd degree\} $\to$ \{odd degree, even degree\}.
\item \{odd degree, odd degree\} $\to$ \{even degree, even degree\}.
\end{itemize}
In each of these cases, the number of vertices with odd degree stays even.
Hence in any graph, the number of vertices of odd degree must be even.
\end{proof}

\subsection*{Exercises}

\begin{enumerate}

\item Is there a graph on 6 vertices with vertices of degrees \ldots
\begin{enumerate}
\item[(a)] 1, 1,  2, 2, 3, 3?
\item[(b)] 1, 1, 2, 2, 4, 6?
\item[(c)] 1, 1, 2, 3, 3, 5?
\item[(d)] 5, 5, 5, 5, 5, 5?
\item[(e)] 2, 2, 2, 2, 2, 4?
\end{enumerate}

\item Is there a graph with:
\begin{enumerate}
    \item  6 vertices with vertices of degrees
    0, 1, 1, 1, 2, 5?
    \item 7 vertices with vertices of degrees 0, 2, 2, 2, 4, 4, 6?
\end{enumerate}

\item True or False: There exists a graph with 7 vertices such the total sum of all vertex degrees is 44.

\item True or False: If a graph has an odd number of vertices, then it has a vertex with even degree.

\item Re-express each of the following assertions as a theorem about degrees in graphs.
The theorem should be of the following form: In any graph $G$ with at least 2 vertices, [some property about vertex degrees holds].
Then explain why the theorem is true.
\begin{enumerate}
\item  At every party there are at least two people who know the same number of other people at the party. 
\item If one person knows everyone at the party, then it is not possible for another person to know no one.
\end{enumerate}

\item In the complete bipartite graph $B_{n,m}$ with $n$ vertices on the left and $m$ vertices on the right, what is the degree of each vertex on the left? what is the degree of each vertex on the right?

\item Is there a graph on 6 vertices whose vertices have degrees 1,1,2,3,5,5?					
\item Is there a graph on 6 vertices whose vertices have degrees 0,1,2,2,2,5?					
\item Is there a graph on 6 vertices whose vertices have degrees 2,2,3,3,3,3?

\end{enumerate}

\section{Subgraphs}

A recurring theme in mathematics is that whenever you have an interesting object, it is often fruitful to also consider its subobjects.
In this section, we consider subgraphs of a graph.
Intuitively, a graph $H$ is a subgraph of $G$ if each of its vertices is a vertex of $G$, and if each of its edges is a edge of $G$.

For example, consider the bus route graph in Figure~\ref{fig:bus}.
Suppose that the bus route from Moby Arena to Corbett Hall closes due to construction.
Then we obtain the subgraph in Figure~\ref{fig:bus-route-subgraphs}(left), with a single edge removed.

\begin{figure}[H]
\begin{center}
    \includegraphics[width=.4\textwidth]{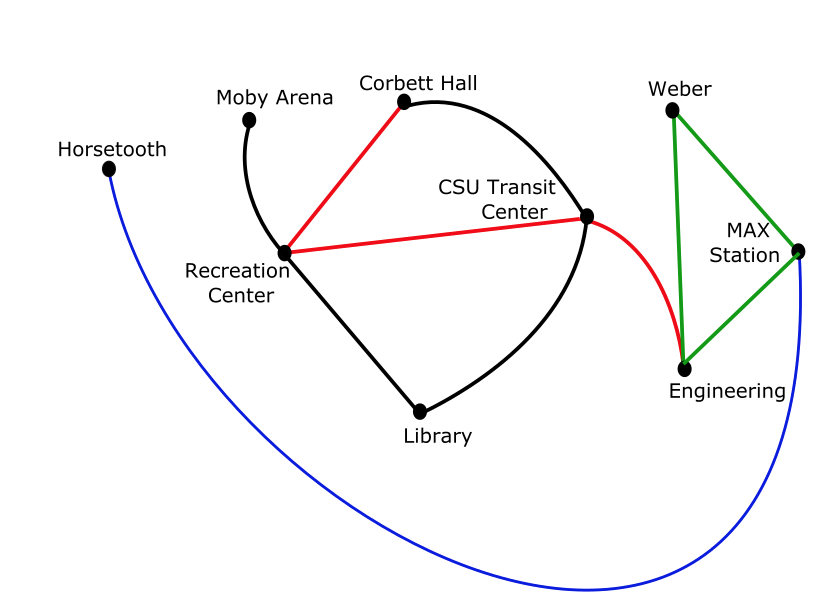}
    \hspace{20mm}
    \includegraphics[width=.4\textwidth]{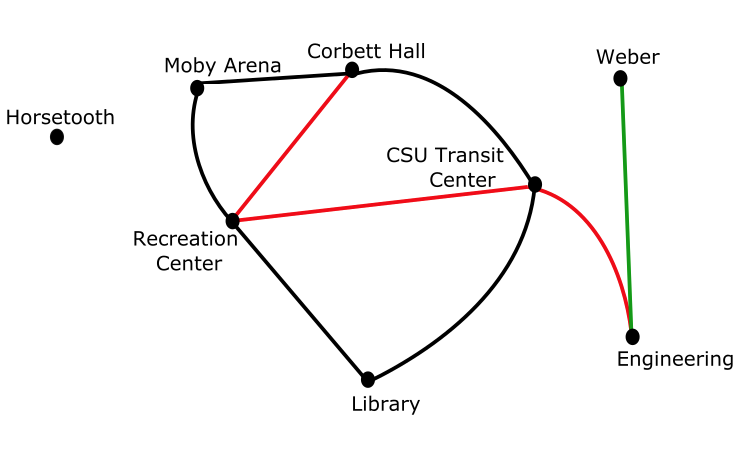}
\end{center}
\caption{Two subgraphs of the bus route graph, as various edges and vertices get removed.}
\label{fig:bus-route-subgraphs}
\end{figure}

Suppose furthermore that repairs to the MAX Station need to be made.
If the MAX Station closes, then necessarily the three routes adjacent to the MAX Station close as well, and we obtain the subgraph in Figure~\ref{fig:bus-route-subgraphs}(right), with one vertex and three more edges removed.

We now formalize what it means for one graph to be a subgraph of another.

\begin{definition}
Let $G=(V,E)$ be a graph.
A graph $G'=(V',E')$ is a \defn{subgraph} of $G$ if 
$V' \subseteq V$ and $E' \subseteq E$.
A \defn{spanning subgraph} of a graph $G$ is a subgraph that contains all of the vertices of $G$. 
\end{definition}

\begin{example}
Let $G$ be the graph drawn in (a).
Then (b) is a subgraph of $G$.
However, the object drawn in (c) is not a subgraph of $G$ because it is not even a graph (it is missing the vertex for one of its edges).
Also, the graph drawn in (d) is not a subgraph of $G$ because it contains an edge that is not in $G$.
\begin{center}
\includegraphics[width=.5\textwidth]{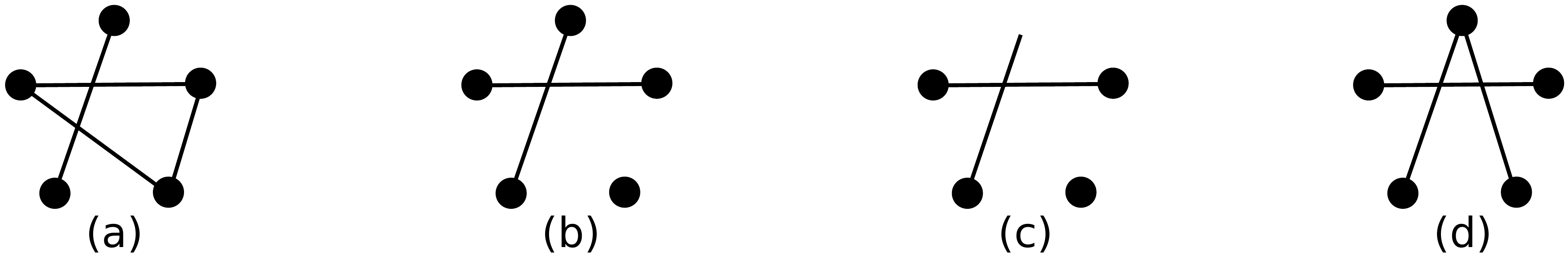}
\end{center}
\end{example}

In Figure~\ref{fig:subgraphs}, we draw a graph with three vertices and one edge, along with all of its subgraphs.

\begin{figure}
\begin{center}
    \includegraphics[width=4in]{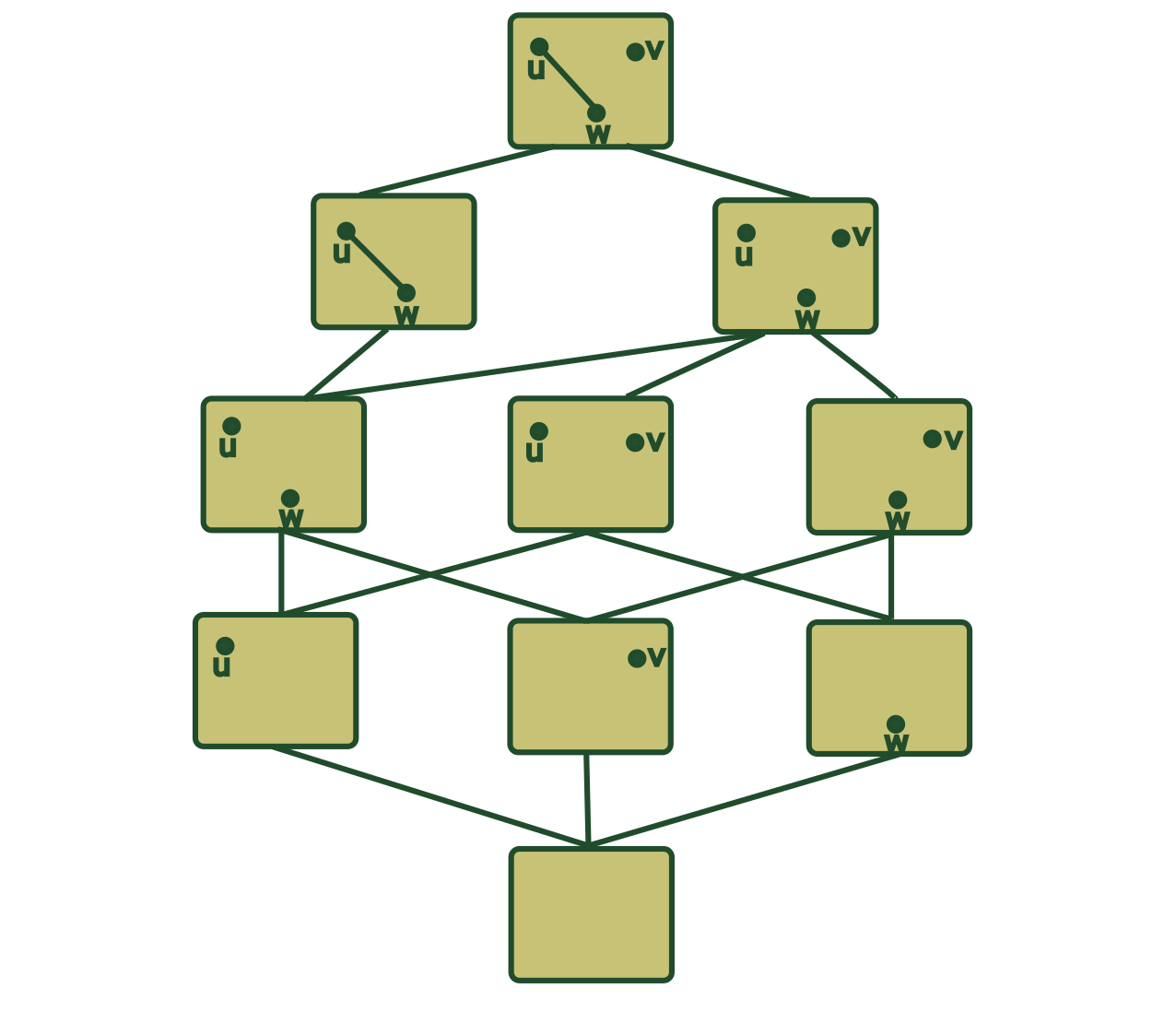}
\end{center}
\caption{
A graph $G$ (drawn in the top box) and all of its subgraphs.
Each connection represents a subgraph that is obtained from the one above it by removing a single vertex or edge.}
\label{fig:subgraphs}
\end{figure}

Note that a graph $G$ is always a subgraph of itself.
Similarly, the empty graph (which has no vertices or edges) is a subgraph of any graph $G$.

\subsection*{Exercises}

\begin{enumerate}

    \item How many subgraphs does a labeled cycle $C_3$ have?
    Assume the three vertices are labelled $a$, $b$, and $c$. We consider the subgraph with two vertices $a,b$ and a single edge $\{a,b\}$ to be different from the subgraph with two vertices $b,c$ and a single edge $\{b,c\}$.
    Hint: count the number of subgraphs with 0 vertices, then the number of subgraphs with 1 vertex, then the number of subgraphs with 2 vertices, etc.

    \item How many subgraphs does a labeled cycle of length 4 have? Assume the four vertices are labelled $a$, $b$, $c$, $d$. We consider the subgraph with two vertices $a,b$ and a single edge $\{a,b\}$ to be different from the subgraph with two vertices $b,c$ and a single edge $\{b,c\}$.

\item Explain why every graph with $n$ vertices is a subgraph of $K_n$.

\item Explain why $P_{n-1}$ is a subgraph of $C_n$

\item Explain why $C_6$ is a subgraph of $B_{3,3}$.

\item If $C_n$ is a subgraph of a bipartite graph, explain why $n$ is even.
 
\item True or False: A labeled graph with $n$ vertices always has at least $2^n$ subgraphs.

\end{enumerate}

\section{Walks and connected graphs}
\label{sec:graphs-connected-comp}

\begin{videobox}
\begin{minipage}{0.1\textwidth}
\href{https://youtu.be/jWZbEoWIQgM}{\includegraphics[width=1cm]{video-clipart-2.png}}
\end{minipage}
\begin{minipage}{0.8\textwidth}
Click on the icon at left or the URL below for a video covering sections 8.1--8.6. \\\vspace{-0.2cm} \\ \href{https://youtu.be/jWZbEoWIQgM}{https://youtu.be/jWZbEoWIQgM}
\end{minipage}
\end{videobox}

After class in Weber, some of the students need to go to their intramural water polo game.
However, Bus 2 in Figure~\ref{fig:bus-map} broke down while they were learning about derangements and now it is not available for service.
This means that the three red edges in Figure~\ref{fig:bus} disappear.
The students are dismayed to find that there is no longer a bus route to travel from Weber to Moby Arena.

We say that a graph is \emph{connected} if you can travel from any vertex of the graph to any other vertex by walking along edges.  For example, when all of the busses are running, the CSU bus system is connected.  However, if the edge between the CSU transit center and engineering disappears, the graph becomes disconnected.
We make this notion more precise by first introducing the concept of a walk.

\begin{definition}
A \defn{walk} in a graph is a sequence of vertices and edges $v_0$, $e_1$, $v_1$, $e_2$, $v_2$, \ldots, $v_{k-1}$, $e_k$, $v_k$ such that edge $e_i$ connects vertices $v_{i-1}$ and $v_i$, for $1 \leq i \leq k$. For example, the sequence of vertices and edges, $v_0,e_1,v_1,e_2,v_2,e_3,v_3$, is a walk on the following graph. 
\end{definition}

\begin{center}
    \includegraphics[width=.3\textwidth]{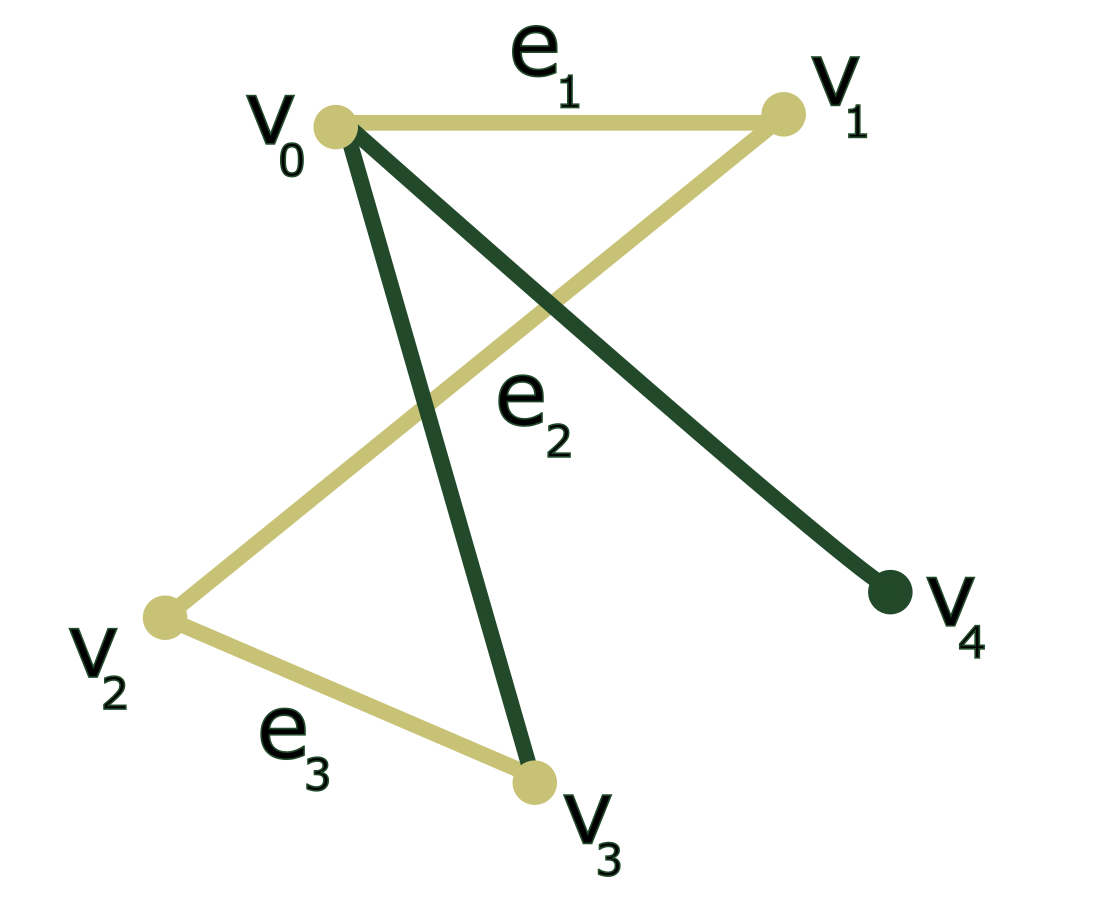}
\end{center}

\begin{definition}
A graph $G$ is \defn{connected} if for any two vertices $u$ and $v$ of $G$, there is a walk in $G$ from $u$ to $v$.
\end{definition}

For example, the graph drawn below is connected.
Indeed, no matter which two vertices you consider in this graph, you can always find a walk between them.
\begin{center}
\includegraphics[width=0.6in]{08-GraphsWalksCycles/graphFixedDegrees.pdf}
\end{center}

By contrast, the graph drawn below is not connected, because there is no walk between a vertex on the left side and a vertex on the right side.

\begin{center}
\includegraphics[height=0.4in]{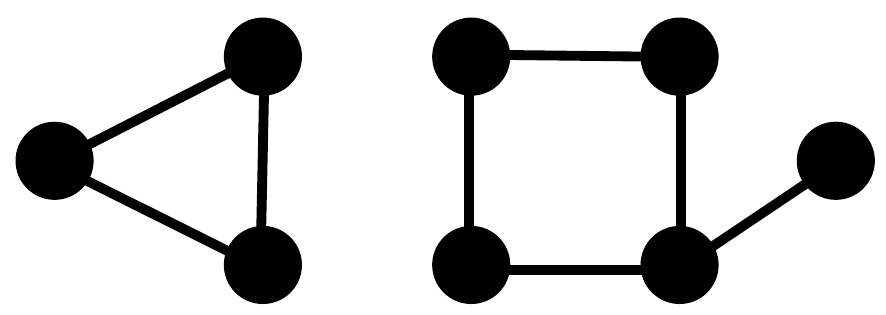}
\end{center}

It is often more convenient to work with connected graphs than with disconnected graphs.
If a graph is disconnected, it breaks up uniquely into its connected components, which are connected subgraphs that we define below.
It is then possible to work separately with each connected component of a graph.
 
\begin{definition}
A \defn{connected component} $H$ of a graph $G$ is a maximal subgraph that is connected.
\end{definition}

The word ``maximal" in the above definition means that there are no connected subgraphs of $G$ that contain $H$ and that are strictly larger than $H$.
In other words, if $H'$ is another connected subgraph of $G$, and if $H$ is a subgraph of $H'$, then $H=H'$.

\begin{example}
If $G$ is the graph
\includegraphics[height=0.2in]{08-GraphsWalksCycles/pathsCyclesComponents1.pdf}, then the connected components of $G$ are \includegraphics[height=0.2in]{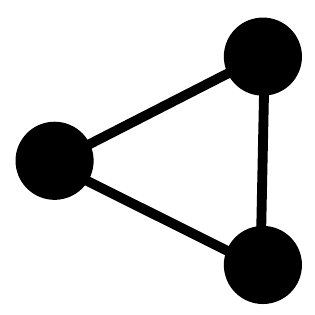}
and 
\includegraphics[height=0.2in]{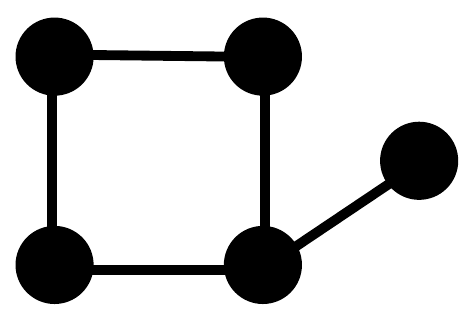}.
Note that
\includegraphics[height=0.2in]{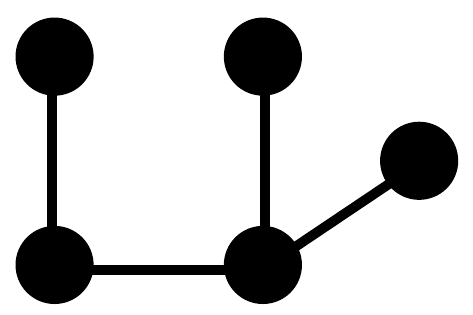}
is not a connected component of $G$ because it is not a \emph{maximal} connected subgraph.
\end{example}

Some of the exercises below illustrate that by deleting an edge from a connected graph, it is possible to make it disconnected.

\subsection*{Exercises}
\begin{enumerate}

\item What is the smallest number $k$ such that it is possible to make the following graph disconnected by removing $k$ edges:
\begin{enumerate}
    \item $K_7$?
    \item $C_7$?
    \item $P_7$?
    \item $B_{3,4}$?
\end{enumerate}

\item What is the smallest number $t$ such that it is possible to divide the following graph into $3$ connected components by removing $t$ edges:
\begin{enumerate}
    \item $K_7$?
    \item $C_7$?
    \item $P_7$?
    \item $B_{3,4}$?
\end{enumerate}

\item
By taking away $3$ edges from the graph below, what is the maximal number of connected components 
that you can make:
\begin{enumerate}
    \item $K_7$?
    \item $C_7$?
    \item $P_7$?
    \item $B_{3,4}$?
    \end{enumerate}

\item Let $G$ be a graph.
Let $H_1=(V_1,E_1)$ and $H_2=(V_2,E_2)$ be connected subgraphs with a vertex $v$ in common.
Form their union $H=(V,E)$, where $V=V_1\cup V_2$ and $E=E_1\cup E_2$.
Prove that $H$ is connected

\item Suppose that graph $G$ is connected and contains a cycle.
Prove that if an edge from this cycle is removed from graph $G$, then the remaining graph is connected.

\item Let $u$ and $v$ be two vertices in a graph $G$ that are not connected by an edge. Show that adding the new edge $uv$ creates a cycle if and only if $u$ and $v$ are in the same connected component of $G$.
\end{enumerate}

\section{Graph complements}

Student A decides to host a picnic, with the goal of introducing new people to each other.
There are 7 attendees for this picnic: Dr.\ Pries, Dr. Adams, Dr. Gillespie, Cam the Ram, Student A, Student B, and Student C.
These are the 7 people (or animals) represented by the vertices of the graph in Figure~\ref{fig:Social}, or in the image below on the left.
Recall that an edge in this graph represents the fact that the two people at the endpoints of this edge have met.

The goal of the picnic is to introduce new folks to each other, and so Student A decides to draw a new graph in which each edge represents that two people have not yet met, and therefore need to be introduced to each other.
This new graph is drawn on the right below; it is the \emph{complement} of the original social network graph drawn on the left.

\begin{center}
    \includegraphics[width=.9\textwidth]{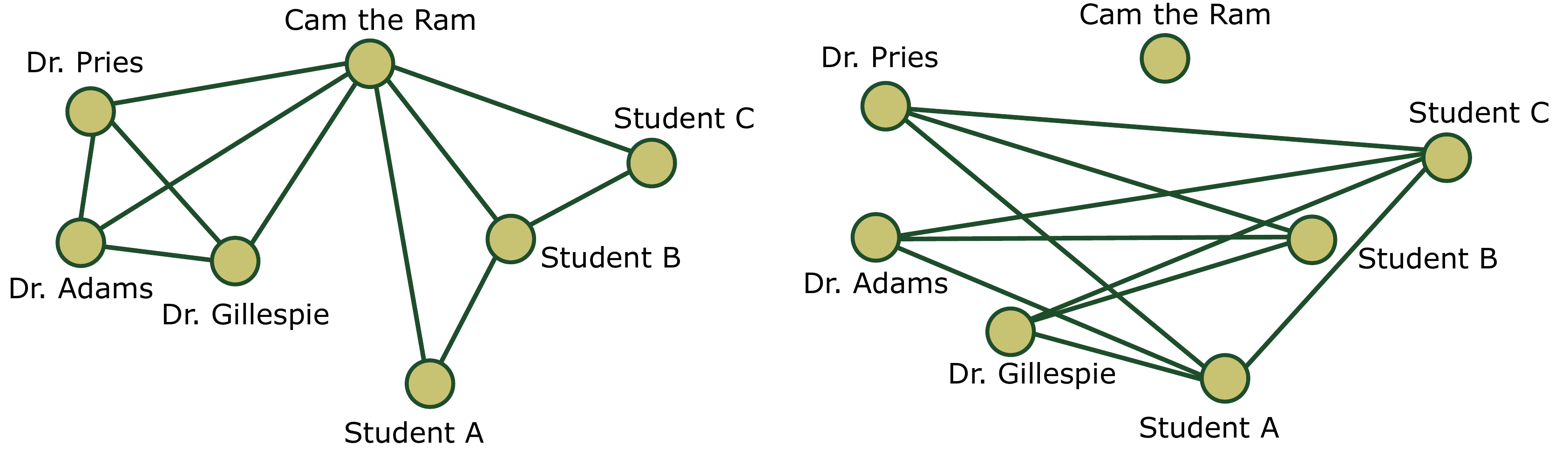}
\end{center}

\begin{definition}
Given a graph $G$, its complement $\overline{G}$ has the same vertex set as $G$, and has an edge between vertices $u$ and $v$ precisely when $G$ does not.
\label{def:graph-complement}
\end{definition}

For example, the figure below has the path graphs $P_2$ through $P_5$ in the left column, and their complement graphs $\overline{P_2}$ through $\overline{P_5}$ in the right column.
\begin{center}
\includegraphics[width=5in]{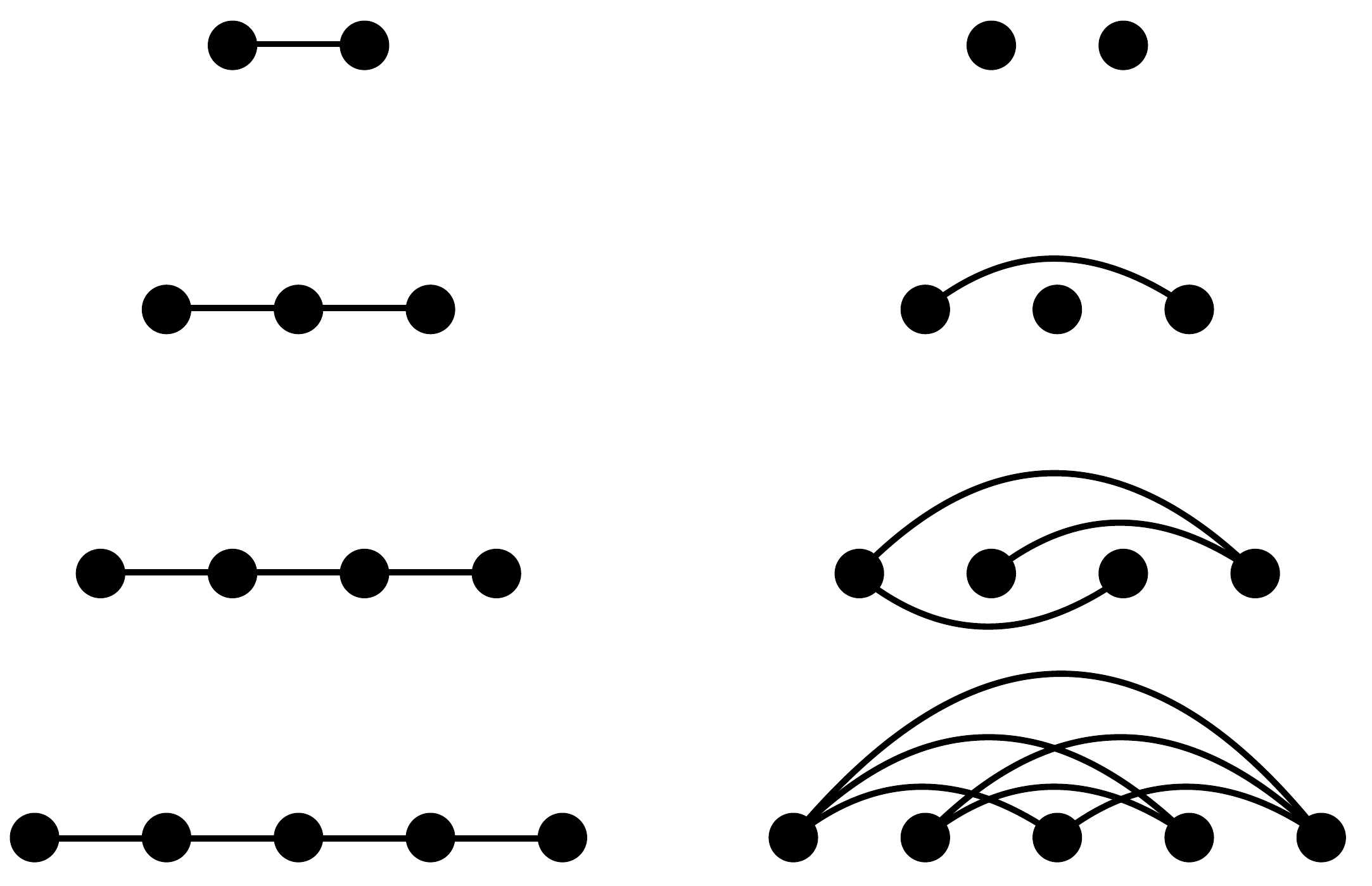}
\end{center}

An equivalent definition of a graph complement is as follows.
Let $G=(V,E)$ be a graph with vertex set $V$ and edge set $E$.
Let $S$ denote the set of all pairs of vertices in $V$.
Note that $E$ is a subset of $S$.
The graph complement $\overline{G}$ can also be defined as $\overline{G}=(V,S-E)$, where $S-E$ is the set complement as defined in Definition~\ref{def:complement}.
Explain why this agrees with the definition of graph complement given in Definition~\ref{def:graph-complement}.

\subsection*{Exercises}

\begin{enumerate}
\item Describe the complement of the following graphs:
\begin{enumerate}
    \item $C_5$;
    \item $B_{3,4}$;
\end{enumerate}
\item Is the complement of $C_6$ a connected graph?		
\item If a vertex $v$ has degree $d$ in a graph $G$ on $n$ vertices, what is the degree of $v$ in the complement of $G$?

    \item 
Prove that a graph complement $\overline{G}$ is a subgraph of the original graph $G$ only if $G$ is the complete graph on its vertex set.

\item
For which $n$, is the complement of $P_n$ a path?

\item
For which $n$, is the complement of $C_n$ a cycle?

\item True or False: Let $G$ and $H$ be graphs. If $G$ is the complement of $H$, then $H$ is the complement of $G$.

\item Prove that the complement of $B_{n,m}$ has two connected components, which are $K_n$ and $K_m$.
\end{enumerate}

\section{Storage structures for graphs}
\label{sec:graphStorage}

Given a graph, how do we store it in a computer?
Once a graph has been stored, how do we edit it?
We give two possible graph representations to address these questions.

The first way to store a graph is as an \defn{adjacency matrix}.
Given a graph $G$ with $n$ vertices, label its vertices from $0$ to $n-1$.
The adjacency matrix encoding this graph will be an $n\times n$ matrix whose entries are either 0 or 1.
Entry $(i,j)$ of the adjacency matrix is 1 when the edge $\{i,j\}$ between vertices $i$ and $j$ is in $G$, and 0 otherwise.
Below is the graph $C_6$ and its adjacency matrix.

\begin{center}
\raisebox{-1.3cm}{\includegraphics{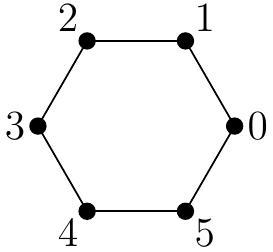}}\hspace{2cm}\begin{tabular}{c|ccccccc}  & 0 & 1 & 2 & 3 & 4 & 5 \\
\hline
0 & 0 & 1 & 0 & 0 & 0 & 1 \\
1 & 1 & 0 & 1 & 0 & 0 & 0 \\
2 & 0 & 1 & 0 & 1 & 0 & 0 \\
3 & 0 & 0 & 1 & 0 & 1 & 0 \\
4 & 0 & 0 & 0 & 1 & 0 & 1 \\
5 & 1 & 0 & 0 & 0 & 1 & 0
\end{tabular}
\end{center}

There are $n^2$ entries in this matrix that are each 0 or 1.
The matrix is symmetric: entry $(i,j)$ equals to entry $(j,i)$, because the edge $\{i,j\}$ is the same as the edge $\{j,i\}$.
Furthermore, each of the $n$ entries along the diagonal is $0$, since self-loops are not allowed in graphs.
Therefore, we only need to store the $\frac{n^2-n}{2}$ entries above the diagonal in order to fully represent the adjacency matrix.

The adjacency matrix depends on the labeling we chose for the vertices.  Changing the labeling changes the adjacency matrix by permuting the rows and the columns.

A second way to store a graph is a list of edges.
We store $n$, the number of vertices, and then maintain a list of size $2\times e$, where $e$ is the number of edges in our graph.
Each column of this edge list stores the two vertices incident to that edge.
Below is the edge list for the graph $C_6$:
\[\begin{matrix}
0 & 1 & 2 & 3 & 4 & 5 \\
1 & 2 & 3 & 4 & 5 & 0
\end{matrix}\]
In this representation, we need to store $2e$ integers between $0$ and $n-1$.

An adjacency matrix has the advantage of being easily editable.
Suppose we want to add (resp.\ remove) the edge $\{i,j\}$ from a graph.  Then to update the storage structure, we only need to change the entries $(i,j)$ and $(j,i)$ in the matrix from $0$ to $1$ 
(resp.\ from $1$ to $0$).

An edge list has the advantage of being a smaller storage method, especially if the graph is \emph{sparse}, meaning that the number of edges is small compared with the number of vertices.
However, an edge list has the disadvantage of being harder to edit.
For example, to remove an edge, we need to search through the entire list to find the location of the edge and then shuffle all of the later entries to the left.

For a graph that is a tree (connected graph with no cycles) and has a root, there is another
disadvantage of an edge list.
For this kind of graph, it is useful to find all the `descendents' of a vertex $v$, meaning all the vertices that are further from the root than $v$ is.
When the graph is stored as an edge list, it is very time-consuming to find all the descendants of a vertex, because you need to make repeated queries. For this reason, in this context, an edge list is usually replaced by a materialized path.  In the storage structure for a materialized path, the idea is to store the entire walk from the root to each vertex.
\subsection*{Exercises}
\begin{enumerate}
    \item Write down the adjacency matrix and the edge list for the graphs below (for some labeling of the vertices):
    \begin{enumerate}
        \item $K_5$;
        \item $P_5$;
        \item $C_5$;
        \item $B_{2,3}$.
    \end{enumerate}
 
 \item 
 Describe the structure of the adjacency matrix and the edge list for the graphs below, for a good choice of labeling of the vertices:
 \begin{enumerate}
     \item $K_n$;
     \item $P_n$;
     \item $C_n$;
     \item $B_{n,m}$.
 \end{enumerate}

\item How many non-zero entries does the adjacency matrix for C4 have?					
\end{enumerate}

\section{Eulerian walks and Hamiltonian cycles}

\begin{videobox}
\begin{minipage}{0.1\textwidth}
\href{https://www.youtube.com/watch?v=oP2paQL8zmM}{\includegraphics[width=1cm]{video-clipart-2.png}}
\end{minipage}
\begin{minipage}{0.8\textwidth}
Click on the icon at left or the URL below for this section's short lecture video. \\\vspace{-0.2cm} \\ \href{https://www.youtube.com/watch?v=oP2paQL8zmM}{https://www.youtube.com/watch?v=oP2paQL8zmM}
\end{minipage}
\end{videobox}

We are now prepared to return to one of our motivating problems from Chapter~1.
Consider the following map of seven bridges crossing the Pregel River in K\"{o}nigsberg, Prussia, which is today called Kaliningrad, Russia.
Is it possible to take a walk in which you cross each bridge \emph{exactly} once?
You do not need to return to your starting location.

\begin{center}
\includegraphics[width=3in]{01-Introduction/PregelRiver3.png}
\end{center}

Euler had the insight to transform this question about bridge crossings into a question about graph theory.  What was truly innovative about this was that no one had ever thought about graph theory abstractly before then.
After this rephrasing, the K\"{o}nigsberg bridge problem can be transformed into the following question: is there an ``Eulerian walk" (which we will later define) in the following graph?

\begin{center}
\includegraphics[width=1in]{01-Introduction/PregelRiverGraph.png}
\end{center}

Note that here we are allowing multiple edges between two vertices in a graph.  This happens because there are multiple bridges between some of the land masses in the picture.

Recall from Section~\ref{sec:graphs-connected-comp} that a \emph{walk} in a graph is a sequence of vertices and edges $v_0$, $e_1$, $v_1$, $e_2$, $v_2$, \ldots, $v_{k-1}$, $e_k$, $v_k$ such that edge $e_i$ connects vertices $v_{i-1}$ and $v_i$ for $1 \leq i \leq k$.

\begin{center}
\includegraphics[width=2in]{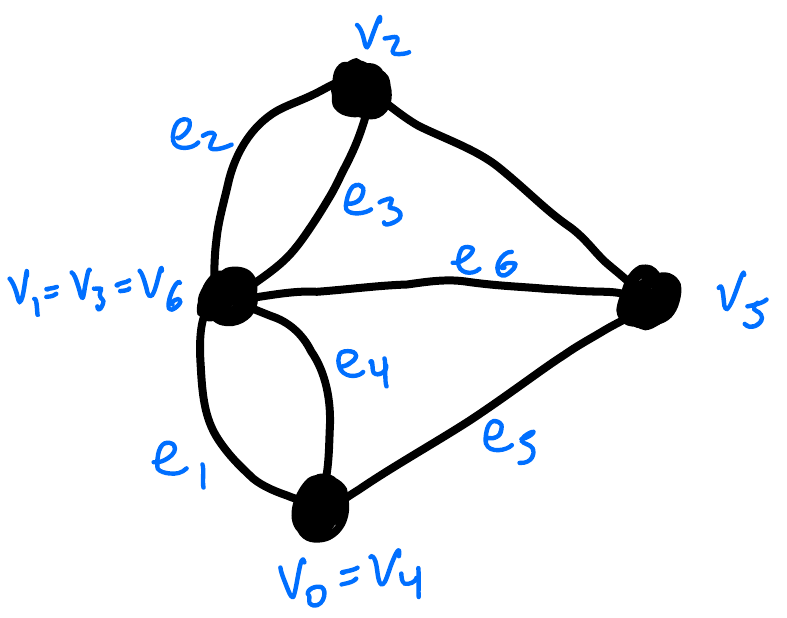}

\end{center}

We say that a walk is \emph{closed} if it begins and ends at the same location, i.e.\ if $v_0=v_k$.

\begin{definition}
An \defn{Eulerian walk} through a connected graph $G$ goes through every edge exactly once.
\end{definition}

\begin{figure}
\begin{center}
\subfigure[No Eulerian walks.]{\includegraphics[width=1.4in]{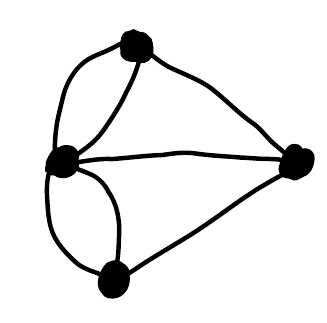}}
\hspace{1mm}
\subfigure[Eulerian walks exist! None of them are closed.]{\includegraphics[width=1.4in]{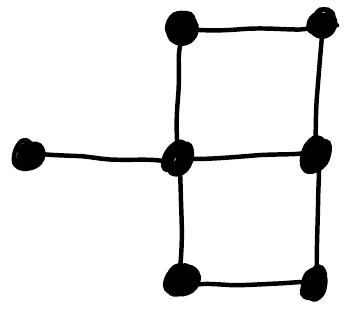}}
\hspace{1mm}
\subfigure[Eulerian walks exist! All of them are closed.]{\includegraphics[width=1.4in]{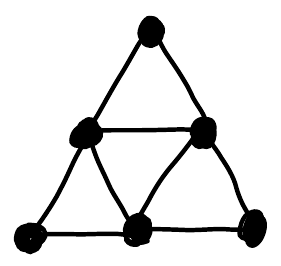}}
\hspace{1mm}
\subfigure[No Eulerian walks.]{\includegraphics[width=1.4in]{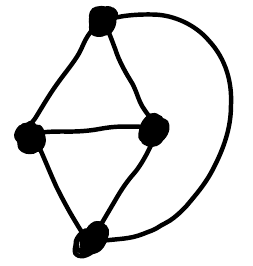}}
\caption{
Examples of graphs with or without (closed) Eulerian walks.
}
\end{center}
\end{figure}

The following theorem is considered to be one of the earliest theorems in graph theory!

\begin{theorem}[Euler, 1736]
\label{thm:euler}
Let $G$ be a connected graph.
\begin{enumerate}
\item[(a)] If $G$ has more than two vertices with odd degree, then it has no Eulerian walks.
\item[(b)] If $G$ has exactly two vertices of odd degree, then it has an Eulerian walk. Every Eulerian walk starts and ends at these two vertices.
\item[(c)] If $G$ has no vertices of odd degree, then it has an Eulerian walk. Every Eulerian walk is necessarily closed.
\end{enumerate}
\end{theorem}

Consider again the graph that results from the K\"{o}nigsberg bridge problem.
This graph has four vertices of odd degree.
Since 4 is more than 2, by part (a) of Theorem~\ref{thm:euler} this graph has no Eulerian walks.
It is therefore not possible to take a walk in K\"{o}nigsberg that crosses each of the seven bridges exactly once!

\begin{question}
Why can't we have one vertex of odd degree in a graph?
\end{question}

\begin{answer}
The sum of all vertex degrees is twice the number of edges (Section~\ref{Sdegree}), and hence an even number. It follows that in any graph, the number of vertices with odd degree is even.
\end{answer}

In the next several claims, we prove parts of Theorem~\ref{thm:euler}.

\begin{claim}
If a vertex $v$ has odd degree, then any Eulerian walk must start or end at $v$.
This proves (a) and the second sentence of (b).
\end{claim}

\begin{proof} $ $
\begin{center}
\includegraphics[width=2.5in]{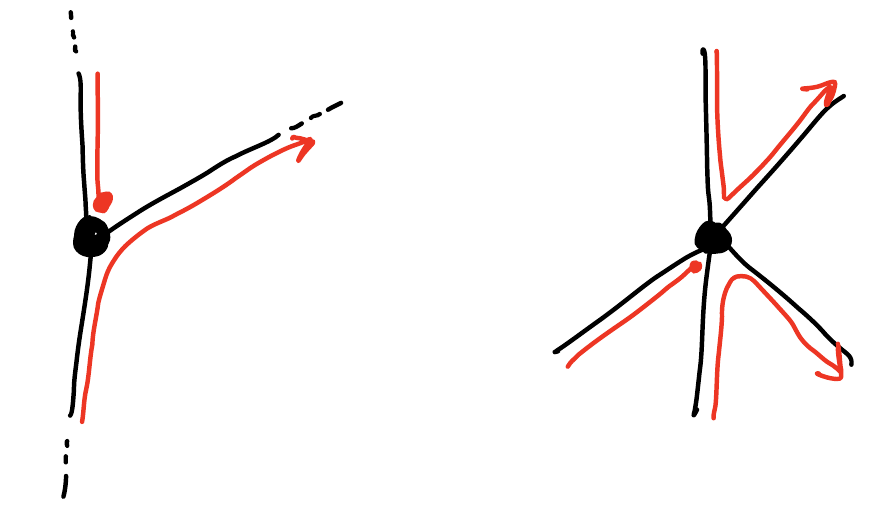}

\end{center}
If the walk doesn't start at $v$, then it:
\begin{itemize}
\item enters and leaves $v$ (using up two edges),
\item enters and leaves $v$ (using up two edges),
\item \ldots
\end{itemize}
This continues until there is one remaining edge incident to $v$ (since $v$ has odd degree).
The walk then enters along this edge and cannot leave.
Hence if the walk doesn't start at $v$, then it must end at $v$.
\end{proof}

\begin{claim}
If a vertex $v$ has even degree, then any Eulerian walk either starts and ends at $v$, or starts and ends somewhere else. This proves the second sentence of (c).
\end{claim}

\begin{proof}[Ideas towards the proof] $ $
Two example walks that start and end at $v$ are drawn as follows:
\begin{center}
\includegraphics[width=3in]{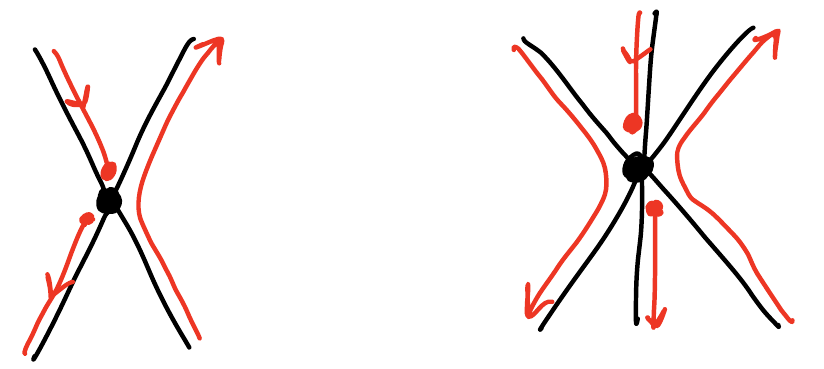}

\end{center}

Two example walks that neither start nor end at $v$ are drawn as follows:
\begin{center}
\includegraphics[width=3in]{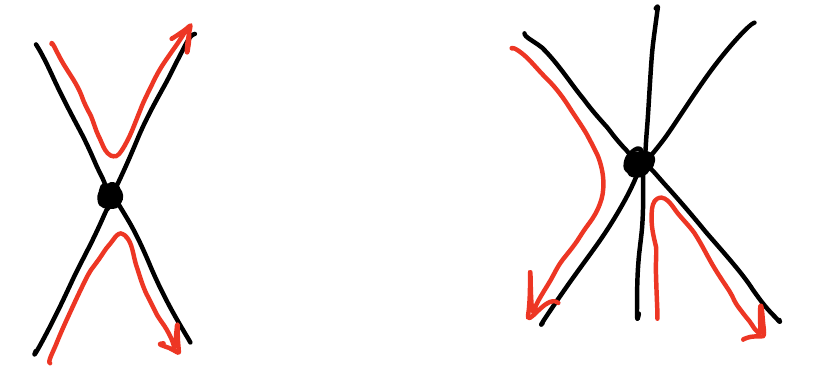}

\end{center}
The complete proof is much like that of the previous claim.
\end{proof}

We omit the proofs that Eulerian walks exist in cases (b) and (c), although they are very interesting.  The main idea for (c) is that if a walk arrives at a vertex $v$ by one edge, then it is possible to continue the walk leaving by another edge, thus reducing the number of unused edges at $v$ by $2$. There is a clever proof of (b) using (c); the main idea of this is to add an edge between the two vertices $u$ and $v$ with odd degree in the graph, then find an Eulerian cycle using (c), and then to adjust the starting point of the cycle to be $u$; deleting the edge between $u$ and $v$ produces an Eulerian walk that starts and ends at $u$ and $v$. 

\begin{example}
\label{Emoreeuler}
Consider the four connected graphs drawn below.
If an Eulerian walk exists, then find one, and say whether all such Eulerian walks are closed or not!
If no Eulerian walk exists, then say why.
\begin{figure}[h]
\begin{center}
\subfigure[]{\includegraphics[width=1in]{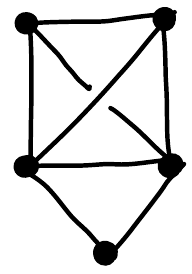}}
\hspace{1mm}
\subfigure[]{\includegraphics[width=1.2in]{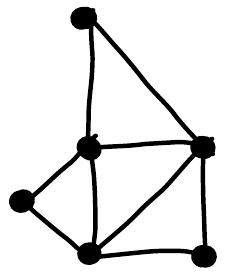}}
\hspace{1mm}
\subfigure[]{\includegraphics[width=1.2in]{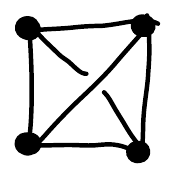}}
\hspace{1mm}
\subfigure[]{\includegraphics[width=1.2in]{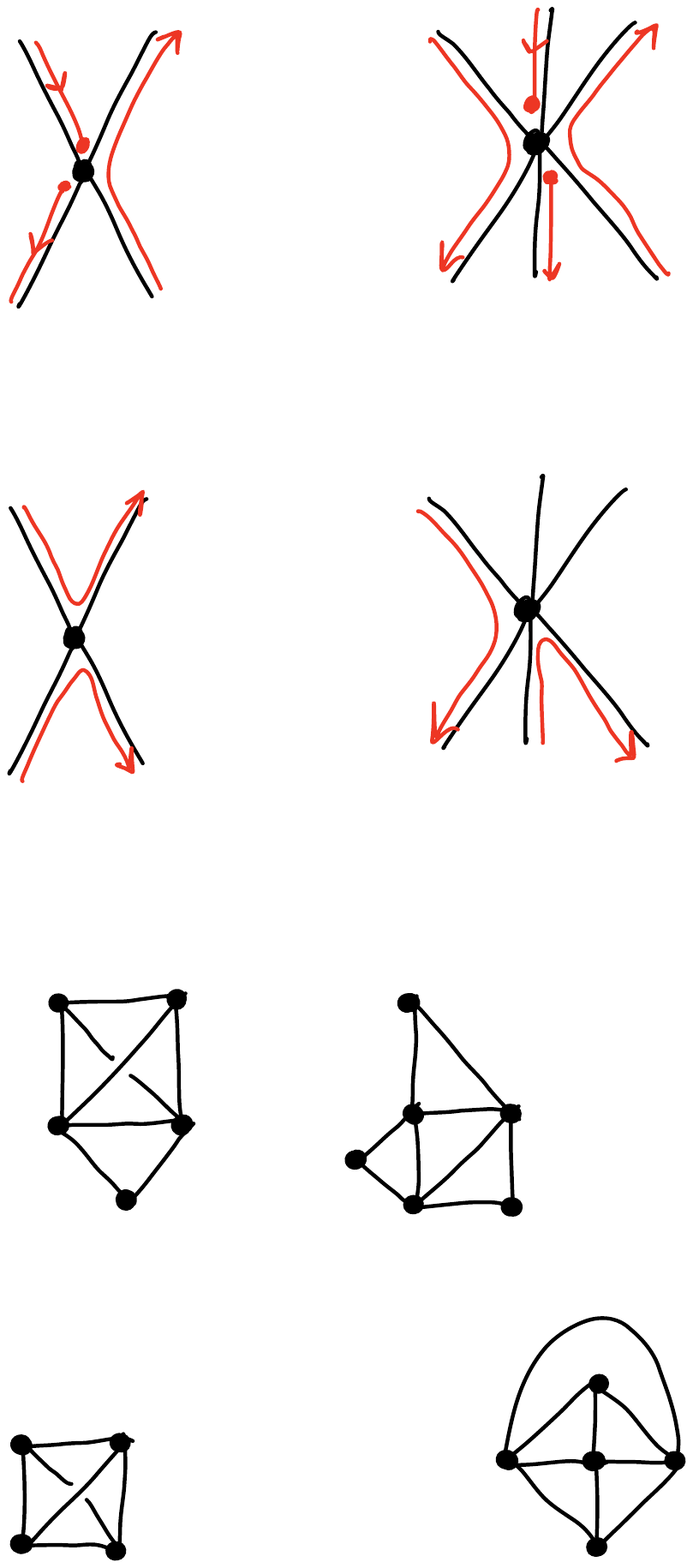}}

\end{center}
\end{figure}
\end{example}

Euler was able to solve the K\"{o}nigsberg bridge problem by transforming it into a question about graph theory.
Though there are efficient techniques for solving many questions in graph theory, it is important to remark that not all graph theory questions have easy answers! 

One such question is as follows: does there exist a walk in a graph that passes through each vertex exactly once? Such a walk is called a \emph{Hamiltonian cycle}.

\begin{definition}
A \defn{Hamiltonian cycle} in a connected graph $G$ is a closed walk that passes through each vertex exactly once.
\end{definition}

Here are two examples of Hamiltonian cycles in graphs, where the Hamiltonian cycle is drawn in red.

\begin{center}
\includegraphics[width=1in]{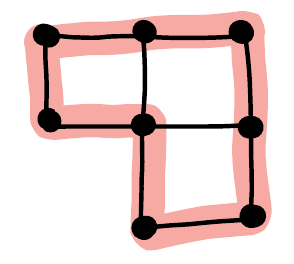}
\hspace{20mm}
\includegraphics[width=2in]{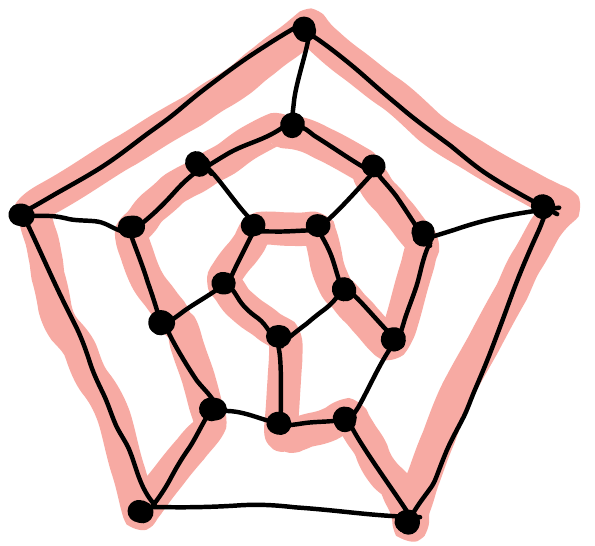}

\end{center}

By contrast, the Petersen graph drawn below has no Hamiltonian cycle.
\begin{center}
\includegraphics[width=1in]{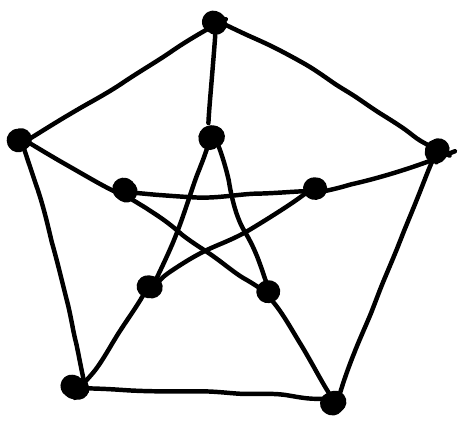}
\end{center}

In general, is very hard to determine whether a graph has a Hamiltonian cycle or not!
Indeed, determining whether a graph has a Hamiltonian cycle or not is \emph{NP-complete}, which is a very difficult class of problems in theoretical computer science.

\subsection*{Exercises}
\begin{enumerate}

\item Does the complete graph $K_5$ have an Eulerian walk?								
\item Does the complete graph $K_5$ have an Eulerian cycle (that is, an Eulerian walk that ends at the same vertex it started at)?								
\item The complete graph $K_6$ have an Eulerian walk?								
\item Does the complete bipartite graph $B_{3,3}$ have an Eulerian walk?								
\item Does the complete bipartite graph $B_{3,3}$ have a Hamiltonian walk?	

\item True or False: The complete graph $K_5$ on 5 vertices contains a closed Eulerian cycle.

\item True or False: The complement of the path graph on 5 vertices has an Eulerian walk.

\item Solve Example~\ref{Emoreeuler}.

\item True or False: There exists a connected graph $G$ with vertices of degrees 2, 2, 2, 3, 3, 4, 4, 4, 6 that contains an Eulerian walk starting at a vertex of degree 2 and ending at a vertex of degree 4.

\item True or False: There exists a graph $G$ containing both closed and non-closed Eulerian walks.

\item True or False: There exists a connected graph $G$ with vertices of degrees 3, 3, 4, 4, 4, 4, 4, 6 that contains an Eulerian walk starting at a vertex of degree 4 and ending at a vertex of degree 4.

\item True or False: There exists a connected graph $G$ with vertices of degrees 2, 2, 2, 2, 3, 4, 4, 6 that contains a non-closed Eulerian walk starting at the vertex of degree 3.

    \item Prove that a graph $G$ with $n$ vertices has a Hamiltonian cycle if and only if $C_n$ is a subgraph of $G$.
    \item For which pairs $(n,m)$ does the complete bipartite graph $B_{n,m}$ contain a Hamiltonian cycle?
\end{enumerate}

\section{Investigation: The number of walks}

This section is about counting the number of walks between two vertices.

\begin{example}
Your hamster's cage looks like this graph:
\begin{center}
\includegraphics[width=0.9in]{08-GraphsWalksCycles/graphFixedDegrees.pdf}
\end{center}
Let's label the vertices of this graph by $1$ (for the vertex of degree 1), $2$ (for the vertex of degree $3$, and then $3$, $4$, $5$ around the cycle (in either direction).

Every minute, the hamster walks along an edge to an adjacent vertex (the hamster never remains put at the same vertex and is always on the move).
Pick a starting vertex $i$ and an ending vertex $j$ and a number of minutes $N$.
Our goal is to count the number of ways that the hamster can travel from vertex $i$ to vertex $j$ in exactly $N$ minutes; let's call that number $h_{i,j}(N)$.

For example, let $i=j=1$.
Then $h_{1,1}(N)$ counts the number of ways the hamster can travel from 
vertex $1$ back to vertex $1$ in $N$ minutes.
When $N=1$, then $h_{1,1}(1)=0$ because after one minute, the hamster is at vertex $2$ and has not returned to vertex $1$.

When $N=2$, then the possible hamster routes are 
\[1 \mapsto 2 \mapsto 1, \ 1 \mapsto 2 \mapsto 3, \ 1 \mapsto 2 \mapsto 5.\]This shows that $h_{1,1}(2)=1$ and $h_{1,3}(2)=1$ and $h_{1,5}(2) =1$.
It also shows that $h_{1,2}(2)=0$ and $h_{1,4}(2)=0$.
When $n=3$, the possible hamster routes starting at vertex $1$ are:
\[1 \mapsto 2 \mapsto 3 \mapsto 4, \ 1 \mapsto 2 \mapsto 3 \mapsto 2, \ 
1 \mapsto 2 \mapsto 5 \mapsto 4, \  1 \mapsto 2 \mapsto 5 \mapsto 2, \ 
1 \mapsto 2 \mapsto 1 \mapsto 2.\]
In particular, this shows that $h_{1,1}(3)=0$, because the hamster is not at vertex $1$ after $3$ minutes.
\end{example}

In the hamster example above, there is a lot of information to keep track of: the starting position $i$, the ending position $j$, and the number of minutes $N$.
To keep track of all that information, for each positive integer $N$, we will make a $5 \times 5$ matrix $M_N$, where the entry in the $i$th row and $j$th column is $h_{i,j}(N)$.
For example, 
\[M_1=\left(\begin{array}{ccccc}
0 & 1 & 0 & 0 & 0\\
1 & 0 & 1 & 0 & 1\\
0 & 1 & 0 & 1 & 0\\
0 & 0 & 1 & 0 & 1\\
0 & 1 & 0 & 1 & 0 \\
\end{array}\right).\]
By definition, $M_1$ is the adjacency matrix, because the $i,j$th entry is $1$ exactly when the $i$th and $j$th vertices are adjacent.

With a lot of work, we compute the matrix for $N=2$:
\[M_2=\left(\begin{array}{ccccc}
1 & 0 & 1 & 0 & 1 \\
0 & 3 & 0 & 2 & 0 \\
1 & 0 & 2 & 0 & 2 \\
0 & 2 & 0 & 2 & 0 \\
1 & 0 & 2 & 0 & 2 \\
\end{array}\right).\]
For example, the first row of $M_2$ contains the values $h_{1,j}(2)$ that we computed before.
Notice that the entries in the diagonal 
of $M_2$ are the degrees of the vertices.
This makes sense, because the number of walks that leave vertex $i$, go to another vertex $j$ and then return to vertex $i$ is exactly the degree of vertex $i$.

We could try to compute $M_3$ by hand, but it would be very time-consuming. We would need to count the number $h_{i,v}(2)$ of routes the hamster could take in two minutes, organized by their ending vertex $v$, and then check if the hamster can travel from $v$ to the endpoint $j$ along an edge. 
For example, when $i=1$ and $j=4$, we would compute the following table.
\[\begin{array}{|c|c|c|c|c|c|}
\hline
v & 1 & 2 & 3 & 4 & 5 \\ \hline
h_{1,v}(2) & 1 & 0 & 1 & 0 & 1 \\ \hline
v \text{ adjacent to } 3 & no & no & yes & no & yes \\ \hline
\end{array}\]
We ignore the $2$-minute routes ending at the vertices $v=1,2,4$ and keep track of the $2$-minute routes ending at the vertices $v=3,5$, allowing us to compute:
\begin{eqnarray*}
h_{1,4}(3)& = & h_{1,1}(2) \cdot 0 + h_{1,2}(2) \cdot 0 +
h_{1,3}(2) \cdot 1 + 
h_{1,4}(2) \cdot 0 + 
h_{1,5}(2) \cdot 1 \\
& = & 1 \cdot 0 + 
0 \cdot 0 + 1\cdot 1 + 0 \cdot 0 + 1 \cdot 1=2.
\end{eqnarray*}
As a sanity check, we can check that there are exactly two $3$-minute routes starting at $i=1$ and ending at $j=4$, namely:
\[1 \mapsto 2 \mapsto 3 \mapsto 4, \ 1 \mapsto 2 \mapsto 5 \mapsto 4.\]

So far, we have computed only one entry of the matrix $M_3$!  

Luckily there is a faster way to compute the matrix $M_3$, by using matrix multiplication from linear algebra.
The first thing we do is to use Sage to check that $M_2$ is the product of the matrix $M_1$ with itself:
\begin{verbatim}
M1 = Matrix([[0,1,0,0,0], [1,0,1,0,1], [0,1,0,1,0], [0,0,1,0,1], [0,1,0,1,0]]);

M2 = Matrix([[1,0,1,0,1], [0,3,0,2,0], [1,0,2,0,2], [0,2,0,2,0], [1,0,2,0,2]]);
\end{verbatim}
Then 
\begin{verbatim}
M2==M1^2
\end{verbatim} 
gives an output of TRUE.

The cube of $M_3$ can be computed on SAGE with the command 
$M3=M1^3$ and it gives that
\[M3=\left(\begin{array}{ccccc}
0 & 3 & 0 & 2 & 0 \\
3 & 0 & 5 & 0 & 5 \\
0 & 5 & 0 & 4 & 0 \\
2 & 0 & 4 & 0 & 4 \\
0 & 5 & 0 & 4 & 0
\end{array} \right).\]
It is time-consuming, but you can check that the $i,j$th entry of this matrix $M3$ equals $h_{i,j}(3)$. For example, the 
entry in the first row and fourth column is $2$ and we computed that $h_{1,4}(3)=2$.

The hamster example generalizes to this very important theorem. 

\begin{theorem} \label{Tnumberwalk}
Let $G$ be a labeled graph with $n$ vertices.
Let $M_N$ be the $n \times n$ matrix whose $i,j$th entry is the number of walks in $G$ of length $N$ starting at vertex $i$ and ending at vertex $j$.
Then $M_1$ is the adjacency matrix of $G$ and $M_N = M_1^N$. 
In other words, the number $h_{i,j}(N)$
of walks in $G$ starting at vertex $i$ and ending at vertex $j$ that have exact length $N$ equals the entry in the $i$th row and $j$th column of the $N$th power of the matrix $M_1$. 
\end{theorem}

\begin{proof}
We provide only the main ideas in the proof.  
When $N=1$, then the $i,j$th 
entry of $M_1$ is the number of $1$-edge routes from vertex $i$ to vertex $j$. 
This number equals $1$ if vertices $i$ and $j$ are adjacent and equals $0$ otherwise.
That is the same as the $i,j$th entry of the adjacency matrix.  

We will now prove the result by induction on $N$, with the case $N=1$ being the previous paragraph.
Let $N \geq 2$. 
The inductive hypothesis is that the $i,v$th entry of $A=M_1^{N-1}$ equals the number of number of walks of length $N-1$ starting at vertex $i$ and ending at vertex $v$; this hypothesis holds for all 
$n^2$ entries of the matrix $A$, so all integer values of $i$ and $v$ such that
$1 \leq i \leq n$ and $1 \leq v \leq n$.
Let $B=M_1$.  We need to show that the 
$i,j$th entry of the matrix multiplication
$M_1^N = A \times B$ is the number of walks in $G$ of length $N$ starting at vertex $i$ and ending at vertex $j$. 

This will rely on the following fact from linear algebra.  Suppose $A$ and $B$ are both $n \times n$ matrices. 
By matrix multiplication, if $C=A \times B$,
then $C$ is also an $n \times n$ matrix. To find the entry $c_{i,j}$ in the $i$th row and $j$th column of $C$, we take the dot product of the $i$th row of $A$ with the $j$th column of $B$.  In other words, 
if the $i$th row of $a$ is 
$(a_{i,1}, \ldots, a_{i,n})$ and the 
$j$th column is (the transpose of) 
$(b_{1,j}, \ldots, b_{n,j})$, then 
\[c_{i,j} = a_{i,1} \cdot b_{1,j} + \cdots + a_{i,n} \cdot b_{n,j}.\]

So, fixing $i$ and $j$, we need to see why this dot product equals the number of walks in $G$ of length $N$ starting at vertex $i$ and ending at vertex $j$.
We separate these walks into groups,
based on the last vertex $v$ in the walk which occurs before returning to the vertex $j$.
If the last vertex of the walk (before returning to $j$) is $v$, then this walk separates into (i) a walk in $G$ of length $N-1$ from $i$ to $v$, and then (ii) a walk of length $1$ from $v$ to $j$.
The number of walks of type (i) is $a_{i,v}$, the $i,v$th coefficient of $A$, by the inductive hypothesis.  The number of walks of type (ii) is 1 if $v$ and $j$ are adjacent in $G$ and is $0$ otherwise; this is $b_{v,j}$, the $v,j$th entry in $B$.
So the number of walks of length $N$ from $i$ to $j$ such that $v$ is the last vertex before $j$ is
$a_{i,v} \cdot b_{v,j}$, which equals $a_{i,v}$ if $v$ and $j$ are adjacent and which equals $0$ if $v$ and $j$ are not adjacent.
Since $1 \leq v \leq n$, the total number of walks is 
\[\sum_{v=1}^n a_{i,v} \cdot b_{v,j}.\] 
By the rules of matrix multiplication, this equals the entry $c_{i,j}$ of $C =M_1^N$. 
\end{proof}

In fact, Theorem~\ref{Tnumberwalk} is true even when the graph has a self-loop or directed edges (meaning edges that can be traveled in only one direction).  In the next example, we show a graph where the Fibonacci numbers occur as the number of walks.

\begin{example} \label{EgraphwalkFibo}
Consider the directed graph drawn below.
\begin{center}
\includegraphics[width=2in]{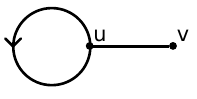}
\end{center}
How many routes of length $N$ are there from $u$ to $v$?
\end{example}

\begin{brainstorm} $ $
We brainstorm by making a table of the number of routes of length $N$ from $u$ to $v$, as $n$ varies from 0 to 5.
\begin{center}
\begin{tabular}{|c|c|c|c|c|c|c|}
\hline
$N$ & 0 & 1 & 2 & 3 & 4 & 5\\
\hline
\# ways & 0 & 1 & 1 & 2 & 3 & 5\\
\hline
picture & & \includegraphics[width=0.3in]{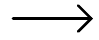} & \includegraphics[width=0.6in]{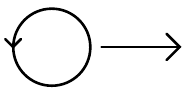} & \includegraphics[width=0.9in]{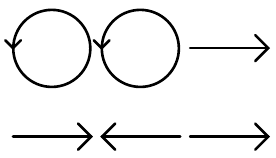} & \includegraphics[width=1.2in]{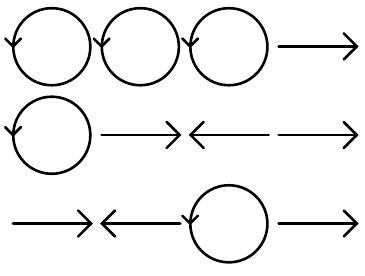} & \includegraphics[width=1.5in]{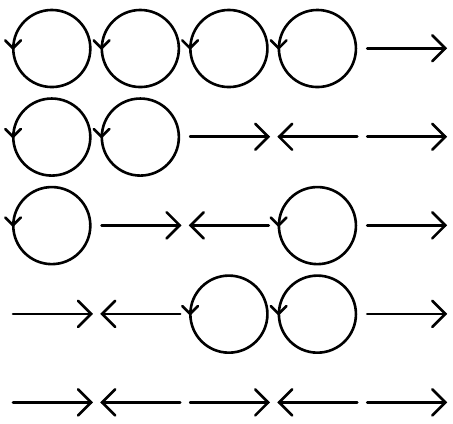} \\
\hline
\end{tabular}
\end{center}
We guess that the number of routes of length $N$ from $u$ to $v$ is the $N$-th Fibonacci number, $F_N$.
\end{brainstorm}

\begin{answer}
We give a proof to show that our above guess is correct.

Let $P_N$ be the number of routes of length $N$ from $u$ to $v$.

Note $P_0=0$ and $P_1=1$ (or, if you prefer, note $P_1=1$ and $P_2=1$).

Note $P_{N+1}=P_N+P_{N-1}$, since every path of length $N+1$ starts with either
\begin{itemize}
\item \includegraphics[width=0.3in]{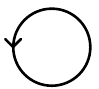}
, after which there are $P_N$ continuing paths, or
\item \includegraphics[width=0.6in]{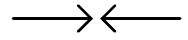}
, after which there are $P_{N-1}$ continuing paths.
\end{itemize}

Since $P_N$ satisfies the same recurrence relation as the Fibonacci numbers $F_n$, and since $P_0=F_0$ and $P_1=F_1$ (or, if you prefer, $P_1=F_1$ and $P_2=F_2$), this shows $P_N=F_N$ for all $N$.
\end{answer}

Here is another way to solve this problem using Theorem~\ref{Tnumberwalk}.
Let $u$ be the first vertex and $v$ be the second vertex.
The adjacency matrix for this graph is 
\[M_1=\left(\begin{array}{cc}
1 & 1 \\
1 & 0 
\end{array}
\right).\]
(The self-loop from $u$ to $u$ leads to the non-zero entry in the diagonal of $M_1$).
By Theorem~\ref{Tnumberwalk}, 
the number of routes from $u$ to $v$ of length $N$ is the entry in the first row and second column of $M_1^N$.
Let $F_N$ be the $N$th Fibonacci number.
For $N \geq 1$, we claim that 
\[M_1^N = \left(\begin{array}{cc}
F_{N+1} & F_N \\
F_N & F_{N-1}
\end{array} \right).\]
We prove this by induction.

\begin{proof}
When $N=1$, then the claim is true because $F_0=0$, $F_1=1$, and $F_2=1$.
The inductive hypothesis is that
\[M_1^{N-1} = \left(\begin{array}{cc}
F_{N} & F_{N-1} \\
F_{N-1} & F_{N-2}
\end{array} \right).\]
Then we compute $M_1^N$ using matrix multiplication:
\begin{eqnarray*}
M_1^N & = & M_1^{N-1} \times M_1 \\
& = & \left(\begin{array}{cc}
F_{N} & F_{N-1} \\
F_{N-1} & F_{N-2}
\end{array} \right) \times 
\left(\begin{array}{cc}
1 & 1 \\
1 & 0 
\end{array}
\right) \\
& = & \left(\begin{array}{cc}
F_N+F_{N-1} & F_N \\
F_{N-1} + F_{N-2} & F_{N-1}
\end{array}
\right)\\
& = & 
\left(\begin{array}{cc}
F_{N+1} & F_N \\
F_{N} & F_{N-1}
\end{array}
\right).
\end{eqnarray*}
This ends the proof of the inductive step.
\end{proof}

\subsection*{Exercises}

\begin{enumerate}
\item Let $G$ be the complete graph $K_4$, with labeled vertices.  
    
\begin{enumerate}
    \item Find the adjacency matrix $M_1$ of $G$.
    \item Using SAGE, compute $M_n=M_1^n$ for $n$ from $2$ to $6$ but do not write down the answer. 
    \item Find the number of routes of length 6 that start and end at vertex $1$. 
    \item Find the number of routes of length $6$ that start at vertex $1$ and end at vertex $2$.
    \item Explain why it makes sense that all the entries on the diagonal of $M_n$ are the same number $d$ and all the entries that are not on the diagonal of $M_n$ are the same number $e$.
    \item If $n$ is even, what do you conjecture 
    that $d-e$ always equals?
    \item If $n$ is odd, what do you conjecture that $d-e$ always equals?
\end{enumerate}
    
\item Let $G$ be the cycle graph $C_4$, with labeled vertices.  
\begin{enumerate}
    \item Find the adjacency matrix $M_1$ of $G$.
    \item Using SAGE, compute $M_n=M_1^n$ for $n$ from $2$ to $6$ but do not write down the answer.
    \item Find the number of routes of length 6 that start and end at vertex $1$. 
    \item Find the number of routes of length $6$ that start at vertex $1$ and end at vertex $2$.
    \item Half the entries of $M_n$ are $0$.  Explain why this makes sense - your answer should split into two cases depending on whether $n$ is even or odd. 
    \item Make a conjecture about which non-zero number appears in the other half of entries of $M_n$.
\end{enumerate}
    
\item Let $G$ be the path graph $P_4$, with labeled vertices.  
\begin{enumerate}
    \item Find the adjacency matrix $M_1$ of $G$.
    \item Using SAGE, compute $M_n=M_1^n$ for $n$ from $2$ to $6$ but do not write down the answer.
    \item Half the entries of $M_n$ are $0$.  Explain why this makes sense - your answer should split into two cases depending on whether $n$ is even or odd. 
    \item Explain what each of the non-zero entries in $M_6$ means. 
    \item Explain why it makes sense that the non-zero entries in the center of $M_6$ are larger than the others. 
\end{enumerate}

\item 
Show there are 149 routes from vertex 0 to vertex 9 in the following directed graph.
Paths in this directed graph are only allowed to cross an edge in the direction indicated by the arrow on that edge.
\begin{center}
\includegraphics[width=5in]{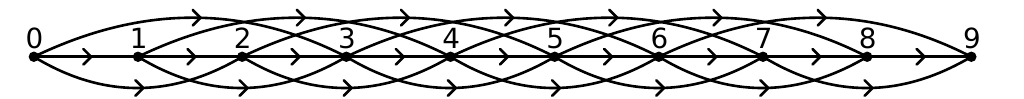}
\end{center}

\begin{center}
\includegraphics[width=5in]{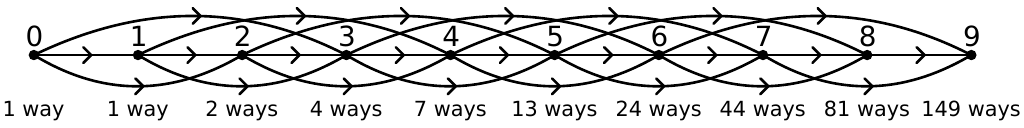}
\end{center}
\end{enumerate}

%% file: 09-Trees/Trees.tex
\chapter{Trees}\label{chap:trees}

\begin{videobox}
\begin{minipage}{0.1\textwidth}
\href{https://youtu.be/cbfz9xeByok}{\includegraphics[width=1cm]{video-clipart-2.png}}
\end{minipage}
\begin{minipage}{0.8\textwidth}
Click on the icon at left or the URL below for this section's short lecture video. \\\vspace{-0.2cm} \\ \href{https://youtu.be/cbfz9xeByok}{https://youtu.be/cbfz9xeByok}
\end{minipage}
\end{videobox}

Trees are in some sense the ``simplest" of all possible graphs, but do not underestimate their usefulness!
Trees can be used to represent hierarchical structure, such as in a family tree of ancestors and descendants.
In computer science, trees are a powerful data structure, and they can be used to implement tasks such as searching or sorting with faster times than na\"{i}ve implementations allow.

For example, suppose you are given a list of $n$ numbers.
You sort these numbers from smallest to largest, and then you would like to store them in a way which
\begin{itemize}
\item does not require much space,
\item allows you to search for a specific item in the list,
\item makes it easy to add new items to the sorted list, and
\item allows you to easily remove items from the list.
\end{itemize}
\begin{center}
\scalebox{.55}{
\begin{tikzpicture}[
            > = stealth, 
            shorten > = 1pt, 
            auto,
            node distance = 3cm, 
            semithick 
        ]

        \tikzstyle{every state}=[
            draw = black,
            thick,
            fill = white,
            minimum size = 4mm
        ]
        \node[state, fill=CSUGold] (0) {$~0~$};
        \node[state, fill=CSUGold] (1) [above right of=0] {$~1~$};
        \node[state, fill=CSUGold] (5) [above right of=1] {$~5~$};
        \node[state, fill=CSUGold] (6) [below right of=5] {$6$};
        \node[state, fill=CSUGold] (7) [below right of=6] {$7$};
        \node[state, fill=CSUGold] (3) [below right of=1] {$~3~$};
        \node[state, fill=CSUGold] (2) [below left of=3] {$~2~$};
        \node[state, fill=CSUGold] (4) [below right of=3] {$~4~$};
        \node[state, fill=CSUGold] (8) [below right of=7] {$8$};

        \path[->, ultra thick, CSUGreen] (5) edge node {} (1);
        \path[->, ultra thick, CSUGreen] (1) edge node {} (0);
        \path[->, ultra thick, CSUGreen] (5) edge node {} (6); 
        \path[->, ultra thick, CSUGreen] (1) edge node {} (3);
        \path[->, ultra thick, CSUGreen] (3) edge node {} (2);
        \path[->, ultra thick, CSUGreen] (6) edge node {} (7);
        \path[->, ultra thick, CSUGreen] (3) edge node {} (4);
        \path[->, ultra thick, CSUGreen] (7) edge node {} (8);

    \end{tikzpicture}}
\end{center}
A binary search tree, such as the one drawn above, is a data structure that is very efficient for these tasks.
Note that each arrow pointing down and to the left points to a smaller number, and each arrow pointing down and to the right points to a larger number.
The amount of space a binary search tree requires is $O(n)$, where $n$ is the number of vertices.
Searching, adding, and deleting numbers are all operations of average running time $O(\log n)$.
Here $O$ is the ``big-O" notation used in computational complexity; these average operation times are considered to be quite fast.

For example, to add the number 5.5 to the tree, we would start at the top vertex, the \emph{root}, which in this case is 5.
Since $5.5>5$, we proceed down and to the right.
We get to the vertex $6$, and since $5.5<6$, we insert a new vertex labeled $5.5$ to the bottom left of vertex 6, along with a directed arrow from 6 to 5.5.
If we were later asked to verify whether 5.5 was present in our search tree or not, we would take the same sequence of steps.

How was this search tree created?
In particular, why have we chosen vertex 5 as the root of this tree?
This tree was created because we were handed the numbers 5, 1, 0, 6, 3, 2, 4, 7, 8, in that order, and asked to store them as a binary search tree.
Since we were handed 5 first, we made it the root (the top vertex).
As 1 came next, and is less than 5, we added 1 to the bottom left of 5, connected by an arrow.
Similarly, since 0 came next, and is less than 5 and less than 1, we added 0 to the bottom left of 1.
Next comes 6, which is more than 5, and so we added 6 to the bottom right of 5.
Next comes 3, which is less than 5 and more than 1, and hence we add it to the bottom right of 1.
We proeceed in this manner to create the entire tree.

In this chapter, we give a mathematical introduction to trees.

\section{The language of trees}

\begin{videobox}
\begin{minipage}{0.1\textwidth}
\href{https://youtu.be/0J4Hc5Iy0T0}{\includegraphics[width=1cm]{video-clipart-2.png}}
\end{minipage}
\begin{minipage}{0.8\textwidth}
Click on the icon at left or the URL below for this section's short lecture video. \\\vspace{-0.2cm} \\ \href{https://youtu.be/0J4Hc5Iy0T0}{https://youtu.be/0J4Hc5Iy0T0}
\end{minipage}
\end{videobox}

Trees are the simplest possible graphs. 

\begin{definition}
A graph $G=(V,E)$ is a \defn{tree} if it is connected and contains no cycle as a subgraph.
\end{definition}

Because a tree is connected, there is a walk between every pair of vertices.  Because a tree has no cycles, for every pair of vertices there is a unique walk between them that does not backtrack.
This is one reason why trees are the most simple graphs.

Below we show four examples of trees.
\begin{center}
\includegraphics[width=0.8\textwidth]{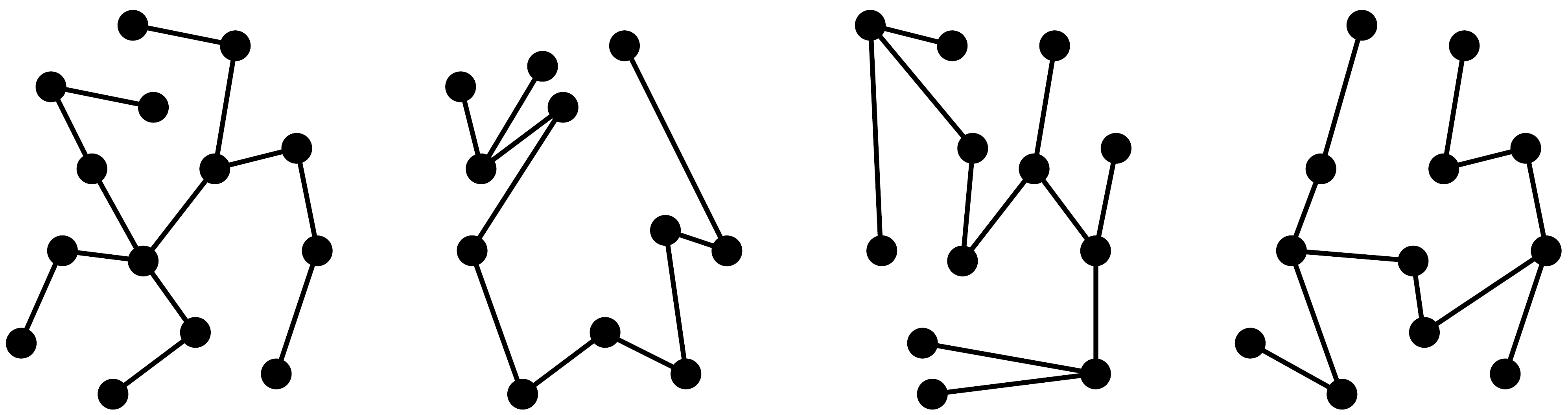}
\end{center}

We also draw four non-examples of trees: the first three graphs are not trees because they contain cycles, and the fourth graph is not a tree because it is not connected.
\begin{center}
\includegraphics[width=0.8\textwidth]{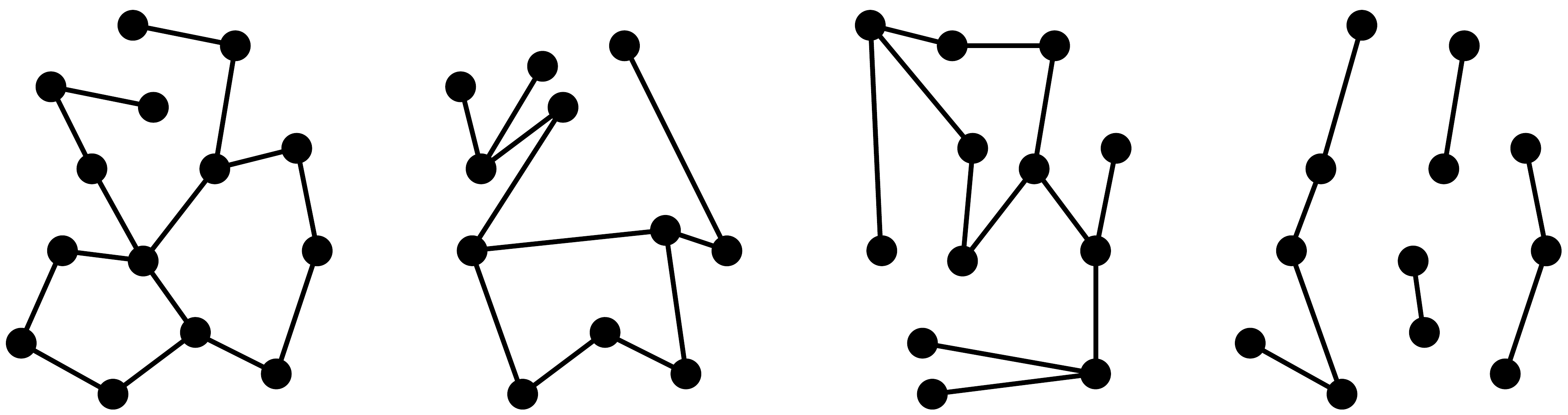}
\end{center}

\begin{definition}
A \defn{leaf} of a tree is a vertex of degree 1.
\end{definition}

Let $G=(V,E)$ be a connected graph.
While keeping the vertex set $V$ fixed, if we add an edge to $G$, then it will still be connected.
In contrast, while keeping the vertex set $V$ fixed, if we remove an edge from $G$ then it might become disconnected.
Trees can be described as the smallest possible connected graphs on a given vertex set, as we make precise in the following theorem.

\begin{theorem}
\label{thm:tree-minimal-connected}
A graph $G$ is a tree if and only if it is connected, but deleting any edge makes it disconnected.
\end{theorem}

\begin{proof}
For the forward direction, let $G$ be a tree.
Then $G$ is connected by definition.
If removing some edge $uv$ produced a connected graph $G'$, then the path from $u$ to $v$ in $G'$ along with the edge $uv$ would create a cycle in $G$, contradicting the fact that $G$ is a tree.
Hence removing any edge from $G$ makes it disconnected.

For the reverse direction, to show $G$ is a tree we must show it is cycle-free (we already know it is connected by assumption).
If $G$ had a cycle, then removing any edge from this cycle would produce a connected graph, which is a contradiction. Hence $G$ is cycle-free, and thus a tree.
\end{proof}

Similarly, if a graph $G$ has no cycles, then the graph remains cycle-free if we remove edges.
By contrast, adding edges can certainly make a formerly cycle-free graph now have cycles.

\begin{theorem}
\label{thm:tree-max-cycle-free}
A graph $G$ is a tree if and only if it contains no cycles, but adding any new edge creates a cycle.
\end{theorem}
We omit the proof; it is included as one of the exercises in this section.

Therefore trees achieve a sweet-spot.
Among all connected graphs, trees are the minimal ones, in the sense that removing any edge from a tree makes it disconnected.
Among all graphs without cycles, trees are the maximal ones, in the sense that adding any edge to a tree creates a cycle.

If $G$ is a connected graph with $n$ vertices, then a \defn{spanning tree} of $G$ is a subgraph of $G$ that is a tree with $n$ vertices.  There can be many spanning trees in a graph.

\subsection*{Exercises}

\begin{enumerate}
    \item Show that every tree is bipartite.
    \item Find an example of a bipartite graph that is not a tree.
    \item How many edges do you need to remove from the complete bipartite graph B4,5 in order to make it into a tree?							
\item How many edges do you need to remove from P7 to divide it into 3 connected components?							
\item Let $T$ be a tree with 10 vertices. How many non-zero entries does the adjacency matrix for T have?							
\item Let $T$ be a tree with 10 vertices. To build an edge list for T, how many numbers between 1 and 10 need to be stored?					

    \item Show that both 
    $B_{1,6}$ and $P_7$ are spanning trees inside $K_7$.
    \item Show that both 
    $B_{1,n-1}$ and $P_n$ are spanning trees inside $K_n$.
    \item Show using strong induction that every connected graph contains a spanning tree.
    
    \item Show that a graph $G$ is a tree if and only if it contains no cycles, but adding any new edge creates a cycle.
    This Theorem~\ref{thm:tree-max-cycle-free}.
\end{enumerate}

\section{The number of edges in a tree}

There are many different trees that one can draw with 14 vertices;
some of them are drawn below.
\begin{center}
\includegraphics[width=0.8\textwidth]{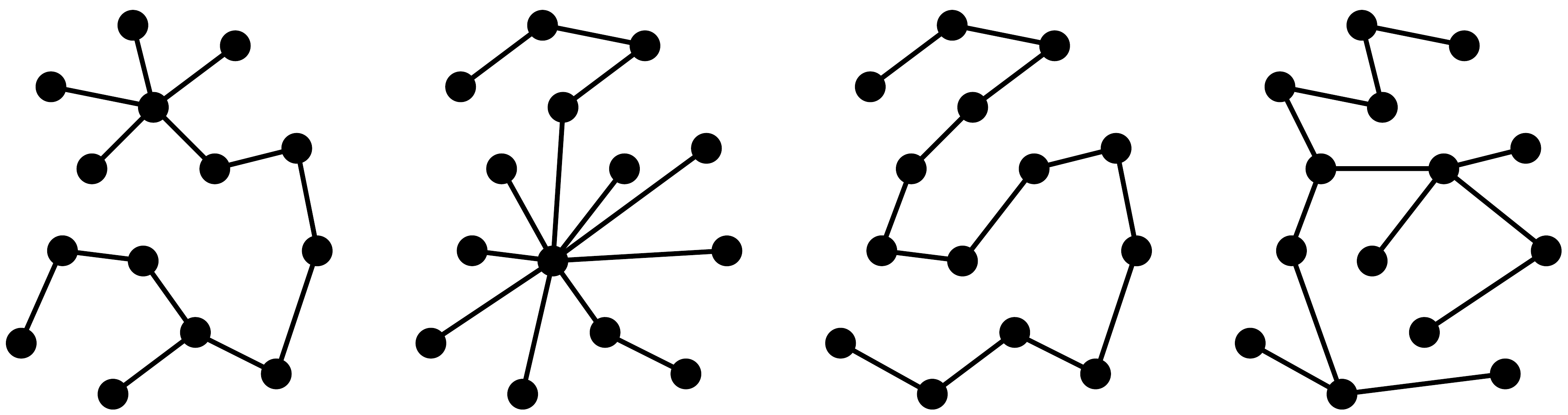}
\end{center}
Despite this wide variety, surprisingly, every tree with 14 vertices has exactly 13 edges.
This is generalized by the following theorem.

\begin{theorem}\label{thm:treeNumEdges}
Every tree with $n$ vertices has $n-1$ edges.
\end{theorem}

\begin{proof}
We give a proof using strong induction.  As the base case, note that every tree with $n=1$ vertex has no edges.

For the inductive step, let $n\ge 2$.
Suppose that every tree with $m<n$ vertices has $m-1$ edges; we must prove that every tree with $n$ vertices has $n-1$ edges.

Let $G$ be an arbitrary tree with $n$ vertices.
Since $G$ is connected with $n\ge 2$, there must be at least one edge $e=v_1v_2$ in $G$.
By Theorem~\ref{thm:tree-minimal-connected}, removing the edge $e$ leaves a graph $G'$ that is disconnected, as shown below.  Also $G'$ has exactly two connected components because adding one edge can only create a connection between two parts.

\begin{center}
    \includegraphics{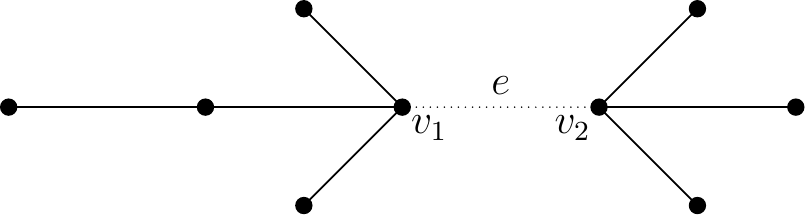}
\end{center}
Let the connected component containing vertex $v_1$ have $m_1$ vertices, and let the connected component containing vertex $v_2$ have $m_2$ vertices.  (In the diagram above, $m_1=5$ and $m_2=4$.)  Note $m_1+m_2=n$.  Since $G$ has no cycles, the connected component containing vertex $v_1$ has no cycles, and is therefore a tree.  Furthermore, since this tree has $m_1<n$ edges, by our inductive assumption the connected component with vertex $v_1$ has $m_1-1$ edges.  Similarly, the connected component containing vertex $v_2$ is a tree with $m_2-1$ edges.
Therefore, when we add edge $e=v_1v_2$ back in, we see that the number of edges in tree $G$ is
\[ (m_1-1)+(m_2-1)+1 = m_1+m_2-1 = n-1. \]
Hence we are done by strong induction.
\end{proof}

\begin{example}
Each of the below trees have their number of edges equal to their number of vertices minus one, as we know must be the case from Theorem~\ref{thm:treeNumEdges}.
\begin{center}
\includegraphics[width=0.6\textwidth]{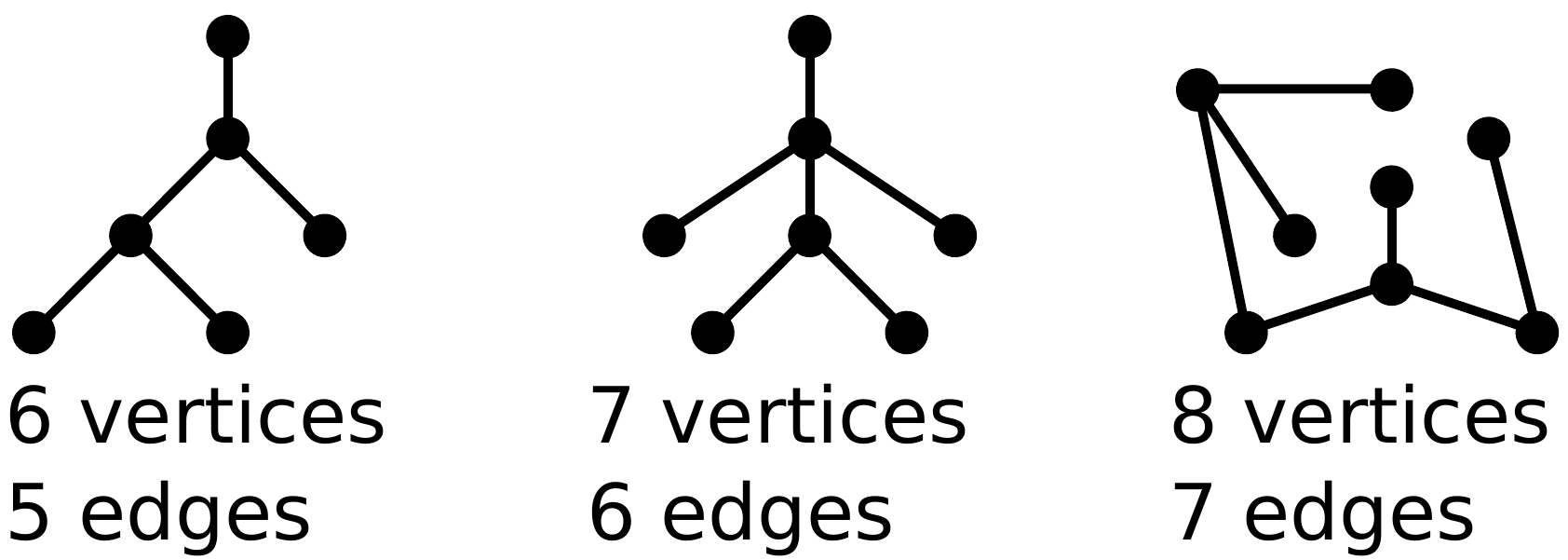}
\end{center}
\end{example}

A \emph{forest} is a graph whose connected components are all trees.
The image below is of a single forest, which consists of six trees.
Note that one tree consists of only a single vertex, in the bottom right.
\begin{center}
\includegraphics[width=0.6\textwidth]{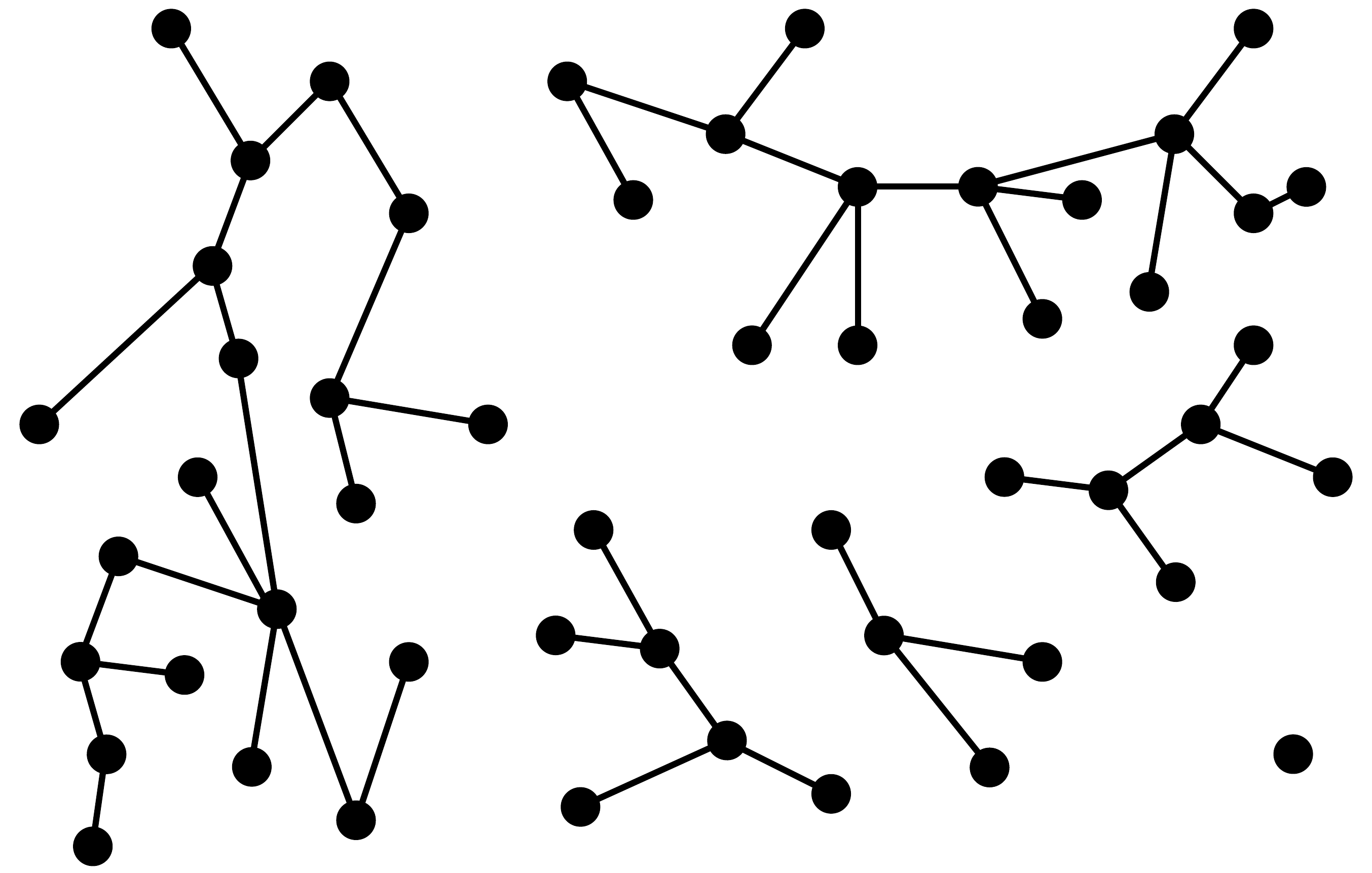}
\end{center}

The following theorem is a generalization of Theorem~\ref{thm:treeNumEdges}; let's not miss the forest for the trees.

\begin{theorem}\label{thm:forrestNumEdges}
Every forest with $n$ vertices and $m$ connected components has $n-m$ edges.
\end{theorem}

Since every tree is a graph, the storage methods for graphs in Section~\ref{sec:graphStorage} can also be used to store trees.
However, since the structure of trees is very specific, these storage structures for graphs are less efficient when storing trees.
Recall from Section~\ref{sec:graphStorage} that a labeled graph with $n$ vertices can be stored as an $n\times n$ adjacency matrix with binary entries, either $0$ or $1$ depending on whether an edge is present or not.
Such a representation can store $2^{(n^2-n)/2}$ different labeled graphs with $n$ vertices.
But as we will learn in Cayley's theorem in Section~\ref{sec:cayley}, the number of labeled trees on $n$ vertices is a lot smaller, only $n^{n-2}$.
Similarly, recall from Section~\ref{sec:graphStorage} that a labeled graph with $n$ vertices can be stored as a $2\times e$ list of edges.
Since a tree with $n$ vertices has $n-1$ edges, a tree can be stored as a $2 \times (n-1)$ list of edges, i.e.\ as $2(n-1)$ entries in $\{0,1,\ldots,n-1\}$.
This representation is closer to being efficient for trees, while still not perfectly efficient, since we could store $n^{2(n-1)}$ different things, even though there are only $n^{n-2}$ labeled trees with $n$ vertices.

Binary search trees, which we described briefly at the very beginning of this chapter, are reasonably efficient storage structure for labeled trees.
Binary search trees are also easy to edit, say by adding or removing nodes from the tree.
Can you think of other good storage structures for trees?

\subsection*{Exercises}

\begin{enumerate}
\item If $G$ is a connected graph with $n$ vertices and $n-1$ edges, show that $G$ is a tree.

\item How many edges need to be removed from $K_7$ to create a spanning tree?  
\item How many edges need to be removed from $K_n$ to create a spanning tree?

\item How many edges need to be removed from $B_{3,4}$ to create a spanning tree?  
\item How many edges need to be removed from $B_{n,m}$ to create a spanning tree?

\item True or False: There exists a tree with 6 vertices of degrees 1, 1, 2, 2, 3, 3.

\item For a tree with $n$ vertices, what is the sum of the degrees of the vertices? 

\item
Prove Theorem~\ref{thm:forrestNumEdges} using the following strategy.
Suppose $G$ is a forest with $n$ vertices and $m$ connected components.
For $i=1,\ldots,m$, let connected component $i$ have $n_i$ vertices.
So $n_1+\ldots+n_m=n$.
Apply Theorem~\ref{thm:treeNumEdges} on each connected component of $G$.
How many edges do you get in total?

\end{enumerate}

\section{Labeled trees: Cayley's theorem}
\label{sec:cayley}

\begin{videobox}
\begin{minipage}{0.1\textwidth}
\href{https://www.youtube.com/watch?v=Z15ofSdio3I}{\includegraphics[width=1cm]{video-clipart-2.png}}
\end{minipage}
\begin{minipage}{0.8\textwidth}
Click on the icon at left or the URL below for this section's short lecture video. \\\vspace{-0.2cm} \\ \href{https://www.youtube.com/watch?v=Z15ofSdio3I}{https://www.youtube.com/watch?v=Z15ofSdio3I}
\end{minipage}
\end{videobox}

A \defn{labeling} of a tree is an assignment of numbers to each vertex.
If a tree has $n$ vertices, then then we will always use the labels $0$ through $n-1$, where each such integer is used exactly once as a label.

For example, the below two trees are different as labeled trees.
\begin{center}
\includegraphics[width=0.28\textwidth]{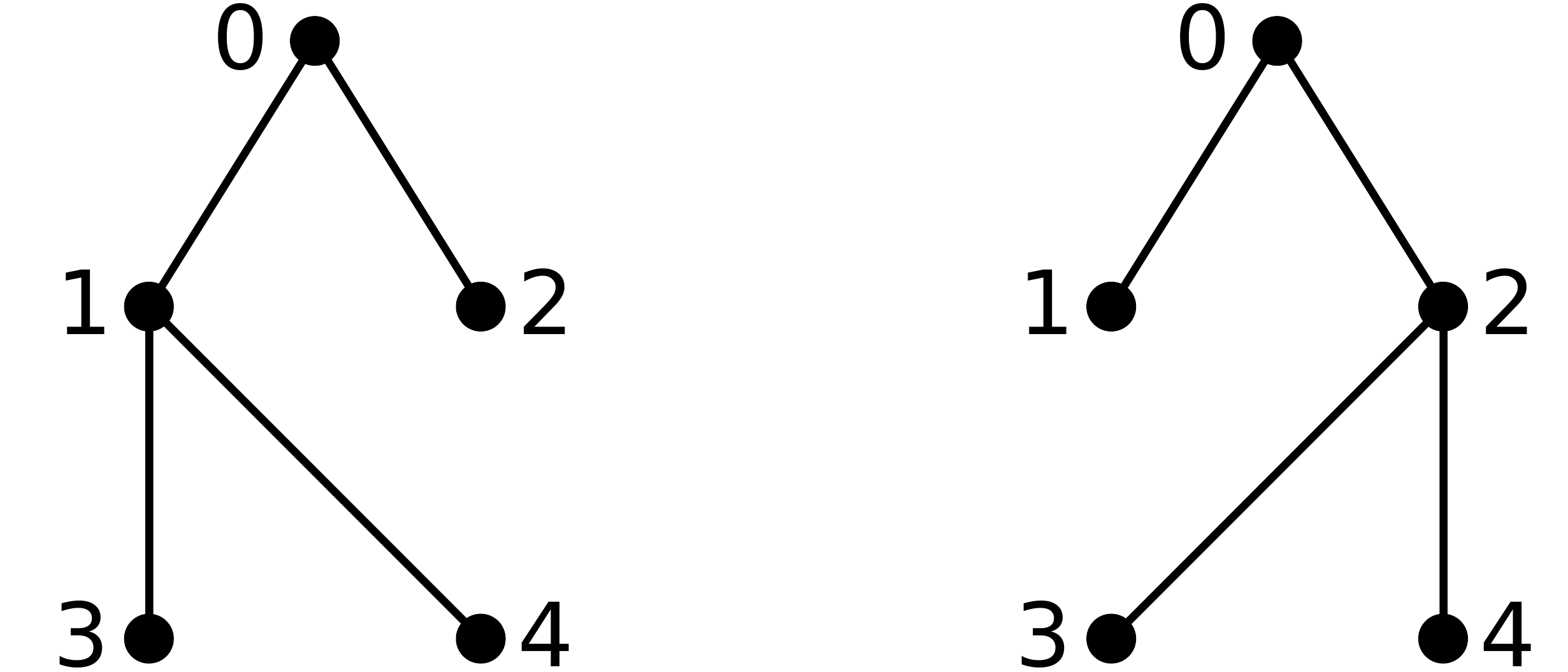}
\end{center}
By contrast, the below two trees are the same as unlabeled trees!
\begin{center}
\includegraphics[width=0.28\textwidth]{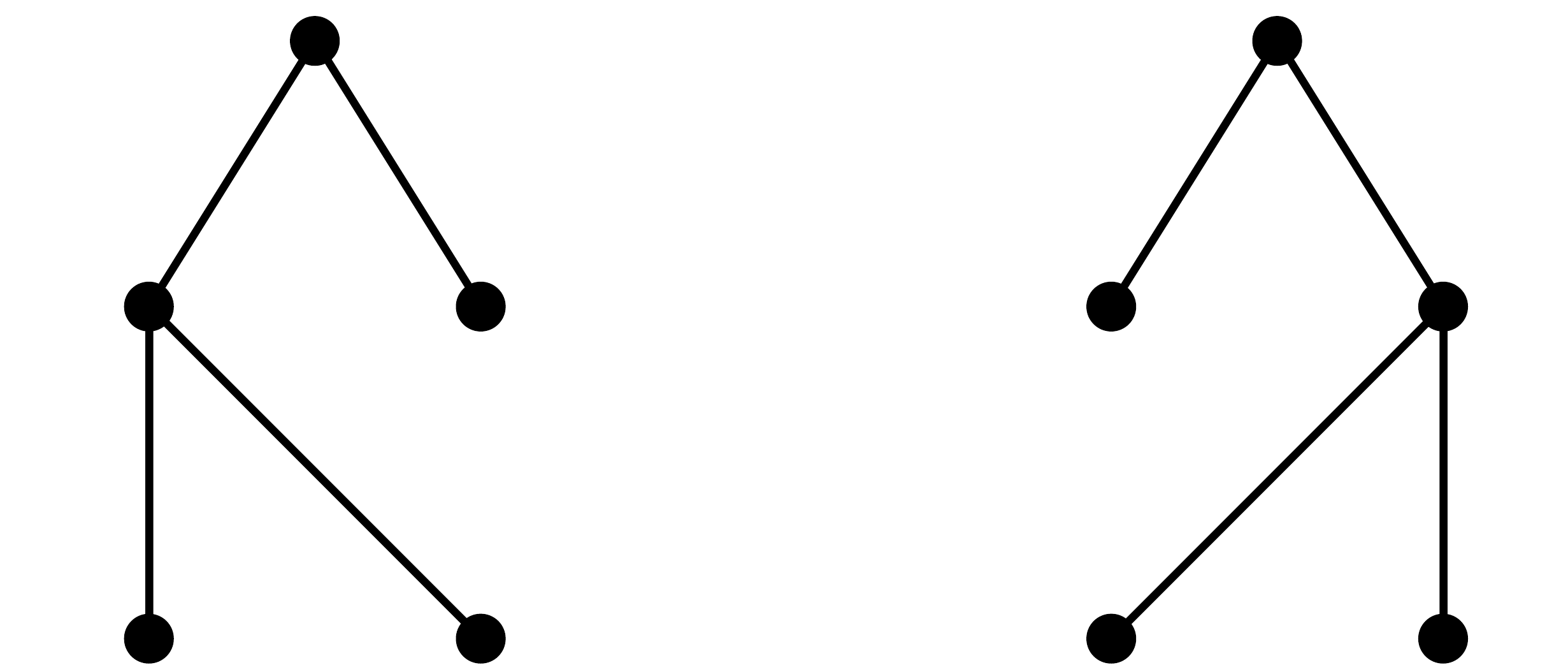}
\end{center}

How many different labeled trees with $n$ are there, as a function of $n$?
Let's begin by doing some counts for $n$ small.

There is only a single labeled tree on $n=1$ vertex --- the single vertex, equipped with label $0$.
\begin{center}
\includegraphics[width=0.07\textwidth]{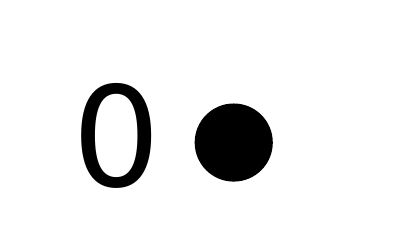}
\end{center}

There is only a single labeled tree on $n=2$ vertices --- an edge with its two vertices labeled 0 and 1.
Note that swapping these vertex labels still results in the same labeled tree!
\begin{center}
\includegraphics[width=0.2\textwidth]{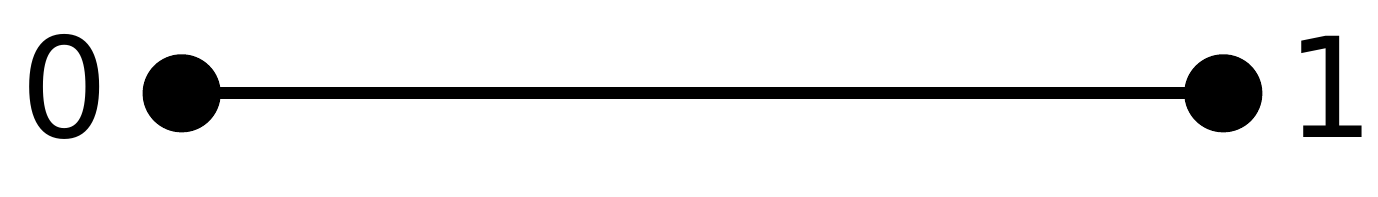}
\end{center}

There are three labeled trees on $n=3$ vertices, drawn below.
\begin{center}
\includegraphics[width=0.7\textwidth]{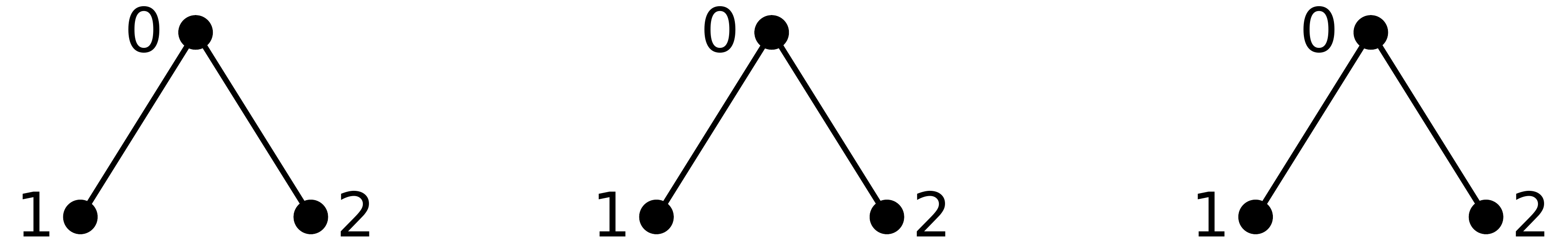}
\end{center}

With work, we can count the number of labeled trees on $n=4$ vertices.
By Theorem~\ref{thm:treeNumEdges} a tree on 4 vertices has 3 edges.
As we see below, there are $\binom{4}{2}=6$ possible locations for these 3 edges.
\begin{center}
\includegraphics[width=0.15\textwidth]{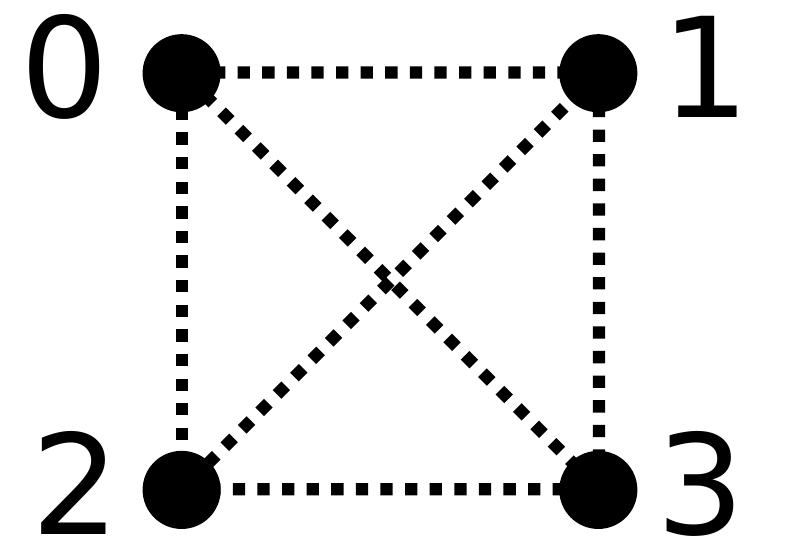}
\end{center}
Hence there are $\binom{6}{3}=20$ possible labeled graphs with 3 edges.
Of these 20 labeled, all except for the following 4 are trees.
\begin{center}
\includegraphics[width=0.7\textwidth]{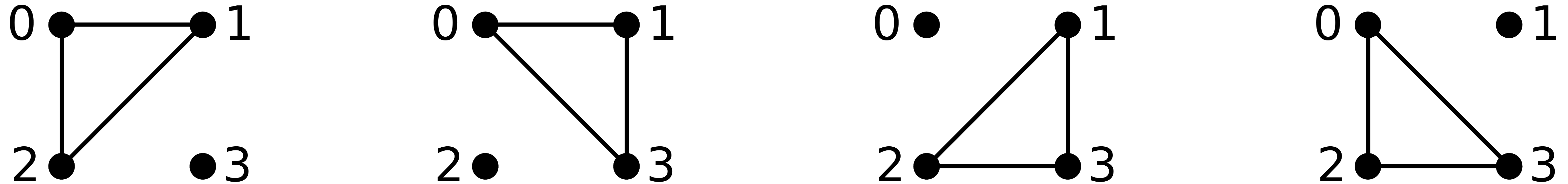}
\end{center}
Hence there are 16 labeled trees on $n=4$ vertices, as counted below.

Let $L(n)$ be the number of labeled trees on $n$ vertices. 
We computed that
$L(1)=1$, $L(2)=1$, $L(3)=3$, and $L(4)=16$.
It is not easy to guess a formula for $L(n)$ in general given the limited number of computations we have done so far.
But a formula is known, and it is included in Cayley's theorem.

\begin{theorem}[Cayley's Theorem]
\label{thm:treeCayley}
The number of labeled trees on $n$ vertices is $L(n)=n^{n-2}$.
\end{theorem}

The proof of Cayley's theorem is much deeper than a simple count.
To see this, try to expand the methods we used to compute $L(4)$ to compute $L(5)$ or $L(6)$.

\subsection*{Exercises}
\begin{enumerate}
\item How many trees on 5 labeled vertices are there?																									
\item How many rooted trees on 5 labeled vertices are there?																									
\item How many doubly rooted trees on 5 labeled vertices are there?	

\item Some of the problems below are about the NCAA basketball tournament.  This is a single-elimination tournament with 64 teams; it may be viewed as a rooted binary tree with 64 leaves. An \emph{unfilled NCAA bracket} is a partially labeled binary tree with 64 leaves, where the only labeled vertices are the 64 leaf vertices in Round 0 (each labeled with one of the 64 teams in the tournament). To fill out a bracket you must label the remaining vertices with which team you think will win each game.  

\begin{enumerate}
\item Explain why a filled out NCAA bracket is not a labeled tree. 

\item Show that the NCAA tournament has 63 games, as follows. Pretend each team brings one new basketball to the tournament. One basketball is used in each match. The losing team departs with the used ball, and the still-unused ball is taken by the match's winner to the next round. How many basketballs remain unused at the end of the tournament? How many basketballs must have been used? Hence, how many games must have been played?

\item How many ways are there to fill out the bracket? (The labels on the 64 leaf nodes are fixed; filling out the bracket means choosing a winner for each game.)

\item Draw all 16 labeled trees with $4$ vertices.

\item Read the proof of Cayley's theorem at 
\url{https://golem.ph.utexas.edu/category/2019/12/a\_visual\_telling\_of\_joyals\_pro.html}
\end{enumerate}
\end{enumerate} 
\section{Binary trees}

A \textit{binary tree} is a recursive tree-like structure defined as follows.

\begin{definition}
A \textbf{binary tree} is either a single vertex with no edges, or a graph on more than one vertex consisting of a \textbf{root vertex} $v$ and two edges, $L$ and $R$, each connecting $v$ to the root vertex of a smaller binary tree.
\end{definition}

Using this definition, we can build up binary trees recursively as follows.  The single vertex graph:
\begin{center}
$\bullet$
\end{center}
is a binary tree by the initial condition case of the recursive definition.  Then, we can join two single vertex graphs (which we label $u$ and $w$) to a new root vertex $v$ to obtain the following graph: 

\begin{center}
    \includegraphics{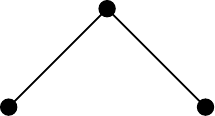}
\end{center}

Here, we draw edges $L$ and $R$ as going down-and-left and down-and-right from $v$, and we call $u$ and $w$ the \textbf{left child} and \textbf{right child} of $v$, and $v$ the \textbf{parent} of $u$ and $w$.  We can then build up more binary trees by combining the ones we have constructed so far.  For instance, we could have the left and right children of the root both connect to the above three-node graph, or just one of the children being the above three-node graph and the other being a single vertex.  This gives us three more binary trees:

\begin{center}
    \includegraphics{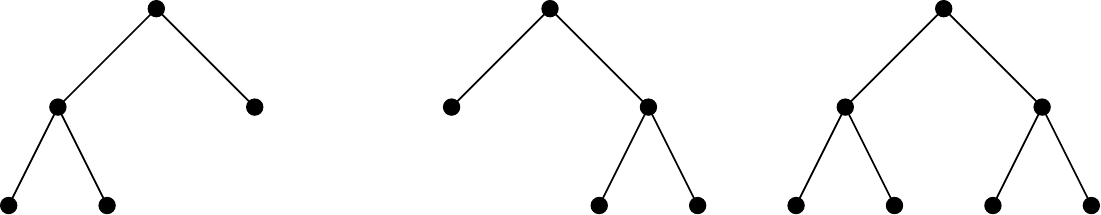}
\end{center}

Because these trees are generated recursively and any new edge is marked as a left edge or a right edge with one vertex being the ``parent'' and the other the ``child'', the notion of \textit{left child} and \textit{right child} is well defined for any vertex in a binary tree.

Using the notion of left and right children, a more informal alternate definition of a binary tree may be stated as follows.

\begin{definition}[Informal definition of binary tree]
A \textbf{binary tree} is a graph formed by starting with a root vertex at the top, and drawing edges and vertices downwards to the left or right in such a way that every vertex either has two children (a left and right child) or no children.
\end{definition}

\subsection{Tournaments and Catalan numbers}

One application of binary trees is as a way to model \textbf{tournaments} in which players face off in pairs, and the winner advances to the next rounds.  In the diagram below, we can think of the labels on the leaves as the initial ``seeding'' of the tournament, where it is determined who will face off against one another first.  

\begin{center}
    \includegraphics{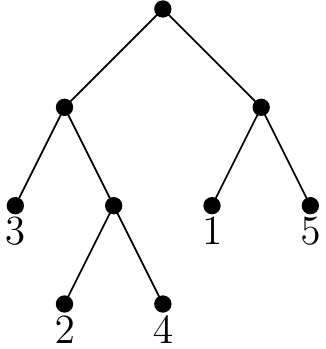}
\end{center}

Suppose it is a chess tournament, so if a left child and right child of a node face off against each other, we assign the left child to be the player who goes first, and the right to be the player who goes second (thus making the distinction between left vs right important). Suppose the numbered labels are some measure of skill, so that larger numbers will likely beat smaller numbers.  Then in the first round of the tournament drawn above, $2$ and $4$ face off against each other and $1$ and $5$ face off against each other, and $4$ and $5$ advance to the appropriate parent nodes:

\begin{center}
    \includegraphics{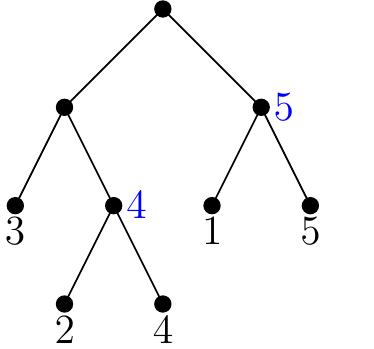}
\end{center}

Then, $4$ defeats $3$ and advances, and finally $5$ defeats $4$ in the last round and advances to the top node, being declared the winner.

\begin{center}
    \includegraphics{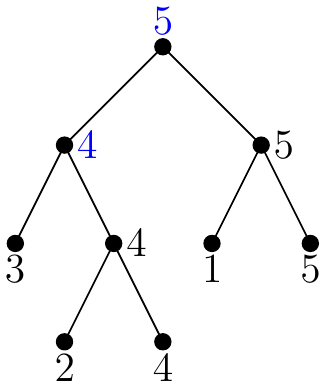}
\end{center}

This was not the only tournament we could have seeded for 5 players!  How many possible ways can we set up such a tournament with $n$ players?  In terms of binary trees, this boils down to the following question.

\begin{question}
How many binary trees have exactly $n$ leaves?
\end{question}

For now we will put aside the question of labeling the leaves with numbers, and just answer the question above for unlabeled binary trees with $n$ leaves.

\begin{theorem}
The number of binary trees with $n$ leaves is the $(n-1)$st Catalan number $C_{n-1}$.
\end{theorem}

\begin{proof}
Let $B_n$ be the number of binary trees with $n$ leaves.  We simply need to show that the numbers $B_1,B_2,B_3,\ldots$ satisfy the same recursion as the one that defines the sequence of Catalan numbers $C_0,C_1,C_2,\ldots$.

We first check the initial condition: we have $B_1=1$ since there is only one binary tree with one leaf, and $C_0=1$, so $B_1=C_0$.

We now need to show the numbers $B_n$ satisfy the Catalan recursion.  To count the number of binary trees on $n+1$ leaves, note that any such tree has some number $k\ge 1$ of leaves on the left subtree branching from the root, and the other $n+1-k$ on the right subtree.  We break into cases based on the value of $k$: if $k=1$ there are $B_1B_{n}$ possibilities, if $k=2$ there are $B_2B_{n-1}$ possibilities, and so on.  Adding these together, we have $$B_{n+1}=B_1B_{n}+B_{2}B_{n-1}+B_{3}B_{n-2}+\cdots +B_{n}B_1.$$  Comparing this to the recursion $$C_{n}=C_0C_{n-1}+C_1C_{n-2}+\cdots+C_{n-1}C_0$$ we see that shifting all the indices up by $1$ and replacing $C$ by $B$ does indeed make the recursions match.  Therefore $B_n=C_{n-1}$ for all $n\ge 1$.
\end{proof}

Let's see how this applies to our tournaments.  There are $C_4$ binary trees with $5$ leaves.  Recall the explicit formula for Catalan numbers that we derived in Section \ref{sec:Catalan}: $$C_n=\frac{1}{n+1}\binom{2n}{n}.$$  Using this formula, we have $C_4=\frac{1}{5}\binom{8}{4}=14$.  Finally, given a binary tree with $5$ leaves, if we have 5 different players, the number of ways we can choose who labels which leaf is just the number of ways of rearranging those players, which is $5!=120$.  Therefore there are $14\cdot 120=1680$ different tournaments you can design for the 5 players!

\subsection{At-most binary trees}

It is somewhat harder to count binary trees according to how many vertices they have, rather than how many leaves.  However, there is a nice answer if we modify the definition slightly.

\begin{definition}
An \textbf{at-most binary} tree is a tree constructed recursively starting from a root node at the top, where every node has either a left child only, a right child only, both a left and right child, or neither.
\end{definition}

Here are all of the at-most binary trees on $3$ vertices:
\begin{center}
    \includegraphics{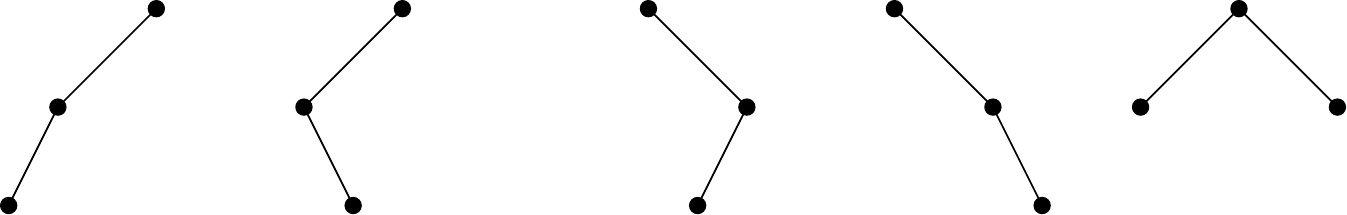}
\end{center}
Notice that there are $5$ of them, which is a Catalan number.  Indeed, the Catalan numbers arise here as well.

\begin{theorem}
The number of at-most binary trees on $n$ vertices is the $n$th Catalan number $C_n$.
\end{theorem}

\begin{proof}
 Let $A_n$ be the number of at-most binary trees on $n$ vertices. We will consider the empty tree as being an at-most binary tree, so that $A_0=1$.  This is equal to $C_0$, so the initial condition of the Catalan recursion is satisfied.
 
 Now, we wish to show that the sequence of numbers $A_n$ satisfies the same recursion as the Catalan numbers.  If $n\ge 1$ there is a root vertex in any at-most binary tree on $n$ vertices, and this vertex has either a left child only, a right child only, both, or neither.  In all cases, both the left and the right branches from the root form a (possibly empty) at-most binary tree, where there are a total of $n-1$ vertices in the left and right trees.  
 
 If there are $k$ vertices on the left, and $n-1-k$ on the right, there are therefore $A_k A_{n-1-k}$ possible trees of this form.  Summing over all $k$, we have $$A_n=A_0A_{n-1}+A_1A_{n-2}+A_2A_{n-3}+\cdots+A_{n-1}A_0,$$ and so indeed the recursion is satisfied.
\end{proof}

\subsection{Investigation: Increasing binary trees and permutations}

At-most binary trees, with appropriate vertex labelings, can give us a new way of thinking about permutations.

\begin{definition}
An \textbf{increasing binary tree} is an at-most binary tree along with a labeling of its vertices by the numbers $1,2,\ldots,n$ such that the labels increase as one reads any downwards path in the tree.
\end{definition}

\textit{Warning:} An increasing binary tree may not be an actual binary tree, because it only needs to be an at-most binary tree!

Here is an example of an increasing binary tree.

\begin{center}
    \includegraphics{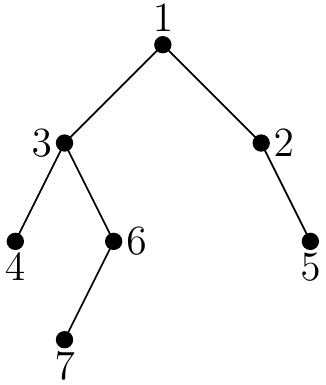}
\end{center}

This tree can be assigned to a permutation as follows.  The $1$ at the root corresponds to the $1$ in the permutation; then, the left subtree consists of the entries to the left of the $1$ and the right subtree consists of the entries to the right of the $1$ in the permutation.  So we know that the $1$ has four numbers before it and two after it in the permutation:

$$\underline{\phantom{2}},\underline{\phantom{2}},\underline{\phantom{2}},\underline{\phantom{2}},1,\underline{\phantom{2}},\underline{\phantom{2}}$$

Then, we repeat this analysis on the right and left subtrees; the $3$ is at the top of the left subtree, with one entry to its left and two to its right, so in the permutation it is in the second position.  Similarly the $2$ is just after the $1$:
$$\underline{\phantom{2}},3,\underline{\phantom{2}},\underline{\phantom{2}},1,2,\underline{\phantom{2}}$$
We then continue this process on the subtrees whose top nodes are $4$, $6$, and $5$, to obtain the permutation
$$4,3,7,6,1,2,5$$

To show that this correspondence gives a bijection, we need to show that we can reverse it.  We illustrate this by continuing with this example.  Starting from the permutation $4,3,7,6,1,2,5$, can we reconstruct the original tree?  Indeed, we must place $1$ at the top, and then its left and right children must be the smallest element left of $1$ and the smallest element right of $1$ respectively.  The left and right children of each successive node is similarly determined. 

We have shown:

\begin{theorem}
Increasing binary trees on $n$ labeled vertices are in bijection with permutations of the set $\{1,2,3,\ldots,n\}$.
\end{theorem}

Since there are $n!$ permutations, this immediately allows us to count the increasing binary trees.

\begin{corollary}
There are exactly $n!$ increasing binary trees on $n$ vertices.
\end{corollary}

\subsection*{Exercises}

\begin{enumerate}

    \item Draw all $14$ binary trees having $5$ leaves.
    \item How many binary trees are there on $7$ labeled leaves?
    \item Draw all $14$ at-most binary trees having $4$ vertices.
    \item How many unlabeled binary trees have exactly 4 leaves?					
\item How many leaf-labeled binary trees are there on labeled leaves 1,2,3,4?					
\item How many unlabeled at-most binary trees on 4 vertices are there?	
    \item How many at-most binary trees have $6$ vertices?
    \item Draw the permutation $2,1,5,3,4$ as an increasing binary tree.
    \item Draw all $24$ increasing binary trees on $4$ vertices.
    \item Define a \textbf{weakly increasing binary tree} to be an at-most binary tree with every vertex labeled by a positive integer, such that for any vertex labeled $v$ with left child $w$ and/or right child $u$, we have $w>v$ and $u\ge v$.  In other words, we can now have repeated numbers, but ties can only go rightwards.  So the tree at left is a weakly increasing binary tree, but the one at right is not:
    \begin{center}
        \includegraphics{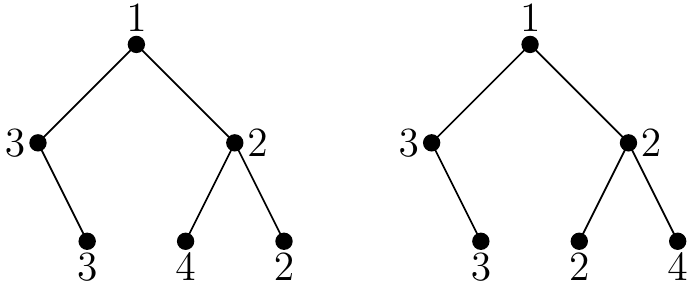}
    \end{center}
    How many weakly increasing binary trees on $6$ vertices use the numbers $1,2,2,3,3,4$ as their labels?  How many permutations are there of the sequence $1,2,2,3,3,4$?  Can you find a bijection between these two objects?
\end{enumerate}

%% file: 10-GraphOptimization/GraphOptimization.tex
\chapter{Graph optimization}\label{chap:optimization}

Consider the below graph, in which each vertex represents a ski lift hut, and in which each edge represents a ski lift.
Each ski lift edge is decorated with an arrow showing in which direction the lift moves up the hill.
Furthermore, each edge is labeled with the number of people that can be transported up the ski lift per minute.
What is the maximum number of people that can be moved from the bottom of the hill to the top of the hill per minute?

\begin{center}
\includegraphics[]{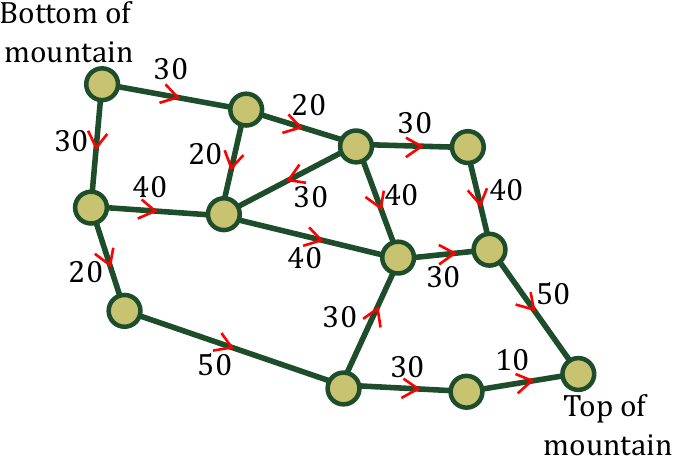}
\end{center}

It is clear that no more than 60 people can go up every minute, since at most 60 can leave the bottom of the mountain.
It turns out that it is possible to transport 60 people per minute from the bottom to the top of the hill.
Below, one strategy for how to transport 60 people per minute is drawn in red, with the number indicating how many people ride a lift.
Note that the number of passengers on a ski lift, as labeled below in red, is always less than or equal to the capacity of that lift, as labeled above in black.
To achieve this flow of 60 people per minute, different skiers need to take different routes up the hill!

\begin{center}
\includegraphics[]{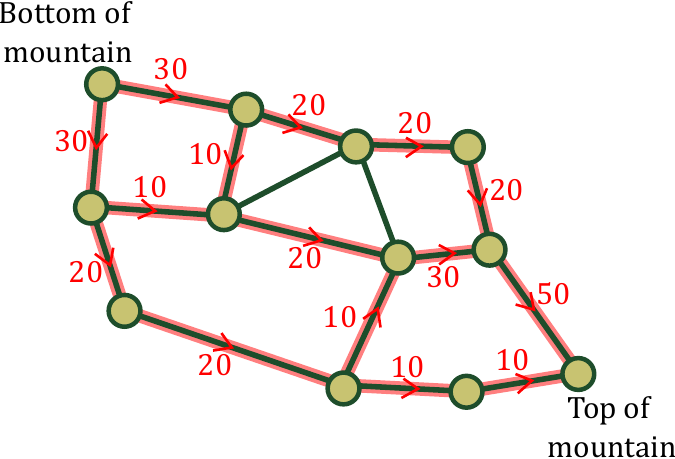}
\end{center}

In this chapter we study optimization problems on graphs.
In an optimization problem, we are trying to maximize or minimize something.
We may want to maximize the number of people transported per minute, or the number of pizzas we can deliver in an hour, or the amount of space between seats in a room.
Alternatively, we may want to minimize the cost we have to pay, the amount of trash we produce, or the amount of time necessary to complete a task.
Some optimization problems on graphs, such as finding the maximum spanning tree in Section~\ref{sec:minimal-spanning-trees}, can be solved efficiently via known algorithms.
Other optimization problems, such as the traveling salesperson problem in Section~\ref{sec:traveling-salesperson}, are difficult to solve exactly, although efficient algorithms to find approximately optimal solutions may exist.

\section{Minimal spanning trees}
\label{sec:minimal-spanning-trees}

\begin{videobox}
\begin{minipage}{0.1\textwidth}
\href{https://youtu.be/ld1tR8Oom5U}{\includegraphics[width=1cm]{video-clipart-2.png}}
\end{minipage}
\begin{minipage}{0.8\textwidth}
Click on the icon at left or the URL below for this section's short lecture video. \\\vspace{-0.2cm} \\ \href{https://youtu.be/ld1tR8Oom5U}{https://youtu.be/ld1tR8Oom5U}
\end{minipage}
\end{videobox}

Suppose the state of Colorado wants to connect the following towns with fiber-optic cable, enabling fast internet.
The cost of installing fiber-optic cable along a single edge is indicated by the number label on that edge.
What is the cheapest way to connect these 6 cities with fiber-optic cable?
\begin{center}
\includegraphics[]{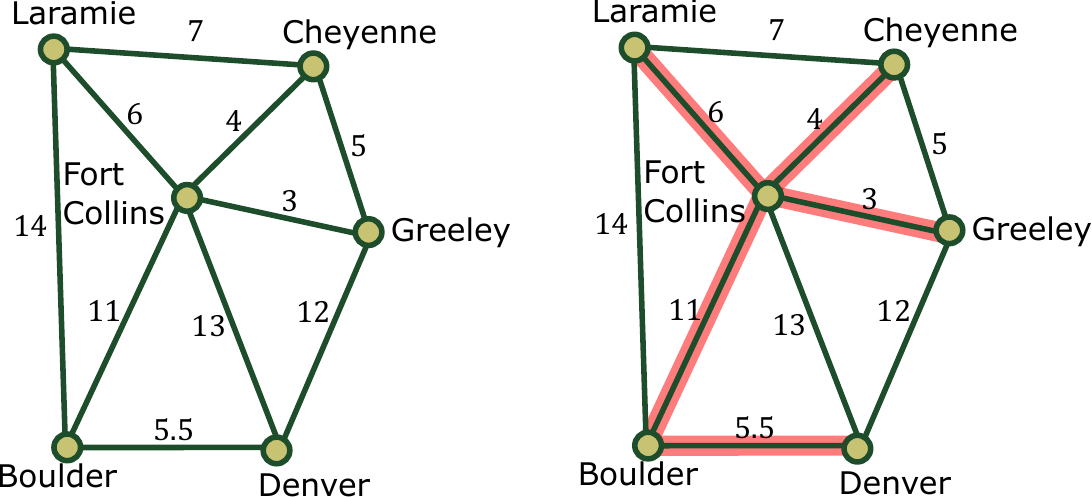}
\end{center}

Let us define the mathematical concept, minimal spanning trees, which underlies this fiber-optic cable example.

\begin{definition}
A \defn{weighted graph} is a graph equipped with a real number weight on each edge.
\end{definition}

A sub-tree of a graph $G$ is a subgraph of $G$ which also happens to be a tree.

\begin{definition}
Given a weighted connected graph $G$, a \defn{minimal spanning tree} is a sub-tree containing all vertices of $G$ whose total sum of edge weights is as small as possible.
\end{definition}

\begin{remark}
\label{rmk:mst-unique}
If no two edges of $G$ have the same weight, then a minimal spanning tree is unique.
\end{remark}

\begin{example}
If some edges of $G$ have the same weights, then a minimal spanning tree may or may not be unique.
\begin{center}
\includegraphics[height=1.7in]{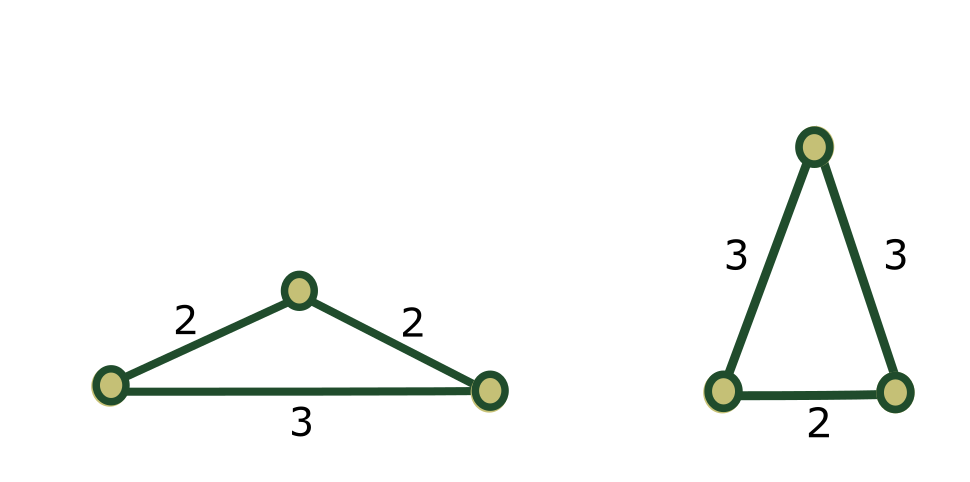}
\end{center}

Indeed, the weighted graph on the left above has a unique minimal spanning tree, whereas the weighted graph on the right has two different minimal spanning trees.
\end{example}

A poor algorithm to compute a minimal spanning tree would to be to compute the cost of all possible trees and then pick the cheapest.
Let's see that this is a poor algorithm when $G$ is the complete labeled graph on $n$ vertices.
In this case, there is a bijection between minimal spanning trees for $G$ and labeled trees on $n$ vertices.
By Cayley's Theorem (Theorem~\ref{thm:treeCayley}), there are $n^{n-2}$ possible labeled trees on $n$ vertices to check.
For $n=50$ vertices, we would have $n^{n-2}=50^{48}\approx 3.6\times 10^{81}$ trees to check, which is on the order of the number of atoms in the universe!

A much better algorithm for finding a minimal spanning tree is Kruskal's algorithm.

\vspace{1em}

\noindent\underline{\textbf{Kruskal's algorithm to produce a minimal spanning tree from a weighted connected graph $G$}}
\begin{enumerate}
\item Include all vertices of the graph $G$.
\item Include the cheapest edge (the one with the smallest weight).
Remove that edge from consideration.
\item If the subgraph is not connected, add the cheapest remaining edge that connects two of its connected components to the subgraph.
Remove that edge from consideration.
\item Repeat step (3) until the graph is connected. 
\end{enumerate}

At each step of the algorithm beyond the first, one edge is added and the output is a subgraph of $G$ with no cycles. 
When the algorithm terminates, the 
output is a connected graph with no cycles, thus a tree, and it contains all the vertices of $G$.
Thus the output is a spanning tree of $G$.
Later in the section, we will prove that
it is a minimal spanning tree.

If $G$ has $n$ vertices, 
then this algorithm will stop after adding exactly $n-1$ edges by Theorem~\ref{thm:treeNumEdges}.

Kruskal's algorithm is a \emph{greedy} algorithm, which means that at each step you do the cheapest possible thing to make progress.
What is surprising is that this greedy algorithm also happens to give you the global optimum!

If there are ties in Kruskal's algorithm, i.e., if there is more than one cheapest remaining edge that does not form a cycle, then these ties can be broken arbitrarily.
Whichever way you proceed to break this tie will still lead to a minimial spanning tree.

\begin{question}
If you and your friend each perform Kruskal's algorithm on a weighted graph, will you necessarily get the same minimal spanning trees as output? Will the minimal spanning trees you get necessarily have the same costs?
\end{question}

\begin{theorem}
When performed on any weighted connected graph $G$, Kruskal's algorithm will return a minimal spanning tree.
\end{theorem}

\begin{proof}
Let $T$ be a spanning tree found by Kruskal's algorithm. Let $G_0$ be any other spanning tree. Our goal is to show $\cost(T) \leq \cost(G_0)$, which will then imply that $T$ is a minimal spanning tree.

Let $e$ be the first edge (when constructing tree $T$ via Kruskal's algorithm) that is not in $G_0$.
Let $G_0\cup e$ denote the graph that is obtained by adding the edge $e$ to the graph $G_0$.
By Theorem~\ref{thm:tree-max-cycle-free}, adding any new edge between pre-existing vertices in a tree creates a cycle, and therefore the graph $G_0\cup e$ has a cycle $C$.
This cycle $C$ is not contained in the tree $T$, and therefore there is an edge $f$ of the cycle $C$ that is not in the tree $T$.

\begin{center}
\includegraphics[width=5in]{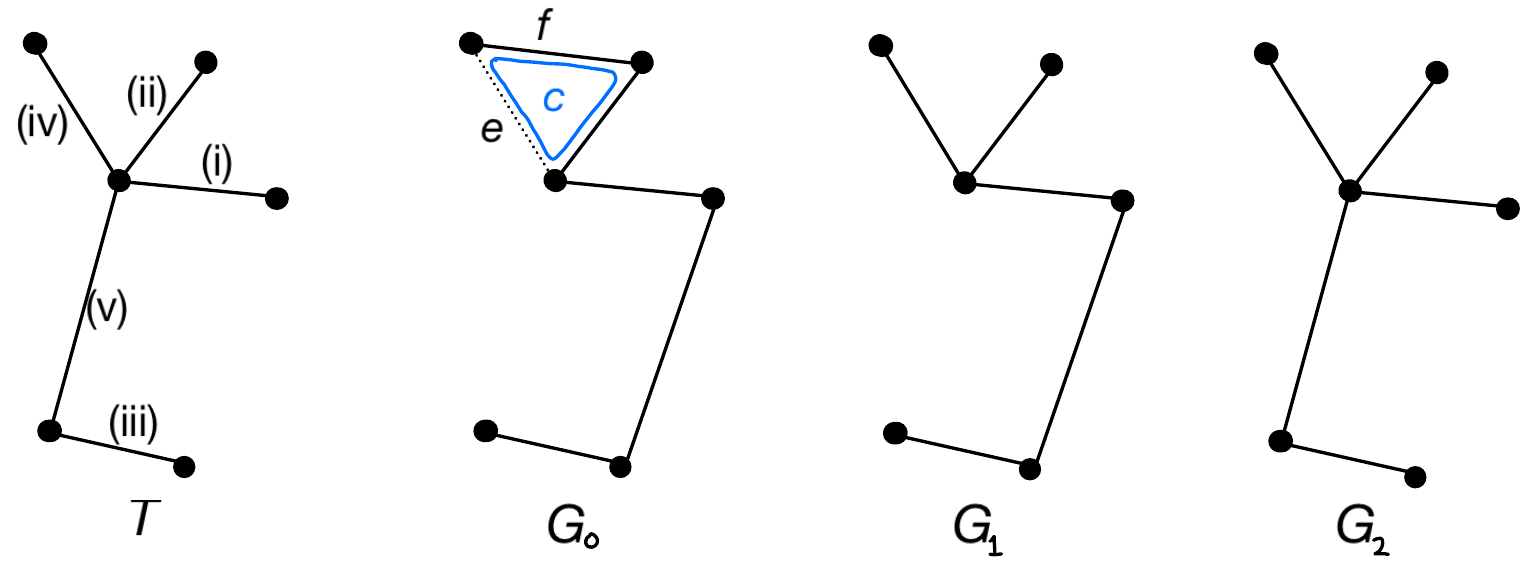}
\end{center}

Let $G_1=(G_0\cup e)\setminus f$ be the graph obtained by removing edge $f$ from $G_0\cup e$.
The weighting of the edge $e$ is less than or equal to the weighting of the edge $f$.  The reason is 
that $e$ was added to the graph of $T$ rather than $f$ at this stage of constructing $T$ in Kruskal's algorithm.  
So $\cost(G_1) \leq \cost(G_0)$.
So the proof will be complete if we can show the stronger statement that
$\cost(T) \leq \cost(G_1)$.

We repeat this same process with the tree $G_1$ in place of the tree $G_0$, obtaining a tree $G_2$.
We note that $G_2$ has one more edge from $T$ added and one more edge not in $T$ removed, and also satisfies $\cost(G_1)\ge\cost(G_2)$.
Since we keep adding edges in $T$ while removing edges not in $T$, eventually we get $G_k=T$ for some integer $k$.
We have
\[ \cost(G_0)\ge\cost(G_1)\ge\cost(G_2)\ge\ldots\ge\cost(G_k)=\cost(T), \]
as desired.
Hence Kruskal's algorithm always produces a minimal spanning tree.
\end{proof}

\subsection*{Exercises}

\begin{enumerate}

\item Consider the complete bipartite graph $B_{3,3}$, and label the vertices in the left set 1,2,3 and the vertices in the right set 1,2,3. Then label each edge by the sum of the two vertices it connects; for instance, the edge connecting left vertex 1 to right vertex 3 is labeled 4, and the edge connecting left vertex 2 to right vertex 2 is labeled 4 as well. What is the weight of a minimal spanning tree of this graph?	

\item Answer the previous problem but for $B_{4,4}$ labeled in the same way; label the vertices on the left and right 1,2,3,4, and label each edge by the sum of its two vertices. What is the weight of a minimal spanning tree?																									
\item A complete graph K4 with vertices labeled 1,2,3,4 has all its edges labeled by the same value 1. How many different minimal spanning trees are there? (Hint: Cayley's Theorem may come in handy.)																									

\item Find a minimal spanning tree on the weighted graph drawn below. 
\begin{center}
\includegraphics[width=.85\textwidth]{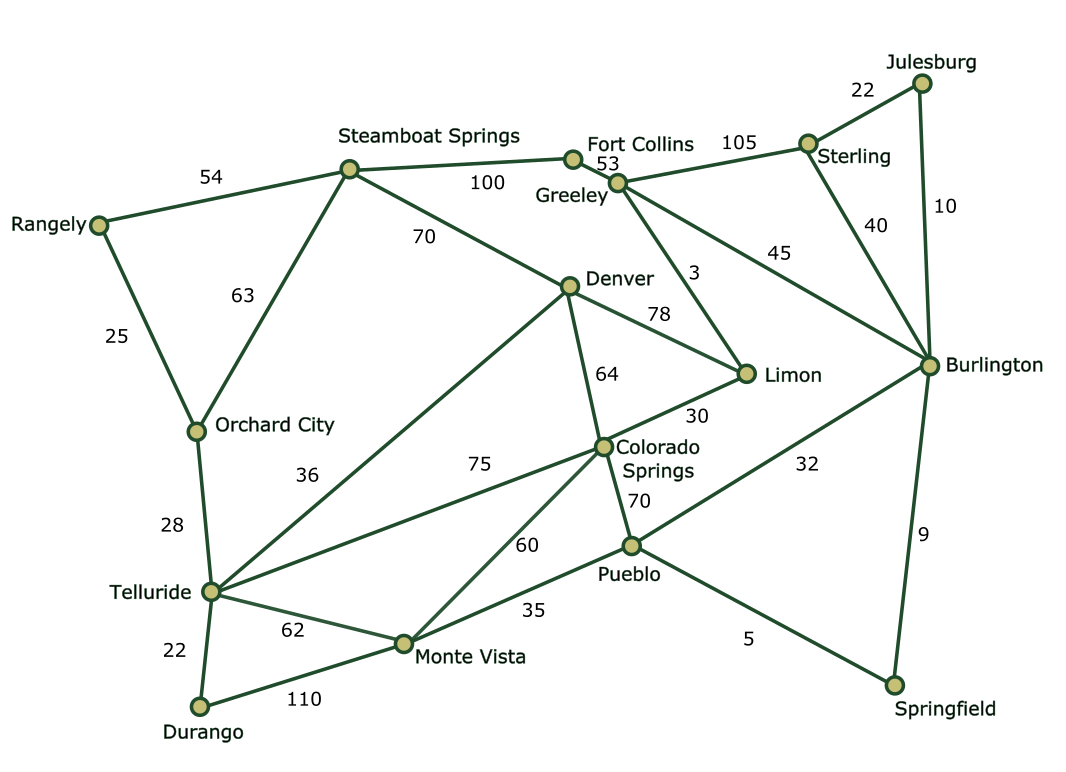}
\end{center}

\item Let $G$ be a weighted connected graph with positive edge costs.
\begin{enumerate}
\item Describe how to find a spanning tree for which the sum of the edge-costs is \emph{maximal}.

\emph{Hint: Create a new weighted graph $G'$ by multiplying each edge weight of $G$ by $-1$.}

\item Describe how to find a spanning tree for which the \emph{product} of the edge-costs is minimal.

\emph{Hint: Create a new weighted graph $G'$ by editing the edge weights of $G$ using logarithms.}
\end{enumerate}

\item 
Another way to measure the cost of a spanning tree is by the maximal cost of its edges;
let's call that the max-edge cost.
For example, the max-edge cost of the diagram at the beginning of this section is 11, for the edge between Boulder and Fort Collins.
Decide whether the following statement is true or false and explain:
the spanning tree produced by Kruskal's algorithm minimizes
the max-edge cost.
In thinking about this, it might be helpful to sort the edges in order of cost so that the first edge has the lowest cost and the last edge has the largest cost.

\item{Draw a weighted connected graph that has at least three different minimal spanning trees.}

\item Prove Remark~\ref{rmk:mst-unique}, which says that if no two edges of a connected graph $G$ have the same weight, then there is a unique minimal spanning tree for $G$.

\item True or False: If two edges in a weighted connected graph have the same cost, then the graph has more than one minimal spanning tree.

\end{enumerate}

\section{Traveling salesperson problem}
\label{sec:traveling-salesperson}

Let $G$ be a connected graph.
A \defn{tour} of $G$ is a closed walk that is allowed to retrace its steps (cross an edge more than once), and that visits every vertex of $G$.

Given a weighted connected graph $G$, can we find a tour of minimal cost?

\begin{example}
We could consider an arbitrary connected graph with arbitrary edge costs.
An optimal tour on this graph is drawn below in red.
\begin{center}
\includegraphics[width=5cm]{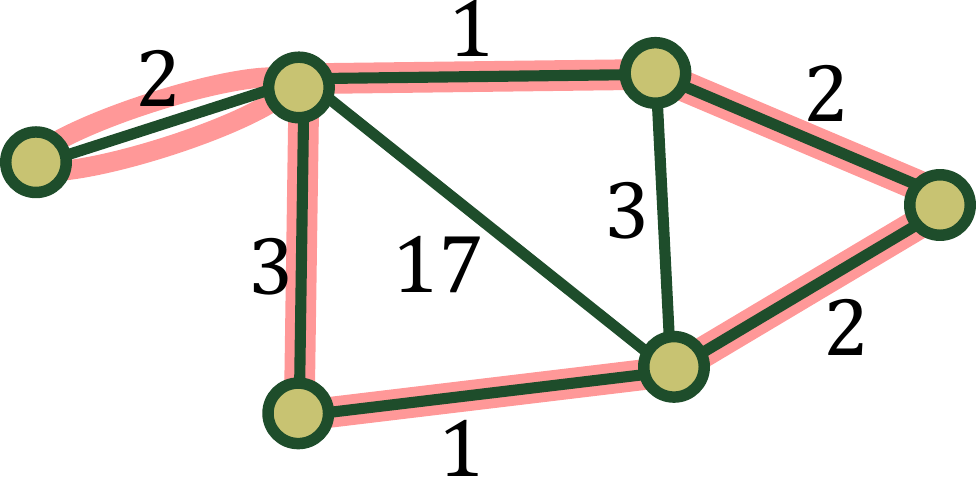}
\end{center}
Notice that an optimal tour on the above graph consists of utilizing an edge twice.
\end{example}

\begin{example}
We could consider the complete graph, where the edge costs are the Euclidean distances.
We have not drawn any of the edges in this complete graph, nor have we labeled the edge weights.
But since the edge costs are the distances on the page, you can tell that the tour drawn below is either optimal, or otherwise very close to optimal.
\begin{center}
\includegraphics[width=2.2in]{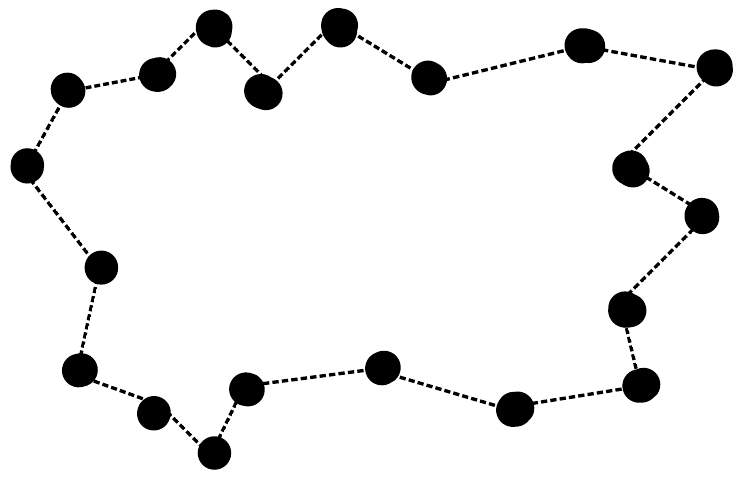}
\end{center}
\end{example}

In what contexts does the traveling salesperson problem arise?
Consider a salesperson visiting cities and then driving back home, while trying to use as little gas as possible.
Or consider a manufacturing plant --- a drill bit needs to drill holes on a metal plate and then return to its starting position before the next metal plate on the assembly line arrives.
Below we show a metal plate with six drilled holes.
The dotted lines show the tour traveled by the drill bit.
As this tour by the drill bit is repeated thousands of times per day in the manufacturing plant, it is important to make the tour as short as possible.
\begin{center}
\includegraphics[width=2.2in]{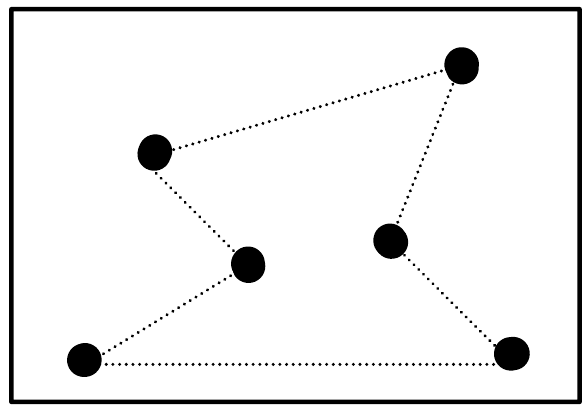}
\end{center}

In Section~\ref{sec:minimal-spanning-trees}, we saw that finding a minimal spanning tree can be solved via a greedy algorithm, namely Kruskal's algorithm.
By contrast, greedy approaches for finding an optimal traveling salesperson tour fail miserably.
Indeed, if you try to find a tour by first adding in the cheapest possible edges, you may be forced later on to choose extremely expensive edges, leading to a suboptimal solution.

In general, it is very hard to find an optimal solution to the traveling salesperson problem.
Finding an optimal tour is an \emph{NP-hard problem}, which roughly speaking means that the running time of a computer program is likely super-polynomial in terms of the number of vertices in a graph.

\subsection*{Triangle inequality}

A weighted complete graph $G$ satisfies the \defn{triangle inequality} if for any triangle (or subgraph $C_3$) in the graph $G$, the edge weights $a$, $b$, and $c$, of this triangle satisfy $a\le b+c$, $b\le a+c$, and $c\le a+b$.

\begin{center}
\includegraphics[width=1.5in]{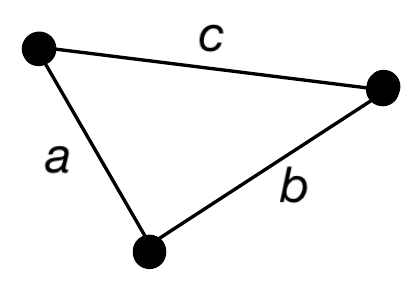}
\end{center}
In other words, the cost of traversing one side of a triangle is never more expensive than the cost of traversing the other two sides of that triangle.
Complete graphs drawn in the plane with edge weights given by Euclidean distances always satisfy the triangle inequality.

In the tree shortcut algorithm, described below, we will find a way to approximate an optimal traveling salesperson tour in a complete graph that satisfies the triangle inequality.  While this tour need not be optimal, we will show that it is optimal up to a factor of 2, so long as the weighted graph $G$ is a complete graph $K_n$ and satisfies the triangle inequality.

\vspace{3mm}
\noindent\underline{\textbf{Tree Shortcut Algorithm}}
\begin{itemize}
\item[(a)] Find a minimal spanning tree for the complete weighted graph $G$.
\item[(b)] Consider the tour that wraps around the outside of the minimal spanning tree, crossing every edge in this tree twice.
\item[(c)] When traversing the tour in (b), whenever possible skip visiting a vertex that has already been visited.
\end{itemize}
This algorithm is drawn below for the case of a complete graph on 7 vertices, with edge labels given by the Euclidean distances between the points.
The (arbitrary) starting edge we use to pass from step (b) to step (c) is indicated by the red arrow.
\begin{center}
\includegraphics[width=4.5in]{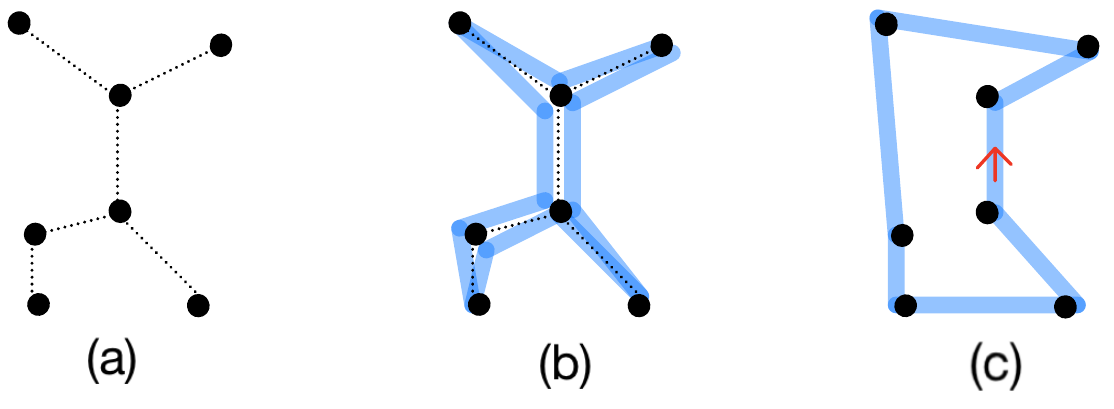}
\end{center}

We note that there is not a unique way to proceed from step (b) to step (c).
Indeed, if you start your shortcut tour by starting at a different edge (indicated by the red arrow), then you may end up with a different shortcut tour, as drawn below.

\begin{center}
\includegraphics[width=4.5in]{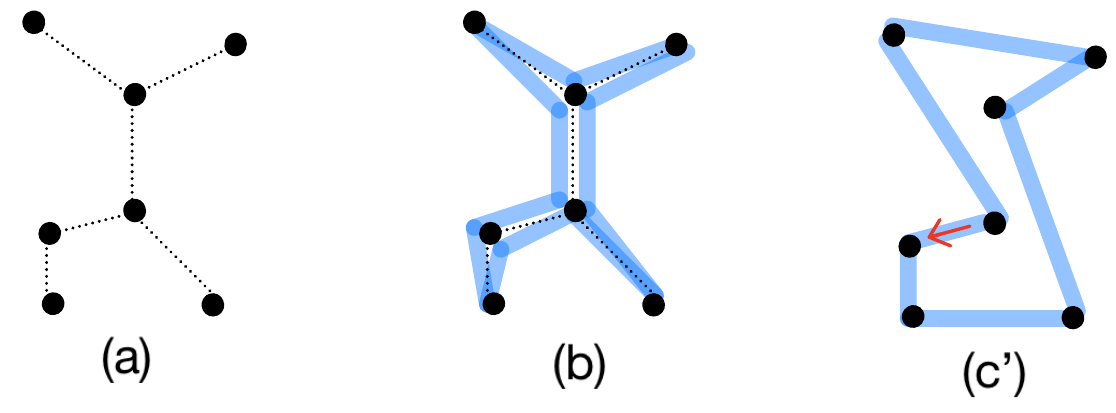}
\end{center}

\begin{theorem}
\label{thm:tree-shortcut}
If the nonnegative costs in a weighted complete graph $G$ satisfy the triangle inequality, then the Tree Shortcut Algorithm finds a tour that costs at most twice as much as an optimal tour.
\end{theorem}

The proof of this theorem will involve two lemmas: the cost of an optimal tour is at least as much as the cost of a minimal spanning tree, and twice the cost of a minimal spanning is at least as much as the cost of the output of the Tree Shortcut Algorithm.

\begin{lemma}\label{lem:tour-tree}
In a connected graph with nonnegative edge weights, the cost of an optimal tour is at least as much as the cost of a minimal spanning tree.
\end{lemma}

\begin{proof}
Let $M$ be the cost of a minimal spanning tree in $G$.
We claim
\[
\text{[cost of an optimal tour]}
\ge\ \text{[cost of an optimal tour minus any edge]} \ge M.
\]
Indeed, the first inequality follows since removing an edge (with a nonnegative cost) cannot increase the total cost.
The second inequality follows since a tour minus an edge is a spanning tree (indeed it is a spanning path), and hence must cost at least as much as $M$, the cost of a \emph{minimal} spanning tree.
\end{proof}

\begin{lemma}\label{lem:tree-shortcut-twiceMST}
If the edge costs in a weighted complete graph $G$ satisfy the triangle inequality, then the Tree Shortcut Algorithm finds a tour that costs at most twice as much as a minimal spanning tree.
\end{lemma}

\begin{proof}
Let $M$ be the cost of a minimal spanning tree in $G$.
Note that the tour produced by step (b) of the Tree Shortcut Algorithm is a tour of cost $2M$, twice the length of a minimal spanning tree.
We claim
\[
\text{[cost of output of Tree Shortcut Algorithm]} 
\le\ 2M.
\]
This follows from the triangle inequality assumption: when progressing from step (b) of the Tree Shortcut Algorithm (which is a tour of cost $2M$) to step (c), the cost of the tour can only decrease by the triangle inequality assumption.
\end{proof}

By combining these two lemmas, we obtain a proof of Theorem~\ref{thm:tree-shortcut}.

\begin{proof}[Proof of Theorem~\ref{thm:tree-shortcut}]
Let $M$ be the cost of a minimal spanning tree in $G$.
We claim
\[
\text{[cost of output of Tree Shortcut Algorithm]} 
\le\ 2M
\le\ 2\cdot\text{[cost of optimal tour]}.
\]
Indeed, the first inequality follows from Lemma~\ref{lem:tree-shortcut-twiceMST}, and the second inequality follows from Lemma~\ref{lem:tour-tree}.
\end{proof}

\begin{question}
Can you find an example to show that Theorem~\ref{thm:tree-shortcut} may fail if the edge weights don't satisfy the triangle inequality?
\end{question}

\subsection*{Exercises}

\begin{enumerate}

\item Describe a real-world example that interests you in which you might need to solve a traveling salesperson type problem.

\item Let $K_4$ be the complete graph on 4 vertices, with vertices labeled $0$, $1$, $2$, and $3$.
How many different tours of $K_4$ are there that start and end at vertex $0$, and that visit every other vertex exactly once?

\item Let $K_{3,3}$ be the complete bipartite graph with 3 vertices on each side.
Fix a single arbitrary vertex $v$ in $K_{3,3}$.
How many different tours of $K_{3,3}$ are there that start and end at vertex $v$, and that visit every other vertex exactly once?


\item Let $G$ be a weighted connected graph in which all edge weights are nonnegative.
Prove that the cost of a minimal spanning tree is less than or equal to the cost of an optimal tour (i.e., a tour solving the Traveling Salesperson Problem).

\item Let $G$ be a weighted connected graph in which all edge weights are positive. Show that the cost of a minimal spanning tree is \emph{strictly} smaller ($<$ instead of $\le$) than the cost of an optimal tour (i.e., a tour solving the Traveling Salesperson Problem).

\item Show by an example that if the triangle inequality is not true on a weighted complete graph, then a tour found by the Tree Shortcut Algorithm can be longer than 1000 times an optimal tour.

\end{enumerate}

\section{Matchings}\label{section:matchings}

\begin{videobox}
\begin{minipage}{0.1\textwidth}
\href{https://www.youtube.com/watch?v=JgWQPKVOl3o}{\includegraphics[width=1cm]{video-clipart-2.png}}
\end{minipage}
\begin{minipage}{0.8\textwidth}
Click on the icon at left or the URL below for this section's short lecture video. \\\vspace{-0.2cm} \\ \href{https://www.youtube.com/watch?v=JgWQPKVOl3o}{https://www.youtube.com/watch?v=JgWQPKVOl3o}
\end{minipage}
\end{videobox}

A \textit{matching} in a graph $G=(V,E)$ is a subset $S$ of the edge set $E$ such that no two of the edges in $S$ share a common vertex.  

For instance, two matchings of the graph $K_5$ are shown below, where the edges we chose to be in the subset are shaded.  The first matching contains only one edge and the second has two:
\begin{center}
    \includegraphics[width=.5\textwidth]{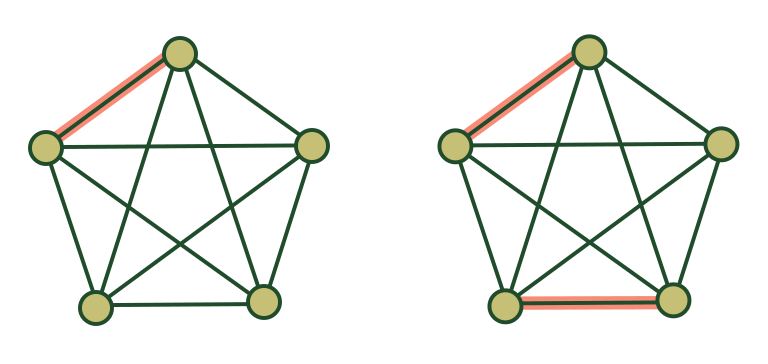}
\end{center}

An important problem in combinatorial optimization is to find 
matchings of the largest possible size. 
In other words, we would like as many vertices as possible to be an endpoint of an edge in the matching.

In the next two definitions, we will define maximal and maximum matchings.
Even though the words differ in only two letters --- maxim\emph{\textbf{al}} vs.\ maxim\emph{\textbf{um}} --- they mean different things!
As you will see, every maximum matching is also maximal, but not necessarily vice versa.

\begin{definition}
  A matching is \textbf{maximal} if it is not contained in a larger matching.
\end{definition}

For instance, the first matching of $K_5$ drawn above is contained in the second matching.  Thus it is not maximal.  However, the second drawing is maximal because there is only one vertex remaining, and connecting it to any other vertex with a shaded edge will result in two bold edges sharing a vertex.

However, it is possible to have a maximal matching that does not include as many edges as possible.  

\begin{example} \label{EP4maxmax}
Consider the path graph $P_4$ on $4$ vertices.  Here are two matchings:
\begin{center}
    \includegraphics[width=.5\textwidth]{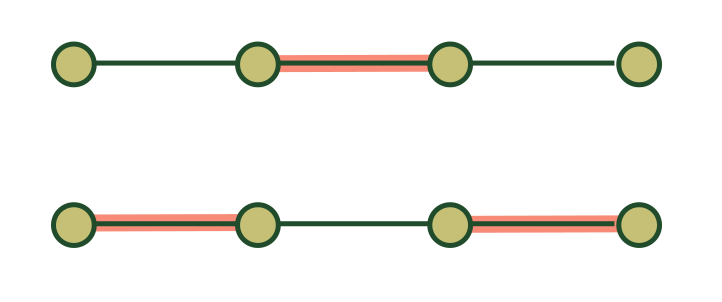}
\end{center}

Both are maximal, because we cannot add any more edges to make a larger matching.
But the second one includes more edges than the first one.
\end{example}

\begin{definition}
  A \textbf{maximum} matching is a matching with the largest possible number of edges.
  A \textbf{perfect} matching is a matching such that every vertex of $V$ 
  is an endpoint of (exactly one) edge in $S$.
\end{definition}

In Example~\ref{EP4maxmax}, the first matching is maximal but not maximum and the second matching is maximal, maximum and perfect.

In $K_5$, every maximal matching has size $2$, and this is the largest possible size of a matching, so it is also a maximum matching.  Exactly four vertices of $K_5$ are contained in a maximum matching of $K_5$, so the maximum matchings are not perfect. 

One algorithm for finding a maximal matching is the \textbf{greedy algorithm}.  Label the vertices $1,2,3,\ldots,n$ and choose the edge $(i,j)$ where $i$ is minimal and $j$ is smallest among all edges from $i$.
Then consider the subgraph formed by deleting vertices $i$ and $j$, and again choose the minimal vertex and edge out of it, and so on until no more edges can be added to make a larger matching.  By definition, this process results in a maximal matching.

However, finding a \textit{maximum} matching is in general much harder.  The above greedy algorithm will not necessarily result in a maximum matching.  The \textit{Hopcroft-Karp} algorithm is a general algorithm for finding a maximum matching in a graph, and while we will not be describing the full algorithm here, we give a hint of the strategy using the key idea of \textit{augmenting paths}.

The idea is to start with a matching that is not maximum and replace it with a different matching that contains one more edge.

\begin{definition}
Let $G$ be a graph and let $M$ be a matching of $G$.
Suppose that $G$ has (at least two) 
vertices $v$ and $w$
that are not endpoints of edges in $M$.
Given this data, an \textit{augmenting path} is a path that starts at $v$ and ends at $w$ that contains an odd number $m$ of distinct edges $e_1,e_2,\ldots,e_m$, whose endpoints are all distinct, and 
such that the edges alternate between 
being in $M$ and not in $M$:
$$e_1\not\in M, e_2\in M, e_3\not\in M, e_4\in M,\ldots,e_m\not\in M.$$  
\end{definition}

\begin{example} \label{Eaugment}
Consider the complete bipartite graph $B_{4,4}$, where the vertices on the left are labeled $A,B,C,D$ and the vertices on the right are labeled $1,2,3,4$.
Let $G$ be the graph obtained from this by removing the edge $\{4,A\}$.
We start with the matching $M$ of $G$ consisting of the edges $\{1,B\}$, $\{2,C\}$, and $\{3,D\}$.
This is a maximal matching, because  the edge $\{4,A\}$ is missing so there 
is no way to increase the number of matched vertices.
On the other hand, it is not a maximum matching.

Here is an example of an augmenting path with edges: \[\{A,1\}, \ \{1,B\}, \ \{B,2\}, \ \{2,C\}, \ \{C,3\}, \ \{3,D\}, \ \{D,4\}.\] 
We illustrate it below, with the black edges being the edges in the path and the red overlined edges being the ones in the matching $M$.
\begin{center}
\includegraphics[width=2in]{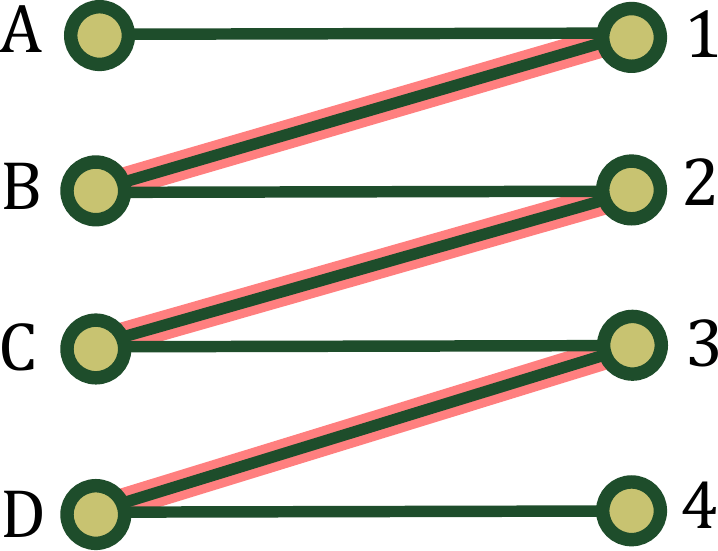}
\end{center}
\end{example}

Suppose $P$ is an augmenting path of a matching $M$.  The set $P-M$ contains all the edges of $P$ that are not in $M$.  We can see that $P-M$ is also a matching of $G$ and the number of edges in $P-M$ is one larger than the number of edges in $M$. 

For instance, in Example~\ref{Eaugment}, $P-M$ contains the edges 
\[\{A,1\}, \ \{B,2\}, \ \{C,3\},
\ \{D,4\},\]
so $P-M$ is a matching of size $4$, which is a maximum matching (and a perfect matching).

\begin{example}
Let $M$ be the set of red edges in the graph below.
Find an augmenting path for $M$ and 
use it to find a larger matching than $M$?  
\begin{center}
\includegraphics[width=2in]{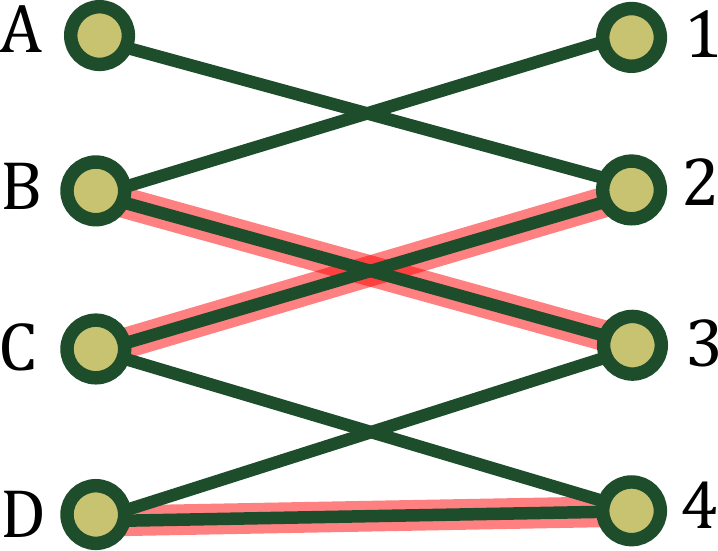}
\end{center}
\end{example}

\subsection{Matchings of bipartite graphs}

In this section, we describe \textit{Hall's Marriage Theorem}
for matchings in bipartite graphs. 
To illustrate it, we start with an example.

\begin{example}
Suppose new CSU students $a$, $b$, $c$,
and $d$ all went to the 
 same high school, and so they would rather not be in the same dorm as each other so that they get more opportunities to meet new people.  

They fill out their dorm preferences form, and unfortunately have quite similar preferences for dorms!  Student $a$ checks only Parmelee, $b$ only checks Braiden, $c$ checks Academic Village, Braiden, and Corbett, and $d$ checks Braiden and Parmelee.  Is there a way to match each of the students to a dorm so that none of the four students are in the same dorm as each other?

To answer this, we can model it as a bipartite graph:

\begin{center}
    \includegraphics[width=.5\textwidth]{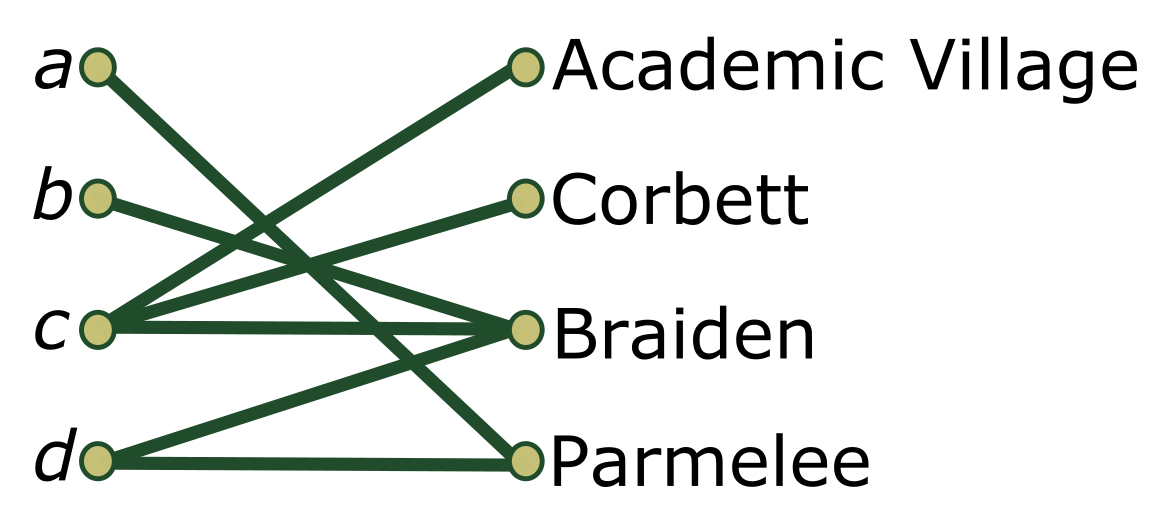}
\end{center}

If we focus on students $a,b,d$, we notice that there is no way to do it; $a$ needs to be in Parmelee to have their preference satisfied, $b$ needs to be in Braiden, and $d$ needs to be in one or the other, forcing it to overlap with $a$ or $b$.  
\end{example}

\begin{example}
Realizing the issue in the above example, the four students decide to change their preferences form so that each checks exactly two preferred dorms.
Student $a$ checks Academic Village and Parmelee, $b$ checks Braiden and Corbett, $c$ checks Academic Village and Corbett, and $d$ checks Braiden and Parmelee.  Is there a way to match them to distinct dorms now? 

\begin{center}
    \includegraphics[width=.5\textwidth]{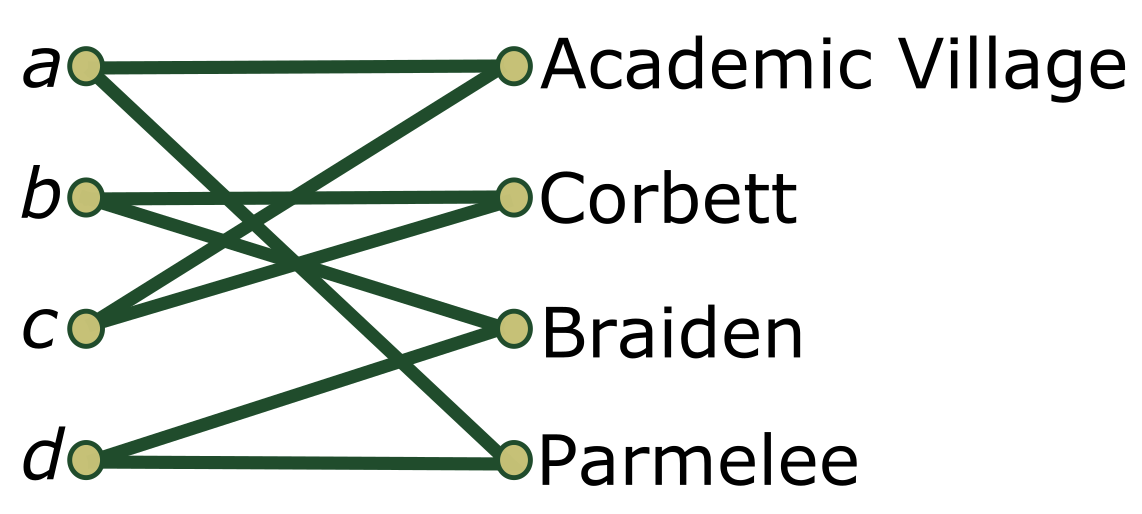}
\end{center}

One way to do it is to match $a$ to Academic Village, $b$ to Braiden, $c$ to Corbett, and $d$ to Parmelee.  But there are other possibilities too; do you see another?
\end{example}

There are many situations in which is it important to find suitable matchings --- for example when matching medical residents to desired medical residencies, when matching pets to suitable homes, and when matching prospective employees to appropriate jobs.

Notice that in the two examples above, we were trying to find a matching in a bipartite graph for which every element on one side of the graph is matched with an element on the other side.  We call such matchings \textit{saturated}, and any saturated matching are necessarily a maximum matching.

\begin{definition}
  Let $G$ be a bipartite graph with left set $L$, right set $R$, and $|L|\le |R|$.  Then any matching of $G$ that matches every element of $L$ with some element of $R$ is called a \textbf{saturated} matching.
\end{definition}

Hall's Marriage Theorem tells us exactly when a saturated matching exists in a bipartite graph.  To state it, we need the concept of the \textit{neighborhood} of a set of vertices in a graph.

\begin{definition}
 Let $G=(V,E)$ be a graph.   The \textbf{neighborhood} of a subset $S\subseteq V$ of the vertices, denoted $N_G(S)$, is the set of all vertices $v$ such that $v$ is connected by an edge to some element of $S$.
\end{definition}

For example, when $S=\{v\}$ contains only one vertex $v$, then the neighborhood of $S$ is the set of edges adjacent to $v$ and the size of the neighborhood is the degree of $v$.

\begin{theorem}
(Hall's Marriage Theorem) Let $G$ be a bipartite graph with left set $L$, right set $R$, and $|L|\le |R|$.
Then a saturated matching exists if and only if, for every subset $X\subseteq L$, $$|X|\le |N_G(X)|.$$
\end{theorem}

The proof of Hall's Marriage Theorem is rather complicated and we will not provide the full details here.  But as intuition for one direction of the statement, suppose there is a subset $X$ of the left set whose neighborhood $N_G(X)$ is strictly smaller than $X$.  Then the points in $X$ cannot all be matched to distinct points in $R$ by the pigeonhole principle.  So the inequality $|X|\le |N_G(X)|$ is necessary for a saturated matching to exist.  The interesting fact is that this condition is also sufficient.

\begin{example}
In the first examples of students and dormitories above, can you find a subset $X$ of the students for which the inequality $|X|\le |N_G(X)|$ is violated?  In the second example, can you show that the inequalities are all satisfied?
\end{example}

\subsection*{Exercises}

\begin{enumerate}
    \item The size of a maximum matching on the complete graph $K_6$?							
\item What is a maximum matching on the complete graph $K_7$?							
\item What is the size of a maximum matching on the path graph $P_6$?							
\item What is the smallest size of a maximal matching on the path graph $P_6$?							
\item What is the size of a maximum matching on the complete bipartite graph $B_{20,21}$?	
\item If $G$ has a perfect matching, explain why the number of vertices of $G$ is even.

\item In the path graph $P_5$, show that every maximal matching is maximum (but not perfect).
\item In the cycle graph $C_5$, show that every maximal matching is maximum (but not perfect).
\item For the path graphs 
$P_6, \ldots, P_{11}$,
find the size of a maximum matching and the smallest possible size of a maximal matching.

\item For the cycle graphs 
$C_6, \ldots, C_{11}$,
find the size of a maximum matching and the smallest possible size of a maximal matching.

    \item Find the size of a maximum matching for the path graph $P_n$.
    Your answer should depend on whether $n$ is even or odd.
    
    \item What is the smallest possible size of a maximal matching for $P_{n}$?
    Your answer should depend on the congruence of $n$ modulo $3$.
    \item Find the size of a maximum matching for the cycle graph $C_n$.
    Your answer should depend on whether $n$ is even or odd.
    
    \item What is the smallest possible size of a maximal matching for $C_{n}$?
    Your answer should depend on the congruence of $n$ modulo $3$.

    \item Does the below bipartite graph have a saturated matching?
    \begin{center}
    \includegraphics[width=1.2in]{08-GraphsWalksCycles/graph-bipartite.pdf}
    \end{center}
    Each node on the left represents a pet, each node on the right represents a home, and tehre is an edge when that home would be suitable for the pet.
    We would like to know if each pet can find a different home.
    
    \item Does the below bipartite graph have a saturated matching?
    \begin{center}
    \includegraphics[width=1.2in]{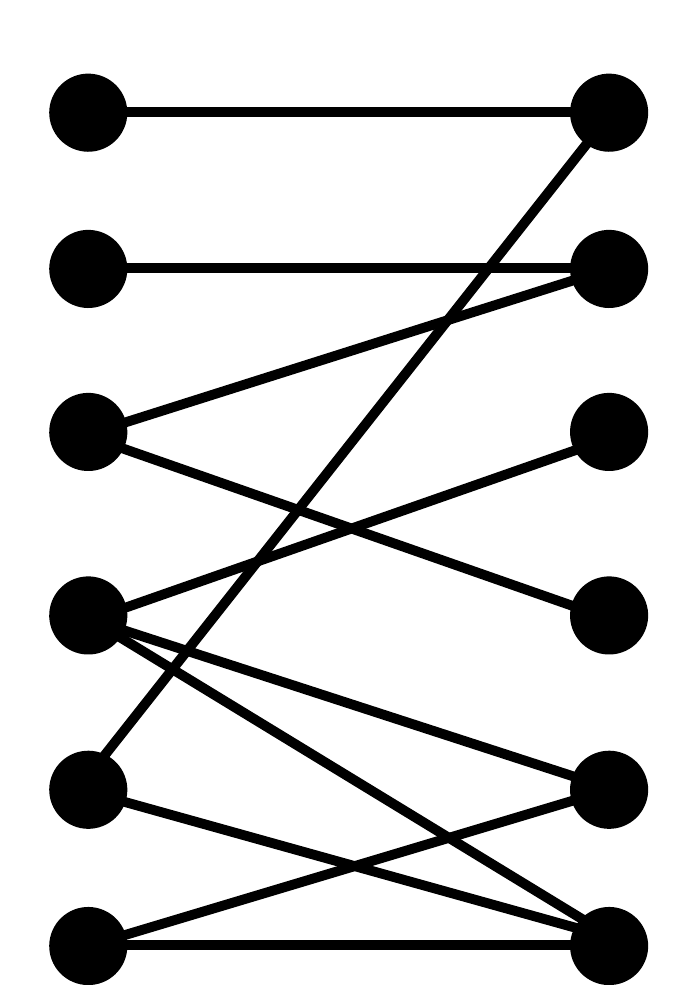}
    \end{center}
    Each node on the left represents a person, each node on the right represents a job, and there is an edge when a job would be suitable for that person.
    We would like to know if each of the jobs can be simultaneously filled, with each person having exactly one job.
    
    \item Does the below bipartite graph have a saturated matching?
    \begin{center}
    \includegraphics[width=1.2in]{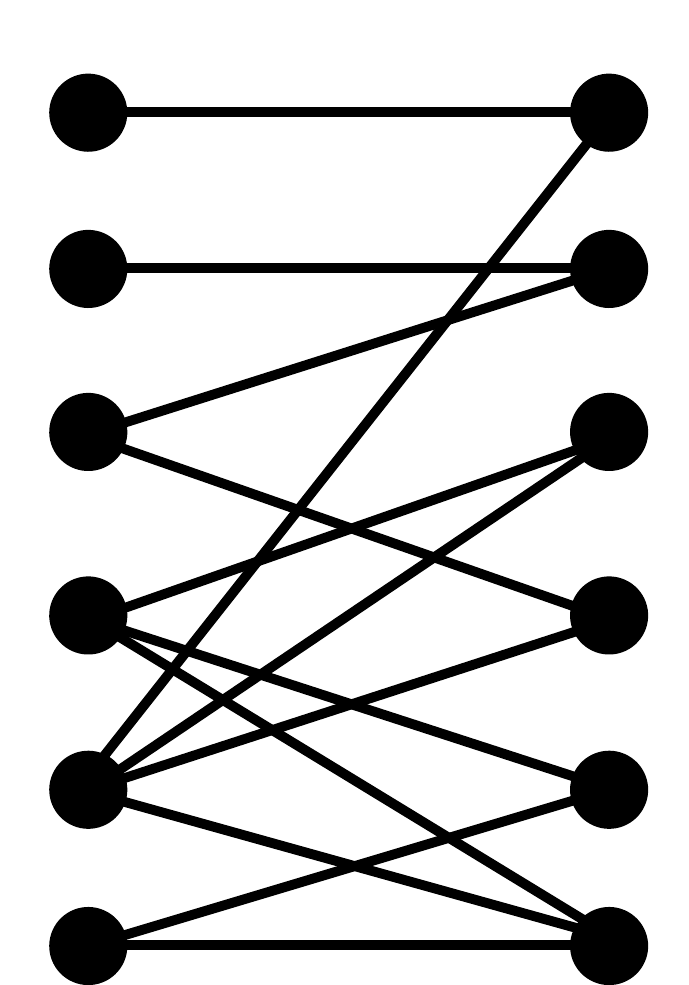}
    \end{center}
    Each node on the left represents a person, each node on the right represents a job, and there is an edge when a job would be suitable for that person.
    We would like to know if each of the jobs can be simultaneously filled, with each person having exactly one job.

    \item 
    Fact: in a complete bipartite graph $B_{n,m}$ with $n\le m$, a maximum matching has size $n$.
    Explain how that fact follows from Hall's Marriage Theorem.
    
    \item Let $G$ be a bipartite graph with the sets of left and right vertices having equal size $n$ and every vertex having the same degree $k\ge 1$. 
    Use Hall's Marriage Theorem to show that $G$ has a perfect matching.

\end{enumerate}


\section{Ramsey theory}

We conclude this chapter with a famous topic in graph optimization known as \textit{Ramsey theory}.  We start with a  classical example.

\begin{example}\label{ex:6-people}
  Six people walk into a room; some of the people know each other already and others do not.  Prove that either some three of them all mutually know each other, or some three of them all mutually do not know each other.
\end{example}

To get a handle on the above problem, we can first model it as a graph theory problem; consider a graph on $6$ vertices that represent the people, and where each pair of people is connected either by a solid or dashed edge according to whether they know each other or not, as shown:

\begin{center}
    \includegraphics{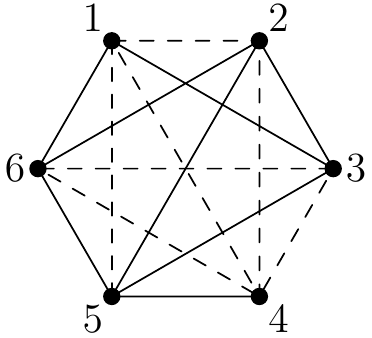}
\end{center}

In the above graph, notice that the vertices numbered $3,4,6$ all mutually do not know each other.  Their induced subgraph is a smaller complete graph $K_3$, with all its edges the same ``color'' (where the two colors are solid and dashed in this case, but any two colors would work).  
To state our problem more precisely in terms of graphs, we make the following two definitions.

\begin{definition}
  An \textbf{$m$-clique} in a graph is a subgraph isomorphic to $K_m$.  
\end{definition}

\begin{definition}
  Suppose a graph $G$ has every edge assigned a color.  A \textbf{monochromatic $m$-clique} in $G$ is an $m$-clique all of whose edges are the same color.
\end{definition}

For example, the graph above has a monochromatic $3$-clique with vertices $3,4,6$, where the edges between them are all dashed.

\begin{question}
How many monochromatic $3$-cliques can you find in the graph above?
\end{question}

We can now rephrase Example \ref{ex:6-people} in terms of graph theory:

\begin{example}
\label{Ek6ramsey}
Prove that if each edge of $K_6$ is colored with one of two colors,
then there exists a monochromatic $3$-clique.
\end{example}

\begin{proof}[Proof of Example~\ref{Ek6ramsey}]
To prove this, pick one of the vertices $x$, and consider all $5$ edges attached to $x$.  Since each of these $5$ edges is either solid or dashed, by the Pigeonhole Principle, at least three edges share the same color.  By symmetry, we can assume that this color is solid, so $x$ is connected to at least three vertices, say $y,z,w$, by solid edges.

\begin{center}
    \includegraphics{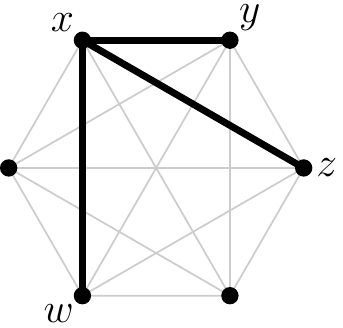}
\end{center}

Now, consider the edges in the triangle formed by $y,z,w$.
If any of these edges is solid, they form a solid triangle with $x$, and we have a monochromatic (solid) $3$-clique, as in the left first example below.
But if none of the three edges is solid, then the triangle $y,z,w$ is a monochromatic (dashed) $3$-clique, as in the right example below.
\begin{center}
    \includegraphics{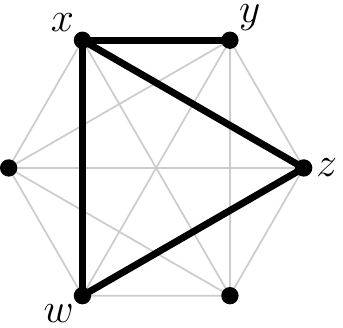} \hspace{2cm}\includegraphics{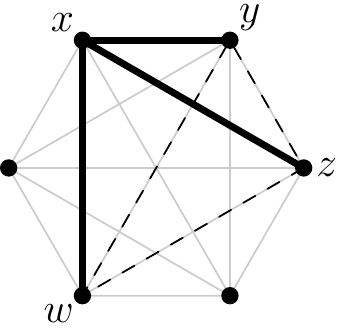}
\end{center}

So in all cases we do indeed have a monochromatic $3$-clique.
This completes the proof.
\end{proof}

We have shown that in a group of $6$ people, either some $3$ people know each other or some $3$ people all do not know each other.
Interestingly, $6$ people is the minimum number of people required for this fact to be true; it is possible to gather 5 people in a room such that no three all know each other and no three all do not know each other.  (See Exercise~\ref{ex:K5no3clique} below.)

\subsection{Investigation: Ramsey numbers}

The example above can be generalized as follows.  If we color the edges of the complete graph $K_n$ with two colors, is there guaranteed to be a monochromatic $m$-clique, for various values of $n$ and $m$?
And, for a given $m$, what is the smallest value of $n$ for which this is true?

We can further generalize this question using the notion of a \textit{Ramsey number}.

\begin{definition}
  The \textbf{Ramsey number} $R(m_1,m_2)$ is the smallest number $n$ such that any coloring of the edges of the complete graph $K_n$ using $2$ colors $c_1$ and $c_2$ must have one of the following:
  \begin{itemize}
      \item A monochromatic $m_1$-clique of color $c_1$, or
      \item A monochromatic $m_2$-clique of color $c_2$.
  \end{itemize}
\end{definition}

\begin{example}
  The number $R(3,3)$ is the smallest number $n$ for which any coloring of the edges of $K_n$ in two colors contains a monochromatic $3$-clique (of one color or the other).
  We found in Example \ref{Ek6ramsey} and Exercise~\ref{ex:K5no3clique} (below) that the smallest such $n$ is $6$, so $R(3,3)=6$.
\end{example}

\begin{example}
  Let's compute $R(2,2)$, which is the smallest number $n$ such that any coloring of the edges of $K_n$ in two colors contains a monochromatic $2$-clique (of one color or the other).  A $2$-clique is just an edge, so we only need one edge of any color, and so $n=2$ is the smallest possible number that satisfies the condition.  Therefore $R(2,2)=2$.
\end{example}

We can define Ramsey numbers not only for 2 colors, but also more generally for $k$ colors.

\begin{definition}
  The \textbf{Ramsey number} $R(m_1,m_2,\ldots,m_k)$ is the smallest number $n$ such that any coloring of the edges of the complete graph $K_n$ using $k$ colors $c_1,c_2,\ldots,c_k$ must have one of the following:
  \begin{itemize}
      \item A monochromatic $m_1$-clique of color $c_1$, or
      \item A monochromatic $m_2$-clique of color $c_2$, or
      \item $\cdots$
      \item A monochromatic $m_k$-clique of color $c_k$.
  \end{itemize}
\end{definition}

\begin{example}
  Let's compute $R(3,3,2)$.  This is the smallest number $n$ such that any coloring of the edges of $K_n$ using colors red, green, and blue has either a red $3$-clique, a green $3$-clique, or a blue $2$-clique.  
  
  First note that since there is an edge-colored graph $G$ on $5$ vertices using only red and green colors that does not have a red triangle or a green triangle (see Exercise 1 below), this graph $G$ does not have a red $3$-clique, a green $3$-clique, or a blue $2$-clique.  So $R(3,3,2)>5$.
  
  We can now show that $R(3,3,2)=6$.  First, if there is any blue edge in a coloring of $K_6$, that is a blue $2$-clique.  On the other hand, if there are no blue edges, then we have a coloring of the edges either red or green, and we know such a coloring either had a red $3$-clique or green $3$-clique since $R(3,3)=6$.  Therefore any coloring either has a red $3$-clique, a green $3$-clique, or a blue $2$-clique, and so $R(3,3,2)=6$.
\end{example}

\textbf{Ramsey's theorem} says that the Ramsey number $R(m_1,m_2,\ldots,m_k)$ is always finite; that is, there does exist some sufficiently large $n$ for which the statement is true.  But what is the minimum such number $n=R(m_1,m_2,\ldots,m_k)$?  This is in general an unsolved problem, and the study of computing the Ramsey numbers is called \textit{Ramsey theory}, and is one of a large class of optimization problems in graph theory involving coloring graphs.

\subsection*{Exercises}

\begin{enumerate}
    \item
    \label{ex:K5no3clique}
    Color the edges of $K_5$ in two colors in a way that avoids creating a monochromatic $3$-clique.
    
    \item Let $G$ be the complete graph $K_5$ on $5$ labeled vertices $1,2,3,4,5$.  How many different ways can you color each edge of $G$ either red or blue such that there exists a monochromatic $4$-clique in the graph?
    
    \item Let $H$ be the complete graph $K_6$ on $6$ labeled vertices $1,2,3,4,5,6$.  How many different $4$-cliques are subgraphs of $G$?
    
    \item Show that $R(n)=n$ for all positive integers $n$.
    
    \item Show that $R(2,n)=n$ for all positive integers $n$.
    
    \item Compute $R(2,2,2)$.
    
    \item Explain why $R(a,b,2)=R(a,b)$ for any positive integers $a$ and $b$.
    
    \item Prove that $R(a,b)=R(b,a)$ for any positive integers $a$ and $b$.
    
    \item Prove that $R(3,4)=9$.
    
    \item Look up online the value of  $R(4,4)$, and explain what this value means.
    \item 
    Look up online what is known about $R(5,5)$, and explain what this means.
    \item It is known that $R(3,3,3)=17$.  Investigate this as follows. Draw a complete graph on $6$ vertices with each edge colored either red, green, or blue that has no monochormatic $3$-clique of any color.  This shows that $R(3,3,3)>6$.  
    
    Then do the same for $7$ vertices that shows that $R(3,3,3)>7$, and so on.  Can you show that $R(3,3,3)>16$?
\end{enumerate}

%% file: 11-PlanarGraphsEuler/PlanarGraphs.tex
\chapter{Planar Graphs}\label{chap:planar}

A graph is \emph{planar} if it can be drawn on a piece of paper with no edges crossing.  In this chapter we will explore the many particularly nice properties of planar graphs and their applications.

\section{Planar graphs}

\begin{videobox}
\begin{minipage}{0.1\textwidth}
\href{https://youtu.be/vbNBpvgeMdY}{\includegraphics[width=1cm]{video-clipart-2.png}}
\end{minipage}
\begin{minipage}{0.8\textwidth}
Click on the icon at left or the URL below for this chapter's  lecture video. \\\vspace{-0.2cm} \\ \href{https://youtu.be/vbNBpvgeMdY}{https://youtu.be/vbNBpvgeMdY}
\end{minipage}
\end{videobox}

Recall that a graph is a set of vertices and a set of edges.
There are many ways to set up the data for a graph, for example an incidence matrix or an edge list.
In this section, we focus on the more visual way of describing a graph, by drawing it on paper.
It is natural to draw a graph as 
simply and cleanly as possible.

For example, here are two ways to draw the graph $K_4$.
\begin{center}
\includegraphics[width=0.8in]{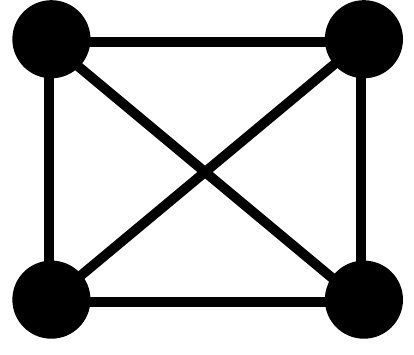}
\hspace{20mm}
\includegraphics[width=1.1in]{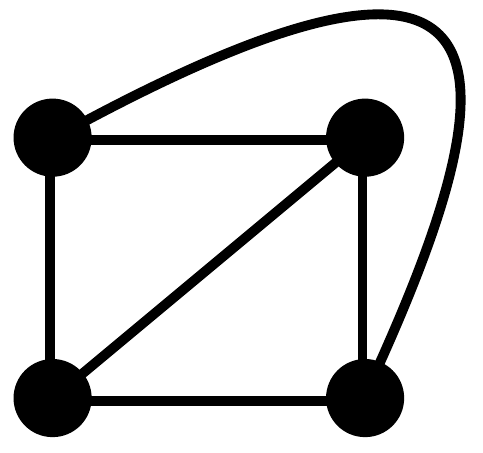}
\end{center}
The graph on the right has the advantage that none of its edges intersect.

\begin{definition}
A \defn{planar graph} is a connected graph drawn in the plane so that its edges do not cross.
\end{definition}

The pictures above show 
$K_4$ drawn as a non-planar graph and as a planar graph.
In other words, the property of being a planar graph depends not only on the graph but also on how it is drawn in the plane. 
We say that a connected graph is \defn{planar} if there is some way of drawing it in the plane as a planar graph.

\begin{example}
Is the below graph planar?
\begin{center}
\includegraphics[width=1in]{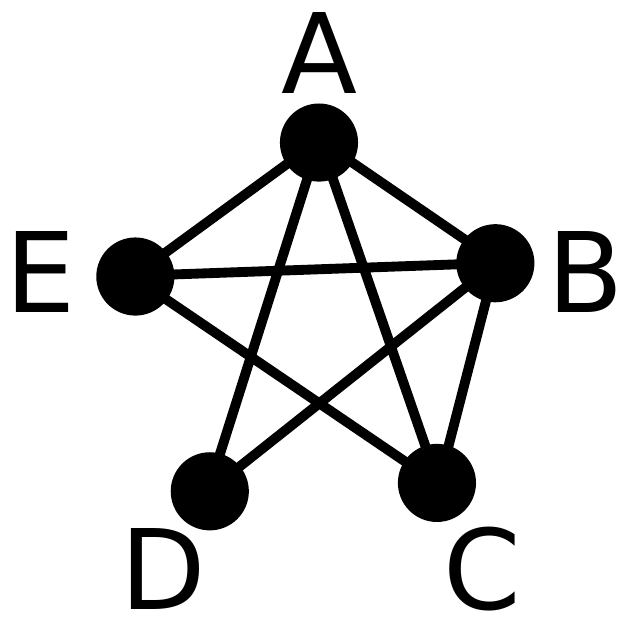}
\end{center}
\end{example}

\begin{answer}
Yes.
Two possible planar maps are drawn below.
\begin{center}
\includegraphics[width=3in]{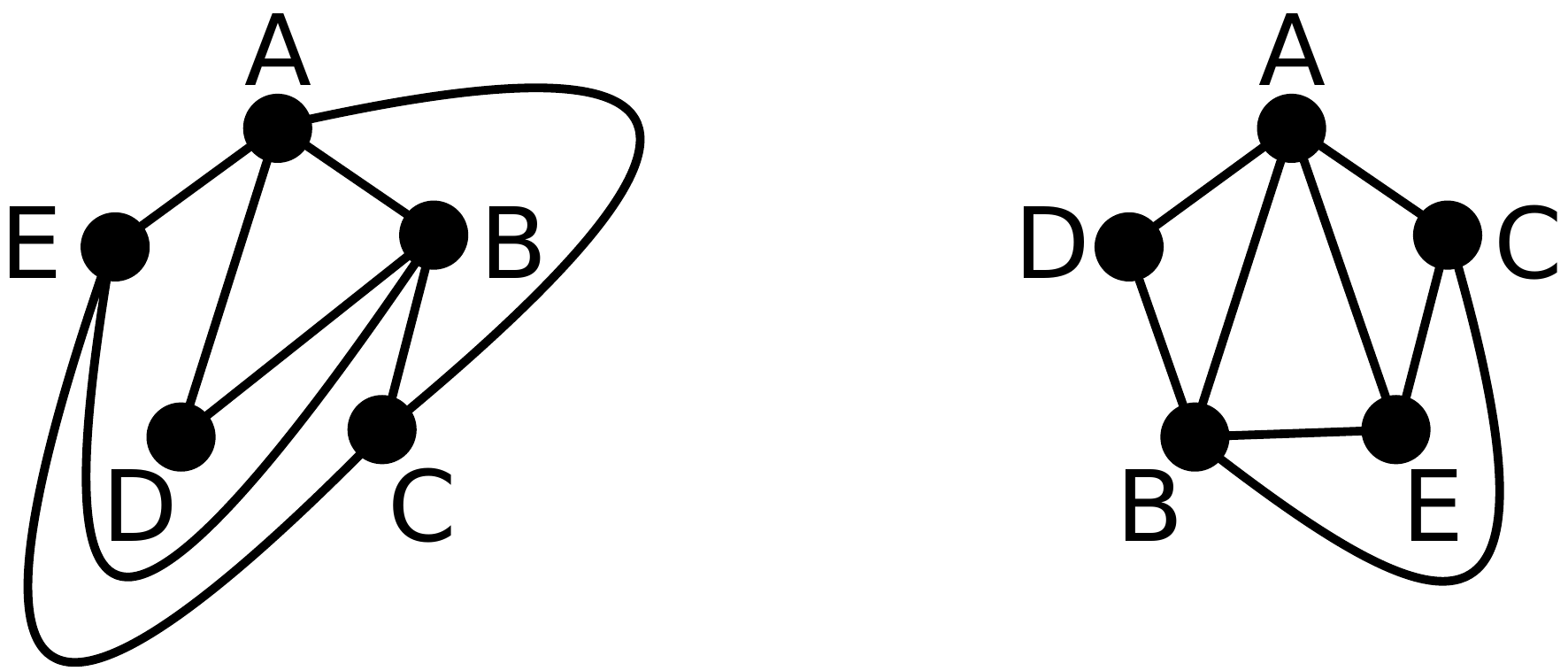}
\end{center}
\end{answer}

\begin{example}
Is the below graph planar? (Adding another edge, CD, only makes this harder.) 
\begin{center}
\includegraphics[width=1in]{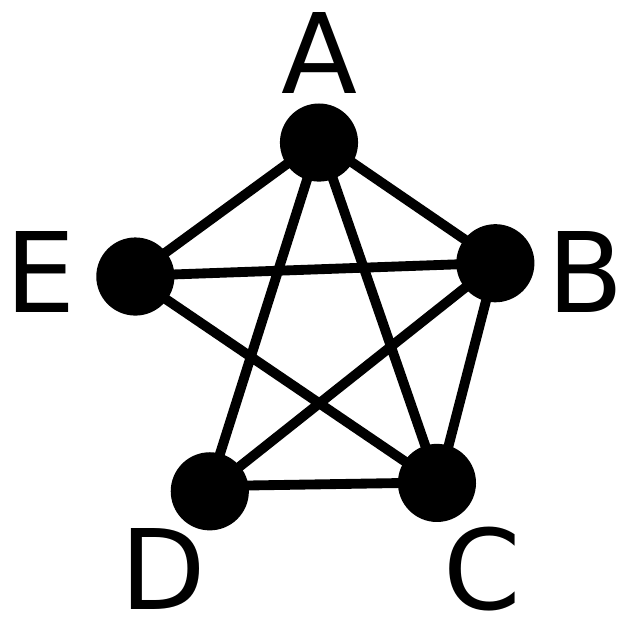}
\end{center}
\end{example}

\begin{answer}
Yes.
Two possible planar maps are drawn below.
\begin{center}
\includegraphics[height=1in]
{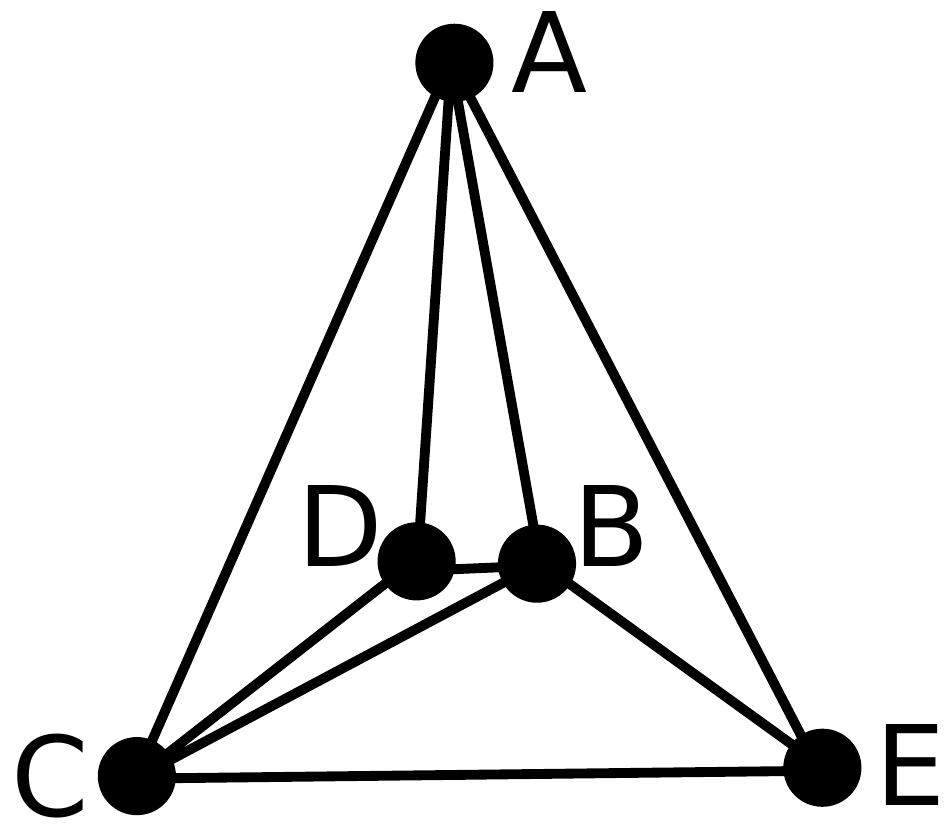}
\hspace{30mm}
\includegraphics[height=1in]
{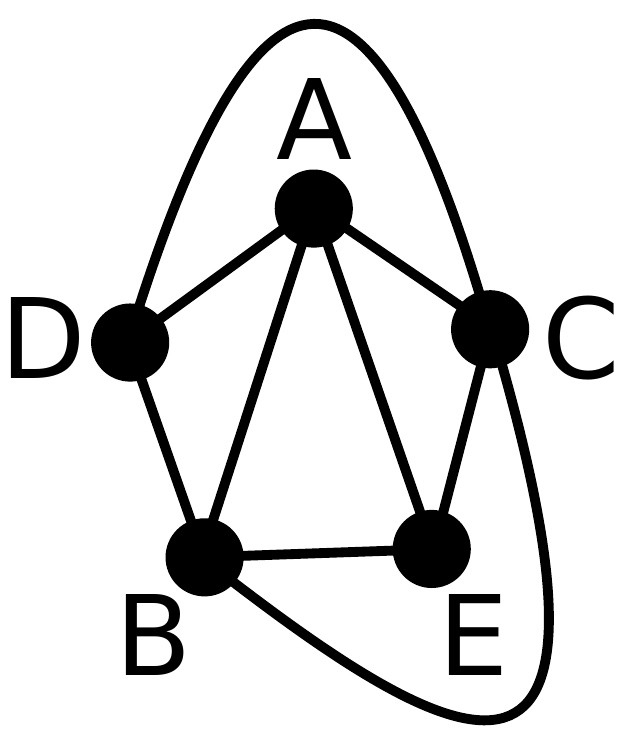}
\end{center}
\end{answer}

\begin{remark}
In designing computer chips, there is an advantage to having a small number of crossings of electrical wires.  This translates into a question on graph theory about
minimizing the number of edge crossings in a graph drawn in the plane.
\end{remark}

In Section~\ref{Seulerplanar}, we study properties of planar graphs. 
Then in later sections of this chapter, we use these properties 
to tackle the problem of deciding whether a graph is a planar graph.

\begin{remark}
\emph{F\'{a}ry's Theorem} states that if a graph is planar, then it can be drawn as a planar graph in which every edge is straight.
See \url{www.jasondavies.com/planarity} for a fun game in which you drag the vertices of a (planar) graph around until you obtain a planar embedding (with straight edges)!
\end{remark}

\begin{remark}
It is not known if every planar graph can be drawn in the plane so that every edge is straight
with length being equal to an integer.
\end{remark}

\subsection*{Exercises}
\begin{enumerate}
\item Draw a 5-pointed star and then redraw it as a planar graph.
\item Draw $B_{2,3}$ as a planar graph.
\item Draw $B_{2,4}$ as a planar graph.
\item Draw $B_{3,3}$ and then remove one of its edges.  Redraw this graph as a planar graph.

\end{enumerate}

\section{Euler's formula for planar graphs} \label{Seulerplanar}

\begin{videobox}
\begin{minipage}{0.1\textwidth}
\href{https://www.youtube.com/watch?v=GJJOIKn-1PA}{\includegraphics[width=1cm]{video-clipart-2.png}}
\end{minipage}
\begin{minipage}{0.8\textwidth}
Click on the icon at left or the URL below for this chapter's  lecture video. \\\vspace{-0.2cm} \\ \href{https://www.youtube.com/watch?v=GJJOIKn-1PA}{https://www.youtube.com/watch?v=GJJOIKn-1PA}
\end{minipage}
\end{videobox}

Euler's formula gives a remarkable connection between the edges, vertices, and regions (or \textit{faces}) of a planar graph.  Let's make these ideas precise.

\begin{definition}
In a planar graph, let $v$ be the number of vertices and let $e$ be the number of edges.
The edges divide the plane into regions, which are called \defn{faces}.  Let $f$ be the number of faces, including the face(s) which touch the edge of the paper.
\end{definition}

\begin{center}
\includegraphics[width=4in]{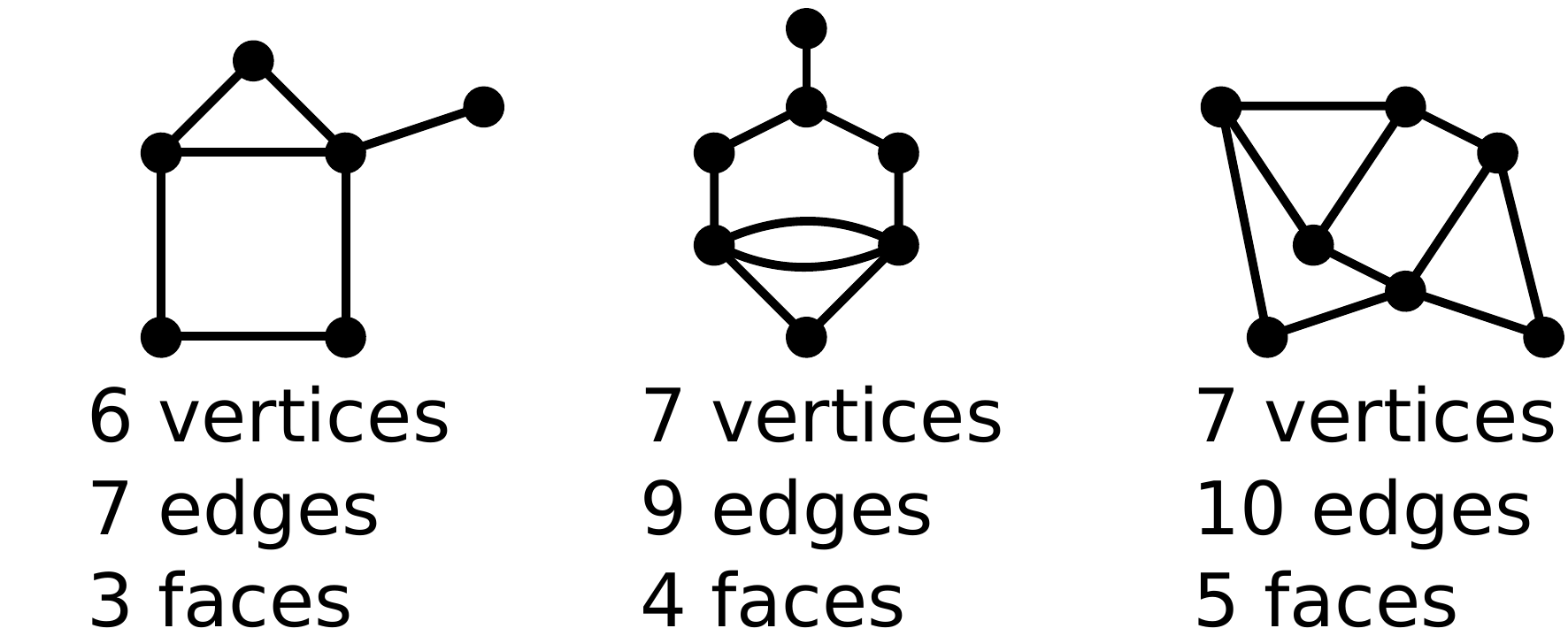}
\end{center}

\begin{definition}
Define $\chi=v-e+f$; we call this quantity the \defn{Euler characteristic} of the graph.
\end{definition}

\begin{example}
For the planar graph $K_4$,
$v=4$ and $e=6$ and $f=4$.
So $\chi =4-6+4 =2$.
\end{example}

\begin{example}
For the graph $C_n$, drawn as a regular $n$-gon in the plane, $v=n$, $e=n$, and $f=2$; (the two faces are the inside and the outside regions).  So $\chi=n-n+2=2$.
\end{example}

\begin{example}
For the graph $P_n$, drawn as a path in the plane, $v=n$, $e=n-1$, and 
$f=1$ (the path makes a slit in the one face).  So $\chi=n-(n-1)+1=2$.
\end{example}

\begin{example}
Find a way to draw the complete bipartite graph $B_{2,3}$ as a planar graph.
Show $v=5$, $e=6$, and $f=3$.
Show $\chi = 2$.
\end{example}

At this point, you may be wondering why $\chi=2$ in all these examples.
The reason is that the Euler characteristic is $2$ for \emph{every} planar graph.

\begin{theorem}[Euler's formula]
\label{Teuler} 
In a connected planar graph, let $v$ be the number of vertices, let $e$ be the number of edges, and let $f$ be the number of faces.
Then $\chi = v-e+f=2$.
\end{theorem}

The proof uses this next lemma about cycles in graphs.

\begin{lemma} \label{Lremoveedgecycle}
Let $G$ be a connected graph which is not a tree.
If $e$ is an edge that belongs to a cycle in $G$, then
the subgraph $G'=(G\text{ with edge }e\text{ removed})$ is still connected.
\end{lemma}

\begin{proof} $ $
\begin{center}
\includegraphics[width=4in]{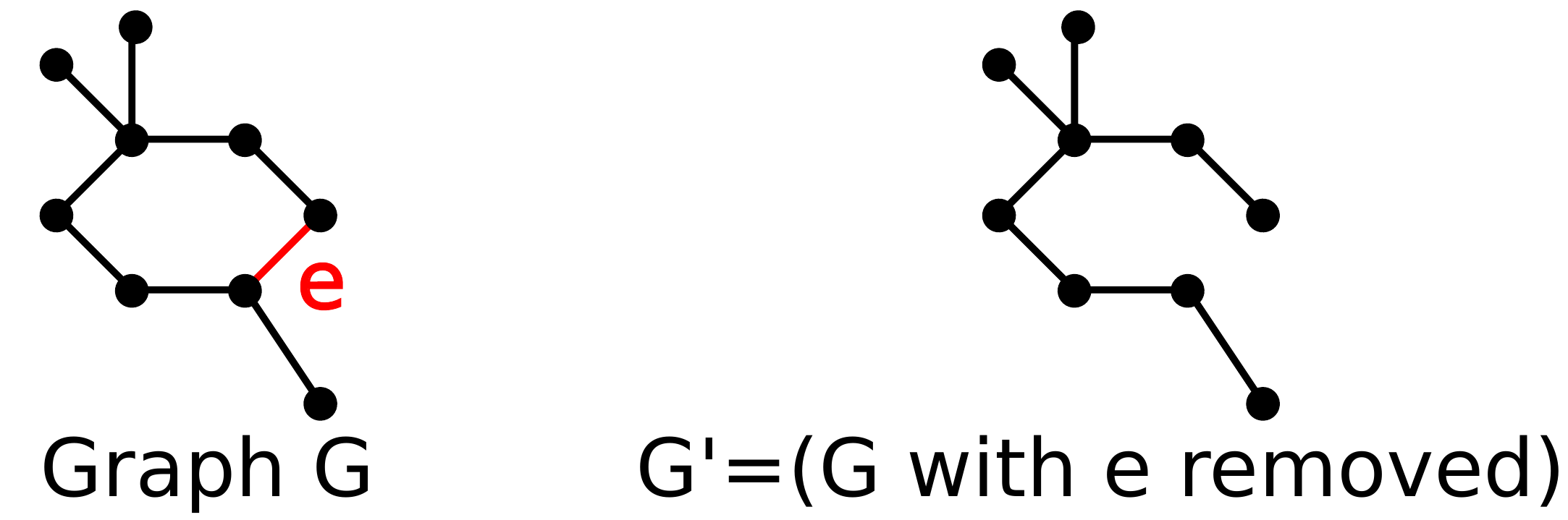}
\end{center}
The vertices of $G$ and $G'$ are the same.  The edges in $G$ are all the edges of $G'$ together with the edge $e$.
Let $u$ and $w$ be arbitrary vertices in $G'$.  To show that $G'$ is connected, we must find a walk in $G'$ from $u$ to $w$. Since $G$ is connected, there is a walk in $G$ from $u$ to $w$.
\begin{itemize}
\item If this walk does not use the edge $e$, then it is a walk in $G'$ and we are done!
\item If this walk does use $e$, then there is a second walk in $G$ from $u$ to $w$ which ``goes the other way around the cycle".  This second walk is also a walk in $G'$.
\end{itemize}
\end{proof}
The following illustrates the proof method above:
\begin{center}
\includegraphics[width=2in]{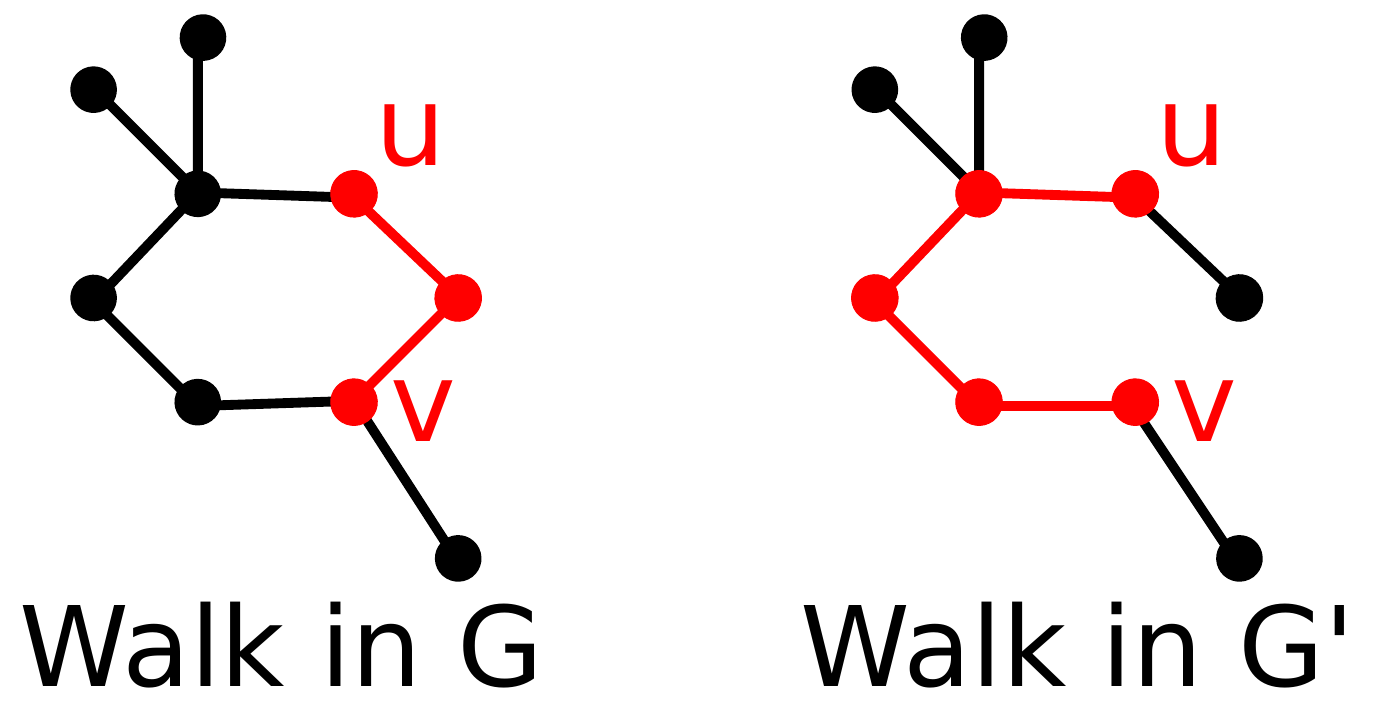}
\end{center}
Here is another example that illustrates the proof method above:
\begin{center}
\includegraphics[width=2in]{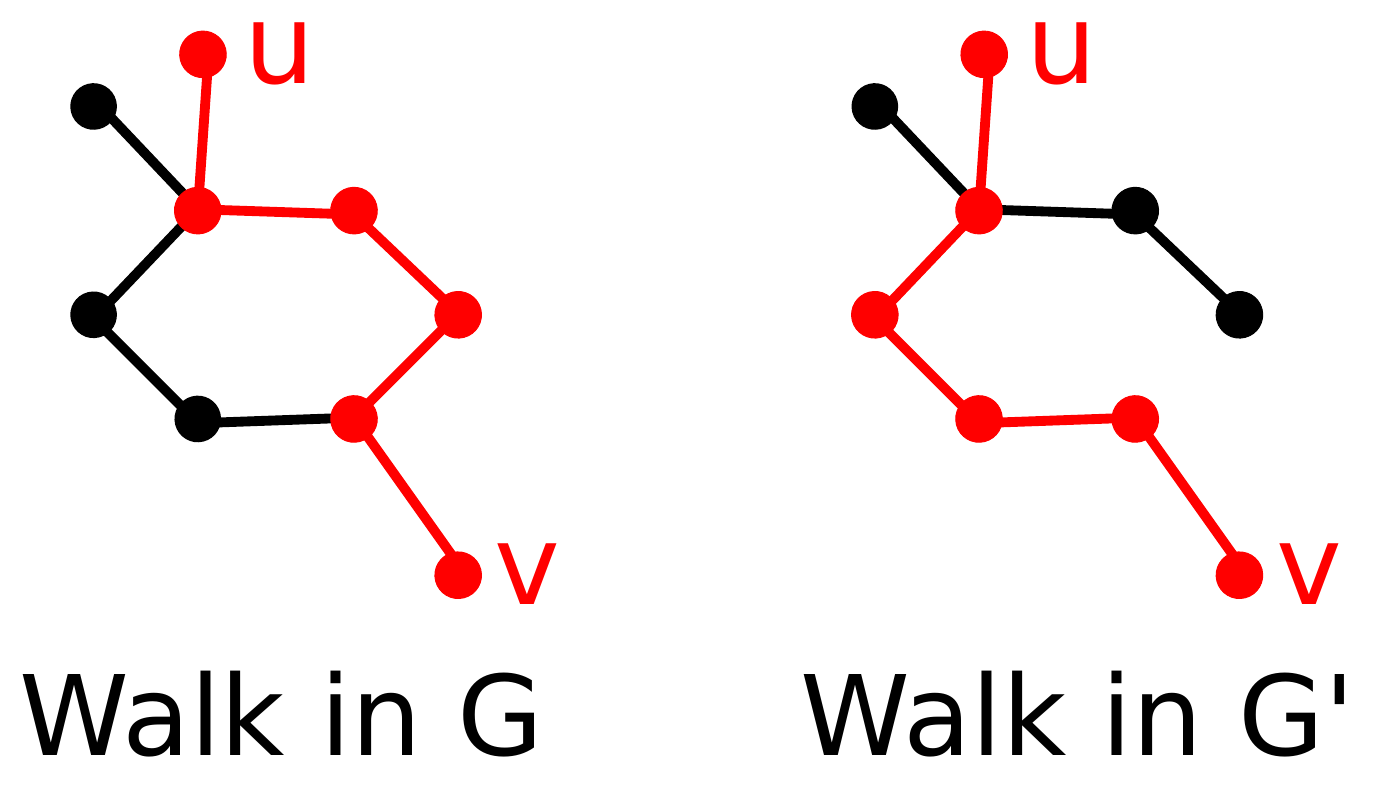}
\end{center}

We are now ready to prove Euler's formula.

\begin{proof} 
This proof will be by induction on the number of faces.  For the base case, if $f=1$, then the graph has no cycles.  So it is a connected tree.
By Theorem~\ref{thm:treeNumEdges}, a tree with $v$ vertices has $v-1$ edges.
So $\chi = v-(v-1)+1=2$.
So Euler's formula is true when $f=1$.

Now suppose $G$ is a connected graph with $f \geq 2$ faces.  We assume for induction that Euler's formula is true for any connected graph with $1 \leq f' < f$ faces.  Since $f \geq 2$, the graph $G$ is not a tree, so it has a cycle.  Pick an edge $e$ in that cycle.

By Lemma~\ref{Lremoveedgecycle}, the subgraph $G'=(G\text{ with edge }e\text{ removed})$ is still connected.  The number of vertices of $G'$ is $v'=v$.  The number of edges is $e'=e-1$ since one edge was removed.  The number of faces is $f'=f-1$.

By the inductive hypothesis, $2=\chi'=v'-e'+f'$.
Substitute that
$v'=v$, $e'=e-1$ and $f'=f-1$ to see that
$2=v-(e-1)+(f-1) = v-e+f$.
\end{proof}
To visualize the inductive proof above, think of the edge $e$ acting like a dam or levee that separates two regions; once it is removed then water will flood those two regions simultaneously.
\begin{center}
\includegraphics[width=1in]{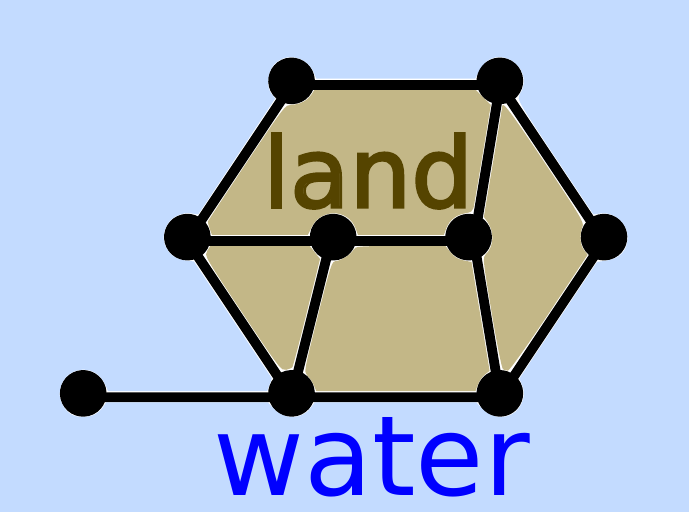}
\includegraphics[width=1in]{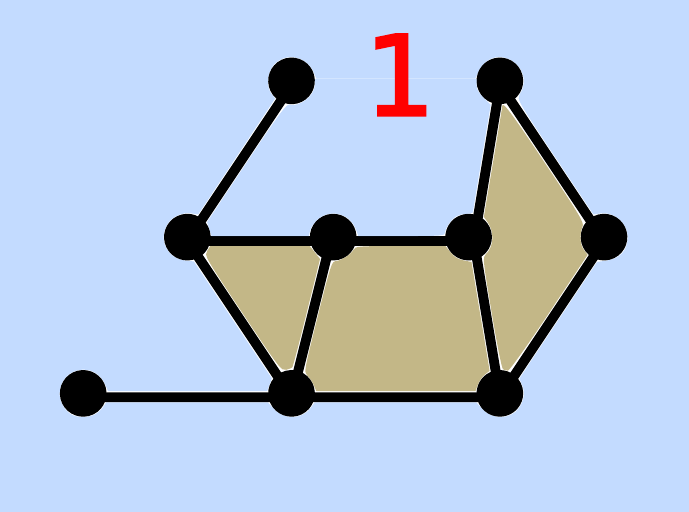}
\includegraphics[width=1in]{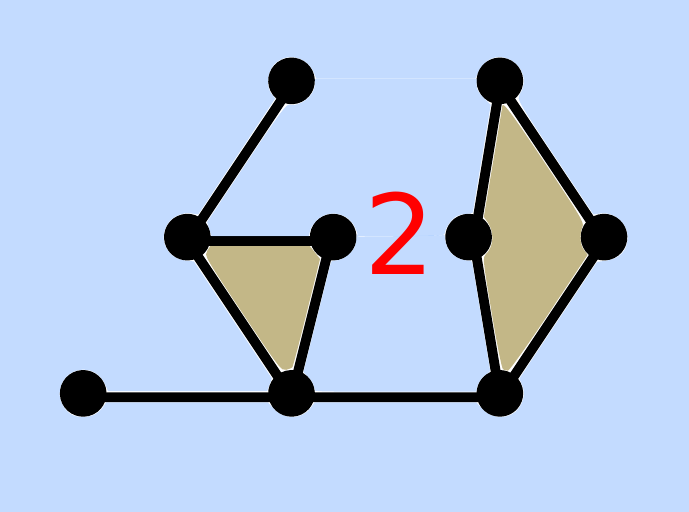}
\includegraphics[width=1in]{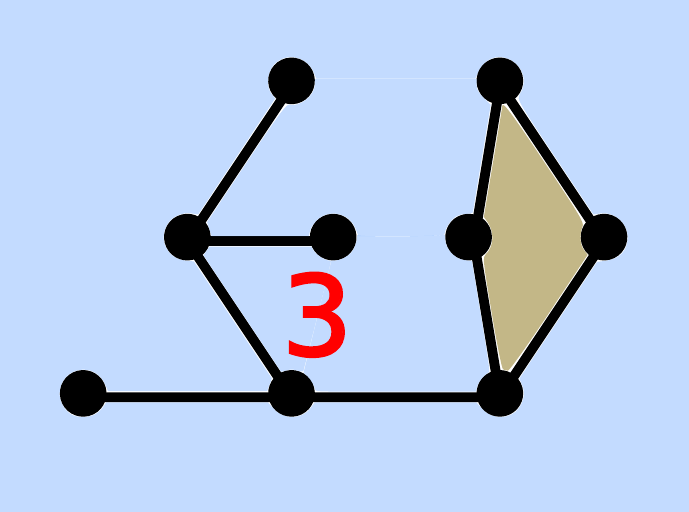}
\includegraphics[width=1in]{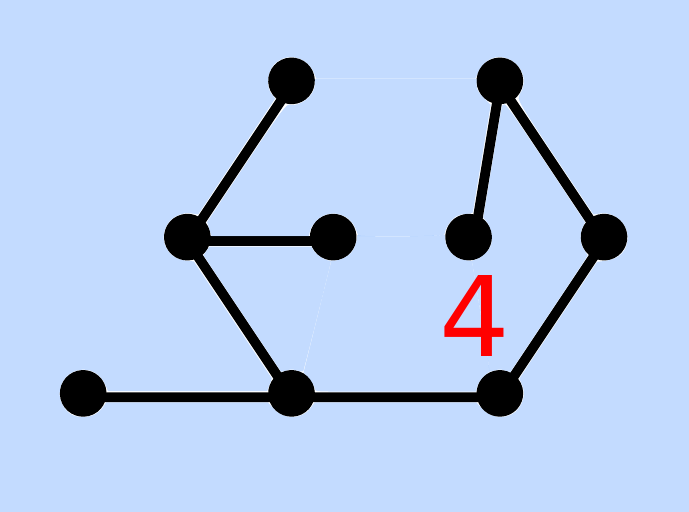}
\end{center}

\begin{remark}
Euler's formula is also true for connected planar graphs that allow loop edges (from a vertex to itself) and multiple edges between vertices.
\end{remark}

\subsection*{Exercises}

\begin{enumerate}
\item Draw a regular $n$-gon.
Add one vertex in the center and draw an edge between the center vertex and every other vertex; your picture should resemble the spokes on a bicycle wheel. 
Compute $v$, $e$, and $f$ for this graph and check that $\chi =2$.

\item If a connected planar graph has $7$ vertices and $11$ edges, what is the number of faces?  Draw a graph with these values.

\item If a connected planar graph has $5$ vertices and $5$ faces, what is the number of edges?
Draw a graph with these values.

\item If a connected planar graph has $8$ edges and $4$ faces, what is the number of vertices?  Draw a graph with these values.

\item A planar graph has 6 vertices and 7 edges. How many faces does it have?																																
\item How many edges do you need to remove from $K_5$ to make it a planar graph?																									
\item How many edges do you need to remove from $B_{3,3}$ to make it a planar graph?																									
\item What is the biggest possible number of faces for a connected bipartite planar graph with 5 edges?																									

\end{enumerate}

\section{Graphs that are not planar}
\label{sec:not-planar}
In this section, we will show that $K_5$, the complete graph on 5 vertices, is not planar.
In other words, it is not possible to draw $K_5$ on a piece of paper without edges crossing.

\begin{center}
    \includegraphics{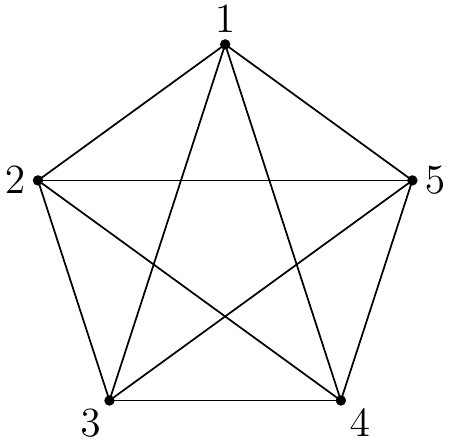}
\end{center}

For a planar graph, recall that $v$ is the number of vertices, $e$ is the number of edges, and $f$ is the number of faces.

\begin{lemma}
\label{Ledgeface}
If $G$ is a planar graph, then $e\ge\frac{3f}{2}$.
\end{lemma}

\begin{proof}
Let $E$ be the sum of the number of edges on every face.
Every edge borders two faces.
So $2e=E$. 
The graph has no multi-edges, so the number of edges on each face is greater than or equal to $3$.  So $E \geq 3f$. 
So $2e \geq 3f$.
\end{proof}

\begin{example}
For example, the below planar graph has $e=\frac{3+4+5}{2}\ge\frac{3+3+3}{2}=\frac{3f}{2}$.
\begin{center}
\includegraphics[width=1in]{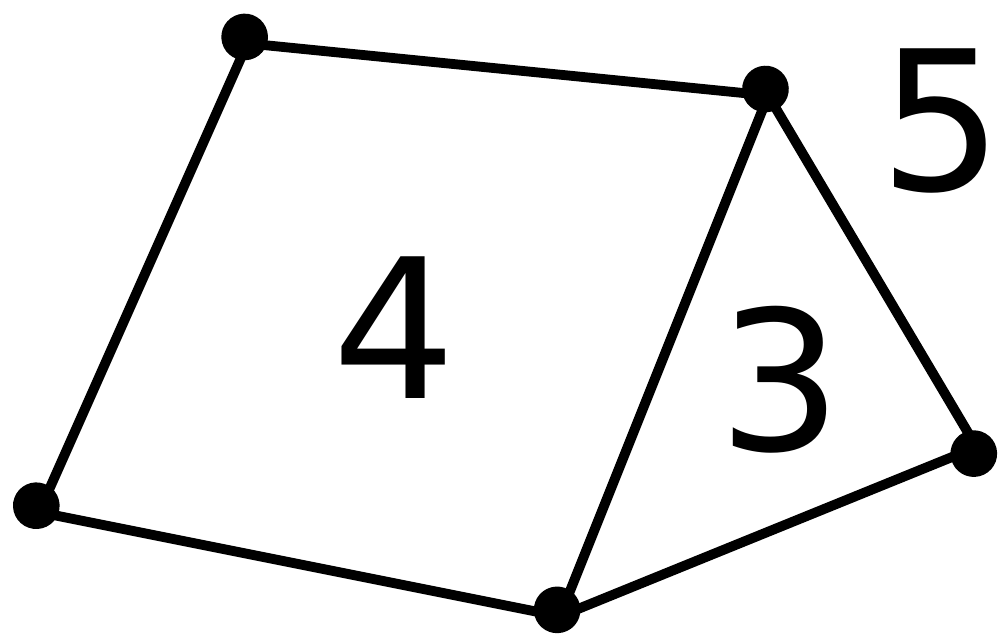}
\end{center}
\end{example}

\begin{example}
The below planar graph has $e=\frac{3+3+5+7}{2}\ge\frac{3+3+3+3}{2}=\frac{3f}{2}$.
\begin{center}
\includegraphics[width=1in]{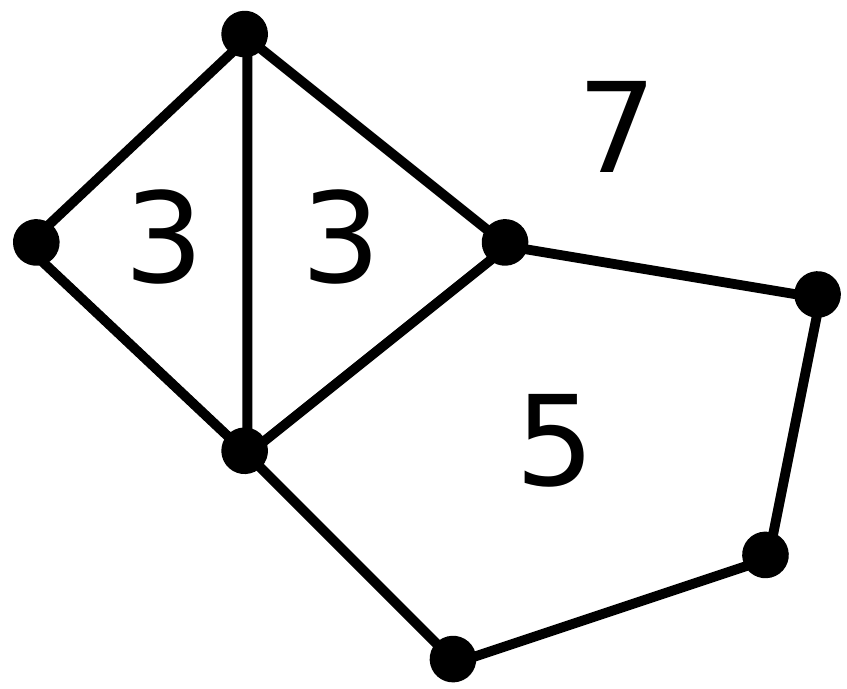}
\end{center}
\end{example}

\begin{theorem}\label{TK5notplanar}
The complete graph $K_5$ on 5 vertices is not planar.
\end{theorem}

\begin{proof}
Suppose for a contradiction that $K_5$ can be drawn as a planar graph. 
Note that $v=5$ and $e=\binom{5}{2}=10$.
By Euler's formula, then $v-e+f=2$.
So $5-10+f=2$.
This implies that $f=7$.

By Lemma~\ref{Ledgeface},
$10=e\ge\frac{3f}{2}=\frac{3\cdot7}{2}=10.5$. This is a contradiction, and hence $K_5$ cannot be planar.
\end{proof}

Since $K_5$ is not planar, any graph that contains $K_5$ as a subgraph is not planar.  One of the exercises is to show that $B_{3,3}$ is not planar; because of this, any graph that contains $B_{3,3}$ as a subgraph is not planar.

The following theorem says that the reason any graph is non-planar is because either $K_5$ or $B_{3,3}$ is hidden in the graph. In order to understand this theorem, we need to generalize the idea of subgraphs to a new concept of \emph{minors}.
(In some texts, 
this theorem is expressed using \emph{subdivisions} rather than minors.)
We include this theorem
just for your general knowledge --- we will often ask you to determine whether a graph is planar or not \emph{without} using Kuratowski's Theorem!

\begin{definition}
Given a graph $G$, 
an \emph{edge contraction} is a way to make a new graph by removing an edge from $G$ and merging the two vertices that were at its endpoints. A graph $H$ is a \emph{minor} of $G$ if $H$ can be obtained from G by contracting some edges, deleting some edges, and deleting some isolated vertices. 
\end{definition}

In the graph below on the left, if we contract the red edge, then we obtain the graph below on the right.
\begin{center}
\includegraphics[width=4in]{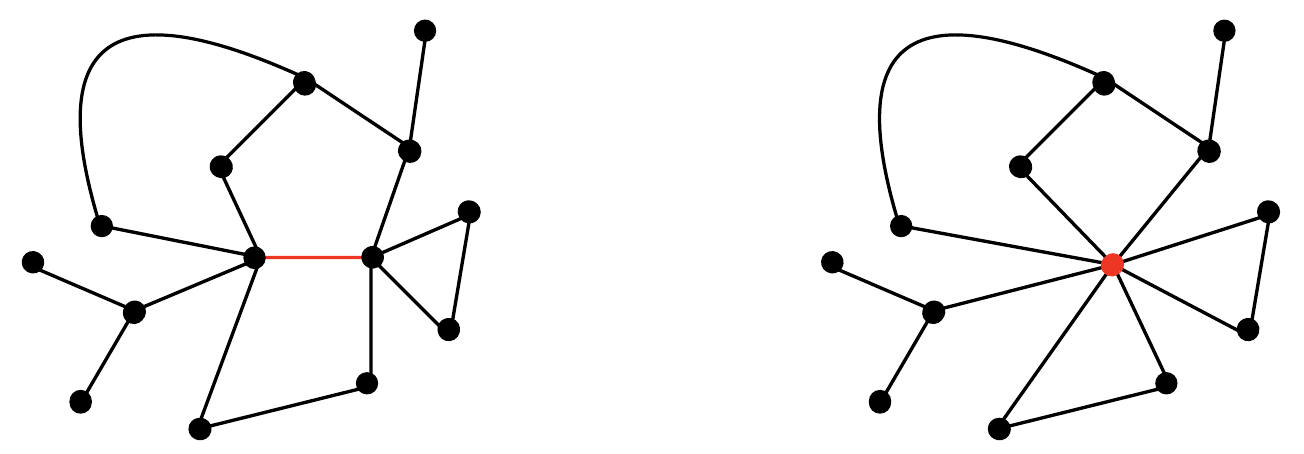}
\end{center}

\begin{theorem}[Kuratowski's Theorem]
A graph $G$ is planar if and only if it does not contain $K_5$ or $K_{3,3}$ as a minor.
\end{theorem}

For example, the Petersen graph is not planar since it has $K_5$ as a minor.  
\begin{center}
\includegraphics[width=1in]{08-GraphsWalksCycles/Eulerian11.pdf}
\end{center}
Indeed, one can contract all five edges connecting the ``inner star'' to the ``outer pentagon'' to form a copy of $K_5$.

\subsection*{Problems}

\begin{enumerate}
\item What is the smallest number of crossings needed for a graph of $K_5$ drawn in the plane? Illustrate this.

\item Is it possible to find any set of nine edges from $K_8$ such that if you remove those 9 edges then you obtain a (connected) planar graph? If so, draw it. If not, explain why not.

\item If $G$ is a bipartite graph which is a planar graph, prove that $e \geq 2f$.

\item Prove that the complete bipartite graph $B_{3,3}$ is not planar. You are not allowed to use Kuratowski's Theorem.
\begin{center}
\includegraphics[width=2.5in]{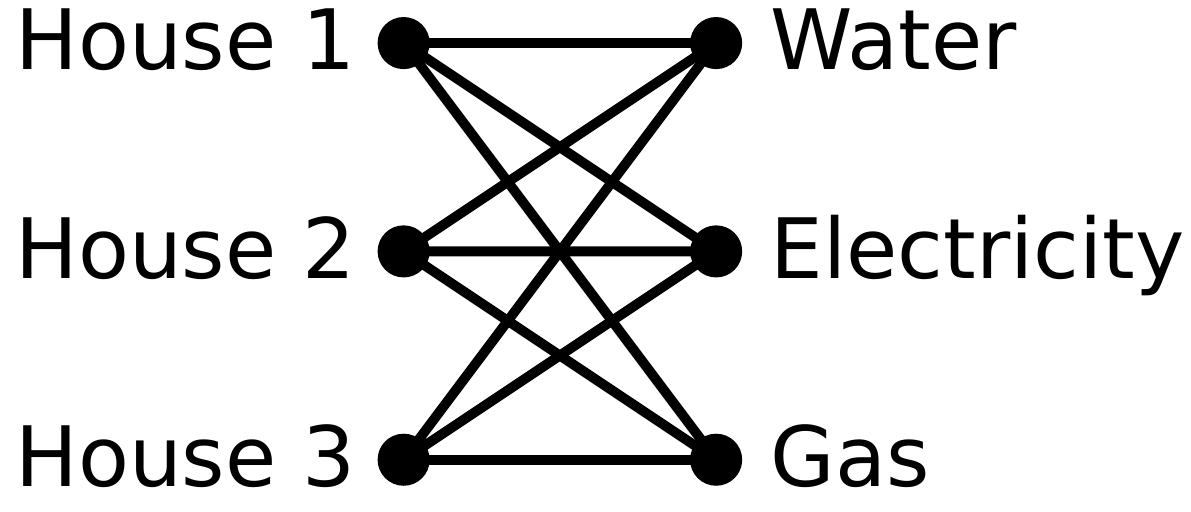}
\end{center}
(As a consequence, it is impossible to connect the above three houses to on the left to the water, electricity and gas utilities on the right, without having some of the utility lines cross.)

\item Draw the graph $B_{3,3}$ in the plane with the smallest number of crossings possible.

\item Show that the Petersen graph  
is not planar.
(You are not allowed to use Kuratowski's Theorem.)
Use the fact (which you don't need to prove) that there are no cycles of length 3 or 4 in the Petersen graph.

\item \label{Exercise7}
Prove that a planar graph with $v$ vertices (where $v\ge3$) has at most $3v-6$ edges.
This result says that a planar graph with a fixed number of vertices cannot have too many edges!
\end{enumerate}

\section{Euler's formula for polyhedra}
\begin{videobox}
\begin{minipage}{0.1\textwidth}
\href{https://youtu.be/hgMS-\_xZxic}{\includegraphics[width=1cm]{video-clipart-2.png}}
\end{minipage}
\begin{minipage}{0.8\textwidth}
Click on the icon at left or the URL below for this chapter's  lecture video. \\\vspace{-0.2cm} \\ \href{https://youtu.be/hgMS-\_xZxic}{https://youtu.be/hgMS-\_xZxic}
\end{minipage}
\end{videobox}

A polyhedron is a three-dimensional shape, whose faces are flat polygons and whose edges are straight line segments intersecting at vertices.
We restrict attention to \emph{convex} polyhedra in three-dimensional space.
The word convex means that the polyhedron is an intersection of half-spaces, or equivalently, it means that the polyhedron is the convex hull of its set of vertices, though we do not expect you to know what all of these words mean.

Here are some prisms, which are examples of polyhedra. 

\begin{center}
\includegraphics[width=1.5in]{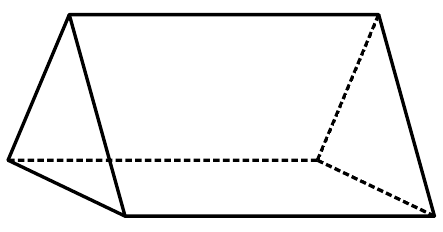}

Triangular prism: 6 vertices, 9 edges, 5 faces
\end{center}

\begin{center}
\includegraphics[width=2in]{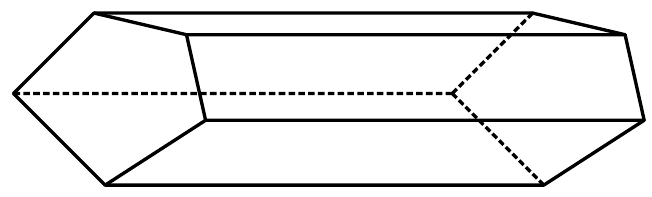}

Pentagonal prism: 10 vertices, 15 edges, 7 faces
\end{center}

The most famous polyhedra are the \emph{Platonic solids}.
These are the only \emph{regular polyhedra}, where the word regular means that all the faces are the same shape, all the edges are the same length, and all the interior angles are the same.

\begin{center}
\includegraphics[width=1in]{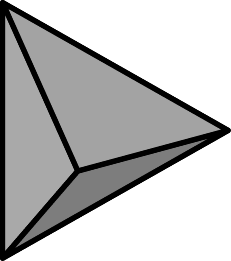}

Tetrahedron: 4 vertices, 6 edges, 4 faces
\end{center}

\begin{center}
\includegraphics[width=1in]{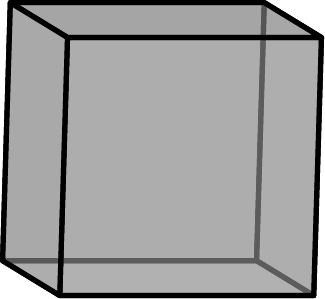}

Cube: 8 vertices, 12 edges, 6 faces
\end{center}

\begin{center}
\includegraphics[width=1in]{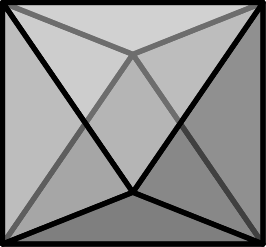}

Octahedron: 6 vertices, 12 edges, 8 faces
\end{center}

\begin{center}
\includegraphics[width=1.5in]{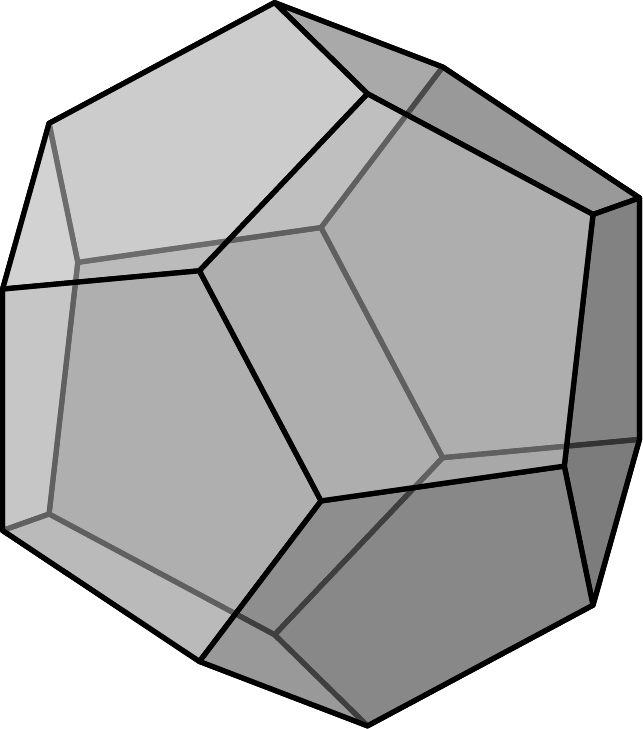}

Dodecahedron: 20 vertices, 30 edges, 12 faces
\end{center}

\begin{center}
\includegraphics[width=1.5in]{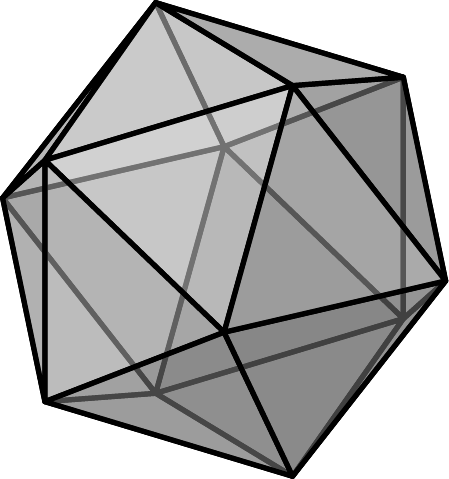}

Icosahedron: 12 vertices, 30 edges, 20 faces
\end{center}

\begin{remark}
Each polyhedron $P$ has a dual $P^D$.  Informally, $P^D$ is found by replacing each vertex with a face and replacing each face with a vertex.  Two vertices of $P^D$ are connected by an edge if and only if the corresponding faces of $P$ share an edge. 

The cube and octrahedron are ``dual" to each other; the dodecahedron and icosahedron are ``dual" to each other; and the tetrahedron is ``self-dual". 
Here is some evidence for this: the cube and the octahedron have the same number of edges, but the number of vertices and faces have been flipped. Similarly, the dodecahedron and icosahedron have the same number of edges, but the number of vertices and faces have been flipped. For the tetrahedron, the number of vertices equals the number of faces.
\end{remark}

Here are some other familiar examples of polyhedra (whose sides have been puffed out).
All of these images are licensed under the Wikipedia Creative Commons.

\begin{center}
\includegraphics[width=1.5in]{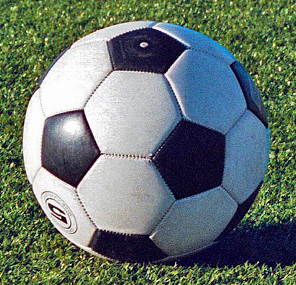}

Soccer ball: 60 vertices, 90 edges, 32 faces (12 pentagons, 20 hexagons)
\end{center}

\begin{center}
\includegraphics[width=1.5in]{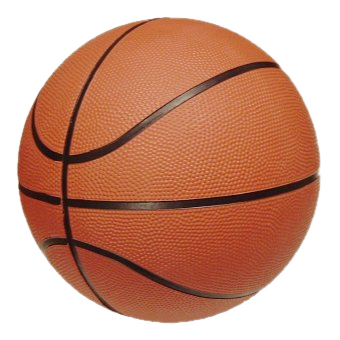}

Basketball: 6 vertices, 12 edges, 8 faces
\end{center}

\begin{center}
\includegraphics[width=1.5in]{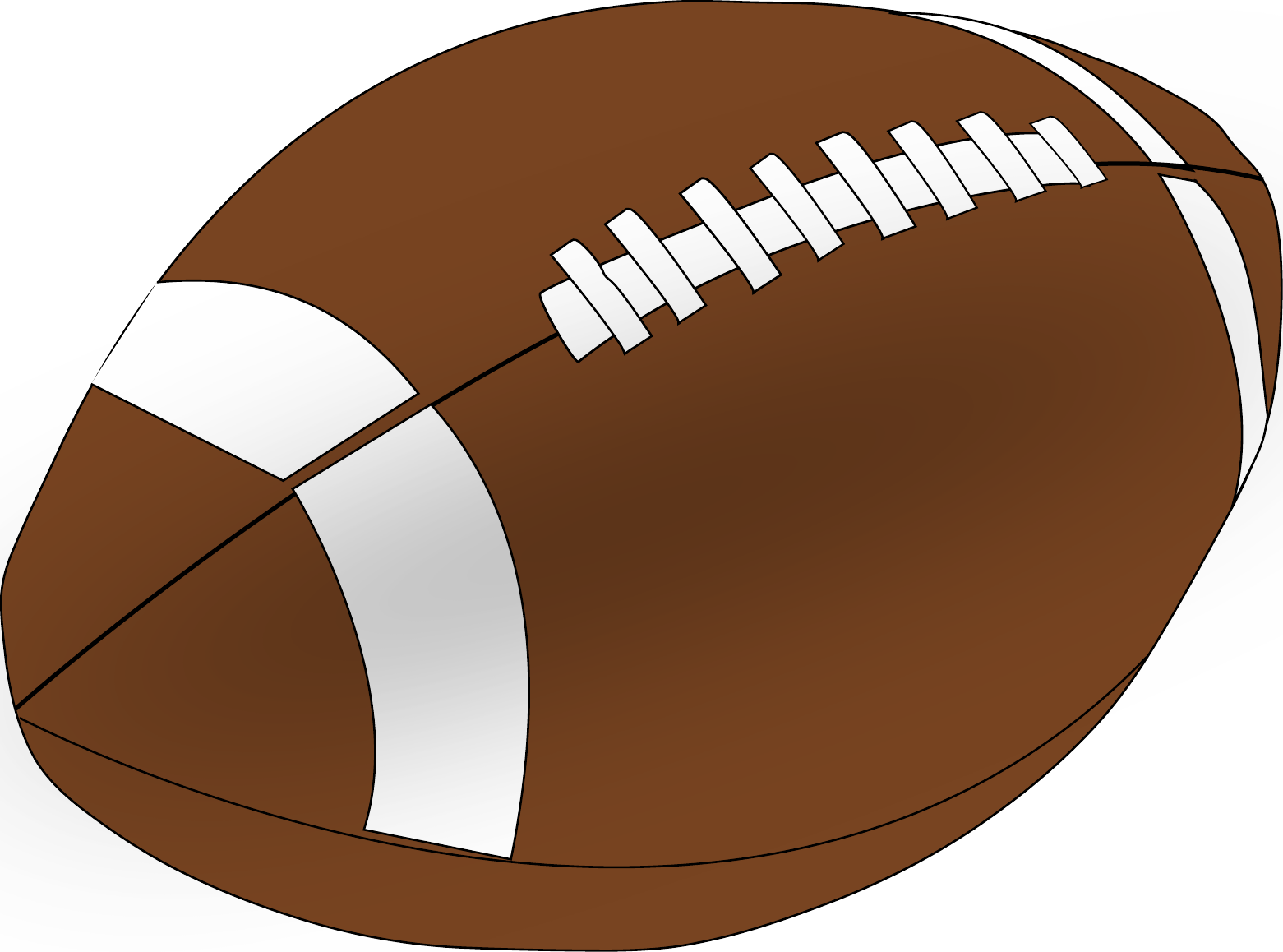}

Football: 2 vertices, 4 edges, 4 faces
\end{center}

\begin{center}
\includegraphics[width=1.5in]{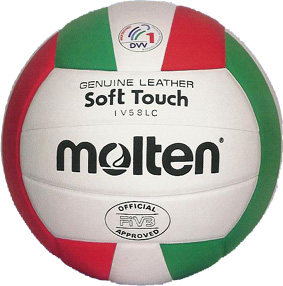}

Volleyball: 32 vertices, 48 edges, 18 faces
\end{center}

You can check that all the polyhedra above have the same Euler characteristic 
$\chi = v -e + f =2$.

\begin{theorem}[Euler]
In any convex polyhedron we have $\chi=v-e+f=2$, where $v$ is the number of vertices, $e$ is the number of edges, and $f$ is the number of faces.
\end{theorem}

\begin{proof}
Remove one face of the polyhedron. Stretch it to lie flat in the plane. This gives a planar graph, where the outside face of the planar map corresponds to the face we removed from the polyhedron.
\begin{center}
\includegraphics[width=4in]{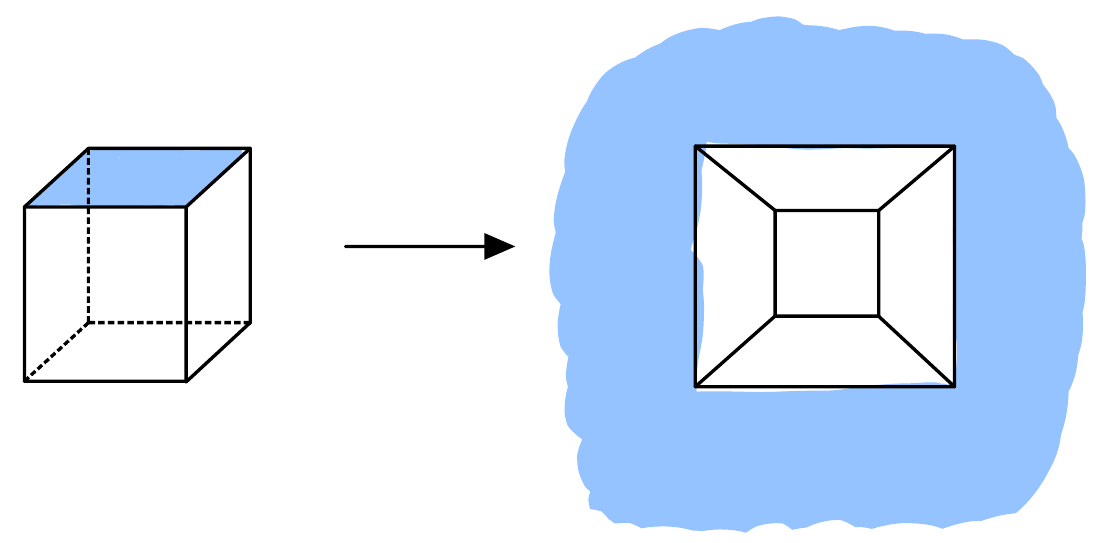}
\end{center}
This planar graph has the same values for $v$, $e$, and $f$.
For any planar graph, $\chi=v-e+f=2$ by Theorem~\ref{Teuler}.  Hence the same is true for the polyhedron.
\end{proof}

\begin{remark}
The above planar graph is called a Schlegel diagram of the polyhedron.
Show this by holding your iPhone flashlight up near the front face of a polyhedron formed by sticks and balls.
The projection of the edges of the polyhedron to the board then gives the Schlegel diagram! 
\end{remark}

\subsection*{Exercises}

\begin{enumerate}
\item Suppose a polyhedron $P$ is a triangulation of the sphere with 57 vertices and 84 edges. How many faces does $P$ have?

\item Draw planar graphs of 
the triangular and pentagonal prisms.

\item Draw planar graphs of the $5$ platonic solids.

 \item Draw several pictures of an orange with a graph on its peel.
    \begin{enumerate}
        \item Make a table of the values of $v,e,f$ for these graphs.
        \item Find the quantity $\chi = v-e+f$.
        \item Make a conjecture about $\chi$ and try  explain why it is true. 
    \end{enumerate}
\item Which platonic solid is dual to the icosahedron?			
\item Which platonic solid is dual to the tetrahedron?			
\item For which platonic solid is the ratio e/v the biggest?			
\item What is the Euler characteristic of the pentagonal prism?			
\item What is the diameter of the cube graph?			
			
\end{enumerate}

\section{Investigation: Graphs on other surfaces}

In earlier sections in this chapter, we looked at connected graphs on the plane or on the sphere.  For these graphs, the Euler characteristic $\chi=v-e+f$ was always equal to $\chi=2$.
We used this to show that $K_5$ and $K_{3,3}$ are not planar graphs.  

The Euler characteristic $\chi$ is an important concept in topology.  In this section, we illustrate the value of the Euler characteristic $\chi$ for graphs that are drawn on the torus.

\begin{center}
    \includegraphics[width=.5\textwidth]{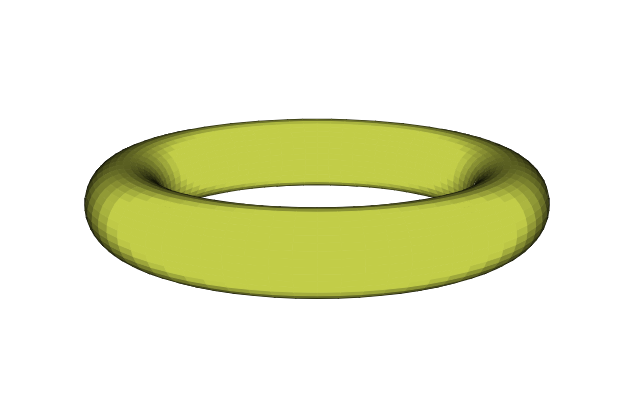}
\end{center}

The torus is the surface of a doughnut.  One way to make a torus is to take a piece of paper and tape the top to the bottom (this makes a cylinder), and then tape the left side to the right.

Another way to make a torus is as a `pacman' board.  This is a flat piece of paper where each point on the right edge is identified with the point on the left edge that has the same distance from the bottom, and each point on the top edge is identified with the point on the bottom edge which has the same distance from the left.
 
These descriptions of the torus are the same surface (in topology, we say that they are \emph{homeomorphic}).  
In the first construction, the tape forms two loops on the torus.
By cutting the torus along the two loops, it can be stretched back out to a flat piece of paper.  

There is additional flexibility when drawing graphs on the torus because edges can loop around the hole or through the hole.  On the pacman board, this extra flexibility can be seen because an edge that goes off the right side of the page comes in on the left, or an edge that goes off the top of the page comes in on the bottom.

In general, the Euler characteristic of a graph on a surface is an integer that depends only on the surface, and not on the graph, so long as the graph subdivides the surface into faces which are connected regions with no non-contractible loops!
It is a surprising fact that the Euler characteristic $\chi=v-e+f$ depends on the surface, but not on the graph.

When studying graphs on the torus, it can be tricky to count the number of vertices, edges, and faces.  Here is an example graph on the pacman board.

\begin{center}
    \includegraphics[width=.2\textwidth]{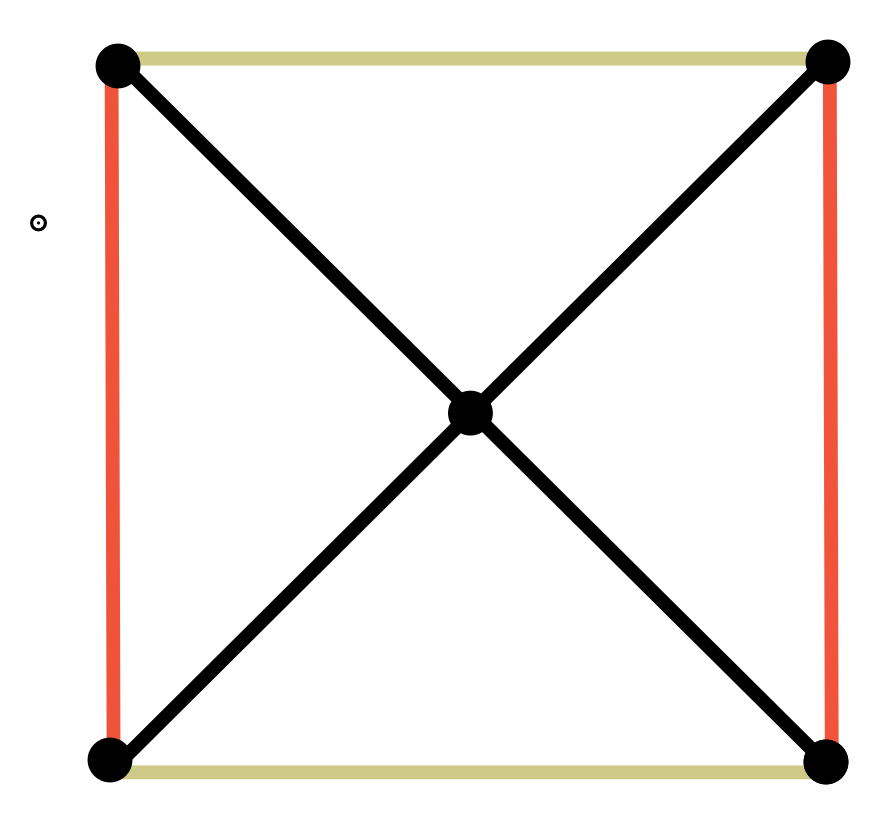}
\end{center}

In the above example on the pacman board, it looks like there are 5 vertices.
However, there are only 2.
The reason is that the 4 corners are all the same point, because of the identifications of the edges!
It looks like there are 8 edges, the 4 radiating from the center and the 
4 sides of the square.  However, there are only 6, because the top and bottom edges are identified and the left and right edges are identified.  
It looks like there are 4 faces and that is correct.  The reason is that the top region is separated from the bottom region 
by the edge along the top and the left region is separated from the right region by the edge along the left. Thus $v=2$, $e=6$, and $f=4$.  Thus $\chi=v-e+f=0$.

Indeed, the Euler characteristic of any graph on the torus (in which the faces carved out by the graph are each connected and contain no non-contractible loops) is $\chi=0$.
This is true even if the number of vertices, edges, and faces are quiet different than the above example --- e.g.\ even if the number of vertices is $v=100$ instead of $v=2$.

One interesting fact is that the complete graphs $K_5$, $K_6$, and $K_7$ can be drawn on a torus without having any edges cross, but $K_8$ cannot.

\subsection*{Exercises}

\begin{enumerate}
\item Draw several graphs on the torus.  (These graphs need to have a cycle that goes around the hole and a cycle that goes through the hole.) 
Make a table of the values of $v,e,f$ for these graphs.
    \begin{enumerate}
        \item Find the quantity $\chi = v-e+f$.
        \item Make a conjecture about $\chi$ and try  explain why it is true. 
\end{enumerate}

\item Draw some graphs on the pac man board.  (These graphs need to have a cycle that goes from the bottom to the top and a cycle that goes from the left to the right.) Make a table of the values of $v,e,f$ for these graphs.
\begin{enumerate}
        \item Find the quantity $\chi = v-e+f$.
        \item Make a conjecture about $\chi$ and try  explain why it is true. \end{enumerate}

\item Draw $K_5$ on the surface of a doughnut, with no edges crossing.

Hint: it is easier to draw this on the pacman board.

\item Draw $K_{3,3}$ on the surface of a doughnut with no edges crossing.

Hint: it is easier to draw this on the pacman board.

\item An IKEA furniture store box arrives with
   32 attachment brackets, 64 metal rods that are 1 foot long, and 32 panels measuring $1$ square foot.
   Explain how to construct these into a $3 \times 3 \times 1$ storage cabinet with a window in the center.  The storage cabinet must be fully enclosed but has no interior shelves.  
   
\item Can you draw either $K_5$ or $K_{3,3}$ on the surface of a Klein bottle?

\end{enumerate}

%% file: 12-ColoringGraphs/ColoringGraphs.tex
\chapter{Graph Coloring}\label{chap:coloring}

\section{The chromatic number of a graph}

Imagine you are in charge of creating a new schedule for final exams at CSU, where students take at most one final each day.
You need to make sure that no student has two finals scheduled on the same day, since the student cannot be in two places taking two exams at once.  An easy way to solve this problem is to give each course its own day; however, this is an impractical solution, since there are hundreds of courses offered a semester and no one wants to turn the week of finals into the year of finals!
Having all finals on the same day also does not work, unless there is in no overlap of students between classes. Our goal for this scenario is to minimize the number of days that the university offers finals, so that the students and faculty can have the longest break possible between semesters.

One way to visualize this problem is with a graph.
We can construct this graph to have $n$ vertices representing the $n$ courses offered that semester, and an edge between two vertices if there is (at least) one student who is in both of those courses.  Now, all the finals on Monday of finals week will be assigned the color blue. 
There will be no conflicts as long as any two adjacent vertices are not both blue, because adjacency represents at least one student being in both classes.
We can continue to introduce new colors for each day needed (orange for Tuesday, yellow for Wednesday, $\ldots$) until we have a graph where no two vertices of the same color are adjacent.  
The smallest number of colors required is the minimum number of days needed for final exams and is called the \textit{chromatic number} of the graph.

\begin{definition}
Let $G$ be a labeled graph.
A \defn{vertex-coloring} is a way of coloring the vertices of $G$ so that any two adjacent vertices have different colors.

A \defn{$k$-coloring} of $G$ is a vertex coloring of $G$ that uses at most $k$ colors.

A graph is \defn{$k$-colorable} if it has a vertex-coloring with $k$ colors.

The \defn{chromatic number} $C(G)$ of a graph $G$ is the smallest integer $k$ such that $G$ is $k$-colorable.
\end{definition}

For example, the following graph has chromatic number 3.
\begin{center}
\includegraphics[width=2in]{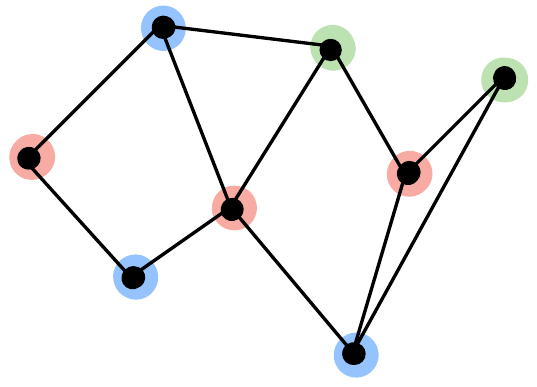}
\end{center}
The following larger graph has the chromatic number 4.
\begin{center}
\includegraphics[width=2in]{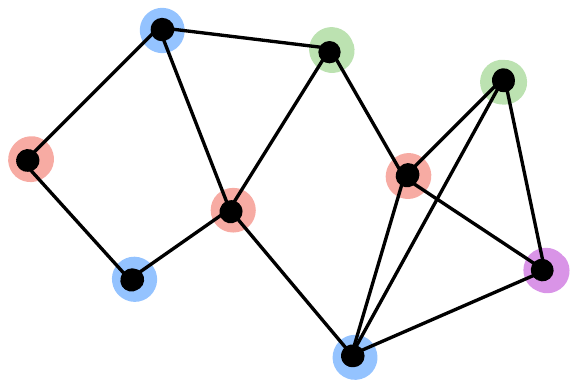}
\end{center}

\begin{example}\label{example:CGKn}
The chromatic number of $K_n$ is 
$C(K_n) =n$.
\end{example}

\begin{proof}
The graph $K_n$ is $n$-colorable because each vertex can be a different color.
The graph $K_n$ is not $(n-1)$-colorable since by the pigeonhole principle, if we have only $n-1$ colors then at least 2 (necessarily adjacent) vertices will have the same color.
\end{proof}

\begin{example}\label{example:CGNn}
The chromatic number of the empty graph with $n\ge 1$ vertices (but no edges) is $1$.
\end{example}

\begin{example} \label{Echromnumber1}
The chromatic number of the path graph $P_n$ with $n$ vertices is $C(P_n)=2$ if $n \geq 2$.
\end{example}

\begin{center}
   \includegraphics[width=.5\textwidth]{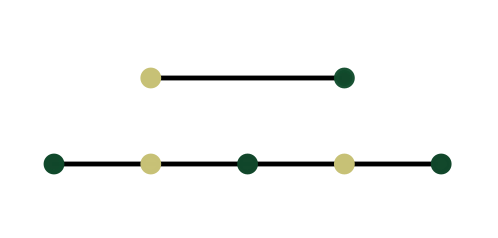}
\end{center}

\begin{example} \label{Echromnumber2}
The chromatic number of the cycle graph $C_n$ is $2$ when $n$ is even and $3$ when $n \geq 3$ is odd.
\end{example}

\begin{center}
     \includegraphics[width=.5\textwidth]{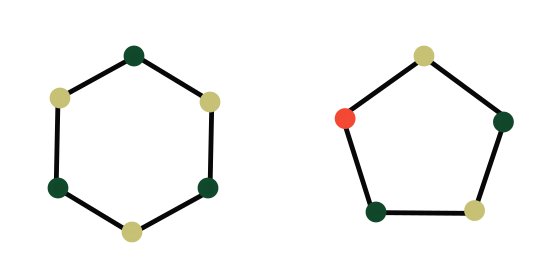}
\end{center}

\begin{remark} \label{Dedgecolor}
An \defn{edge-coloring} of $G$ is a way of coloring the edges of $G$ such that
all the edges ending at a vertex $v$ have different colors, for each vertex $v$ of $G$. 
The \defn{edge-coloring number} of $G$ is the smallest number $k$ such that $G$ has an edge-coloring with $k$ colors.
\end{remark}

In this book will consider vertex colorings more than edge colorings, and whenever we write \emph{chromatic number} we mean the minimum number of colors required for a vertex coloring (as defined above).
Nevertheless, we also ask a few questions about edge colorings at the end of this section.

\subsection*{Exercises}

\begin{enumerate}

\item Explain Examples~\ref{example:CGNn}, \ref{Echromnumber1}, and  \ref{Echromnumber2}.

\item Find the chromatic number of the following graph.

\begin{tikzpicture}
\draw[CSUGreen, line width=1mm] (0,0) circle [radius=2];
\draw[CSUGreen, line width=1mm] (90:4cm) -- (0,0) -- (270:2cm);
\draw[CSUGreen, line width=1mm] (90:4cm) -- (30:2cm) -- (330:4cm);
\draw[CSUGreen, line width=1mm] (90:4cm) -- (150:2cm) -- (210:4cm);
\draw[CSUGreen, line width=1mm] (210:4cm) -- (0,0) -- (30:2cm);
\draw[CSUGreen, line width=1mm] (210:4cm) -- (270:2cm) -- (330:4cm);
\draw[CSUGreen,line width=1mm] (330:4cm) -- (0,0) -- (150:2cm);

\fill[CSUGold, draw=CSUGreen,line width=1mm] (0,0) circle [radius=0.5];
\fill[CSUGold, draw=CSUGreen,line width=1mm] (90:4cm) circle [radius=0.5];
\fill[CSUGold, draw=CSUGreen,line width=1mm] (210:4cm) circle [radius=0.5];
\fill[CSUGold, draw=CSUGreen,line width=1mm] (330:4cm) circle [radius=0.5];
\fill[CSUGold, draw=CSUGreen,line width=1mm] (150:2cm) circle [radius=0.5];
\fill[CSUGold, draw=CSUGreen,line width=1mm] (270:2cm) circle [radius=0.5];
\fill[CSUGold, draw=CSUGreen,line width=1mm] (30:2cm) circle [radius=0.5];

\node[font=\Large] at (90:4cm) {$A$};
\node[font=\Large] at (0,0) {$B$};
\node[font=\Large] at (30:2cm) {$C$};
\node[font=\Large] at (150:2cm) {$D$};
\node[font=\Large] at (210:4cm) {$E$};
\node[font=\Large] at (330:4cm) {$F$};
\node[font=\Large] at (270:2cm) {$G$};

\end{tikzpicture}

\item 
Find the chromatic number of the 5 platonic solids.

\item Find the chromatic number of the Petersen graph.

\item Show that the edge number of a graph is at least as big as the maximum degree of the vertices of $G$.

\item Find the edge coloring number of the following graphs:
\begin{enumerate}
\item $P_n$;
\item $C_n$;
\item $K_4$;
\item $K_5$;
\item cube and octahedron. 
\end{enumerate}

\item Let $G$ and $H$ be two graphs that have no vertices (and hence no edges) in common.
Let $G\cup H$ be the graph that consists of all of the vertices and edges that live in either $G$ or $H$.
How is the chromatic number $C(G\cup H)$ of the union related to the two chromatic numbers $C(G)$ and $C(H)$ of the individual graphs?

\end{enumerate}

\section{What graphs are 2-colorable?}

It turns out that we have already seen the graphs that can be colored with two colors.
Indeed, they are the bipartite graphs!

\begin{proposition}
Let $G$ be a connected graph with at least 2 vertices.
Then $G$ is bipartite if and only if $C(G)=2$. 
\end{proposition}

\begin{proof}
A graph is bipartite if and only if it is $2$-colorable.  The reason is that we can identify vertices on the left side with the color gold and vertices on the right side with the color green.
The condition for a 2-coloring that adjacent vertices have different colors is the same as the condition for a bipartite graph that all edges connect vertices on opposite sides.
At least two colors are needed if the graph is connected and has at least two vertices.
\end{proof}
\begin{center}
\includegraphics{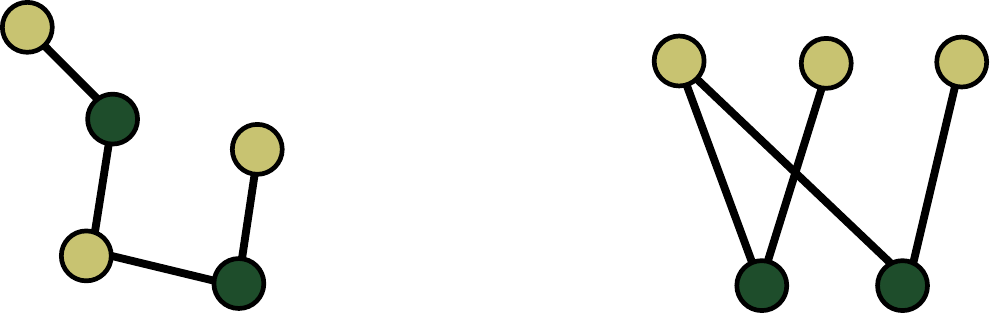}
\end{center}

The next theorem shows that 2-colorable graphs are those graphs that do not have any odd cycles.

\begin{theorem}
\label{thm:odd-cycles}
A graph $G$ has $C(G)\leq 2$ if and only if it contains no odd cycles.
\end{theorem}

\begin{proof}
The forward direction ($\Rightarrow$) follows since an odd cycle is not 2-colorable, and hence any graph containing an odd cycle is not 2-colorable.

For the reverse direction ($\Leftarrow$), suppose that $G$ contains no odd cycles. Pick a vertex ``$a$" and color it red. Color all neighbors of $a$ blue. Then, color all neighbors of the new set of colored vertices red, etc. Continue, alternating coloring the new neighbors red or blue, until all vertices are colored.  The first three steps of this process are shown here:
\begin{center}
\includegraphics{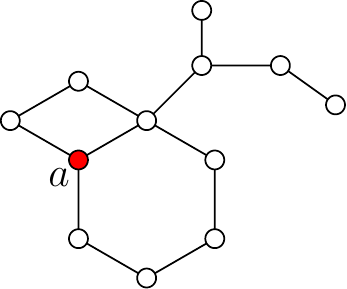}\hspace{1cm} \includegraphics{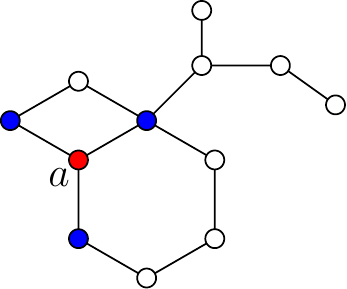}\hspace{1cm} \includegraphics{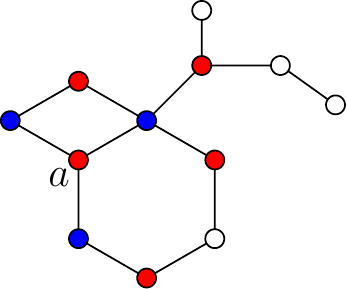}
\end{center}
Suppose for contradiction that two adjacent vertices $u$ and $v$ have the same color in this coloring.
\begin{center}
\includegraphics{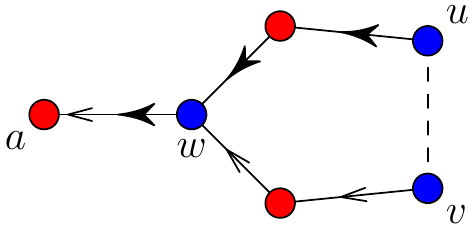}
\end{center}
Pick shortest paths from $u$ to $a$ and from $v$ to $a$; these paths alternate colors.  Let $w$ be the first junction point of these two paths.  The paths from $u$ to $w$ and then reversed from $w$ to $v$, along with edge $uv$, forms an odd cycle; this is a contradiction. Hence $G$ must be 2-colorable.
\end{proof}

\begin{remark}
The proof of Theorem~\ref{thm:odd-cycles} also gives an algorithm for deciding if a graph is 2-colorable (i.e., if it has no odd cycles).
\end{remark}

\subsection*{Exercises}

\begin{enumerate}

\item If $G$ is a connected bipartite graph with at least one edge, then show that $G$ has exactly two different 2-colorings.

\item Which platonic solids are 2-colorable?

\item Let $n$ be a positive integer.  Let $G$ be the graph whose vertices are the set of binary strings of length $n$, where two vertices are adjacent if and only if the strings differ in exactly one position.
For example, when $n=3$, then $G$ is the graph of the cube.
Is $G$ 2-colorable?  Explain why or why not.
\end{enumerate}

\section{Bounds on chromatic numbers}
Given a graph $G$ with $n$ vertices, what are the upper and lower bounds on the chromatic number $C(G)$?  For example, take a graph with three vertices.  If there were no edges in our graph (the empty graph on three vertices), we would have that $C(G)=1$.  If there was one edge ($P_2$ and a vertex) or two edges ($P_3$), we would have $C(G)=2$.  Finally, if we had three edges, our graph ($C_3=K_3$) would have $C(G)=3.$  Thus, we have that a graph with $n=3$ vertices has a lower bound of $C(G)=1$ and upper bound of $C(G)=3$.  We generalize this process in the following theorem.
\begin{center}
\includegraphics[]{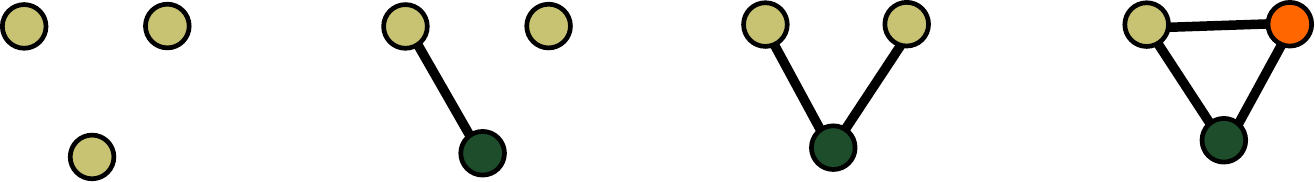}
\end{center}

\begin{theorem}
Let $G$ be a graph with $n$ vertices, where $n\ge 1$.
Then $1\leq C(G) \leq n$.
\end{theorem}

\begin{proof}
As shown in Example \ref{example:CGKn} and Example \ref{example:CGNn}, the chromatic number of the complete graph $K_n$ is $n$ and the chromatic number of the empty graph is 1.  Now, consider a graph $G_1$ that is not complete.  Then, there exists two vertices $x_1$ and $y_1$ that are not adjacent.  We can therefore assign $x_1$ and $y_1$ to be the same color and assign each of the remaining $n-2$ vertices to be its own color.  Then, we see that the chromatic number of this graph is $C(G_1)\leq n-1 < n$.  Similarly, take another graph $G_2$ that is not empty.  Then, there exists two vertices $x_2$ and $y_2$ that are adjacent.  We must assign $x_2$ and $y_2$ to be two different colors.  Therefore, $C(G_2)\geq 2 >1.$
\end{proof}

In the beginning of this section, we showed that the chromatic number of $K_3$ or $C_3$ was 3.  This means that any graph that contains a triangle $C_3=K_3$ is not $2$-colorable.  
This can be generalized in two ways: by looking at the cycle graph $C_n$ when $n \geq 3$ is odd; or by looking at the complete graph $K_n$ when $n \geq 3$. 

The below graph contains a copy of $K_5$, the complete graph on 5 vertices.
Since $K_5$ is not 4-colorable, this means that the entire graph is not 4-colorable.

\begin{center}
\includegraphics[width=2in]{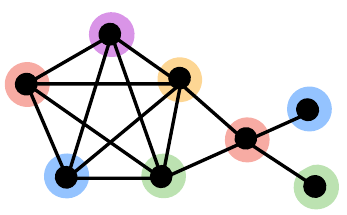}
\end{center}

\begin{remark}
Any graph containing a copy of $K_n$ is not $(n-1)$-colorable.
\end{remark}

Recall also that any cycle graph with an odd number of vertices is not 2-colorable.
It follows that the following below graph, which contains a 7-cycle, is not 2-colorable.

\begin{center}
\includegraphics[width=2in]{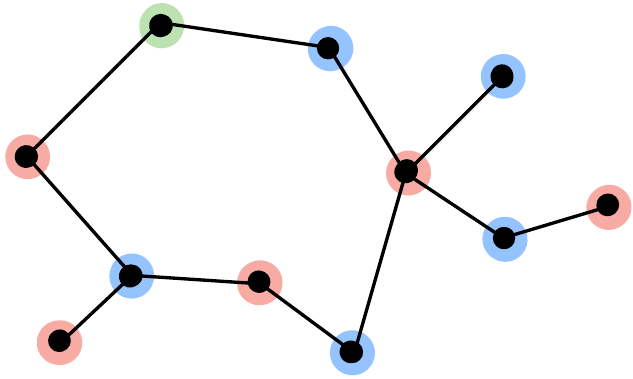}
\end{center}

\begin{remark}
Any graph containing a copy of $C_n$ for $n\geq 3$ odd is not $2$-colorable.
\end{remark}

More generally, whenever $H$ is a subgraph of $G$, then the chromatic number of $H$ provides a lower bound on the chromatic number of $G$

\begin{proposition}
If $G$ is a graph and $H$ is a subgraph of $G$, then $C(H)\le C(G)$.
\end{proposition}

\begin{proof}
Let $k<C(H)$.
Suppose for a contradiction that $G$ were $k$-colorable.
If we simply removed all of the vertices in $G$ that are not in $H$ (and also all of the edges in $G$ that are not in $H$), then we would obtain a valid coloring of $H$ with at most $k$ colors.
This contradicts the fact that $k<C(H)$.
\end{proof}

\subsection*{Exercises}

\begin{enumerate}

\item Show that the chromatic number of the following graph is at least 3.
\begin{center}
    \includegraphics[width=.25\textwidth]{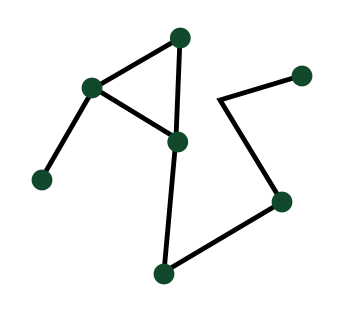}
\end{center}

\item Show that the chromatic number of the following graph is at least 3.
\begin{center}
    \includegraphics[width=.25\textwidth]{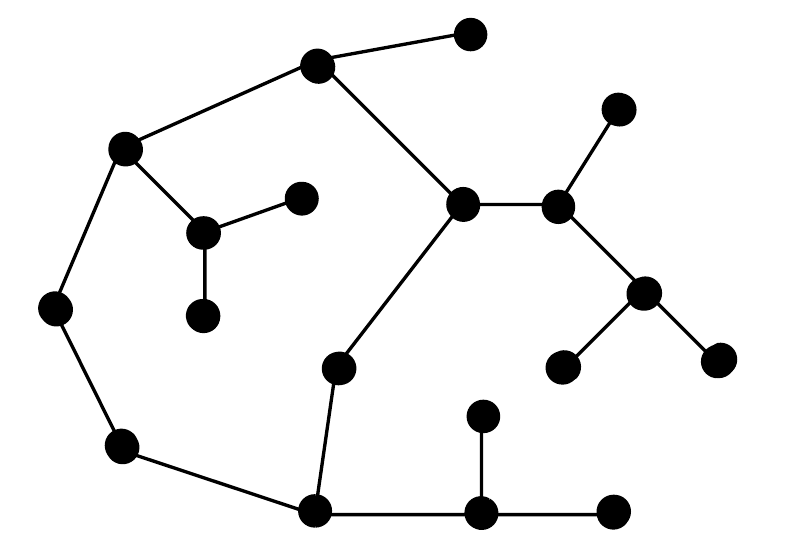}
\end{center}

\item Show that the chromatic number of the following graph is at least 4.
\begin{center}
    \includegraphics[width=.25\textwidth]{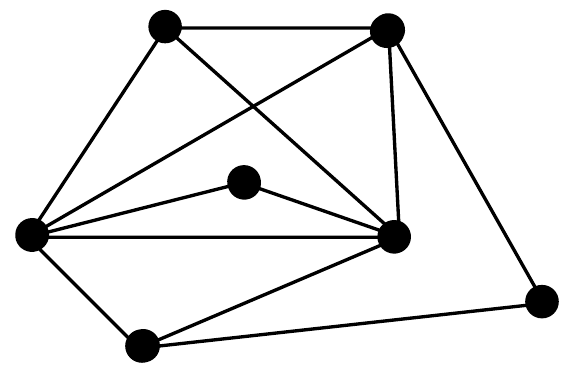}
\end{center}

\item Show that the chromatic number of the following graph is at least 5.

\begin{center}
    \includegraphics[width=.25\textwidth]{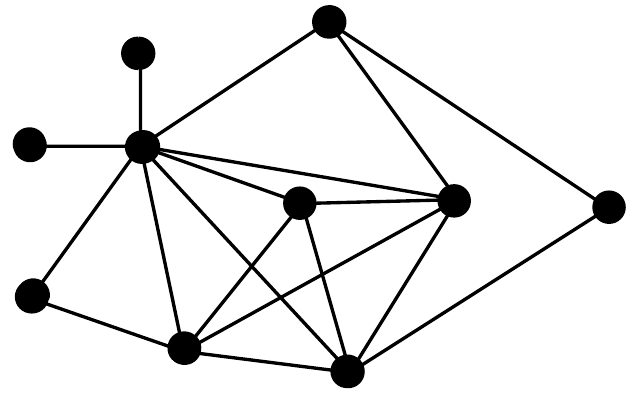}
\end{center}

\end{enumerate}

\section{Brooks' Theorem}

What happens if we try to 3-color a graph in the same way we 2-colored a graph?
Without loss of generality, fix three of the vertices in the left-most triangle to be Red, Blue, and Green.
How many ways can we try to continue this coloring?
\begin{center}
\includegraphics[width=3in]{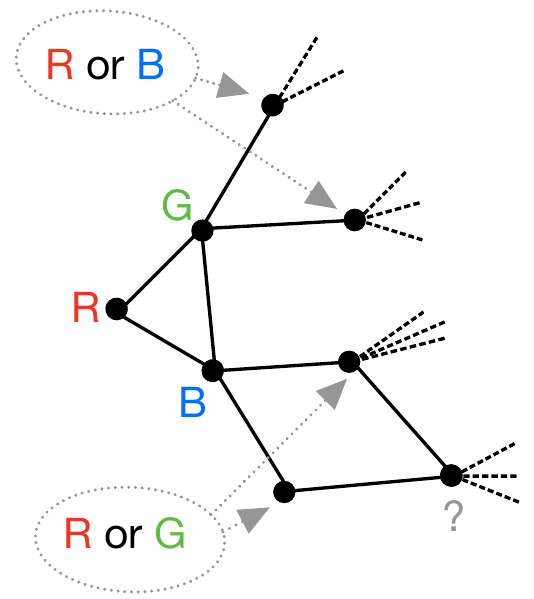}\\
R = Red, B = Blue, G = Green
\end{center}
The number of options to try explodes!

\begin{remark}
One example application of coloring graphs is to scheduling. Suppose for example that you are the math department chair, with a collection of math classes that need to be taught. The times for these classes are already  fixed. What is the minimum number of instructors you need to hire?

You can solve this problem by building a graph with a vertex for each math class, and by adding an edge between two vertices if the corresponding classes overlap in time. The minimum number of colors needed to color this graph is the same as the minimum possible number of instructors you need to hire.

See Problem \#4 in Section~\ref{sec:coloring-problems} for another example scheduling application.
\end{remark}

\begin{remark}
Deciding if a graph with $n$ vertices is 3-colorable in general takes running time $\approx 2^{n/2}$. This is expensive.
Nothing substantially better is known for $k$-colorings for any $k\ge 3$. 
\end{remark}

Here's a special case where $k$-colorings are possible:

\begin{theorem}[Brooks' Theorem]
\label{thm:Brooks}
If every vertex in a graph $G$ has degree at most $d$, then $G$ is $(d+1)$-colorable.
\end{theorem}

\begin{proof}
Fix $d$. We proceed by induction on $n$, the number of vertices in $G$.

For the base case, if $n\le d+1$, then clearly the graph is $(d+1)$-colorable, since we can use a different color for each vertex.

Now, assume any such graph with $n$ vertices is $(d+1)$-colorable. (We must show the same is true for any such graph with $n+1$ vertices.)

Given a graph $G$ with $n$ vertices, delete one vertex $v$. By the inductive assumption we can color what remains. Since $v$ has at most $d$ neighbors, there is at least one free color out of the $(d+1)$-colors we can use to color $v$ compatibly. Hence we are done by induction.

Below is a picture in the case when $d=3$.
\begin{center}
$G$\hspace{30mm}$G\setminus v$\hspace{30mm}$G$\\
\includegraphics[width=4in]{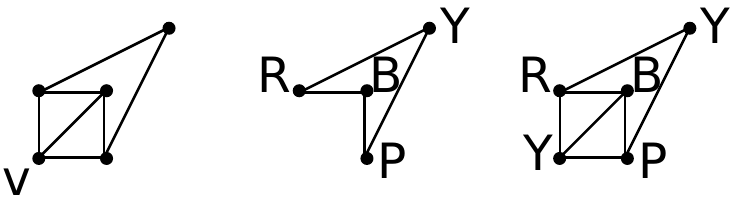}\\
R = Red, B = Blue, Y = Yellow, P = Purple
\end{center}
\end{proof}

\begin{aside}
Furthermore, a graph in which every vertex has degree at most $d$ is $d$-colorable unless
\begin{itemize}
\item $d=2$ and some connected component is an odd cycle, or
\item $d>2$ and some connected component is $K_{d+1}$.
\end{itemize}
The proof of this aside is very hard!
\end{aside}

\begin{exercise}
Show the following graph is not 3-colorable.
\begin{center}
\begin{center}
    \includegraphics[width=.25\textwidth]{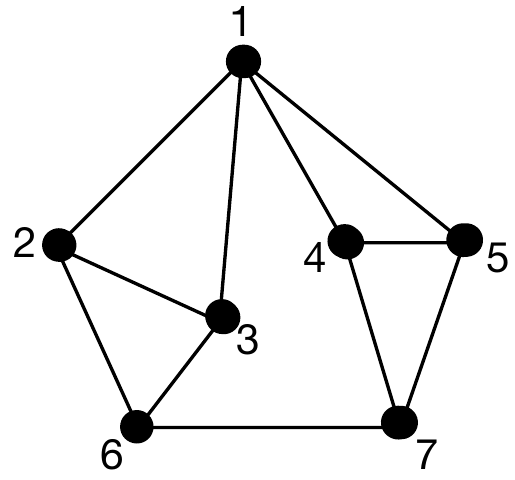}
\end{center}
\end{center}
\end{exercise}

\begin{remark}
Sadly there is no $K_4$ in this graph, because then we'd be done immediately.
\end{remark}

\begin{answer}
Suppose for a contradiction that this graph were 3-colorable. Without loss of generality let vertex 1 be Red.
Then vertices 2 and 3 are Blue and Green or Green and Blue, which implies 6 is Red.
Also vertices 4 and 5 are Blue and Green or Green and Blue, which implies 7 is Red.  This is a contradiction since 6 and 7 are adjacent! Hence this graph is not 3-colorable.
\end{answer}

\begin{exercise}
Show the following graph is not 3-colorable.
\begin{center}
\includegraphics{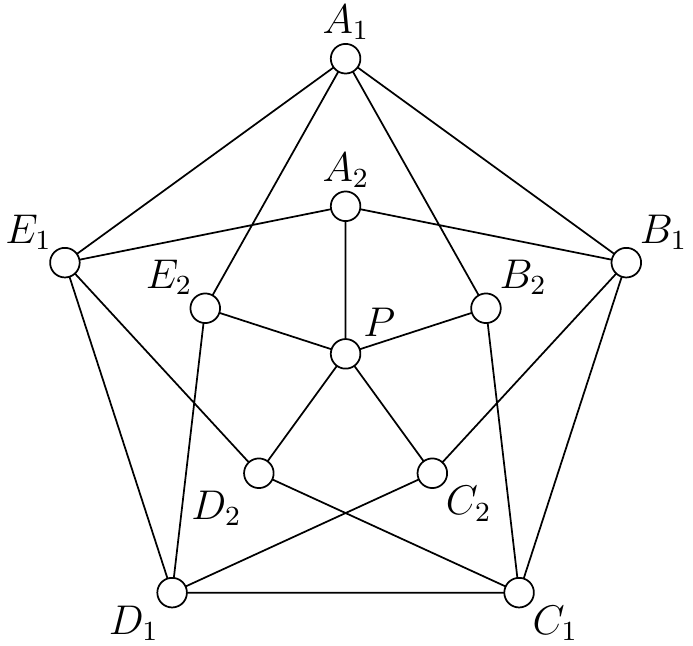}
\end{center}
\end{exercise}

In the following solution, we will call the inner vertex $P$ the \textit{central vertex}, and the other vertices are in five pairs of \textit{twins}: $A_1$ and $A_2$ are twins, $B_1$ and $B_2$ are twins, and so on.  In its pair, $A_1$ is called the \textit{outer twin} and $A_2$ is called the \textit{inner twin} (and similarly for all other twin pairs).

\begin{answer}
Suppose for a contradiction this graph were 3-colorable. Without loss of generality let the central vertex be Red. So all 5 inner vertices are Blue or Green.  A way of coloring the vertices that agrees with this choice is shown below (but naturally is not a $3$-coloring - we will just pretend that it is a valid coloring as an illustration).

\begin{center}
    \includegraphics[width=15cm]{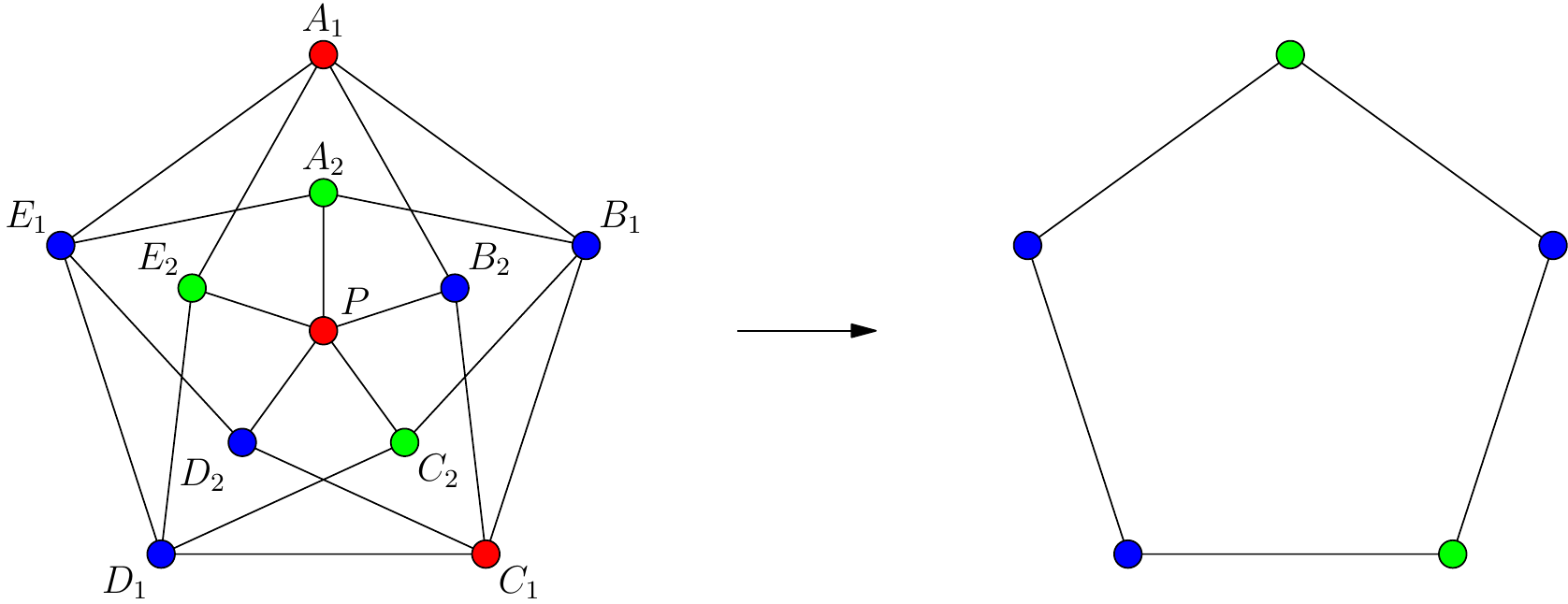}
\end{center}

We will show that, if this coloring is a valid $3$-coloring, then we can use it to  produce a valid 2-coloring of the outer 5-cycle (but \emph{not} of the entire graph) as follows: if any outer twin is colored Red, replace Red with the color of its inner twin (Blue or Green). Then in this coloring of the outer cycle with Blue and Green, if an outer twin $v$ is changed from Red to Blue (for example), then both of $v$'s adjacent outer twins must be Green. Indeed, neither of $v$'s adjacent outer twins could have been Blue since they are each connected to $v$'s inner twin which is Blue. Also, neither of $v$'s adjacent outer twins could have been changed from Red (to any other color) since $v$ was originally Red.

But this is a contradiction---no odd cycle can be 2-colored. Hence the original graph must not be 3-colorable.
\end{answer}

\subsection*{Exercises}

\begin{enumerate}
    \item 
    Brook's Theorem guarantees that $G$ is $k$-colorable for some $k$, but $k$ might be larger than the chromatic number $C(G)$.
    Compare $k$ and $C(G)$ for each of the 
    following graphs:
    \begin{enumerate}
    \item the path graph $P_n$;
    \item the cycle graph $C_n$;
    \item each of the Platonic solids; 
    \item the Petersen graph.
    \end{enumerate}
\end{enumerate}

\section{The chromatic polynomial}

In this section, we think about the number $k$-colorings of a labeled graph $G$ and we want to keep track of the outcome when $k$ is any positive integer.  We define the {\it chromatic polynomial} of $G$ as the polynomial whose value at the input $k$ is the number of $k$-colorings.  

Here is an application to a real-world scenario.
Suppose you have $k$ employees and $n$ jobs and some of the jobs conflict with each other, then how many ways are there to assign the employees to the jobs?  In this scenario, let $G$ be the labeled graph with $n$ vertices, where the vertices are labeled by the $n$ jobs, and an edge is drawn between the vertices for two jobs exactly when those jobs conflict.  Each employee is assigned one of the $k$ colors.  Making an assignment of the employees to the jobs
is the same as making a $k$-coloring of $G$.
If the number of employees is less than the chromatic number $C(G)$, then there is no way to create a proper vertex-coloring of a graph.  However, for any $k\geq C(G),$ we want to find a way to represent the number of different vertex-colorings.  In fact, we can do such a thing with a polynomial.
Now we want to study the number of ways to $k$-color a labeled graph.

\begin{theorem}
If $G$ is a labeled graph with $n$ vertices,
there is a polynomial
$P_G(x)$ 
having the property that $P_G(k)$ is the number of ways to $k$-color the vertices of $G$ for any positive integer $k$.
\end{theorem}

What is amazing about this theorem is that the number of colors can
be as large as we want, including much larger than the number of coefficients of the polynomial.

The chromatic polynomial is not a generating function because the number of $k$-colorings is the value $P_G(k)$ rather than its $k$th coefficient.

\begin{example} \label{Echrompoly2}
If $G$ is the labeled complete graph on $n$ vertices, then $P_G(x) = x(x-1)\cdots(x-n+1)$.
\end{example}

\begin{example} \label{Echrompoly3}
If $G$ is the empty labeled graph on $n$ vertices, then $P_G(x)=x^n$.
\end{example}

\begin{example} \label{Echrompoly4}
If $G$ is the labeled path graph
$P_n$, then $P_G(x)=x(x-1)^{n-1}$.
\end{example}

\begin{theorem}
If $T$ is a labeled tree with $n$ vertices, then $P_T(x)=x(x-1)^{n-1}.$
\end{theorem}

\begin{proof}
There are $k$ choices for the color of the first labeled vertex.
We now consider the vertices adjacent to this first one.
Since these vertices are adjacent to the first labeled vertex, 
the color used for the first labeled vertex is not an option.
So there are $k-1$ choices for the color of any vertex adjacent to the first labeled vertex.
We continue in this fashion.
We can color the starting vertex with any of the $k$ colors.
Then, each branch can be colored with any of the $(k-1)$ colors remaining.
The branches off those vertices are not connected to the starting vertex anymore, so while one color is removed, the color of the starting vertex is added back.
Thus, there will be $k-1$ colors left at all times for the $n-1$ vertices remaining.
\end{proof}

\begin{theorem}
Let $G$ be a labeled graph and $\alpha$ be an edge of $G$.
Let $G_{\ominus\alpha}$ be the labeled graph obtained by deleting $\alpha$ and $G_{\otimes\alpha}$ 
be the labeled graph obtained by contracting the edge $\alpha$. 
Then \[p_G(x)=p_{G_{\ominus\alpha}}(x)-p_{G_{\otimes\alpha}}(x).\]
\end{theorem}

\begin{proof}
Let $v_1,v_2$ be the two vertices of $G$ that are connected by the edge $\alpha$.  If we consider the graph $G_{\ominus\alpha}$, we have two disjoint methods of creating a $k$-coloring.  The first method is a $k$-coloring where $v_1$ and $v_2$ are the same color (since they are no longer adjacent), and the number of these $k$-coverings is $p_{G_{\otimes \alpha}}(k)$. 
The second method is a $k$-coloring where $v_1$ and $v_2$ are different colors, and the number of these $k$-colorings is $p_G(k)$.  Thus, $p_{G_{\ominus \alpha}}(k)=p_G(k)+p_{G_{\otimes \alpha}}(k)$.  Thus  $p_G(k)=p_{G_{\ominus\alpha}}(k)-p_{G_{\otimes\alpha}}(k)$.
\end{proof}
 
\begin{example}
Let $G$ be the labeled graph $C_4$.
Then $G_{\ominus \alpha}$ is a labeled $P_4$ and $G_{\otimes \alpha}$ is a labeled $C_3$.
The chromatic polynomial for the labeled $P_4$
\[p_{P_4}(x) = x(x-1)^3 = x^4-3x^3+3x^2-x.\]
The chromatic polynomial for the labeled $C_3$
is 
\[p_{C_3}(x) = x(x-1)(x-2)= x^3-3x^2+2x.\]
So
\[p_{C_4}(x) = p_{P_4}(x) - p_{C_3}(x) = 
x^4 - 4x^3+6x^2-3x=x(x-1)(x^2-3x+3).\]
\end{example} 

\begin{example}
Let $G$ be the labeled cycle graph $C_5$. 
Then $p_G(x)=x(x-1)(x-2)(x^2-2x+2)$.
\end{example}

\begin{proof}
In this case, $G_{\ominus \alpha}$ is a labeled path $P_5$ and $G_{\otimes \alpha}$ is a labeled $C_4$. Thus,  \[p_G(x)=x(x-1)^4-(x^4 - 4x^3+6x^2-3x) = x(x-1)(x-2)(x^2-2x+2)\] as desired.
\end{proof}

Since $0$ and $1$ and $2$ are roots of $p_{C_5}$ but 
$3$ is not, we see that $\chi(C_5)=3$.

\begin{corollary}
If $G$ is a labeled graph with $n$ vertices, then there is a polynomial $P_G(x)$ such that $P_G(k)$ is the
number of $k$-colorings of $G$.
Also, $P_G(x)$ has degree $n$ and is monic (meaning that the coefficient of $x^n$ in $P_G(x)$ is $1$).
\end{corollary}

\subsection*{Exercises}

\begin{enumerate}
    \item Explain Example~\ref{Echrompoly2}.
    \item Explain Example~\ref{Echrompoly3}.
    \item Explain Example~\ref{Echrompoly4}.
    \item Find the chromatic polynomial of 
    the labeled cycle $C_6$.
    
\end{enumerate}

\section{Coloring regions with two colors}

Suppose we draw circles in the plane; these circles are allowed to intersect.
The circles divide the plane into regions, where the number of regions depends on the intersection pattern of the circles.
We would like to color these regions with as few colors as possible, so that adjacent regions have different colors.
We say that two regions are adjacent if they share a boundary arc of non-zero length (two regions which only intersect at vertices are not considered to be adjacent).
How many colors are needed to color these regions?

\begin{center}
\includegraphics[width=3in]{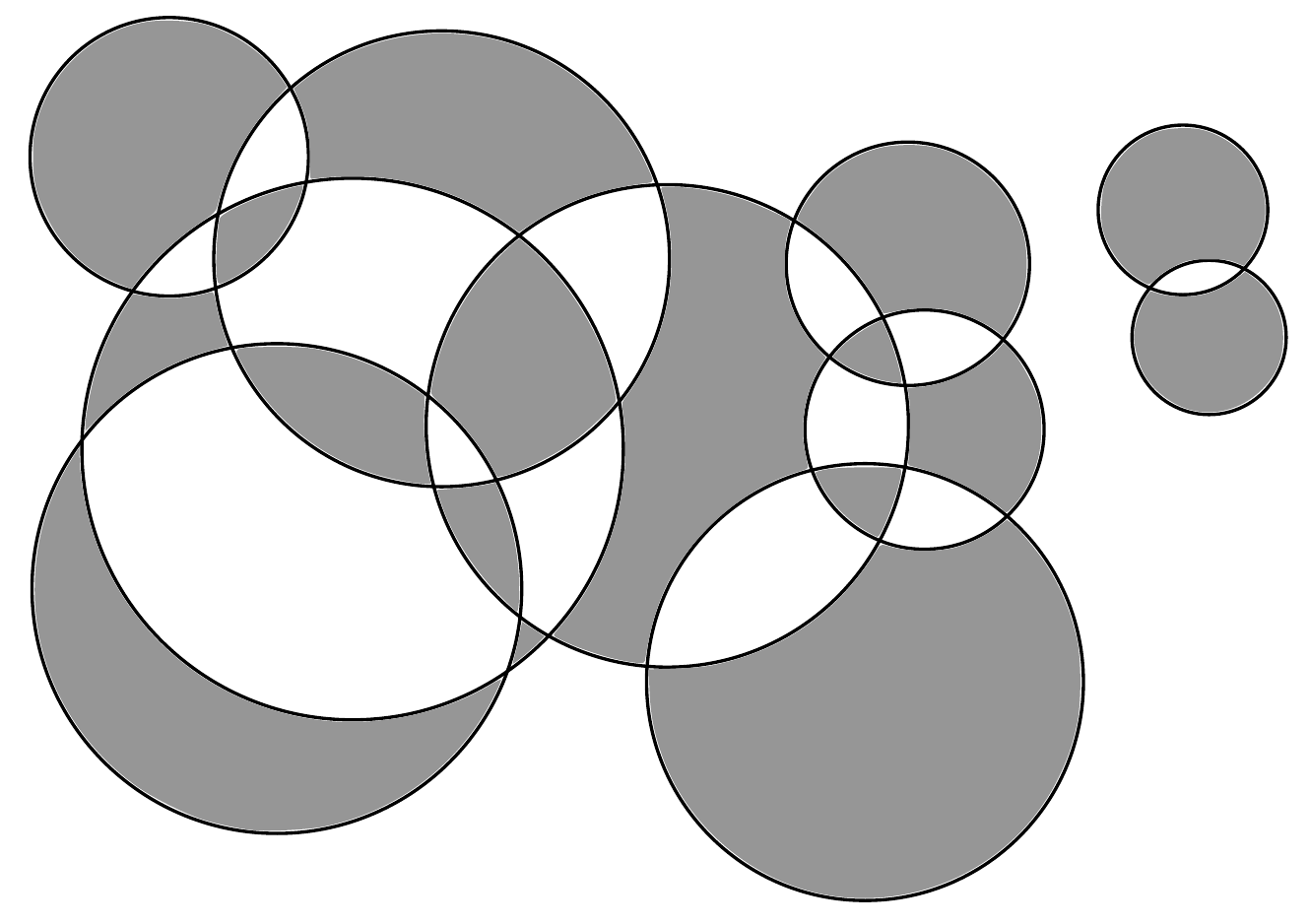}
\end{center}

\begin{theorem}
\label{thm:circles}
The regions formed by circles in the plane can be colored with two colors so that regions sharing a boundary arc are colored differently.
\end{theorem}

\begin{proof}
Let $n$ be the number of circles drawn in the plane; we proceed by induction on $n$.
The two colors we will use are gray and white.

For the base case $n=1$, note that we can draw the inside of our single circle gray, and the outside white.
\begin{center}
\includegraphics[width=0.7in]{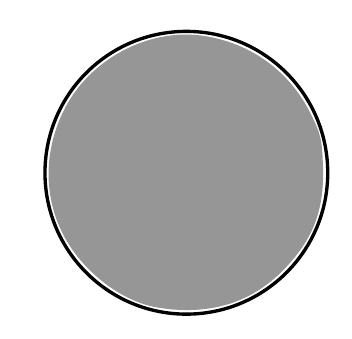}
\end{center}

For the inductive step, suppose the theorem is true for any $n-1$ circles.
Now consider $n$ circles drawn in the plane.
Select any circle $C$ and remove it.
\begin{figure}[h]
\begin{center}
\includegraphics[height=1.5in]{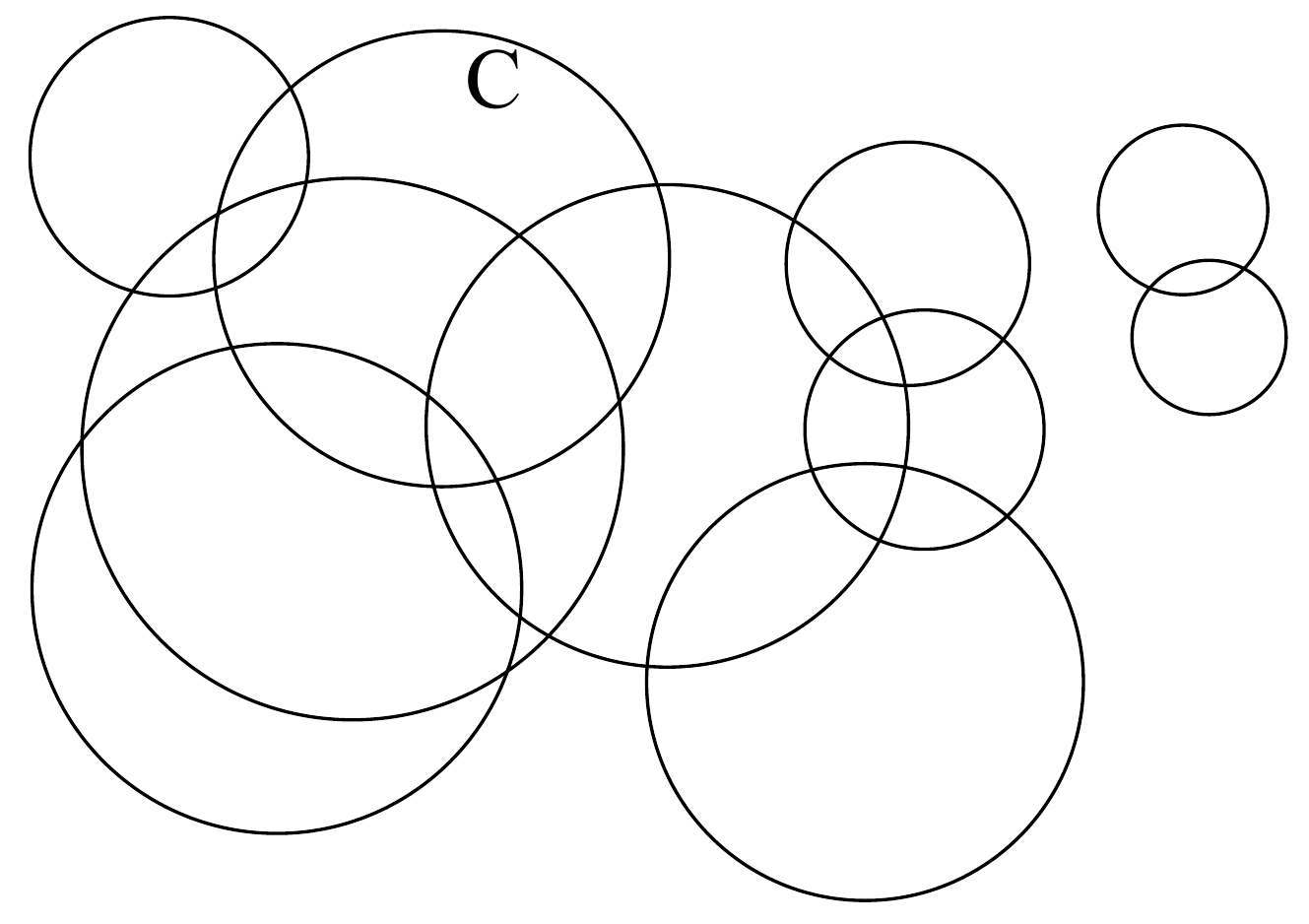}
\includegraphics[height=1.5in]{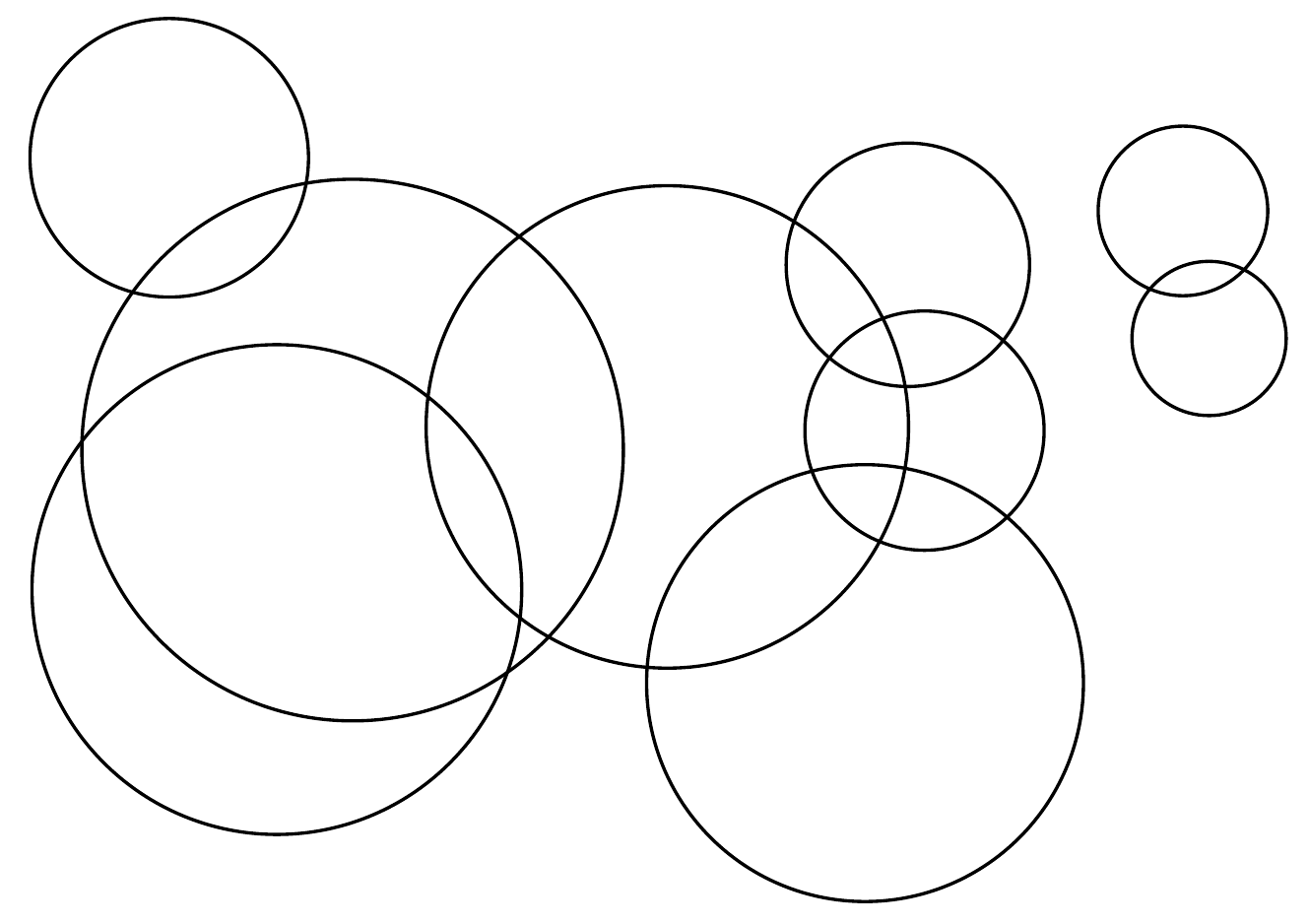}
\includegraphics[height=1.5in]{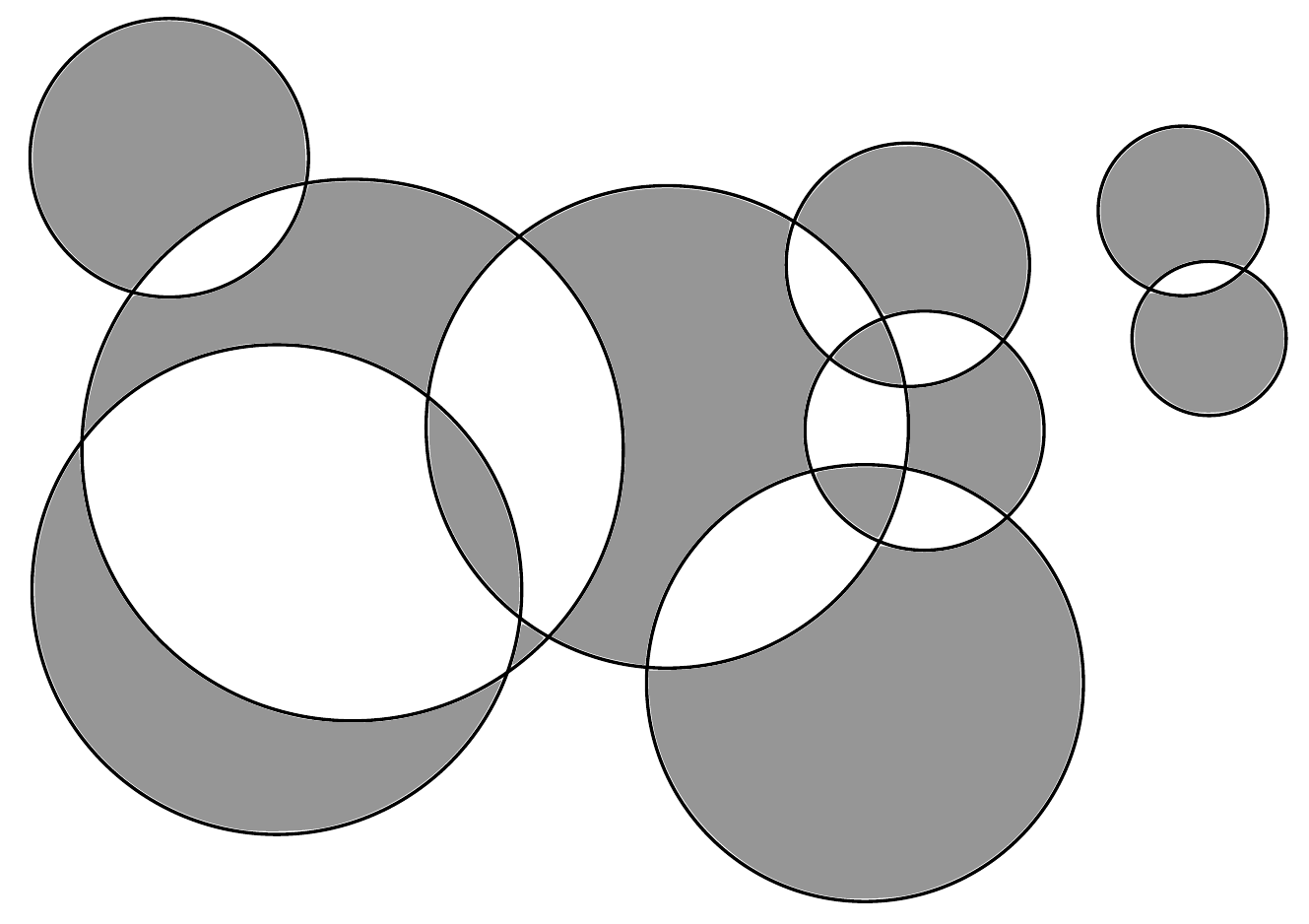}
\caption{(Left) $n$ circles.
(Middle) With circle $C$ removed, $n-1$ circles remain.
(Right) By the inductive hypothesis, we can color the regions defined by $n-1$ circles.}
\end{center}
\end{figure}
\FloatBarrier
By the inductive hypothesis, we can color the picture on the right with $n-1$ circles with two colors.
Now to color the picture with $n$ circles, simply reverse the color of each region inside circle $C$.
\begin{center}
\includegraphics[height=2in]{12-ColoringGraphs/TwoColors1.pdf}
\end{center}
Note this coloring works for arcs outside $C$, for arcs inside $C$, and for arcs which are part of $C$.
\end{proof}

\begin{remark}
\label{rem:elementary}
One can also color ``elementary school drawings" such as the following with two colors.
\begin{center}
\includegraphics[height=3in]{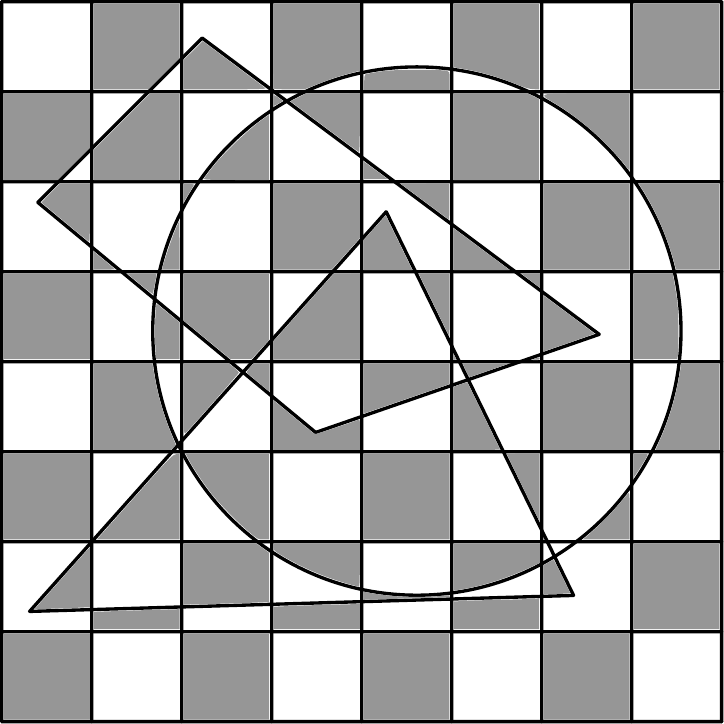}
\end{center}
Indeed, one can certainly color the underlying grid in a ``checkerboard" fashion. Inductively add objects such as the triangle or circle or quadrilateral above, repeating the inductive argument in the proof of Theorem~\ref{thm:circles}.
\end{remark}

\begin{remark} Consider the picture of $n$ circles in the plane.  The regions of the plane made by these circles can be colored with 2 colors.  Now, let's make a graph by putting one vertex in each region and drawing an edge between two vertices exactly when their regions are adjacent.  Then this graph is 2-colorable, or bipartite.

\begin{center}
\includegraphics[height=3in]{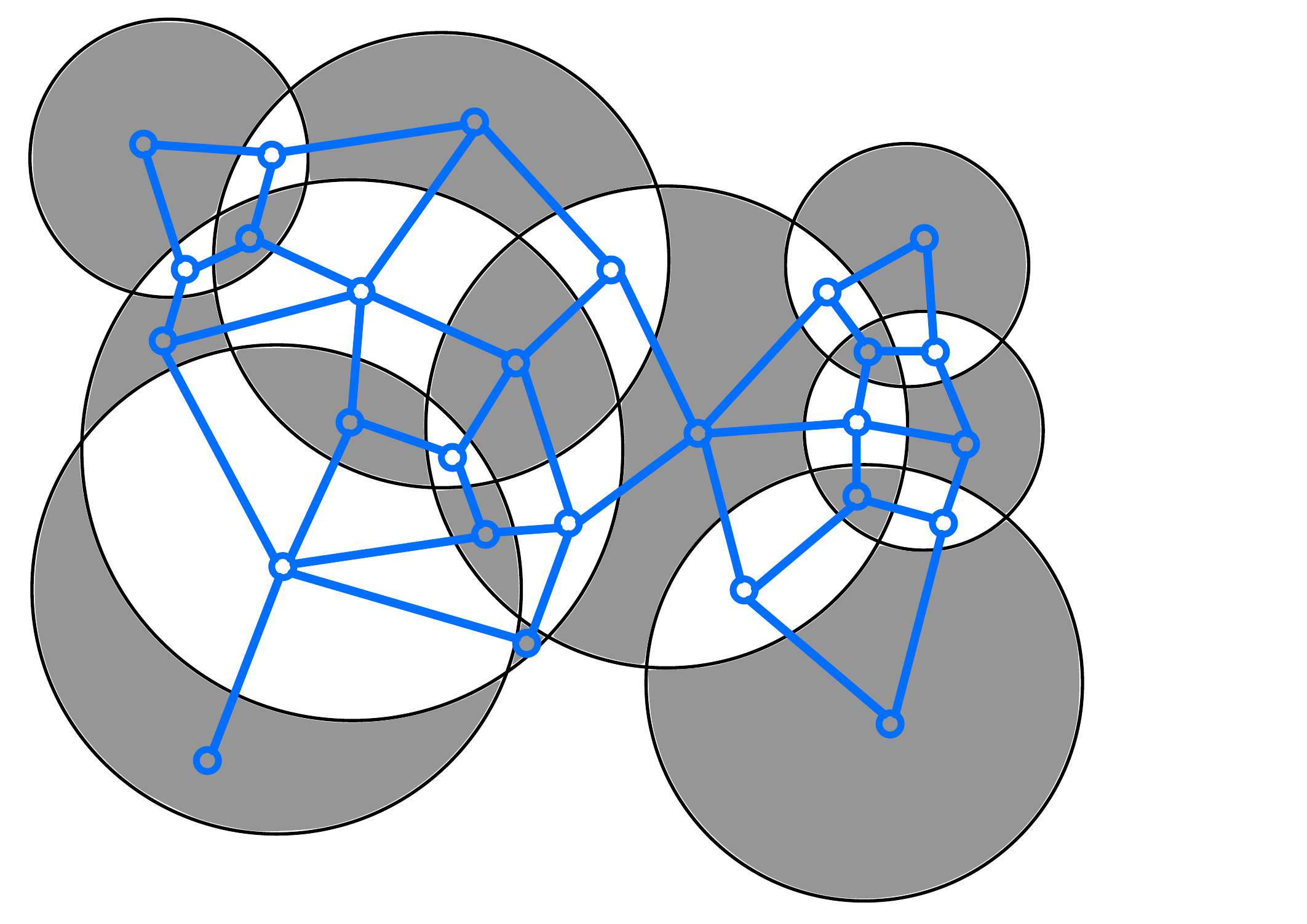}
\end{center}

By Theorem~\ref{thm:odd-cycles}, the graph drawn above has no odd cycles because it is 2-colorable.
We can also prove directly 
that this graph obtained from regions formed by circles in the plane contains no odd cycles.
The reason is that each edge of this graph crosses a single circle.
But when walking around a cycle in this graph, each circle $C$ contributes an even number of edges to the length of this cycle (alternating in, out, in, out, \ldots of $C$).
Hence every cycle in this graph has even length.
\end{remark}

\subsection*{Exercises}

\begin{enumerate}

\item Draw an ``elementary school drawing'' of your own design on top of a checkerboard, as in Remark~\ref{rem:elementary}.
Then 2-color the regions you have created.

\item Is it possible to color the below regions in the plane so that no two adjacent regions have the same color?
\begin{center}
\includegraphics[height=1.2in]{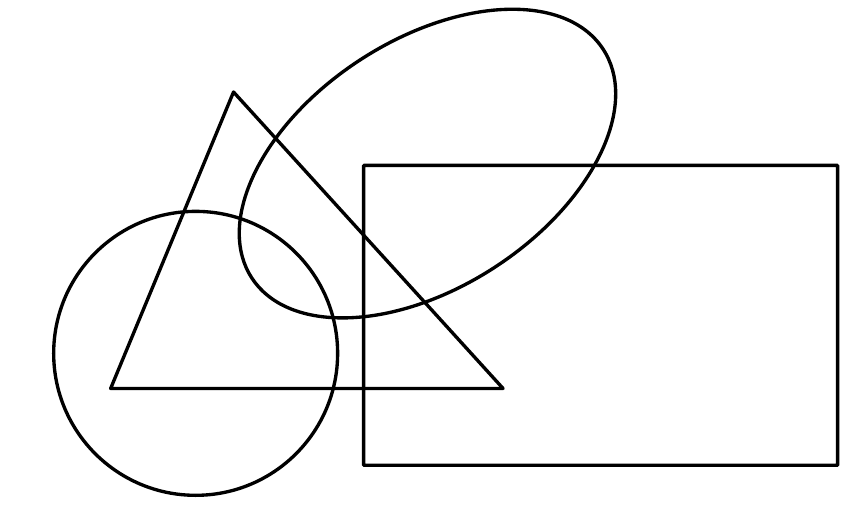}
\end{center}

\item Is it possible to color the below regions in the plane so that no two adjacent regions have the same color?
\begin{center}
\includegraphics[height=1.2in]{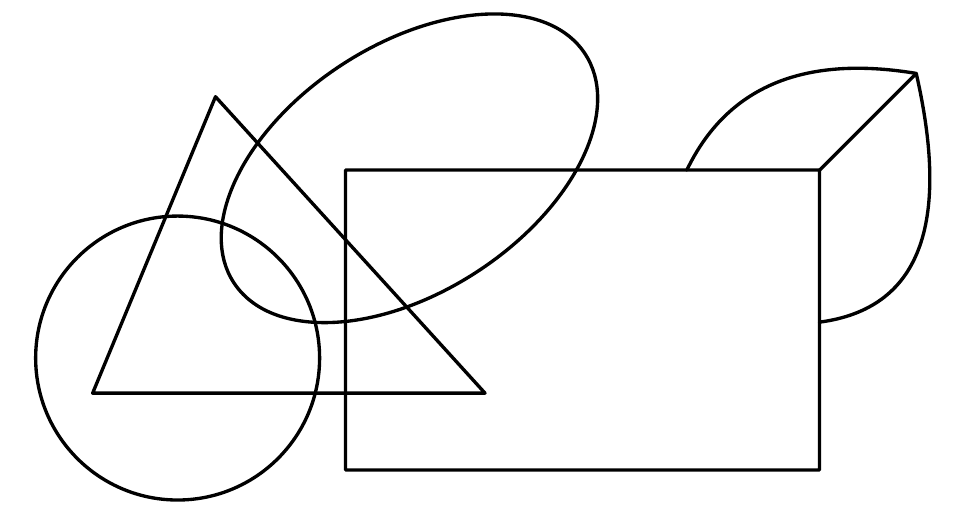}
\end{center}

\item If $G$ is a connected bipartite graph, then how many different ways are there to 2-color $G$?

\item If $G$ is a bipartite graph with $n$ different connected components, then how many different ways are there to 2-color $G$?
    
\end{enumerate}

\section{The four color theorem}

Imagine the life of an ambitious cartographer (a person who makes maps) in 1852.
To properly distinguish one region from another, it is crucial to utilize a different color for territories sharing a common border.
After making many maps for various royal affairs, you begin to question the coloring of your maps.
How many colors are needed to color planar maps in a satisfactory way?
The below map can be colored with four colors.

\begin{center}
\includegraphics[width=2in]{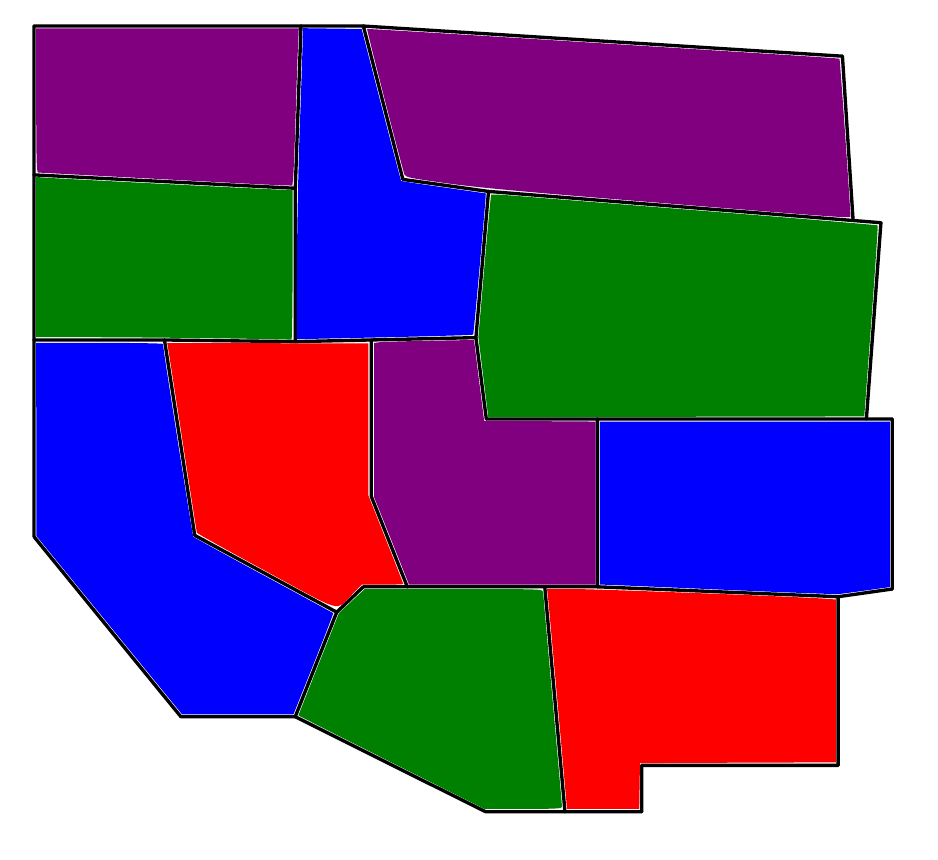}
\end{center}

Coloring maps can be related to coloring planar graphs.
Indeed, below we show how to turn this map into a planar graph, and how a valid coloring of this map translates into a vertex coloring of the corresponding planar graph.
To form the graph, add a vertex for each region, and add an edge between two vertices if the corresponding regions are adjacent.
A valid coloring of the map, with no adjacent regions of the same color, corresponds to a vertex coloring of the graph, meaning that no adjacent vertices have the same color.

\begin{figure}
\begin{center}
\includegraphics[width=4in]{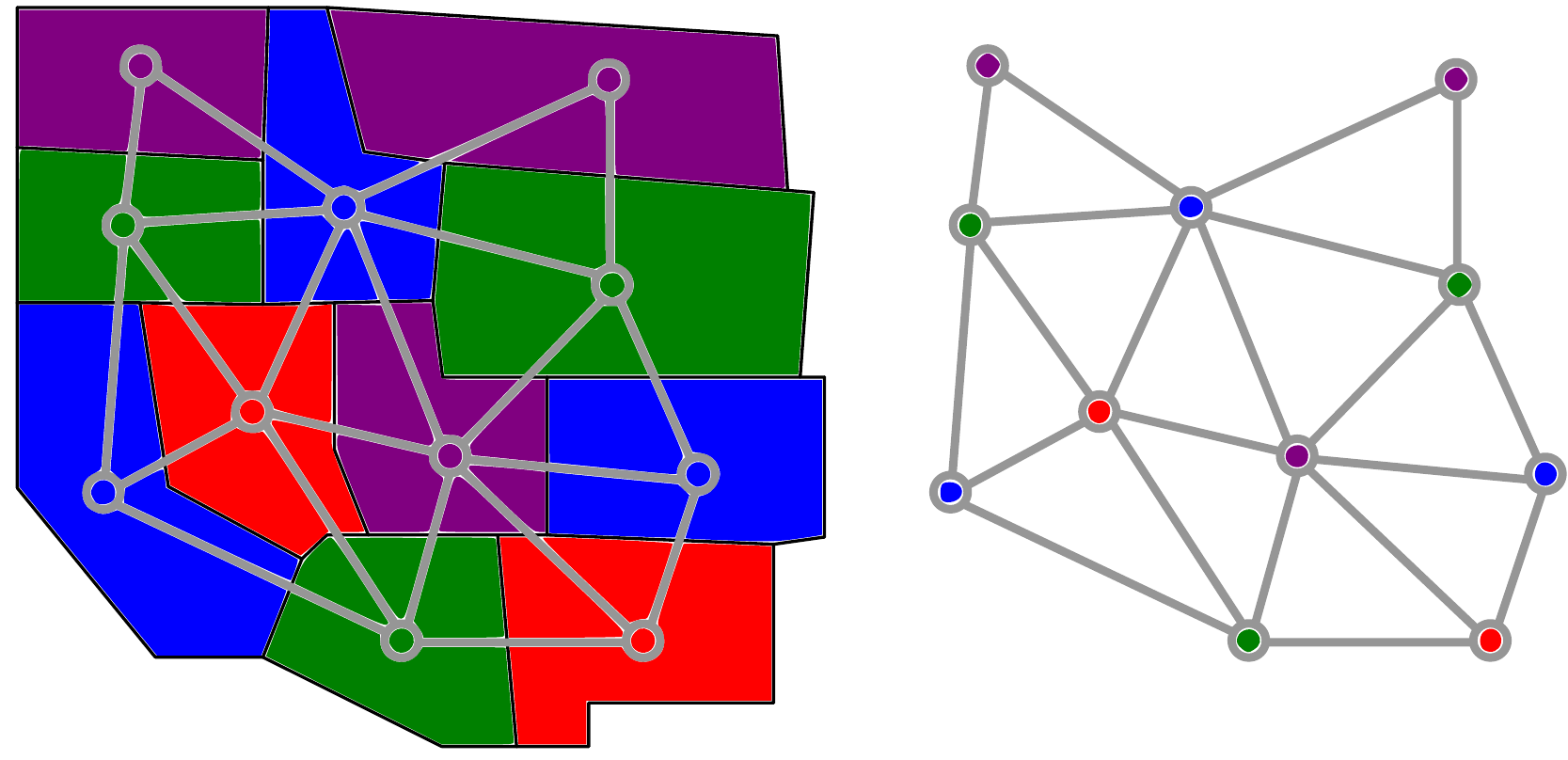}
\end{center}
\caption{On the left is a planar map, and on the right is its corresponding planar graph.}
\label{fig:planar-map-graph}
\end{figure}

Let us explain why this map requires at least 4 colors, in order for adjacent regions to have different colors.
Suppose for a contradiction 3 colors sufficed.
Without loss of generality, color the region `N' red, color the region `C' blue, and color the region `O' green.
Then necessarily the region `I' must be blue, as it is already adjacent to green and red regions, and we have only three colors at our disposal.
And similarly, the region `U' must now necessarily be green.
There is now no valid color for region `A', as it is already adjacent to regions labeled blue, red, and green.
Hence at least 4 colors are needed to color this map!

\begin{center}
\includegraphics[width=4in]{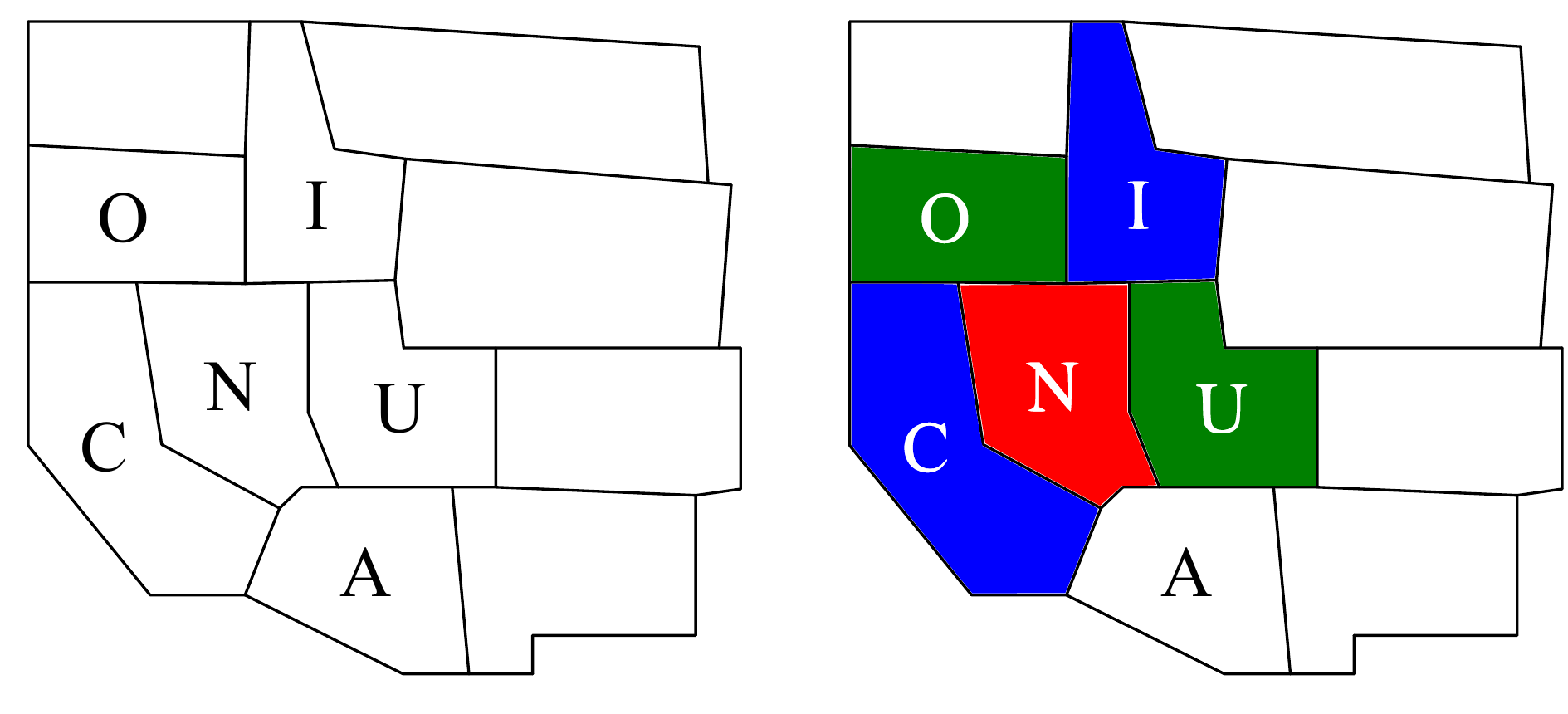}
\end{center}

Below is a smaller map that requires 4 colors so that adjacent regions are not of the same color.
Indeed, the planar graph corresponding to this map is $K_4$, the complete graph on four vertices, which requires at least 4 colors.
\begin{center}
\includegraphics[width=1.4in]{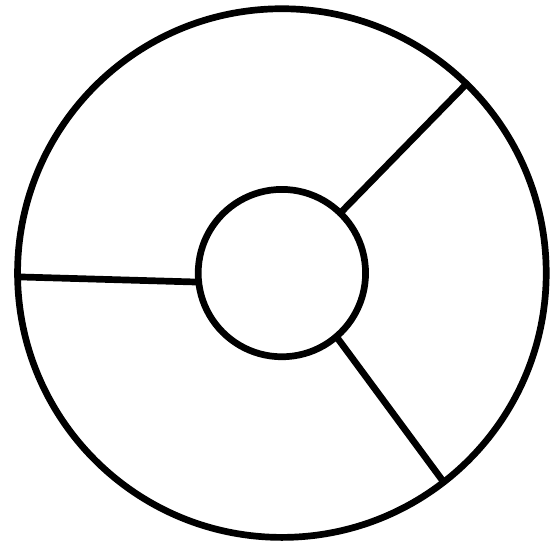}
\end{center}

After successfully coloring many large maps with only four colors, our cartographer begins to wonder if \emph{all} planar maps can be colored with four colors.
This is true and is now called the four-color theorem; it states that any map drawn on the plane or the surface of a sphere can be colored using four colors in such a way that areas with a common boundary are not the same color.

\begin{theorem}[Four color theorem]
Every planar map can be colored with only 4 colors, so that adjacent regions (sharing a boundary of positive length) have different colors.
\end{theorem}

The history of the four color theorem is quite interesting.
The question ``Can every planar map be 4-colored'' was first raised by Francis Guthrie in England in 1852.
In 1879, Alfred Kempe published an incorrect proof, which for a decade was regarded as correct before the error was found!
Indeed, in 1886 the problem of showing that any planar map can be 4-colored was posed to the students of Clifton College, with the (extremely difficult) instructions that ``no solution may exceed 1 page of manuscript and 1 page of diagrams.''
In 1976, the first correct proof the four color theorem was given by Appel \& Haken.
This proof relies on computer simulations, including 1,000 hours of computer time to check various cases.
This proof technique has raised interesting questions at the intersection of philosophy, mathematics, and computer science (recall Example~\ref{ex:four-color} in the chapter on proof techniques), such as:
\begin{itemize}
    \item What is a proof?
    \item What does it mean for a proof to be verified or certified by a computer?
    \item What are the implications when a proof is too long for every step to be understood by a team of human readers?
\end{itemize}
Computer assisted proofs have since grown in number, scope, and importance.
The present-day proof of the four color theorem has simplified greatly since 1976, but it still involves assistance from computers.

We next show, via a relatively short proof not requiring computers, that every planar map is 6-colorable.
One can give similar proofs that every planar map is 5-colorable, though they are a bit longer.

Proving that every planar map is 6-colorable is the same thing as showing that every planar graph is 6-colorable, since every planar map can be turned into a planar graph via the process shown in Figure~\ref{fig:planar-map-graph}.
To show that every planar map is 6-colorable, we will need the following two lemmas.

\begin{lemma}
\label{lem:degreeAtMost5}
Every planar graph has a vertex of degree at most 5.
\end{lemma}

\begin{proof}
We assume that every vertex has degree at least 6 and will find a contradiction. Let $e$ be the number of edges and $v$ be the number of vertices. Then
\begin{align*}
2e&=\text{total sum of all vertex degrees}&&\text{by Theorem~\ref{thm:graphDegreeTwiceEdges}}\\
&\ge6v,
\end{align*}
since we assumed that every vertex has degree at least 6.
Dividing by two, this gives $e\ge 3v$.
This contradicts Exercise \ref{Exercise7} in Section \ref{sec:not-planar}, which says a planar graph can have at most $3v-6$ edges.
\end{proof}

Here is a generalization of Brook's Theorem (Theorem~\ref{thm:Brooks}).

\begin{theorem}
\label{T:BrooksExtension}
Let $d$ be a positive integer.
Suppose every subgraph of a graph $G$ contains a vertex whose degree is less than or equal to $d$. Then $G$ is $(d+1)$-colorable.
\end{theorem}

\begin{proof}
Let $v$ be the number of vertices in the graph $G$.
We will remove vertices (and all incident edges) one-by-one to form a sequence of graphs $G=G_v$, $G_{v-1}$, $G_{v-2}$, \ldots, $G_3$, $G_2$, $G_1$, where each $G_i$ has exactly $i$ vertices.

By assumption, the graph $G=G_v$ has a vertex of degree at most $d$;
let's call this vertex $b_v$.  We remove the vertex $b_v$ and all its incident edges to form the subgraph $G_{v-1}$.
We remark that if $G_{v-1}$ is $(d+1)$-colorable, then so is $G_v$ since the number of colors ($d+1$) is greater than the number of vertices $b_v$ is adjacent to.

By assumption, the subgraph $G_{v-1}$ has a vertex of degree at most $d$; let's call this vertex $b_{v-1}$.  We remove $b_{v-1}$ and all its incident edges to form subgraph $G_{v-2}$.
Similarly, if $G_{v-2}$ is $(d+1)$-colorable, then so is $G_{v-1}$ since the number of colors is greater than the degree of the vertex $b_{v-1}$.

\begin{figure}
\centering
\includegraphics[width=\textwidth]{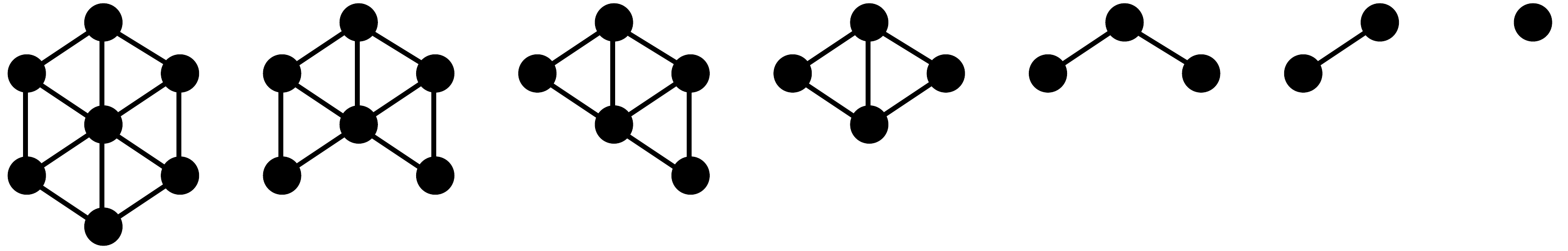}
\caption{Figure for the proof of Theorem~\ref{T:BrooksExtension}, showing how to collapse from the graph $G=G_7$ down to the graph $G_1$ which is a single vertex, at each point removing a vertex of degree at most $3$.}
\end{figure}

We continue in this way until we obtain a subgraph $G_1$ with a single vertex.
At each stage, we note that if $G_{i-1}$ is $(d+1)$-colorable, then so is $G_i$.

Note $G_1$ is clearly $(d+1)$-colorable because it is 1-colorable.
Hence it follows that $G_2$ is $(d+1)$-colorable, and therefore $G_3$ is $(d+1)$-colorable.
We continue in this way until we observe that $G_{v-2}$ is $(d+1)$-colorable, and therefore $G_{v-1}$ is $(d+1)$-colorable, and therefore $G_v$ is $(d+1)$-colorable.
\end{proof}

\begin{corollary}
Any planar graph $G$, and hence any planar map, is 6-colorable.
\end{corollary}

\begin{proof}
Any subgraph of a planar graph $G$ is planar and hence has a vertex of degree at most $d=5$ by Lemma~\ref{lem:degreeAtMost5}.
So Theorem~\ref{T:BrooksExtension} implies that $G$ is colorable with $d+1=5+1=6$ colors.
\end{proof}

\subsection*{Exercises}\label{sec:coloring-problems}

\begin{enumerate}
\item Read the philosophical objections section on the Wikipedia page on computer-assisted proofs.
Explain your opinion on the validity of computer-assisted proofs. 

\item Show that the graph drawn below is not 3-colorable but is 4-colorable.

\begin{center}
\includegraphics[width=.35\textwidth]{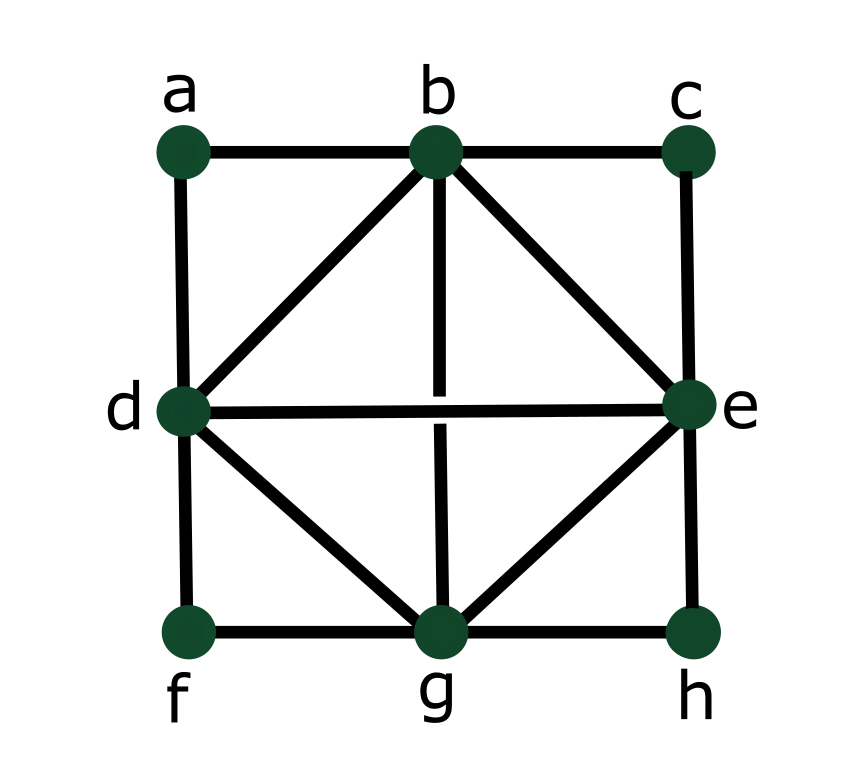}
\end{center}

\item Let $G$ be a graph such that $d+1$ vertices have an arbitrarily large degree and all of the other vertices have degree at most $d$.
Prove that $G$ is $(d+1)$-colorable.

\item
Show that a planar graph $G$ with 8 vertices and 13 edges cannot be 2-colored.

\emph{Hint: Show that in a planar map corresponding to $G$, one of the faces must be a triangle.}

\item
\begin{enumerate}
\item A set of solar experiments is to be made at observatories. Each experiment begins on a given day of the year and ends on a given day (each experiment is repeated for several years). An observatory can perform only one experiment at a time.\\
Experiment A runs Sept 2 to Jan 3,\\
experiment B from Oct 15 to March 10,\\
experiment C from Nov 20 to Feb 17,\\
experiment D from Jan 23 to May 30,\\
experiment E from April 4 to July 28,\\
experiment F from April 30 to July 28, and\\
experiment G from June 24 to Sept 30.\\
The problem we'll solve in (b) is: ``What is the minimimum number of observatories required to perform a given set of experiments annually?" Model this scheduling problem as a graph-coloring problem, and draw the corresponding graph.
\item Find a minimal coloring of the graph in (a) (i.e.\ the smallest $k$ such that this graph is $k$-colorable), and show that fewer colors will not suffice.
\end{enumerate}

\item Show that a planar graph $G$ with 10 vertices and 17 edges cannot be 2-colored.

\item True or False: A graph that is not $d$-colorable has a vertex of degree at least $d$.

\item True or False: A 2-colorable graph with $n$ connected components can be 2-colored in exactly $2^n$ different ways.

\end{enumerate}